\documentclass{amsart}
\usepackage{amsmath,graphicx,amsfonts,amssymb,amsthm,hyperref,amscd,esint,bbm,verbatim,mathabx,abstract}
\usepackage[margin=1in]{geometry}
\usepackage{mathtools}
\mathtoolsset{showonlyrefs}
\newtheorem{theorem}{Theorem}[section]
\newtheorem{proposition}{Proposition}[section]

\newtheorem{lemma}[theorem]{Lemma}
\title{Infinite time blow-up solutions to the energy critical wave maps equation}
\date{}
\author{Mohandas Pillai} 
\subjclass[2010]{Primary 35L05, 35Q75}
\numberwithin{equation}{section}
\begin{document}
\maketitle

\begin{abstract} We consider the wave maps problem with domain $\mathbb{R}^{2+1}$ and target $\mathbb{S}^{2}$ in the 1-equivariant, topological degree one setting. In this setting, we recall that the soliton is a harmonic map from $\mathbb{R}^{2}$ to $\mathbb{S}^{2}$, with polar angle equal to $Q_{1}(r) = 2 \arctan(r)$. By applying the scaling symmetry of the equation, $Q_{\lambda}(r) = Q_{1}(r \lambda)$ is also a harmonic map, and the family of all such $Q_{\lambda}$ are the unique minimizers of the harmonic map energy among finite energy, 1-equivariant, topological degree one maps. In this work, we construct infinite time blowup solutions along the $Q_{\lambda}$ family. More precisely, for $b>0$, and for all $\lambda_{0,0,b} \in C^{\infty}([100,\infty))$ satisfying, for some $C_{l}, C_{m,k}>0$,
$$\frac{C_{l}}{\log^{b}(t)} \leq \lambda_{0,0,b}(t) \leq \frac{C_{m}}{\log^{b}(t)}, \quad |\lambda_{0,0,b}^{(k)}(t)| \leq \frac{C_{m,k}}{t^{k} \log^{b+1}(t) }, k\geq 1 \quad t \geq 100$$
there exists a wave map with the following properties. If $u_{b}$ denotes the polar angle of the wave map into $\mathbb{S}^{2}$, we have
$$u_{b}(t,r) = Q_{\frac{1}{\lambda_{b}(t)}}(r) + v_{2}(t,r) + v_{e}(t,r), \quad t \geq T_{0}$$ 
where
$$-\partial_{tt}v_{2}+\partial_{rr}v_{2}+\frac{1}{r}\partial_{r}v_{2}-\frac{v_{2}}{r^{2}}=0$$
$$||\partial_{t}(Q_{\frac{1}{\lambda_{b}(t)}}+v_{e})||_{L^{2}(r dr)}^{2}+||\frac{v_{e}}{r}||_{L^{2}(r dr)}^{2} + ||\partial_{r}v_{e}||_{L^{2}(r dr)}^{2} \leq \frac{C}{t^{2} \log^{2b}(t)}, \quad t \geq T_{0}$$
and $$\lambda_{b}(t) = \lambda_{0,0,b}(t) + O\left(\frac{1}{\log^{b}(t) \sqrt{\log(\log(t))}}\right)$$
 \end{abstract}
\tableofcontents
\section{Introduction}
This paper considers the wave maps equation, with domain $\mathbb{R}^{2+1}$ and target $\mathbb{S}^{2}$. This equation is the Euler-Lagrange equation associated to the functional
$$\mathcal{L}(\Phi) = \int_{\mathbb{R}^{2+1}} \langle \partial^{\alpha} \Phi(t,x),\partial_{\alpha}\Phi(t,x) \rangle_{g(\Phi(t,x))} dtdx$$
where $g$ denotes the round metric on $\mathbb{S}^{2}$, and $\Phi: \mathbb{R}^{2+1} \rightarrow \mathbb{S}^{2}$. We will only work with 1-equivariant maps $\Phi$, which we describe by first regarding $\Phi$ as a map into $\mathbb{R}^{3}$ with unit norm, and then writing
$$\Phi_{u}(t,r,\phi) = (\cos(\phi) \sin(u(t,r)), \sin(\phi)\sin(u(t,r)),\cos(u(t,r)))$$
where $(r,\phi)$ are polar coordinates on $\mathbb{R}^{2}$.  Then, the wave maps equation becomes
\begin{equation}\label{wavemaps}-\partial_{tt}u+\partial_{rr}u+\frac{1}{r}\partial_{r}u=\frac{\sin(2 u)}{2r^{2}}\end{equation}
\cite{stz} studied (a more general problem which includes) the Cauchy problem associated to \eqref{wavemaps}, with data $(u_{0},u_{1})$ such that 
$$(x_{1},x_{2}) \mapsto (\frac{x_{1}u_{0}(r)}{r},\frac{x_{2} u_{0}(r)}{r})\in H^{1}_{\text{loc}}(\mathbb{R}^{2}), \quad (x_{1},x_{2}) \mapsto (\frac{x_{1}u_{1}(r)}{r},\frac{x_{2} u_{1}(r)}{r})\in L^{2}_{\text{loc}}(\mathbb{R}^{2})$$ 
We will say that $u$ is a finite energy solution to \eqref{wavemaps} if $u$ is a distributional solution, with $\Phi_{u}\in C^{0}_{t}\dot{H}^{1}(\mathbb{R}^{2})$ and $\partial_{t}\Phi_{u}\in C^{0}_{t}L^{2}(\mathbb{R}^{2})$.  
Define the energy by
\begin{equation}\label{equivariantenergy} E_{\text{WM}}(u,v) = \pi \int_{0}^{\infty} \left(v^{2}+\frac{\sin^{2}(u)}{r^{2}} + \left(\partial_{r}u\right)^{2}\right) r dr\end{equation}
Then, if $u$ solves \eqref{wavemaps}, $E_{\text{WM}}(u,\partial_{t}u)$ is formally independent of time. We recall the ground state soliton, $Q_{1}(r) = 2\arctan(r)$ is a solution to \eqref{wavemaps}, with the property that the family of all $Q_{\lambda}(r) = Q_{1}(r \lambda)$ are the unique minimizers of $E_{\text{WM}}(u,0)$ among finite energy $u$ with $\Phi_{u}$ having topological degree one. In order to state our main theorem, it will also be useful to consider the following non-degenerate energy which will be used to measure perturbations of $Q_{\lambda}$.
\begin{equation} \label{rsquaredpotenergy} E(u,v) = \pi \int_{0}^{\infty}\left(v^{2}+(\partial_{r}u)^{2}+\frac{u^{2}}{r^{2}}\right) r dr\end{equation} 
The quantity $E(u,\partial_{t}u)$ is formally conserved for solutions to the wave equation
\begin{equation} \label{rsquaredpoteqn} -\partial_{tt}u+\partial_{rr}u+\frac{1}{r}\partial_{r}u-\frac{u}{r^{2}}=0\end{equation}
 We consider the problem of constructing 1-equivariant, topological degree one, global, non-scattering solutions to \eqref{wavemaps}, which have energy strictly greater than $E_{\text{WM}}((Q_{1},0))$. The authors of \cite{ckls} classified 1-equivariant topological degree one solutions to the wave maps equation with energy strictly between $E_{\text{WM}}((Q_{1},0))$ and $3 E_{\text{WM}}((Q_{1},0))$. The part of their result which is relevant for this paper is the statement that any such solutions which are global in time admit a generic decomposition into the form 
\begin{equation}\label{udecomp}u=Q_{\frac{1}{\lambda(t)}}+\phi_{L}+\epsilon\end{equation}
where $\phi_{L}$ solves \ref{rsquaredpoteqn}, $\lambda(t) = o(t), \quad t \rightarrow \infty$, and $\epsilon \rightarrow 0$ (in an appropriate sense) as $t \rightarrow \infty$.
There are many possible asymptotic behaviors of $\lambda$ allowed by the above result, and, according to \cite{ckls}, there were no known constructions of solutions of the above form, with $\lambda(t) \rightarrow 0$ or $\lambda(t) \rightarrow \infty$. \\
\\
This paper constructs a family of finite energy solutions to \eqref{wavemaps}, say $\{u_{b}\}_{b>0}$, where each $u_{b}$ can be decomposed as in \eqref{udecomp}, with $\lambda_{b}(t)=\lambda_{0,0,b}(t) + O\left(\frac{1}{\log^{b}(t)\sqrt{\log(\log(t))}}\right)$  (see the main theorem below for the sense in which the $\epsilon$ term for our solution, which is called $v_{e}$, vanishes as $t$ goes to infinity). Here, $\lambda_{0,0,b} \in C^{\infty}([100,\infty))$ satisfies, for some $b, C_{l}, C_{m,k}>0$,
$$\frac{C_{l}}{\log^{b}(t)} \leq \lambda_{0,0,b}(t) \leq \frac{C_{m}}{\log^{b}(t)}, \quad |\lambda_{0,0,b}^{(k)}(t)| \leq \frac{C_{m,k}}{t^{k} \log^{b+1}(t) },k \geq 1 \quad t \geq 100$$ To the author's knowledge, these are the first examples of such solutions of \eqref{wavemaps}. First, we will prove the following theorem, which corresponds to the special case of $\lambda_{0,0,b}(t) = \frac{1}{\log^{b}(t)}$. Solutions with more general $\lambda_{0,0,b}$ can be constructed with a slight modification of the proof of this special case, outlined in the appendix. 
\begin{theorem}\label{mainthm} For each $b>0$, there exists $T_{0}>0$ and a (real-valued) finite energy solution $u_{b}$ to \eqref{wavemaps} for $t \geq T_{0}$, of the form
$$u_{b}(t,r) = Q_{\frac{1}{\lambda_{b}(t)}}(r) + v_{2}(t,r) + v_{e}(t,r)$$
where $\lambda_{b} \in C^{4}([T_{0},\infty))$,
$$E_{\emph{WM}}(u_{b},\partial_{t}u_{b}) < \infty$$ 
$$-\partial_{tt}v_{2}+\partial_{rr}v_{2}+\frac{1}{r}\partial_{r}v_{2}-\frac{v_{2}}{r^{2}}=0, \quad E(v_{2},\partial_{t}v_{2}) < \infty, \quad v_{2} \in C^{\infty}([T_{0},\infty) \times [0,\infty))$$
$$E(v_{e},\partial_{t}\left(Q_{\frac{1}{\lambda_{b}}}+v_{e}\right)) \leq \frac{C}{t^{2} \log^{2b}(t)}, \quad t \geq T_{0}$$
and $$\lambda_{b}(t) = \frac{1}{\log^{b}(t)} + e(t), \quad |e(t)| \leq \frac{C}{\log^{b}(t)\sqrt{\log(\log(t))}},\quad |e^{(j)}(t)| \leq \begin{cases}\frac{C}{t^{j} \log^{b+1}(t) \sqrt{\log(\log(t))}}, \quad j=1,2\\
\frac{C}{t^{j}\log^{b+1}(t)}, \quad j=3,4\end{cases}$$
\end{theorem}
Regarding more general $\lambda_{0,0,b}$, we have the following. For $b>0$, let $\Lambda_{b}$ be the set of $f \in C^{\infty}([100,\infty))$  such that there exist $C_{l},C_{m},C_{m,k} >0$ with 
$$\frac{C_{l}}{\log^{b}(t)} \leq f(t) \leq \frac{C_{m}}{\log^{b}(t)}, \quad |f^{(k)}(t)| \leq \frac{C_{m,k}}{t^{k} \log^{b+1}(t)}, k\geq 1, \quad t \geq 100$$
\begin{theorem}\label{genthm}
Let $b>0$. For all $\lambda_{0,0,b} \in \Lambda_{b}$, there exists $T_{0}>100$ and a (real-valued) finite energy solution $u_{b}$ to \eqref{wavemaps} for $t \geq T_{0}$, of the form
$$u_{b}(t,r) = Q_{\frac{1}{\lambda_{b}(t)}}(r) + v_{2}(t,r) + v_{e}(t,r)$$
where $\lambda_{b} \in C^{4}([T_{0},\infty))$,
$$E_{\emph{WM}}(u_{b},\partial_{t}u_{b}) < \infty$$ 
$$-\partial_{tt}v_{2}+\partial_{rr}v_{2}+\frac{1}{r}\partial_{r}v_{2}-\frac{v_{2}}{r^{2}}=0, \quad E(v_{2},\partial_{t}v_{2}) < \infty, \quad v_{2} \in C^{\infty}([T_{0},\infty) \times [0,\infty))$$
$$E(v_{e},\partial_{t}\left(Q_{\frac{1}{\lambda_{b}}}+v_{e}\right)) \leq \frac{C}{t^{2} \log^{2b}(t)}, \quad t \geq T_{0}$$
and $$\lambda_{b}(t) = \lambda_{0,0,b}(t) + e(t), \quad |e(t)| \leq \frac{C}{\log^{b}(t)\sqrt{\log(\log(t))}},\quad |e^{(j)}(t)| \leq \begin{cases}\frac{C}{t^{j} \log^{b+1}(t) \sqrt{\log(\log(t))}}, \quad j=1,2\\
\frac{C}{t^{j}\log^{b+1}(t)}, \quad j=3,4\end{cases}$$
\end{theorem}
\bigskip
\noindent \emph{Remark} 1. Our proof yields more information about the regularity of $v_{e}$ appearing in the main theorem. In particular, we have
$$v_{e}=v_{e,0}+v_{6}$$
where the function $v_{e,0}$ is fairly explicit (but complicated), and $v_{6}$ is constructed with a fixed point argument, but has more Sobolev regularity than what follows from $u$ having finite energy, namely:  if $(r,\theta)$ denote polar coordinates on $\mathbb{R}^{2}$, then
$$(t,r,\theta) \mapsto e^{i \theta} v_{6}(t,r) \in C^{0}_{t}([T_{0},\infty),H^{2}(\mathbb{R}^{2}))$$
$$(t,r,\theta) \mapsto e^{i \theta} \partial_{t}v_{6}(t,r) \in C^{0}_{t}([T_{0},\infty),H^{1}(\mathbb{R}^{2}))$$
\emph{Remark} 2. Theorem \ref{genthm} includes solutions with mildly oscillating $\lambda_{0,0,b}$, such as
$$\lambda_{0,0,b} = \frac{2+\sin(\log(\log(t)))}{\log^{b}(t)}, \quad t \geq 100$$
\emph{Remark} 3. It is expected that the method used in this paper can be extended to allow for the construction of such solutions with $\lambda(t) \rightarrow \infty, \quad t \rightarrow \infty$.\\
\\
Finally, this method should also be applicable to higher equivariance classes, the energy critical Yang-Mills problem in $4$ dimensions with gauge group $SO(4)$, as well as the quintic, focusing semilinear wave equation in $\mathbb{R}^{1+3}$. All of these extensions are work in progress by the author. \\
\\
\emph{Remark} 4. The leading behavior of $\lambda(t)$ is partially motivated by the fact that, for $k=1,2$,
$$\frac{|\lambda^{(k)}(t)|}{\lambda(t)} \leq \frac{C}{t^{k}\log(t)} =o(t^{-k}), \quad t \rightarrow \infty$$
This fact is very important so that certain error terms appearing in this work which involve the ``transferrence operator'' defined in \cite{kst} can be treated perturbatively. In particular, the method of this work can not be directly applied to construct solutions with $\lambda(t)$ being all powers of $t$.\\
\\
In order to understand this work in a larger context, we review previous work regarding dynamical behavior of solutions to the critical wave maps problem with $\mathbb{S}^{2}$ target. Firstly, the works \cite{stthreshold1} and \cite{stthreshold2} (which apply to much more general targets than $\mathbb{S}^{2}$) show that for data with energy strictly less than that of the lowest energy non-trivial harmonic map, one has global well-posedness and scattering. In the 1-equivariant setting, \cite{ckls1} showed that global existence and scattering holds for  smooth, topological degree zero data with energy less than twice that of $Q_{1}$ (which is the appropriate threshold in this setting). An analogous result, but without the equivariance assumption, is implied by \cite{lo}.  Then, \cite{j} constructed, for all $k \geq 2$, $k$ equivariant, topological degree zero, two-bubble solutions with energy exactly equal to 2 times the energy of $Q_{k}(r) := Q_{1}(r^{k})$.  In \cite{jl}, all $k$-equivariant solutions (with $k \geq 2$) of toplogical degree zero, and energy exactly equal to twice the energy of $Q_{k}(r)$ were classified. A classification of 1-equivariant, topological degree 0 solutions with energy equal to twice $E(Q_{1})$ was obtained in \cite{r}, and shows that the dynamical behavior of such solutions can be quite different than in the higher equivariance case. \cite{r} also constructs a finite-time blow-up solution in this setting. The methods used in this work differ significantly from these previous works.\\
\\
1-equivariant, topological degree one, finite time blow-up solutions have been constructed in \cite{kst}, and the work \cite{gk} extended the range of possible blow-up rates of these solutions. The method of construction of the ansatz of this paper differs significantly from that used in \cite{kst}. However, we do eventually use some technical information, most importantly, the distorted Fourier transform, from \cite{kst} as part of the process to complete our ansatz to an exact solution. Analogs of the solutions of \cite{kst} for the $4+1$-dimensional Yang-Mills equation with gauge group $SO(4)$ and the quintic, focusing nonlinear wave equation in $\mathbb{R}^{1+3}$ were also constructed in \cite{kstym} and \cite{kstslw}. The work of \cite{km} studies the stability of the solutions of \cite{kst} and \cite{gk} under certain equivariant perturbations. In addition, \cite{rs} constructs finite-time blow-up, $k-$equivariant solutions with $k \geq 4$, and \cite{rr} constructs finite time blow-up solutions for the $(4+1)$-dimensional Yang-Mills problem with gauge group $SO(4)$, as well as in all equivariance classes for energy critical wave maps. The method of this work is quite different from the methods used in \cite{rr} and \cite{rs}. \\
\\
The work of \cite{bkt} constructs modulated soliton solutions, where $\lambda$ is bounded away from zero and infinity for all time. Some facts from \cite{bkt} about the wave maps equation linearized around $Q_{1}$ will be utilized in this paper, but the ansatz construction is again quite different. Finally, infinite time blow-up and infinite time relaxation solutions to the quintic, focusing, nonlinear wave equation in $\mathbb{R}^{1+3}$ have been constructed in \cite{dk}, but the method of this paper is again quite different.
\subsection{Acknowledgements}
The author thanks his adviser, Daniel Tataru, for suggesting the problem, and many useful discussions. This material is based upon work partially supported by the National Science Foundation under Grant No.  DMS-1800294.
\section{Notation}
We will occasionally use the Hankel transform of order $1$, and it will be denoted as
\begin{equation} \widehat{f}(\xi) = \int_{0}^{\infty} f(r) J_{1}(r \xi) r dr\end{equation}
The main Fourier transform we will use is the distorted Fourier transform of \cite{kst}, which we denote by
\begin{equation} \mathcal{F}(f)(\xi) = \int_{0}^{\infty} \phi(r,\xi) f(r) dr\end{equation}
Briefly, we will make some use of the distorted Fourier transforms of \cite{bkt}, which are denoted by $\mathcal{F}_{H}$ and $\mathcal{F}_{\tilde{H}}$. Finally, we use the same notation as \cite{bkt} for the following norm
\begin{equation} ||f||_{\dot{H}^{1}_{e}}^{2} = ||\partial_{r}f||_{L^{2}(r dr)}^{2} + ||\frac{f}{r}||_{L^{2}(r dr)}^{2}\end{equation}
We denote by $\phi_{0}$, the zero resonance of the elliptic part of the wave equation linearized around $Q_{1}$:
\begin{equation}\phi_{0}(r) = \frac{d}{d\lambda}\Bigr|_{\lambda=1}Q_{\lambda}(r)=\frac{2r}{1+r^{2}}\end{equation}
(This notation for $\phi_{0}$ is the same as that used in \cite{bkt}, but is different from that used in \cite{kst}).

\section{Overview of the proof}
We remind the reader that we will prove theorem \ref{mainthm}, with the extra arguments needed to establish theorem \ref{genthm} summarized in the appendix. The argument used in this paper proceeds in two parts. First, we construct an approximate solution, $u$, to \eqref{wavemaps}. Second, we construct an exact solution to \eqref{wavemaps} which is close to our approximate one. \\
\\
\textbf{Part 1: Constructing the approximate solution} \\
\\
One way to understand the intuition behind our approximate solution is as follows. One could start by looking for an approximate solution to \eqref{wavemaps} which consists of a dynamically rescaled soliton and a radiation field, along with an appropriate compatibility condition between the soliton length scale, and the radiation field. This would correspond to an approximate solution of the form $u_{a}=Q_{\frac{1}{\lambda(t)}}+v_{2}$, where $v_{2}$ is some solution to \eqref{rsquaredpoteqn}, $\lambda(t)$ is not yet chosen, and one can look for an appropriate relation between $v_{2}$ and $\lambda(t)$ so as to make the approximate solution accurate. Two difficulties immediately arise with this procedure. One is that $\partial_{t}Q_{\frac{1}{\lambda(t)}}(r) \not \in L^{2}(r dr)$, which means that $E_{\text{WM}}(u_{a},\partial_{t}u_{a})$ is not finite. Another key difficulty is that the elliptic part of the wave maps equation linearized around $Q_{\frac{1}{\lambda(t)}}$ has a zero resonance, $\phi_{0}(\frac{r}{\lambda(t)})$. So, one would like the principal part of the error term of $u_{a}$ to be orthogonal to $\phi_{0}(\frac{r}{\lambda(t)})$. On the other hand, the soliton error term 
$$\partial_{t}^{2}Q_{\frac{1}{\lambda(t)}}(r)=\frac{-2 r \lambda''(t)}{r^{2}+\lambda(t)^{2}} + \frac{4 r \lambda(t) \lambda'(t)^{2}}{(r^{2}+\lambda(t)^{2})^{2}}$$
does not decay fast enough in $r$ for its inner product with $\phi_{0}(\frac{r}{\lambda(t)})$ to even be defined in the first place. \\
\\
Hence, we start by introducing an additional correction, $v_{1}$, which is independent of $v_{2}$, and whose purpose is both to make the ansatz have finite energy and to eliminate an appropriate principal part of the soliton error term, $\partial_{t}^{2}Q_{\frac{1}{\lambda(t)}}$, for large $r$. More precisely the starting point for our ansatz is 
$$u_{a,1}=Q_{\frac{1}{\lambda(t)}}(r)+v_{1}(t,r)+v_{2}(t,r)$$
Here, $\lambda(t)$ is not yet chosen, and $v_{1}$ solves the nondegenerate wave equation
$$ -\partial_{tt}v_{1}+\partial_{rr}v_{1}+\frac{1}{r}\partial_{r} v_{1} - \frac{v_{1}}{r^{2}}=-2\lambda''(t) \frac{r}{1+r^{2}}$$
with $0$ Cauchy data at infinity. As explained in more detail later on, we do not have the entire soliton error term on the right-hand side of the $v_{1}$ equation, because doing so would be more difficult and also would hide the leading, linear part in the modulation equation for $\lambda$. \\
\\
On the other hand, $v_{2}$ solves \eqref{rsquaredpoteqn} with Cauchy data 
$$v_{2}(0,r)=0, \quad \partial_{t}v_{2}(0,r)=v_{2,0}(r)$$
where $v_{2,0}$ will be a prescribed function depending on a fixed parameter $b>0$. The choices for $v_{2,0}$ and $\lambda(t)$ are closely related. Later, we will choose $\lambda(t)$ to solve a modulation equation involving $v_{2,0}$, thereby correlating the length scale of the dynamically rescaled soliton with the radiation profile.  \\
\\
At this stage, one could choose $\lambda(t)$ by requiring that the principal part of the error term associated to $u_{a,1}$, say $e_{a,1}$, is orthogonal to $\phi_{0}(\frac{r}{\lambda(t)})$. This is not exactly what is done in our argument, since we will need to add additional corrections to $u_{a,1}$ before imposing the orthogonality condition on the principal part of the error term of our final ansatz. Nevertheless, computing $\langle e_{a,1}(t,R \lambda(t)),\phi_{0}(R)\rangle_{L^{2}(R dR)}$ allows one to see, in a simpler context than our final equation for $\lambda(t)$, a relation between the leading order behavior of $\lambda''(t)$ and $v_{2}$. For the purposes of this discussion, the principal part of the $u_{a,1}$ error term is 
$$e_{a,1}(t,r)=\partial_{tt}Q_{\frac{1}{\lambda(t)}}(r)+\frac{2 \lambda''(t) r}{1+r^{2}} + \left(\frac{\cos(2Q_{\frac{1}{\lambda(t)}}(r))-1}{r^{2}}\right)v_{1}(t,r)+ \left(\frac{\cos(2Q_{\frac{1}{\lambda(t)}}(r))-1}{r^{2}}\right)v_{2}(t,r)$$
Using the Hankel transform of order 1 to express $v_{2}$, the contribution to $\langle e_{a,1}(t,R \lambda(t)),\phi_{0}(R)\rangle_{L^{2}(R dR)}$ from the $v_{2}$-related term above is given by
$$  \langle \left(\frac{\cos(2Q_{1}(R))-1}{R^{2}\lambda(t)^{2}}\right)v_{2}(t,R\lambda(t)),\phi_{0}(R)\rangle_{L^{2}(R dR)}=-2  \int_{0}^{\infty}  \sin(t\xi) \widehat{v_{2,0}}(\xi) \xi^{2} K_{1}(\xi \lambda(t))d\xi$$
where $K_{1}$ denotes the modified Bessel function of the second kind. We thus get
\begin{equation}\begin{split}\nonumber&\langle e_{a,1}(t,R \lambda(t)),\phi_{0}(R)\rangle_{L^{2}(R dR)} \\
&= -\frac{4}{\lambda(t)} \int_{t}^{\infty} \frac{\lambda''(s)}{1+s-t} ds -2  \int_{0}^{\infty}  \sin(t\xi) \widehat{v_{2,0}}(\xi) \xi^{2} K_{1}(\xi \lambda(t))d\xi +\frac{4 \lambda''(t) \log (\lambda(t))}{\lambda(t)}+f_{1}(\lambda(t),\lambda'(t),\lambda''(t))\end{split}\end{equation}
We will choose $v_{2,0}$ and the principal part of $\lambda$, denoted by $\lambda_{0,0}$, in order to have a leading order cancellation in the above equation. The term $f_{1}(\lambda(t),\lambda'(t),\lambda''(t))$ turns out to be subleading for all $\lambda$ close (in a $C^{2}$ sense) to the choice of $\lambda_{0,0}$ which we make. Our choice of $v_{2,0}$ gives the following equation
\begin{equation}\begin{split}\nonumber \langle e_{a,1}(t,R \lambda(t)),\phi_{0}(R)\rangle_{L^{2}(R dR)} &=-\frac{4}{\lambda(t)} \int_{t}^{\infty} \frac{\lambda''(s)}{1+s-t} ds + \frac{4 b}{\lambda(t)t^{2}\log^{b}(t)} +\frac{4 \lambda''(t) \log (\lambda(t))}{\lambda(t)}+f_{2}(\lambda(t),\lambda'(t),\lambda''(t))\end{split}\end{equation}
Again, the term $f_{2}(\lambda(t),\lambda'(t),\lambda''(t))$ is subleading, for all $\lambda(t)$ close (in a $C^{2}$ sense) to $\lambda_{0,0}(t)=\frac{1}{\log^{b}(t)}$, which is a leading order solution to 
$$-\frac{4}{\lambda(t)} \int_{t}^{\infty} \frac{\lambda''(s)}{1+s-t} ds + \frac{4 b}{\lambda(t)t^{2}\log^{b}(t)} +\frac{4 \lambda''(t) \log (\lambda(t))}{\lambda(t)}=0$$
The term 
$$\frac{4 b}{\lambda(t)t^{2}\log^{b}(t)}$$
occurs in the above computation, as a consequence of our particular choice of Cauchy data for $v_{2}$. Despite the fact that our final equation for $\lambda(t)$ is not exactly the one given above, the leading behavior of our $\lambda(t)$ is indeed $\lambda_{0,0}(t)$, due to the same cancellation seen above.\\
\\
Even if the principal part of the error term of $u_{a,1}$ is chosen to be orthogonal to $\phi_{0}(\frac{\cdot}{\lambda(t)})$, it does not have enough decay to be treated perturbatively via our methods. So, we need to add three more corrections to $u_{a,1}$, denoted by $v_{3},v_{4}$, and $v_{5}$, in order to achieve an acceptable error term.\\
\\
When the error term resulting from $Q_{\frac{1}{\lambda(t)}}$ and $v_{1}$ is computed in the renormalized spatial coordinate $R$, defined by $r=R\lambda(t)$, it has insufficient decay for large $R$ (because factors of $\frac{1}{\lambda(t)}$ will turn out to correspond to logarithmic growth in $t$, once we choose $\lambda(t)$). On the other hand, if this error term is completely eliminated, the resulting modulation equation for $\lambda$ becomes much more difficult to study. The third correction, $v_{3}$, solves an inhomogeneous version of \eqref{rsquaredpoteqn} with $0$ Cauchy data at infinity in order to correct the problem of this error term for large $R$ while also not complicating the final modulation equation for $\lambda$. In the process of doing this, we introduce a small, positive parameter $\alpha$. \\
\\
The soliton, along with these three corrections can be regarded as the principal components of the ansatz. However, we introduce two more corrections, $v_{4}$ and $v_{5}$, both of which solve inhomogeneous versions of \eqref{rsquaredpoteqn}, with 0 Cauchy data at infinity, in order to improve the error terms resulting from the previous terms in the ansatz. More precisely, $v_{4}$ eliminates a large $r$ portion of linear error terms associated to $v_{k}, \quad k=1,2,3$, as well as an error term arising from the combination of the right-hand sides of the $v_{3}$ and $v_{1}$ equations and $\partial_{t}^{2}Q_{\frac{1}{\lambda(t)}}$. $v_{5}$ eliminates error terms associated to the nonlinear interaction between $v_{k}, \quad k=1,2,3,4$.\\
\\
Our final ansatz 
$$u_{ansatz}= Q_{\frac{1}{\lambda(t)}}(r) + \sum_{k=1}^{5} v_{k}(t,r)$$ 
satisfies
$$E_{\text{WM}}(u_{ansatz},\partial_{t}u_{ansatz}) < \infty$$
$$ -\partial_{tt}u_{ansatz}+\partial_{rr}u_{ansatz}+\frac{1}{r}\partial_{r}u_{ansatz}-\frac{\sin(2 u_{ansatz})}{2r^{2}} = -\left(F_{4}+F_{5}+F_{6}\right)$$ 
where $F_{5}+F_{6}$ has sufficiently fast time decay in sufficiently many norms, and is perturbative:
$$\frac{||F_{5}(t,r)+F_{6}(t,r)||_{L^{2}(r dr)}}{\lambda(t)^{2}} \leq \frac{C}{t^{4} \log^{3b+2N-1}(t)}$$
$$\frac{||F_{5}+F_{6}||_{\dot{H}^{1}_{e}}}{\lambda(t)} \leq \frac{C \log^{6+b}(t)}{t^{35/8}}$$
and $F_{4}$, the principal part of the error term, does not have as fast decay as $F_{5}+F_{6}$, but satisfies
$$\langle F_{4}(t,\cdot),\phi_{0}(\frac{\cdot}{\lambda(t)})\rangle=0$$
and the following estimate (in a symbol-type fashion, see theorem \ref{approxsolnthm} for the precise statement):
$$ |F_{4}(t,r)| \leq \frac{C \mathbbm{1}_{\{r \leq \log^{N}(t)\}} r}{t^{2} \log^{3b+1-2\alpha b}(t) (r^{2}+\lambda(t)^{2})^{2}}+\frac{C \mathbbm{1}_{\{r \leq \frac{t}{2}\}} r}{t^{2} \log^{5b+2N-2}(t) (r^{2}+\lambda(t)^{2})^{2}}$$
where $\alpha$ and $N$ are parameters associated to $v_{3}$ and $v_{4}$, respectively. $\alpha$ is small relative to $b$ and $1$, while $N$ and is large relative to $b$ and 1.\\
\\
The modulation equation for $\lambda$, 
$$\langle F_{4}(t,\cdot),\phi_{0}(\frac{\cdot}{\lambda(t)})\rangle=0$$
has a principal part which is of the form of a Volterra equation of the second kind in the variable $\lambda''$ (with kernel and coefficients weakly depending on $\lambda$). In particular, this equation is of the form
$$-4 \int_{t}^{\infty} \frac{\lambda''(s)}{1+s-t} ds + \frac{4 b}{t^{2}\log^{b}(t)} + 4 \alpha \log(\lambda(t)) \lambda''(t) - 4 \int_{t}^{\infty} \frac{\lambda''(s)}{(\lambda(t)^{1-\alpha}+s-t)(1+s-t)^{3}} ds=f_{3}(\lambda(t),\lambda'(t),\lambda''(t))$$
We emphasize that the leading behavior of $\lambda$ is independent of the small parameter $\alpha>0$, associated to $v_{3}$, and the crucial source term 
$$\frac{4b}{t^{2} \log^{b}(t)}$$
comes from $v_{2}$, and is a consequence of the particular choice of data for $v_{2}$.  The other terms on the left-hand side of the above equation come from $v_{1},v_{3},Q_{\frac{1}{\lambda(t)}}$. In particular, $v_{4}$ and $v_{5}$ do not contribute to the leading order part of the equation for $\lambda(t)$ because the terms contained in $f_{3}$ are subleading, for all $\lambda$ close (in a $C^{2}$ sense) to $\frac{1}{\log^{b}(t)}$. \\
\\
To the knowledge of the author, a modulation equation of the above form is quite different from that of previous works. Moreover, in the context of our ansatz, the Volterra form of the modulation equation seems to be related to the fact that $\frac{d}{d\lambda}\Bigr|_{\lambda=1}Q_{\lambda}(r) \not \in L^{2}(r dr)$. Indeed, the integral operators acting on $\lambda''$ arising in the principal part of the modulation equation come from $v_{1}$ and $v_{3}$, which were introduced to correct soliton related error terms for large $r$.\\
\\
As motivated by the discussion of the simpler ansatz $u_{a,1}$ above, we find an approximate solution to the modulation equation \eqref{modulationfinal} for $\lambda$ of the form
$$\lambda_{0,0}(t) = \frac{1}{\log^{b}(t)}$$
The equation for $\lambda$ is then exactly solved around $\lambda_{0,0}$ with a fixed point argument in a weighted $C^{2}$ space, and we obtain an exact solution
$$\lambda(t) = \frac{1}{\log^{b}(t)} + O\left(\frac{1}{\log^{b}(t) \sqrt{\log(\log(t))}}\right)$$ 
Afterwards, we show that $\lambda$ is in fact a $C^{4}$ function, and prove quantitative estimates on its derivatives. Along the way, we thus obtain estimates on higher time derivatives of the corrections $v_{k}$, whose right-hand sides depend on $\lambda(t)$ (for all $k \neq 2$).\\
\\
\textbf{Part 2: Construction of the exact solution}\\
\\
Once we construct $u_{ansatz}$, we substitute $u=u_{ansatz}+v_{6}$ into \eqref{wavemaps}, thereby obtaining an equation for $v_{6}$. Our goal is to solve this equation for $v_{6}$  perturbatively, with $0$ Cauchy data at infinity. We achieve this, by studying the distorted Fourier transform (in the sense of \cite{kst}), of $v_{6}$. In particular, we (formally) derive the equation for
$$y(t,\omega) = \mathcal{F}(\sqrt{\cdot} v_{6}(t,\cdot \lambda(t)))(\omega \lambda(t)^{2})$$
where we recall that $\mathcal{F}$ denotes the distorted Fourier transform of \cite{kst}. The particular choice of the rescaling used in the definition of $y$ is explained by noting that the equation for $y$ takes the form
$$\partial_{tt}y+\omega y = -\mathcal{F}(\sqrt{\cdot}F(t,\cdot \lambda(t)))(\omega \lambda(t)^{2})+F_{2}(y)(t,\omega) -\mathcal{F}(\sqrt{\cdot}F_{3}(v_{6}(y))(t,\cdot \lambda(t)))(\omega \lambda(t)^{2})$$
for appropriate $F_{2}$, $F_{3}$, where 
$$F=F_{4}+F_{5}+F_{6}$$
is the sum of the error terms of $u_{ansatz}$. In deriving the $y$ equation, a few properties about the distorted Fourier transform, most importantly, the transferrence identity from \cite{kst} are used. The equation for $y$ is solved by showing that the map
$$T(y)(t,\omega) = -\int_{t}^{\infty} \frac{\sin((t-x)\sqrt{\omega})}{\sqrt{\omega}}\left(F_{2}(y)(x,\omega) - \mathcal{F}(\sqrt{\cdot}F(x,\cdot\lambda(x)))(\omega \lambda(x)^{2}) - \mathcal{F}(\sqrt{\cdot}F_{3}(u(y))(x,\cdot\lambda(x)))(\omega \lambda(x)^{2})\right)dx$$
has a fixed point in an appropriate Banach space (whose norm is roughly the sum of  weighted $L^{\infty}_{t}L^{2}_{\omega}$ norms of $y$ and $\partial_{t}y$, see \eqref{znorm} for the precise definition) via the contraction mapping theorem. The most delicate term to estimate is
$$-\int_{t}^{\infty} \frac{\sin((t-x)\sqrt{\omega})}{\sqrt{\omega}}\left(- \mathcal{F}(\sqrt{\cdot}F_{4}(x,\cdot\lambda(x)))(\omega \lambda(x)^{2})\right)dx$$
In order to obtain sufficient estimates on this term (and its time derivative) in appropriate norms, we must utilize the previously discussed orthogonality condition on $F_{4}$. The orthogonality condition on $F_{4}$ is uitilized by noting that it implies that $\mathcal{F}(\sqrt{\cdot}F_{4}(x,\cdot\lambda(x)))(\omega \lambda(x)^{2})$ has a certain degree of vanishing at small frequencies, and this allows us to integrate by parts in the $x$ variable in the integral above. This, combined with the symbol-type nature of the pointwise estimates on $F_{4}$ turn out provide sufficient decay of the integral above in all norms required by the fixed point argument. Since the density of the spectral measure associated to the distorted Fourier transform of \cite{kst} (which appears in the weighted norms of our iteration space) has a singularity at low frequencies, such integration by parts would be impossible without the orthogonality condition. In particular, we can not integrate by parts for the analogous integrals involving $F_{5}+F_{6}$. The faster time decay of the $L^{2}$ and $\dot{H}^{1}_{e}$ norms of this term, relative to $F_{4}$ is what allows it to be perturbative, despite the non-orthogonality to $\phi_{0}(\frac{\cdot}{\lambda(t)})$.
 
\section{Construction of the ansatz}
Fix $b>0$, $0 < \alpha < \text{min}\{\frac{1}{b (1040!)},\frac{1}{1040!}\}$, and $N > (5000!) (b+1)$. We consider \eqref{wavemaps} for $t \geq T_{0}$, where
\begin{equation}\label{T0initialconstraint} T_{0}> 2e^{e^{1000(b+1)}} +T_{0,1}\end{equation}
and is otherwise arbitrary for now, and where
$T_{0,1} > e^{(128000)^{\frac{4}{b}}}+e^{N^{2}}$ is  such that
$$|\frac{d^{j}}{dt^{j}}\left(\frac{\log(\log(t))}{\log^{b+1}(t)}\right)|+|\frac{d^{j}}{dt^{j}}\left(\frac{1}{b \log^{b}(t) \sqrt{\log(\log(t))}}\right)| \leq \frac{1}{200} |\frac{d^{j}}{dt^{j}}\left(\frac{1}{\log^{b}(t)}\right)|, \quad t \geq T_{0,1}, \quad  j=0,1,2$$
$\lambda:[T_{0},\infty) \rightarrow (0,\infty)$ denotes a $C^{2}([T_{0},\infty))$ map satisfying, for some $C>0$, \emph{independent} of $T_{0}$
\begin{equation}\label{lambdarestr}\lambda'(t) <0, \quad |\lambda''(t)| \leq \frac{C}{t^{2} \log^{b+1}(t)}, \quad \text{ and } \frac{1}{C \log^{C}(t)}<\lambda(t) < \frac{1}{2}, \quad t \geq T_{0}\end{equation}
and is otherwise arbitrary for now. (Note that the first and third requirements above are not strictly necessary for the validity of most of our procedure. However, for this work, there is no loss of generality in assuming them, since $\lambda$ will later on be restricted to a class of functions of the form $\lambda=\lambda_{0}+e$, where $\lambda_{0}$ is some explicit function to be specified later, and $e$ belongs to a certain space of functions such that (among other things) \eqref{lambdarestr} holds for $\lambda_{0}+e$). \\
\\
Also for all estimates appearing in the entire paper, we use the convention that $C$ will always denote a positive constant \emph{independent} of $T_{0}$, whose value may change from line to line. Although we have already summarized the ansatz construction in the overview of the proof, we now provide an outline which clarifies the logical structure of the process.
\subsection{Outline}
\textbf{Step 1, Definitions of the corrections, $v_{k}$:} For all $T_{0}$ and $\lambda$ as above, and $t \geq T_{0}$, we define functions $v_{k}, \quad 1\leq k \leq 5$, which were roughly described in the overview of the proof, thereby obtaining
$$u_{ansatz}(t,r) = Q_{\frac{1}{\lambda(t)}}(r)+\sum_{k=1}^{5} v_{k}(t,r)$$
\textbf{Step 2, Splitting of the error term of $u_{ansatz}$, and setup of the modulation equation for $\lambda$:} We define functions $F_{4},F_{5}$, and $F_{6}$, which split the error term of $u_{ansatz}$ into $-\left(F_{4}+F_{5}+F_{6}\right)$. Then, we consider the modulation equation for $\lambda(t)$ resulting from setting
$$\langle F_{4}(t,\cdot),\phi_{0}(\frac{\cdot}{\lambda(t)})\rangle =0, \quad t \geq T_{0}$$
(Note that $F_{4}$ depends on $\lambda(t)$). As described (to a certain extent) in the overview of the proof, this modulation equation takes the form of a Volterra equation of the second kind in the variable $\lambda''(t)$ (with kernel and coefficients weakly depending on $\lambda$), and our choice of Cauchy data for $v_{2}$ gives rise to a leading order solution 
$$\lambda_{0}=\frac{1}{\log^{b}(t)}+ \int_{t}^{\infty} \int_{t_{1}}^{\infty} \frac{-b^{2} \log(\log(t_{2}))}{t_{2}^{2}\log^{b+2}(t_{2})}dt_{2}dt_{1}$$ 
to this equation. We then introduce a weighted $C^{2}$ space, $X$, with norm \eqref{ynorm}. From this point on, we further restrict $\lambda(t)$ to be of the form
$$\lambda(t) = \lambda_{0}(t)+e(t), \quad e \in B$$
where
$$B=\overline{B_{1}(0)} \subset X$$
(Note that $\lambda_{0}+B$ is contained within the class of $\lambda(t)$ we initially started with (which is described by \eqref{lambdarestr})).\\
\\
\textbf{Step 3, Solving the modulation equation for $\lambda(t)$:} Writing $F_{4}(t,r) = F_{4}^{\lambda(t)}(t,r)$ to emphasize the $\lambda$ dependence of $F_{4}$, we show that there exists $T_{3}>0$ such that, for all $T_{0} \geq T_{3}$, the equation
$$\langle F_{4}^{\lambda_{0}(t)+e(t)}(t,\cdot),\phi_{0}(\frac{\cdot}{\lambda_{0}(t)+e(t)})\rangle =0, \quad t \geq T_{0}$$
can be solved for $e(t) \in B$, using the contraction mapping principle. An important part of this procedure is that the kernel appearing in the Volterra equation for $e$ (which is independent of $e$, modulo error terms which can be treated perturbatively) satisfies the structural condition \eqref{kineq}. From here on, we work under the constraint $T_{0} \geq T_{3}$.\\
\\
\textbf{Step 4, Estimates on higher derivatives of $\lambda''$:}
Denoting the solution to the above equation for $e$ by $e_{0}(t)$, we fix $\lambda(t) = \lambda_{0}(t)+e_{0}(t)$. Then, we prove that $\lambda(t)$, which is apriori only in $C^{2}([T_{0},\infty))$, is actually in $C^{4}([T_{0},\infty))$, and establish quantitative estimates on $\lambda'''$ and $\lambda''''$.\\
\\
\textbf{Step 5, Estimates on $F_{k}$:} At this stage, we prove pointwise estimates on $F_{4}$, as well as estimates on $||F_{k}||_{L^{2}(r dr)}$, $||F_{k}||_{\dot{H}^{1}_{e}}$ for $k=5,6$. This completes the ansatz construction.
More precisely, our main result of this section is
\begin{theorem}[Approximate solution to \eqref{wavemaps}]\label{approxsolnthm}
For all $b>0$, there exists $T_{3}>0$ such that, for all $T_{0} \geq T_{3}$, there exists $v_{corr} \in C^{4}([T_{0},\infty); C^{2}((0,\infty)))$, and $\lambda \in C^{4}([T_{0},\infty))$ such that, if $u=Q_{\frac{1}{\lambda(t)}}+v_{corr}$, then
$$E_{\text{WM}}(u,\partial_{t}u) < \infty$$
\begin{equation} -\partial_{tt}u+\partial_{rr}u+\frac{1}{r}\partial_{r}u-\frac{\sin(2 u)}{2r^{2}} = -\left(F_{4}+F_{5}+F_{6}\right)\end{equation}
where
\begin{equation}\label{f5f6l2thm} \frac{1}{\lambda(t)^{2}} ||\left(F_{5}+F_{6}\right)(t,r)||_{L^{2}(r dr)} \leq \frac{C}{t^{4} \log^{3b+2N-1}(t)}\end{equation}
\begin{equation}\label{f5f6h1thm}\frac{||\left(F_{5}+F_{6}\right)(t)||_{\dot{H}^{1}_{e}}}{\lambda(t)} \leq \frac{C \log^{6+b}(t)}{t^{35/8}}\end{equation}

$$\langle F_{4}(t,\cdot),\phi_{0}(\frac{\cdot}{\lambda(t)}) \rangle =0$$
For $0 \leq k \leq 2, \quad 0 \leq j \leq 1, \quad j+k \leq 2$, we have
\begin{equation}  t^{j} r^{k} |\partial_{r}^{k}\partial_{t}^{j} F_{4}(t,r)| \leq \frac{C \mathbbm{1}_{\{r \leq \log^{N}(t)\}} r}{t^{2} \log^{3b+1-2\alpha b}(t) (r^{2}+\lambda(t)^{2})^{2}}+\frac{C \mathbbm{1}_{\{r \leq \frac{t}{2}\}} r}{t^{2} \log^{5b+2N-2}(t) (r^{2}+\lambda(t)^{2})^{2}}\end{equation}
In addition, we have
\begin{equation}  |\partial_{t}^{2} F_{4}(t,r)| \leq \frac{C \mathbbm{1}_{\{r \leq \log^{N}(t)\}} r}{t^{4} \log^{3b+1-2\alpha b}(t) (r^{2}+\lambda(t)^{2})^{2}}+\frac{C \mathbbm{1}_{\{r \leq \frac{t}{2}\}} r}{t^{4} \log^{5b+2N-2}(t) (r^{2}+\lambda(t)^{2})^{2}} + \frac{C \mathbbm{1}_{\{r \leq \frac{t}{2}\}}}{t^{4} \log^{5b+N-2}(t) (r^{2}+\lambda(t)^{2})^{2}}\end{equation}
We also have the following estimates on $v_{corr}$.
\begin{equation}\label{vcorrcofthm} ||\frac{v_{corr}(x,R\lambda(x))}{R \lambda(x)}||_{L^{\infty}}^{2}+||\frac{v_{corr}(x,R\lambda(x))}{R\lambda(x)^{2}(1+R^{2})}||_{L^{\infty}} \leq \frac{C \log(\log(x))}{x^{2}\log(x)}\end{equation}

\begin{equation}\label{1cofthm} 1+||\frac{v_{corr}(x,R\lambda(x))}{R}||_{L^{\infty}}+||\partial_{R}(v_{corr}(x,R\lambda(x)))||_{L^{\infty}} \leq C\end{equation}

\begin{equation}\label{vcorrdrvcorrcofthm}\begin{split} &||\frac{v_{corr}(x,R\lambda(x)) \partial_{R}(v_{corr}(x,R\lambda(x)))}{R\lambda(x)^{2}}||_{L^{\infty}_{R}((0,1))}+||\frac{v_{corr}(x,R\lambda(x)) \partial_{R}(v_{corr}(x,R\lambda(x)))}{R^{2}\lambda(x)^{2}}||_{L^{\infty}_{R}((1,\infty))}\\
&+||\frac{\partial_{R}(v_{corr}(x,R\lambda(x)))}{(1+R^{2})\lambda(x)^{2}}||_{L^{\infty}} \\
&\leq \frac{C \log(\log(x))}{x^{2}\log(x)}\end{split}\end{equation}
Finally, $\lambda$ is described by
\begin{equation} \lambda(t) = \frac{1}{\log^{b}(t)} + e(t), \quad |e(t)| \leq \frac{C}{\log^{b}(t)\sqrt{\log(\log(t))}},\quad |e^{(j)}(t)| \leq \begin{cases}\frac{C}{t^{j} \log^{b+1}(t) \sqrt{\log(\log(t))}}, \quad j=1,2\\
\frac{C}{t^{j}\log^{b+1}(t)}, \quad j=3,4\end{cases}\end{equation}
\end{theorem}
  
\subsection{Correcting the large $r$ behavior of $Q_{\frac{1}{\lambda(t)}}$}
The error term $\partial_{t}^{2}Q_{\frac{1}{\lambda(t)}}$, which arises from inserting $Q_{\frac{1}{\lambda(t)}}$ into \eqref{wavemaps}, does not decay quickly enough for its inner product with $\phi_{0}(\frac{\cdot}{\lambda(t)})$ to be defined. The first term in the ansatz, $v_{1}$, is designed to correct this problem. Note that we do not choose $v_{1}$ to solve an equation whose right-hand side is equal to $\partial_{t}^{2}Q_{\frac{1}{\lambda(t)}}$, as doing so would result in a much more difficult equation to solve when eventually choose $\lambda(t)$. Instead, $v_{1}$ is  defined as the solution to the equation 
\begin{equation} \label{v1eqn} -\partial_{tt}v_{1}+\partial_{rr}v_{1}+\frac{1}{r}\partial_{r} v_{1} - \frac{v_{1}}{r^{2}}=-2\lambda''(t) \frac{r}{1+r^{2}}\end{equation}
with $0$ Cauchy data at infinity. I.e., by Duhamel's principle, we have
\begin{equation}\label{v1eqn} v_{1}(t,r) = \int_{t}^{\infty} v_{s}(t,r) ds\end{equation}
where $v_{s}$ is the solution to the following Cauchy problem
\begin{equation}\label{vseqn}\begin{cases} -\partial_{tt} v_{s} + \partial_{rr}v_{s} + \frac{1}{r} \partial_{r}v_{s}-\frac{v_{s}}{r^{2}} = 0\\
v_{s}(s,r)=0\\
\partial_{t}v_{s}(s,r) = -2 \lambda''(s) \frac{r}{1+r^{2}}
\end{cases}
\end{equation}
We can determine a fairly explicit formula for $v_{s}$. The procedure used to determine the formula for $v_{s}$ may be slightly formal, but the final expression obtained can be seen to be the solution to the Cauchy problem \eqref{vseqn}.
Firstly, note that if $u$ is a solution to 
\begin{equation} -\partial_{tt}u + \partial_{rr} u + \frac{1}{r} \partial_{r} u=0\end{equation}
then, $w=\partial_{r}u$ is formally a solution to
\begin{equation} -\partial_{tt}w+\partial_{rr}w+\frac{1}{r}\partial_{r}w-\frac{w}{r^{2}}=0\end{equation}
So, we will first write down the spherical means representation formula for the problem
\begin{equation}\label{ueqn}\begin{cases} -\partial_{tt}u_{1}+\Delta u_{1}=0\\
u_{1}(s)=0\\
\partial_{t}u_{1}(s)=-\lambda''(s)\log(1+|x|^{2})
\end{cases}
\end{equation}
then define $u$ by $u(t,|x|) = u_{1}(t,x)$ (which is possible, since $u_{1}$ is radially symmetric). Then, 
$$v_{s}(t,r)=\partial_{r}u(t,r)$$ 
will be seen to be the solution to \eqref{vseqn}. From the spherical means representation formula, for $t>s$, we have
\begin{equation} \begin{split} u_{1}(t,x) &= -\frac{\lambda''(s)}{2\pi} \int_{B(0,t-s)} \frac{ \log(1+|y+x|^{2})}{((t-s)^{2}-|y|^{2})^{1/2}} dy\\
&=-\frac{\lambda''(s)}{2\pi}\int_{0}^{t-s} \rho d\rho \int_{0}^{2 \pi} d\theta \frac{ \log(1+|x|^{2}+2|x|\rho \cos(\theta)+\rho^{2})}{((t-s)^{2}-\rho^{2})^{1/2}}
\end{split}
\end{equation} 
Recalling that $u$ is the radial coordinate representative of $u_{1}$, we have
\begin{equation}\label{dxu} \begin{split} \partial_{|x|} u &=-\frac{\lambda''(s)}{2 \pi}\int_{0}^{t-s} \rho d\rho \int_{0}^{2 \pi} d\theta \frac{2(|x|+\rho \cos(\theta))}{(1+|x|^{2}+2|x|\rho \cos(\theta)+\rho^{2})((t-s)^{2}-\rho^{2})^{1/2}}
\end{split}
\end{equation}
To do the integral over the angular coordinate, we can regard it as 
\begin{equation} \int_{C} \frac{dz}{i z} \frac{2(|x|+\frac{\rho}{2} (z+\frac{1}{z}))}{(1+|x|^{2}+|x|\rho(z+\frac{1}{z})+\rho^{2})((t-s)^{2}-\rho^{2})^{1/2}}\end{equation}
where $C$ is the boundary of the unit circle in the complex plane, traversed in the counterclockwise direction. \\
\\
The integrand is a meromorphic function on $\mathbb{C}$, with poles at $$z=0, -\left(\frac{1+|x|^{2}+\rho^{2}\pm \sqrt{-4|x|^{2}\rho^{2}+(1+|x|^{2}+\rho^{2})^{2}}}{2|x|\rho}\right)$$

Note that $$-4 |x|^{2}\rho^{2}+(1+|x|^{2}+\rho^{2})^{2} = 1+2(|x|^{2}+\rho^{2})+(|x|^{2}-\rho^{2})^{2} \geq 1$$
So, $$\frac{1+|x|^{2}+\rho^{2} + \sqrt{-4|x|^{2}\rho^{2}+(1+|x|^{2}+\rho^{2})^{2}}}{2|x|\rho} > \frac{1+|x|^{2}+\rho^{2}}{2|x|\rho} >\frac{|x|^{2}+\rho^{2}}{2|x|\rho} \geq 1$$
On the other hand, we have
$$\sqrt{-4|x|^{2}\rho^{2}+(1+|x|^{2}+\rho^{2})^{2}} \leq (1+|x|^{2}+\rho^{2})$$
so that
$$\frac{1+|x|^{2}+\rho^{2}-\sqrt{-4|x|^{2}\rho^{2}+(1+|x|^{2}+\rho^{2})^{2}}}{2|x|\rho} >0$$
and
\begin{equation} \frac{1+|x|^{2}+\rho^{2}-\sqrt{-4|x|^{2}\rho^{2}+(1+|x|^{2}+\rho^{2})^{2}}}{2|x|\rho} = \frac{2 |x| \rho}{1+|x|^{2}+\rho^{2}+\sqrt{-4 |x|^{2}\rho^{2}+(1+|x|^{2}+\rho^{2})^{2}}} \leq 1\end{equation}

So, the only two poles of the integrand located inside the unit disk in the complex plane are $z=0$ and $z=z_{1}=-\left(\frac{1+|x|^{2}+\rho^{2}-\sqrt{-4|x|^{2}\rho^{2}+(1+|x|^{2}+\rho^{2})^{2}}}{2|x|\rho}\right)$. The corresponding residues are
$$Res_{0} = \frac{1}{i |x| \sqrt{(s-t)^{2}-\rho^{2}}}$$
$$Res_{z_{1}} = \frac{-1+|x|^{2}-\rho^{2}}{i |x| \sqrt{-4|x|^{2}\rho^{2}+(1+|x|^{2}+\rho^{2})^{2}}\sqrt{(s-t)^{2}-\rho^{2}}}$$

Returning to \eqref{dxu}, we get, for $t>s$, \begin{equation}\begin{split}\label{finaldxu} (\partial_{|x|} u)(t,r) &= -\frac{\lambda''(s)}{|x|} \int_{0}^{t-s} \frac{\rho d\rho}{\sqrt{(t-s)^{2}-\rho^{2}}} \left(1+\frac{|x|^{2}-1-\rho^{2}}{\sqrt{(1+|x|^{2}+\rho^{2})^{2}-4 |x|^{2}\rho^{2}}}\right)
\end{split}
\end{equation}
By substitution, we see that $\partial_{|x|} u(t,r)$ solves \eqref{vseqn} for $t>s$. We can extend the solution to $t<s$ with the same Cauchy data at $t=s$ by defining $$v_{s}(t,r) = (-\partial_{|x|} u)(s-(t-s),r), \quad t<s$$ so that, for $t<s$, we have
\begin{equation}\label{finalvs} v_{s}(t,r) = \frac{\lambda''(s)}{r} \int_{0}^{s-t} \frac{\rho d\rho}{\sqrt{(s-t)^{2}-\rho^{2}}} \left(1+\frac{r^{2}-1-\rho^{2}}{\sqrt{(1+r^{2}+\rho^{2})^{2}-4 r^{2}\rho^{2}}}\right)\end{equation}
We will only need to use pointwise estimates on $v_{1}$, which will be proven shortly, but to give the reader some idea of the behavior of $v_{s}$, we note that
\begin{equation} v_{s}(t,r) \sim \frac{2 \lambda''(s)\left(1-\sqrt{1-a^{2}}\right)}{a}, \quad 0<a=\frac{r}{(s-t)}<1\ll s-t\end{equation}
We have
\begin{equation}\label{v1formula}v_{1}(t,r) = \int_{t}^{\infty} ds\frac{\lambda''(s)}{r} \int_{0}^{s-t} \frac{\rho d\rho}{\sqrt{(s-t)^{2}-\rho^{2}}} \left(1+\frac{r^{2}-1-\rho^{2}}{\sqrt{(1+r^{2}+\rho^{2})^{2}-4 r^{2}\rho^{2}}}\right) \end{equation}

\subsubsection{Pointwise estimates on $\partial^{j}_{r}v_{1}$}
In this section, we will prove the following
\begin{lemma} 
\begin{equation}\label{v1smallrest} v_{1}(t,r) = r \int_{t}^{\infty} \frac{\lambda''(s)}{1+s-t} ds +\emph{Err}(t,r)\end{equation}
where
\begin{equation} |\emph{Err}(t,r)| \leq C r \log(3+2r) \sup_{x \geq t} |\lambda''(x)|, \quad r >0\end{equation}\\
\\
In addition,
\begin{equation} \label{v1largerest} |v_{1}(t,r)| \leq \frac{C}{r} \int_{t}^{\infty}  |\lambda''(s)| (s-t) ds, \quad r>0\end{equation}
Moreover, similar results are true for $\partial_{r}v_{1}$:
\begin{equation}\label{drv1smallrest} \partial_{r}v_{1}(t,r) = \int_{t}^{\infty}  \frac{\lambda''(s)}{1+s-t}ds + E_{\partial_{r}v_{1}}(t,r)\end{equation}
with
$$|E_{\partial_{r}v_{1}}(t,r)| \leq C \sup_{x \geq t}\left(|\lambda''(x)|\right) \log(3+2r), \quad r>0$$
and
\begin{equation}\label{largerdrv1} |\partial_{r}v_{1}(t,r)| \leq \frac{C}{r^{2}}\left(\sup_{x \geq t} \left(|\lambda''(x)| (1+(x-t)^{2})\right)+\int_{t}^{\infty} ds |\lambda''(s)| (s-t)\right), \quad r>2\end{equation}

\end{lemma}
\begin{proof}
We start with
\begin{equation}\label{v1eqn2}\begin{split} v_{1}(t,r) &= \int_{t}^{\infty}ds \frac{\lambda''(s)}{r} \int_{0}^{s-t}d\rho \rho \left(\frac{1}{\sqrt{(s-t)^{2}-\rho^{2}}}-\frac{1}{s-t}\right)\left(1+\frac{r^{2}-1-\rho^{2}}{\sqrt{(r^{2}-1-\rho^{2})^{2}+4r^{2}}}\right) \\
&+\int_{t}^{\infty}ds \frac{\lambda''(s)}{r} \int_{0}^{s-t}d\rho \frac{\rho}{(s-t)} \left(1+\frac{r^{2}-1-\rho^{2}}{\sqrt{(r^{2}-1-\rho^{2})^{2}+4r^{2}}}\right) 
\end{split}
\end{equation}
The first line of \eqref{v1eqn2} can be estimated as follows
\begin{equation} \begin{split} &|\int_{t}^{\infty}ds \frac{\lambda''(s)}{r} \int_{0}^{s-t}d\rho \rho \left(\frac{1}{\sqrt{(s-t)^{2}-\rho^{2}}}-\frac{1}{s-t}\right)\left(1+\frac{r^{2}-1-\rho^{2}}{\sqrt{(r^{2}-1-\rho^{2})^{2}+4r^{2}}}\right)|\\
&\leq \frac{\sup_{x \geq t}\left(|\lambda''(x)|\right)}{r} \int_{0}^{\infty} \rho d\rho \int_{\rho+t}^{\infty} ds \left(\frac{1}{\sqrt{(s-t)^{2}-\rho^{2}}}-\frac{1}{s-t}\right)\left(1+\frac{r^{2}-1-\rho^{2}}{\sqrt{(r^{2}-1-\rho^{2})^{2}+4r^{2}}}\right)\\
&\leq \frac{\sup_{x \geq t}\left(|\lambda''(x)|\right)}{r}\int_{0}^{\infty} \rho d\rho \log(2) \left(1+\frac{r^{2}-1-\rho^{2}}{\sqrt{(r^{2}-1-\rho^{2})^{2}+4r^{2}}}\right)\\
&\leq C \frac{\sup_{x \geq t}\left(|\lambda''(x)|\right)}{r} \cdot r^{2}\\
&\leq C r \sup_{x \geq t}\left(|\lambda''(x)|\right)\end{split}
\end{equation}
The second line of \eqref{v1eqn2} is split into the following two terms
\begin{equation}\label{v1eqn2part1} \begin{split} &\int_{t}^{\infty}ds \frac{\lambda''(s)}{r} \int_{0}^{s-t}d\rho \frac{\rho}{(s-t)} \left(1+\frac{r^{2}-1-\rho^{2}}{\sqrt{(r^{2}-1-\rho^{2})^{2}+4r^{2}}}\right)\\
&=\int_{t}^{t+2(r+1)}ds \frac{\lambda''(s)}{r} \int_{0}^{s-t}d\rho \frac{\rho}{(s-t)} \left(1+\frac{r^{2}-1-\rho^{2}}{\sqrt{(r^{2}-1-\rho^{2})^{2}+4r^{2}}}\right)\\
&+\int_{t+2(r+1)}^{\infty}ds \frac{\lambda''(s)}{r} \int_{0}^{s-t}d\rho \frac{\rho}{(s-t)} \left(1+\frac{r^{2}-1-\rho^{2}}{\sqrt{(r^{2}-1-\rho^{2})^{2}+4r^{2}}}\right)
\end{split}
\end{equation}
For the second line of \eqref{v1eqn2part1}, we have
\begin{equation} \begin{split} &\int_{t}^{t+2(r+1)}ds \frac{\lambda''(s)}{r} \int_{0}^{s-t}d\rho \frac{\rho}{(s-t)} \left(1+\frac{r^{2}-1-\rho^{2}}{\sqrt{(r^{2}-1-\rho^{2})^{2}+4r^{2}}}\right)\\
&=\int_{t}^{t+2(r+1)} ds \frac{\lambda''(s)}{r(s-t)} \frac{\left(1+r^{2}+(s-t)^{2}-\sqrt{(1+(r+s-t)^{2})(1+(r-(s-t))^{2})}\right)}{2}\\
&=\int_{t}^{t+2(r+1)} ds \frac{\lambda''(s)}{r(s-t)} \frac{2 r^{2}(s-t)^{2}}{(1+r^{2}+(s-t)^{2}+\sqrt{(1+(r+s-t)^{2})(1+(r-(s-t))^{2})})} 
\end{split}
\end{equation}
So, we have
\begin{equation}\begin{split}&|\int_{t}^{t+2(r+1)}ds \frac{\lambda''(s)}{r} \int_{0}^{s-t}d\rho \frac{\rho}{(s-t)} \left(1+\frac{r^{2}-1-\rho^{2}}{\sqrt{(r^{2}-1-\rho^{2})^{2}+4r^{2}}}\right)|\\
&\leq C r \sup_{x \geq t}\left(|\lambda''(x)|\right) \int_{t}^{t+2(r+1)} \frac{(s-t)}{1+r^{2}} ds\\
&\leq C r \sup_{x \geq t}\left(|\lambda''(x)|\right)\end{split}\end{equation}
For the third line of \eqref{v1eqn2part1}, we have
\begin{equation} \label{partialv1est} \begin{split} &\int_{t+2(r+1)}^{\infty}  \lambda''(s) \frac{2 r (s-t)}{(1+r^{2}+(s-t)^{2}+\sqrt{(1+(r+s-t)^{2})(1+(r-(s-t))^{2})}}ds\\
&=\int_{t+2(r+1)}^{\infty} \lambda''(s) \frac{2r}{(s-t)} \left(\frac{1}{2}+O\left(\frac{(1+r)^{2}}{(s-t)^{2}}\right)\right) ds\\
&=\int_{t+2(r+1)}^{\infty}  \lambda''(s) \frac{r}{(s-t)}ds +E_{1}(t,r)
\end{split}\end{equation}
where
\begin{equation}\begin{split} |E_{1}(t,r)| &\leq \sup_{x \geq t}\left(|\lambda''(x)|\right) r (1+r)^{2} \int_{t+2(r+1)}^{\infty} \frac{ds}{(s-t)^{3}}\\
&\leq C r \sup_{x \geq t}\left(|\lambda''(x)|\right)\end{split}\end{equation}
Next, 
\begin{equation} \begin{split} \int_{t+2(r+1)}^{\infty} \frac{\lambda''(s) r}{(s-t)} ds &=\int_{t+2(r+1)}^{\infty} \frac{\lambda''(s) r}{1+s-t} ds + \int_{t+2(r+1)}^{\infty} \lambda''(s) r \left(\frac{1}{s-t}-\frac{1}{1+s-t}\right) ds\end{split}\end{equation}
and
\begin{equation}\begin{split} |\int_{t+2(r+1)}^{\infty} \lambda''(s) r\left(\frac{1}{s-t}-\frac{1}{1+s-t}\right) ds| &\leq r \sup_{x \geq t} \left(|\lambda''(x)|\right) \log(1+\frac{1}{2+2r})\\
&\leq C \sup_{x \geq t}\left(|\lambda''(x)|\right) r\end{split}\end{equation}
Finally, 
\begin{equation} \int_{t+2(r+1)}^{\infty} \frac{\lambda''(s) r}{1+s-t} ds = \int_{t}^{\infty} \frac{\lambda''(s) r}{1+s-t} ds - \int_{t}^{t+2(r+1)} \frac{\lambda''(s) r}{1+s-t} ds\end{equation}
with
\begin{equation}|\int_{t}^{t+2(r+1)} \frac{\lambda''(s) r}{1+s-t} ds| \leq C \sup_{x \geq t}\left(|\lambda''(x)|\right) r \log(3+2r)\end{equation}
Thus, we obtain the decomposition \eqref{v1smallrest}, with the desired estimate on $\text{Err}$.\\
\\
Next, we have
\begin{equation} \begin{split} |v_{1}(t,r)| &= |\int_{t}^{\infty} ds \frac{\lambda''(s)}{r} \int_{0}^{s-t} \frac{\rho d\rho}{\sqrt{(s-t)^{2}-\rho^{2}}} \left(1+\frac{r^{2}-1-\rho^{2}}{\sqrt{(r^{2}-1-\rho^{2})^{2}+4r^{2}}}\right)|\\
&\leq \frac{C}{r} \int_{t}^{\infty}  |\lambda''(s)| (s-t) ds
\end{split}
\end{equation}
Now we treat $\partial_{r}v_{1}$: 
\begin{equation}\label{drv1est} \begin{split} \partial_{r}v_{1}(t,r) &= \int_{t}^{\infty} ds \frac{-\lambda''(s)}{r^{2}} \int_{0}^{s-t} \frac{\rho d\rho}{\sqrt{(s-t)^{2}-\rho^{2}}} \left(1+\frac{r^{2}-1-\rho^{2}}{\sqrt{(r^{2}-1-\rho^{2})^{2}+4r^{2}}}\right)\\
&+\int_{t}^{\infty} ds \frac{\lambda''(s)}{r} \int_{0}^{s-t} \frac{\rho d\rho}{\sqrt{(s-t)^{2}-\rho^{2}}} \frac{4 r \left(p^2+r^2+1\right)}{\left(\left(p^2-r^2+1\right)^2+4 r^2\right)^{3/2}}
\end{split}
\end{equation}
Even though the first line of \eqref{drv1est} is equal to $\frac{-v_{1}}{r}$, the principal contribution to $\partial_{r}v_{1}$ near the origin actually comes from a combination of effects from appropriate parts of both the first and second lines of \eqref{drv1est}. So, we do not simply divide the previous $v_{1}$ estimates by $-r$ to treat the first line of \eqref{drv1est}. Instead, we split the first line of \eqref{drv1est} as
\begin{equation} \begin{split} &\int_{t}^{\infty} ds \frac{-\lambda''(s)}{r^{2}} \int_{0}^{s-t} \frac{\rho d\rho}{\sqrt{(s-t)^{2}-\rho^{2}}} \left(1+\frac{r^{2}-1-\rho^{2}}{\sqrt{(r^{2}-1-\rho^{2})^{2}+4r^{2}}}\right)\\
&= \int_{t}^{\infty} ds \frac{-\lambda''(s)}{r^{2}} \int_{0}^{s-t} \rho d\rho\left(\frac{1}{\sqrt{(s-t)^{2}-\rho^{2}}}-\frac{1}{(s-t)}\right) \left(1+\frac{r^{2}-1-\rho^{2}}{\sqrt{(r^{2}-1-\rho^{2})^{2}+4r^{2}}}\right)\\
&+\int_{t}^{\infty} ds \frac{-\lambda''(s)}{r^{2}} \int_{0}^{s-t} \frac{\rho d\rho}{(s-t)} \left(1+\frac{r^{2}-1-\rho^{2}}{\sqrt{(r^{2}-1-\rho^{2})^{2}+4r^{2}}}\right)
\end{split}
\end{equation}
and
\begin{equation} \begin{split} &|\int_{t}^{\infty} ds \frac{-\lambda''(s)}{r^{2}} \int_{0}^{s-t} \rho d\rho\left(\frac{1}{\sqrt{(s-t)^{2}-\rho^{2}}}-\frac{1}{(s-t)}\right) \left(1+\frac{r^{2}-1-\rho^{2}}{\sqrt{(r^{2}-1-\rho^{2})^{2}+4r^{2}}}\right)|\\
&\leq C\frac{\sup_{x \geq t}\left(|\lambda''(x)|\right)}{r^{2}} \int_{0}^{\infty} \rho d\rho \left(1+\frac{r^{2}-1-\rho^{2}}{\sqrt{(r^{2}-1-\rho^{2})^{2}+4r^{2}}}\right) \int_{\rho+t}^{\infty} ds \left(\frac{1}{\sqrt{(s-t)^{2}-\rho^{2}}}-\frac{1}{s-t}\right)\\
&\leq \frac{C \sup_{x \geq t}\left(|\lambda''(x)|\right)}{r^{2}} \int_{0}^{\infty} \rho d\rho \left(1+\frac{r^{2}-1-\rho^{2}}{\sqrt{(r^{2}-1-\rho^{2})^{2}+4r^{2}}}\right)\\
&\leq C \sup_{x \geq t}\left(|\lambda''(x)|\right)\end{split}\end{equation}
The second line of \eqref{drv1est} is split in the same way
\begin{equation} \begin{split} &\int_{t}^{\infty} ds \frac{\lambda''(s)}{r} \int_{0}^{s-t} \frac{\rho d\rho}{\sqrt{(s-t)^{2}-\rho^{2}}} \frac{4 r \left(p^2+r^2+1\right)}{\left(\left(p^2-r^2+1\right)^2+4 r^2\right)^{3/2}}\\
&=\int_{t}^{\infty} ds \frac{\lambda''(s)}{r} \int_{0}^{s-t} \rho d\rho \left(\frac{1}{\sqrt{(s-t)^{2}-\rho^{2}}}-\frac{1}{(s-t)}\right) \frac{4 r \left(p^2+r^2+1\right)}{\left(\left(p^2-r^2+1\right)^2+4 r^2\right)^{3/2}}\\
&+\int_{t}^{\infty} ds \frac{\lambda''(s)}{r} \int_{0}^{s-t} \frac{\rho d\rho}{(s-t)} \frac{4 r \left(p^2+r^2+1\right)}{\left(\left(p^2-r^2+1\right)^2+4 r^2\right)^{3/2}}
\end{split}
\end{equation}
and we again have
\begin{equation}\begin{split} &|\int_{t}^{\infty} ds \frac{\lambda''(s)}{r} \int_{0}^{s-t} \rho d\rho \left(\frac{1}{\sqrt{(s-t)^{2}-\rho^{2}}}-\frac{1}{(s-t)}\right) \frac{4 r \left(p^2+r^2+1\right)}{\left(\left(p^2-r^2+1\right)^2+4 r^2\right)^{3/2}}|\\
&\leq C \sup_{x \geq t}\left(|\lambda''(x)|\right) \int_{0}^{\infty} \rho d\rho \frac{(\rho^{2}+r^{2}+1)}{((\rho^{2}-r^{2}+1)^{2}+4r^{2})^{3/2}} \int_{\rho+t}^{\infty} ds \left(\frac{1}{\sqrt{(s-t)^{2}-\rho^{2}}}-\frac{1}{s-t}\right) \\
&\leq C \sup_{x \geq t}\left(|\lambda''(x)|\right)\end{split}\end{equation}
So far, we have
\begin{equation}\label{drv1middleest} \begin{split} \partial_{r}v_{1}(t,r) &=\int_{t}^{\infty} ds \frac{-\lambda''(s)}{r^{2}} \int_{0}^{s-t} \frac{\rho d\rho}{(s-t)} \left(1+\frac{r^{2}-1-\rho^{2}}{\sqrt{(r^{2}-1-\rho^{2})^{2}+4r^{2}}}\right)\\
&+\int_{t}^{\infty} ds \frac{\lambda''(s)}{r} \int_{0}^{s-t} \frac{\rho d\rho}{(s-t)} \frac{4 r \left(p^2+r^2+1\right)}{\left(\left(p^2-r^2+1\right)^2+4 r^2\right)^{3/2}}\\
&+E_{0,\partial_{r}v_{1}}\end{split}\end{equation}
where
$$|E_{0,\partial_{r}v_{1}}(t,r)| \leq C \sup_{x \geq t}\left(|\lambda''(x)|\right)$$
Now, we combine the first and second lines of \eqref{drv1middleest}, to get
\begin{equation} \begin{split} &\int_{t}^{\infty} ds \frac{-\lambda''(s)}{r^{2}} \int_{0}^{s-t} \frac{\rho d\rho}{(s-t)} \left(1+\frac{r^{2}-1-\rho^{2}}{\sqrt{(r^{2}-1-\rho^{2})^{2}+4r^{2}}}\right)\\
&+\int_{t}^{\infty} ds \frac{\lambda''(s)}{r} \int_{0}^{s-t} \frac{\rho d\rho}{(s-t)} \frac{4 r \left(p^2+r^2+1\right)}{\left(\left(p^2-r^2+1\right)^2+4 r^2\right)^{3/2}}\\
&= -2 \int_{t}^{\infty} ds \lambda''(s)(s-t) \left(\frac{(r^{2}-1-(s-t)^{2})}{\sqrt{\beta} (1+r^{2}+(s-t)^{2}+\sqrt{\beta})}\right)\end{split}\end{equation}
where
$$\beta = 4r^{2}+(r^{2}-1-(s-t)^{2})^{2}$$
Now, we proceed as in the estimates for $v_{1}$.
\begin{equation} \begin{split} &|-2 \int_{t}^{t+2(r+1)} ds \lambda''(s)(s-t) \left(\frac{(r^{2}-1-(s-t)^{2})}{\sqrt{\beta} (1+r^{2}+(s-t)^{2}+\sqrt{\beta})}\right)|\\
&\leq C \sup_{x \geq t}\left(|\lambda''(x)|\right) \int_{t}^{t+2(r+1)} ds \frac{|s-t|}{1+r^{2}}\\
&\leq C \sup_{x \geq t}\left(|\lambda''(x)|\right)\end{split}\end{equation}
Next, note that
\begin{equation}\begin{split} \left(\frac{(r^{2}-1-(s-t)^{2})}{\sqrt{\beta} (1+r^{2}+(s-t)^{2}+\sqrt{\beta})}\right)&=\frac{-\left(1+\frac{1-r^{2}}{(s-t)^{2}}\right)}{(s-t)^{2}} \left(\frac{1}{\sqrt{1+q}}\cdot \frac{1}{1+y+\sqrt{1+q}}\right)\end{split}\end{equation}
where
$$q=\frac{2(1-r^{2})}{(s-t)^{2}} + \frac{(r^{2}+1)^{2}}{(s-t)^{4}}$$
$$y=\frac{1+r^{2}}{(s-t)^{2}}$$
So, for $s-t \geq 2(r+1)$,
$$|q| \leq \frac{9}{16}, \quad |y| \leq \frac{1}{4}$$
Using this, we have
\begin{equation} \left(\frac{(r^{2}-1-(s-t)^{2})}{\sqrt{\beta} (1+r^{2}+(s-t)^{2}+\sqrt{\beta})}\right) = \frac{-1}{2(s-t)^{2}} \left(1+O\left(\frac{1+r^{2}}{(s-t)^{2}}\right)\right)\end{equation}
So,
\begin{equation}\begin{split} &-2 \int_{t+2(r+1)}^{\infty} ds \lambda''(s)(s-t) \left(\frac{(r^{2}-1-(s-t)^{2})}{\sqrt{\beta} (1+r^{2}+(s-t)^{2}+\sqrt{\beta})}\right)\\
&= \int_{t+2(r+1)}^{\infty} ds \frac{\lambda''(s)}{(s-t)} + E_{1,\partial_{r}v_{1}}(t,r)\end{split}\end{equation}
where
$$|E_{1,\partial_{r}v_{1}}(t,r)| \leq C \int_{t+2(r+1)}^{\infty} ds \frac{|\lambda''(s)|}{(s-t)} \frac{(1+r^{2})}{(s-t)^{2}} \leq C \sup_{x \geq t}\left(|\lambda''(x)|\right)$$
Then,
\begin{equation} \int_{t+2(r+1)}^{\infty} ds \frac{\lambda''(s)}{(s-t)} = \int_{t+2(r+1)}^{\infty} ds \frac{\lambda''(s)}{1+s-t} + E_{2,\partial_{r}v_{1}}(t,r)\end{equation}
with
$$|E_{2,\partial_{r}v_{1}}(t,r)|\leq \int_{t+2(r+1)}^{\infty} ds |\lambda''(s)| \left(\frac{1}{s-t}-\frac{1}{1+s-t}\right) \leq C \sup_{x \geq t}\left(|\lambda''(x)|\right)$$
Finally,
\begin{equation}\begin{split} \int_{t+2(r+1)}^{\infty} ds \frac{\lambda''(s)}{1+s-t} &= \int_{t}^{\infty} ds \frac{\lambda''(s)}{1+s-t} - \int_{t}^{t+2(r+1)} ds \frac{\lambda''(s)}{1+s-t}\\
&=\int_{t}^{\infty} ds \frac{\lambda''(s)}{1+s-t} +E_{3,\partial_{r}v_{1}}(t,r)\end{split}\end{equation}
where
$$|E_{3,\partial_{r}v_{1}}(t,r)| \leq C \sup_{x \geq t}\left(|\lambda''(x)|\right) \log(3+2r)$$
and we get the desired result
\begin{equation}\partial_{r}v_{1}(t,r) = \int_{t}^{\infty}  \frac{\lambda''(s)}{1+s-t}ds + E_{\partial_{r}v_{1}}(t,r)\end{equation}
with
$$|E_{\partial_{r}v_{1}}(t,r)| \leq C \sup_{x \geq t}\left(|\lambda''(x)|\right) \log(3+2r)$$
It remains to prove the last estimate in the lemma statement: The first term of \eqref{drv1est} is estimated by
\begin{equation}\begin{split} |\int_{t}^{\infty} ds \frac{-\lambda''(s)}{r^{2}} \int_{0}^{s-t} \frac{\rho d\rho}{\sqrt{(s-t)^{2}-\rho^{2}}} \left(1+\frac{r^{2}-1-\rho^{2}}{\sqrt{(r^{2}-1-\rho^{2})^{2}+4r^{2}}}\right)|&\leq \frac{C}{r^{2}} \int_{t}^{\infty} ds |\lambda''(s)|(s-t)\end{split}\end{equation}
Turning to the second term of \eqref{drv1est}, we have
\begin{equation} \begin{split} &|\int_{t}^{\infty} ds \frac{\lambda''(s)}{r} \int_{0}^{s-t} \frac{\rho d\rho}{\sqrt{(s-t)^{2}-\rho^{2}}} \frac{4 r \left(p^2+r^2+1\right)}{\left(\left(p^2-r^2+1\right)^2+4 r^2\right)^{3/2}}|\\
&\leq C \int_{0}^{\infty} d\rho \frac{\rho(\rho^{2}+r^{2}+1)}{((\rho^{2}-r^{2}+1)^{2}+4r^{2})^{3/2}} \int_{\rho+t}^{\infty} ds \frac{|\lambda''(s)| (1+(s-t)^{2})}{\sqrt{(s-t)^{2}-\rho^{2})}(1+(s-t)^{2})}\\
&\leq C \text{sup}_{x \geq t} \left(|\lambda''(x)| (1+(x-t)^{2})\right) \int_{0}^{\infty}d\rho \frac{\rho(\rho^{2}+r^{2}+1)}{((\rho^{2}-r^{2}+1)^{2}+4r^{2})^{3/2}} \int_{\rho+t}^{\infty} \frac{ds}{\sqrt{(s-t)^{2}-\rho^{2}}} \frac{1}{(1+(s-t)^{2})}\\
&\leq C\text{sup}_{x \geq t} \left(|\lambda''(x)| (1+(x-t)^{2})\right) \int_{0}^{\infty}d\rho \frac{(\rho^{2}+r^{2}+1)}{((\rho^{2}-r^{2}+1)^{2}+4r^{2})^{3/2}} \frac{1}{\sqrt{1+\rho^{2}}} \end{split}\end{equation}
Let us first note that, for $ \rho \leq \frac{r}{2}$, we have
$$r^{2}-1-\rho^{2} > C r^{2}, \quad \text{ for } r >2$$
so, 
$$\frac{(\rho^{2}+r^{2}+1)}{((\rho^{2}-r^{2}+1)^{2}+4r^{2})^{3/2}} \leq \frac{C}{r^{4}}$$
and
\begin{equation}\begin{split} \int_{0}^{r/2}d\rho \frac{(\rho^{2}+r^{2}+1)}{((\rho^{2}-r^{2}+1)^{2}+4r^{2})^{3/2}} \frac{1}{\sqrt{1+\rho^{2}}} \leq \frac{C}{r^{4}} \sinh^{-1}\left(\frac{r}{2}\right)\end{split}\end{equation}
On the other hand, when $\rho \geq r/2$, we have
\begin{equation} \begin{split}&\int_{r/2}^{\infty}d\rho \frac{(\rho^{2}+r^{2}+1)}{((\rho^{2}-r^{2}+1)^{2}+4r^{2})^{3/2}} \frac{1}{\sqrt{1+\rho^{2}}} \\
&\leq C \int_{r/2}^{\infty} d\rho \frac{(\rho^{2}+r^{2}+1)}{\rho((\rho^{2}-r^{2}+1)^{2}+4r^{2})^{3/2}}\\
&=\frac{r^2+\frac{3 r^4}{\sqrt{9 r^4+40 r^2+16}}-\frac{9 r^2}{\sqrt{9 r^4+40 r^2+16}}-\frac{12}{\sqrt{9 r^4+40 r^2+16}}}{4 \left(r^2+1\right)^2}\\
&+\frac{2 \log \left(3 r^4+\left(\sqrt{9 r^4+40 r^2+16}+9\right) r^2+\sqrt{9 r^4+40 r^2+16}+4\right)-4 \log (r)+1-\log (4)}{4 \left(r^2+1\right)^2}\\
&\leq \frac{C}{r^{2}}\end{split}\end{equation}
So, we conclude that
\begin{equation} \begin{split} &|\int_{t}^{\infty} ds \frac{\lambda''(s)}{r} \int_{0}^{s-t} \frac{\rho d\rho}{\sqrt{(s-t)^{2}-\rho^{2}}} \frac{4 r \left(p^2+r^2+1\right)}{\left(\left(p^2-r^2+1\right)^2+4 r^2\right)^{3/2}}|\\
&\leq C\text{sup}_{x \geq t} \left(|\lambda''(x)| (1+(x-t)^{2})\right) \frac{1}{r^{2}}\end{split}\end{equation}
Combining these estimates, we get
\begin{equation} |\partial_{r}v_{1}(t,r)| \leq \frac{C}{r^{2}}\left(\text{sup}_{x \geq t} \left(|\lambda''(x)| (1+(x-t)^{2})\right)+\int_{t}^{\infty}  |\lambda''(s)| (s-t) ds\right), \quad r >2\end{equation}

\end{proof}
\subsection{The free wave correction}
The next correction is designed to produce a crucial source term in the equation that we will eventually use to choose a specific $\lambda$. In fact, it is this correction which ultimately determines the asymptotics of the $\lambda$ which we will choose.
Define $v_{2}$ to be the solution to
\begin{equation} \begin{cases} -\partial_{tt}v_{2}+\partial_{rr}v_{2}+\frac{1}{r}\partial_{r}v_{2}-\frac{v_{2}}{r^{2}}=0\\
v_{2}(0)=0\\
\partial_{t}v_{2}(0)=v_{2,0}\end{cases}\end{equation}
where $$\widehat{v_{2,0}}(\xi) = \int_{0}^{\infty} v_{2,0}(r) J_{1}(r\xi)r dr =  \begin{cases}\frac{4b}{\pi(b-1)}\frac{\chi_{\leq \frac{1}{4}}(\xi)}{\log^{b-1}(\frac{1}{\xi})}, \quad b \neq 1\\
\frac{-4}{\pi} \chi_{\leq \frac{1}{4}}(\xi) \log(\log(\frac{1}{\xi})), \quad b=1\end{cases}$$
where $\chi_{\leq\frac{1}{4}} \in C^{\infty}_{c}([0,\infty))$, $0 \leq \chi_{\leq \frac{1}{4}}(x) \leq 1$ for all $x$, and $\chi_{\leq \frac{1}{4}}$ satisfies $$\chi_{\leq \frac{1}{4}}(x) = \begin{cases}1, \quad 0\leq x \leq \frac{1}{8}\\
0, \quad x >\frac{1}{4}\end{cases}$$ and is otherwise arbitrary. For ease of notation, let
\begin{equation} c_{b} = \begin{cases}\frac{4 b}{\pi(b-1)}, \quad b \neq 1\\
\frac{-4}{\pi}, \quad b=1\end{cases}\end{equation} 
Note that this particular form of $c_{b}$ is due to the fact that part of $\widehat{v_{2,0}}(\xi)$ involves an antiderivative of $\frac{1}{\xi \log^{b}(\frac{1}{\xi})}$.   
We have a formula for $v_{2}$, namely: 
\begin{equation}\label{v2def}v_{2}(t,r) = c_{b} \int_{0}^{\infty} d\xi \sin(t \xi) J_{1}(r\xi) \chi_{\leq \frac{1}{4}}(\xi)\cdot \begin{cases}\frac{1}{\log^{b-1}(\frac{1}{\xi})}, \quad b \neq 1\\
\log(\log(\frac{1}{\xi})), \quad b=1 \end{cases}\end{equation}
The significance of this particular choice of Cauchy data will be seen later on, when we identify the $v_{2}$-related contribution to the equation we use to choose $\lambda$.\\
\\
We will prove pointwise estimates on $v_{2}$ later on, but, to give the reader some idea of the behavior of $v_{2}$, we note that (for instance, for $b > 1$)
\begin{equation}\begin{split} v_{2}(t,r) &= \frac{c_{b}(1-\text{sgn}(t-r))}{2\sqrt{2 r} \sqrt{|t-r|} \log^{b-1}(|t-r|)} +E_{2}(t,r)\end{split}\end{equation}
with
\begin{equation}\begin{split} |E_{2}(t,r)| &\leq C\left(\frac{1}{\sqrt{r}\sqrt{|t-r|}\log^{b}(|t-r|)}\right)+\frac{C}{\sqrt{r}\sqrt{t+r}\log^{b-1}(t+r)}+\frac{C}{r \log^{b-1}(r)}, \quad r>\frac{t}{2}, \quad |t-r| > 5\end{split}\end{equation}
This can be established by a procedure similar to the one which we use later on to compute the inner product of the $v_{2}$ linear error term with $\phi_{0}(\frac{\cdot}{\lambda(t)})$.

\subsection{Further improvement of the soliton error term}
If we substitute $u=Q_{\frac{1}{\lambda(t)}}+v_{1}+v_{2}+u_{3}$ into the wave maps equation $$-\partial_{tt}u+\partial_{rr}u+\frac{1}{r}\partial_{r}u-\frac{\sin(2 u)}{2r^{2}} =0$$
we get
\begin{equation}\label{v3eqn} \begin{split} -\partial_{tt}u_{3}+\partial_{rr}u_{3}+\frac{1}{r}\partial_{r}u_{3}-\frac{\cos(2Q_{\frac{1}{\lambda(t)}})}{r^{2}}u_{3} &= \partial_{tt}Q_{\frac{1}{\lambda(t)}} + \frac{2\lambda''(t) r}{1+r^{2}}+\left(\frac{\cos(2Q_{\frac{1}{\lambda(t)}})-1}{r^{2}}\right)v_{1}\\
&+\left(\frac{\cos(2Q_{\frac{1}{\lambda(t)}})-1}{r^{2}}\right)v_{2}\\
&+N_{2}(v_{1}+v_{2})+N(u_{3})+L(u_{3})\end{split}\end{equation}
where $$N(f) = \left(\frac{\sin(2f)-2f}{2r^{2}}\right)\cos(2Q_{\frac{1}{\lambda(t)}}) + \left(\frac{\cos(2f)-1}{2r^{2}}\right)\sin(2(Q_{\frac{1}{\lambda(t)}}+v_{1}+v_{2}))$$
$$L(f)=\frac{\sin(2f)}{2r^{2}} \cos(2Q_{\frac{1}{\lambda(t)}})(\cos(2(v_{1}+v_{2}))-1) -\frac{\sin(2f)}{2r^{2}}\sin(2Q_{\frac{1}{\lambda(t)}})\sin(2(v_{1}+v_{2}))$$
$$N_{2}(f) = \frac{\sin(2Q_{\frac{1}{\lambda(t)}})}{2r^{2}}(\cos(2f)-1)+\frac{\cos(2Q_{\frac{1}{\lambda(t)}})}{2r^{2}}(\sin(2f)-2f)$$
Note that
$$F_{0}(t,r) = \partial_{tt}Q_{\frac{1}{\lambda(t)}} + \frac{2\lambda''(t) r}{1+r^{2}}$$
appears on the right-hand side of \eqref{v3eqn}. When the spatial coordinate is renormalized, via $$r=R \lambda(t)$$ one term arising in the the large $R$ expansion of $F_{0}$ has insufficient decay in time.  To remedy this, we will add another correction, to be denoted $v_{3}$. On one hand, choosing $v_{3}$ to solve an equation whose right-hand side is excatly equal to $F_{0}$ would eventually lead to a much more difficult equation that we use to determine $\lambda(t)$. On the other hand, the error terms remaining after adding the correction $v_{3}$ should no longer have insufficient decay in time for large values of $R$. We therefore proceed as follows: recall that $\alpha$ has been introduced just before \ref{T0initialconstraint}, and satisfies 
$$0 < \alpha < \text{min}\{\frac{1}{b (1040!)},\frac{1}{1040!}\}$$
 and let
$$F_{0,1}(t,r) = \frac{2 r \lambda''(t)}{(\lambda(t)^{2}+r^{2})} \left(\frac{-1+\lambda(t)^{2}}{1+r^{2}} + \frac{1-\lambda(t)^{2\alpha}}{1+r^{2}\lambda(t)^{2\alpha -2}}\right)$$
Next, we consider $v_{3}$, defined as the solution (with 0 Cauchy data at infinity) to the equation
\begin{equation} -\partial_{tt}v_{3}+\partial_{rr}v_{3}+\frac{1}{r}\partial_{r}v_{3}-\frac{v_{3}}{r^{2}} = F_{0,1}(t,r)\end{equation}
Following the same steps as for $v_{1}$, we get
\begin{equation}v_{3}(t,r) = -\frac{1}{2\pi} \int_{t}^{\infty} ds \int_{0}^{s-t} \frac{\rho d\rho}{\sqrt{(s-t)^{2}-\rho^{2}}} \int_{0}^{2\pi} d\theta \frac{F_{0,1}(s,\sqrt{r^{2}+\rho^{2}+2 r \rho \cos(\theta)})(r+\rho \cos(\theta))}{\sqrt{r^{2}+\rho^{2}+2 r \rho \cos(\theta)}}\end{equation}
which gives
\begin{equation}\label{v3def}\begin{split}v_{3}(t,r)=-\frac{1}{r} \int_{t}^{\infty} ds \int_{0}^{s-t} \frac{\rho d\rho}{\sqrt{(s-t)^{2}-\rho^{2}}} \lambda''(s)&\left(\frac{-1-\rho^{2}+r^{2}}{\sqrt{(1+\rho^{2}-r^{2})^{2}+4r^{2}}}\right.\\
&\left.+\frac{\lambda(s)^{2}-(r^{2}-\rho^{2})\lambda(s)^{2\alpha}}{\lambda(s)^{2}\sqrt{1+2(\rho^{2}+r^{2})\lambda(s)^{2\alpha-2} + (\rho^{2}-r^{2})^{2}\lambda(s)^{4\alpha-4}}}\right)\end{split} \end{equation}
The main result of this section is a decomposition of $v_{3}$ which will be useful in understanding its contribution to the equation for $\lambda$:
\begin{lemma} \begin{equation}\label{v3preciseforip} v_{3}(t,r) = -r \int_{t+6r}^{\infty} ds \lambda''(s)(s-t)\left(\frac{1}{1+(s-t)^{2}}-\frac{1}{\lambda(t)^{2-2\alpha}+(s-t)^{2}}\right)+E_{5}(t,r)\end{equation}
where
\begin{equation}\begin{split} |E_{5}(t,r)|&\leq C \left(\sup_{x \geq t}|\lambda''(x)|\right) \cdot\sup_{x \geq t}\left(\frac{|\lambda'(x)| \lambda(t)^{\alpha}}{\lambda(x)^{\alpha}}\right) \lambda(t)^{1-2\alpha}\\
&+C r \sup_{x \geq t}\left(|\lambda''(x)| \lambda(x)^{\alpha-1}(\lambda(x)^{\alpha-1}-\lambda(t)^{\alpha-1})\right)\lambda(t)^{2-2\alpha}\\
&+C r \sup_{x \geq t}|\lambda''(x)| \end{split}\end{equation}
\end{lemma}
\begin{proof}
\begin{equation}\begin{split}v_{3}(t,r)=-\frac{1}{r} \int_{t}^{\infty} ds \int_{0}^{s-t} \frac{\rho d\rho}{\sqrt{(s-t)^{2}-\rho^{2}}} \lambda''(s)&\left(\frac{-1-\rho^{2}+r^{2}}{\sqrt{(1+\rho^{2}-r^{2})^{2}+4r^{2}}}\right.\\
&\left.+\frac{\lambda(s)^{2}-(r^{2}-\rho^{2})\lambda(s)^{2\alpha}}{\lambda(s)^{2}\sqrt{1+2(\rho^{2}+r^{2})\lambda(s)^{2\alpha-2} + (\rho^{2}-r^{2})^{2}\lambda(s)^{4\alpha-4}}}\right)\end{split} \end{equation}
$v_{3}$ is then decomposed as 
\begin{equation} v_{3}(t,r) = v_{3,1}(t,r)+v_{3,2}(t,r)\end{equation}
where
\begin{equation}\label{v31def}\begin{split} v_{3,1}(t,r) =-\frac{1}{r} \int_{t}^{\infty} ds \int_{0}^{s-t} \frac{\rho d\rho}{(s-t)} \lambda''(s)&\left(\frac{-1-\rho^{2}+r^{2}}{\sqrt{(1+\rho^{2}-r^{2})^{2}+4r^{2}}}\right.\\
&\left.+\frac{\lambda(s)^{2}-(r^{2}-\rho^{2})\lambda(s)^{2\alpha}}{\lambda(s)^{2}\sqrt{1+2(\rho^{2}+r^{2})\lambda(s)^{2\alpha-2} + (\rho^{2}-r^{2})^{2}\lambda(s)^{4\alpha-4}}}\right)\end{split} \end{equation} 
$$v_{3,2}=v_{3}-v_{3,1}$$
Then, we record some pointwise estimates on $v_{3,2}$. If
$$F_{3}(r,\rho,\lambda(s)) = \frac{1-(r^{2}-\rho^{2})\lambda(s)^{2\alpha-2}}{\sqrt{1+2(\rho^{2}+r^{2})\lambda(s)^{2\alpha-2} + (\rho^{2}-r^{2})^{2} \lambda(s)^{4\alpha-4}}}=\frac{1-(r^{2}-\rho^{2})\lambda(s)^{2\alpha-2}}{\sqrt{4r^{2}\lambda(s)^{2\alpha-2}+(1-(r^{2}-\rho^{2})\lambda(s)^{2\alpha-2})^{2}}}$$
then,
\begin{equation}\label{v32eqn}\begin{split}v_{3,2}(t,r)&= \frac{-1}{r} \int_{t}^{\infty} ds \int_{0}^{s-t} \rho d\rho \left(\frac{1}{\sqrt{(s-t)^{2}-\rho^{2}}}-\frac{1}{(s-t)}\right)\lambda''(s)\left(\frac{-1-\rho^{2}+r^{2}}{\sqrt{(1+\rho^{2}-r^{2})^{2}+4r^{2}}}+F_{3}(r,\rho,\lambda(t))\right)\\
&-\frac{1}{r} \int_{t}^{\infty} ds \int_{0}^{s-t} \rho d\rho \left(\frac{1}{\sqrt{(s-t)^{2}-\rho^{2}}}-\frac{1}{(s-t)}\right)\lambda''(s) \left(F_{3}(r,\rho,\lambda(s))-F_{3}(r,\rho,\lambda(t))\right)\end{split}\end{equation}
For the first line of \eqref{v32eqn}, we have
\begin{equation}\begin{split}&|-\frac{1}{r} \int_{t}^{\infty} ds \int_{0}^{s-t} \rho d\rho \left(\frac{1}{\sqrt{(s-t)^{2}-\rho^{2}}}-\frac{1}{s-t}\right) \lambda''(s) \left(\frac{-1-\rho^{2}+r^{2}}{\sqrt{(-1-\rho^{2}+r^{2})^{2}+4r^{2}}}+F_{3}(r,\rho,\lambda(t))\right)|\\
&\leq \frac{C}{r} \left(\sup_{x \geq t}|\lambda''(x)|\right) \int_{0}^{\infty} \rho d\rho |\frac{-1-\rho^{2}+r^{2}}{\sqrt{(-1-\rho^{2}+r^{2})^{2}+4r^{2}}}+1-1+F_{3}(r,\rho,\lambda(t))| \int_{t+\rho}^{\infty} \left(\frac{1}{\sqrt{(s-t)^{2}-\rho^{2}}}-\frac{1}{s-t}\right) ds\\
&\leq \frac{C}{r} \left(\sup_{x \geq t}|\lambda''(x)|\right)\left( \int_{0}^{\infty} \rho d\rho \left(1-\frac{1+\rho^{2}-r^{2}}{\sqrt{(-1-\rho^{2}+r^{2})^{2}+4r^{2}}}\right) +\int_{0}^{\infty} \rho d\rho \left(1-F_{3}(r,\rho,\lambda(t))\right)\right)\\
&\leq C r \left(\sup_{x \geq t}|\lambda''(x)|\right)\end{split}\end{equation}
where we used the facts that
$$\frac{-1-\rho^{2}+r^{2}}{\sqrt{(-1-\rho^{2}+r^{2})^{2}+4r^{2}}} \geq -1$$
and
$$F_{3}(r,\rho,\lambda(t)) \leq 1$$
To estimate the second line of \eqref{v32eqn}, we first note that
\begin{equation} F_{3}(r,\rho,\lambda(s))-F_{3}(r,\rho,\lambda(t)) = \int_{0}^{1} \frac{-4 r^{2} z_{\sigma} (1+(r^{2}-\rho^{2})z_{\sigma}^{2})}{(1+2(\rho^{2}+r^{2})z_{\sigma}^{2}+(\rho^{2}-r^{2})^{2}z_{\sigma}^{4})^{3/2}}\left(\lambda(s)^{\alpha-1}-\lambda(t)^{\alpha-1}\right) d\sigma\end{equation}
where
$$z_{\sigma} = \sigma \lambda(s)^{\alpha-1}+(1-\sigma) \lambda(t)^{\alpha-1}$$
First, we note that
\begin{equation} 1+2(\rho^{2}+r^{2}) z^{2}+(\rho^{2}-r^{2})^{2} z^{4} = 4\rho^{2}z^{2}+(1+(r^{2}-\rho^{2})z^{2})^{2}\end{equation}
So,
\begin{equation} |\frac{-4 r^{2} z_{\sigma} (1+(r^{2}-\rho^{2})z_{\sigma}^{2})}{(1+2(\rho^{2}+r^{2})z_{\sigma}^{2}+(\rho^{2}-r^{2})^{2}z_{\sigma}^{4})^{3/2}}\left(\lambda(s)^{\alpha-1}-\lambda(t)^{\alpha-1}\right)| \leq C \frac{r^{2} |z_{\sigma}|\left(\lambda(s)^{\alpha-1}-\lambda(t)^{\alpha-1}\right)}{(1+2(\rho^{2}+r^{2})z_{\sigma}^{2}+(\rho^{2}-r^{2})^{2}z_{\sigma}^{4})}\end{equation} 
Since $\lambda$ is a decreasing function, and $0<\alpha<\frac{1}{8}$, we have
\begin{equation} \lambda(t)^{\alpha -1} \leq|z_{\sigma}|= |\sigma \lambda(s)^{\alpha -1}+(1-\sigma)\lambda(t)^{\alpha-1}| \leq \lambda(s)^{\alpha -1}, \quad 0 \leq \sigma \leq 1, \quad s \geq t\end{equation}
So, \begin{equation} |F_{3}(r,\rho,\lambda(s))-F_{3}(r,\rho,\lambda(t))| \leq C\frac{r^{2} \lambda(s)^{\alpha -1} |\lambda(s)^{\alpha -1}-\lambda(t)^{\alpha -1}|}{(1+2(\rho^{2}+r^{2})\lambda(t)^{2\alpha-2}+(\rho^{2}-r^{2})^{2}\lambda(t)^{4\alpha -4})}, \quad s \geq t\end{equation}
This gives
\begin{equation} \begin{split} &|-\frac{1}{r} \int_{t}^{\infty} ds \int_{0}^{s-t} \rho d\rho \left(\frac{1}{\sqrt{(s-t)^{2}-\rho^{2}}}-\frac{1}{(s-t)}\right)\lambda''(s) \left(F_{3}(r,\rho,\lambda(s))-F_{3}(r,\rho,\lambda(t))\right)|\\
&\leq C r \int_{0}^{\infty} \rho d\rho \int_{\rho+t}^{\infty} ds \left(\frac{1}{\sqrt{(s-t)^{2}-\rho^{2}}}-\frac{1}{(s-t)}\right) |\lambda''(s)| \frac{\lambda(s)^{\alpha -1}|\lambda(s)^{\alpha -1}-\lambda(t)^{\alpha -1}|}{(1+2(\rho^{2}+r^{2})\lambda(t)^{2\alpha-2}+(\rho^{2}-r^{2})^{2}\lambda(t)^{4\alpha -4})}\\
&\leq C r \sup_{x \geq t}\left(|\lambda''(x)| \lambda(x)^{\alpha-1} |\lambda(x)^{\alpha-1}-\lambda(t)^{\alpha-1}|\right)\int_{0}^{\infty} \frac{\rho d\rho}{(1+2(\rho^{2}+r^{2})\lambda(t)^{2\alpha-2} +(\rho^{2}-r^{2})^{2} \lambda(t)^{4\alpha-4})}\end{split}\end{equation}
Then, if $2r > \lambda(t)^{1-\alpha}$, we have \begin{equation}\begin{split} \int_{0}^{\infty} \frac{\rho d\rho}{(1+2(\rho^{2}+r^{2})\lambda(t)^{2\alpha-2} +(\rho^{2}-r^{2})^{2} \lambda(t)^{4\alpha-4})}&\leq C\left( \int_{0}^{\lambda(t)^{1-\alpha}} \rho d\rho + \int_{\lambda(t)^{1-\alpha}}^{2r} \frac{\rho d\rho}{r^{2}\lambda(t)^{2\alpha-2}} + \int_{2r}^{\infty} \frac{\rho d\rho}{\rho^{4} \lambda(t)^{4\alpha-4}}\right)\\
&\leq C \lambda(t)^{2-2\alpha}\end{split}\end{equation}
On the other hand, if $2r \leq \lambda(t)^{1-\alpha}$, we have
\begin{equation}\begin{split} \int_{0}^{\infty} \frac{\rho d\rho}{(1+2(\rho^{2}+r^{2})\lambda(t)^{2\alpha-2} +(\rho^{2}-r^{2})^{2} \lambda(t)^{4\alpha-4})}&\leq \int_{0}^{\lambda(t)^{1-\alpha}} \rho d\rho + \int_{\lambda(t)^{1-\alpha}}^{\infty} \frac{d\rho}{\rho^{3}\lambda(t)^{4\alpha-4}}\\
&\leq C \lambda(t)^{2-2\alpha}\end{split}\end{equation}
In total, we get
\begin{equation}\begin{split} &|-\frac{1}{r} \int_{t}^{\infty} ds \int_{0}^{s-t} \rho d\rho \left(\frac{1}{\sqrt{(s-t)^{2}-\rho^{2}}}-\frac{1}{(s-t)}\right)\lambda''(s) \left(F_{3}(r,\rho,\lambda(s))-F_{3}(r,\rho,\lambda(t))\right)|\\
& \leq C r \sup_{x \geq t}\left(|\lambda''(x)| \lambda(x)^{\alpha-1} |\lambda(x)^{\alpha-1}-\lambda(t)^{\alpha-1}|\right) \lambda(t)^{2-2\alpha}\end{split}\end{equation}
It then suffices to study $v_{3,1}$. Firstly, we have
\begin{equation}\begin{split} &|-\frac{1}{r} \int_{t}^{t+6r} ds \int_{0}^{s-t} \frac{\rho d\rho}{(s-t)} \lambda''(s)\left(\frac{-1-\rho^{2}+r^{2}}{\sqrt{(-1-\rho^{2}+r^{2})^{2}+4r^{2}}}+F_{3}(r,\rho,\lambda(s))\right)| \\
&\leq \frac{1}{r} \int_{t}^{t+6r} ds \int_{0}^{s-t} \frac{\rho d\rho}{(s-t)} |\lambda''(s)|\cdot 2\\
&\leq C r \sup_{x \geq t}|\lambda''(x)|\end{split}\end{equation}
Next, we have
\begin{equation}\begin{split}&-\frac{1}{r} \int_{t+6r}^{\infty} ds \int_{0}^{s-t} \frac{\rho d\rho}{(s-t)} \lambda''(s)\left(\frac{-1-\rho^{2}+r^{2}}{\sqrt{(-1-\rho^{2}+r^{2})^{2}+4r^{2}}}+F_{3}(r,\rho,\lambda(s))\right)\\
=-2r \int_{6r}^{\infty} dw &\lambda''(t+w) w\left(\frac{1}{(1+w^{2})\left(\frac{r^{2}}{1+w^{2}}+1+\sqrt{(\frac{r^{2}}{1+w^{2}}+1)^{2}-4\frac{r^{2}w^{2}}{(1+w^{2})^{2}}}\right)}\right.\\
&\left.+\frac{-1}{(\lambda(t+w)^{2-2\alpha}+w^{2})\left(\frac{r^{2}}{\lambda(t+w)^{2-2\alpha}+w^{2}}+1+\sqrt{(\frac{r^{2}}{\lambda(t+w)^{2-2\alpha}+w^{2}}+1)^{2}-\frac{4r^{2}w^{2}}{(\lambda(t+w)^{2-2\alpha}+w^{2})^{2}}}\right)}\right)\end{split}\end{equation}
Using the fact that $w=s-t \geq 6r$ in the integral below, we get
\begin{equation}\begin{split} &-\frac{1}{r} \int_{t+6r}^{\infty} ds \int_{0}^{s-t} \frac{\rho d\rho}{(s-t)} \lambda''(s)\left(\frac{-1-\rho^{2}+r^{2}}{\sqrt{(-1-\rho^{2}+r^{2})^{2}+4r^{2}}}+F_{3}(r,\rho,\lambda(s))\right)\\
&=-2r \int_{6r}^{\infty} dw \lambda''(t+w) w\left(\frac{1}{2(1+w^{2})}-\frac{1}{2(\lambda(t+w)^{2-2\alpha}+w^{2})}\right)+E_{4}\end{split}\end{equation}
where
\begin{equation}\begin{split}|E_{4}| &\leq C r \sup_{x \geq t}|\lambda''(x)| \int_{6r}^{\infty} dw \frac{r^{2} w}{w^{2}} \left(\frac{1}{1+w^{2}}+\frac{1}{\lambda(t+w)^{2-2\alpha} + w^{2}}\right)\\
&\leq C r \sup_{x \geq t}|\lambda''(x)|\end{split}\end{equation}
So, it suffices to study 
\begin{equation}  -2r \int_{6r}^{\infty} dw \lambda''(t+w) w\left(\frac{1}{2(1+w^{2})}-\frac{1}{2(\lambda(t+w)^{2-2\alpha}+w^{2})}\right)\end{equation}
We will make one more reduction, which is to replace $\lambda(t+w)$ in the above expression with $\lambda(t)$. The error in doing this replacement is
\begin{equation} \begin{split} &-r \int_{t+6r}^{\infty} ds \lambda''(s)(s-t)\left(\frac{1}{(\lambda(t)^{2-2\alpha}+(s-t)^{2})}-\frac{1}{(\lambda(s)^{2-2\alpha}+(s-t)^{2})}\right)\end{split}\end{equation}
But, if
$$F_{4}(x,s-t)=\frac{1}{x^{2}+(s-t)^{2}}$$
then,
\begin{equation}\begin{split}|F_{4}(\lambda(s)^{1-\alpha},s-t)-F_{4}(\lambda(t)^{1-\alpha},s-t)| &\leq  C \frac{\lambda(t)^{1-\alpha}}{(s-t)^{4}} \cdot |\lambda(t)^{1-\alpha}-\lambda(s)^{1-\alpha}|\\
&\leq C \frac{\lambda(t)^{1-2\alpha}}{(s-t)^{3}} \cdot \sup_{x \geq t}\left(\frac{|\lambda'(x)| \lambda(t)^{\alpha}}{\lambda(x)^{\alpha}}\right)\end{split}\end{equation}
where we use the fact that $\lambda$ is a decreasing function. 
So, 
\begin{equation} \begin{split} &|-r \int_{t+6r}^{\infty} ds \lambda''(s)(s-t)\left(\frac{1}{(\lambda(t)^{2-2\alpha}+(s-t)^{2})}-\frac{1}{(\lambda(s)^{2-2\alpha}+(s-t)^{2})}\right)|\\
&\leq C r \left(\sup_{x \geq t}|\lambda''(x)|\right)\cdot \sup_{x \geq t}\left(\frac{|\lambda'(x)| \lambda(t)^{\alpha}}{\lambda(x)^{\alpha}}\right) \lambda(t)^{1-2\alpha} \int_{6r}^{\infty} dw \frac{w}{w^{3}}\\
&\leq C \left(\sup_{x \geq t}|\lambda''(x)|\right) \cdot\sup_{x \geq t}\left(\frac{|\lambda'(x)| \lambda(t)^{\alpha}}{\lambda(x)^{\alpha}}\right) \lambda(t)^{1-2\alpha}\end{split}\end{equation}
This finally gives \eqref{v3preciseforip}.
\end{proof}
\subsection{The linear error terms for large $r$}
Despite the decay of $\frac{1-\cos(2Q_{1}(\frac{r}{\lambda(t)}))}{r^{2}}$ for large $r$, we will still need to add another correction which improves the linear error terms of $v_{1},v_{2},v_{3}$, as well as $F_{0,2}=-F_{0,1}+\frac{2 \lambda''(t) r}{1+r^{2}}+\partial_{t}^{2} Q_{\frac{1}{\lambda(t)}}$, for large $r$. The addition of this correction will not change the leading order contribution of these error terms to the modulation equation, but, will improve the overall error term of the final ansatz for large $r$.
Let $\chi_{\geq 1} \in C^{\infty}(\mathbb{R})$ satisfy
$$\chi_{\geq 1}(x) = \begin{cases} 1, \quad x \geq 2\\
0, \quad x < 1\end{cases}$$
and
$$0 \leq \chi_{\geq 1}(x) \leq 1, \quad x \in \mathbb{R}$$
Then, we recall that $N$ has been defined just before \eqref{T0initialconstraint}, let 
$$v_{4,c}(t,r) = \chi_{\geq 1}(\frac{2r}{\log^{N}(t)}) \left(\left(\frac{\cos(2Q_{1}(\frac{r}{\lambda(t)}))-1}{r^{2}}\right)\left(v_{1}+v_{2}+v_{3}\right)+F_{0,2}(t,r)\right)$$
and define $v_{4}$ as the solution to
$$-\partial_{tt}v_{4}+\partial_{rr}v_{4}+\frac{1}{r}\partial_{r}v_{4}-\frac{v_{4}}{r^{2}} = v_{4,c}(t,r)$$
with 0 Cauchy data at infinity. In other words, we have
\begin{equation} v_{4}(t,r) = \int_{t}^{\infty} v_{4,s}(t,r) ds\end{equation}
where
$v_{4,s}$ solves
\begin{equation}\begin{cases} -\partial_{tt}v_{4,s} + \partial_{rr}v_{4,s} + \frac{1}{r} \partial_{r}v_{4,s} - \frac{v_{4,s}}{r^{2}}=0\\
v_{4,s}(s,r) = 0\\
\partial_{t}v_{4,s}(s,r) = v_{4,c}(s,r)\end{cases}\end{equation}
So, $$v_{4,s}(t,r) = \partial_{r}w_{4,s}(t,r)$$
where
$$u_{4,s}(t,y) = w_{4,s}(t,|y|)$$
and $u_{4,s}:[T_{0},\infty) \times \mathbb{R}^{2} \rightarrow \mathbb{R}$ solves
\begin{equation}\begin{cases} \partial_{tt} u_{4,s} - \Delta u_{4,s} = 0\\
u_{4,s}(s,x) =0\\
\partial_{t}u_{4,s}(s,x) = -\int_{|x|}^{\infty} v_{4,c}(s,q) dq\end{cases}\end{equation}
We get, for $s \geq t$,
\begin{equation} u_{4,s}(t,x) = \frac{1}{2\pi} \int_{B_{s-t}(x)} \frac{\left(\int_{|y|}^{\infty} v_{4,c}(s,q) dq\right)}{\sqrt{(s-t)^{2}-|y-x|^{2}}} dy=\frac{1}{2\pi} \int_{B_{s-t}(0)} \frac{\left(\int_{|z+x|}^{\infty} v_{4,c}(s,q) dq\right)}{\sqrt{(s-t)^{2}-|z|^{2}}} dz\end{equation}
We use polar coordinates in the $z$ variable, with origin $0$, and polar axis $\hat{x}$ for $x \neq 0$. Then, we obtain, apriori for $r \neq 0$, 
\begin{equation}\label{v4prelimformula} v_{4}(t,r) = \frac{-1}{2\pi} \int_{t}^{\infty} ds \int_{0}^{s-t} \frac{\rho d\rho}{\sqrt{(s-t)^{2}-\rho^{2}}} \int_{0}^{2\pi} d\theta \frac{v_{4,c}(s,\sqrt{r^{2}+2 r \rho \cos(\theta) + \rho^{2}})}{\sqrt{r^{2}+2 r \rho \cos(\theta) + \rho^{2}}} \left(r+\rho \cos(\theta)\right)\end{equation}
If we let
\begin{equation}\label{gforv4}\begin{split} G(s,r,\rho) &=\int_{0}^{2\pi} d\theta \frac{v_{4,c}(s,\sqrt{r^{2}+2 r \rho \cos(\theta) + \rho^{2}})}{\sqrt{r^{2}+2 r \rho \cos(\theta) + \rho^{2}}} \left(r+\rho \cos(\theta)\right)\\
& s \geq t,\quad r \geq 0,\quad s-t \geq \rho \geq 0\end{split}\end{equation}
Then,
$$G(s,0,\rho)=0$$
and
\begin{equation}\begin{split} G(s,r,\rho) &= r \int_{0}^{1} \partial_{2}G(s,r\beta,\rho) d\beta\\
&=r \int_{0}^{1} d\beta \int_{0}^{2\pi} d\theta \left(\partial_{2}v_{4,c}(s,\sqrt{\beta^{2}r^{2}+2 \beta r \rho \cos(\theta) + \rho^{2}}) \frac{(\beta r+\rho \cos(\theta))^{2}}{\beta^{2} r^{2}+2 \beta r \rho \cos(\theta) + \rho^{2}}\right.\\
&-\left.v_{4,c}(s,\sqrt{\beta^{2} r^{2}+2 \beta r \rho \cos(\theta) + \rho^{2}}) \frac{(\beta r+\rho \cos(\theta))^{2}}{(\beta^{2} r^{2}+2 \beta r \rho \cos(\theta) + \rho^{2})^{3/2}}+\frac{v_{4,c}(s,\sqrt{\beta^{2} r^{2}+2 \beta r \rho \cos(\theta) + \rho^{2}})}{\sqrt{\beta^{2} r^{2}+2 \beta r \rho \cos(\theta) + \rho^{2}}}\right)\end{split}\end{equation}
So, $v_{4}(t,\cdot)$ is (for instance) continuous on $[0,\infty)$ and we have, for all $r \geq 0$ (including $r=0$) 
\begin{equation}\label{v4repformsimp} \begin{split} v_{4}(t,r) &=\frac{-r}{2\pi} \int_{t}^{\infty} ds \int_{0}^{s-t} \frac{\rho d\rho}{\sqrt{(s-t)^{2}-\rho^{2}}} \int_{0}^{1} \partial_{2}G(s,r\beta,\rho) d\beta\end{split}\end{equation}
We will use this formula to prove estimates on $v_{4}$, but this will be done later on, once we futher restrict the class of functions $\lambda$ under consideration (which will be done once we introduce an iteration space in which to solve the eventual equation for $\lambda$).

\subsection{The nonlinear error terms involving $v_{1},v_{2},v_{3},v_{4}$}
Let
$$f_{v_{5}}=v_{1}+v_{2}+v_{3}+v_{4}$$
and 
$$N_{2}(f_{v_{5}})(t,r) = \frac{\sin(2Q_{\frac{1}{\lambda(t)}}(r))}{2r^{2}} \left(\cos(2 f_{v_{5}})-1\right)+\frac{\cos(2Q_{\frac{1}{\lambda(t)}}(r))}{2r^{2}} \left(\sin(2f_{v_{5}})-2f_{v_{5}}\right)$$

Then, we consider $v_{5}$, defined as the solution with $0$ Cauchy data at infinity, to the problem
\begin{equation} -\partial_{tt}v_{5}+\partial_{rr}v_{5}+\frac{1}{r}\partial_{r}v_{5}-\frac{v_{5}}{r^{2}} = N_{2}(f_{v_{5}})(t,r)\end{equation}
Following the same steps used to obtain \eqref{v4repformsimp}, we obtain the analogous formula for $v_{5}$:
\begin{equation} \begin{split} v_{5}(t,r) &=\frac{-r}{2\pi} \int_{t}^{\infty} ds \int_{0}^{s-t} \frac{\rho d\rho}{\sqrt{(s-t)^{2}-\rho^{2}}} \int_{0}^{1} \partial_{2}G_{5}(s,r\beta,\rho) d\beta\end{split}\end{equation}
where
\begin{equation}\begin{split} G_{5}(s,r,\rho) &= r \int_{0}^{1} \partial_{2}G_{5}(s,r\beta,\rho) d\beta\\
&=r \int_{0}^{1} d\beta \int_{0}^{2\pi} d\theta \left(\partial_{2}N_{2}(f_{v_{5}})(s,\sqrt{\beta^{2}r^{2}+2 \beta r \rho \cos(\theta) + \rho^{2}}) \frac{(\beta r+\rho \cos(\theta))^{2}}{\beta^{2} r^{2}+2 \beta r \rho \cos(\theta) + \rho^{2}}\right.\\
&-\left.N_{2}(f_{v_{5}})(s,\sqrt{\beta^{2} r^{2}+2 \beta r \rho \cos(\theta) + \rho^{2}}) \frac{(\beta r+\rho \cos(\theta))^{2}}{(\beta^{2} r^{2}+2 \beta r \rho \cos(\theta) + \rho^{2})^{3/2}}+\frac{N_{2}(f_{v_{5}})(s,\sqrt{\beta^{2} r^{2}+2 \beta r \rho \cos(\theta) + \rho^{2}})}{\sqrt{\beta^{2} r^{2}+2 \beta r \rho \cos(\theta) + \rho^{2}}}\right)\end{split}\end{equation}
As was the case for $v_{4}$, we will prove estimates on $v_{5}$ later on, once we further restrict the class of functions $\lambda$ under consideration.

\subsection{The equation resulting from $u_{ansatz}$}
If we substitute $u=Q_{\frac{1}{\lambda(t)}}+v_{1}+v_{2}+v_{3}+v_{4}+v_{5}+v_{6}$ into the wave maps equation $$-\partial_{tt}u+\partial_{rr}u+\frac{1}{r}\partial_{r}u-\frac{\sin(2 u)}{2r^{2}} =0$$
and use the equations solved by $v_{1},v_{2},v_{3},v_{4},v_{5}$, then, we obtain
\begin{equation}\label{v6eqn}\begin{split} &-\partial_{tt}v_{6}+\partial_{rr}v_{6}+\frac{1}{r}\partial_{r}v_{6}-\frac{\cos(2Q_{\frac{1}{\lambda(t)}})}{r^{2}} v_{6}\\
&=F_{0,2}(t,r) + N(v_{6})+L_{1}(v_{6})+N_{2}(v_{5})\\
&+\frac{\sin(2(v_{1}+v_{2}+v_{3}+v_{4}))}{2r^{2}}\left(\cos(2Q_{\frac{1}{\lambda(t)}}+2v_{5})-\cos(2Q_{\frac{1}{\lambda(t)}})\right)\\
&+\left(\frac{\cos(2(v_{1}+v_{2}+v_{3}+v_{4}))-1}{2r^{2}}\right)\left(\sin(2Q_{\frac{1}{\lambda(t)}}+2v_{5})-\sin(2Q_{\frac{1}{\lambda(t)}})\right)\\
&+\left(\frac{\cos(2Q_{\frac{1}{\lambda(t)}})-1}{r^{2}}\right)\left(v_{1}+v_{2}+v_{3}+v_{4}+v_{5}\right)\\
&-\chi_{\geq 1}(\frac{2r}{\log^{N}(t)}) \left(\frac{\cos(2Q_{\frac{1}{\lambda(t)}})-1}{r^{2}}\right)\left(v_{1}+v_{2}+v_{3}\right)-\chi_{\geq 1}(\frac{2r}{\log^{N}(t)}) F_{0,2}(t,r)\end{split}\end{equation}

where $$N(f) = \left(\frac{\sin(2f)-2f}{2r^{2}}\right)\cos(2Q_{\frac{1}{\lambda(t)}}) + \left(\frac{\cos(2f)-1}{2r^{2}}\right)\sin(2(Q_{\frac{1}{\lambda(t)}}+v_{1}+v_{2}+v_{3}+v_{4}+v_{5}))$$
$$L_{1}(f)=\frac{\sin(2f)}{2r^{2}} \cos(2Q_{\frac{1}{\lambda(t)}})(\cos(2(v_{1}+v_{2}+v_{3}+v_{4}+v_{5}))-1) -\frac{\sin(2f)}{2r^{2}}\sin(2Q_{\frac{1}{\lambda(t)}})\sin(2(v_{1}+v_{2}+v_{3}+v_{4}+v_{5}))$$
$$N_{2}(f) = \frac{\sin(2Q_{\frac{1}{\lambda(t)}})}{2r^{2}}(\cos(2f)-1)+\frac{\cos(2Q_{\frac{1}{\lambda(t)}})}{2r^{2}}(\sin(2f)-2f)$$
$$F_{0,2}(t,r) = -F_{0,1}(t,r)+\frac{2 \lambda''(t) r}{1+r^{2}}+\partial_{t}^{2}Q_{\frac{1}{\lambda(t)}}$$
Note that we do not combine the terms involving $\chi_{\geq 1}$ with analogous terms having coefficient $1$ because the terms involving $\chi_{\geq 1}$ will turn out to have a subleading contribution to the modulation equation for $\lambda$. When we solve the final equation, after choosing $\lambda$, we will of course make use of the fact that $1-\chi_{\geq 1}(x)$ is supported on the set $x \leq 2$.\\
We can re-write the right-hand side of \eqref{v6eqn} as $F+F_{3}$
where
$$F=F_{4}+F_{5}+F_{6}$$
with 
\begin{equation}\label{f4def}\begin{split} F_{4}(t,r) &=F_{0,2}(t,r)+\left(\frac{\cos(2Q_{\frac{1}{\lambda(t)}})-1}{r^{2}}\right)\left(v_{1}+v_{2}+v_{3}+\left(1-\chi_{\geq 1}(\frac{4 r}{t})\right)(v_{4}+v_{5})\right)\\
&-\chi_{\geq 1}(\frac{2r}{\log^{N}(t)}) \left(\frac{\cos(2Q_{\frac{1}{\lambda(t)}})-1}{r^{2}}\right)\left(v_{1}+v_{2}+v_{3}\right)-\chi_{\geq 1}(\frac{2r}{\log^{N}(t)}) F_{0,2}(t,r)\end{split}\end{equation}
\begin{equation}\label{f5def}\begin{split} F_{5}(t,r) &= N_{2}(v_{5}) + \frac{\sin(2(v_{1}+v_{2}+v_{3}+v_{4}))}{2r^{2}} \left(\cos(2Q_{\frac{1}{\lambda(t)}}+2v_{5})-\cos(2Q_{\frac{1}{\lambda(t)}})\right)\\
&+\left(\frac{\cos(2(v_{1}+v_{2}+v_{3}+v_{4}))-1}{2r^{2}}\right)\left(\sin(2Q_{\frac{1}{\lambda(t)}}+2v_{5})-\sin(2Q_{\frac{1}{\lambda(t)}})\right)\end{split}\end{equation}
\begin{equation}\label{f6def}F_{6}(t,r) = \chi_{\geq 1}(\frac{4r}{t})\left(\frac{\cos(2Q_{\frac{1}{\lambda(t)}})-1}{r^{2}}\right)\left(v_{4}(t,r)+v_{5}(t,r)\right)\end{equation} 
and $$F_{3} = N(v_{6})+L_{1}(v_{6})$$

\subsection{Choosing $\lambda(t)$}
$\lambda$ will be chosen so that the term
\begin{equation}\begin{split}\label{mdef}F_{4}(t,r) &= F_{0,2}(t,r) + \left(\frac{\cos(2Q_{\frac{1}{\lambda(t)}}(r))-1}{r^{2}}\right) \left(v_{1}+v_{2}+v_{3}+\left(1-\chi_{\geq 1}(\frac{4 r}{t})\right)(v_{4}+v_{5})\right)\\
&-\chi_{\geq 1}(\frac{2r}{\log^{N}(t)})\left(F_{0,2}(t,r)+ \left(v_{1}(t,r) + v_{2}(t,r) + v_{3}(t,r)\right) \left(\frac{\cos(2Q_{\frac{1}{\lambda(t)}}(r))-1}{r^{2}}\right)\right)\end{split}\end{equation}
which appears on the right-hand side of \eqref{v6eqn}, is orthogonal to $\phi_{0}\left(\frac{\cdot}{\lambda(t)}\right)$. \\
\\
Define the space $(X,||\cdot||_{X})$ by
$$X=\{f \in C^{2}([T_{0},\infty))| ||f||_{X}< \infty\}$$
where
\begin{equation}\label{ynorm}||f||_{X} = \text{sup}_{t \geq T_{0}} \left(|f(t)| b \log^{b}(t) \sqrt{\log(\log(t))} + |f'(t)| t \log^{b+1}(t) \sqrt{\log(\log(t))} +|f''(t)| t^{2} \log^{b+1}(t) \sqrt{\log(\log(t))}\right)\end{equation}
In this section, we will first prove the following proposition
\begin{proposition}\label{lambdaexist} There exists $T_{3}>0$ such that, for all $T_{0} \geq T_{3}$, there exists $\lambda \in C^{2}([T_{0},\infty))$ which solves 
$$\langle F_{4}(t),\phi_{0}\left(\frac{\cdot}{\lambda(t)}\right)\rangle =0, \quad t \geq T_{0}$$ 
Moreover, 
\begin{equation} \lambda(t) = \lambda_{0}(t) + e_{0}(t), \quad ||e_{0}||_{X} \leq 1\end{equation}
and
$$ \lambda_{0}(t) = \frac{1}{\log^{b}(t)}+ \int_{t}^{\infty} \int_{t_{1}}^{\infty} \frac{-b^{2} \log(\log(t_{2}))}{t_{2}^{2}\log^{b+2}(t_{2})}dt_{2}dt_{1}$$
\end{proposition} 
After fixing $\lambda$ as in the above proposition, we will then show that $\lambda$, apriori only in $C^{2}([T_{0},\infty))$, is actually in $C^{4}([T_{0},\infty))$, with the estimates
\begin{equation} |\lambda^{(2+k)}(t)| \leq \frac{C}{t^{2+k} \log^{b+1}(t)}, \quad t \geq T_{0}, \quad k=1,2\end{equation}
To prove the proposition, we will first show that the equation
$$\langle F_{4}(t),\phi_{0}\left(\frac{\cdot}{\lambda(t)}\right)\rangle =0$$
is equivalent to
\begin{equation}\label{modulationfinal1}\begin{split} &-4 \int_{t}^{\infty} \frac{\lambda''(s)}{1+s-t} ds + \frac{4 b}{t^{2}\log^{b}(t)} + 4 \alpha \log(\lambda(t)) \lambda''(t) - 4 \int_{t}^{\infty} \frac{\lambda''(s)}{(\lambda(t)^{1-\alpha}+s-t)(1+s-t)^{3}} ds\\
&= -\lambda(t) E_{0,1}(\lambda(t),\lambda'(t),\lambda''(t)) - 16 \int_{t}^{\infty} \lambda''(s) \left(K_{3}(s-t,\lambda(t))-K_{3,0}(s-t,\lambda(t))\right) ds\\
&+\frac{16}{\lambda(t)^{2}} \int_{t}^{\infty} K(s-t,\lambda(t)) \lambda''(s) ds - \lambda(t) E_{v_{2},ip}(t,\lambda(t)) +\frac{16}{\lambda(t)^{2}} \int_{t}^{\infty} ds \lambda''(s) \left(K_{1}(s-t,\lambda(t))-\frac{\lambda(t)^{2}}{4(1+s-t)}\right)\\
&- \lambda(t) \langle \left(\frac{\cos(2Q_{\frac{1}{\lambda(t)}})-1}{r^{2}}\right)\left((v_{4}+v_{5})\left(1-\chi_{\geq 1}(\frac{4r}{t})\right)+E_{5}-\chi_{\geq 1}(\frac{2r}{\log^{N}(t)})\left(v_{1}+v_{2}+v_{3}\right)\right)\vert_{r=R\lambda(t)},\phi_{0}\rangle\\
&+\lambda(t) \langle \chi_{\geq 1}(\frac{2r}{\log^{N}(t)}) F_{0,2}(t,r) \vert_{r=R\lambda(t)},\phi_{0}\rangle\\
&:=G(t,\lambda(t))\end{split}\end{equation}
Then, we will substitute $\lambda(t) = \lambda_{0}(t)+e_{0}(t)$, for $e_{0} \in \overline{B_{1}(0)} \subset X$, and solve the resulting equation for $e_{0}$ with a fixed point argument. \\
\\
We start by studying the relevant inner products of the $v_{1},v_{2}$ and $v_{3}$ terms above.

\subsubsection{The inner product of the (rescaled) $v_{1}$ linear error term with $\phi_{0}$}

In this section, we will prove
\begin{lemma} For $v_{1}$ defined by \eqref{v1formula}, we have
\begin{equation} \label{innerproductkernelrepresentation} \begin{split} \langle \left(\frac{\cos(2Q_{\frac{1}{\lambda(t)}})-1}{r^{2}}\right)v_{1}\vert_{r=R\lambda(t)},\phi_{0}\rangle &= \frac{-16}{\lambda(t)^{3}} \int_{t}^{\infty} ds \lambda''(s) K(s-t,\lambda(t)) + \frac{-16}{\lambda(t)^{3}} \int_{t}^{\infty} ds \lambda''(s) K_{1}(s-t,\lambda(t))\end{split}\end{equation}
where
\begin{equation} \int_{t}^{\infty} |K(s-t,\lambda(t))| ds \leq C \lambda(t)^{2}\end{equation}
\begin{equation} |K_{1}(x,\lambda(t))| \leq \frac{C \lambda(t)^{2} x}{1+x^{2}}\end{equation}
\begin{equation} \label{k1diffint} \int_{t}^{\infty} |K_{1}(s-t,\lambda(t))-\frac{\lambda(t)^{2}}{4(1+s-t)}| ds \leq C \lambda(t)^{2}\end{equation}
\begin{equation} |K_{1}(s-t,\lambda(t))-\frac{\lambda(t)^{2}}{4(s-t)}| \leq \frac{C \lambda(t)^{2} (1+\lambda(t)^{2})}{(s-t) (1+(s-t)^{2})}, \quad s-t \geq 1\end{equation}
\end{lemma}
\begin{proof}

We have
\begin{equation}\label{innerproduct}\begin{split}\langle\left.\frac{(\cos(2Q_{\frac{1}{\lambda(t)}})-1)}{r^{2}} v_{1}\right\vert_{r=R \lambda(t)},\phi_{0}\rangle&=\frac{-1}{\lambda(t)^{2}}\int_{0}^{\infty}v_{1}(t,R \lambda(t)) \frac{8}{(1+R^{2})^{2}}\frac{2R}{1+R^{2}} RdR\\
&=-\frac{16}{\lambda(t)^{2}} \int_{0}^{\infty}  \text{  } v_{1}(t, R\lambda(t))\frac{R^{2}}{(1+R^{2})^{3}}dR
\end{split}
\end{equation}
\begin{equation}\label{initialsplitting} \begin{split} &-\frac{16}{\lambda(t)^{2}} \int_{0}^{\infty}  v_{1}(t, R\lambda(t))\frac{R^{2}}{(1+R^{2})^{3}}dR\\
&=-\frac{16}{\lambda(t)^{3}} \int_{t}^{\infty} ds \lambda''(s) \int_{0}^{\infty} dR \frac{R}{(1+R^{2})^{3}} \int_{0}^{s-t} \frac{\rho d\rho}{\sqrt{(s-t)^{2}-\rho^{2}}} \left(1+\frac{R^{2}\lambda(t)^{2}-1-\rho^{2}}{\sqrt{(R^{2}\lambda(t)^{2}-1-\rho^{2})^{2}+4R^{2}\lambda(t)^{2}}}\right)\\
&=-\frac{16}{\lambda(t)^{3}} \int_{t}^{\infty} ds \lambda''(s) \int_{0}^{\infty} dR \frac{R}{(1+R^{2})^{3}} \int_{0}^{s-t} \frac{\rho d\rho}{s-t} \left(1+\frac{R^{2}\lambda(t)^{2}-1-\rho^{2}}{\sqrt{(R^{2}\lambda(t)^{2}-1-\rho^{2})^{2}+4R^{2}\lambda(t)^{2}}}\right)\\
&-\frac{16}{\lambda(t)^{3}} \int_{t}^{\infty} ds \lambda''(s) \int_{0}^{\infty} dR \frac{R}{(1+R^{2})^{3}} \int_{0}^{s-t} \rho d\rho\left(\frac{1}{\sqrt{(s-t)^{2}-\rho^{2}}}-\frac{1}{s-t}\right) \left(1+\frac{R^{2}\lambda(t)^{2}-1-\rho^{2}}{\sqrt{(R^{2}\lambda(t)^{2}-1-\rho^{2})^{2}+4R^{2}\lambda(t)^{2}}}\right)
\end{split}
\end{equation}
The last line of \eqref{initialsplitting} is of the form
\begin{equation} -\frac{16}{\lambda(t)^{3}} \int_{t}^{\infty} ds \lambda''(s) K(s-t,\lambda(t))\end{equation}
where
\begin{equation}\label{kitself}K(x,\lambda(t)) = \int_{0}^{\infty} dR \frac{R}{(1+R^{2})^{3}} \int_{0}^{x} \rho d\rho\left(\frac{1}{\sqrt{x^{2}-\rho^{2}}}-\frac{1}{x}\right) \left(1+\frac{R^{2}\lambda(t)^{2}-1-\rho^{2}}{\sqrt{(R^{2}\lambda(t)^{2}-1-\rho^{2})^{2}+4R^{2}\lambda(t)^{2}}}\right) \geq 0\end{equation}
and
\begin{equation}\label{kintegralestimate} \begin{split} &\int_{t}^{\infty}  K(s-t,\lambda(t)) ds \\
&= \int_{t}^{\infty} ds \int_{0}^{\infty} dR \frac{R}{(1+R^{2})^{3}} \int_{0}^{s-t} \rho d\rho\left(\frac{1}{\sqrt{(s-t)^{2}-\rho^{2}}}-\frac{1}{s-t}\right) \left(1+\frac{R^{2}\lambda(t)^{2}-1-\rho^{2}}{\sqrt{(R^{2}\lambda(t)^{2}-1-\rho^{2})^{2}+4R^{2}\lambda(t)^{2}}}\right)\\
&=\int_{0}^{\infty} \rho d\rho \int_{0}^{\infty} dR \frac{R}{(1+R^{2})^{3}} \int_{\rho + t}^{\infty} ds \left(\frac{1}{\sqrt{(s-t)^{2}-\rho^{2}}}-\frac{1}{s-t}\right)\left(1+\frac{R^{2}\lambda(t)^{2}-1-\rho^{2}}{\sqrt{(R^{2}\lambda(t)^{2}-1-\rho^{2})^{2}+4R^{2}\lambda(t)^{2}}}\right)\\
&=\int_{0}^{\infty} \rho d\rho \int_{0}^{\infty} dR \frac{R}{(1+R^{2})^{3}} \left(1+\frac{R^{2}\lambda(t)^{2}-1-\rho^{2}}{\sqrt{(R^{2}\lambda(t)^{2}-1-\rho^{2})^{2}+4R^{2}\lambda(t)^{2}}}\right)\log(2)\\
&=\int_{0}^{\infty} dR \frac{R}{(1+R^{2})^{3}} \log(2) R^{2} \lambda(t)^{2}\\
&=\frac{\log(2)}{4} \lambda(t)^{2}
\end{split}
\end{equation}

Hence, it remains to calculate 
\begin{equation} \label{innerprod0}\begin{split} &-\frac{16}{\lambda(t)^{3}} \int_{t}^{\infty} ds \lambda''(s) \int_{0}^{\infty} dR \frac{R}{(1+R^{2})^{3}} \int_{0}^{s-t} \frac{\rho d\rho}{s-t} \left(1+\frac{R^{2}\lambda(t)^{2}-1-\rho^{2}}{\sqrt{(R^{2}\lambda(t)^{2}-1-\rho^{2})^{2}+4R^{2}\lambda(t)^{2}}}\right)\\
&=-\frac{16}{\lambda(t)^{3}} \int_{t}^{\infty} ds \lambda''(s) \int_{0}^{\infty} dR \frac{R}{(1+R^{2})^{3}} \left(\frac{r^2-\sqrt{\left((r+s-t)^2+1\right) \left((r-s+t)^2+1\right)}+(s-t)^2+1}{2 (s-t)}\right)\vert_{r=R\lambda(t)}\\
&=-\frac{16}{\lambda(t)^{3}} \int_{t}^{\infty} ds \lambda''(s) K_{1}(s-t,\lambda(t))
\end{split}
\end{equation}
where 
\begin{equation} \label{k1} \begin{split} K_{1}(w,\lambda(t))&=\int_{0}^{\infty} \frac{R dR}{(1+R^{2})^{3}} \left(\frac{R^{2} \lambda(t)^{2}+w^{2}+1-\sqrt{(1+R^{2}\lambda(t)^{2}+w^{2})^{2}-4R^{2}\lambda(t)^{2}w^{2}}}{2w}\right)\\
&=\frac{w \lambda(t)^2 \left(\lambda(t)^2+w^2+1\right)}{4 (y^{2}+4w^{2})}+\frac{w \lambda(t)^4 \log \left(\lambda(t)^2 \left(\sqrt{y^{2}+4w^{2}}-y\right)\right)}{2 \left(y^{2}+4w^{2}\right)^{3/2}}\\
&-w \lambda(t)^{4}\frac{\log \left(4w^{2}+y^{2}-\lambda(t)^{2}y+\left(w^2+1\right) \sqrt{y^{2}+4w^{2}}\right)}{2 \left(y^{2}+4w^{2}\right)^{3/2}}
\end{split}
\end{equation}
where
\begin{equation} y=\lambda(t)^2+w^2-1\end{equation} 
and, for the first line of \eqref{k1}, we used the fact that
\begin{equation} \label{algsimp}((r+w)^{2}+1)((r-w)^{2}+1) = (r^{2}+w^{2}+1)^{2}-4r^{2}w^{2}\end{equation}
Now, we will prove a pointwise estimate on $K_{1}$. Using \eqref{algsimp}, we have
\begin{equation}\begin{split} 0 \leq  K_{1}(w,\lambda(t)) &= \int_{0}^{\infty} \frac{R dR}{2 w (1+R^{2})^{3}} \left(\frac{4 R^{2}\lambda(t)^{2} w^{2}}{1+R^{2}\lambda(t)^{2}+w^{2}+\sqrt{(1+(R \lambda(t)+w)^{2})(1+(R\lambda(t)-w)^{2})}}\right)\\
&\leq C \int_{0}^{\infty} \frac{R^{3} dR}{(1+R^{2})^{3}} \frac{\lambda(t)^{2} w}{1+w^{2}} \leq C \lambda(t)^{2} \frac{w}{1+w^{2}}\end{split}\end{equation}
So,
\begin{equation}\label{k1ptwseest}|K_{1}(x,\lambda(t))| \leq C \lambda(t)^{2} \frac{x}{1+x^{2}} \end{equation}
For use later on, we will also need to estimate
\begin{equation} \int_{t}^{\infty} ds |K_{1}(s-t,\lambda(t))-\frac{\lambda(t)^{2}}{4(1+s-t)}|\end{equation}
Note that $K_{1}(w,\lambda(t)) \geq 0$, by its definition as the integral of a non-negative function. We start with
\begin{equation}\begin{split} \int_{t}^{t+1} ds |K_{1}(s-t,\lambda(t))-\frac{\lambda(t)^{2}}{4(1+s-t)}|&\leq \int_{t}^{t+1} ds K_{1}(s-t,\lambda(t) + \int_{t}^{t+1} ds \frac{\lambda(t)^{2}}{4(1+s-t)}\\
&\leq C \lambda(t)^{2} \end{split}\end{equation}
where we used \eqref{k1ptwseest}. \\
\\
Now, we consider the region $s-t \geq 1$. Returning to \eqref{k1}, we see that
\begin{equation}\label{k1intstep} \begin{split} K_{1}(w,\lambda(t)) &= \int_{0}^{\infty} dR \frac{R}{(1+R^{2})^{3}} \frac{R^{2}\lambda(t)^{2}}{w} \\
&- \int_{0}^{\infty} dR \frac{R}{(1+R^{2})^{3}} \frac{4 R^{2}\lambda(t)^{2}}{2 w \left(\sqrt{\left(-R^{2}\lambda(t)^{2}+w^2+1\right)^2+4 R^{2}\lambda(t)^{2}}-R^{2}\lambda(t)^{2}+w^2+1\right)}\end{split}\end{equation}
The right-hand side of the first line of \eqref{k1intstep} is equal to
$$\int_{0}^{\infty} dR \frac{R}{(1+R^{2})^{3}} \frac{R^{2}\lambda(t)^{2}}{w} = \frac{\lambda(t)^{2}}{4w}$$
So,
\begin{equation} K_{1}(w,\lambda(t))-\frac{\lambda(t)^{2}}{4 w} = -\int_{0}^{\infty} \frac{R dR}{(1+R^{2})^{3}} \frac{4 R^{2}\lambda(t)^{2}}{2w\left(\sqrt{(1+w^{2}-R^{2}\lambda(t)^{2})^{2}+4R^{2}\lambda(t)^{2}}+1+w^{2}-R^{2}\lambda(t)^{2}\right)}\end{equation}
First, we note that
$$1+w^{2}-R^{2}\lambda(t)^{2} \geq C(1+w^{2}), \quad\text{ if } R\lambda(t) \leq \frac{w}{2}$$
So,
\begin{equation}\begin{split} &\int_{0}^{\frac{w}{2\lambda(t)}} \frac{R dR}{(1+R^{2})^{3}} \frac{4 R^{2}\lambda(t)^{2}}{2w\left(\sqrt{(1+w^{2}-R^{2}\lambda(t)^{2})^{2}+4R^{2}\lambda(t)^{2}}+1+w^{2}-R^{2}\lambda(t)^{2}\right)} \\
&\leq C \int_{0}^{\infty} dR \frac{R}{(1+R^{2})^{3}} \frac{R^{2} \lambda(t)^{2}}{w(1+w^{2})}\leq \frac{C \lambda(t)^{2}}{w (1+w^{2})}, \quad w \geq 1\end{split}\end{equation}
Next,
\begin{equation} \frac{1}{\sqrt{(1+w^{2}-R^{2}\lambda(t)^{2})^{2}+4R^{2}\lambda(t)^{2}}+1+w^{2}-R^{2}\lambda(t)^{2}} = \frac{1+w^{2}-R^{2}\lambda(t)^{2} - \sqrt{(1+w^{2}-R^{2}\lambda(t)^{2})^{2}+4R^{2}\lambda(t)^{2}}}{-4R^{2}\lambda(t)^{2}}\end{equation}
and
\begin{equation}\begin{split}&|\frac{1+w^{2}-R^{2}\lambda(t)^{2} - \sqrt{(1+w^{2}-R^{2}\lambda(t)^{2})^{2}+4R^{2}\lambda(t)^{2}}}{-4R^{2}\lambda(t)^{2}}| \leq C \frac{R^{2}\lambda(t)^{2}}{R^{2}\lambda(t)^{2}} \leq C, \\
&\text{ if} \quad R \lambda(t) > \frac{w}{2}, \quad w \geq 1 \end{split}\end{equation}

So, \begin{equation}\begin{split}&\int_{\frac{w}{2\lambda(t)}}^{\infty} \frac{R dR}{(1+R^{2})^{3}} \frac{4 R^{2}\lambda(t)^{2}}{2w\left(\sqrt{(1+w^{2}-R^{2}\lambda(t)^{2})^{2}+4R^{2}\lambda(t)^{2}}+1+w^{2}-R^{2}\lambda(t)^{2}\right)} \\
&\leq \int_{\frac{w}{2\lambda(t)}}^{\infty} \frac{dR}{R^{5}} \frac{R^{2} \lambda(t)^{2}}{w} \leq \frac{C\lambda(t)^{4}}{w^{3}}, \quad w \geq 1\end{split}\end{equation}
Combining these, we get
\begin{equation} \label{k1diffptwse} |K_{1}(w,\lambda(t))-\frac{\lambda(t)^{2}}{4w}| \leq \frac{C \lambda(t)^{2}(1+\lambda(t)^{2})}{w(1+w^{2})}, \quad w \geq 1\end{equation}
So,
\begin{equation}\label{k1intstep2}\begin{split} \int_{t+1}^{\infty} ds |K_{1}(s-t,\lambda(t))-\frac{\lambda(t)^{2}}{4(1+s-t)}| &\leq C \int_{t+1}^{\infty} ds \frac{\lambda(t)^{2}}{4}\left(\frac{1}{s-t}-\frac{1}{1+s-t}\right) + C \int_{t+1}^{\infty}ds \frac{C \lambda(t)^{2}(1+\lambda(t)^{2})}{(s-t)(1+(s-t)^{2})}\\
&\leq C \lambda(t)^{2}(1+\lambda(t)^{2})\end{split}\end{equation}
Recalling \eqref{lambdarestr}, we conclude
\begin{equation} \int_{t}^{\infty} ds |K_{1}(s-t,\lambda(t))-\frac{\lambda(t)^{2}}{4(1+s-t)}| \leq C \lambda(t)^{2}\end{equation}
 \end{proof}

\subsubsection{The inner product of the (rescaled) $v_{2}$ linear error term with $\phi_{0}$}
\begin{lemma} For all $b>0$, and $v_{2}$ defined by \eqref{v2def}, we have
\begin{equation}\label{v2linearinnerproduct} \int_{0}^{\infty} R dR \left(\frac{\cos(2Q_{1}(R))-1}{R^{2}\lambda(t)^{2}}\right)\phi_{0}(R) v_{2}(t,R\lambda(t)) = \frac{4b}{\lambda(t) t^{2} \log^{b}(t)} +E_{v_{2},ip}(t,\lambda(t))\end{equation}
where
\begin{equation} \label{ev2ipest}|E_{v_{2},ip}(t,\lambda(t))| \leq \frac{C}{\lambda(t) t^{2}\log^{b+1}(t)}\end{equation}
\end{lemma}
\begin{proof}
We start with the case $b \neq 1$. Using our formula for $v_{2}$, we get
 \begin{equation}\label{v2innerp}\int_{0}^{\infty} R dR \left(\frac{\cos(2Q_{1}(R))-1}{R^{2} \lambda(t)^{2}}\right) \phi_{0}(R) v_{2}(t,R\lambda(t)) =-2c_{b}\int_{0}^{\infty}\xi^{2} d\xi \sin(t \xi)\frac{\chi_{\leq \frac{1}{4}}(\xi)}{\log^{b-1}(\frac{1}{\xi})} K_{1}(\xi\lambda(t))\end{equation}
where $K_{1}$ denotes the modified Bessel function of the second kind. (This follows, for instance, from equation (\textbf{6.532} 4) of the table of integrals \cite{gr}). Recalling that $$K_{1}(x) = \frac{1}{x} +O(x \log(x)), \quad x \rightarrow 0$$
we can integrate by parts two times, and get
$$-2 c_{b} \int_{0}^{\infty} \xi^{2} d\xi \sin(t\xi) \frac{\chi_{\leq \frac{1}{4}}(\xi)}{\log^{b-1}(\frac{1}{\xi})} K_{1}(\xi \lambda(t)) = 2 c_{b} \int_{0}^{\infty} d\xi \frac{\sin(t \xi)}{t^{2}} \partial_{\xi}^{2} \left(\xi^{2} \frac{\chi_{\leq \frac{1}{4}}(\xi)}{\log^{b-1}(\frac{1}{\xi})} K_{1}(\xi \lambda(t))\right) $$
Let $$F_{b}(\xi) = \partial_{\xi}^{2}\left(\xi^{2} \frac{\chi_{\leq \frac{1}{4}}(\xi)}{\log^{b-1}(\frac{1}{\xi})} K_{1}(\xi \lambda(t))\right)$$
Note that \begin{equation}\label{psidef2}F_{b}(\xi) = \chi_{\leq \frac{1}{4}}(\xi) \partial_{\xi}^{2}\left(\xi^{2}\frac{K_{1}(\xi \lambda(t))}{\log^{b-1}(\frac{1}{\xi})}\right)+\psi_{v_{2}}(\xi,\lambda(t))\end{equation}
where $$\psi_{v_{2}} \in C^{\infty}_{c}([\frac{1}{8},\frac{1}{4}])$$
So, integration by parts (for instance, once) gives
$$|2 c_{b} \int_{0}^{\infty} d\xi \frac{\sin(t\xi)}{t^{2}} \psi_{v_{2}}(\xi,\lambda(t))| \leq \frac{C}{t^{3}\lambda(t)}$$
To study $$2 c_{b} \int_{0}^{\infty} d\xi \chi_{\leq \frac{1}{4}}(\xi) \frac{\sin(t\xi)}{t^{2}} \partial_{\xi}^{2}\left(\xi^{2}\frac{K_{1}(\xi \lambda(t))}{\log^{b-1}(\frac{1}{\xi})}\right)$$
we can use the asymptotics of $K_{1}$ to write
\begin{equation}\label{Fdef2}\partial_{\xi}^{2}\left(\xi^{2}\frac{K_{1}(\xi\lambda(t))}{\log^{b-1}(\frac{1}{\xi})}\right)= \partial_{\xi}^{2}\left(\frac{\xi}{\lambda(t)\log^{b-1}(\frac{1}{\xi})}\right) + F_{v_{2}}(\xi,\lambda(t)), \quad \xi \leq \frac{1}{4}\end{equation}
where
\begin{equation}\begin{split}F_{v_{2}}(\xi,y) &= \frac{\xi ^2 y^2 \log \left(\frac{1}{\xi }\right) \left(-3 \log \left(\frac{1}{\xi }\right)-2 b+2\right) K_0(y \xi )+(b-1) \left(\log \left(\frac{1}{\xi }\right)+b\right) (\xi  y K_1(y \xi )-1)}{\xi y \log^{b+1}(\frac{1}{\xi})}\\
&+\frac{\xi ^3 y^3 \log ^2\left(\frac{1}{\xi }\right) K_1(y \xi )}{\log^{b+1}(\frac{1}{\xi})\xi  y}, \quad \xi \leq \frac{1}{4}\end{split}\end{equation}
Using the simple estimates, valid for all $x >0$,
\begin{equation}\label{k1simpest}\begin{split} |-1+x K_{1}(x)| &\leq C x^{2}(|\log(x)|+1)\\
|K_{0}(x)| &\leq C(|\log(x)|+1)\\
|x K_{1}(x)| &\leq C\end{split}\end{equation}
we get, for any $\lambda(t) >0$,
$$|F_{v_{2}}(\xi,\lambda(t))| \leq C \frac{\xi \lambda(t)}{\log^{b-1}(\frac{1}{\xi})}\left(|\log(\xi)|+|\log(\lambda(t))|\right), \quad \xi \leq \frac{1}{4}$$
Similarly, we have
\begin{equation}\begin{split}\partial_{\xi}F_{v_{2}}(\xi,y) &= \frac{(b-1) \left(-\log ^2\left(\frac{1}{\xi }\right)+b^2+b\right) \log ^{-b-2}\left(\frac{1}{\xi }\right) (\xi  y K_1(y \xi )-1)}{\xi ^2 y}\\
&+y  \frac{\left(\xi  y \log \left(\frac{1}{\xi }\right) \left(4 \log \left(\frac{1}{\xi }\right)+3 b-3\right) K_1(y \xi )\right)}{\log ^{b+1}\left(\frac{1}{\xi }\right)}\\
&-y K_0(y \xi ) \frac{\left(6 (b-1) \log \left(\frac{1}{\xi }\right)+3 (b-1) b+\log ^2\left(\frac{1}{\xi }\right) \left(\xi ^2 y^2+3\right)\right)}{\log^{b+1}(\frac{1}{\xi})}\end{split}\end{equation}
Again, using \eqref{k1simpest}, we get, for all $\lambda(t)>0$,
$$|\partial_{\xi}F_{v_{2}}(\xi,\lambda(t))| \leq \frac{C \lambda(t) \left(|\log(\xi)|+|\log(\lambda(t))|\right)}{\log^{b-1}(\frac{1}{\xi})}\left(1+\lambda(t)^{2} \xi^{2}\right), \quad \xi \leq \frac{1}{4}$$
So, we can integrate by parts (for instance) once, and recall \eqref{lambdarestr}, to get
$$|2 c_{b}\int_{0}^{\infty} d\xi \chi_{\leq \frac{1}{4}}(\xi) \frac{\sin(t\xi)}{t^{2}} F_{v_{2}}(\xi,\lambda(t))| \leq \frac{C \lambda(t) |\log(\lambda(t))|}{t^{3}}$$
Next, we need to consider
$$2c_{b}\int_{0}^{\infty} d\xi \chi_{\leq\frac{1}{4}}(\xi) \frac{\sin(t \xi)}{t^{2}} \partial_{\xi}^{2}\left(\frac{\xi}{\lambda(t)\log^{b-1}(\frac{1}{\xi})}\right)=\frac{2 c_{b}}{\lambda(t)}\int_{0}^{1/2} d\xi \chi_{\leq \frac{1}{4}}(\xi) \frac{\sin(t\xi)}{t^{2}} \left(\frac{b-1}{\xi\log^{b}(\frac{1}{\xi})}+\frac{b(b-1)}{\xi\log^{b+1}(\frac{1}{\xi})}\right)$$
From one integration by parts we see that $$|\frac{2 c_{b}}{\lambda(t)} \int_{0}^{1/2} d\xi \left(\chi_{\leq \frac{1}{4}}(\xi)-1\right) \frac{\sin(t \xi)}{t^{2}}\left(\frac{b-1}{\xi \log^{b}(\frac{1}{\xi})}+\frac{b(b-1)}{\xi\log^{b+1}(\frac{1}{\xi})}\right)| \leq \frac{C}{t^{3} \lambda(t)}$$
So, we need only consider
\begin{equation}\label{innerprodlogs} \frac{2 c_{b}}{\lambda(t)} \int_{0}^{1/2} d\xi \frac{\sin(t \xi)}{t^{2}}\left(\frac{b-1}{\xi \log^{b}(\frac{1}{\xi})}+\frac{b(b-1)}{\xi\log^{b+1}(\frac{1}{\xi})}\right)= \frac{2 c_{b}}{\lambda(t)} \int_{0}^{t/2} du \frac{\sin(u)}{t^{2}}\left(\frac{b-1}{u \log^{b}(\frac{t}{u})}+\frac{b(b-1)}{u\log^{b+1}(\frac{t}{u})}\right)\end{equation}
Let us start by studying, for $a>0$,
$$\int_{0}^{t/2} du \frac{\sin(u)}{u \log^{a}(t/u)}$$
Let$$f(z) = \frac{e^{iz}}{z\log^{a}(t/z)}$$
where we use the principal branch of $\log$ and $(\cdot)^{a}$ Then, $f$ is analytic on (for instance) $D$ given by $$D =\mathbb{C}\setminus ((-\infty,0] \cup [2t/3,\infty))$$
For $0 < \epsilon \leq \frac{1}{2}$, consider the contour $C_{t,\epsilon}$ in $D$ given by:
$$C_{t,\epsilon} = \{\frac{t}{2} e^{i \theta}| 0\leq \theta \leq \frac{\pi}{2}\} \cup \{i y| \epsilon \leq y \leq \frac{t}{2}\} \cup \{\epsilon e^{i \theta} | 0 \leq \theta \leq \frac{\pi}{2}\} \cup [\epsilon,t/2]$$
traversed in the counter-clockwise direction. By Cauchy's residue theorem, $$\int_{C_{t,\epsilon}} f(z) dz =0$$
On the other hand, 
$$\int_{C_{t,\epsilon}} f(z) dz = \int_{0}^{\frac{\pi}{2}} i d\theta \frac{e^{\frac{i t}{2} \cos(\theta)} e^{-\frac{t}{2} \sin(\theta)}}{\log^{a}(2 e^{-i \theta})}-\int_{\epsilon}^{t/2} dy \frac{e^{-y}}{y \log^{a}(\frac{t}{i y})}-\int_{0}^{\frac{\pi}{2}} i d\theta \frac{e^{i \epsilon e^{i \theta}}}{\log^{a}(\frac{t e^{-i\theta}}{\epsilon})}+\int_{\epsilon}^{t/2} \frac{e^{i u}}{u \log^{a}(t/u)} du$$ 
So, $$\int_{\epsilon}^{t/2} \frac{\sin(u)}{u \log^{a}(t/u)} du = \text{Im}\left(\int_{\epsilon}^{t/2} dy \frac{e^{-y}}{y\log^{a}(\frac{t}{i y})}\right)+\text{Im}\left(\int_{0}^{\frac{\pi}{2}} i d\theta \frac{e^{i \epsilon e^{i \theta}}}{\log^{a}(\frac{t e^{-i\theta}}{\epsilon})}\right)-\text{Im}\left(\int_{0}^{\frac{\pi}{2}} i d\theta \frac{e^{\frac{i t}{2} \cos(\theta)} e^{-\frac{t}{2} \sin(\theta)}}{\log^{a}(2 e^{-i \theta})}\right)$$
Letting $\epsilon \rightarrow 0$, we have
\begin{equation}\label{sinintegral}\int_{0}^{t/2} \frac{\sin(u)}{u \log^{a}(t/u)} du = \lim_{\epsilon \rightarrow 0}\text{Im}\left(\int_{\epsilon}^{t/2} dy \frac{e^{-y}}{y\log^{a}(\frac{t}{i y})}\right)-\text{Im}\left(\int_{0}^{\frac{\pi}{2}} i d\theta \frac{e^{\frac{i t}{2} \cos(\theta)} e^{-\frac{t}{2} \sin(\theta)}}{\log^{a}(2 e^{-i \theta})}\right)\end{equation}
Note that 
$$\text{Im}\left(\int_{\epsilon}^{t/2} dy \frac{e^{-y}}{y\log^{a}(\frac{t}{i y})}\right) = \int_{\epsilon}^{t/2} dy \frac{e^{-y}\sin(a\tan^{-1}(\frac{\pi}{2(\log(t)-\log(y))}))}{y((\log(t)-\log(y))^{2}+\frac{\pi^{2}}{4})^{a/2}}$$
So, \begin{equation}\label{verticalcontour}\lim_{\epsilon \rightarrow 0} \text{Im}\left(\int_{\epsilon}^{t/2} dy \frac{e^{-y}}{y\log^{a}(\frac{t}{i y})}\right) = \int_{0}^{t/2} dy \frac{e^{-y}\sin(a\tan^{-1}(\frac{\pi}{2(\log(t)-\log(y))}))}{y((\log(t)-\log(y))^{2}+\frac{\pi^{2}}{4})^{a/2}}\end{equation}
We have 
$$\frac{e^{-y}\sin(a\tan^{-1}(\frac{\pi}{2(\log(t)-\log(y))}))}{y((\log(t)-\log(y))^{2}+\frac{\pi^{2}}{4})^{a/2}} = a\frac{\pi}{2}\frac{e^{-y}}{y \log^{a+1}(t/y)} + \text{Err}, \quad 0 \leq y \leq t/2$$
where
$$|\text{Err}| \leq C \frac{e^{-y}}{y}\frac{1}{\log^{3+a}(t/y)}, \quad 0 \leq y \leq t/2$$
So, 
\begin{equation} \label{seriesofintegrand}\begin{split} \int_{0}^{t/2} dy \frac{e^{-y}\sin(a\tan^{-1}(\frac{\pi}{2(\log(t)-\log(y))}))}{y((\log(t)-\log(y))^{2}+\frac{\pi^{2}}{4})^{a/2}}&=a\frac{\pi}{2}\int_{0}^{t/2} dy \frac{e^{-y}}{y\log^{a+1}(t/y)} + \int_{0}^{t/2} dy \text{Err} \end{split}\end{equation}
Let us start by considering
\begin{equation}\label{leadingexp}\int_{0}^{1} \frac{dy}{y} \frac{e^{-y}}{\log^{a+1}(t/y)} = \int_{0}^{1} \frac{dy}{y \log^{a+1}(t/y)} +\int_{0}^{1} dy \frac{\left(e^{-y}-1\right)}{y \log^{a+1}(\frac{t}{y})}\end{equation}
Treating each term individually, we have:
$$ \int_{0}^{1} \frac{dy}{y \log^{a+1}(t/y)} = \frac{1}{a \log^{a}(t)}$$
\begin{equation}\label{logintegral2}|\int_{0}^{1} dy \frac{\left(e^{-y}-1\right)}{y \log^{a+1}(\frac{t}{y})}| \leq \frac{C}{\log^{a+1}(t)} \int_{0}^{1} dy \frac{|e^{-y}-1|}{y} \leq \frac{C}{\log^{a+1}(t)} \end{equation}
where we used the fact that
$$y \mapsto \frac{1}{\log^{a+1}(\frac{t}{y})},\text{ is increasing on }(0,1]$$
So, $$\int_{0}^{1} \frac{dy}{y} \frac{e^{-y}}{\log^{a+1}(t/y)} = \frac{1}{a \log^{a}(t)} +O\left(\frac{1}{\log^{a+1}(t)}\right)$$
Also,
$$\int_{1}^{\sqrt{t}} \frac{dy}{y} \frac{e^{-y}}{(\log(t)-\log(y))^{a+1}} \leq 2^{a+1} \int_{1}^{\sqrt{t}} \frac{dy}{y} \frac{e^{-y}}{\log^{a+1}(t)} =O(\frac{1}{\log^{a+1}(t)}) $$ 
Finally,
$$\int_{\sqrt{t}}^{t/2} \frac{dy}{y} \frac{e^{-y}}{(\log(t)-\log(y))^{a+1}} \leq \frac{1}{\sqrt{t}} \int_{\sqrt{t}}^{t/2} dy \frac{e^{-y}}{\log^{a+1}(2)}=O(\frac{e^{-\sqrt{t}}}{\sqrt{t}})$$
So, \begin{equation}\label{fullleadingexp}\int_{0}^{t/2} \frac{dy}{y} \frac{e^{-y}}{\log^{a+1}(t/y)} = \frac{1}{a \log^{a}(t)} + O\left(\frac{1}{\log^{a+1}(t)}\right)\end{equation}

Treating the $\text{Err}$ term in \eqref{seriesofintegrand} in the same way as was used to obtain \eqref{fullleadingexp}, we have
\begin{equation}\begin{split} \int_{0}^{t/2} dy \frac{e^{-y}\sin(a\tan^{-1}(\frac{\pi}{2(\log(t)-\log(y))}))}{y((\log(t)-\log(y))^{2}+\frac{\pi^{2}}{4})^{a/2}} &= a \frac{\pi}{2}\left(\frac{1}{a \log^{a}(t)} + O\left(\frac{1}{\log^{a+1}(t)}\right)\right) + O\left(\frac{1}{\log^{a+2}(t)}\right) \\
&=\frac{\pi}{2 \log^{a}(t)} + O\left(\frac{1}{\log^{a+1}(t)}\right)\end{split}\end{equation}
Lastly, let us estimate
\begin{equation} |-\text{Im}\left(\int_{0}^{\frac{\pi}{2}} i d\theta \frac{e^{\frac{i t}{2} \cos(\theta)} e^{-\frac{t}{2} \sin(\theta)}}{\log^{a}(2 e^{-i \theta})}\right)| \leq C \int_{0}^{\frac{\pi}{2}}d\theta \frac{e^{-\frac{t}{2} \sin(\theta)}}{\log^{a}(2)} \leq C \int_{0}^{\frac{\pi}{2}}d\theta e^{-\frac{t}{2}\frac{2}{\pi}\theta} =O\left(\frac{1}{t}\right)\end{equation}
So, \eqref{sinintegral} can be written as
\begin{equation} \int_{0}^{t/2} \frac{\sin(u)}{u\log^{a}(t/u)} du = \frac{\pi}{2\log^{a}(t)} +O\left(\frac{1}{\log^{a+1}(t)}\right)\end{equation}
Since \eqref{sinintegral} is valid for all $a >0$, we can repeat the identical steps done in this section for the case $b=1$, and apply the above result to \eqref{innerprodlogs} (and its analog for $b=1$), thereby obtaining, for all $b>0$:
\begin{equation} \int_{0}^{\infty} R dR \left(\frac{\cos(2Q_{1}(R))-1}{R^{2}\lambda(t)^{2}}\right)\phi_{0}(R) v_{2}(t,R\lambda(t)) = \frac{4b}{\lambda(t) t^{2} \log^{b}(t)} +E_{v_{2},ip}(t,\lambda(t))\end{equation}
where, for $b \neq 1$:
\begin{equation}\label{ev2ip}\begin{split}E_{v_{2},ip}(t,\lambda(t))&=2 c_{b}\int_{0}^{\infty} d\xi \frac{\sin(t\xi)}{t^{2}}\psi_{v_{2}}(\xi,\lambda(t))+2 c_{b}\int_{0}^{\infty} d\xi \chi_{\leq \frac{1}{4}}(\xi) \frac{\sin(t\xi)}{t^{2}} F_{v_{2}}(\xi,\lambda(t))\\
&+2 \frac{c_{b}}{\lambda(t)} \int_{0}^{\frac{1}{2}} d\xi \left(\chi_{\leq \frac{1}{4}}(\xi)-1\right) \frac{\sin(t\xi)}{t^{2}}\left(\frac{b-1}{\xi\log^{b}(\frac{1}{\xi})}+\frac{b(b-1)}{\xi \log^{b+1}(\frac{1}{\xi})}\right)\\
&+\frac{2 c_{b}}{\lambda(t)} \left(\int_{0}^{\frac{t}{2}} \frac{du \sin(u) (b-1)}{t^{2}u \log^{b}(\frac{t}{u})} -\frac{(b-1)\pi}{2 t^{2}\log^{b}(t)} + \int_{0}^{\frac{t}{2}} du \frac{\sin(u)b(b-1)}{t^{2}u\log^{b+1}(\frac{t}{u})}\right)\end{split}\end{equation}
When $b=1$, $E_{v_{2},ip}$ has the same form as \eqref{ev2ip}, except the second and third lines changed to
\begin{equation}\begin{split} &\frac{2 c_{1}}{\lambda(t)} \int_{0}^{\frac{1}{2}} d\xi \left(\chi_{\leq \frac{1}{4}}(\xi)-1\right) \frac{\sin(t\xi)}{\xi t^{2}} \left(-\frac{1}{\log^{2}(\frac{1}{\xi})}-\frac{1}{\log(\frac{1}{\xi})}\right)\\
&+\frac{2 c_{1}}{\lambda(t)} \left(\int_{0}^{\frac{1}{2}} d\xi \frac{\sin(t\xi)}{\xi t^{2}} \left(\frac{-1}{\log(\frac{1}{\xi})}-\frac{1}{\log^{2}(\frac{1}{\xi})}\right) + \frac{\pi}{2 t^{2} \log(t)}\right)\end{split}\end{equation}
For all $b>0$,
$$|E_{v_{2},ip}(t,\lambda(t))| \leq \frac{C}{\lambda(t) t^{2}\log^{b+1}(t)}$$
\end{proof}
\subsubsection{Pointwise estimates on $\partial_{t}^{i}\partial_{r}^{j}v_{2}$}
\begin{lemma} For any $b>0$, there exists $C>0$ such that\\
\\
For $0 \leq j \leq 2, \quad 0 \leq k \leq 1, \quad j+k \leq 2$,
\begin{equation}\label{v2precisenearorigin} \partial_{t}^{j}\partial_{r}^{k}v_{2}(t,r) = \frac{\partial_{t}^{j}\partial_{r}^{k}\left(\frac{-b r}{t^{2}}\right)}{\log^{b}(t)} + E_{\partial_{t}^{j}\partial_{r}^{k}v_{2}}(t,r)\end{equation}
where
$$|E_{\partial_{t}^{j}\partial_{r}^{k}v_{2}}(t,r)| \leq C\partial_{t}^{j}\partial_{r}^{k}\left(\frac{r}{t^{2} \log^{b+1}(t)}\right) + C \partial_{t}^{j}\partial_{r}^{k}\left(\frac{r^{2}}{t^{3} \log^{b}(t)}\right), \quad r \leq \frac{t}{2}$$

For $0 \leq j \leq 2, \quad 0 \leq k \leq 2, \quad j+k \leq 2$,
\begin{equation}\label{v2sqrtrest}|\partial_{t}^{j}\partial_{r}^{k}v_{2}(t,r)| \leq \frac{C}{\sqrt{r}}, \quad r \geq \frac{t}{2}\end{equation}

For $0 \leq j \leq 1, \quad 0 \leq k \leq 1, \quad j+k \leq 2$,
\begin{equation}\label{v2singularconeest}|\partial_{t}^{j}\partial_{r}^{k}v_{2}(t,r)| \leq \frac{C\log(r)}{|t-r|^{1+j+k}}, \quad r \geq \frac{t}{2}, \quad r \neq t\end{equation}

Finally, we also have the estimates
\begin{equation} \label{drrv2mainest}|\partial_{r}^{2} v_{2}(t,r)| \leq \begin{cases} \frac{C r}{t^{4} \log^{b}(t)}, \quad r \leq \frac{t}{2}\\
\frac{C \log(r)}{|t-r|^{3}}, \quad \frac{t}{2} < r \neq t\\
\end{cases}\end{equation}

\end{lemma}
\begin{proof}
In the course of proving these estimates, we will occasionally make use of the following formula, which can be found, for example, in Appendix B of \cite{ktv}: for $n > -\frac{1}{2}$
\begin{equation}\label{besselrep}J_{n}(x) = \frac{1}{\Gamma(n+\frac{1}{2}) \sqrt{\pi}} \left(\frac{x}{2}\right)^{n} \int_{0}^{\pi}\cos(x \cos(\theta))\sin^{2n}(\theta) d\theta\end{equation}
Similarly, we will need, Lemma 8.1 of \cite{ktv}, which states
\begin{equation}\label{jnktv} J_{\frac{d-2}{2}}(x) = \frac{\left(\frac{x}{2}\right)^{\frac{d-2}{2}}}{\pi^{\frac{d-1}{2}}} \text{Re}\left(e^{-i x} \Phi_{\frac{d-2}{2}}(x)\right), \quad x>0\end{equation}
where
$$|\Phi_{\frac{d-2}{2}}^{(k)}|(x) \leq C_{k} x^{-\frac{(d-1)}{2}-k}$$
We will first prove the lemma for the case $b\neq 1$, and then remark what the identical procedures give, for the case $b=1$.\\ 
We will start with the region $0 \leq r \leq \frac{t}{2}$, and use \eqref{besselrep} for $n=1$ in the formula for $v_{2}$:
\begin{equation}\label{v2betternear0} \begin{split} v_{2}(t,r) &= c_{b} \int_{0}^{\infty} d\xi \sin(t\xi) J_{1}(r \xi) \frac{\chi_{\leq \frac{1}{4}}(\xi)}{\log^{b-1}(\frac{1}{\xi})}\\
&=\frac{c_{b}}{\pi} \int_{0}^{\pi} \sin^{2}(\theta) d\theta \int_{0}^{\infty} d\xi \sin(t\xi) r \xi \cos(r \xi \cos(\theta)) \frac{\chi_{\leq \frac{1}{4}}(\xi)}{\log^{b-1}(\frac{1}{\xi})}\\
&=\frac{c_{b}}{2\pi}\int_{0}^{\pi} \sin^{2}(\theta) d\theta \int_{0}^{\infty} d\xi r \xi \left(\sin(\xi(t+r\cos(\theta)))+\sin(\xi(t-r\cos(\theta)))\right) \frac{\chi_{\leq \frac{1}{4}}(\xi)}{\log^{b-1}(\frac{1}{\xi})}\end{split}\end{equation}
Integrating by parts twice, we get
\begin{equation} v_{2}(t,r) = -\frac{c_{b} r}{2 \pi} \int_{0}^{\pi} \sin^{2}(\theta) d\theta \int_{0}^{\infty} d\xi \partial_{\xi}^{2} \left(\frac{\xi \chi_{\leq \frac{1}{4}}(\xi)}{\log^{b-1}(\frac{1}{\xi})}\right)\left(\frac{\sin(\xi t_{+})}{t_{+}^{2}}+\frac{\sin(\xi t_{-})}{t_{-}^{2}}\right)\end{equation}
where
\begin{equation}\label{tplusminus}t_{+/-} = t\pm r \cos(\theta)\end{equation}
We divide the terms analogously to how similar terms were treated in the previous subsection. In particular, we have
\begin{equation} v_{2}(t,r) = \frac{-c_{b}r}{2\pi} \int_{0}^{\pi} \sin^{2}(\theta) d\theta \int_{0}^{\infty} d\xi \left(\chi_{\leq \frac{1}{4}}(\xi) \left(\frac{b-1}{\xi \log^{b}(\frac{1}{\xi})}+\frac{b(b-1)}{\xi \log^{b+1}(\frac{1}{\xi})}\right)+\psi(\xi)\right)\left(\frac{\sin(\xi t_{+})}{t_{+}^{2}} + \frac{\sin(\xi t_{-})}{t_{-}^{2}}\right)\end{equation}
where $\psi \in C^{\infty}_{c}([\frac{1}{8},\frac{1}{4}])$. So, we can integrate by parts (for example, once) to treat the $\psi$ term:
\begin{equation}\begin{split} &|\frac{-c_{b}r}{2\pi} \int_{0}^{\pi} \sin^{2}(\theta) d\theta \int_{0}^{\infty} d\xi \psi(\xi)\left(\frac{\sin(\xi t_{+})}{t_{+}^{2}} + \frac{\sin(\xi t_{-})}{t_{-}^{2}}\right)| \leq C r \int_{0}^{\pi} d\theta \int_{0}^{\infty} \frac{|\psi'(\xi)| d\xi}{t^{3}}\\
&\leq \frac{C r}{t^{3}}, \quad r \leq \frac{t}{2}\end{split}\end{equation}
where we use
$$\frac{1}{|t_{\pm}|} \leq \frac{C}{t}, \quad r \leq \frac{t}{2}$$
Then, it remains to study 
\begin{equation}\label{v2near0intest}\begin{split} &\frac{-c_{b}r}{2\pi} \int_{0}^{\pi} \sin^{2}(\theta) d\theta \int_{0}^{\infty} d\xi \left(\chi_{\leq \frac{1}{4}}(\xi) \left(\frac{b-1}{\xi \log^{b}(\frac{1}{\xi})}+\frac{b(b-1)}{\xi \log^{b+1}(\frac{1}{\xi})}\right)\right)\left(\frac{\sin(\xi t_{+})}{t_{+}^{2}} + \frac{\sin(\xi t_{-})}{t_{-}^{2}}\right)\\
&=\frac{-c_{b}r}{2\pi} \int_{0}^{\pi} \sin^{2}(\theta)d\theta \int_{0}^{\frac{1}{2}} \left(\frac{b-1}{\xi \log^{b}(\frac{1}{\xi})} + \frac{b(b-1)}{\xi \log^{b+1}(\frac{1}{\xi})}\right) \left(\frac{\sin(\xi t_{+})}{t_{+}^{2}} + \frac{\sin(\xi t_{-})}{t_{-}^{2}}\right)\\
&+\frac{-c_{b}r}{2\pi} \int_{0}^{\pi} \sin^{2}(\theta)d\theta \int_{0}^{\frac{1}{2}} \left(\chi_{\leq \frac{1}{4}}(\xi)-1\right) \left(\frac{b-1}{\xi \log^{b}(\frac{1}{\xi})}+\frac{b(b-1)}{\xi \log^{b+1}(\frac{1}{\xi})}\right)\left(\frac{\sin(\xi t_{+})}{t_{+}^{2}} + \frac{\sin(\xi t_{-})}{t_{-}^{2}}\right)\end{split}\end{equation}
We start with the second line of \eqref{v2near0intest}. Since all other terms will use the identical argument, we first consider the term involving $\frac{1}{\log^{b}(t)}$ and $t_{+}$, namely
\begin{equation}\label{v2near0intest2}\begin{split} &\frac{-c_{b}r}{2\pi} \int_{0}^{\pi} \sin^{2}(\theta) d\theta \left((b-1) \int_{0}^{\frac{t_{+}}{2}} \frac{du \sin(u)}{u \log^{b}(\frac{t_{+}}{u}) t_{+}^{2}}\right)= -\frac{c_{b}r}{2\pi} \int_{0}^{\pi} \sin^{2}(\theta) d\theta (b-1) \left(\frac{\pi}{2 \log^{b}(t_{+}) t_{+}^{2}} + \text{Err}(t,r,\theta) \right)\end{split}\end{equation}
where
$$|\text{Err}(t,r,\theta)| \leq \frac{C}{t_{+}^{2} \log^{b+1}(t_{+})}$$
and we use our calculation of \eqref{sinintegral} from the previous subsection. To treat this last integral, we start with the first term:
\begin{equation}\begin{split} \frac{-c_{b} r}{2 \pi} \cdot (b-1)\frac{\pi}{2} \int_{0}^{\pi} \frac{\sin^{2}(\theta) d\theta}{\log^{b}(t_{+})t_{+}^{2}} &= \frac{-c_{b} r (b-1)}{4} \frac{1}{t^{2} \log^{b}(t)} \int_{0}^{\pi} \frac{\sin^{2}(\theta) d\theta}{(1+\frac{r \cos(\theta)}{t})^{2}} \frac{1}{\left(1+\frac{\log(1+\frac{r \cos(\theta)}{t})}{\log(t)}\right)^{b}}\\
&=\frac{-c_{b} r (b-1)}{4} \left(\frac{\pi}{2 t^{2} \log^{b}(t)} + E_{v_{2},1}(t,r)\right), \quad r \leq \frac{t}{2}\end{split}\end{equation}
with
$$|E_{v_{2},1}(t,r)| \leq C \frac{r}{t^{3} \log^{b}(t)}, \quad r \leq \frac{t}{2}$$
For the second term of \eqref{v2near0intest2}, we have
\begin{equation}\begin{split}&|\frac{-c_{b}r}{2\pi} \int_{0}^{\pi} \sin^{2}(\theta) d\theta (b-1)\text{Err}(t,r,\theta)| \leq C r \int_{0}^{\pi} \frac{d\theta}{t^{2} \log^{b+1}(t)} \\
&\leq \frac{C r}{t^{2} \log^{b+1}(t)}, \quad r \leq \frac{t}{2}\end{split}\end{equation}
where we used
\begin{equation} \frac{1}{t_{+}^{2} \log^{b+1}(t_{+})} \leq \frac{C}{t^{2} \log^{b+1}(t)}, \quad r \leq \frac{t}{2}\end{equation}
Combining these, we get
\begin{equation}\frac{-c_{b}r}{2\pi} \int_{0}^{\pi} \sin^{2}(\theta) d\theta \left((b-1) \int_{0}^{\frac{t_{+}}{2}} \frac{du \sin(u)}{u \log^{b}(\frac{t_{+}}{u}) t_{+}^{2}}\right)= \frac{-b r}{2 t^{2} \log^{b}(t)} + E_{v_{2},2}(t,r), \quad r \leq \frac{t}{2}\end{equation}
$$|E_{v_{2},2}(t,r)| \leq C\left(\frac{r^{2}}{t^{3} \log^{b}(t)}+C\frac{r}{t^{2} \log^{b+1}(t)}\right), \quad r \leq \frac{t}{2}$$
We use the identical procedure to treat all the other terms in the second line of \eqref{v2near0intest}, and get
\begin{equation} \begin{split} &\frac{-c_{b}r}{2\pi} \int_{0}^{\pi} \sin^{2}(\theta)d\theta \int_{0}^{\frac{1}{2}} d\xi \left(\frac{b-1}{\xi \log^{b}(\frac{1}{\xi})} + \frac{b(b-1)}{\xi \log^{b+1}(\frac{1}{\xi})}\right) \left(\frac{\sin(\xi t_{+})}{t_{+}^{2}} + \frac{\sin(\xi t_{-})}{t_{-}^{2}}\right)\\
&= \frac{-b r}{t^{2} \log^{b}(t)} + E_{v_{2},3}(t,r)\end{split}\end{equation}
with
$$|E_{v_{2},3}(t,r)| \leq C\left(\frac{r}{t^{2} \log^{b+1}(t)}+\frac{r^{2}}{t^{3} \log^{b}(t)}\right), \quad r \leq \frac{t}{2}$$
Finally, for the third line of \eqref{v2near0intest}, we integrate by parts, identically to how a similar term was treated in the last subsection, and get
\begin{equation} |\frac{-c_{b}r}{2\pi} \int_{0}^{\pi} \sin^{2}(\theta)d\theta \int_{0}^{\frac{1}{2}} \left(\chi_{\leq \frac{1}{4}}(\xi)-1\right) \left(\frac{b-1}{\xi \log^{b}(\frac{1}{\xi})}+\frac{b(b-1)}{\xi \log^{b+1}(\frac{1}{\xi})}\right)\left(\frac{\sin(\xi t_{+})}{t_{+}^{2}} + \frac{\sin(\xi t_{-})}{t_{-}^{2}}\right)| \leq \frac{C r}{t^{3}}, \quad r \leq \frac{t}{2}\end{equation}
This completes the proof of \eqref{v2precisenearorigin} (for $v_{2}$).\\
\\
We will need two more estimates on $v_{2}$. One is given by using the simple estimate
$$|J_{1}(x)| \leq \frac{C}{\sqrt{x}}$$
which yields
\begin{equation} |v_{2}(t,r)| \leq \frac{C}{\sqrt{r}}\end{equation}
We will obtain another estimate in the region $r \geq \frac{t}{2}$, by first decomposing $v_{2}$ as follows:
\begin{equation}\label{v2forest}\begin{split} &v_{2}(t,r) = c_{b}\int_{0}^{\infty} d\xi \sin(t\xi) J_{1}(r \xi) \frac{\chi_{\leq \frac{1}{4}}(\xi)}{\log^{b-1}(\frac{1}{\xi})}\\
&=c_{b}\int_{0}^{\infty} d\xi \chi_{\leq 1}(r \xi) \sin(t\xi) J_{1}(r \xi) \frac{\chi_{\leq \frac{1}{4}}(\xi)}{\log^{b-1}(\frac{1}{\xi})}+c_{b}\int_{0}^{\infty} d\xi \left(1-\chi_{\leq 1}(r \xi)\right) \sin(t\xi) J_{1}(r \xi) \frac{\chi_{\leq \frac{1}{4}}(\xi)}{\log^{b-1}(\frac{1}{\xi})}\end{split}\end{equation}
where $\chi_{\leq 1} \in C^{\infty}_{c}([0,\infty))$, $0 \leq \chi_{\leq 1} \leq 1$,
$$\chi_{\leq 1}(x) = \begin{cases} 1, \quad x \leq \frac{1}{2}\\
0, \quad x \geq 1\end{cases}$$
and $\chi_{\leq 1}$ is otherwise arbitrary.\\
We start with the second line of \eqref{v2forest}. Using the simple estimate
$$|J_{1}(x)| \leq C x$$
we get
\begin{equation}\begin{split}|\int_{0}^{\infty} d\xi \chi_{\leq 1}(r \xi) \sin(t\xi) J_{1}(r \xi) \frac{\chi_{\leq \frac{1}{4}}(\xi)}{\log^{b-1}(\frac{1}{\xi})}|& \leq C r \int_{0}^{\infty} d\xi \chi_{\leq 1}(r \xi) \xi \frac{\chi_{\leq \frac{1}{4}}(\xi)}{\log^{b-1}(\frac{1}{\xi})}\\
&\leq \frac{C}{r \log^{b-1}(r)}, \quad r \geq \frac{t}{2} > 4\end{split}\end{equation}\\
\\
We need only estimate the remaining integral in \eqref{v2forest}, namely,
\begin{equation}\int_{0}^{\infty} d\xi \left(1-\chi_{\leq 1}(r \xi)\right) \sin(t\xi) J_{1}(r \xi) \frac{\chi_{\leq \frac{1}{4}}(\xi)}{\log^{b-1}(\frac{1}{\xi})}\end{equation}
Note that this integral vanishes if $r<2$, by the support properties of $\chi_{\leq 1}$ and $\chi_{\leq \frac{1}{4}}$. We use \eqref{jnktv}, for $d=4$.
Then, for $r>4$,  we have
\begin{equation}\label{awayfromcone}\begin{split} |\int_{0}^{\infty} d\xi \left(1-\chi_{\leq 1}(r \xi)\right) \sin(t\xi) J_{1}(r \xi) \frac{\chi_{\leq \frac{1}{4}}(\xi)}{\log^{b-1}(\frac{1}{\xi})}| &= |\text{Re}\left(\int_{0}^{\infty} d\xi \left(1-\chi_{\leq 1}(r \xi)\right) \sin(t\xi) \frac{r \xi}{2 \pi^{3/2}} e^{-i r \xi} \Phi_{1}(r \xi) \frac{\chi_{\leq \frac{1}{4}}(\xi)}{\log^{b-1}(\frac{1}{\xi})}\right)|\\
&\leq \frac{C}{|t-r|} \int_{0}^{\infty} d\xi |\partial_{\xi}\left(\left(1-\chi_{\leq 1}(r \xi)\right) r \xi \Phi_{1}(r \xi) \frac{\chi_{\leq \frac{1}{4}}(\xi)}{\log^{b-1}(\frac{1}{\xi})}\right)|\end{split}\end{equation}
Note that \begin{equation}\begin{split} \int_{0}^{\infty} d \xi |\chi_{\leq 1}'(r \xi)| r |r \xi \Phi_{1}(r \xi)| \frac{|\chi_{\leq \frac{1}{4}}(\xi)|}{\log^{b-1}(\frac{1}{\xi})}&\leq C r \int_{0}^{1/r}\frac{d\xi}{\log^{b-1}(\frac{1}{\xi})}\\
&\leq \frac{C}{\log^{b-1}(r)}, \quad r\geq\frac{t}{2}\end{split}\end{equation} 
where we used the support properties of $\chi$ and symbol property of $\Phi_{1}$.\\
\\
Next, \begin{equation} \begin{split} \int_{0}^{\infty} d\xi \left(1-\chi_{\leq 1}(r \xi)\right) |\partial_{\xi}(r \xi \Phi_{1}(r \xi))| \frac{|\chi_{\leq \frac{1}{4}}(\xi)|}{\log^{b-1}(\frac{1}{\xi})} &\leq C r \int_{0}^{\infty} d\xi \frac{\left(1-\chi_{\leq 1}(r \xi)\right)}{(r \xi)^{3/2}} \frac{|\chi_{\leq \frac{1}{4}}(\xi)|}{\log^{b-1}(\frac{1}{\xi})}\\
&\leq \frac{C}{\sqrt{r}} \int_{\frac{1}{2r}}^{\frac{1}{4}} \frac{d\xi}{\xi^{3/2}\log^{b-1}(\frac{1}{\xi})}\\
&\leq  \frac{C}{\log^{b-1}(r)}, \quad r\geq\frac{t}{2}\end{split}\end{equation}

Finally, we use the same procedure to get \begin{equation} \begin{split} \int_{0}^{\infty} d\xi \left(1-\chi_{\leq 1}(r \xi)\right) |r \xi \Phi_{1}(r \xi)| |\partial_{\xi}\left(\frac{\chi_{\leq \frac{1}{4}}(\xi)}{\log^{b-1}(\frac{1}{\xi})}\right)|&\leq \frac{C}{\sqrt{r}} \int_{\frac{1}{2 r}}^{\frac{1}{4}} d\xi \left(\frac{|\chi_{\leq \frac{1}{4}}'(\xi)|}{\sqrt{\xi}\log^{b-1}(\frac{1}{\xi})} + \frac{1}{\xi^{3/2} \log^{b}(\frac{1}{\xi})}\right)\\
&\leq \frac{C}{\log^{b}(r)}, \quad r \geq \frac{t}{2}\end{split}\end{equation}\\
\\
Returning to \eqref{awayfromcone}, we get, for $r \geq \frac{t}{2}$,
\begin{equation}\begin{split}\frac{C}{|t-r|} \int_{0}^{\infty} d\xi |\partial_{\xi}\left(\left(1-\chi_{\leq 1}(r \xi)\right) r \xi \phi(r \xi) \frac{\chi_{\leq \frac{1}{4}}(\xi)}{\log^{b-1}(\frac{1}{\xi})}\right)| \leq \frac{C}{|t-r|\log^{b-1}(r)}\end{split}\end{equation}
\\
\\
Note that, if $\frac{t}{2} \leq r$, then, $|t-r| \leq r$, so $$\frac{C}{r \log^{b-1}(r)} \leq \frac{C}{|t-r|\log^{b-1}(r)}, \quad r \geq \frac{t}{2}, \quad r \neq t $$
So, we can combine the previous two estimates to conclude \eqref{v2singularconeest}, thereby concluding the proof of the $v_{2}$ estimates (in the case $b \neq 1$). \\
\\
We now proceed to prove the similar statements about $\partial_{r}v_{2}$. We have
\begin{equation} \partial_{r}v_{2}(t,r) = \frac{c_{b}}{2} \int_{0}^{\infty} \sin(t \xi) \xi \left(J_{0}(r \xi)-J_{2}(r \xi)\right) \frac{\chi_{\leq \frac{1}{4}}(\xi)}{\log^{b-1}(\frac{1}{\xi})} d\xi\end{equation}
and start by treating the $J_{0}$ term. Using \eqref{besselrep} for $n=0$, we have
\begin{equation}\begin{split}&\frac{c_{b}}{2} \int_{0}^{\infty} \sin(t \xi) \xi J_{0}(r \xi) \frac{\chi_{\leq \frac{1}{4}}(\xi)}{\log^{b-1}(\frac{1}{\xi})} d\xi=\frac{c_{b}}{4\pi} \int_{0}^{\pi} d\theta \int_{0}^{\infty} d\xi \left(\sin(\xi t_{+})+\sin(\xi t_{-})\right)\frac{\xi \chi_{\leq \frac{1}{4}}(\xi)}{\log^{b-1}(\frac{1}{\xi})}\\
&=\frac{-c_{b}}{4\pi} \int_{0}^{\pi} d\theta \int_{0}^{\infty} d\xi \left(\frac{\sin(\xi t_{+})}{t_{+}^{2}}+\frac{\sin(\xi t_{-})}{t_{-}^{2}}\right)\partial_{\xi}^{2}\left(\frac{\xi \chi_{\leq \frac{1}{4}}(\xi)}{\log^{b-1}(\frac{1}{\xi})}\right)\end{split}\end{equation}
We follow the identical procedure used to prove \eqref{v2precisenearorigin} for $v_{2}$, the only major difference being an extra $\sin^{2}(\theta)$ in the $\theta$ integral in the $v_{2}$ case. We get
\begin{equation}\frac{c_{b}}{2} \int_{0}^{\infty} \sin(t \xi) \xi J_{0}(r \xi) \frac{\chi_{\leq \frac{1}{4}}(\xi)}{\log^{b-1}(\frac{1}{\xi})} d\xi = \frac{-b}{t^{2} \log^{b}(t)} + E_{\partial_{r}v_{2},1}(t,r)\end{equation}
where
$$|E_{\partial_{r}v_{2}}(t,r)| \leq \frac{C r}{t^{3} \log^{b}(t)} + \frac{C}{t^{2} \log^{b+1}(t)}, \quad r \leq \frac{t}{2}$$
Now, we treat the $J_{2}$ term, again using \eqref{besselrep}, this time for $n=2$.
\begin{equation}\begin{split}&-\frac{c_{b}}{2} \int_{0}^{\infty} d\xi \sin(t\xi) \xi J_{2}(r\xi) \frac{\chi_{\leq \frac{1}{4}}(\xi)}{\log^{b-1}(\frac{1}{\xi})}\\
&=\frac{-c_{b}r^{2}}{6\pi} \int_{0}^{\infty} d\xi \sin(t\xi) \xi^{3} \int_{0}^{\pi} \cos(r\xi \cos(\theta))\sin^{4}(\theta)d\theta \frac{\chi_{\leq \frac{1}{4}}(\xi)}{\log^{b-1}(\frac{1}{\xi})}\\
&=\frac{-c_{b} r^{2}}{12\pi} \int_{0}^{\pi} \sin^{4}(\theta) d\theta \int_{0}^{\infty} d\xi \left(\sin(\xi(t+r\cos(\theta)))+\sin(\xi(t-r\cos(\theta)))\right) \frac{\xi^{3}\chi_{\leq \frac{1}{4}}(\xi)}{\log^{b-1}(\frac{1}{\xi})}\\
&=\frac{-c_{b} r^{2}}{12\pi} \int_{0}^{\pi} \sin^{4}(\theta)d\theta \int_{0}^{\infty} d\xi \left(\frac{\sin(\xi(t+r\cos(\theta)))}{(t+r\cos(\theta))^{4}}+\frac{\sin(\xi(t-r\cos(\theta)))}{(t-r\cos(\theta))^{4}}\right)\partial_{\xi}^{4}\left(\frac{\xi^{3} \chi_{\leq \frac{1}{4}}(\xi)}{\log^{b-1}(\frac{1}{\xi})}\right)\end{split}\end{equation}
Note that
\begin{equation} \begin{split} \partial_{\xi}^{4}\left(\frac{\xi^{3} \chi_{\leq \frac{1}{4}}(\xi)}{\log^{b-1}(\frac{1}{\xi})}\right)&=\chi_{\leq \frac{1}{4}}(\xi)\left( \frac{11b(b-1)}{\xi \log^{b+1}(\frac{1}{\xi})}+\frac{6(b-1)}{\xi \log^{b}(\frac{1}{\xi})}+\frac{b(b-1)(b+1)(b+2)}{\xi\log^{b+3}(\frac{1}{\xi})}+\frac{6b(b-1)(b+1)}{\xi\log^{b+2}(\frac{1}{\xi})}\right)\\
&+\psi(\xi)\end{split}\end{equation}
where $$\psi \in C^{\infty}_{c}([\frac{1}{8},\frac{1}{4}])$$
We treat the $\psi$ term in the same way as in the $v_{2}$ estimate, and use the same argument as in the $v_{2}$ estimate to handle the other terms. In total, we get
\begin{equation}\begin{split} |-\frac{c_{b}}{2} \int_{0}^{\infty} d\xi \sin(t\xi) \xi J_{2}(r\xi) \frac{\chi_{\leq \frac{1}{4}}(\xi)}{\log^{b-1}(\frac{1}{\xi})}| &\leq \frac{C r^{2}}{t^{4}\log^{b}(t)}, \quad r \leq \frac{t}{2}\end{split}\end{equation}
Combining our results, we then get \eqref{v2precisenearorigin} for $\partial_{r}v_{2}$.\\
\\
Now, we study the region $r \geq \frac{t}{2}$. Again using the simple estimate
$$|J_{1}'(x)| \leq \frac{C}{\sqrt{x}}$$
we get \begin{equation}|\partial_{r}v_{2}(t,r)| \leq \frac{C}{\sqrt{r}}\end{equation}
Then, we decompose
\begin{equation} \partial_{r}v_{2}(t,r) = I_{r}+II_{r}\end{equation}
with
\begin{equation} I_{r} = \frac{c_{b}}{2} \int_{0}^{\infty} d\xi \sin(t\xi) \xi \chi_{\leq 1}(r\xi) \left(J_{0}(r\xi)-J_{2}(r\xi)\right) \frac{\chi_{\leq \frac{1}{4}}(\xi)}{\log^{b-1}(\frac{1}{\xi})}\end{equation}
Using the simple estimate
$$|J_{k}(x)| \leq C, \quad k=0,2$$
we get \begin{equation}\label{largerIr}\begin{split} |I_{r}| &\leq |\frac{c_{b}}{2 } \int_{0}^{\infty} \sin(t \xi) \chi_{\leq 1}(r \xi)  \xi \left(J_{0}(r\xi) - J_{2}(r \xi)\right) \frac{\chi_{\leq \frac{1}{4}}(\xi)}{\log^{b-1}(\frac{1}{\xi})} d\xi|\\
&\leq C\int_{0}^{\text{min}\{\frac{1}{r},\frac{1}{4}\}} \frac{\xi d \xi}{\log^{b-1}(\frac{1}{\xi})} \leq \frac{C}{r^{2} \log^{b-1}(r)}, \quad r \geq \frac{t}{2} \geq 4\end{split}\end{equation}
It remains to estimate $II_{r}$. We write 
$$II_{r}=II_{r,0}+II_{r,2}$$
with
$$II_{r,0} = \frac{c_{b}}{2} \int_{0}^{\infty} \sin(t\xi) \left(1-\chi_{\leq 1}(r \xi)\right) \xi J_{0}(r \xi) \frac{\chi_{\leq \frac{1}{4}}(\xi)}{\log^{b-1}(\frac{1}{\xi})}d\xi$$
$$II_{r,2} = -\frac{c_{b}}{2} \int_{0}^{\infty} \sin(t\xi) \left(1-\chi_{\leq 1}(r \xi)\right) \xi J_{2}(r \xi) \frac{\chi_{\leq \frac{1}{4}}(\xi)}{\log^{b-1}(\frac{1}{\xi})}d\xi$$ 
Again using \eqref{jnktv} for $d=2$, we get
\begin{equation}\label{IIr0} \begin{split} |II_{r,0}| &\leq C \int_{0}^{\infty} |\partial_{\xi}^{2}\left(\left(1-\chi_{\leq 1}(r \xi)\right)\xi \Phi_{0}(r \xi) \frac{\chi_{\leq \frac{1}{4}}(\xi)}{\log^{b-1}(\frac{1}{\xi})}\right)| \frac{d\xi}{(t-r)^{2}}\leq \frac{C}{(t-r)^{2}} \int_{0}^{\infty} \frac{\mathbbm{1}_{\{r \xi \geq \frac{1}{2}\}} \mathbbm{1}_{\{\xi \leq \frac{1}{4}\}}}{\sqrt{r}\xi^{3/2}\log^{b-1}(\frac{1}{\xi})} d\xi \\
&\leq \frac{C}{(t-r)^{2}} \int_{\frac{1}{2r}}^{\frac{1}{4}} \frac{d\xi}{\sqrt{r}\xi^{3/2} \log^{b-1}(\frac{1}{\xi})}\leq \frac{C}{(t-r)^{2}} \frac{1}{\sqrt{r}} \int_{\frac{1}{2 r}}^{\frac{1}{4}} \frac{d\xi}{\xi^{3/2}\log^{b-1}(\frac{1}{\xi})}\\
&\leq \frac{C}{(t-r)^{2}\log^{b-1}(r)}, \quad r \geq \frac{t}{2}\end{split}\end{equation} 

For $II_{r,2}$, we use \eqref{jnktv} for $d=6$. Using the same argument as for $II_{r,0}$, we get \begin{equation}\begin{split} |II_{r,2}| &\leq C \int_{0}^{\infty} |\partial_{\xi}^{2}\left(\left(1-\chi_{\leq 1}(r \xi)\right) \xi (r\xi)^{2} \Phi_{2}(r \xi) \frac{\chi_{\leq \frac{1}{4}}(\xi)}{\log^{b-1}(\frac{1}{\xi})}\right)|\frac{d\xi}{(t-r)^{2}}\\
&\leq \frac{C}{(t-r)^{2}\log^{b-1}(r)}, \quad r \geq \frac{t}{2}\end{split}\end{equation} 
Combining these, we have $$|II_{r}| \leq \frac{C}{(t-r)^{2}\log^{b-1}(r)}, \quad r \geq \frac{t}{2}$$
Using the same reasoning as was used to prove the analogous estimate for $v_{2}$, if $r \geq \frac{t}{2}, \quad r \neq t$, we have
\begin{equation} \frac{1}{r^{2}\log^{b-1}(r)} \leq \frac{C}{(t-r)^{2}\log^{b-1}(r)}\end{equation}
which gives \eqref{v2singularconeest} for $\partial_{r}v_{2}$.\\
\\
Next, we study $\partial_{t}v_{2}$, in the region $0 \leq r \leq \frac{t}{2}$. Using \eqref{besselrep}, we have
\begin{equation}\begin{split}\partial_{t}v_{2}(t,r) &= c_{b} \int_{0}^{\infty} d\xi \cos(t\xi) \frac{\xi J_{1}(r\xi) \chi_{\leq \frac{1}{4}}(\xi)}{\log^{b-1}(\frac{1}{\xi})}\\
&=\frac{c_{b}}{\pi} \int_{0}^{\pi} \sin^{2}(\theta) d\theta \int_{0}^{\infty} d\xi \xi \cos(t\xi) r \xi \cos(r \xi \cos(\theta)) \frac{\chi_{\leq \frac{1}{4}}(\xi)}{\log^{b-1}(\frac{1}{\xi})}\\
&=\frac{c_{b} r}{\pi} \int_{0}^{\pi} \sin^{2}(\theta) d\theta \int_{0}^{\infty} \frac{d\xi \xi^{2}}{2} \left(\cos(\xi(t+r\cos(\theta)))+\cos(\xi(t-r\cos(\theta)))\right)\frac{\chi_{\leq \frac{1}{4}}(\xi)}{\log^{b-1}(\frac{1}{\xi})}\\
&=\frac{c_{b} r}{2\pi} \int_{0}^{\pi} \sin^{2}(\theta) d\theta \int_{0}^{\infty} d\xi \left(\frac{\sin(\xi(t+r\cos(\theta)))}{(t+r\cos(\theta))^{3}}+\frac{\sin(\xi(t-r\cos(\theta)))}{(t-r\cos(\theta))^{3}}\right) \partial_{\xi}^{3}\left(\frac{\xi^{2}\chi_{\leq \frac{1}{4}}(\xi)}{\log^{b-1}(\frac{1}{\xi})}\right)\end{split}\end{equation}
We note that
\begin{equation} \partial_{\xi}^{3}\left(\frac{\xi^{2}}{\log^{b-1}(\frac{1}{\xi})}\chi_{\leq \frac{1}{4}}(\xi)\right) = \left(\frac{(b-1)b(b+1)}{\xi\log^{b+2}(\frac{1}{\xi})} + \frac{3 b(b-1)}{\xi\log^{b+1}(\frac{1}{\xi})} + \frac{2(b-1)}{\xi\log^{b}(\frac{1}{\xi})}\right)\chi_{\leq \frac{1}{4}}(\xi)+\psi(\xi)\end{equation}
where
$$\psi \in C^{\infty}_{c}([\frac{1}{8},\frac{1}{4}])$$
The integral to estimate is therefore of exactly the same form as that treated in estimating $v_{2}$, and repeating this procedure gives \eqref{v2precisenearorigin} for $\partial_{t}v_{2}$.
For the region $r \geq \frac{t}{2}$, we can again use
$$|J_{1}(x)| \leq \frac{C}{\sqrt{x}}$$ to get
\begin{equation}|\partial_{t}v_{2}(t,r)| \leq \frac{C}{\sqrt{r}}\end{equation}

We now consider the region $r \geq \frac{t}{2}$. First, let us write
\begin{equation} \partial_{t}v_{2}(t,r) = I_{t}(t,r)+II_{t}(t,r)\end{equation}
where
$$I_{t} = c_{b}\int_{0}^{\infty} \chi_{\leq 1}(r\xi) \xi \cos(t\xi) J_{1}(r \xi) \frac{\chi_{\leq \frac{1}{4}}(\xi)}{\log^{b-1}(\frac{1}{\xi})} d\xi$$

$$II_{t}=c_{b}\int_{0}^{\infty}\left(1-\chi_{\leq 1}(r\xi)\right) \xi \cos(t\xi) J_{1}(r \xi) \frac{\chi_{\leq \frac{1}{4}}(\xi)}{\log^{b-1}(\frac{1}{\xi})} d\xi$$
Then, we use $$|J_{1}(x)| \leq C x$$ to get
\begin{equation}\label{largerIt} |I_{t}| \leq C \int_{0}^{\text{min}\{\frac{1}{r},\frac{1}{4}\}}\frac{r\xi^{2}}{\log^{b-1}(\frac{1}{\xi})} d\xi \leq \frac{C}{r^{2}\log^{b-1}(r)}, \quad r \geq 4\end{equation}
As usual, we use \eqref{jnktv} with $d=4$ to get
\begin{equation} \begin{split} II_{t}&=\frac{c_{b}}{4 \pi^{3/2}} \text{Re}\left(\int_{0}^{\infty} \left(1-\chi_{\leq 1}(r\xi)\right) \xi \left(e^{i(t-r)\xi}+e^{-i(t+r)\xi}\right)\frac{r \xi \Phi_{1}(r\xi) \chi_{\leq \frac{1}{4}}(\xi)}{\log^{b-1}(\frac{1}{\xi})}d\xi \right)\end{split}\end{equation}
and this gives
\begin{equation} |II_{t}| \leq C \int_{0}^{\infty} \frac{1}{|t-r|^{2}} |\partial_{\xi}^{2} \left(\frac{\left(1-\chi_{\leq 1}(r\xi)\right)\xi^{2} r \Phi_{1}(r\xi) \chi_{\leq \frac{1}{4}}(\xi)}{\log^{b-1}(\frac{1}{\xi})}\right)| d\xi\end{equation}
But, comparing this expression with \eqref{IIr0}, we see that the only difference is
$$ \Phi_{0}(r\xi)\text{ is replaced by } \Phi_{1}(r\xi)\cdot r\xi$$
Comparing the symbol-type estimates on $\Phi_{d}$ which follow \eqref{jnktv}, we can use the same procedure used to treat \eqref{IIr0} to get
\begin{equation} |II_{t}| \leq \frac{C}{(t-r)^{2}\log^{b-1}(r)}, \quad r \geq \frac{t}{2}\end{equation} 
This gives \eqref{v2singularconeest} for $\partial_{t}v_{2}$. \\
\\
Next, we obtain estimates on $\partial_{r}^{2} v_{2}$, using the same procedure as above. In particular,
we use
\begin{equation} J_{1}''(x) = \frac{1}{4} (J_{3}(x)-3 J_{1}(x))\end{equation}
to get
\begin{equation} \partial_{r}^{2}v_{2}(t,r) = \frac{c_{b}}{4} \int_{0}^{\infty} \sin(t\xi) \xi^{2}\left(-3 J_{1}(r\xi) + J_{3}(r\xi)\right) \frac{\chi_{\leq \frac{1}{4}}(\xi)}{\log^{b-1}(\frac{1}{\xi})} d\xi\end{equation}
Then, we use \eqref{besselrep} for $J_{k}, \quad k=1,3$. For the term involving $J_{k}$, for $k=1$ or $k=3$, we integrate by parts $k+3$ times in $\xi$ and get:

\begin{equation} \begin{split} \partial_{r}^{2}v_{2}(t,r) &= \frac{-3 c_{b} r}{8 \pi} \int_{0}^{\pi} d\theta \sin^{2}(\theta) \int_{0}^{\infty} \left(\frac{\sin(\xi(t+r\cos(\theta)))}{(t+r \cos(\theta))^{4}}+\frac{\sin((t-r\cos(\theta))\xi)}{(t-r \cos(\theta))^{4}}\right) \partial_{\xi}^{4} \left(\frac{\xi^{3} \chi_{\leq \frac{1}{4}}(\xi)}{\log^{b-1}(\frac{1}{\xi})}\right)d\xi\\
&-\frac{c_{b} r^{3}}{120 \pi} \int_{0}^{\pi} d\theta \sin^{6}(\theta) \int_{0}^{\infty} \left(\frac{\sin(\xi(t+r\cos(\theta)))}{(t+r\cos(\theta))^{6}}+\frac{\sin(\xi(t-r\cos(\theta)))}{(t-r\cos(\theta))^{6}}\right) \partial_{\xi}^{6} \left(\frac{\xi^{5} \chi_{\leq \frac{1}{4}}(\xi)}{\log^{b-1}(\frac{1}{\xi})}\right)d\xi\end{split}\end{equation}

We will not need as precise of a description of $\partial_{r}^{2} v_{2}$ here, since such a description for small $r$ will be obtained later using the equation solved by $v_{2}$. So, we use the identical procedure used for $\partial_{r}v_{2}$, and get
\begin{equation}\begin{split} |\partial_{r}^{2}v_{2}(t,r)| &\leq \frac{C r}{t^{4} \log^{b}(t)}, \quad r \leq \frac{t}{2}\end{split}\end{equation}
For the larger $r$ estimates, we can again use $$|J_{k}(x)| \leq \frac{C}{\sqrt{x}}, \quad k=1,3$$ to get
\begin{equation} |\partial_{r}^{2}v_{2}(t,r)| \leq C \int_{0}^{\infty} \frac{\xi^{2}}{\sqrt{r\xi}} \frac{\chi_{\leq \frac{1}{4}}(\xi)}{\log^{b-1}(\frac{1}{\xi})} d\xi \leq \frac{C}{\sqrt{r}}\end{equation}
Lastly, for another estimate for large $r$, we first split $\partial_{r}^{2}v_{2}$ as follows
\begin{equation} \partial_{r}^{2}v_{2}(t,r) = I_{rr}+II_{rr}\end{equation}
and
\begin{equation} I_{rr} = \frac{c_{b}}{4} \int_{0}^{\infty} \sin(t\xi) \xi^{2} \chi_{\leq 1}(r\xi) \left(-3 J_{1}(r\xi) + J_{3}(r\xi)\right) \frac{\chi_{\leq \frac{1}{4}}(\xi)}{\log^{b-1}(\frac{1}{\xi})} d\xi\end{equation}
\begin{equation} II_{rr} = \frac{c_{b}}{4} \int_{0}^{\infty} \sin(t\xi) \xi^{2} \left(1-\chi_{\leq 1}(r\xi)\right) \left(-3 J_{1}(r\xi) + J_{3}(r\xi)\right) \frac{\chi_{\leq \frac{1}{4}}(\xi)}{\log^{b-1}(\frac{1}{\xi})} d\xi\end{equation}
which give
\begin{equation} |I_{rr}| \leq C \int_{0}^{\text{min}\{\frac{1}{r}, \frac{1}{4}\}} \frac{\xi^{2} r\xi}{\log^{b-1}(\frac{1}{\xi})} d\xi \leq \frac{C}{r^{3}\log^{b-1}(r)}, \quad r \geq 4\end{equation}
Next, we again use \eqref{jnktv} for $J_{k}, \quad k=1,3$, and integrate by parts 3 times, to get
\begin{equation}\begin{split} |II_{rr}| &\leq \sum_{k \in \{1,3\}} \frac{C}{|t-r|^{3}} \int_{0}^{\infty} |\partial_{\xi}^{3}\left(\frac{\xi^{2} \chi_{\leq \frac{1}{4}}(\xi)}{\log^{b-1}(\frac{1}{\xi})} (r\xi)^{k} \Phi_{k}(r\xi) \left(1-\chi_{\leq 1}(r\xi)\right)\right)|d\xi\\
&\leq \frac{C}{|t-r|^{3}} \int_{0}^{\infty}\left(\frac{\mathbbm{1}_{\{r\xi \geq \frac{1}{2}\}} \mathbbm{1}_{\{\xi \leq \frac{1}{4}\}}}{\sqrt{r} \xi^{3/2} \log^{b-1}(\frac{1}{\xi})} + \sqrt{\frac{r}{\xi}} \frac{\mathbbm{1}_{\{1 \geq r \xi \geq \frac{1}{2}\}} \mathbbm{1}_{\{\xi \leq \frac{1}{4}\}}}{\log^{b-1}(\frac{1}{\xi})}\right)d\xi\\
&\leq \frac{C}{|t-r|^{3}\log^{b-1}(r)}, \quad r > \frac{t}{2}\end{split}\end{equation}
Combining the above estimates, we get \eqref{drrv2mainest}

Next, we study $\partial_{t}^{2} v_{2}$ in the region $r \leq \frac{t}{2}$. We proceed just as for $v_{2}$:
\begin{equation} \begin{split} \partial_{t}^{2}v_{2}(t,r) &= \frac{-c_{b}}{2\pi} \int_{0}^{\pi} \sin^{2}(\theta)d\theta \int_{0}^{\infty} d\xi r \left(\frac{\sin(\xi t_{+})}{t_{+}^{4}}+\frac{\sin(\xi t_{-})}{t_{-}^{4}}\right) \partial_{\xi}^{4} \left(\frac{\chi_{\leq \frac{1}{4}}(\xi) \cdot \xi^{3}}{\log^{b-1}(\frac{1}{\xi})}\right)\end{split}\end{equation}
We then use
\begin{equation}\begin{split}&\partial_{\xi}^{4} \left(\frac{\chi_{\leq \frac{1}{4}}(\xi) \cdot \xi^{3}}{\log^{b-1}(\frac{1}{\xi})}\right)\\
&= \chi_{\leq \frac{1}{4}}(\xi)\left(\frac{6 b \left(b^2-1\right) }{\xi \log ^{b+2}\left(\frac{1}{\xi }\right) }+\frac{(b-1) b (b+1) (b+2) }{\xi \log ^{b+3}\left(\frac{1}{\xi }\right) }+\frac{11 (b-1) b }{\xi \log ^{b+1}\left(\frac{1}{\xi }\right) }+\frac{6 (b-1) }{\xi \log ^{b}\left(\frac{1}{\xi }\right) }\right)+\psi(\xi)\\
&, \quad \psi \in C^{\infty}_{c}([\frac{1}{8},\frac{1}{4}])\end{split}\end{equation}
The integral is of the same form as that considered during the $v_{2}$ estimates, and we get \eqref{v2precisenearorigin} for $\partial_{t}^{2}v_{2}$. (Note that any other estimates on $\partial_{t}^{2} v_{2}$ which we may need can be obtained using the other estimates of the lemma, and the equation solved by $v_{2}$).

Finally, we will need a similar formula for $\partial_{tr}v_{2}$: We have
\begin{equation} \partial_{tr}v_{2}(t,r) = \frac{c_{b}}{2} \int_{0}^{\infty} \cos(t\xi) \xi^{2} (J_{0}(r\xi)-J_{2}(r\xi)) \frac{\chi_{\leq \frac{1}{4}}(\xi) d\xi}{\log^{b-1}(\frac{1}{\xi})}\end{equation}
Applying \eqref{besselrep}, the $J_{0}$ term is
\begin{equation} \frac{c_{b}}{4\pi} \int_{0}^{\pi} d\theta \int_{0}^{\infty} d\xi \left(\frac{\sin(\xi t_{+})}{t_{+}^{3}}+\frac{\sin(\xi t_{-})}{t_{-}^{3}}\right) \partial_{\xi}^{3} \left(\frac{\xi^{2} \chi_{\leq \frac{1}{4}}(\xi)}{\log^{b-1}(\frac{1}{\xi})}\right)\end{equation}
Comparing this to the analogous integral treated while studying $\partial_{t}v_{2}$, we get
\begin{equation}\begin{split}&\frac{c_{b}}{4\pi} \int_{0}^{\pi} d\theta \int_{0}^{\infty} d\xi \left(\frac{\sin(\xi t_{+})}{t_{+}^{3}}+\frac{\sin(\xi t_{-})}{t_{-}^{3}}\right) \partial_{\xi}^{3} \left(\frac{\xi^{2} \chi_{\leq \frac{1}{4}}(\xi)}{\log^{b-1}(\frac{1}{\xi})}\right)\\
&=\frac{2b}{t^{3} \log^{b}(t)} + E_{\partial_{12}v_{2},1}(t,r)\end{split}\end{equation}
where
$$|E_{\partial_{12}v_{2},1}(t,r)| \leq C\left(\frac{r}{t^{4} \log^{b}(t)}+\frac{1}{t^{3} \log^{b+1}(t)}\right),\quad r \leq \frac{t}{2}$$
For the $J_{2}$ term, we use \eqref{besselrep}, and get
\begin{equation} \begin{split} &\frac{c_{b}}{2} \int_{0}^{\infty} \cos(t\xi) \xi^{2} (-J_{2}(r\xi)) \frac{\chi_{\leq \frac{1}{4}}(\xi) d\xi}{\log^{b-1}(\frac{1}{\xi})}\\
&=\frac{c_{b}r^{2}}{12 \pi} \int_{0}^{\pi} \sin^{4}(\theta)d\theta \int_{0}^{\infty} \left(\frac{\sin(\xi t_{+})}{t_{+}^{5}}+\frac{\sin(\xi t_{-})}{t_{-}^{5}}\right) \partial_{\xi}^{5}\left(\frac{\xi^{4} \chi_{\leq \frac{1}{4}}(\xi)}{\log^{b-1}(\frac{1}{\xi})}\right) d\xi\end{split}\end{equation}
we then treat this integral in exactly the same was as the $J_{0}$ term, and get
\begin{equation} |\frac{c_{b}}{2} \int_{0}^{\infty} \cos(t\xi) \xi^{2} (-J_{2}(r\xi)) \frac{\chi_{\leq \frac{1}{4}}(\xi) d\xi}{\log^{b-1}(\frac{1}{\xi})}| \leq \frac{C r^{2}}{t^{5}\log^{b}(t)}, \quad r \leq \frac{t}{2}\end{equation}
This gives \eqref{v2precisenearorigin} for $\partial_{tr}v_{2}$.\\
\\
Finally, \eqref{v2sqrtrest} is proven for $\partial_{t}^{2}v_{2}$ and $\partial_{tr}v_{2}$, and \eqref{v2singularconeest} is proven for $\partial_{tr}v_{2}$ in exactly the same way as for the other derivatives.\\
\\
The identical procedure shows that \eqref{v2precisenearorigin} is still true for $b=1$. In the region $r \geq \frac{t}{2}$, \eqref{v2sqrtrest} is still true for $b=1$, and for the second estimate in the region $r \geq \frac{t}{2}$, we get\\
 for $0 \leq j \leq 1, \quad 0 \leq k \leq 1, \quad j+k \leq 2$
\begin{equation}|\partial_{t}^{j}\partial_{r}^{k}v_{2}(t,r)| \leq \frac{C}{|t-r|^{1+j+k}} \log(\log(r)), \quad t \neq r \geq \frac{t}{2}, \quad b=1\end{equation}
(Note the difference with the case $b \neq 1$, which had $\frac{C}{|t-r|^{1+j+k} \log^{b-1}(r)}$ on the right hand side). In any case, \eqref{v2singularconeest} is true for all $b>0$ (for simplicity, we use $\log(r)$ instead of $\begin{cases} \log^{1-b}(r), \quad b \neq 1\\
\log(\log(r)), \quad b=1\end{cases}$ in \eqref{v2singularconeest}). \end{proof}

\subsubsection{The inner product of the (rescaled) $v_{3}$ linear error term}
\begin{lemma}For $v_{3}$ defined in \eqref{v3def}, we have
\begin{equation} \begin{split} &\int_{0}^{\infty} \left(\frac{\cos(2Q_{1}(R))-1}{R^{2}\lambda(t)^{2}}\right)v_{3}(t,R\lambda(t)) \phi_{0}(R) R dR\\
&= \frac{16}{\lambda(t)} \int_{t}^{\infty} K_{3}(s-t,\lambda(t)) \lambda''(s) ds + \int_{0}^{\infty} \left(\frac{\cos(2Q_{1}(R))-1}{R^{2}\lambda(t)^{2}}\right)E_{5}(t,R\lambda(t)) \phi_{0}(R) R dR\end{split}\end{equation}
where
$E_{5}$ is as in \eqref{v3preciseforip}, 
\begin{equation}\begin{split}K_{3}(w,\lambda(t))&= \left(\frac{w}{1+w^{2}}-\frac{w}{\lambda(t)^{2-2\alpha}+w^{2}}\right)\frac{w^{4}}{4(w^{2}+36\lambda(t)^{2})^{2}}\end{split}\end{equation}
and
\begin{equation}\label{k3minusk30} \int_{0}^{\infty}|K_{3}(w,\lambda(t))-K_{3,0}(w,\lambda(t))|dw \leq C \end{equation}
where 
$$K_{3,0}(w,\lambda(t)) = -\frac{1}{4(\lambda(t)^{1-\alpha}+w)(1+w)^{3}}$$
\end{lemma}
\begin{proof}

We start with
\begin{equation}\label{v3ip} \begin{split} &\int_{0}^{\infty} \left(\frac{\cos(2Q_{1}(R))-1}{R^{2}\lambda(t)^{2}}\right)v_{3}(t,R\lambda(t)) \phi_{0}(R) R dR\\
&=\frac{16}{\lambda(t)} \int_{0}^{\infty} \frac{R^{3}}{(1+R^{2})^{3}} \int_{t+6 R \lambda(t)}^{\infty} ds \lambda''(s) \left(\frac{s-t}{1+(s-t)^{2}}-\frac{(s-t)}{\lambda(t)^{2-2\alpha}+(s-t)^{2}}\right)\\
&+\int_{0}^{\infty} \left(\frac{\cos(2Q_{1}(R))-1}{R^{2}\lambda(t)^{2}}\right)E_{5}(t,R\lambda(t)) \phi_{0}(R) R dR\end{split}\end{equation}
where we recall the decomposition \eqref{v3preciseforip}.\\
We study the second line of \eqref{v3ip} in more detail, since the third line is as in the lemma statement.
\begin{equation} \begin{split} &\frac{16}{\lambda(t)} \int_{0}^{\infty} \frac{R^{3}}{(1+R^{2})^{3}} \int_{t+6 R \lambda(t)}^{\infty} ds \lambda''(s) \left(\frac{s-t}{1+(s-t)^{2}}-\frac{(s-t)}{\lambda(t)^{2-2\alpha}+(s-t)^{2}}\right)\\
&=\frac{16}{\lambda(t)} \int_{t}^{\infty} K_{3}(s-t,\lambda(t)) \lambda''(s) ds\end{split}\end{equation}
where
\begin{equation}\begin{split}K_{3}(w,\lambda(t))&= \left(\frac{w}{1+w^{2}}-\frac{w}{\lambda(t)^{2-2\alpha}+w^{2}}\right)\frac{w^{4}}{4(w^{2}+36\lambda(t)^{2})^{2}}\end{split}\end{equation}

We now estimate
\begin{equation} \int_{0}^{\infty}|K_{3}(w,\lambda(t))-K_{3,0}(w,\lambda(t))|dw \end{equation}

where 
$$K_{3,0}(w,\lambda(t)) = -\frac{1}{4(\lambda(t)^{1-\alpha}+w)(1+w)^{3}}$$

We start with
\begin{equation}\begin{split} &\int_{0}^{\lambda(t)^{1-\alpha}} |K_{3}(w,\lambda(t))| dw + \int_{0}^{\lambda(t)^{1-\alpha}} \frac{dw}{4(\lambda(t)^{1-\alpha}+w)(1+w)^{3}}\\
&\leq C \int_{0}^{\lambda(t)^{1-\alpha}} \frac{w dw}{(1+w^{2})} + C \int_{0}^{\lambda(t)^{1-\alpha}} \frac{w dw}{\lambda(t)^{2-2\alpha}} +C \int_{0}^{\lambda(t)^{1-\alpha}} \frac{dw}{\lambda(t)^{1-\alpha}}\\
&\leq C\end{split}\end{equation}
Next,
\begin{equation}\label{v3ip2terms}\begin{split} &\int_{\lambda(t)^{1-\alpha}}^{\frac{1}{2}} |K_{3}(w,\lambda(t))+\frac{1}{4(\lambda(t)^{1-\alpha}+w)(1+w)^{3}}|dw\\
&\leq \int_{\lambda(t)^{1-\alpha}}^{\frac{1}{2}} |w \left(\frac{1}{1+w^{2}}-\frac{1}{\lambda(t)^{2-2\alpha}+w^{2}}\right)\left(\frac{w^{4}}{4(w^{2}+36\lambda(t)^{2})^{2}}-\frac{1}{4}\right)|dw\\
&+\int_{\lambda(t)^{1-\alpha}}^{\frac{1}{2}} |\frac{w}{4}\left(\frac{1}{1+w^{2}}-\frac{1}{\lambda(t)^{2-2\alpha}+w^{2}}\right)+\frac{1}{4(\lambda(t)^{1-\alpha}+w)}| dw\\
&+\int_{\lambda(t)^{1-\alpha}}^{\frac{1}{2}} |\frac{-1}{4(\lambda(t)^{1-\alpha}+w)}+\frac{1}{4(\lambda(t)^{1-\alpha}+w)(1+w)^{3}}| dw\end{split}\end{equation}
For the second line of \eqref{v3ip2terms}, we have
\begin{equation}\begin{split}&\int_{\lambda(t)^{1-\alpha}}^{\frac{1}{2}} |w \left(\frac{1}{1+w^{2}}-\frac{1}{\lambda(t)^{2-2\alpha}+w^{2}}\right)\left(\frac{w^{4}}{4(w^{2}+36\lambda(t)^{2})^{2}}-\frac{1}{4}\right)|dw\\
&\leq C \int_{\lambda(t)^{1-\alpha}}^{\frac{1}{2}} w\left(\frac{1}{\lambda(t)^{2-2\alpha}+w^{2}}-\frac{1}{1+w^{2}}\right)\frac{\lambda(t)^{2}}{w^{2}} dw\\
&\leq C \lambda(t)^{2\alpha} |\log(\lambda(t)^{2-2\alpha})| \leq C\end{split}\end{equation}
For the third line of \eqref{v3ip2terms}, we have
\begin{equation}\begin{split}&\int_{\lambda(t)^{1-\alpha}}^{\frac{1}{2}} |\frac{w}{4}\left(\frac{1}{1+w^{2}}-\frac{1}{\lambda(t)^{2-2\alpha}+w^{2}}\right)+\frac{1}{4(\lambda(t)^{1-\alpha}+w)}| dw\\
&\leq \int_{\lambda(t)^{1-\alpha}}^{\frac{1}{2}} \frac{w dw}{4(1+w^{2})} + \int_{\lambda(t)^{1-\alpha}}^{\frac{1}{2}} dw |\frac{1}{\lambda(t)^{1-\alpha}+w}-\frac{w}{\lambda(t)^{2-2\alpha}+w^{2}}|\\
&\leq C \end{split}\end{equation}
Finally, the last line of \eqref{v3ip2terms} is estimated as follows:
\begin{equation}\begin{split}&\int_{\lambda(t)^{1-\alpha}}^{\frac{1}{2}} |\frac{-1}{4(\lambda(t)^{1-\alpha}+w)}+\frac{1}{4(\lambda(t)^{1-\alpha}+w)(1+w)^{3}}| dw\\
&\leq C \int_{\lambda(t)^{1-\alpha}}^{\frac{1}{2}} \frac{w dw}{(\lambda(t)^{1-\alpha}+w)}\\
&\leq C\end{split}\end{equation}
Then, we consider
\begin{equation} \begin{split} &\int_{\frac{1}{2}}^{\infty} dw |K_{3}(w,\lambda(t))+\frac{1}{4(\lambda(t)^{1-\alpha}+w)(1+w)^{3}}|\\
&\leq \int_{\frac{1}{2}}^{\infty} dw |K_{3}(w,\lambda(t))| + \int_{\frac{1}{2}}^{\infty} \frac{dw}{(\lambda(t)^{1-\alpha}+w)(1+w)^{3}}\\
&\leq C \int_{\frac{1}{2}}^{\infty} dw \frac{w|\lambda(t)^{2-2\alpha}-1|}{w^{4}} + C\\
&\leq C\end{split}\end{equation}

Combining the above, we get that
\begin{equation} \int_{0}^{\infty}|K_{3}(w,\lambda(t))-K_{3,0}(w,\lambda(t))|dw \leq C \end{equation}
\end{proof}

\subsubsection{Leading behavior of $\lambda$ and set-up of the modulation equation}
From the previous subsections, we have
\begin{equation} \begin{split} &\langle F_{4}(t,\cdot \lambda(t)), \phi_{0}\rangle \\
&= \frac{-16}{\lambda(t)^{3}} \int_{t}^{\infty} ds \lambda''(s) K_{1}(s-t,\lambda(t)) + \frac{4 b}{\lambda(t) t^{2}\log^{b}(t)} + \frac{4 \alpha \log(\lambda(t)) \lambda''(t)}{\lambda(t)}+ \frac{16}{\lambda(t)} \int_{t}^{\infty} ds \lambda''(s) K_{3,0}(s-t,\lambda(t))\\
&+E_{0,1}(\lambda(t),\lambda'(t),\lambda''(t)) + \frac{16}{\lambda(t)} \int_{t}^{\infty} ds \lambda''(s) \left(K_{3}(s-t,\lambda(t))-K_{3,0}(s-t,\lambda(t))\right) \\
&+\frac{-16}{\lambda(t)^{3}} \int_{t}^{\infty} ds \lambda''(s) K(s-t,\lambda(t)) + E_{v_{2},ip}(t,\lambda(t)) \\
&+ \langle \left(\frac{\cos(2Q_{\frac{1}{\lambda(t)}})-1}{r^{2}}\right)\left((v_{4}+v_{5})\left(1-\chi_{\geq 1}(\frac{4r}{t})\right)+E_{5}-\chi_{\geq 1}(\frac{2r}{\log^{N}(t)})\left(v_{1}+v_{2}+v_{3}\right)\right)\vert_{r=R\lambda(t)},\phi_{0}\rangle\\
&-\langle \chi_{\geq 1}(\frac{2r}{\log^{N}(t)}) F_{0,2}(t,r) \vert_{r=R\lambda(t)},\phi_{0}\rangle \end{split}\end{equation}
where
\begin{equation}\begin{split}E_{0,1}(\lambda(t),\lambda'(t),\lambda''(t)) &= 2\frac{\lambda'(t)^{2}}{\lambda(t)^{2}} + \frac{2 \lambda''(t)}{\lambda(t)}-\frac{4 \alpha \lambda''(t) \log(\lambda(t))}{\lambda(t)} \left(\frac{1}{-1+\lambda(t)^{2\alpha}}+1\right)\end{split}\end{equation}
So, the equation resulting from 
$$\langle F_{4}(t),\phi_{0}(\frac{\cdot}{\lambda(t)})\rangle =0$$
is what we recorded at the beginning of this section, namely \eqref{modulationfinal1}. For convenience, we repeat the equation here.
\begin{equation}\label{modulationfinal}\begin{split} &-4 \int_{t}^{\infty} \frac{\lambda''(s)}{1+s-t} ds + \frac{4 b}{t^{2}\log^{b}(t)} + 4 \alpha \log(\lambda(t)) \lambda''(t) - 4 \int_{t}^{\infty} \frac{\lambda''(s)}{(\lambda(t)^{1-\alpha}+s-t)(1+s-t)^{3}} ds\\
&= -\lambda(t) E_{0,1}(\lambda(t),\lambda'(t),\lambda''(t)) - 16 \int_{t}^{\infty} \lambda''(s) \left(K_{3}(s-t,\lambda(t))-K_{3,0}(s-t,\lambda(t))\right) ds\\
&+\frac{16}{\lambda(t)^{2}} \int_{t}^{\infty} K(s-t,\lambda(t)) \lambda''(s) ds - \lambda(t) E_{v_{2},ip}(t,\lambda(t)) +\frac{16}{\lambda(t)^{2}} \int_{t}^{\infty} ds \lambda''(s) \left(K_{1}(s-t,\lambda(t))-\frac{\lambda(t)^{2}}{4(1+s-t)}\right)\\
&- \lambda(t) \langle \left(\frac{\cos(2Q_{\frac{1}{\lambda(t)}})-1}{r^{2}}\right)\left((v_{4}+v_{5})\left(1-\chi_{\geq 1}(\frac{4r}{t})\right)+E_{5}-\chi_{\geq 1}(\frac{2r}{\log^{N}(t)})\left(v_{1}+v_{2}+v_{3}\right)\right)\vert_{r=R\lambda(t)},\phi_{0}\rangle\\
&+\lambda(t) \langle \chi_{\geq 1}(\frac{2r}{\log^{N}(t)}) F_{0,2}(t,r) \vert_{r=R\lambda(t)},\phi_{0}\rangle\\
&:=G(t,\lambda(t))\end{split}\end{equation}

The key point is that there is leading order cancellation between the four terms on the first line of \eqref{modulationfinal} (which means cancellation of terms of size $\frac{1}{t^{2}\log^{b}(t)}$ and terms of size $\frac{\log(\log(t))}{t^{2} \log^{b+1}(t)}$) when we substitute $\lambda=\lambda_{0}$ into these terms, where
$$ \lambda_{0}(t) = \frac{1}{\log^{b}(t)}+ \int_{t}^{\infty} \int_{t_{1}}^{\infty} \frac{-b^{2} \log(\log(t_{2}))}{t_{2}^{2}\log^{b+2}(t_{2})}dt_{2}dt_{1}:=\lambda_{0,0}+\lambda_{0,1}$$
In order to show this cancellation, we first let $\lambda(t) = \lambda_{0}(t) + e(t)$, and re-write the equation for $e$ in the following way:

\begin{equation}\label{eeqnv1}\begin{split} &-4 \int_{t}^{\infty} \frac{\lambda_{0,0}''(s)}{1+s-t} ds + \frac{4b}{t^{2}\log^{b}(t)}\\
&-4 \int_{t}^{\infty} \frac{\lambda_{0,1}''(s)}{1+s-t} ds + 4 \alpha \log(\lambda_{0,0}(t)) \lambda_{0,0}''(t) - 4 \int_{t}^{\infty} \frac{\lambda_{0,0}''(s) ds}{(\lambda_{0,0}(t)^{1-\alpha}+s-t)(1+s-t)^{3}}\\
&-4 \int_{t}^{\infty} \frac{e''(s) ds}{1+s-t} + 4 \alpha \log(\lambda_{0}(t)) e''(t) - 4 \int_{t}^{\infty} \frac{e''(s) ds}{(\lambda_{0}(t)^{1-\alpha}+s-t)(1+s-t)^{3}}\\
& + 4 \alpha (\log(\lambda(t))-\log(\lambda_{0,0}(t)))\lambda_{0,0}''(t)+4 \alpha \log(\lambda(t)) \left(\lambda_{0}''(t)-\lambda_{0,0}''(t)\right) + 4\alpha e''(t) \left(\log(\lambda(t))-\log(\lambda_{0}(t))\right)\\
&-4 \int_{t}^{\infty} \frac{e''(s) ds}{(1+s-t)^{3}}\left(\frac{1}{\lambda(t)^{1-\alpha}+s-t}-\frac{1}{\lambda_{0}(t)^{1-\alpha}+s-t}\right)\\
&-4 \int_{t}^{\infty} \frac{(\lambda_{0}''(s)-\lambda_{0,0}''(s))ds}{(\lambda_{0,0}(t)^{1-\alpha}+s-t)(1+s-t)^{3}}-4 \int_{t}^{\infty} \frac{\lambda_{0}''(s) ds}{(1+s-t)^{3}} \left(\frac{1}{(\lambda(t)^{1-\alpha}+s-t)}-\frac{1}{(\lambda_{0,0}(t)^{1-\alpha}+s-t)}\right)\\
&=G(t,\lambda(t))\end{split}\end{equation}

Now, we will show that the terms in each of the first two lines of \eqref{eeqnv1} cancel to leading order. More precisely, this means that both terms in the first line of \eqref{eeqnv1} are of size $\frac{1}{t^{2} \log^{b}(t)}$, but their sum has size bounded above by $\frac{1}{t^{2} \log^{b+1}(t)}$. Similarly, each term on the second line of \eqref{eeqnv1} is of size $\frac{\log(\log(t))}{t^{2}\log^{b+1}(t)}$, but their sum has size bounded above by $\frac{1}{t^{2} \log^{b+1}(t)}$.

\begin{equation}\begin{split}-4 \int_{t}^{\infty} ds \frac{\lambda_{0,0}''(s)}{1+s-t} &= -4\int_{t}^{2t} ds \frac{\lambda_{0,0}''(s)}{1+s-t} -4\int_{2t}^{\infty} ds \frac{\lambda_{0,0}''(s)}{1+s-t}
\\&=-4 \lambda_{0,0}''(t)\int_{t}^{2t} \frac{ds}{1+s-t} -4\int_{t}^{2t} ds \frac{\lambda_{0,0}''(s)-\lambda_{0,0}''(t)}{1+s-t} -4\int_{2t}^{\infty} ds \frac{\lambda_{0,0}''(s)}{1+s-t}\\
&=-\frac{4b}{t^{2}\log^{b+1}(t)}\log(1+t) +O(\frac{1}{t^{2}\log^{b+1}(t)}) +\text{Err}\end{split}\end{equation} 
where
\begin{equation}\begin{split}|\text{Err}| &\leq 4 \int_{t}^{2t} ds \frac{||\lambda_{0,0}'''||_{L^{\infty}(t,2t)}(s-t)}{1+s-t} + \frac{4}{1+t}\int_{2t}^{\infty} ds |\lambda_{0,0}''(s)|\\
&\leq \frac{C}{t^{2}\log^{b+1}(t)}\end{split}\end{equation}
So, we get 
\begin{equation}\label{lambda00comp} -4 \int_{t}^{\infty} ds \frac{\lambda_{0,0}''(s)}{1+s-t} = -\frac{4b}{t^{2}\log^{b}(t)}+E_{\lambda_{0,0}}\end{equation}
where
$$|E_{\lambda_{0,0}}| \leq \frac{C}{t^{2}\log^{b+1}(t)}$$
Next, we have
\begin{equation}\begin{split} -4 \int_{t}^{\infty} \frac{\lambda_{0,0}''(x) dx}{(\lambda_{0,0}(t)^{1-\alpha}+x-t)(1+x-t)^{3}} &= -4 \int_{t}^{\infty} \frac{b}{x^{2}\log^{b+1}(x)} \frac{dx}{(\lambda_{0,0}(t)^{1-\alpha}+x-t)(1+x-t)^{3}} + E_{v_{3,ip}}\end{split}\end{equation}
where
$$|E_{v_{3,ip}}| \leq C \int_{t}^{\infty} \frac{dx}{x^{2}\log^{b+2}(x) (\lambda_{0,0}(t)^{1-\alpha}+x-t)(1+x-t)^{3}}$$
Then, we have
\begin{equation} \begin{split} &|-4b \int_{t}^{t+\lambda_{0,0}(t)^{1-\alpha}} \frac{dx}{x^{2}\log^{b+1}(x) (\lambda_{0,0}(t)^{1-\alpha} +x-t)(1+x-t)^{3}}|\\
&\leq C \int_{t}^{t+\lambda_{0,0}(t)^{1-\alpha}} \frac{dx}{x^{2} \log^{b+1}(x)} \frac{1}{\lambda_{0,0}(t)^{1-\alpha}}\\
&\leq \frac{C}{t^{2} \log^{b+1}(t)}\end{split}\end{equation}

The second term to consider is
\begin{equation}\label{v3lambda001b}\begin{split} &-4 b \int_{t+\log^{(\alpha-1)b}(t)}^{t+\log^{(\alpha-1)b}(t)+\frac{1}{2}} \frac{1}{x^{2}\log^{b+1}(x)} \frac{1}{(\log^{(\alpha-1)b}(t)+x-t)}\frac{dx}{(1+x-t)^{3}}\\
&=\frac{-4b}{t^{2}\log^{b+1}(t)} \int_{t+\log^{(\alpha-1)b}(t)}^{t+\log^{(\alpha-1)b}(t)+\frac{1}{2}} \frac{dx}{(\log^{(\alpha-1)b}(t)+x-t)}\\
&-4 b \int_{t+\log^{(\alpha-1)b}(t)}^{t+\log^{(\alpha-1)b}(t)+\frac{1}{2}} \left(\frac{1}{x^{2}\log^{b+1}(x)}-\frac{1}{t^{2}\log^{b+1}(t)}\right) \frac{1}{(\log^{(\alpha-1)b}(t)+x-t)} dx\\
&-4b \int_{t+\log^{(\alpha-1)b}(t)}^{t+\log^{(\alpha-1)b}(t)+\frac{1}{2}} \frac{dx}{x^{2}\log^{b+1}(x)} \frac{1}{(\log^{(\alpha-1)b}(t)+x-t)} \left(\frac{1}{(1+x-t)^{3}}-1\right)\end{split}\end{equation}

The second line of \eqref{v3lambda001b} is  treated as follows:
\begin{equation}\begin{split}&\frac{-4b}{t^{2}\log^{b+1}(t)} \int_{t+\log^{(\alpha-1)b}(t)}^{t+\log^{(\alpha-1)b}(t)+\frac{1}{2}} \frac{dx}{(\log^{(\alpha-1)b}(t)+x-t)}|\\
&= \frac{-4b}{t^{2}\log^{b+1}(t)} \left(\log(2\log^{(\alpha-1)b}(t)+\frac{1}{2}) - \log(2\log^{(\alpha-1)b}(t))\right)\\
&=\frac{4b^{2}(\alpha-1) \log(\log(t))}{t^{2}\log^{b+1}(t)} + E_{v_{3,ip},1b}\end{split}\end{equation}

where
$$|E_{v_{3,ip},1b}| \leq \frac{C}{t^{2}\log^{b+1}(t)}$$
The third line of \eqref{v3lambda001b} is estimated by:
\begin{equation}\begin{split}&|-4 b \int_{t+\log^{(\alpha-1)b}(t)}^{t+\log^{(\alpha-1)b}(t)+\frac{1}{2}} \left(\frac{1}{x^{2}\log^{b+1}(x)}-\frac{1}{t^{2}\log^{b+1}(t)}\right) \frac{1}{(\log^{(\alpha-1)b}(t)+x-t)} dx|\\
&\leq \frac{C}{t^{3}\log^{b+1}(t)} \int_{t+\log^{(\alpha-1)b}(t)}^{t+\log^{(\alpha-1)b}(t)+\frac{1}{2}} \frac{(x-t) dx}{(\log^{(\alpha-1)b}(t)+x-t)}\\
&\leq \frac{C}{t^{3}\log^{b+1}(t)}\end{split}\end{equation}
The last line of \eqref{v3lambda001b} is estimated by
\begin{equation} \begin{split}&|-4b \int_{t+\log^{(\alpha-1)b}(t)}^{t+\log^{(\alpha-1)b}(t)+\frac{1}{2}} \frac{dx}{x^{2}\log^{b+1}(x)} \frac{1}{(\log^{(\alpha-1)b}(t)+x-t)} \left(\frac{1}{(1+x-t)^{3}}-1\right)|\\
&\leq C \int_{t+\log^{(\alpha-1)b}(t)}^{t+\log^{(\alpha-1)b}(t)+\frac{1}{2}} \frac{dx}{x^{2}\log^{b+1}(x)} \frac{x-t}{(\log^{(\alpha-1)b}(t)+x-t)}\\
&\leq \frac{C}{t^{2}\log^{b+1}(t)}\end{split}\end{equation}

The thrid term to consider is
\begin{equation}\begin{split} &|-4 b \int_{t+\log^{(\alpha-1)b}(t)+\frac{1}{2}}^{\infty} \frac{1}{x^{2}\log^{b+1}(x)} \frac{1}{(\log^{(\alpha-1)b}(t)+x-t)} \frac{dx}{(1+x-t)^{3}}|\\
& \leq \frac{C}{t^{2}\log^{b+1}(t)} \int_{t+\log^{(\alpha-1)b}(t)+\frac{1}{2}}^{\infty} \frac{dx}{(x-t)^{4}}\\
&\leq \frac{C}{t^{2}\log^{b+1}(t)}\end{split}\end{equation}

Then, we estimate $E_{v_{3},ip}$:
\begin{equation}\begin{split} |E_{v_{3},ip}| &\leq \frac{C}{t^{2}\log^{b+2}(t)} \int_{t}^{\infty} \frac{dx}{(\log^{(\alpha-1)b}(t)+x-t)(1+x-t)^{3}}\\
&\leq \frac{C \log(\log(t))}{t^{2}\log^{b+2}(t)}\end{split}\end{equation}

In total, we have
\begin{equation} -4 \int_{t}^{\infty} \frac{\lambda_{0,0}''(x) dx}{(\log^{(\alpha-1)b}(t)+x-t)(1+x-t)^{3}} = E_{v_{3},ip,f}+\frac{4b^{2} (\alpha-1)\log(\log(t))}{t^{2}\log^{b+1}(t)}\end{equation}
where
$$|E_{v_{3},ip,f}| \leq \frac{C}{t^{2}\log^{b+1}(t)}$$
By the same procedure used to study the analogous term involving $\lambda_{0,0}$, we have
\begin{equation} -4 \int_{t}^{\infty} \frac{\lambda_{0,1}''(s) ds}{1+s-t} = \frac{4b^{2} \log(\log(t))}{t^{2}\log^{b+1}(t)} + E_{v_{3,ip},01}\end{equation}
where
$$|E_{v_{3,ip},01}| \leq \frac{C \log(\log(t))}{t^{2}\log^{b+2}(t)}$$
Combining the above, we can show the cancellation, to leading order, of the terms on each of the first two lines of \eqref{eeqnv1}:
\begin{equation}\begin{split}&-4 \int_{t}^{\infty} \frac{\lambda_{0,0}''(s)}{1+s-t} ds + \frac{4b}{t^{2}\log^{b}(t)}=E_{\lambda_{0,0}}\end{split}\end{equation}
with
$$|E_{\lambda_{0,0}}| \leq \frac{C}{t^{2}\log^{b+1}(t)}$$
and
\begin{equation}\begin{split}
&-4 \int_{t}^{\infty} \frac{\lambda_{0,1}''(s)}{1+s-t} ds + 4 \alpha \log(\lambda_{0,0}(t)) \lambda_{0,0}''(t) - 4 \int_{t}^{\infty} \frac{\lambda_{0,0}''(s) ds}{(\lambda_{0,0}(t)^{1-\alpha}+s-t)(1+s-t)^{3}}\\
&=E_{v_{3},ip,01}+E_{v_{3},ip,f}+E_{\lambda_{0,1}}\end{split}\end{equation}
with
$$E_{\lambda_{0,1}}(t) = 4 \alpha \log(\lambda_{0,0}(t)) \left(\lambda_{0,0}''(t) - \frac{b}{t^{2}\log^{b+1}(t)}\right)$$
$$|E_{v_{3},ip,01}| + |E_{v_{3},ip,f}|+|E_{\lambda_{0,1}}| \leq \frac{C}{t^{2}\log^{b+1}(t)}$$

We recall that we are considering \eqref{modulationfinal} for $t \in [T_{0},\infty)$, with $T_{0}$ satisfying \eqref{T0initialconstraint}.\\
\\
The equation for $e$ can be written as follows:

\begin{equation}\label{emodulation}\begin{split}&-4 \int_{t}^{\infty} \frac{e''(s) ds}{\log(\lambda_{0}(s))(1+s-t)} + 4 \alpha e''(t) - 4 \int_{t}^{\infty} \frac{e''(s) ds}{\log(\lambda_{0}(s))(\lambda_{0}(t)^{1-\alpha}+s-t)(1+s-t)^{3}}\\
&= \frac{1}{\log(\lambda_{0}(t))}\left(-E_{\lambda_{0,0}}-E_{v_{3},ip,01}-E_{v_{3},ip,f}-E_{\lambda_{0,1}}\right)\\
&+\frac{1}{\log(\lambda_{0}(t))}\left(G(t,\lambda(t))- 4 \alpha (\log(\lambda(t))-\log(\lambda_{0,0}(t)))\lambda_{0,0}''(t)-4 \alpha \log(\lambda(t)) \left(\lambda_{0}''(t)-\lambda_{0,0}''(t)\right)\right)\\
& +\frac{1}{\log(\lambda_{0}(t))}\left(- 4\alpha e''(t) \left(\log(\lambda(t))-\log(\lambda_{0}(t))\right)\right)\\
&+\frac{1}{\log(\lambda_{0}(t))}\left(4 \int_{t}^{\infty} \frac{e''(s) ds}{(1+s-t)^{3}}\left(\frac{1}{\lambda(t)^{1-\alpha}+s-t}-\frac{1}{\lambda_{0}(t)^{1-\alpha}+s-t}\right)\right)\\
&+\frac{4}{\log(\lambda_{0}(t))} \int_{t}^{\infty} \frac{(\lambda_{0}''(s)-\lambda_{0,0}''(s))ds}{(\lambda_{0,0}(t)^{1-\alpha}+s-t)(1+s-t)^{3}}\\
&+\frac{4}{\log(\lambda_{0}(t))} \int_{t}^{\infty} \frac{\lambda_{0}''(s) ds}{(1+s-t)^{3}} \left(\frac{1}{(\lambda(t)^{1-\alpha}+s-t)}-\frac{1}{(\lambda_{0,0}(t)^{1-\alpha}+s-t)}\right)\\
&+4 \int_{t}^{\infty} e''(s) \left(\frac{1}{\log(\lambda_{0}(t))}-\frac{1}{\log(\lambda_{0}(s))}\right) \frac{1}{(1+s-t)} ds\\
&+4 \int_{t}^{\infty} e''(s) \left(\frac{1}{\log(\lambda_{0}(t))}-\frac{1}{\log(\lambda_{0}(s))}\right)\frac{1}{(\lambda_{0}(t)^{1-\alpha}+s-t)(1+s-t)^{3}} ds\\
&:=RHS(e,t)\end{split}\end{equation}
where $\lambda(t) = \lambda_{0}(t)+e(t)$.
In order to study this equation, let us first consider the problem of solving an equation of the form
$$-\int_{t}^{\infty} \frac{y(s)}{\log(\lambda_{0}(s))}\left(\frac{1}{1+s-t}+\frac{1}{(\lambda_{0}(t)^{1-\alpha}+s-t)(1+s-t)^{3}}\right) ds + \alpha  y(t) = F(t), \quad t \geq T_{0}$$
where 
\begin{equation}\label{Fassump}F \in C([T_{0},\infty)), \quad |F(x)| \leq \frac{C}{x^{2}}\end{equation}
If $$x(t) = y(-t), \quad H(t) = \frac{F(-t)}{\alpha}, \quad t \leq -T_{0}$$
then our equation becomes
$$-\int_{-\infty}^{t} \frac{x(s)}{\alpha \log(\lambda_{0}(-s))} \left(\frac{1}{1-s+t}+\frac{1}{(\lambda_{0}(-t)^{1-\alpha} -s+t)(1-s+t)^{3}}\right) ds + x(t) = H(t), \quad t \leq -T_{0}$$
which can be written as
\begin{equation}\label{conveqn}\int_{J}  x(s) K(t,s) ds + x(t) = H(t), \quad t \in J\end{equation}
where
$$J=(-\infty, -T_{0}]$$
and, for $(t,s) \in J^{2}$,
\begin{equation}\begin{split}K(t,s) &= -\frac{\mathbbm{1}_{\leq 0}(s-t)}{\alpha \log(\lambda_{0}(-s))} \left(\frac{1}{1-s+t}+\frac{1}{(\lambda_{0}(-t)^{1-\alpha}-s+t)(1-s+t)^{3}}\right)\\
&=\frac{\mathbbm{1}_{\leq 0}(s-t)}{\alpha |\log(\lambda_{0}(-s))|} \left(\frac{1}{1-s+t}+\frac{1}{(\lambda_{0}(-t)^{1-\alpha}-s+t)(1-s+t)^{3}}\right)\end{split}\end{equation}
Note that $K(t,s)=0$ when $s>t$, so $K$ is a (non-negative) Volterra kernel on $J^{2}$.\\
\\
Moreover, $K$ is of type $L^{\infty}_{\text{loc}}$ on $J$. To see this, it suffices to consider $C=[a,d] \subset J$ a compact subinterval, $g \in L^{1}(C), f \in L^{\infty}(C)$ with $||g||_{L^{1}(C)} \leq 1, ||f||_{L^{\infty}(C)} \leq 1$ and estimate
\begin{equation}\begin{split} &\int_{C} \int_{C}  |g(t)| |K(t,s)| |f(s)| ds dt\leq \int_{C} \frac{|g(t)|}{2\alpha} \left(\int_{a}^{t} \frac{ds}{1-s+t} + \frac{1}{\lambda_{0}(-t)^{1-\alpha}} \int_{-\infty}^{t} \frac{ds}{(1-s+t)^{3}}\right)dt\\ 
&\leq\frac{\log(1+d-a)+\frac{1}{2\lambda_{0}(-a)^{1-\alpha}}}{2\alpha} \end{split}\end{equation}
where we used the facts (which follow from \eqref{T0initialconstraint} and \eqref{lambdarestr}) 
$$t \mapsto \frac{1}{\lambda_{0}(-t)^{1-\alpha}} \text{ is decreasing}, \quad \frac{1}{|\log(\lambda_{0}(-t))|} \leq \frac{1}{2}, \quad t \in J$$\\
\\
Moreover, if $s \leq u \leq v \leq t$, and $(t,v,u,s) \in J^{4}$, then, 
\begin{equation} \label{kineq} K(v,s)K(t,u) \leq K(t,s)K(v,u)\end{equation}
To verify this, let us note that, by the given conditions on $s,u,v,t$, we have
$$1=\mathbbm{1}_{\leq 0}(s-v) =\mathbbm{1}_{\leq 0}(u-t) =\mathbbm{1}_{\leq 0}(s-t) = \mathbbm{1}_{\leq 0}(u-v)$$
and
$$K(t,s) = \frac{\mathbbm{1}_{\leq 0}(s-t)}{\alpha |\log(\lambda_{0}(-s))|} k(t,s)$$
for
$$k(t,s) = \frac{1}{1-s+t} + \frac{1}{(\lambda_{0}(-t)^{1-\alpha}-s+t)(1-s+t)^{3}}$$
So, it suffices to show that
$$\frac{\partial^{2}}{\partial s \partial t}\log(k(t,s)) \leq 0, \quad s \leq t$$ 
Then, we note that
$$\log(k(t,s)) = \log(\frac{1}{1-s+t})+\log\left(1+\frac{1}{(\lambda_{0}(-t)^{1-\alpha}-s+t)(1-s+t)^{2}}\right)$$
and we have
\begin{equation} \begin{split} &\partial_{st} \log\left(1+\frac{1}{(\lambda_{0}(-t)^{1-\alpha}-s+t)(1-s+t)^{2}}\right)\\
&=\left(\frac{-\left(2(1-s+t)(\lambda_{0}(-t)^{1-\alpha}-s+t)+(1-s+t)^{2}(-\lambda_{0}'(-t)(1-\alpha)\lambda_{0}(-t)^{-\alpha}+1)\right)}{((1-s+t)^{2}(\lambda_{0}(-t)^{1-\alpha}-s+t)+1)^{2}}\right)\\
&\cdot\left(\frac{2}{1-s+t}+\frac{1}{(\lambda_{0}(-t)^{1-\alpha}-s+t)}\right)\\
&+\left(\frac{1}{(1-s+t)^{2}(\lambda_{0}(-t)^{1-\alpha}-s+t)+1}\right)\left(\frac{-2}{(1-s+t)^{2}}-\frac{(-(1-\alpha)\lambda_{0}(-t)^{-\alpha}\lambda_{0}'(-t)+1)}{(\lambda_{0}(-t)^{1-\alpha}-s+t)^{2}}\right)\\
&\leq 0\end{split}\end{equation}
where we recall $\alpha \leq \frac{1}{4}$ and $\lambda_{0}'(x) \leq 0, \quad x \geq T_{0}$. This completes the verification of \eqref{kineq}.

Now, by Theorem 8.6 (sec. 9.8, pg. 259) of \cite{inteqns}, $K$ has a non-negative resolvent, $r$, which is locally of type $L^{\infty}$ on $J^{2}$. Recall that the resolvent kernel, $r$, satisfies (a.e) the equations
\begin{equation}\label{resolventeqns} r + K * r =K, \quad r+r * K =K \end{equation}
where the $*$ operation, (as defined in \cite{inteqns}, Definition 2.3) between two measureable functions on $J^{2}$, or between a measureable function on $J^{2}$ and one on $J$ (when the integrands are integrable)  is 
\begin{equation}\begin{split} (a*b)(t,s) &= \int_{J} a(t,u)b(u,s) du\\
(a*f)(t) &= \int_{J} a(t,u)f(u) du\end{split}\end{equation}

In fact, we also have 
\begin{equation}\label{kquotest} \frac{K(t,s)}{K(u,s)} \leq 2, \quad s \leq u \leq t\end{equation}
because
\begin{equation} \begin{split} \frac{K(t,s)}{K(u,s)}&=\frac{\frac{1}{1-s+t}+\frac{1}{(\lambda_{0}(-t)^{1-\alpha}-s+t)(1-s+t)^{3}}}{\frac{1}{1-s+u}+\frac{1}{(\lambda_{0}(-u)^{1-\alpha}-s+u)(1-s+u)^{3}}}\\
&\leq \frac{1-s+u}{1-s+t} + \frac{(\lambda_{0}(-u)^{1-\alpha}-s+u)(1-s+u)^{3}}{(\lambda_{0}(-t)^{1-\alpha}-s+t)(1-s+t)^{3}}\\
&\leq 2, \quad s \leq u \leq t \end{split}\end{equation}
Hence, by Theorem 8.5(sec. 9.8, pg. 258) of \cite{inteqns} $r$ is in fact a non-negative Volterra kernel of type $L^{\infty}$, and we have the estimate 
\begin{equation}\label{restimate} \int_{-\infty}^{t} r(t,u) du \leq 2, \quad \text{a.e. } t \in J\end{equation}
By our assumptions on $F$, namely \eqref{Fassump}, $H$ satisfies the property that (e.g) \begin{equation} \label{hdecay}H(\cdot)\left(\cdot\right)^{2} \in L^{\infty}(J)\end{equation}
Then, we have a solution to \eqref{conveqn} (a.e.) given by the formula
\begin{equation}\label{existenceresult} x=H-(r * H)\end{equation}
To see this, we first note that \eqref{conveqn} is
$$K * x+x=H$$
Then, $$K*(r*H) = (K*r)*H$$ by Fubini's theorem. Fubini's theorem is applicable because
\begin{equation}\begin{split}&\int_{J} \int_{J} |K(t,u)| |r(u,s)| |H(s)| ds du  =\int_{-\infty}^{t}\int_{-\infty}^{u} |K(t,u)| |r(u,s)| |H(s)| ds du\\
&\leq ||H(\cdot)\left(\cdot\right)^{2}||_{L^{\infty}(J)} \int_{-\infty}^{t}\int_{-\infty}^{u}K(t,u) \frac{r(u,s)}{s^{2}} ds du \leq||H(\cdot)\left(\cdot\right)^{2}||_{L^{\infty}(J)} \int_{-\infty}^{t} \int_{-\infty}^{u} \frac{K(t,u)}{u^{2}} r(u,s) ds du\\
&\leq 2 ||H(\cdot)\left(\cdot\right)^{2}||_{L^{\infty}(J)} \int_{-\infty}^{t} \frac{K(t,u)}{u^{2}} du \leq \frac{2 ||H(\cdot)\left(\cdot\right)^{2}||_{L^{\infty}(J)}}{\alpha |\log(\lambda_{0}(T_{0}))|} \left(1+\frac{1}{\lambda_{0}(-t)^{1-\alpha}}\right)\int_{-\infty}^{t} \frac{du}{u^{2}}\\
&\leq \frac{2 ||H(\cdot)\left(\cdot\right)^{2}||_{L^{\infty}(J)}}{\alpha |\log(\lambda_{0}(T_{0}))|} \left(1+\frac{1}{\lambda_{0}(-t)^{1-\alpha}}\right)\frac{1}{|t|},\quad t \in J\end{split}\end{equation}  by \eqref{restimate}, \eqref{hdecay}, and inspection of the formula for $K$. Now, substituting 
$$x=H-(r*H)$$
using Fubini's theorem as above, and using the resolvent equation 
$$r+K*r=K$$
we see that \eqref{existenceresult} is a solution to \eqref{conveqn}.  
\\
But, this means that we have a solution to the equation
\begin{equation}  \alpha  y(t) = F(t)+\int_{t}^{\infty} \frac{y(s)}{\log(\lambda_{0}(s))}\left(\frac{1}{1+s-t}+\frac{1}{(\lambda_{0}(t)^{1-\alpha}+s-t)(1+s-t)^{3}}\right) ds, \quad \text{a.e. } t \geq T_{0}\end{equation}
Since we are considering this equation for $F$ satisfying \eqref{Fassump} the right-hand side of the equation above is a continuous function of $t$. So, $y$ agrees with a continuous function a.e., and hence, we may extend $y$ given a.e. by \eqref{existenceresult} to a continuous function of $t \in [T_{0},\infty)$.\\
\\
\eqref{existenceresult}, written in terms of $y$ reads
\begin{equation}\label{lambdasolution} y(t) = \frac{F(t)}{\alpha} -\int_{t}^{\infty}  \frac{F(s)}{\alpha} r(-t,-s) ds\end{equation}
and \eqref{restimate} implies that
\begin{equation}\label{rintest} \int_{t}^{\infty} r(-t,-z) dz \leq 2, \quad \text{a.e. } t \geq T_{0}\end{equation}
Now, we are finally ready to solve \eqref{emodulation}. We recall the complete, normed vector space $(X,||\cdot||_{X})$ defined in the beginning of this section by
$$X=\{f \in C^{2}([T_{0},\infty))| ||f||_{X}< \infty\}$$
where
\begin{equation}||f||_{X} = \text{sup}_{t \geq T_{0}} \left(|f(t)| b \log^{b}(t) \sqrt{\log(\log(t))} + |f'(t)| t \log^{b+1}(t) \sqrt{\log(\log(t))} +|f''(t)| t^{2} \log^{b+1}(t) \sqrt{\log(\log(t))}\right)\end{equation}
We quickly remark that all previous manipulations done on $v_{k}$, including estimates and representations of inner products, are valid for all $\lambda = \lambda_{0} + e, \quad e \in \overline{B_{1}(0)} \subset X$, since \eqref{lambdarestr} is valid for all $\lambda$ of this form. For $e \in \overline{B_{1}(0)} \subset X$, $RHS(e,t)$ is a continuous function of $t \in [T_{0},\infty)$. We will now estimate $RHS(e,t)$ for an arbitrary $e \in \overline{B_{1}(0)} \subset X$. The estimate we will obtain will then allow us to define a map, $T$, on $\overline{B_{1}}(0) \subset X$, and prove some properties about it, using the discussion above. Eventually, we will show that $T$ has a fixed point.

We start by estimating all the terms of $RHS(e,t)$, except for the one involving $G$, for $e \in \overline{B_{1}(0)} \subset X$. From our previous calculations, we have
\begin{equation} |E_{\lambda_{0,0}}| + |E_{v_{3},ip,01}|+|E_{v_{3},ip,f}|+|E_{\lambda_{0,1}}| \leq \frac{C}{t^{2}\log^{b+1}(t)}\end{equation}
So, the first line of $RHS(e,t)$ is bounded above in absolute value by
$$\frac{C}{\log(\log(t))t^{2}\log^{b+1}(t)}$$
Next, we note that
\begin{equation}\begin{split} \lambda_{0,1}(t) &= \int_{t}^{\infty}\int_{t_{1}}^{\infty} \frac{-b^{2}\log(\log(t_{2}))}{t_{2}^{2}\log^{b+2}(t_{2})}dt_{2}dt_{1}=\frac{-b^{2}\log(\log(t))}{(b+1)\log^{b+1}(t)} + O\left(\frac{1}{\log^{b+1}(t)}\right)\end{split}\end{equation}
Then, using the fact that $e \in \overline{B_{1}(0)} \subset X$, we get that the terms in the second line of $RHS(e,t)$, except for the one depending on $G$ is bounded above in absolute value by
$$\frac{C}{(\log(\log(t)))^{3/2} t^{2}\log^{b+1}(t)}$$
Similarly, the third, fourth, fifth, sixth, seventh, and eighth lines of $RHS(e,t)$ are bounded above in absolute value (respectively) by
\begin{equation} \begin{split} &\frac{C}{(\log(\log(t)))^{2} t^{2}\log^{b+1}(t)} +  \frac{C}{(\log(\log(t)))^{2} \log^{b+1}(t) t^{2}} + \frac{C \log(\log(t))}{t^{2}\log^{b+2}(t)}+ \frac{C}{(\log(\log(t)))^{3/2} t^{2}\log^{b+1}(t)} \\
&+ \frac{C}{t^{2}\log^{b+2}(t) (\log(\log(t))^{5/2}} + \frac{C}{(\log(\log(t))^{5/2} t^{3}\log^{b+2}(t)}\end{split}\end{equation}

Now, we proceed to estimate the terms from $G$. In the below expressions, we note that $\lambda(t) = \lambda_{0}(t) + e(t)$, and $e \in \overline{B}_{1}(0) \subset X$ is arbitrary. Recalling the definition of $E_{0,1}$, we have
\begin{equation} |\lambda(t) E_{0,1}(\lambda(t),\lambda'(t),\lambda''(t))| \leq \frac{C}{t^{2} \log^{b+1}(t)}\end{equation}
Then, from \eqref{k3minusk30}, we have
\begin{equation} |-16 \int_{t}^{\infty} \lambda''(s) \left(K_{3}(s-t,\lambda(t))-K_{3,0}(s-t,\lambda(t))\right) ds| \leq \frac{C}{t^{2}\log^{b+1}(t)}\end{equation}
From \eqref{ev2ipest}, we have
\begin{equation} |\lambda(t) E_{v_{2},ip}(t,\lambda(t))| \leq \frac{C}{t^{2}\log^{b+1}(t)}\end{equation} 
Using \eqref{kintegralestimate} and \eqref{k1diffint}, we get
\begin{equation} \begin{split} &|\frac{16}{\lambda(t)^{2}} \int_{t}^{\infty} ds \lambda''(s)K(s-t,\lambda(t))| \leq \frac{C}{t^{2}\log^{b+1}(t)}\\
&|\frac{16}{\lambda(t)^{2}} \int_{t}^{\infty} dx \lambda''(x)\left(K_{1}(x-t,\lambda(t))-\frac{\lambda(t)^{2}}{4(1+x-t)}\right)| \leq \frac{C}{t^{2}\log^{b+1}(t)}\end{split}\end{equation}
In order to proceed, we will use some pointwise estimates on $v_{3}$, $v_{4}$ and $v_{5}$.
\subsubsection{Pointwise estimates on $v_{3},\partial_{r}^{j}v_{3}$}
Here, we will prove two simple pointwise estimate on $v_{3}$ which do not directly follow from \eqref{v3preciseforip}. As with all of our work from now on, every estimate is valid for any $\lambda$ of the form 
$$\lambda(t) = \lambda_{0}(t)+e(t), \quad e \in \overline{B}_{1}(0) \subset X$$
\begin{lemma} We have the following pointwise estimates on $\partial_{r}^{j}v_{3}, \quad j=0,1,2$:
\begin{equation}\label{v3laterest} |v_{3}(t,r)| \leq \frac{C r \log(\log(t))}{t^{2} \log^{b+1}(t)}\end{equation}
\begin{equation}\label{v3largerest} |v_{3}(t,r)| \leq \frac{C}{r} \int_{t}^{\infty} ds |\lambda''(s)|(s-t)\end{equation}
\begin{equation}\label{drv3est} |\partial_{r}v_{3}(t,r)| \leq \frac{C}{t^{2}\log^{b}(t)}\end{equation}

\end{lemma}
\begin{proof}
We again make the decomposition
$$v_{3} = v_{3,1}+v_{3,2}$$
and use the same estimates on $v_{3,2}$ proven while obtaining \eqref{v3preciseforip}, to get
\begin{equation} |v_{3,2}(t,r)| \leq \frac{C r}{t^{2} \log^{b+1}(t)}\end{equation}
For $v_{3,1}$, whose definition is \eqref{v31def}, we have
\begin{equation} |v_{3,1}(t,r)| \leq \frac{C r}{t^{2} \log^{b+1}(t)} + Cr \int_{6r}^{\infty} dw |\lambda''(t+w)| w |\frac{1}{(\lambda(t+w)^{2-2\alpha}+w^{2})}-\frac{1}{1+w^{2}}|\end{equation}
In order to estimate this integral, we use the fact that $\lambda'(x) \leq 0, \quad x \geq T_{0}$, and get
\begin{equation}\label{v31bi1initial}\begin{split} \int_{0}^{1} dw w \frac{|\lambda''(t+w)|}{(\lambda(t+w)^{2-2\alpha}+w^{2})} &\leq \frac{C}{t^{2} \log^{b+1}(t)} \frac{1}{\log^{(2\alpha -2)b}(t)} \int_{0}^{1} \frac{w dw}{1+w^{2}\lambda(t)^{2\alpha -2}}\\
&\leq \frac{C \log(\log(t))}{t^{2} \log^{b+1}(t)}\end{split}\end{equation}
Also,
\begin{equation} \int_{0}^{1} \frac{w dw |\lambda''(t+w)|}{1+w^{2}} \leq \frac{C}{t^{2} \log^{b+1}(t)}\end{equation}
On the other hand, if $w \geq 1$, then, 
\begin{equation} |\frac{1}{(\lambda(t+w)^{2-2\alpha} +w^{2})}-\frac{1}{1+w^{2}}| \leq \frac{C}{w^{4}}\end{equation}
which gives
\begin{equation} \int_{1}^{\infty} dw |\lambda''(t+w)| w |\frac{1}{(\lambda(t+w)^{2-2\alpha} + w^{2})} - \frac{1}{1+w^{2}}| \leq \frac{C}{t^{2} \log^{b+1}(t)}\end{equation}
In total, we get 
\begin{equation} |v_{3,1}(t,r)| \leq \frac{C r \log(\log(t))}{t^{2} \log^{b+1}(t)}\end{equation}
which gives \eqref{v3laterest}. For the second pointwise estimate on $v_{3}$ in the lemma, we use
\begin{equation}\begin{split} |v_{3}(t,r)| &\leq \frac{1}{r} \int_{t}^{\infty} ds \int_{0}^{s-t} \frac{\rho d\rho}{\sqrt{(s-t)^{2}-\rho^{2}}} |\lambda''(s)|\cdot 2\\
&\leq \frac{C}{r} \int_{t}^{\infty} ds |\lambda''(s)|(s-t)\end{split}\end{equation}

We now prove the estimate on $\partial_{r}v_{3}$ in the lemma statement. We recall the definition of $v_{3}$
\begin{equation}\label{v3recall} v_{3}(t,r) = \frac{-1}{r}\int_{t}^{\infty} ds \int_{0}^{s-t} \frac{\rho d\rho}{\sqrt{(s-t)^{2}-\rho^{2}}} \lambda''(s) \left(\frac{-1-\rho^{2}+r^{2}}{\sqrt{(1+\rho^{2}-r^{2})^{2}+4r^{2}}}+F_{3}(r,\rho,\lambda(s))\right)\end{equation}

Then, we make a decomposition analogous to $v_{3}=v_{3,1}+v_{3,2}$, used previously, and treat each term seperately:
\begin{equation}\begin{split} &|\int_{t}^{\infty} ds \int_{0}^{s-t} \rho d\rho \left(\frac{1}{\sqrt{(s-t)^{2}-\rho^{2}}}-\frac{1}{(s-t)}\right)\frac{\lambda''(s)}{r}\partial_{r}\left(\frac{-1-\rho^{2}+r^{2}}{\sqrt{(1+\rho^{2}-r^{2})^{2}+4r^{2}}}\right)|\\
&\leq C \left(\sup_{x \geq t} |\lambda''(x)|\right) \int_{0}^{\infty} \rho d\rho \left(\frac{1+\rho^{2}+r^{2}}{(4r^{2}+(1+\rho^{2}-r^{2})^{2})^{3/2}}\right)\\
&\leq C \sup_{x \geq t} |\lambda''(x)|\end{split}\end{equation}

\begin{equation}\begin{split} &|\int_{t}^{\infty} ds \int_{0}^{s-t} \rho d\rho \left(\frac{1}{\sqrt{(s-t)^{2}-\rho^{2}}}-\frac{1}{(s-t)}\right)\frac{\lambda''(s)}{r}\partial_{r}\left(F_{3}(r,\rho,\lambda(s))\right)|\\
&\leq C \sup_{x \geq t}\left(|\lambda''(x)|\lambda(x)^{4\alpha-4}\right) \int_{0}^{\infty} \rho d\rho \frac{(\rho^{2}+r^{2}+\lambda(t)^{2-2\alpha})}{(1+2(\rho^{2}+r^{2})\lambda(t)^{2\alpha-2}+(\rho^{2}-r^{2})^{2}\lambda(t)^{4\alpha-4})^{3/2}}\\
&\leq C \sup_{x \geq t}\left(|\lambda''(x)| \lambda(x)^{4\alpha-4}\right) \lambda(t)^{4-4\alpha}\end{split}\end{equation}
where we used
\begin{equation} \label{drf3est}\frac{|\partial_{r}F_{3}|}{r} \leq \frac{C ((\rho^{2}+r^{2})\lambda(s)^{4\alpha -4} +\lambda(s)^{2\alpha -2})}{(1+2(\rho^{2}+r^{2})\lambda(s)^{2\alpha -2}+(\rho^{2}-r^{2})^{2} \lambda(s)^{4\alpha -4})^{3/2}}\end{equation}

Next, we have the term where the $r$ derivative acts on the $\frac{1}{r}$ factored out of the integrals in \eqref{v3recall}. For this term, we simply note that
\begin{equation} \frac{1}{r^{2}} \int_{t}^{\infty} ds \int_{0}^{s-t} \frac{\rho d\rho \lambda''(s)}{\sqrt{(s-t)^{2}-\rho^{2}}} \left(\frac{-1-\rho^{2}+r^{2}}{\sqrt{(1+\rho^{2}-r^{2})^{2}+4r^{2}}}+F_{3}(r,\rho,\lambda(s))\right) = \frac{-v_{3}(t,r)}{r}\end{equation}

Then, using \eqref{v3laterest}, we have
\begin{equation} \begin{split}&|\frac{1}{r^{2}} \int_{t}^{\infty} ds \int_{0}^{s-t} \frac{\rho d\rho \lambda''(s)}{\sqrt{(s-t)^{2}-\rho^{2}}} \left(\frac{-1-\rho^{2}+r^{2}}{\sqrt{(1+\rho^{2}-r^{2})^{2}+4r^{2}}}+F_{3}(r,\rho,\lambda(s))\right)|\\
&\leq C \frac{\log(\log(t))}{t^{2}\log^{b+1}(t)} \end{split}\end{equation}
The last term to estimate is
\begin{equation} -\int_{t}^{\infty} \frac{ds}{(s-t)} \int_{0}^{s-t} \rho d\rho \frac{\lambda''(s)}{r} \partial_{r}\left(\frac{-1-\rho^{2}+r^{2}}{\sqrt{(1+\rho^{2}-r^{2})^{2}+4r^{2}}}+F_{3}(r,\rho,\lambda(s))\right)\end{equation}
If $s-t \leq \frac{1}{2}$, we start with
\begin{equation}\begin{split}&|\int_{0}^{s-t}\frac{\rho d\rho}{r}\partial_{r}\left(\frac{-1-\rho^{2}+r^{2}}{\sqrt{(1+\rho^{2}-r^{2})^{2}+4r^{2}}}\right)|\\
&=|\frac{-1-r^{2}+(s-t)^{2}}{\sqrt{4(s-t)^{2}+(1+r^{2}-(s-t)^{2})^{2}}}+1|\\
& = \frac{4(s-t)^{2}}{\sqrt{4(s-t)^{2}+(1+r^{2}-(s-t)^{2})^{2}}(1+r^{2}-(s-t)^{2}+\sqrt{(1+r^{2}-(s-t)^{2})^{2}+4(s-t)^{2}})}\\
&\leq C (s-t)^{2}, \quad s-t \leq \frac{1}{2} \end{split}\end{equation}

This gives
\begin{equation} \begin{split}&|-\int_{t}^{t+\frac{1}{2}} \frac{ds}{(s-t)} \int_{0}^{s-t} \rho d\rho \frac{\lambda''(s)}{r} \partial_{r}\left(\frac{-1-\rho^{2}+r^{2}}{\sqrt{(1+\rho^{2}-r^{2})^{2}+4r^{2}}}\right)|\\
&\leq \frac{C}{t^{2}\log^{b+1}(t)}\end{split}\end{equation}
On the other hand, we have
\begin{equation}\begin{split} &|-\int_{t+\frac{1}{2}}^{\infty} \frac{ds}{(s-t)} \int_{0}^{s-t} \rho d\rho \frac{\lambda''(s)}{r} \partial_{r}\left(\frac{-1-\rho^{2}+r^{2}}{\sqrt{(1+\rho^{2}-r^{2})^{2}+4r^{2}}}\right)|\\
&\leq \int_{t+\frac{1}{2}}^{\infty} \frac{ds}{(s-t)} |\lambda''(s)| \cdot 2\\
&\leq C \sup_{x \geq t}\left(x|\lambda''(x)|\right) \frac{\log(t)}{t}\end{split}\end{equation}

Now, we have to treat the $F_{3}$ related terms. We start by recalling \eqref{drf3est}.\\
Then, we use a slightly different procedure:
\begin{equation} \begin{split} &|-\int_{t}^{\infty} \frac{ds}{(s-t)} \int_{0}^{s-t} \rho d\rho \frac{\lambda''(s)}{r} \partial_{r}\left(F_{3}(r,\rho,\lambda(s))\right)|\\
&\leq C \int_{0}^{\infty} \rho d\rho \int_{\rho+t}^{\infty} \frac{ds |\lambda''(s)|}{(s-t)} \frac{(\rho^{2}+r^{2}+\lambda(s)^{2-2\alpha}) \lambda(s)^{4\alpha -4}}{(1+2(\rho^{2}+r^{2})\lambda(s)^{2\alpha -2} +(\rho^{2}-r^{2})^{2} \lambda(s)^{4\alpha -4})^{3/2}}\\
&\leq C \int_{0}^{\infty} \rho d\rho \int_{\rho+t}^{\infty} \frac{ds |\lambda''(s)|}{(s-t)} \frac{(\rho^{2}+r^{2}+\lambda(t)^{2-2\alpha}) \lambda(s)^{4\alpha -4}}{(1+2(\rho^{2}+r^{2})\lambda(t)^{2\alpha -2} +(\rho^{2}-r^{2})^{2} \lambda(t)^{4\alpha -4})^{3/2}}\\
&\leq C \int_{0}^{\infty} \frac{\rho d\rho (\rho^{2}+r^{2}+\lambda(t)^{2-2\alpha})}{(1+2(\rho^{2}+r^{2})\lambda(t)^{2\alpha -2} +(\rho^{2}-r^{2})^{2} \lambda(t)^{4\alpha -4})^{3/2}} \frac{1}{\log^{(4\alpha -4)b}(t) t \log^{b+1}(t)} \int_{\rho+t}^{\infty} \frac{ds}{s(s-t)}\\
&\leq C \int_{0}^{t} \frac{\rho d\rho (\rho^{2}+r^{2}+\lambda(t)^{2-2\alpha})}{(1+2(\rho^{2}+r^{2})\lambda(t)^{2\alpha -2} +(\rho^{2}-r^{2})^{2} \lambda(t)^{4\alpha -4})^{3/2}} \frac{\log(1+\frac{t}{\rho})}{\log^{(4\alpha -4)b}(t) t^{2} \log^{b+1}(t)} \\
&+ C \int_{t}^{\infty} \frac{\rho d\rho (\rho^{2}+r^{2}+\lambda(t)^{2-2\alpha})}{(1+2(\rho^{2}+r^{2})\lambda(t)^{2\alpha -2} +(\rho^{2}-r^{2})^{2} \lambda(t)^{4\alpha -4})^{3/2}} \frac{1}{\log^{(4\alpha -4)b}(t) t^{2} \log^{b+1}(t)}\end{split}\end{equation}
We finally consider two subsets of the region $r \leq t$ seperately:
\begin{equation} \begin{split} &\int_{0}^{\lambda(t)^{1-\alpha}} \frac{\rho d\rho (\rho^{2}+r^{2}+\lambda(t)^{2-2\alpha})}{(1+2(\rho^{2}+r^{2})\lambda(t)^{2\alpha -2} +(\rho^{2}-r^{2})^{2} \lambda(t)^{4\alpha -4})^{3/2}}\frac{(\log(t)+|\log(\rho)|)}{\log^{(4\alpha -4)b}(t) t^{2}\log^{b+1}(t)}\\
&\leq C \int_{0}^{\lambda(t)^{1-\alpha}} \frac{\rho d\rho \lambda(t)^{2-2\alpha} (\log(t)+|\log(\rho)|)}{(1+2(\rho^{2}+r^{2})\lambda(t)^{2\alpha -2} +(\rho^{2}-r^{2})^{2} \lambda(t)^{4\alpha -4})^{1/2}} \frac{1}{\log^{(4\alpha -4)b}(t) t^{2} \log^{b+1}(t)}\\
&\leq \frac{C}{t^{2} \log^{b}(t)}\end{split}\end{equation}
Then,
\begin{equation} \begin{split} &\int_{\lambda(t)^{1-\alpha}}^{t} \frac{\rho d\rho (\rho^{2}+r^{2}+\lambda(t)^{2-2\alpha})}{(1+2(\rho^{2}+r^{2})\lambda(t)^{2\alpha -2} +(\rho^{2}-r^{2})^{2} \lambda(t)^{4\alpha -4})^{3/2}}\frac{(\log(t)+|\log(\rho)|)}{\log^{(4\alpha -4)b}(t) t^{2}\log^{b+1}(t)}\\
&\leq \frac{C\log(t)}{t^{2} \log^{(4\alpha -4)b}(t) \log^{b+1}(t)} \int_{0}^{\infty} \frac{\rho d\rho (\rho^{2}+r^{2}+\lambda(t)^{2-2\alpha})}{(1+2(\rho^{2}+r^{2})\lambda(t)^{2\alpha -2} +(\rho^{2}-r^{2})^{2} \lambda(t)^{4\alpha -4})^{3/2}}\\
&\leq \frac{C}{t^{2} \log^{b}(t)}\end{split}\end{equation}
The final integral to estimate is then
\begin{equation}\begin{split} &\int_{t}^{\infty} \frac{\rho d\rho (\rho^{2}+r^{2}+\lambda(t)^{2-2\alpha})}{(1+2(\rho^{2}+r^{2})\lambda(t)^{2\alpha -2} +(\rho^{2}-r^{2})^{2} \lambda(t)^{4\alpha -4})^{3/2}}\frac{1}{\log^{(4\alpha -4)b}(t) t^{2}\log^{b+1}(t)}\\
&\leq \frac{C}{t^{2} \log^{b+1}(t)}\end{split}\end{equation}
This gives \eqref{drv3est}.

\end{proof}

\subsubsection{Pointwise estimates on $v_{4}$, $\partial_{r}v_{4}$}
In this section, we prove
\begin{lemma} For all $\lambda$ of the form
$$\lambda(t) = \lambda_{0}(t)+e(t), \quad e \in \overline{B}_{1}(0) \subset X$$
we have the pointwise estimates

\begin{equation} \label{v4finalest} |v_{4}(t,r)| \leq \begin{cases} \frac{C r}{t^{2} \log^{3b+2N-1}(t)}, \quad r \leq \frac{t}{2}\\
\frac{C}{\sqrt{r} t \log^{\frac{3N}{2}+3b-1}(t)}, \quad r > \frac{t}{2}\end{cases}\end{equation}

\begin{equation}\label{drv4finalest} |\partial_{r}v_{4}(t,r)| \leq \begin{cases} \frac{C}{t^{2}\log^{3b+2N-1}(t)}, \quad r \leq \frac{t}{2}\\
\frac{C}{\sqrt{r}t\log^{3b-1+\frac{5N}{2}}(t)}, \quad r \geq \frac{t}{2}\end{cases}\end{equation}
and
\begin{equation} \label{dtv4largerest} |\partial_{t}v_{4}(t,r)| \leq \frac{C}{\sqrt{r}t\log^{3b-1+\frac{5N}{2}}(t)}, \quad r \geq \frac{t}{2}\end{equation}
We also have the $L^{2}$ estimates
\begin{equation}\label{v4energyest} ||\partial_{t}v_{4}||_{L^{2}(r dr)} + ||\partial_{r}v_{4}||_{L^{2}(r dr)}+||\frac{v_{4}}{r}||_{L^{2}(r dr)} \leq \frac{C}{t \log^{2N+3b}(t)}\end{equation}
\end{lemma}
\begin{proof}
We start by considering $v_{4}(t,r)$, for $\frac{t}{2}>r >0$. In order to ease notation, let $x \in \mathbb{R}^{2}$ be defined by $x = r \textbf{e}_{1}$. Then, we recall \eqref{v4repformsimp} and get
\begin{equation}\begin{split} v_{4}(t,r) = \frac{-r}{2 \pi} \int_{0}^{1} d\beta \int_{t}^{\infty} ds \int_{B_{s-t}(0)} \frac{dA(y)}{\sqrt{(s-t)^{2}-|y|^{2}}} &\left(\frac{\partial_{2}v_{4,c}(s,|\beta x+y|)\left((\beta x+y)\cdot \hat{x}\right)^{2}}{|\beta x+y|^{2}}\right.\\
&-\frac{v_{4,c}(s,|\beta x+y|)\left((\beta x+y)\cdot \hat{x}\right)^{2}}{|\beta x+y|^{3}}\\
&\left. + \frac{v_{4,c}(s,|\beta x+y|)}{|\beta x+y|}\right)\end{split}\end{equation}

Note that, for $|x| \leq \frac{t}{2}$, $0<\beta <1$, and $|y| \leq s-t$, we have
\begin{equation} s-|\beta x+y| \geq \frac{t}{2}\end{equation}
This means that, for the purposes of estimating $v_{4}$ in the region $ r \leq \frac{t}{2}$, we can use \eqref{v2singularconeest} to estimate $v_{2}$, for all $r \geq \frac{t}{2}$. We then combine this with the estimates for $v_{1},v_{3}$, and $F_{0,2}$, to get
\begin{equation} \label{v4cest1} |v_{4,c}(t,r)| \leq  C |\chi_{\geq 1}(\frac{2r}{\log^{N}(t)})| \begin{cases} \frac{1}{t^{2} r^{3} \log^{3b}(t)} , \quad r \leq \frac{t}{2}\\
\frac{\log(r)}{\log^{2b}(t) r^{4}|t-r|} + \frac{\log^{2b\alpha}(t)}{t^{2}r^{3}\log^{3b+1}(t)}, \quad r \geq \frac{t}{2}\end{cases}\end{equation}
and similarly, for the derivatives, we have
\begin{equation}\begin{split} |\partial_{r}v_{4,c}(t,r)| &\leq C \chi_{\geq 1}(\frac{2r}{\log^{N}(t)}) \begin{cases} \frac{1}{r^{4}t^{2}\log^{3b}(t)}, \quad r \leq \frac{t}{2}\\
\frac{\log(r)}{\log^{2b}(t) r^{4}(t-r)^{2}} + \frac{1}{\log^{3b}(t) t^{2}r^{4}}, \quad r \geq \frac{t}{2}\end{cases}\\
&+\frac{C |\chi'(\frac{2r}{\log^{N}(t)})|}{\log^{5N+2b}(t)} \frac{r}{t^{2}\log^{b}(t)} \end{split}\end{equation}
Then, we get
\begin{equation}\label{v4formulaforestimates}\begin{split}  |v_{4}(t,r)| &\leq C r \int_{0}^{1} d\beta \int_{t}^{\infty} ds \int_{B_{s-t}(0) \cap B_{\frac{s}{2}}(-\beta x)} \frac{dA(y)}{\sqrt{(s-t)^{2}-|y|^{2}}} |\partial_{2}v_{4,c}(s,|\beta x+y|)| \\
&+C r \int_{0}^{1} d\beta \int_{t}^{\infty} ds \int_{B_{s-t}(0) \cap B_{\frac{s}{2}}(-\beta x)} \frac{dA(y)}{\sqrt{(s-t)^{2}-|y|^{2}}} \frac{|v_{4,c}(s,|\beta x+y|)|}{|\beta x+y|}\\
&+C r \int_{0}^{1} d\beta \int_{t}^{\infty} ds \int_{B_{s-t}(0) \cap (B_{\frac{s}{2}}(-\beta x))^{c}} \frac{dA(y)}{\sqrt{(s-t)^{2}-|y|^{2}}} |\partial_{2}v_{4,c}(s,|\beta x+y|)|\\
&+C r \int_{0}^{1} d\beta \int_{t}^{\infty} ds \int_{B_{s-t}(0) \cap (B_{\frac{s}{2}}(-\beta x))^{c}} \frac{dA(y)}{\sqrt{(s-t)^{2}-|y|^{2}}}\frac{|v_{4,c}(s,|\beta x+y|)}{|\beta x+y|}\end{split}\end{equation}

Each line of \eqref{v4formulaforestimates} is then further split into two terms, based on the decomposition  $$\frac{1}{\sqrt{(s-t)^{2}-|y|^{2}}}=\frac{1}{s-t}+\left(\frac{1}{\sqrt{(s-t)^{2}-|y|^{2}}}-\frac{1}{s-t}\right)$$
and estimated seperately. For the first term of the first line, we have

\begin{equation}\begin{split}&r \int_{0}^{1} d\beta \int_{t}^{\infty} ds \int_{B_{s-t}(0) \cap B_{\frac{s}{2}}(-\beta x)} \frac{dA(y)}{(s-t)} |\partial_{2}v_{4,c}(s,|\beta x+y|)| \\
&\leq C r \int_{0}^{1}d\beta \int_{t}^{t+\frac{1}{2}} ds \int_{B_{s-t}(0)} \frac{dA(y)}{(s-t)} \frac{1}{\log^{4N}(s) s^{2}\log^{3b}(s)} +C r \int_{0}^{1} d\beta \int_{t+\frac{1}{2}}^{\infty} \frac{ds}{(s-t)} \int_{B_{\frac{s}{2}}(-\beta x)} \frac{dA(y) \mathbbm{1}_{\{|\beta x+y| \geq \log^{N}(s)\}}}{|\beta x+y|^{4} s^{2}\log^{3b}(s)}\\
&\leq C r \int_{0}^{1}d\beta \int_{t}^{t+\frac{1}{2}} ds \frac{(s-t)}{s^{2}\log^{3b+4N}(s)} +C r \int_{0}^{1}d\beta \int_{t+\frac{1}{2}}^{\infty} \frac{ds}{(s-t)} \int_{B_{\frac{s}{2}}(0)} \frac{dA(z) \mathbbm{1}_{\{|z| \geq \log^{N}(s)\}}}{|z|^{4}s^{2}\log^{3b}(s)}\\
&\leq \frac{C r}{t^{2}\log^{3b+4N}(t)}+ C r \int_{0}^{1}d\beta \int_{t+\frac{1}{2}}^{\infty} \frac{ds}{(s-t)} \int_{\log^{N}(s)}^{\frac{s}{2}} \rho d\rho \int_{0}^{2\pi} \frac{d\theta}{\rho^{4}s^{2}\log^{3b}(s)}\\
&\leq \frac{C r}{\log^{3b+2N-1}(t)t^{2}}\end{split}\end{equation}
where we used 
$$\frac{|\chi'(\frac{|\beta x+y|}{\log^{N}(s)})| |\beta x+y|}{\log^{5N+3b}(s) s^{2}} \leq C \frac{\mathbbm{1}_{\{|\beta x+y| \geq \log^{N}(s)\}}}{|\beta x+y|^{4}\log^{3b}(s)s^{2}}$$
Next, we estimate
\begin{equation}\label{term1b}\begin{split}&r \int_{0}^{1} d\beta \int_{t}^{\infty} ds \int_{B_{s-t}(0)\cap B_{\frac{s}{2}}(-\beta x)} dA(y) \left(\frac{1}{\sqrt{(s-t)^{2}-|y|^{2}}}-\frac{1}{(s-t)}\right) |\partial_{r}v_{4,c}(s,|\beta x+y|)|\\
&\leq C r \int_{0}^{1} d\beta \int_{t}^{\infty} ds \int_{B_{s-t}(0)} dA(y) \left(\frac{1}{\sqrt{(s-t)^{2}-|y|^{2}}}-\frac{1}{(s-t)}\right) \frac{1}{\log^{2N}(s)}\frac{1}{(\log^{2N}(t)+|\beta x+y|^{2})s^{2}\log^{3b}(s)}\\
&\leq C r \int_{0}^{1} d\beta \int_{0}^{\infty} \rho d\rho \int_{0}^{2\pi} \frac{d\theta}{\log^{2N}(t)} \frac{1}{(\log^{2N}(t)+\beta^{2} r^{2} + \rho^{2} + 2 \beta r \rho \cos(\theta))} \frac{1}{\log^{3b}(t)} \int_{\rho +t}^{\infty} \frac {ds}{s^{2}}\left(\frac{1}{\sqrt{(s-t)^{2}-\rho^{2}}}-\frac{1}{(s-t)}\right)\\
&\leq C r \int_{0}^{1} d\beta \int_{0}^{\infty} \rho d\rho \int_{0}^{2\pi} \frac{d\theta}{\log^{2N+3b}(t)} \frac{1}{(\log^{2N}(t) + \beta^{2}r^{2}+\rho^{2}+2 \beta r \rho \cos(\theta))}\frac{1}{(\rho+t)^{2}}\\
&\leq \frac{C r}{\log^{2N+3b}(t)} \int_{0}^{1}d\beta \int_{0}^{\infty} \frac{\rho d\rho}{(\rho+t)^{2}} \frac{1}{\sqrt{(\log^{2N}(t)+(\beta r+\rho)^{2})(\log^{2N}(t)+(\beta r-\rho)^{2})}}\\
&\leq \frac{C r}{\log^{2N+3b}(t)} \int_{0}^{1}d\beta \int_{0}^{t} \frac{\rho d\rho}{(\rho+t)^{2}} \frac{1}{\sqrt{(\log^{2N}(t)+(\beta r+\rho)^{2})(\log^{2N}(t)+(\beta r-\rho)^{2})}}+\frac{C r}{\log^{2N+3b}(t)} \int_{0}^{1}d\beta \int_{t}^{\infty} \frac{\rho d\rho}{\rho^{4}}\\
&\leq \frac{C r (\log(2+r^{2})+\log(2+t^{2}))}{t^{2}\log^{2N+3b}(t)}\\
&\leq \frac{C r}{t^{2}\log^{2N+3b-1}(t)}, \quad r \leq \frac{t}{2}\end{split}\end{equation}
where, we used the fact that 
$$\rho \geq t, \quad r \leq \frac{t}{2} \implies |\beta r -\rho| = \rho-\beta r \geq \frac{\rho}{2}$$
Next, we estimate \begin{equation}\begin{split} &r\int_{0}^{1}d\beta \int_{t}^{\infty} ds \int_{B_{s-t}(0) \cap B_{\frac{s}{2}}(-\beta x)} \frac{dA(y)}{(s-t)} \frac{|v_{4,c}(s,|\beta x+y|)|}{|\beta x+y|}\\
&\leq C r \int_{0}^{1}d\beta \int_{t}^{\infty} ds \int_{B_{s-t}(0) \cap B_{\frac{s}{2}}(-\beta x)} \frac{dA(y)}{(s-t)} \frac{\chi_{\geq 1}(\frac{|\beta x+y|}{\log^{N}(s)})}{\log^{3b}(s) s^{2}|\beta x+y|^{4}}\end{split}\end{equation}
The last line in the above equation has already been estimated above, and we get
\begin{equation}\begin{split} &r \int_{0}^{1}d\beta \int_{t}^{\infty} ds \int_{B_{s-t}(0) \cap B_{\frac{s}{2}}(-\beta x)} \frac{dA(y)}{(s-t)} \frac{|v_{4,c}(s,|\beta x+y|)|}{|\beta x+y|}\leq \frac{C r}{t^{2}\log^{3b+2N-1}(t)}\end{split}\end{equation}
Next, we have
\begin{equation}\begin{split} &r \int_{0}^{1}d\beta \int_{t}^{\infty} ds \int_{B_{s-t}(0) \cap B_{\frac{s}{2}}(-\beta x)} dA(y)\left(\frac{1}{\sqrt{(s-t)^{2}-|y|^{2}}}-\frac{1}{(s-t)}\right) \frac{|v_{4,c}(s,|\beta x+y|)|}{|\beta x+y|}\\
&\leq C r \int_{0}^{1}d\beta \int_{t}^{\infty} ds \int_{B_{s-t}(0) \cap B_{\frac{s}{2}}(-\beta x)} dA(y) \left(\frac{1}{\sqrt{(s-t)^{2}-|y|^{2}}}-\frac{1}{(s-t)}\right) \frac{\chi_{\geq 1}(\frac{|\beta x+y|}{\log^{N}(s)})}{|\beta x+y|^{4} s^{2}\log^{3b}(s)}\end{split}\end{equation}
The last line in the above equation has also been estimated above. Next, we recall that
$$r \leq \frac{t}{2}, |y| \leq s-t \implies s-|\beta x+y| \geq \frac{t}{2}$$
and consider
\begin{equation}\begin{split} &r \int_{0}^{1}d\beta \int_{t}^{\infty} ds \int_{B_{s-t}(0)\cap (B_{\frac{s}{2}}(-\beta x))^{c}} \frac{dA(y)}{(s-t)} |\partial_{2}v_{4,c}(s,|\beta x+y|)|\end{split}\end{equation}
We first treat the portion of the integral when $s-t \leq \frac{1}{2}$:
\begin{equation}\begin{split} &r \int_{0}^{1}d\beta \int_{t}^{t+\frac{1}{2}} ds \int_{B_{s-t}(0)\cap (B_{\frac{s}{2}}(-\beta x))^{c}} \frac{dA(y)}{(s-t)} |\partial_{2}v_{4,c}(s,|\beta x+y|)|\\
&\leq Cr \int_{0}^{1}d\beta \int_{t}^{t+\frac{1}{2}} ds \int_{B_{s-t}(0)} \frac{dA(y)}{(s-t)} \left(\frac{1}{\log^{2b-1}(s) s^{4}t^{2}} + \frac{1}{\log^{3b}(s) s^{6}}\right)\\
&\leq \frac{C r}{t^{6}\log^{2b-1}(t)}\end{split}\end{equation}
We next consider the region $s-t \geq \frac{1}{2}$, and get
\begin{equation}\begin{split} &r \int_{0}^{1}d\beta \int_{t+\frac{1}{2}}^{\infty} ds \int_{B_{s-t}(0)\cap (B_{\frac{s}{2}}(-\beta x))^{c}} \frac{dA(y)}{(s-t)} |\partial_{2}v_{4,c}(s,|\beta x+y|)|\\
&\leq C r \int_{0}^{1}d\beta \int_{t+\frac{1}{2}}^{\infty} \frac{ds}{(s-t)} \int_{(B_{\frac{s}{2}}(0))^{c}}dA(z) \left(\frac{\log(|z|)}{\log^{2b}(s) |z|^{4} t^{2}}+\frac{1}{\log^{3b}(s) s^{2}|z|^{4}}\right)\\
&\leq C r \int_{t+\frac{1}{2}}^{\infty} \frac{ds}{(s-t)} \frac{1}{s^{2}t^{2}\log^{2b-1}(s)}\\
&\leq \frac{Cr}{t^{4}\log^{2b-2}(t)}\end{split}\end{equation}
Next, we have
\begin{equation}\begin{split} &r \int_{0}^{1}d\beta \int_{t}^{\infty} ds \int_{B_{s-t}(0)\cap (B_{\frac{s}{2}}(-\beta x))^{c}}dA(y)\left(\frac{1}{\sqrt{(s-t)^{2}-|y|^{2}}}-\frac{1}{(s-t)}\right) |\partial_{2}v_{4,c}(s,|\beta x+y|)|\\
&\leq C r \int_{0}^{1} d\beta \int_{t}^{\infty} ds \int_{B_{s-t}(0)} dA(y) \left(\frac{1}{\sqrt{(s-t)^{2}-|y|^{2}}}-\frac{1}{(s-t)}\right) \frac{1}{s^{2}}\frac{\log(s)}{(t^{2}+|\beta x+y|^{2})t^{2}\log^{2b}(s)}\\
&\leq C r \int_{0}^{1}d\beta \int_{0}^{\infty} \rho d\rho \int_{0}^{2\pi} \frac{d\theta}{(\rho+t) t^{3} \log^{2b-1}(t) (t^{2}+\beta^{2} r^{2} + \rho^{2} + 2 \beta r \rho \cos(\theta))}\end{split}\end{equation}
We can then treat the last line of the above equation in the same way as we treated \eqref{term1b}. This results in
\begin{equation}\begin{split} &r \int_{0}^{1}d\beta \int_{t}^{\infty} ds \int_{B_{s-t}(0)\cap (B_{\frac{s}{2}}(-\beta x))^{c}}dA(y)\left(\frac{1}{\sqrt{(s-t)^{2}-|y|^{2}}}-\frac{1}{(s-t)}\right) |\partial_{2}v_{4,c}(s,|\beta x+y|)|\\
&\leq \frac{C r}{t^{4} \log^{2b-2}(t)}\end{split}\end{equation}
Next, we have
\begin{equation} \begin{split}&r \int_{0}^{1} d\beta \int_{t}^{\infty} ds \int_{B_{s-t}(0) \cap(B_{\frac{s}{2}}(-\beta x))^{c}} \frac{dA(y)}{(s-t)} |\frac{v_{4,c}(s,|\beta x+y|)}{|\beta x+y|}|\end{split}\end{equation}
The contribution to the integral from the region $s-t \leq \frac{1}{2}$ is treated in an identical manner as above, and we estimate the other contribution, using the same procedure as used above:
\begin{equation}\begin{split} &r \int_{0}^{1} d\beta \int_{t+\frac{1}{2}}^{\infty} ds \int_{B_{s-t}(0) \cap(B_{\frac{s}{2}}(-\beta x))^{c}} \frac{dA(y)}{(s-t)} \frac{|v_{4,c}(s,|\beta x+y|)|}{|\beta x+y|}\\
&\leq C r \int_{0}^{1}d\beta \int_{t+\frac{1}{2}}^{\infty} \frac{ds}{(s-t)} \int_{\frac{s}{2}}^{\infty} \rho d\rho \int_{0}^{2\pi} d\theta \left(\frac{1}{t \log^{2b-1}(s) \rho^{5}}+\frac{\log^{2b\alpha}(s)}{s^{2} \rho^{4} \log^{3b+1}(s)}\right)\\
&\leq C r \int_{t+\frac{1}{2}}^{\infty} \frac{ds}{(s-t) t s^{3}\log^{2b-1}(s)} + C r \int_{t+\frac{1}{2}}^{\infty} \frac{\log^{2b\alpha -3b-1}(s) ds}{(s-t) s^{4}}\\
&\leq \frac{C r}{t^{4}\log^{2b-2}(t)}\end{split}\end{equation}
The last term to estimate is
\begin{equation}\begin{split} &r \int_{0}^{1} d\beta \int_{t}^{\infty} ds \int_{B_{s-t}(0) \cap (B_{\frac{s}{2}}(-\beta x))^{c}} dA(y)\left(\frac{1}{\sqrt{(s-t)^{2}-|y|^{2}}}-\frac{1}{(s-t)}\right)\frac{|v_{4,c}(s,|\beta x+y|)|}{|\beta x+y|}\\
&\leq C r \int_{0}^{1} d\beta \int_{t}^{\infty} ds \int_{B_{s-t}(0)} dA(y)\left(\frac{1}{\sqrt{(s-t)^{2}-|y|^{2}}}-\frac{1}{(s-t)}\right) \left(\frac{\log(s)}{\log^{2b}(s) s^{5}t}+\frac{\log^{2b\alpha}(s)}{s^{6} \log^{3b+1}(s)}\right)\\
&\leq C r \int_{0}^{\infty} \rho d\rho \int_{\rho+t}^{\infty} ds \left(\frac{1}{\sqrt{(s-t)^{2}-\rho^{2}}}-\frac{1}{(s-t)}\right) \left(\frac{1}{(\rho+t)^{4} t^{2} \log^{2b-1}(t)} + \frac{\log^{2b\alpha -3b-1}(t)}{(\rho+t)^{6}}\right)\\
&\leq \frac{C r}{t^{4}\log^{2b-1}(t)}\end{split}\end{equation}
Combining these estimates, we conclude that
\begin{equation} |v_{4}(t,r)| \leq \frac{C r}{t^{2}\log^{3b+2N-1}(t)}, \quad r \leq \frac{t}{2}\end{equation}

Next, we estimate $\partial_{r}v_{4}$, starting with the region $r \leq \frac{t}{2}$. We recall the function $G$ defined in \eqref{gforv4}:
\begin{equation}\begin{split} G(s,r,\rho) &=\int_{0}^{2\pi} d\theta \frac{v_{4,c}(s,\sqrt{r^{2}+2 r \rho \cos(\theta) + \rho^{2}})}{\sqrt{r^{2}+2 r \rho \cos(\theta) + \rho^{2}}} \left(r+\rho \cos(\theta)\right)\\
& s \geq t,\quad r \geq 0,\quad s-t \geq \rho \geq 0\end{split}\end{equation} 
and start with
\begin{equation}\begin{split}v_{4}(t,r) &= \frac{-1}{2\pi} \int_{t}^{\infty} ds \int_{0}^{s-t} \frac{\rho d\rho}{\sqrt{(s-t)^{2}-\rho^{2}}} G(s,r,\rho)\end{split}\end{equation}
Note that, when we estimated $v_{4}$ in the region $r \leq \frac{t}{2}$, we used $$G(s,r,\rho) = r \int_{0}^{1} \partial_{2}G(s,r\beta,\rho) d\beta$$
Now, we have
\begin{equation} \partial_{r}v_{4}(t,r) = \frac{-1}{2\pi} \int_{t}^{\infty} ds \int_{0}^{s-t}\frac{\rho d\rho}{\sqrt{(s-t)^{2}-\rho^{2}}} \partial_{2}G(s,r,\rho)\end{equation}
Therefore, the identical procedure gives the estimate
\begin{equation} |\partial_{r}v_{4}(t,r)| \leq \frac{C}{t^{2}\log^{3b+2N-1}(t)}, \quad r \leq \frac{t}{2}\end{equation}
Now, we treat the region $r \geq \frac{t}{2}$. Here, we use a different combination of the $v_{2}$ estimates in the various regions, to obtain
\begin{equation}\label{v4cest2} |v_{4,c}(t,r)| \leq C |\chi_{ \geq 1}(\frac{2r}{\log^{N}(t)})| \begin{cases} \frac{1}{r^{3}t^{2}\log^{3b}(t)}, \quad r \leq \frac{t}{2}\\
\frac{\log(r)}{\log^{2b}(t) r^{4}|t-r|}+\frac{\log^{2b\alpha}(t)}{t^{2} r^{3} \log^{3b+1}(t)}, \quad \frac{t}{2} \leq r \leq t-\sqrt{t}, \quad r > t+\sqrt{t}\\
\frac{1}{\log^{2b}(t)r^{9/2}}, \quad t-\sqrt{t} \leq r \leq t+\sqrt{t}\end{cases}\end{equation}
and
\begin{equation} \label{drv4cforfourier} \begin{split} |\partial_{r}v_{4,c}(t,r)|&\leq \frac{C |\chi'(\frac{2r}{\log^{N}(t)})|}{\log^{2b+N}(t) r^{4}} \left( \frac{r}{t^{2}\log^{b}(t)}\right)\\
&+\frac{C \chi_{\geq 1}(\frac{2r}{\log^{N}(t)})}{r^{4}\log^{2b}(t)} \begin{cases} \frac{1}{t^{2}\log^{b}(t)}, \quad r \leq \frac{t}{2}\\
\frac{1}{t^{2}\log^{b}(t)} + \frac{\log(r)}{(t-r)^{2}}, \quad \frac{t}{2} \leq r \leq t-t^{1/4} \text{, or }r \geq t+t^{1/4}\\
\frac{1}{\sqrt{r}}, \quad t-t^{1/4} \leq r \leq t+t^{1/4}\end{cases}\end{split}\end{equation} 

Next, we will use a different representation formula for $v_{4}$ to estimate $v_{4}$ and $\partial_{r}v_{4}$ in the region $r \geq \frac{t}{2}$. In particular, we have
\begin{equation} v_{4}(t,r) = \int_{t}^{\infty} dx \int_{0}^{\infty} d\xi J_{1}(r\xi) \sin((t-x)\xi) \widehat{v_{4,c}}(x,\xi)\end{equation}
So, it suffices to estimate $\widehat{v_{4,c}}$. To do this, we consider seperately the regions
\begin{equation}\frac{1}{\xi} \geq t+\sqrt{t}, \quad t+\sqrt{t}\geq \frac{1}{\xi} \geq t-\sqrt{t}, \quad t-\sqrt{t} \geq \frac{1}{\xi} \geq \frac{t}{2}, \quad \frac{t}{2} \geq \frac{1}{\xi}\geq \log^{N}(t)\end{equation}
Then, for example, in the case $\frac{1}{\xi} \geq t+\sqrt{t}$, we have
\begin{equation}\begin{split} \int_{0}^{\infty} J_{1}(r \xi) r v_{4,c}(t,r) dr&=\sum_{k=1}^{5}\int_{I_{k}} J_{1}(r\xi) r v_{4,c}(t,r) dr\end{split}\end{equation}
where
\begin{equation} I_{1}=[\frac{1}{\xi},\infty), \quad I_{2}=[t+\sqrt{t},\frac{1}{\xi}], \quad I_{3}=[t-\sqrt{t},t+\sqrt{t}], \quad I_{4}=[\frac{t}{2},t-\sqrt{t}], \quad I_{5}=[0,\frac{t}{2}]\end{equation}
and we use $$|J_{1}(x)| \leq \begin{cases} Cx, \quad 0<x<1\\
\frac{C}{\sqrt{x}}, \quad x>1\end{cases}$$
and \eqref{v4cest2}. The analogous decomposition is done for all cases of $\xi$ mentioned above.\\
\\
This procedure results in
\begin{equation}\label{v4chatest} |\widehat{v_{4,c}(t,\xi)}| \leq \begin{cases} \frac{C \xi \log(t) (\log(t)+|\log(\xi)|)}{t^{2}\log^{2b}(t)} + \frac{C \xi^{2}|\log(\xi)|\cdot |\log(1-t\xi)|}{t \log^{2b}(t)}, \quad \xi \leq \frac{1}{t+\sqrt{t}}\\
\frac{C \xi}{t^{2} \log^{2b-2}(t)}, \quad \frac{1}{t+\sqrt{t}} \leq \xi \leq \frac{1}{t-\sqrt{t}}\\
\frac{C\xi\log^{2}(t)}{t^{2}\log^{2b}(t)}, \quad \frac{1}{t-\sqrt{t}} \leq \xi \leq \frac{2}{t}\\
\frac{C \xi(|\log(\xi)|+\log(\log(t)))}{t^{2}\log^{3b}(t)}+\frac{C}{\sqrt{\xi}t^{7/2} \log^{2b-2}(t)}, \quad \frac{2}{t} \leq \xi \leq \frac{1}{\log^{N}(t)}\end{cases}\end{equation}
Rather than recording pointwise estimates on $\widehat{v_{4,c}}(t,\xi)$ in the region $\frac{1}{\xi} \leq \log^{N}(t)$, we use the following argument to infer an integral estimate on $\widehat{v_{4,c}}(t,\xi)$. From \eqref{v4cest2} and \eqref{drv4cforfourier}, 
\begin{equation}\begin{split} &||\left(\partial_{r}+\frac{1}{r}\right)v_{4,c}(t,r)||_{L^{2}(r dr)} \leq ||\partial_{r}v_{4,c}(t,r)||_{L^{2}(r dr)} + ||\frac{v_{4,c}(t,r)}{r}||_{L^{2}(r dr)}\\
&\leq \frac{C}{t^{2} \log^{3b+3N}(t)}\end{split}\end{equation}

On the other hand, 
\begin{equation}\begin{split} &\int_{0}^{\infty} J_{0}(r \xi) \left(\partial_{r}v_{4,c}(t,r) + \frac{v_{4,c}(t,r)}{r}\right) r dr = -\int_{0}^{\infty} v_{4,c}(t,r) J_{0}'(r\xi) \xi r dr = \xi \int_{0}^{\infty} v_{4,c}(t,r) J_{1}(r\xi) r dr \\
&=\xi \widehat{v_{4,c}}(t,\xi)\end{split}\end{equation}
where the vanishing of the boundary terms arising from integration by parts is justified by \eqref{v4cest2}.\\
By the $L^{2}$ isometry property of the Hankel transform of order $0$, this implies
\begin{equation}\label{xidxil2est}\frac{C}{t^{2} \log^{3b+3N}(t)} \geq ||\partial_{r}v_{4,c}(t,r) + \frac{v_{4,c}(t,r)}{r}||_{L^{2}(r dr)} = ||\xi \widehat{v_{4,c}}(t,\xi)||_{L^{2}(\xi d\xi)}\end{equation}

Then, we use $$|J_{1}(x)| + |J_{1}'(x)|\leq \frac{C}{\sqrt{x}}$$
in the formulae
\begin{equation} v_{4}(t,r) = \int_{t}^{\infty} dx \int_{0}^{\infty} d\xi J_{1}(r\xi) \sin((t-x)\xi) \widehat{v_{4,c}}(x,\xi)\end{equation}
and
\begin{equation} \partial_{r}v_{4}(t,r) = \int_{t}^{\infty} dx \int_{0}^{\infty} d\xi \xi J_{1}'(r\xi) \sin((t-x)\xi) \widehat{v_{4,c}}(x,\xi)\end{equation}
(Note that the differentiation under the integral sign is justified by the pointwise estimates on $\widehat{v_{4,c}}$, as well as \eqref{xidxil2est}).\\ 
Thus, we get
\begin{equation}\begin{split} |v_{4}(t,r)| &\leq C \int_{t}^{\infty} \frac{dx}{\sqrt{r}} \int_{0}^{\frac{1}{\log^{N}(x)}} d\xi \frac{|\widehat{v_{4,c}}(x,\xi)|}{\sqrt{\xi}} + C \int_{t}^{\infty} \frac{dx}{\sqrt{r}} \int_{\frac{1}{\log^{N}(x)}}^{\infty} d\xi \frac{|\widehat{v_{4,c}}(x,\xi)|}{\sqrt{\xi}} \cdot \frac{\xi^{3/2}}{\xi^{3/2}}\\
&\leq C \int_{t}^{\infty} \frac{dx}{\sqrt{r}} \int_{0}^{\frac{1}{\log^{N}(x)}} d\xi \frac{|\widehat{v_{4,c}}(x,\xi)|}{\sqrt{\xi}} + C \int_{t}^{\infty} \frac{dx}{\sqrt{r}} \left(\int_{\frac{1}{\log^{N}(x)}}^{\infty} \frac{d\xi}{\xi^{4}}\right)^{1/2} ||\xi \widehat{v_{4,c}}(x,\xi)||_{L^{2}(\xi d\xi)}\end{split}\end{equation}
and
\begin{equation}\begin{split} |\partial_{r}v_{4}(t,r)| &\leq C \int_{t}^{\infty} \frac{dx}{\sqrt{r}} \int_{0}^{\frac{1}{\log^{N}(x)}} d\xi |\widehat{v_{4,c}}(x,\xi)|\sqrt{\xi} + C \int_{t}^{\infty} \frac{dx}{\sqrt{r}} \int_{\frac{1}{\log^{N}(x)}}^{\infty} d\xi |\widehat{v_{4,c}}(x,\xi)|\sqrt{\xi} \cdot \frac{\xi}{\xi}\\
&\leq C \int_{t}^{\infty} \frac{dx}{\sqrt{r}} \int_{0}^{\frac{1}{\log^{N}(x)}} d\xi |\widehat{v_{4,c}}(x,\xi)|\sqrt{\xi} + C \int_{t}^{\infty} \frac{dx}{\sqrt{r}} \left(\int_{\frac{1}{\log^{N}(x)}}^{\infty} \frac{d\xi}{\xi^{2}}\right)^{1/2} ||\xi \widehat{v_{4,c}}(x,\xi)||_{L^{2}(\xi d\xi)}\end{split}\end{equation} 
Using our pointwise estimates on $\widehat{v_{4,c}}$, as well as \eqref{xidxil2est}, we get
\begin{equation} |v_{4}(t,r)| \leq \frac{C}{\sqrt{r} t \log^{\frac{3N}{2}+3b-1}(t)}, \quad r \geq \frac{t}{2}\end{equation}
and
\begin{equation} |\partial_{r}v_{4}(t,r)| \leq \frac{C}{\sqrt{r}t\log^{3b-1+\frac{5N}{2}}(t)}, \quad r \geq \frac{t}{2}\end{equation}
In addition, we have
\begin{equation} \partial_{r}v_{4}(t,r)+\frac{v_{4}(t,r)}{r} = \int_{t}^{\infty} dx \int_{0}^{\infty} \xi d\xi J_{0}(r\xi) \sin((t-x)\xi) \widehat{v_{4,c}}(x,\xi)\end{equation}

First using Minkowski's inequality, then the $L^{2}$ isometry property of the Hankel transform of order $0$, and then that of the Hankel transform of order $1$, we get
\begin{equation}\begin{split} ||\partial_{r}v_{4}(t,r)+\frac{v_{4}(t,r)}{r}||_{L^{2}(r dr)} &\leq \int_{t}^{\infty} dx ||\int_{0}^{\infty} \xi d\xi J_{0}(r\xi) \sin((t-x)\xi) \widehat{v_{4,c}}(x,\xi)||_{L^{2}(r dr)}\\
&\leq \int_{t}^{\infty} dx ||\sin((t-x)\xi) \widehat{v_{4,c}}(x,\xi)||_{L^{2}(\xi d\xi)} \leq \int_{t}^{\infty} dx ||\widehat{v_{4,c}}(x,\xi)||_{L^{2}(\xi d\xi)}\\
&\leq \int_{t}^{\infty} dx ||v_{4,c}(x)||_{L^{2}(r dr)} \leq \int_{t}^{\infty} dx \frac{C}{\log^{2N+3b}(x) x^{2}}\\
&\leq \frac{C}{t \log^{2N+3b}(t)} \end{split}\end{equation}
where we used \eqref{v4cest2}. Note that this  is simply the energy estimate for the equation solved by $v_{4}$. Next, we note that the pointwise estimates established for $v_{4}$ imply that 
$$v_{4}(t,0)=0, \quad \lim_{r \rightarrow \infty} v_{4}(t,r) =0$$
So, 
\begin{equation}\begin{split} ||\left(\partial_{r}+\frac{1}{r}\right)v_{4}||_{L^{2}(r dr)}^{2} &= \int_{0}^{\infty} \left((\partial_{r} v_{4})^{2}+ \frac{\partial_{r}(v_{4}^{2})}{r} + \frac{v_{4}^{2}}{r^{2}}\right) r dr\\
&=||\partial_{r}v_{4}||_{L^{2}(r dr)}^{2}+||\frac{v_{4}}{r}||_{L^{2}(r dr)}^{2}\end{split}\end{equation}
\\
We now treat $\partial_{t}v_{4}$: 
\begin{equation} \partial_{t}v_{4}(t,r) =\int_{t}^{\infty} dx \int_{0}^{\infty} d\xi J_{1}(r\xi) \cos((t-x) \xi) \xi \widehat{v_{4,c}}(x,\xi)\end{equation}
where the differentiation under the integral sign is again justified by the pointwise estimates on $\widehat{v_{4,c}}$ and \eqref{xidxil2est}.\\
This gives
\begin{equation} |\partial_{t}v_{4}(t,r)| \leq \frac{C}{\sqrt{r}}\int_{t}^{\infty} dx \int_{0}^{\infty} d\xi \sqrt{\xi} |\widehat{v_{4,c}}(x,\xi)|\end{equation}
So, the same exact procedure used for $\partial_{r}v_{4}$ pointwise estimates in the region $r \geq \frac{t}{2}$ also applies to $\partial_{t}v_{4}$.\\
Finally,
\begin{equation}\begin{split} ||\partial_{t}v_{4}(t,r)||_{L^{2}(r dr)} &\leq \int_{t}^{\infty} dx ||\int_{0}^{\infty} \xi d\xi J_{1}(r\xi) \cos((t-x)\xi) \widehat{v_{4,c}}(x,\xi)||_{L^{2}(r dr)}\\
&\leq \int_{t}^{\infty} dx ||\cos((t-x)\xi) \widehat{v_{4,c}}(x,\xi)||_{L^{2}(\xi d\xi)} \leq \int_{t}^{\infty} dx ||\widehat{v_{4,c}}(x,\xi)||_{L^{2}(\xi d\xi)}\\
&\leq \int_{t}^{\infty} dx ||v_{4,c}(x)||_{L^{2}(r dr)} \leq \int_{t}^{\infty} dx \frac{C}{\log^{2N+3b}(x) x^{2}}\\
&\leq \frac{C}{t \log^{2N+3b}(t)} \end{split}\end{equation} 
This completes the proof of the lemma. 
\end{proof}
\subsubsection{Pointwise estimates on $v_{5}$}
\begin{lemma} For all $\lambda$ of the form
$$\lambda(t) = \lambda_{0}(t)+e(t), \quad e \in \overline{B_{1}(0)} \subset X$$
we have the pointwise estimates
\begin{equation}\label{v5finalest} |v_{5}(t,r)| \leq \begin{cases} \frac{C r}{t^{7/2} \log^{3b-3+\frac{5N}{2}}(t)}, \quad r \leq \frac{t}{2}\\
\frac{C \log^{4}(t)}{\sqrt{r} t^{3/2}}, \quad r > \frac{t}{2}\end{cases}\end{equation}

\begin{equation}\label{drv5finalest} |\partial_{r}v_{5}(t,r)| \leq \begin{cases} \frac{C}{t^{7/2} \log^{3b-3+\frac{5N}{2}}(t)}, \quad r \leq \frac{t}{2}\\
\frac{C \log^{3}(t)}{\sqrt{r} t^{3/2}}, \quad r > \frac{t}{2}\end{cases}\end{equation}

\begin{equation} |\partial_{t}v_{5}(t,r)| \leq \frac{C \log^{3}(t)}{\sqrt{r} t^{3/2}}, \quad r > \frac{t}{2}\end{equation}
In addition, we have the $L^{2}$ estimates
\begin{equation}\label{v5energyest} ||\partial_{t}v_{5}||_{L^{2}(r dr)} + ||\partial_{r}v_{5}||_{L^{2}(r dr)}+||\frac{v_{5}}{r}||_{L^{2}(r dr)} \leq C \frac{\log^{3}(t)}{t^{7/4}}\end{equation}

\end{lemma}
\begin{proof}
We start with the region $\frac{t}{2}>r >0$. In order to ease notation, let $x \in \mathbb{R}^{2}$ be defined by $x = r \textbf{e}_{1}$. Then,
\begin{equation}\begin{split} v_{5}(t,r) = \frac{-r}{2 \pi} \int_{0}^{1} d\beta \int_{t}^{\infty} ds \int_{B_{s-t}(0)} \frac{dA(y)}{\sqrt{(s-t)^{2}-|y|^{2}}} &\left(\frac{\partial_{2}N_{2}(f_{v_{5}})(s,|\beta x+y|)\left((\beta x+y)\cdot \hat{x}\right)^{2}}{|\beta x+y|^{2}}\right.\\
&-\frac{N_{2}(f_{v_{5}})(s,|\beta x+y|)\left((\beta x+y)\cdot \hat{x}\right)^{2}}{|\beta x+y|^{3}}\\
&\left. + \frac{N_{2}(f_{v_{5}})(s,|\beta x+y|)}{|\beta x+y|}\right)\end{split}\end{equation}
We then decompose $$v_{5}(t,r)=v_{5,1}(t,r)+v_{5,2}(t,r)$$
where
\begin{equation}\begin{split} v_{5,1}(t,r) = \frac{-r}{2 \pi} \int_{0}^{1} d\beta \int_{t}^{\infty} ds \int_{B_{s-t}(0)\cap B_{\frac{s}{2}}(-\beta x)} \frac{dA(y)}{\sqrt{(s-t)^{2}-|y|^{2}}} &\left(\frac{\partial_{2}N_{2}(f_{v_{5}})(s,|\beta x+y|)\left((\beta x+y)\cdot \hat{x}\right)^{2}}{|\beta x+y|^{2}}\right.\\
&-\frac{N_{2}(f_{v_{5}})(s,|\beta x+y|)\left((\beta x+y)\cdot \hat{x}\right)^{2}}{|\beta x+y|^{3}}\\
&\left. + \frac{N_{2}(f_{v_{5}})(s,|\beta x+y|)}{|\beta x+y|}\right)\end{split}\end{equation}
and
\begin{equation} v_{5,2}(t,r) = v_{5}(t,r)-v_{5,1}(t,r)\end{equation}
Finally, we use the pointwise estimates on $v_{1},v_{2},v_{3},v_{4}$ to record pointwise estimates on $N_{2}(f_{v_{5}})$:
We start with
\begin{equation} |N_{2}(f)(t,r)| \leq \frac{C r(f(t,r))^{2}}{\lambda(t) (1+\frac{r^{2}}{\lambda(t)^{2}}) r^{2}} + C \frac{|f(t,r)|^{3}}{r^{2}}\end{equation}
The first set of estimates we will require on $N_{2}(f_{v_{5}})$ concern the regions $r \leq \frac {t}{2}$, and $t > r > \frac{t}{2}$. In the region $r \leq \frac{t}{2}$, we get
\begin{equation} |N_{2}(f_{v_{5}})(t,r)| \leq \frac{C r}{(\lambda(t)^{2}+r^{2})t^{4}\log^{3b}(t)}+\frac{C r}{t^{6}\log^{3b}(t)}, \quad r \leq \frac{t}{2}\end{equation}

In the region $t > r > \frac{t}{2}$, we use \eqref{v2singularconeest} to estimate $v_{2}$ in the region $t > r \geq \frac{t}{2}$, exactly as was done while studying $v_{4}$, and we get
\begin{equation} |N_{2}(f_{v_{5}})(t,r)| \leq \frac{C \log^{3}(r)}{r^{2}|t-r|^{3}} + \frac{C \lambda(t)}{r^{7/2}t^{5/2}\log^{3N+6b-2}(t)}, \quad t> r \geq \frac{t}{2}\end{equation}

Next, we consider $\partial_{r}N_{2}(f_{v_{5}})$:
\begin{equation} |\partial_{r}(N_{2}(f))(t,r)| \leq \frac{C (f(t,r))^{2}}{r^{2}\lambda(t)\left(1+\frac{r^{2}}{\lambda(t)^{2}}\right)} +  \frac{C|f(t,r) \partial_{r}f(t,r)|}{\lambda(t) r \left(1+\frac{r^{2}}{\lambda(t)^{2}}\right)} +  \frac{C|f(t,r)|^{3}}{r^{3}} +  \frac{C|f(t,r)^{2} \partial_{r}f(t,r)|}{r^{2}}\end{equation}
We again treat the regions $t > r > \frac{t}{2}$ and $r \leq \frac{t}{2}$. We get
\begin{equation}|\partial_{r}N_{2}(f_{v_{5}})(t,r)| \leq  \frac{C}{t^{4}\log^{3b}(t) (\lambda(t)^{2}+r^{2})}+\frac{C}{t^{6}\log^{3b}(t)}, \quad r \leq \frac{t}{2}\end{equation}

\begin{equation}\begin{split}|\partial_{r}N_{2}(f_{v_{5}})(t,r)| &\leq \frac{C}{r^{3} t^{3} \log^{4N+7b-2}(t)}+\frac{C \lambda(t) \log(r)}{r^{3} t^{3/2} |t-r| \log^{3b-1+\frac{5N}{2}}(t)} + \frac{C \log^{2}(r)}{r^{2} t^{3/2}(t-r)^{2} \log^{3b-1+\frac{5N}{2}}(t)}\\
&+ \frac{C \log^{3}(r)}{r^{2}(t-r)^{4}}, \quad t> r \geq \frac{t}{2}\end{split}\end{equation}
where we used
$$\frac{1}{r} \leq \frac{1}{|t-r|}, \quad r\geq \frac{t}{2}$$
So,
\begin{equation} \label{v51near0} |v_{5,1}(t,r)| \leq C r \int_{0}^{1} d\beta \int_{t}^{\infty} ds \int_{B_{s-t}(0) \cap B_{\frac{s}{2}}(-\beta x)} \frac{dA(y)}{\sqrt{(s-t)^{2}-|y|^{2}}} \left(\frac{1}{s^{4}\log^{b}(s)\left(1+\frac{|\beta x+y|^{2}}{\lambda(s)^{2}}\right)} +\frac{1}{s^{6}\log^{3b}(s)}\right)\end{equation}

We treat several terms comprising \eqref{v51near0} sperately. First, we have
\begin{equation}\begin{split} & r \int_{0}^{1} d\beta \int_{t}^{t+\frac{1}{2}} ds \int_{B_{s-t}(0) \cap B_{\frac{s}{2}}(-\beta x)} \frac{dA(y)}{(s-t)} \frac{1}{s^{4}\log^{b}(s)\left(1+\frac{|\beta x+y|^{2}}{\lambda(s)^{2}}\right)}\\
&\leq C r \int_{0}^{1} d\beta \int_{t}^{t+\frac{1}{2}} ds \int_{0}^{s-t} \frac{\rho d\rho}{(s-t)} \int_{0}^{2\pi} \frac{d\theta}{s^{4} \log^{b}(s)}\\
&\leq \frac{C r}{t^{4} \log^{b}(t)}\end{split}\end{equation} 

Next,
\begin{equation}\begin{split} &r \int_{0}^{1} d\beta \int_{t+\frac{1}{2}}^{\infty} ds \int_{B_{s-t}(0) \cap B_{\frac{s}{2}}(-\beta x)} \frac{dA(y)}{(s-t)} \frac{1}{s^{4}\log^{b}(s)\left(1+\frac{|\beta x+y|^{2}}{\lambda(s)^{2}}\right)}\\
&\leq C r \int_{0}^{1}d\beta \int_{t+\frac{1}{2}}^{\infty} \frac{ds}{(s-t)} \int_{0}^{s-t} \rho d\rho \int_{0}^{2\pi} \frac{d\theta}{s^{2}(\rho+t)^{2} \log^{b}(s)(1+\beta^{2}r^{2}+\rho^{2}+2\beta r \rho\cos(\theta))}\\
&\leq C r \int_{0}^{1}d\beta \int_{t+\frac{1}{2}}^{\infty} ds \frac{(\log(2+r^{2}\beta^{2})+\log(2+t^{2}))}{(s-t) s^{2} \log^{b}(s) t^{2}}\\
&\leq \frac{C r}{t^{4} \log^{b-2}(t)}, \quad r \leq \frac{t}{2}\end{split}\end{equation}
where we used the same procedure that we used in \eqref{term1b}.\\
\\
The third term to consider is
\begin{equation}\begin{split} &r \int_{0}^{1} d\beta \int_{t}^{\infty} ds \int_{B_{s-t}(0) \cap B_{\frac{s}{2}}(-\beta x)} dA(y) \left(\frac{1}{\sqrt{(s-t)^{2}-|y|^{2}}}-\frac{1}{(s-t)}\right) \left(\frac{1}{s^{4}\log^{b}(s)\left(1+\frac{|\beta x+y|^{2}}{\lambda(s)^{2}}\right)}\right)\\
&\leq C r \int_{0}^{1}d\beta \int_{0}^{\infty} \frac{\rho d\rho}{(\rho+t)^{2}} \int_{0}^{2\pi} \frac{d\theta}{(1+\beta^{2}r^{2}+\rho^{2}+2 \beta r \rho \cos(\theta))} \int_{\rho+t}^{\infty} ds \left(\frac{1}{\sqrt{(s-t)^{2}-\rho^{2}}}-\frac{1}{(s-t)}\right) \frac{1}{s^{2} \log^{b}(s)}\\
&\leq C r \int_{0}^{1} d\beta \left(\frac{\log(2+r^{2}\beta^{2})+\log(2+t^{2})}{t^{2}}\right) \frac{1}{t^{2}\log^{b}(t)}\\
&\leq \frac{C r}{t^{4} \log^{b-1}(t)}, \quad r \leq \frac{t}{2}\end{split}\end{equation}

Finally, the last term to consider is
\begin{equation}\begin{split} &r \int_{0}^{1} d\beta \int_{t}^{\infty} ds \int_{B_{s-t}(0) \cap B_{\frac{s}{2}}(-\beta x)} \frac{dA(y)}{\sqrt{(s-t)^{2}-|y|^{2}}} \frac{1}{s^{6}\log^{3b}(s)}\\
&\leq C r \int_{0}^{1} d\beta \int_{t}^{\infty} ds \frac{1}{s^{6}\log^{3b}(s)} \int_{0}^{s-t} \rho d\rho \int_{0}^{2 \pi} \frac{d\theta}{\sqrt{(s-t)^{2}-\rho^{2}}}\\
&\leq C r \int_{t}^{\infty} ds  \frac{(s-t)}{s^{6}\log^{3b}(s)}\\
&\leq \frac{C r}{t^{4} \log^{3b}(t)}\end{split}\end{equation}

Combining the above estimates, we conclude:
\begin{equation}\begin{split} |v_{5,1}(t,r)| &\leq C \frac{r}{t^{4}\log^{b-2}(t)}, \quad r \leq \frac{t}{2} \end{split} \end{equation}
Next, we treat $v_{5,2}$.
\begin{equation}\begin{split} v_{5,2}(t,r) = \frac{-r}{2\pi} \int_{0}^{1}d\beta \int_{t}^{\infty} ds \int_{B_{s-t}(0) \cap (B_{\frac{s}{2}}(-\beta x))^{c}} \frac{dA(y)}{\sqrt{(s-t)^{2}-|y|^{2}}} &\left(\frac{\partial_{2}N_{2}(f)(s,|\beta x+y|) \left((\beta x+y)\cdot \hat{x}\right)^{2}}{|\beta x+y|^{2}}\right. \\
&- \frac{N_{2}(f)(s,|\beta x+y|)\left((\beta x+y)\cdot \hat{x}\right)^{2}}{|\beta x+y|^{3}} \\
&\left.+ \frac{N_{2}(f)(s,|\beta x+y|)}{|\beta x+y|}\right)\end{split}\end{equation}
So,
\begin{equation} \begin{split} &|v_{5,2}(t,r)| \\
&\leq C r \int_{0}^{1} d\beta \int_{t}^{\infty} ds \int_{B_{s-t}(0) \cap (B_{\frac{s}{2}}(-\beta x))^{c}} \frac{dA(y)}{\sqrt{(s-t)^{2}-|y|^{2}}} \left(|\partial_{2}N_{2}(f)(s,|\beta x+y|)|+\frac{|N_{2}(f)(s,|\beta x+y|)|}{|\beta x+y|}\right)\\
&\leq C r \int_{0}^{1}d\beta \int_{t}^{\infty} ds \int_{B_{s-t}(0)\cap(B_{\frac{s}{2}}(-\beta x))^{c}}\frac{dA(y)}{\sqrt{(s-t)^{2}-|y|^{2}}}\text{integrand}_{v_{5,2}}(s,|\beta x+y|)\end{split}\end{equation}
where
\begin{equation}\begin{split} \text{integrand}_{v_{5,2}}(s,|\beta x+y|)&= \frac{1}{|\beta x+y|^{3} s^{3} \log^{4N+7b-2}(s)} + \frac{\lambda(s) \log(|\beta x+y|)}{|\beta x +y|^{3} |s-|\beta x+y|| s^{3/2} \log^{3b-1+\frac{5N}{2}}(s)}\\
&+\frac{\log^{2}(|\beta x+y|)}{|\beta x+y|^{2} s^{3/2} (s-|\beta x+y|)^{2} \log^{3b-1+\frac{5N}{2}}(s)} + \frac{\log^{3}(|\beta x +y|)}{|\beta x+y|^{2}(s-|\beta x+y|)^{4}}\end{split}\end{equation}

Exactly as in the case of $v_{4}$, we note that, for $z \in B_{s-t}(\beta x) \cap (B_{\frac{s}{2}}(0))^{c}$, we have
\begin{equation} |z| =|z-\beta x+\beta x| \leq |z-\beta x| + r < s-t+r \leq s-\frac{t}{2}\end{equation}
So,
\begin{equation} s-|z| \geq \frac{t}{2}\end{equation}
We use this estimate for every term in $\text{integrand}_{v_{5,2}}$, except for the term $\frac{\log^{3}(|\beta x+y|)}{|\beta x+y|^{2}(s-|\beta x+y|)^{4}} \leq \frac{C \log^{3}(s)}{s^{2} t^{3}(s-|\beta x+y|)}$.
Then, 
\begin{equation}\label{v52intest} \begin{split} &|v_{5,2}(t,r)| \\
&\leq C r \int_{0}^{1}d\beta \int_{t}^{\infty} ds \int_{0}^{s-t} \frac{\rho d\rho}{\sqrt{(s-t)^{2}-\rho^{2}}} \int_{0}^{2\pi} d\theta \left(\frac{ \log^{2}(s)}{s^{7/2} t^{2} \log^{3b-1+\frac{5N}{2}}(s)}+\frac{\log^{3}(s)}{s^{3}t^{3}}\right)\\
&+C r \int_{0}^{1}d\beta \int_{t}^{\infty} ds \int_{0}^{s-t} \frac{\rho d\rho}{\sqrt{(s-t)^{2}-\rho^{2}}} \int_{0}^{2\pi}d\theta \frac{\log^{3}(s)}{s^{2} t^{3}(s-\sqrt{\beta^{2}r^{2}+\rho^{2}+2 \beta r \rho \cos(\theta)})}\\
&\leq \frac{C r}{t^{7/2} \log^{3b-3+\frac{5N}{2}}(t)}\\
&+C r \int_{0}^{1}d\beta \int_{t}^{\infty} ds \int_{0}^{s-t} \frac{\rho d\rho}{\sqrt{(s-t)^{2}-\rho^{2}}} \int_{0}^{2\pi}d\theta \frac{\log^{3}(s)}{s^{2}t^{3}(s-\sqrt{\beta^{2}r^{2}+\rho^{2}+2 \beta r \rho \cos(\theta)})}\end{split}\end{equation}
We need only continue to estimate the last line of \eqref{v52intest}.
\begin{equation}\label{v5specialterm} \begin{split} & r \int_{0}^{1} d\beta \int_{t}^{\infty} ds \int_{0}^{s-t} \frac{\rho d\rho}{ \sqrt{(s-t)^{2}-\rho^{2}}} \int_{0}^{2\pi} d\theta \frac{\log^{3}(s)}{s^{2} t^{3}(s-\sqrt{\beta^{2}r^{2}+\rho^{2}+2 \beta r \rho \cos(\theta)})}\\
&\leq C r \int_{0}^{1} d\beta \int_{0}^{\infty} \rho d\rho \int_{0}^{2 \pi}\frac{d\theta}{t^{3}} \int_{t+\rho}^{\infty} \frac{ds}{\sqrt{(s-t)^{2}-\rho^{2}}} \frac{\log^{3}(s)}{s^{2}} \frac{1}{(s-\sqrt{\rho^{2}+2 \rho \beta r \cos(\theta) + \beta^{2}r^{2}})}\\
&\leq C r \int_{0}^{1} d\beta \int_{0}^{\infty} \rho d\rho \int_{0}^{2\pi} \frac{d\theta}{t^{3}} \int_{t+\rho}^{\infty} \frac{ds}{\sqrt{s-t+\rho}} \frac{1}{\sqrt{s-t-\rho}} \frac{\log^{3}(s)}{s^{2}} \frac{1}{(s-\sqrt{\rho^{2}+2 \rho \beta r \cos(\theta) + \beta^{2}r^{2}})}\\
&\leq C r \int_{0}^{1} d\beta \int_{0}^{\infty} \rho d\rho \int_{0}^{2 \pi} \frac{d\theta}{t^{3}} \frac{1}{\sqrt{\rho}} \frac{\log^{3}(t+\rho)}{(t+\rho)^{2}} \int_{t+\rho}^{\infty} \frac{ds}{\sqrt{s-t-\rho}} \frac{1}{(s-\sqrt{\rho^{2}+2 \rho \beta r \cos(\theta) + \beta^{2}r^{2}})}\\
&\leq C r \int_{0}^{1} d\beta \int_{0}^{\infty} \frac{\sqrt{\rho} \log^{3}(t+\rho)}{(t+\rho)^{2}} d\rho \int_{0}^{2 \pi} \frac{d\theta}{t^{3}} \frac{1}{\sqrt{t}}\\
&\leq C \frac{r \log^{3}(t)}{t^{4}}\end{split}\end{equation}
So, we get
\begin{equation} |v_{5,2}(t,r)| \leq \frac{C r}{t^{7/2} \log^{3b-3+\frac{5N}{2}}(t)}, \quad r \leq \frac{t}{2}\end{equation}

In total, we have
\begin{equation}\begin{split} |v_{5}(t,r)| &\leq \frac{C r}{t^{7/2}\log^{3b-3+\frac{5N}{2}}(t)}, \quad r \leq \frac{t}{2}\end{split}\end{equation}\\
\\
Exactly as for $\partial_{r}v_{4}$, we have
\begin{equation}\label{drv5formula}\begin{split} |\partial_{r}v_{5}(t,r)| \leq C \int_{t}^{\infty} ds \int_{0}^{s-t} &\frac{\rho d\rho}{\sqrt{(s-t)^{2}-\rho^{2}}} \int_{0}^{2\pi} d\theta \\
&\left(|\partial_{2}N_{2}(f_{v_{5}})(s,\sqrt{r^{2}+\rho^{2}+2 r \rho \cos(\theta)})| + \frac{|N_{2}(f_{v_{5}})(s,\sqrt{r^{2}+\rho^{2}+2 r \rho \cos(\theta)})|}{\sqrt{r^{2}+\rho^{2}+2 r \rho \cos(\theta)}}\right)\end{split}\end{equation}
As was also the case for $v_{4}$, the integrals appearing in \eqref{drv5formula} were already estimated above. So, we get
\begin{equation}\label{drv5near0} |\partial_{r}v_{5}(t,r)| \leq \frac{C}{t^{7/2}\log^{3b-3+\frac{5N}{2}}(t)}, \quad r \leq \frac{t}{2}\end{equation}
We will now prove estimates on $v_{5}(t,r)$ and $\partial_{r}v_{5}$ in the region $r \geq \frac{t}{2}$. For the next steps, we will use slightly different combinations of the estimates \eqref{v2singularconeest} and \eqref{v2sqrtrest} for $v_{2}$ in various subsets of the region $r \geq \frac{t}{2}$, in order to estimate $N_{2}(f_{v_{5}})$, resulting in
\begin{equation}\label{n2estpart2} |N_{2}(f_{v_{5}})(t,r)| \leq \begin{cases} \frac{C r}{(\lambda(t)^{2}+r^{2})t^{4}\log^{3b}(t)}+\frac{C r}{t^{6}\log^{3b}(t)}, \quad r \leq \frac{t}{2}\\
\frac{C \log^{3}(r)}{r^{2}|t-r|^{3}} + \frac{C}{r^{7/2}t^{5/2}\log^{3N+7b-2}(t)}, \quad \frac{t}{2} \leq r \leq t-\sqrt{t} \text{ or } r\geq t+\sqrt{t}\\
\frac{C}{r^{7/2}}, \quad t-\sqrt{t} \leq r \leq t+\sqrt{t}\end{cases}\end{equation}
We will also need estimates on $\partial_{r}N_{2}(f_{v_{5}})$. To obtain these, we use  \eqref{v2precisenearorigin} in the region $r \leq \frac{t}{2}$. In the region $r > \frac{t}{2}$, we use \eqref{v2sqrtrest} in the region $t-\sqrt{t} \leq r \leq t+\sqrt{t}$, and \eqref{v2singularconeest} in the regions $\frac{t}{2} < r < t-\sqrt{t}$ and $r>t+\sqrt{t}$ for $v_{2}$. For $\partial_{r}v_{2}$, we use \eqref{v2sqrtrest} in the region $t-t^{1/4} \leq r \leq t+t^{1/4}$ and \eqref{v2singularconeest} in the regions $\frac{t}{2} \leq r \leq t-t^{1/4}$ and $t+t^{1/4} \leq r $. Then, we get
\begin{equation} |\partial_{r}N_{2}(f_{v_{5}})(t,r)| \leq  \frac{C}{t^{4}\log^{3b}(t) (\lambda(t)^{2}+r^{2})}+\frac{C}{t^{6}\log^{3b}(t)}, \quad r \leq \frac{t}{2}\end{equation}
\begin{equation}\begin{split}|\partial_{r}N_{2}(f_{v_{5}})(t,r)| &\leq \frac{C}{r^{3} t^{3} \log^{4N+7b-2}(t)} + \frac{C \log(r)}{r^{3} |t-r| t^{3/2} \log^{4b-1+\frac{5N}{2}}(t)} + \frac{C \log^{2}(r)}{r^{2} t^{3/2} (t-r)^{2} \log^{3b-1+\frac{5N}{2}}(t)}\\
&+\frac{C \log^{3}(r)}{r^{2}(t-r)^{4}},\quad t-\sqrt{t} > r > \frac{t}{2}\text{  or  } r > t+\sqrt{t}\end{split}\end{equation}
\begin{equation}\begin{split}|\partial_{r}N_{2}(f_{v_{5}})(t,r)| &\leq \begin{cases} \frac{C \log(r)}{r^{3}(t-r)^{2}}, \quad t-\sqrt{t} \leq r \leq t-t^{1/4} \text{, or } t+t^{1/4} \leq r \leq t+\sqrt{t}\\
\frac{C}{r^{7/2}}, \quad t-t^{1/4} \leq r \leq t+t^{1/4}\end{cases}\end{split}\end{equation}
Then, we have
\begin{equation} v_{5}(t,r) = \int_{t}^{\infty} ds \int_{0}^{\infty} d\xi J_{1}(r\xi) \sin((t-s)\xi) \widehat{N_{2}(f_{v_{5}})}(s,\xi)\end{equation}
and we will prove estimates on $\widehat{N_{2}(f_{v_{5}})}(t,\xi)$, for $\frac{1}{\xi} \geq 1$. We consider seperately the cases
$$\frac{1}{\xi} \geq t+\sqrt{t}, \quad t-\sqrt{t} \leq \frac{1}{\xi} \leq t+\sqrt{t}, \quad \frac{t}{2} \leq \frac{1}{\xi} \leq t-\sqrt{t}, \quad 1 \leq \frac{1}{\xi} \leq \frac{t}{2}$$
and proceed in exactly the same manner as was done for $\widehat{v_{4,c}}$.\\
\\
This results in
\begin{equation}\begin{split} |\widehat{N_{2}(f_{v_{5}})}(t,\xi)| &\leq \frac{C \xi}{t} \left(\log^{3}(t)+|\log(\xi)|^{3}\right), \quad \frac{1}{\xi} \geq t+\sqrt{t}\end{split}\end{equation}
\begin{equation}\begin{split} |\widehat{N_{2}(f_{v_{5}})}(t,\xi)| &\leq \frac{C \log^{3}(t)}{t^{2}}, \quad t-\sqrt{t} \leq \frac{1}{\xi} \leq t+\sqrt{t}\end{split}\end{equation}
\begin{equation}\begin{split} |\widehat{N_{2}(f_{v_{5}})}(t,\xi)| &\leq \frac{C\log^{3}(t)}{t^{2}}, \quad \frac{t}{2} \leq \frac{1}{\xi} \leq t-\sqrt{t}\end{split}\end{equation}
\begin{equation}\begin{split} |\widehat{N_{2}(f_{v_{5}})}(t,\xi)| &\leq  \frac{C}{\xi t^{4} \log^{3b}(t)} + \frac{C}{\xi^{3} t^{6} \log^{3b}(t)} + \frac{C \log^{3}(t)}{\sqrt{\xi} t^{5/2}} , \quad 1 \leq \frac{1}{\xi} \leq \frac{t}{2}\end{split}\end{equation}

For the region $\frac{1}{\xi} <1$, we again use the following argument:\\
From \eqref{n2estpart2} and the $\partial_{r}N_{2}(f_{v_{5}})$ estimates which follow it, 
\begin{equation}\begin{split} &||\left(\partial_{r}+\frac{1}{r}\right)N_{2}(f_{v_{5}})(t,r)||_{L^{2}(r dr)} \leq ||\partial_{r}N_{2}(f_{v_{5}})(t,r)||_{L^{2}(r dr)} + ||\frac{N_{2}(f_{v_{5}})(t,r)}{r}||_{L^{2}(r dr)}\\
&\leq \frac{C \log(t)}{t^{23/8}}\end{split}\end{equation}
Then, the same procedure used for $v_{4,c}$ gives

\begin{equation}\label{xidxil2estn2}\frac{C \log(t)}{t^{23/8}} \geq ||\partial_{r}N_{2}(f_{v_{5}})(t,r) + \frac{N_{2}(f_{v_{5}})(t,r)}{r}||_{L^{2}(r dr)} = ||\xi \widehat{N_{2}(f_{v_{5}})}(t,\xi)||_{L^{2}(\xi d\xi)}\end{equation}

Now, we return to \begin{equation} v_{5}(t,r) = \int_{t}^{\infty} ds \int_{0}^{\infty} d\xi J_{1}(r\xi) \sin((t-s)\xi) \widehat{N_{2}(f_{v_{5}})}(s,\xi)\end{equation}
use $$|J_{1}(x)| \leq \frac{C}{\sqrt{x}}$$ 
and the same procedure used for $v_{4}$:
\begin{equation}\begin{split} &|v_{5}(t,r)| \leq \frac{C}{\sqrt{r}} \int_{t}^{\infty} ds \int_{0}^{1} \frac{d\xi}{\sqrt{\xi}} |\widehat{N_{2}(f_{v_{5}})}(s,\xi)| + \frac{C}{\sqrt{r}} \int_{t}^{\infty} ds \left(\int_{1}^{\infty} \frac{d\xi}{\xi^{4}}\right)^{1/2} ||\xi \widehat{N_{2}(f_{v_{5}})}(s,\xi)||_{L^{2}(\xi d\xi)}\\
&\leq \frac{C}{\sqrt{r}} \frac{\log^{4}(t)}{t^{3/2}}\end{split}\end{equation}

Next, the pointwise estimates on $\widehat{N_{2}(f_{v_{5}})}$, as well as \eqref{xidxil2estn2} justify the following differentiation under the integral sign:
\begin{equation} \partial_{r}v_{5}(t,r) = \int_{t}^{\infty} ds \int_{0}^{\infty} d\xi \xi J_{1}'(r\xi) \sin((t-s)\xi) \widehat{N_{2}(f_{v_{5}})}(s,\xi)\end{equation}
and we also have
$$|J_{1}'(x)| \leq \frac{C}{\sqrt{x}}$$
So,
\begin{equation}\begin{split} &|\partial_{r}v_{5}(t,r)| \leq \frac{C}{\sqrt{r}}\int_{t}^{\infty} ds \int_{0}^{1} d\xi \sqrt{\xi} |\widehat{N_{2}(f_{v_{5}})}(s,\xi)| + \frac{C}{\sqrt{r}}\int_{t}^{\infty} ds ||\xi \widehat{N_{2}(f_{v_{5}})}(s)||_{L^{2}(\xi d\xi)}\\
&\leq \frac{C \log^{3}(t)}{\sqrt{r}t^{3/2}}\end{split}\end{equation}

For $\partial_{t}v_{5}$, we can similarly differentiate under the integral sign to get
\begin{equation} \partial_{t}v_{5}(t,r) = \int_{t}^{\infty} dx \int_{0}^{\infty} d\xi J_{1}(r \xi) \cos((t-x)\xi) \xi \widehat{N_{2}(f_{v_{5}})}(x,\xi)\end{equation}
so, the same argument used for $\partial_{r}v_{5}$ also applies to $\partial_{t}v_{5}$.

Finally, the same procedure used for the energy estimates for $v_{4}$ also applies here, to give
\begin{equation}\begin{split} ||\partial_{t}v_{5}||_{L^{2}(r dr)} + ||\left(\partial_{r}+\frac{1}{r}\right)v_{5}||_{L^{2}(r dr)} &\leq \int_{t}^{\infty} ||N_{2}(f_{v_{5}})(x,r)||_{L^{2}(r dr)} dx \\
&\leq C \frac{\log^{3}(t)}{t^{7/4}}\end{split}\end{equation}
(We again have
$$||\left(\partial_{r}+\frac{1}{r}\right)v_{5}||_{L^{2}(r dr)}^{2}=||\partial_{r}v_{5}||_{L^{2}(r dr)}^{2}+||\frac{v_{5}}{r}||^{2}_{L^{2}(r dr)}$$
by the same argument used while studying $v_{4}$).
\end{proof}
\subsubsection{Solving the modulation equation}

The main result of this section is to prove Proposition \ref{lambdaexist}, which we recall is the following statement:\\
There exists $T_{3}>0$ such that, for all $T_{0} \geq T_{3}$, there exists $\lambda \in C^{2}([T_{0},\infty))$ which solves \eqref{modulationfinal}. Moreover, 
\begin{equation} \lambda(t) = \lambda_{0}(t) + e_{0}(t), \quad ||e_{0}||_{X} \leq 1\end{equation}
($||\cdot||_{X}$ was defined in \eqref{ynorm}).\\
\\
We use our pointwise estimates on $v_{k},E_{5}$ to get
\begin{equation} \begin{split} &|-\lambda(t) \langle \left(\frac{\cos(2Q_{\frac{1}{\lambda(t)}})-1}{r^{2}}\right)v_{4}\left(1-\chi_{\geq 1}(\frac{4r}{t})\right)\vert_{r=R\lambda(t)},\phi_{0}\rangle|\leq  \frac{C}{\lambda(t)} \int_{0}^{\frac{t}{2 \lambda(t)}} \frac{R^{2}}{(1+R^{2})^{3}} |v_{4}(t,R\lambda(t))| dR\\
&\leq \frac{C}{t^{2}\log^{3b+2N-1}(t)}\end{split}\end{equation}
\begin{equation} \begin{split} |-\lambda(t) \langle \left(\frac{\cos(2Q_{\frac{1}{\lambda(t)}})-1}{r^{2}}\right)v_{5}\left(1-\chi_{\geq 1}(\frac{4r}{t})\right)\vert_{r=R\lambda(t)},\phi_{0}\rangle|&\leq \frac{C}{t^{7/2} \log^{3b-3+\frac{5N}{2}}(t)}\end{split}\end{equation}
\begin{equation} \begin{split} |-\lambda(t) \langle \left(\frac{\cos(2Q_{\frac{1}{\lambda(t)}})-1}{r^{2}}\right)E_{5}\vert_{r=R\lambda(t)},\phi_{0}\rangle|&\leq \frac{C}{t^{2}\log^{b+1}(t)}\end{split}\end{equation}
\begin{equation} \begin{split} &|-\lambda(t) \langle \left(\left(\frac{\cos(2Q_{\frac{1}{\lambda(t)}})-1}{r^{2}}\right)\left(v_{1}+v_{2}+v_{3}\right)+F_{0,2}\right)\chi_{\geq 1}(\frac{2r}{\log^{N}(t)}) \vert_{r=R\lambda(t)},\phi_{0}\rangle|\\
&\leq C\lambda(t) \int_{0}^{\infty} |v_{4,c}(t,R\lambda(t))| \phi_{0}(R) R dR \leq \frac{C}{t^{2}\log^{3b+2N}(t)}\end{split}\end{equation}

Combining the above with our previous estimates, there exists a constant $D_{1}>0$ such that
$$|G(t,\lambda(t))| \leq \frac{D_{1}}{t^{2}\log^{b+1}(t)}, \quad \lambda = \lambda_{0} + e, \quad e \in \overline{B}_{1}(0) \subset X$$ 
Then, combining this with the estimate of all the terms comprising $RHS$ except for $G$, we get, for a constant $D_{2}$ independent of $e$, $T_{0}$:
\begin{equation} |RHS(e,t)| \leq \frac{D_{2}}{\log(\log(t))t^{2}\log^{b+1}(t)}, \quad e \in \overline{B}_{1}(0) \subset X\end{equation}
Let $T_{0,2}>e^{e^{900}}$ be such that the following three inequalities are true:
\begin{equation}\frac{1}{\sqrt{\log(\log(T_{0,2}))}} \leq \frac{1}{100} \cdot \frac{\alpha}{3 D_{2}}\end{equation}
\begin{equation} \frac{(b+1)}{\log(T_{0,2})} + \frac{1}{2 \log(T_{0,2}) \log(\log(T_{0,2}))} \leq \frac{1}{100}\end{equation}
\begin{equation}\begin{split}&\frac{1}{2 b \log(\log(T_{0,2}))}+\frac{b}{\log(T_{0,2})}+\frac{b}{2(b+1)\log(T_{0,2})\log(\log(T_{0,2}))} \leq \frac{1}{100}\end{split}\end{equation}
Finally, there exists $T_{0,3}>700+700b$, such that
$$\frac{\log^{700+700b}(T_{0,3})}{T_{0.3}} < \frac{1}{2}$$
From now on, we will work under the additional restriction that 
$$T_{0} \geq 2e^{e^{1000(b+1)}} +T_{0,1}+T_{0,2}+T_{0,3}$$
(Recall that \eqref{T0initialconstraint} was, up to now, the only constraint on $T_{0}$). From the discussion preceeding \eqref{lambdasolution}, and our estimates on $RHS(e,t)$ above, for $e \in \overline{B_{1}}(0) \subset X$, the map
$$x \mapsto \frac{RHS(e,x)}{4\alpha}-\int_{x}^{\infty} \frac{RHS(e,z)}{4\alpha}r(-x,-z)dz, \text{ a.e. } x \in [T_{0},\infty)$$
a priori defined only almost everywhere, admits a continuous extension to $[T_{0},\infty)$. Then, we define $T$ on $\overline{B_{1}}(0) \subset X$ by
$$T(f)(t) = \int_{t}^{\infty} dq_{1} \int_{q_{1}}^{\infty} dq_{2} S_{f}(q_{2})$$
where we define $S_{f}$ by (the above mentioned continuous extension of)
$$S_{f}(x) = \frac{RHS(f,x)}{4\alpha}-\int_{x}^{\infty} \frac{RHS(f,z)}{4\alpha}r(-x,-z)dz$$
By definition of $T$, we have 
$$T(e)''(t) = S_{e}(t) = \frac{RHS(e,t)}{4\alpha} - \int_{t}^{\infty} \frac{RHS(e,z)}{4\alpha} r(-t,-z) dz$$
Using \eqref{rintest}, we get 
\begin{equation}\begin{split} |T(e)''(t)| &\leq \frac{3 D_{2}}{\alpha \log(\log(t))t^{2}\log^{b+1}(t)}\\
&\leq \frac{1}{100 \sqrt{\log(\log(t))} t^{2}\log^{b+1}(t)}, \quad e \in \overline{B}_{1}(0) \subset X\end{split}\end{equation}
Next, we have
\begin{equation}\begin{split} &T(e)'(t) = -\int_{t}^{\infty} T(e)''(x) dx\\
&|T(e)'(t)| \leq \frac{1}{100} \int_{t}^{\infty} \frac{dx}{\sqrt{\log(\log(x))} x^{2}\log^{b+1}(x)}\end{split}\end{equation}
Then, we note that 
\begin{equation}\begin{split} \int_{t}^{\infty} \frac{dx}{\sqrt{\log(\log(x))} x^{2}\log^{b+1}(x)} &= \frac{1}{t \log^{b+1}(t) \sqrt{\log(\log(t))}} -(b+1) \int_{t}^{\infty} \frac{dx}{x^{2}\log^{b+2}(x) \sqrt{\log(\log(x))}}\\
&-\frac{1}{2} \int_{t}^{\infty} \frac{dx}{x^{2}} \frac{1}{\log^{b+2}(x) (\log(\log(x)))^{3/2}}\\
&= \frac{1}{t \log^{b+1}(t) \sqrt{\log(\log(t))}} +E_{int,1}\end{split}\end{equation}
where
\begin{equation}\begin{split} |E_{int,1}| &\leq \frac{(b+1)}{t \log^{b+2}(t) \sqrt{\log(\log(t))}} + \frac{1}{2 t \log^{b+2}(t) (\log(\log(t)))^{3/2}} \\
&\leq \frac{1}{100} \frac{1}{t \log^{b+1}(t) \sqrt{\log(\log(t))}}, \quad t \geq T_{0}\end{split}\end{equation}
which gives
\begin{equation} |T(e)'(t)| \leq \frac{1}{100} \left(1+\frac{1}{100}\right) \left(\frac{1}{t \log^{b+1}(t) \sqrt{\log(\log(t))}}\right), \quad t \geq T_{0}\end{equation}
Similarly,
\begin{equation} T(e)(t) = \int_{t}^{\infty} dx_{1}\int_{x_{1}}^{\infty} dx_{2} T(f)''(x_{2})\end{equation}
Then we use
\begin{equation} \begin{split} &\int_{t}^{\infty} dx_{1}\int_{x_{1}}^{\infty} \frac{dx_{2}}{\sqrt{\log(\log(x_{2}))} x_{2}^{2}\log^{b+1}(x_{2})}=\frac{1}{b \log^{b}(t) \sqrt{\log(\log(t))}} + E_{int,2}\end{split}\end{equation}
where
\begin{equation}\begin{split} |E_{int,2}| &\leq \frac{1}{2 b^{2} (\log(\log(t)))^{3/2} \log^{b}(t)} + \frac{1}{\sqrt{\log(\log(t))} \log^{b+1}(t)}+\frac{1}{2(b+1) (\log(\log(t)))^{3/2} \log^{b+1}(t)}\\
&\leq \frac{1}{100} \frac{1}{b \log^{b}(t) \sqrt{\log(\log(t))}}, \quad t \geq T_{0}\end{split}\end{equation}
to get
\begin{equation} |T(e)(t)| \leq \frac{1}{100} \left(1+\frac{1}{100}\right) \frac{1}{b \log^{b}(t) \sqrt{\log(\log(t))}}, \quad t \geq T_{0}\end{equation}
Then, we conclude that $T(e) \in \overline{B_{1}(0)}$ for $e \in \overline{B_{1}(0)}$.\\
\\
Now, we will show that $T$ is a contraction on the space $\overline{B_{1}(0)} \subset X$.  We start by estimating, for $e_{1},e_{2} \in \overline{B_{1}(0)} \subset X$, the expression
\begin{equation}\label{Flip}RHS(e_{1}(t),t)-RHS(e_{2}(t),t)\end{equation}
In particular, we will prove the following proposition. 
\begin{proposition}\label{rhslipprop} There exists $C_{lip}>0$ \emph{independent} of $T_{0}$, such that for $e_{1},e_{2} \in \overline{B_{1}(0)} \subset X$,
\begin{equation} |RHS(e_{1},t)-RHS(e_{2},t)| \leq \frac{C_{lip} ||e_{1}-e_{2}||_{X}}{t^{2}\log^{b+1}(t) (\log(\log(t)))^{3/2}}\end{equation}\end{proposition}
To prove the proposition, it will be useful to define the functions
$$\lambda_{i}(t) = \lambda_{0}(t) + e_{i}(t), \quad i =1,2$$

We start with the terms of $RHS$ which don't involve $G$:
\begin{equation} \begin{split} |\frac{-4 \alpha \lambda_{0}''(t) \left(\log(\lambda_{1}(t))-\log(\lambda_{2}(t))\right)}{\log(\lambda_{0}(t))}| &\leq \frac{C}{t^{2}\log^{b+1}(t) \log(\log(t))} \frac{|\lambda_{1}(t)-\lambda_{2}(t)|}{\lambda_{0,0}(t)}\\
&\leq \frac{C ||e_{1}-e_{2}||_{X}}{t^{2} \log^{b+1}(t) (\log(\log(t)))^{3/2}}\end{split}\end{equation}

\begin{equation}\begin{split} &|\frac{-4 \alpha}{\log(\lambda_{0}(t))} \left(e_{1}''(t) \left(\log(\lambda_{1}(t))-\log(\lambda_{0}(t))\right)-e_{2}''(t) \left(\log(\lambda_{2}(t))-\log(\lambda_{0}(t))\right)\right)|\\
&\leq \frac{C}{\log(\log(t))} \left(|e_{1}''(t)-e_{2}''(t)| \log(\frac{\lambda_{0}(t)+e_{1}(t)}{\lambda_{0}(t)}) + |e_{2}''(t)| \left(\log(\lambda_{1}(t))-\log(\lambda_{2}(t))\right)\right)\\
&\leq C \frac{||e_{1}-e_{2}||_{X}}{t^{2}\log^{b+1}(t) (\log(\log(t)))^{2}}\end{split}\end{equation}

\begin{equation}\begin{split} &|\frac{4}{\log(\lambda_{0}(t))} \left(\int_{t}^{\infty} \frac{(e_{1}''(s)-e_{2}''(s))}{(1+s-t)^{3}}\left(\frac{1}{\lambda_{1}(t)^{1-\alpha}+s-t} - \frac{1}{\lambda_{0}(t)^{1-\alpha}+s-t}\right)ds\right.\\
&\left.+\int_{t}^{\infty} \frac{e_{2}''(s)}{(1+s-t)^{3}} \left(\frac{1}{\lambda_{1}(t)^{1-\alpha}+s-t}-\frac{1}{\lambda_{2}(t)^{1-\alpha}+s-t}\right) ds\right)|\\
&\leq  C \frac{||e_{1}-e_{2}||_{X}}{(\log(\log(t)))^{2} t^{2}\log^{b+1}(t)}\end{split}\end{equation} 

\begin{equation} \begin{split} \frac{1}{\log(\log(t))} \int_{t}^{\infty} \frac{ds}{s^{2}\log^{b+1}(s)(1+s-t)^{3}} |\frac{1}{\lambda_{1}(t)^{1-\alpha}+s-t}-\frac{1}{\lambda_{2}(t)^{1-\alpha}+s-t}| \leq \frac{C ||e_{1}-e_{2}||_{X}}{(\log(\log(t)))^{3/2} t^{2} \log^{b+1}(t)}\end{split}\end{equation}

\begin{equation} \begin{split}  \int_{t}^{\infty} |e_{1}''(s)-e_{2}''(s)| |\frac{1}{\log(\lambda_{0}(t))}-\frac{1}{\log(\lambda_{0}(s))}| \frac{ds}{1+s-t} \leq \frac{C ||e_{1}-e_{2}||_{X}}{t^{2} \log^{b+2}(t) (\log(\log(t)))^{5/2}}\end{split}\end{equation}

\begin{equation} \begin{split} &\int_{t}^{\infty} |e_{1}''(s)-e_{2}''(s)| |\frac{1}{\log(\lambda_{0}(t))}-\frac{1}{\log(\lambda_{0}(s))}| \frac{ds}{(\lambda_{0}(t)^{1-\alpha} +s-t)(1+s-t)^{3}}\\
&\leq \frac{C ||e_{1}-e_{2}||_{X}}{t^{3}\log^{b+2}(t) (\log(\log(t)))^{5/2}}\end{split}\end{equation}

We now proceed to study each term appearing in the expression
$$G(t,\lambda_{1}(t))-G(t,\lambda_{2}(t))$$

By writing
\begin{equation}\begin{split} &\lambda_{1}(t) E_{0,1}(\lambda_{1}(t),\lambda_{1}'(t),\lambda_{1}''(t)) - \lambda_{2}(t) E_{0,1}(\lambda_{2}(t),\lambda_{2}'(t),\lambda_{2}''(t))\\
&=\int_{0}^{1} DF_{0,0,1}(\lambda_{\sigma}) \cdot (\lambda_{1}(t)-\lambda_{2}(t),\lambda_{1}'(t)-\lambda_{2}'(t),\lambda_{1}''(t)-\lambda_{2}''(t)) d\sigma\end{split}\end{equation}
where
$$F_{0,0,1}(x,y,z) = x E_{0,1}(x,y,z), \quad \lambda_{\sigma} = \sigma (\lambda_{1}(t),\lambda_{1}'(t),\lambda_{1}''(t)) + (1-\sigma) (\lambda_{2}(t),\lambda_{2}'(t),\lambda_{2}''(t))$$
we get
\begin{equation} \begin{split} &\lambda_{1}(t) E_{0,1}(\lambda_{1}(t),\lambda_{1}'(t),\lambda_{1}''(t)) - \lambda_{2}(t) E_{0,1}(\lambda_{2}(t),\lambda_{2}'(t),\lambda_{2}''(t))\\
&\leq \frac{C ||e_{1}-e_{2}||_{X}}{t^{2}\log^{b+1}(t) \sqrt{\log(\log(t))}}\end{split}\end{equation}

Next, we consider 
\begin{equation} \begin{split} &K_{3}(w,\lambda_{1}(t))-K_{3,0}(w,\lambda_{1}(t)) - \left(K_{3}(w,\lambda_{2}(t))-K_{3,0}(w,\lambda_{2}(t))\right)\\
&= \frac{w^{5}}{4(1+w^{2})}\left(\frac{1}{(w^{2}+36\lambda_{1}(t)^{2})^{2}}-\frac{1}{(w^{2}+36\lambda_{2}(t)^{2})^{2}}\right)\\
&-\frac{1}{4}\left(\frac{w^{5}}{(\lambda_{1}(t)^{2-2\alpha}+w^{2})(w^{2}+36\lambda_{1}(t)^{2})^{2}}-\frac{w^{5}}{(\lambda_{2}(t)^{2-2\alpha}+w^{2})(w^{2}+36\lambda_{1}(t)^{2})^{2}}\right.\\
&+\frac{w^{5}}{(\lambda_{2}(t)^{2-2\alpha}+w^{2})(w^{2}+36\lambda_{1}(t)^{2})^{2}} - \frac{w^{5}}{(\lambda_{2}(t)^{2-2\alpha}+w^{2})(w^{2}+36\lambda_{2}(t)^{2})^{2}}\\
&\left.-\frac{1}{(\lambda_{1}(t)^{1-\alpha}+w)(1+w)^{3}}+\frac{1}{(\lambda_{2}(t)^{1-\alpha}+w)(1+w)^{3}}\right)\end{split}\end{equation}

So,
\begin{equation} \begin{split} \int_{t}^{\infty}  &|K_{3}(s-t,\lambda_{1}(t))-K_{3,0}(s-t,\lambda_{1}(t)) - \left(K_{3}(s-t,\lambda_{2}(t))-K_{3,0}(s-t,\lambda_{2}(t))\right)| ds\\
&\leq \frac{C ||e_{1}-e_{2}||_{X}}{\log^{b}(t)\sqrt{\log(\log(t))}} + C \frac{||e_{1}-e_{2}||_{X}}{\sqrt{\log(\log(t))}}+\frac{C ||e_{1}-e_{2}||_{X}}{\log^{b\alpha}(t)\sqrt{\log(\log(t))}} + \frac{C ||e_{1}-e_{2}||_{X}}{\sqrt{\log(\log(t))}}\\
&\leq \frac{C||e_{1}-e_{2}||_{X}}{\sqrt{\log(\log(t))}}\end{split}\end{equation}

From this, we get
\begin{equation} \begin{split} &|\int_{t}^{\infty} ds \lambda_{1}''(s) \left(K_{3}(s-t,\lambda_{1}(t))-K_{3,0}(s-t,\lambda_{1}(t))\right) - \int_{t}^{\infty} ds \lambda_{2}''(s) \left(K_{3}(s-t,\lambda_{2}(t))-K_{3,0}(s-t,\lambda_{2}(t))\right)|\\
&\leq C \int_{t}^{\infty} |e_{1}''(s)-e_{2}''(s)| |K_{3}(s-t,\lambda_{1}(t))-K_{3,0}(s-t,\lambda_{1}(t))| ds\\
&+ C \int_{t}^{\infty} |\lambda_{2}''(s)| |K_{3}(s-t,\lambda_{1}(t))-K_{3,0}(s-t,\lambda_{1}(t)) - \left(K_{3}(s-t,\lambda_{2}(t))-K_{3,0}(s-t,\lambda_{2}(t))\right)|\\
&\leq C \frac{||e_{1}-e_{2}||_{X}}{t^{2}\log^{b+1}(t) \sqrt{\log(\log(t))}} \int_{t}^{\infty} |K_{3}(s-t,\lambda_{1}(t))-K_{3,0}(s-t,\lambda_{1}(t))| ds\\
&+\frac{C}{t^{2}\log^{b+1}(t)} \int_{t}^{\infty} |K_{3}(s-t,\lambda_{1}(t))-K_{3,0}(s-t,\lambda_{1}(t)) - \left(K_{3}(s-t,\lambda_{2}(t))-K_{3,0}(s-t,\lambda_{2}(t))\right)|\\
&\leq C \frac{||e_{1}-e_{2}||_{X}}{t^{2}\log^{b+1}(t) \sqrt{\log(\log(t))}}+\frac{C||e_{1}-e_{2}||_{X}}{t^{2}\log^{b+1}(t)\sqrt{\log(\log(t))}}\\
&\leq \frac{C||e_{1}-e_{2}||_{X}}{t^{2}\log^{b+1}(t)\sqrt{\log(\log(t))}}\end{split}\end{equation}
(where we recall that $\int_{t}^{\infty} |K_{3}(s-t,\lambda_{1}(t))-K_{3,0}(s-t,\lambda_{1}(t))| ds$ was previously estimated).

Next, let us consider the following term which appears in the expression \\
$G(t,\lambda_{0}(t)+e_{1}(t))-G(t,\lambda_{0}(t)+e_{2}(t))$:
\begin{equation}\begin{split}&\frac{16}{(\lambda_{0}(t)+e_{1}(t))^{2}}\int_{t}^{\infty} dx (\lambda_{0}''(x)+e_{1}''(x))\left(K_{1}(x-t,\lambda_{0}(t)+e_{1}(t))-\frac{(\lambda_{0}(t)+e_{1}(t))^{2}}{4(1+x-t)}\right)\\
&-\frac{16}{(\lambda_{0}(t)+e_{2}(t))^{2}}\int_{t}^{\infty} dx (\lambda_{0}''(x)+e_{2}''(x))\left(K_{1}(x-t,\lambda_{0}(t)+e_{2}(t))-\frac{(\lambda_{0}(t)+e_{2}(t))^{2}}{4(1+x-t)}\right)
\end{split}\end{equation}
which we re-write as
\begin{equation}\label{g1}\begin{split} &\left(\frac{16}{(\lambda_{0}(t)+e_{1}(t))^{2}}-\frac{16}{(\lambda_{0}(t)+e_{2}(t))^{2}}\right)\int_{t}^{\infty} dx (\lambda_{0}''(x)+e_{1}''(x))\left(K_{1}(x-t,\lambda_{0}(t)+e_{1}(t))-\frac{(\lambda_{0}(t)+e_{1}(t))^{2}}{4(1+x-t)}\right)\\&+\frac{16}{(\lambda_{0}(t)+e_{2}(t))^{2}} \int_{t}^{\infty} dx (e_{1}''(x)-e_{2}''(x))\left(K_{1}(x-t,\lambda_{0}(t)+e_{1}(t))-\frac{(\lambda_{0}(t)+e_{1}(t))^{2}}{4(1+x-t)}\right)\\
&+\frac{16}{(\lambda_{0}(t)+e_{2}(t))^{2}} \int_{t}^{\infty} dx (\lambda_{0}''(x)+e_{2}''(x))\left(L_{1}(x-t,\lambda_{0}(t)+e_{1}(t))-L_{1}(x-t,\lambda_{0}(t)+e_{2}(t))\right)
\end{split}\end{equation}
where
$$L_{1}(w,z) = K_{1}(w,z)-\frac{z^{2}}{4(1+w)}$$
From \eqref{k1ptwseest} and \eqref{k1diffptwse} there exists $C>0$, independent of $z$ such that 
$$|K_{1}(w,z)-\frac{z^{2}}{4(1+w)}| \leq \frac{C z^{2}}{1+w^{2}}, \quad |z| \leq \frac{1}{2}$$
(Recall that, by the choice of $T_{0}$, $|\lambda_{0}(t)+f_{i}(t)| \leq \frac{1}{2}, \quad i=1,2$).
Finally, let us note that
\begin{equation} |\partial_{z} \left(K_{1}(w,z)-\frac{z^{2}}{4(1+w)}\right)| \leq  \frac{C z}{1+w^{2}}, \quad w \geq 0\end{equation}

Using the facts that 
$$|\frac{16}{(\lambda_{0}(t)+e_{1}(t))^{2}}-\frac{16}{(\lambda_{0}(t)+e_{2}(t))^{2}}| \leq C \frac{\log^{3b}(t) ||e_{1}-e_{2}||_{X}}{\sqrt{\log(\log(t))} \log^{b}(t)}$$
and
$$|L_{1}(w,z_{1})-L_{1}(w,z_{2})| \leq ||\partial_{2}L_{1}(w,sz_{1}+(1-s)z_{2})||_{L^{\infty}_{s}([0,1])} |z_{1}-z_{2}|$$
we get that the absolute value of \eqref{g1} is bounded above by
$$\frac{C ||e_{1}-e_{2}||_{X}}{t^{2}\log^{b+1}(t) \sqrt{\log(\log(t))}}$$

Another term arising in\\
$G(t,\lambda_{0}(t)+e_{1}(t))-G(t,\lambda_{0}(t)+e_{2}(t))$ is
\begin{equation}\begin{split} &\frac{16}{(\lambda_{0}(t)+e_{1}(t))^{2}} \int_{t}^{\infty} dx (\lambda_{0}''(x)+e_{1}''(x)) K(x-t,\lambda_{0}(t)+e_{1}(t))\\
&-\frac{16}{(\lambda_{0}(t)+e_{2}(t))^{2}} \int_{t}^{\infty} dx (\lambda_{0}''(x)+e_{2}''(x)) K(x-t,\lambda_{0}(t)+e_{2}(t))
\end{split}\end{equation}
This can be split analogously to the previous term:\\
\\
\begin{equation}\label{g2} \begin{split} &\left(\frac{16}{(\lambda_{0}(t)+e_{1}(t))^{2}}-\frac{16}{(\lambda_{0}(t)+e_{2}(t))^{2}}\right)\int_{t}^{\infty} dx (\lambda_{0}''(x)+e_{1}''(x)) K(x-t,\lambda_{0}(t)+e_{1}(t))\\
&+\frac{16}{(\lambda_{0}(t)+e_{2}(t))^{2}} \int_{t}^{\infty} dx (e_{1}''(x)-e_{2}''(x)) K(x-t,\lambda_{0}(t)+e_{1}(t))\\
&+\frac{16}{(\lambda_{0}(t)+e_{2}(t))^{2}} \int_{t}^{\infty} dx (\lambda_{0}''(x)+e_{2}''(x))\left(K(x-t,\lambda_{0}+e_{1}(t))-K(x-t,\lambda_{0}(t)+e_{2}(t))\right)\end{split}\end{equation}
Note that
\begin{equation} K(x-t,z_{1})-K(x-t,z_{2}) = (z_{1}-z_{2})\int_{0}^{1} dq \partial_{2}K(x-t,z_{2}+q(z_{1}-z_{2}))\end{equation}
So, 
\begin{equation}\label{kiterationest}\begin{split} \int_{t}^{\infty} dx |K(x-t,z_{1})-K(x-t,z_{2})| &\leq |z_{1}-z_{2}| \int_{0}^{1} dq \int_{0}^{\infty} dw |\partial_{2}K(w,z_{2}+(z_{1}-z_{2})q)|\end{split}\end{equation}
But, using the formula for $K$ from the previous sections, we see that $\partial_{z}K(w,z)$ has a sign:
\begin{equation} \partial_{z}K(w,z) = \int_{0}^{\infty} dR \int_{0}^{w} d\rho \left(\frac{4 p R^3 z  \left(p^2+R^2 z^2+1\right)}{\left(R^2+1\right)^3  \left(\left(p^2-R^2 z^2+1\right)^2+4 R^2 z^2\right)^{3/2}}\right)\left(\frac{1}{\sqrt{w^{2}-\rho^{2}}}-\frac{1}{w}\right) \geq 0\end{equation}
So, to control the integral in \eqref{kiterationest}, we can note the following: If $$z_{q} = z_{2}+(z_{1}-z_{2})q$$
then
\begin{equation}\label{d2kintest}\begin{split} &\int_{0}^{1} dq \int_{0}^{\infty} dw |\partial_{2}K(w,z_{q})|\\
&= \int_{0}^{1} dq \int_{0}^{\infty} dw \int_{0}^{\infty} dR \int_{0}^{w} d\rho \left(\frac{4 p R^3 z_{q}  \left(p^2+R^2 z_{q}^2+1\right)}{\left(R^2+1\right)^3  \left(\left(p^2-R^2 z_{q}^2+1\right)^2+4 R^2 z_{q}^2\right)^{3/2}}\right)\left(\frac{1}{\sqrt{w^{2}-\rho^{2}}}-\frac{1}{w}\right) \\
&=\int_{0}^{1} dq \int_{0}^{\infty} dR \int_{0}^{\infty} d\rho \int_{\rho}^{\infty} dw \left(\frac{4 p R^3 z_{q}  \left(p^2+R^2 z_{q}^2+1\right)}{\left(R^2+1\right)^3  \left(\left(p^2-R^2 z_{q}^2+1\right)^2+4 R^2 z_{q}^2\right)^{3/2}}\right)\left(\frac{1}{\sqrt{w^{2}-\rho^{2}}}-\frac{1}{w}\right)\\
&=\int_{0}^{1} dq \int_{0}^{\infty} dR \int_{0}^{\infty} d\rho\frac{4 p R^3 z_{q} \log (2) \left(p^2+R^2 z_{q}^2+1\right)}{\left(R^2+1\right)^3 \left(\left(p^2-R^2 z_{q}^2+1\right)^2+4 R^2 z_{q}^2\right)^{3/2}}\\
&=\int_{0}^{1} dq \int_{0}^{\infty} dR\frac{R^3 z_{q} \log (4)}{\left(R^2+1\right)^3}\\
&=\int_{0}^{1} dq \frac{1}{2} z_{q} \log (2)\\
&=\frac{1}{4} \log (2) (z_{1}+z_{2})
\end{split}\end{equation}
Using \eqref{kintegralestimate}, we see that the absolute value of \eqref{g2} is bounded above by
$$\frac{C ||e_{1}-e_{2}||_{X}}{t^{2}\log^{b+1}(t) \sqrt{\log(\log(t))}}$$

Next, we will consider the terms involving $E_{v_{2},ip}$, starting with the case $b \neq 1$. We recall
\begin{equation}\label{ev2ip2}\begin{split} \lambda(t) E_{v_{2},ip}(t,\lambda(t)) &= 2 c_{b}\lambda(t) \int_{0}^{\infty} d\xi \frac{\sin(t\xi)}{t^{2}} \psi_{v_{2}}(\xi,\lambda(t)) + 2 c_{b}\lambda(t) \int_{0}^{\infty} d\xi \chi_{\leq \frac{1}{4}}(\xi) \frac{\sin(t\xi)}{t^{2}} F_{v_{2}}(\xi,\lambda(t))\\
&+2 c_{b} \int_{0}^{\frac{1}{2}} d\xi \left(\chi_{\leq \frac{1}{4}}(\xi)-1\right) \frac{\sin(t\xi)}{t^{2}}\left(\frac{(b-1)}{\xi\log^{b}(\frac{1}{\xi})} + \frac{b(b-1)}{\xi \log^{b+1}(\frac{1}{\xi})}\right)\\
&+2 c_{b}\left(\int_{0}^{\frac{t}{2}} \frac{du \sin(u)(b-1)}{t^{2}u\log^{b}(\frac{t}{u})} - \frac{(b-1)\pi}{2t^{2}\log^{b}(t)} + \int_{0}^{\frac{t}{2}} \frac{du \sin(u) b(b-1)}{t^{2}u\log^{b+1}(\frac{t}{u})}\right)\end{split}\end{equation}
where
$F_{v_{2}}$ and $\psi_{v_{2}}$ were defined in \eqref{Fdef2} and \eqref{psidef2}, respectively. Notice that the last two lines of \eqref{ev2ip2} are \emph{independent} of $\lambda(t)$. So,
\begin{equation}\label{ev2ipdiff}\begin{split} &\lambda_{1}(t)E_{v_{2},ip}(t,\lambda_{1}(t))-\lambda_{2}(t)E_{v_{2},ip}(t,\lambda_{2}(t))\\
&=2 c_{b}\lambda_{1}(t) \int_{0}^{\infty} d\xi \frac{\sin(t\xi)}{t^{2}} \psi_{v_{2}}(\xi,\lambda_{1}(t))-2 c_{b}\lambda_{2}(t) \int_{0}^{\infty} d\xi \frac{\sin(t\xi)}{t^{2}} \psi_{v_{2}}(\xi,\lambda_{2}(t))\\
&+2 c_{b}\lambda_{1}(t) \int_{0}^{\infty} d\xi \chi_{\leq \frac{1}{4}}(\xi) \frac{\sin(t\xi)}{t^{2}} F_{v_{2}}(\xi,\lambda_{1}(t))-2 c_{b}\lambda_{2}(t) \int_{0}^{\infty} d\xi \chi_{\leq \frac{1}{4}}(\xi) \frac{\sin(t\xi)}{t^{2}}F_{v_{2}}(\xi,\lambda_{2}(t))\end{split}\end{equation}
First, we consider the second line of \eqref{ev2ipdiff}:
\begin{equation}\label{ev2ipintdiff}\begin{split} &|2 c_{b}\lambda_{1}(t) \int_{0}^{\infty} d\xi \frac{\sin(t\xi)}{t^{2}} \psi_{v_{2}}(\xi,\lambda_{1}(t))-2 c_{b}\lambda_{2}(t) \int_{0}^{\infty} d\xi \frac{\sin(t\xi)}{t^{2}} \psi_{v_{2}}(\xi,\lambda_{2}(t))|\\
&\leq |2c_{b}(e_{1}(t)-e_{2}(t)) \int_{0}^{\infty} d\xi \frac{\sin(t\xi)}{t^{2}}\psi_{v_{2}}(\xi,\lambda_{1}(t))|\\
&+|2 c_{b}\lambda_{2}(t) \int_{0}^{\infty} d\xi \frac{\sin(t\xi)}{t^{2}}\left(\psi_{v_{2}}(\xi,\lambda_{1}(t))-\psi_{v_{2}}(\xi,\lambda_{2}(t))\right)|\\
\end{split}\end{equation}
The second line of \eqref{ev2ipintdiff} is estimated as follows:
\begin{equation} \begin{split} &|2c_{b}(e_{1}(t)-e_{2}(t)) \int_{0}^{\infty} d\xi \frac{\sin(t\xi)}{t^{2}}\psi_{v_{2}}(\xi,\lambda_{1}(t))|\\
&\leq C |e_{1}(t)-e_{2}(t)| |\int_{0}^{\infty} d\xi \frac{\sin(t\xi)}{t^{2}}\psi_{v_{2}}(\xi,\lambda_{1}(t))|\\
&\leq C \frac{|e_{1}(t)-e_{2}(t)|}{t^{3}\lambda_{1}(t)}\end{split}\end{equation}
where we use the same procedure as used in the previous subsections to estimate the integral. For the third line of \eqref{ev2ipintdiff}, we first note that
\begin{equation} \psi_{v_{2}}(\xi,\lambda(t)) = 2 \chi_{\leq \frac{1}{4}}'(\xi) \partial_{\xi} \left(\frac{\xi^{2}K_{1}(\xi \lambda(t))}{\log^{b-1}(\frac{1}{\xi})}\right) + \frac{\chi_{\leq \frac{1}{4}}''(\xi) \xi^{2}K_{1}(\xi\lambda(t))}{\log^{b+1}(\frac{1}{\xi})}\end{equation}
\begin{equation}\begin{split} \partial_{\xi} \psi_{v_{2}}(\xi,\lambda(t))&=3 \chi_{\leq \frac{1}{4}}''(\xi) \partial_{\xi}\left(\frac{\xi^{2}K_{1}(\xi\lambda(t))}{\log^{b-1}(\frac{1}{\xi})}\right)+2 \chi_{\leq \frac{1}{4}}'(\xi) \partial_{\xi}^{2}\left(\frac{\xi^{2}K_{1}(\xi\lambda(t))}{\log^{b-1}(\frac{1}{\xi})}\right)\\
&+\chi_{\leq \frac{1}{4}}'''(\xi) \frac{\xi^{2}K_{1}(\xi\lambda(t))}{\log^{b-1}(\frac{1}{\xi})}\end{split}\end{equation}
Then, we use 
\begin{equation}\begin{split} K_{1}'(y) &= -\left(\frac{y K_{0}(y)+K_{1}(y)}{y}\right)\\
K_{1}''(y) &= K_{1}(y)+\frac{K_{2}(y)}{y}\\
K_{1}'''(y)&=\frac{-(3+y^{2})}{y^{2}} K_{2}(y)\end{split}\end{equation}

and the estimates, for $0 <x<\frac{1}{2}$ and $\frac{1}{8} \leq \xi \leq \frac{1}{4}$,
\begin{equation}\label{kjestimates} |K_{1}(\xi x)| \leq \frac{C}{\xi x}, \quad |K_{0}(\xi x)| \leq C |\log(\xi x)|, \quad |K_{2}(\xi x)| \leq \frac{C}{\xi^{2}x^{2}}\end{equation}
Since $\text{supp}\left(\chi^{(j)}_{\leq \frac{1}{4}}\right) \subset [\frac{1}{8},\frac{1}{4}]$, we have
\begin{equation}\label{d12psiest}\begin{split} |\partial_{12}\psi_{v_{2}}(\xi,y)| &\leq \frac{C \mathbbm{1}_{[\frac{1}{8},\frac{1}{4}]}(\xi)}{\xi y^{2} \log^{s-1}(\frac{1}{\xi})}, \quad 0<y<\frac{1}{2}\end{split}\end{equation}
Note that 
\begin{equation}\begin{split} &\int_{0}^{\infty} d\xi \frac{\sin(t\xi)}{t^{2}}\left(\psi_{v_{2}}(\xi,\lambda_{1}(t))-\psi_{v_{2}}(\xi,\lambda_{2}(t))\right)=\int_{0}^{\infty} d\xi \frac{\cos(t\xi)}{t^{3}}\partial_{\xi}\left(\psi_{v_{2}}(\xi,\lambda_{1}(t))-\psi_{v_{2}}(\xi,\lambda_{2}(t))\right)\end{split}\end{equation}
and
\begin{equation}\begin{split} &|\partial_{\xi}\left(\psi_{v_{2}}(\xi,\lambda_{1}(t))-\psi_{v_{2}}(\xi,\lambda_{2}(t))\right)|\leq \sup_{x \in [\frac{\lambda_{0}(t)}{2},\frac{1}{2}]} |\partial_{12}\psi_{v_{2}}(\xi,x)| |\lambda_{1}(t)-\lambda_{2}(t)|\end{split}\end{equation}
where we used the fact that, for $e_{i} \in \overline{B_{1}(0)} \subset X, \quad i=1,2$, we have
\begin{equation} \frac{1}{2} \lambda_{0}(t)<\lambda_{i}(t)<2\lambda_{0}(t)<\frac{1}{2}, \quad i=1,2\end{equation}
Now, we use \eqref{d12psiest} to get
\begin{equation} \begin{split} &|\int_{0}^{\infty} d\xi \frac{\cos(t\xi)}{t^{3}}\partial_{\xi}\left(\psi_{v_{2}}(\xi,\lambda_{1}(t))-\psi_{v_{2}}(\xi,\lambda_{2}(t))\right)|\\
&\leq \frac{C}{t^{3}} \int_{0}^{\infty} \sup_{x \in [\frac{\lambda_{0}(t)}{2},\frac{1}{2}]} |\partial_{12}\psi_{v_{2}}(\xi,x)| |\lambda_{1}(t)-\lambda_{2}(t)|\\
&\leq C \frac{|e_{1}(t)-e_{2}(t)|}{t^{3}} \int_{0}^{\infty}d\xi\frac{\mathbbm{1}_{[\frac{1}{8},\frac{1}{4}]}(\xi)}{\lambda_{0}(t)^{2} \xi \log^{s-1}(\frac{1}{\xi})} \\
&\leq C \frac{|e_{1}(t)-e_{2}(t)| \log^{2b}(t)}{t^{3}}\end{split}\end{equation}
We conclude that
\begin{equation}\begin{split}&|2c_{b}(\lambda_{0}(t)+e_{2}(t))\int_{0}^{\infty} d\xi \frac{\cos(t\xi)}{t^{3}}\left(\partial_{\xi}\psi_{v_{2}}(\xi,\lambda_{1}(t))-\partial_{\xi}\psi_{v_{2}}(\xi,\lambda_{2}(t))\right)|\\
&\leq \frac{C ||e_{1}-e_{2}||_{X}}{t^{3}\sqrt{\log(\log(t))}}\end{split}\end{equation}
We now need to consider the third line of \eqref{ev2ipdiff}. Here, we use a similar procedure as above:
\begin{equation}\label{fv2diff2}\begin{split}&|2 c_{b}\lambda_{1}(t)\int_{0}^{\infty} \chi_{\leq \frac{1}{4}}(\xi) \frac{\sin(t\xi)}{t^{2}} F_{v_{2}}(\xi,\lambda_{1}(t))-2c_{b}\lambda_{2}(t) \int_{0}^{\infty} d\xi \chi_{\leq \frac{1}{4}}(\xi) \frac{\sin(t\xi)}{t^{2}} F_{v_{2}}(\xi,\lambda_{2}(t))|\\
&\leq C |\left(e_{1}(t)-e_{2}(t)\right)\int_{0}^{\infty} d\xi \chi_{\leq \frac{1}{4}}(\xi) \frac{\sin(t\xi)}{t^{2}} F_{v_{2}}(\xi,\lambda_{1}(t))|\\
&+C \lambda_{2}(t) |\int_{0}^{\infty} d\xi \frac{\chi_{\leq \frac{1}{4}}(\xi) \sin(t\xi)}{t^{2}} \left(F_{v_{2}}(\xi,\lambda_{1}(t))-F_{v_{2}}(\xi,\lambda_{2}(t))\right)|\end{split}\end{equation}
For the second line of \eqref{fv2diff2}, we use the same procedure to estimate the integral as that used in previous subsections, to get
\begin{equation} \begin{split} &|\left(e_{1}(t)-e_{2}(t)\right)\int_{0}^{\infty} d\xi \chi_{\leq \frac{1}{4}}(\xi) \frac{\sin(t\xi)}{t^{2}} F_{v_{2}}(\xi,\lambda_{1}(t))|\leq C |e_{1}(t)-e_{2}(t)| \frac{\lambda_{1}(t) |\log(\lambda_{1}(t))|}{t^{3}}\end{split}\end{equation}
For the third line of \eqref{fv2diff2}, we first integrate by parts once to get
\begin{equation} \begin{split} &\int_{0}^{\infty} d\xi \frac{\chi_{\leq \frac{1}{4}}(\xi) \sin(t\xi)}{t^{2}} \left(F_{v_{2}}(\xi,\lambda_{1}(t))-F_{v_{2}}(\xi,\lambda_{2}(t))\right)\\
&=\int_{0}^{\infty} d\xi \frac{\cos(t\xi)}{t^{3}}\left(\chi_{\leq \frac{1}{4}}'(\xi) \left(F_{v_{2}}(\xi,\lambda_{1}(t))-F_{v_{2}}(\xi,\lambda_{2}(t))\right)+\chi_{\leq \frac{1}{4}}(\xi) \partial_{\xi}\left(F_{v_{2}}(\xi,\lambda_{1}(t))-F_{v_{2}}(\xi,\lambda_{2}(t))\right)\right)\end{split}\end{equation}
Now, we recall that
$$F_{v_{2}}(\xi,y) = \partial_{\xi}^{2}\left(\frac{\xi^{2}}{\log^{b-1}(\frac{1}{\xi})}\left(K_{1}(\xi y)-\frac{1}{\xi y}\right)\right)$$
Using the same estimates on $K_{j}$ as used to obtain \eqref{d12psiest}, we get
$$|\partial_{2}F_{v_{2}}(\xi,y)| \leq \frac{C\xi (|\log(\xi)|+|\log(y)|)}{\log^{b-1}(\frac{1}{\xi})}, \quad 0 < \xi < \frac{1}{4}, \quad 0<y<\frac{1}{2}$$
and
$$|\partial_{12}F_{v_{2}}(\xi,y)| \leq \frac{C(|\log(\xi)|+|\log(y)|)}{\log^{b-1}(\frac{1}{\xi})} , \quad 0 < y < \frac{1}{2}, \quad 0 < \xi < \frac{1}{4}$$
We then get 
$$|F_{v_{2}}(\xi,\lambda_{2}(t))-F_{v_{2}}(\xi,\lambda_{1}(t))| \leq \left(\sup_{x \in [\frac{\lambda_{0}(t)}{2},\frac{1}{2}]}|\partial_{2}F_{v_{2}}(\xi,x)|\right)|\lambda_{1}(t)-\lambda_{2}(t)|$$
and 
$$|\partial_{1}F_{v_{2}}(\xi,\lambda_{2}(t))-\partial_{1}F_{v_{2}}(\xi,\lambda_{1}(t))| \leq \left(\sup_{x \in [\frac{\lambda_{0}(t)}{2},\frac{1}{2}]}|\partial_{12}F_{v_{2}}(\xi,x)|\right)|\lambda_{1}(t)-\lambda_{2}(t)|$$
Using the above estimates, we obtain, 
\begin{equation}\begin{split} &|2c_{b}\lambda_{2}(t) \int_{0}^{\infty} d\xi \chi_{\leq \frac{1}{4}}(\xi) \frac{\sin(t\xi)}{t^{2}}\left(F_{v_{2}}(\xi,\lambda_{1}(t))-F_{v_{2}}(\xi,\lambda_{2}(t))\right)|\leq C \frac{|\lambda_{1}(t)-\lambda_{2}(t)| \lambda_{0}(t) |\log(\lambda_{0}(t))|}{t^{3}}\\
&\leq C \frac{||e_{1}-e_{2}||_{X} \sqrt{\log(\log(t))}}{t^{3}\log^{2b}(t)}\end{split}\end{equation}
The same procedure for the case $b=1$ yields the same estimates, since $\psi_{v_{2}}$ and $F_{v_{2}}$ have the same form for all $b$.\\
Combining the above estimates then gives, for all $b >0$:
\begin{equation}\begin{split}&|-\lambda_{1}(t)E_{v_{2},ip}(t,\lambda_{1}(t))+\lambda_{2}(t)E_{v_{2},ip}(t,\lambda_{2}(t))| \leq \frac{C ||e_{1}-e_{2}||_{X}}{t^{3}\sqrt{\log(\log(t))}} \end{split}\end{equation}

If 
$$G_{2}(t,\lambda(t)) = - \lambda(t) \langle \left(\frac{\cos(2Q_{\frac{1}{\lambda(t)}})-1}{r^{2}}\right)\left((v_{4}+v_{5})\left(1-\chi_{\geq 1}(\frac{4r}{t})\right)+E_{5}-\chi_{\geq 1}(\frac{2r}{\log^{N}(t)})\left(v_{1}+v_{2}+v_{3}\right)\right)\vert_{r=R\lambda(t)},\phi_{0}\rangle$$
Then, we need to estimate
$$G_{2}(t,\lambda_{1}(t))-G_{2}(t,\lambda_{2}(t))$$

From now on, the fact that $v_{k}, \quad k \neq 2$ depends on $\lambda(t)$ is important, so we will denote these functions by $v_{k}^{\lambda}$  to emphasize the dependence of these functions on $\lambda$.

We first note that 
\begin{equation}\begin{split}& \lambda(t) \langle \left(\frac{\cos(2Q_{\frac{1}{\lambda(t)}})-1}{r^{2}}\right)v_{4}^{\lambda}\left(1-\chi_{\geq 1}(\frac{4r}{t})\right)\vert_{r=R\lambda(t)}, \phi_{0}\rangle\\
&= \frac{1}{\lambda(t)} \int_{0}^{\infty} \left(\frac{\cos(2Q_{1}(\frac{r}{\lambda(t)}))-1}{r^{2}}\right) v_{4}^{\lambda}(t,r)\left(1-\chi_{\geq 1}(\frac{4r}{t})\right) \phi_{0}(\frac{r}{\lambda(t)}) r dr\end{split}\end{equation}

so,
\begin{equation}\begin{split} & \lambda_{1}(t) \langle \left(\frac{\cos(2Q_{\frac{1}{\lambda_{1}(t)}})-1}{r^{2}}\right)v_{4}^{\lambda_{1}}\left(1-\chi_{\geq 1}(\frac{4r}{t})\right)\vert_{r=R\lambda_{1}(t)}, \phi_{0}\rangle\\
&-\left( \lambda_{2}(t) \langle \left(\frac{\cos(2Q_{\frac{1}{\lambda_{2}(t)}})-1}{r^{2}}\right)v_{4}^{\lambda_{2}}\left(1-\chi_{\geq 1}(\frac{4r}{t})\right)\vert_{r=R\lambda_{2}(t)}, \phi_{0}\rangle\right)\\
&= -16 \int_{0}^{\infty} \left(\frac{\lambda_{1}(t)^{2}}{(r^{2}+\lambda_{1}(t)^{2})^{3}} - \frac{\lambda_{2}(t)^{2}}{(r^{2}+\lambda_{2}(t)^{2})^{3}}\right) v_{4}^{\lambda_{1}}(t,r)\left(1-\chi_{\geq 1}(\frac{4r}{t})\right) r^{2} dr\\
&+ \int_{0}^{\infty} \left(\frac{\cos(2Q_{1}(\frac{r}{\lambda_{2}(t)})-1}{r^{2} \lambda_{2}(t)}\right) \phi_{0}(\frac{r}{\lambda_{2}(t)}) \left(v_{4}^{\lambda_{1}}-v_{4}^{\lambda_{2}}\right)(t,r)\left(1-\chi_{\geq 1}(\frac{4r}{t})\right) r dr\end{split}\end{equation}
and we have the analogous formulae for all the other terms in $G_{2}$. Using the pointwise estimates for $v_{i}, E_{5}, \quad 1 \leq i \leq 5$, we get
\begin{equation} \begin{split} &|-16 \int_{0}^{\infty} \left(\frac{\lambda_{1}(t)^{2}}{(r^{2}+\lambda_{1}(t)^{2})^{3}} - \frac{\lambda_{2}(t)^{2}}{(r^{2}+\lambda_{2}(t)^{2})^{3}}\right) \left((v_{4}^{\lambda_{1}}+v_{5}^{\lambda_{1}})\left(1-\chi_{\geq 1}(\frac{4r}{t})\right)+E_{5}^{\lambda_{1}}-\chi_{\geq 1}(\frac{2r}{\log^{N}(t)})\left(v_{1}^{\lambda_{1}}+v_{2}+v_{3}^{\lambda_{1}}\right)\right) r^{2} dr|\\
&\leq \frac{C ||e_{1}-e_{2}||_{X}}{t^{2} \log^{3b+2N-1}(t) \sqrt{\log(\log(t))}} + \frac{C ||e_{1}-e_{2}||_{X}}{t^{7/2} \log^{3b-3+\frac{5N}{2}}(t) \sqrt{\log(\log(t))}} + \frac{C ||e_{1}-e_{2}||_{X}}{t^{2} \log^{b+1}(t) \sqrt{\log(\log(t))}}\\
& + \frac{C ||e_{1}-e_{2}||_{X}}{t^{2}\log^{3b+2N}(t) \sqrt{\log(\log(t))}}\\
&\leq \frac{C ||e_{1}-e_{2}||_{X}}{t^{2} \log^{b+1}(t) \sqrt{\log(\log(t))}}\end{split}\end{equation}

In order to estimate the rest of the terms arising in $G_{2}(t,\lambda_{1})-G_{2}(t,\lambda_{2})$, we must first prove estimates on $v_{3,1,2}$ defined by 
$$v_{3,1,2}:=v_{3}^{\lambda_{1}}-v_{3}^{\lambda_{2}}$$
\begin{lemma} \begin{equation}\label{v3lip} |v_{3,1,2}(t,r)| \leq \begin{cases} \frac{C r ||e_{1}-e_{2}||_{X}}{t^{2} \sqrt{\log(\log(t))} \log^{b}(t)}, \quad r \leq \frac{t}{2}\\
\frac{C ||e_{1}-e_{2}||_{X}}{r \log^{b}(t) \sqrt{\log(\log(t))}}, \quad r > \frac{t}{2}\end{cases}\end{equation}\end{lemma}
\begin{proof}
Note that $v_{3,1,2}$ solves the following equation with $0$ Cauchy data at infinity:
\begin{equation} -\partial_{tt}v_{3,1,2}+\partial_{rr}v_{3,1,2}+\frac{1}{r} \partial_{r} v_{3,1,2} - \frac{v_{3,1,2}}{r^{2}} = F_{0,1}^{\lambda_{1}}(t,r) - F_{0,1}^{\lambda_{2}}(t,r)\end{equation}
With the definitions from the $v_{3}$ subsection:
\begin{equation} v^{\lambda}_{3,1}(t,r)= \frac{-1}{r} \int_{t}^{\infty} ds \int_{0}^{s-t} \frac{\rho d\rho}{(s-t)} \lambda''(s) \left(\frac{-1-\rho^{2}+r^{2}}{\sqrt{(1+\rho^{2}-r^{2})^{2}+4r^{2}}} + F_{3}(r,\rho,\lambda(s))\right)\end{equation}
and
\begin{equation} v^{\lambda}_{3,2} = v^{\lambda}_{3} - v^{\lambda}_{3,1}\end{equation}
we then consider
\begin{equation} v_{3,2,a}^{\lambda}(t,r) := \frac{-1}{r} \int_{t}^{\infty} ds \int_{0}^{s-t} \rho d\rho \left(\frac{1}{\sqrt{(s-t)^{2}-\rho^{2}}}-\frac{1}{(s-t)}\right) \lambda''(s) \left(\frac{-1-\rho^{2}+r^{2}}{\sqrt{(1+\rho^{2}-r^{2})^{2}+4r^{2}}} + 1\right)\end{equation}
\begin{equation} v_{3,2,b}^{\lambda}(t,r) := \frac{-1}{r} \int_{t}^{\infty} ds \int_{0}^{s-t} \rho d\rho \left(\frac{1}{\sqrt{(s-t)^{2}-\rho^{2}}}-\frac{1}{(s-t)}\right) \lambda''(s) \left(-1+F_{3}(r,\rho,\lambda(s))\right)\end{equation}
and get
\begin{equation} \begin{split} &|v_{3,2,a}^{\lambda_{1}}(t,r) - v_{3,2,a}^{\lambda_{2}}(t,r)| \\
&\leq \frac{C}{r} \int_{t}^{\infty} ds \int_{0}^{s-t} \rho d\rho \left(\frac{1}{\sqrt{(s-t)^{2}-\rho^{2}}}-\frac{1}{(s-t)}\right) |\lambda_{1}''(s)-\lambda_{2}''(s)| \left(\frac{-1-\rho^{2}+r^{2}}{\sqrt{(1+\rho^{2}-r^{2})^{2}+4r^{2}}} +1\right)\\
&\leq \frac{C}{r} \int_{0}^{\infty} \rho d\rho \left(1+\frac{-1-\rho^{2}+r^{2}}{\sqrt{(1+\rho^{2}-r^{2})^{2}+4r^{2}}}\right) \int_{\rho + t}^{\infty} ds \left(\frac{1}{\sqrt{(s-t)^{2}-\rho^{2}}}-\frac{1}{(s-t)}\right) \frac{||e_{1}-e_{2}||_{X}}{s^{2} \log^{b+1}(s) \sqrt{\log(\log(s))}}\\
\end{split}\end{equation}
which gives
\begin{equation} |v_{3,2,a}^{\lambda_{1}}(t,r) - v_{3,2,a}^{\lambda_{2}}(t,r)| \leq \frac{C r ||e_{1}-e_{2}||_{X} }{t^{2}\log^{b+1}(t) \sqrt{\log(\log(t))}}\end{equation}
For $v_{3,2,b}$, we have
\begin{equation}\label{v32b}\begin{split} &v_{3,2,b}^{\lambda_{1}}(t,r) - v_{3,2,b}^{\lambda_{2}}(t,r) \\
&= \frac{-1}{r} \int_{t}^{\infty} ds \int_{0}^{s-t} \rho d\rho \left(\frac{1}{\sqrt{(s-t)^{2}-\rho^{2}}}-\frac{1}{(s-t)}\right) \left(\lambda_{1}''(s) - \lambda_{2}''(s)\right) \left(-1+F_{3}(r,\rho,\lambda_{1}(s)\right)\\
&+\frac{-1}{r} \int_{t}^{\infty} ds \int_{0}^{s-t} \rho d\rho \left(\frac{1}{\sqrt{(s-t)^{2}-\rho^{2}}}-\frac{1}{(s-t)}\right) \lambda_{2}''(s) \left(F_{3}(r,\rho,\lambda_{1}(s))-F_{3}(r,\rho,\lambda_{2}(s))\right)\end{split}\end{equation}
The first line on the right-hand side of \eqref{v32b} is treated identically to the analogous term arising in the pointwise estimates for $v_{3}$, and it is bounded above in absolute value by
$$\frac{C r ||e_{1}-e_{2}||_{X}}{t^{2}\log^{b+1}(t) \sqrt{\log(\log(t))}}$$
For the second line on the right-hand side of \eqref{v32b}, we start with 
\begin{equation} |F_{3}(r,\rho,z_{1}) - F_{3}(r,\rho,z_{2})| \leq \max_{\sigma \in [0,1]} |\partial_{3}F_{3}(r,\rho,\sigma z_{1}+(1-\sigma) z_{2})| |z_{1}-z_{2}|\end{equation}
we then note that
\begin{equation} \partial_{3}F_{3}(r,\rho,z) = \frac{-4(-1+\alpha) r^{2} z^{-3+2\alpha} (1+(r^{2}-\rho^{2})z^{2\alpha-2})}{((1+(r^{2}-\rho^{2})z^{2\alpha-2})^{2}+4\rho^{2}z^{2\alpha-2})^{3/2}}\end{equation}
so, for $z = \sigma \lambda_{1}(s)+(1-\sigma)\lambda_{2}(s)$, 
\begin{equation} \begin{split}|\partial_{3}F_{3}(r,\rho,z)| &\leq \frac{C r^{2} \lambda_{0,0}(s)^{2\alpha-3}}{(1+(-\rho^{2}+r^{2})\lambda_{0,0}(s)^{2\alpha-2})^{2}+4 \rho^{2} \lambda_{0,0}(s)^{2\alpha-2}}\\
&\leq \frac{C r^{2} \lambda_{0,0}(s)^{2\alpha-3}}{1+2(\rho^{2}+r^{2})\lambda_{0,0}(t)^{2\alpha-2}+(\rho^{2}-r^{2})^{2} \lambda_{0,0}(t)^{4\alpha-4}}\end{split}\end{equation}
Then, we get
\begin{equation} \begin{split} &|\frac{-1}{r} \int_{t}^{\infty} ds \int_{0}^{s-t} \rho d\rho \left(\frac{1}{\sqrt{(s-t)^{2}-\rho^{2}}}-\frac{1}{(s-t)}\right) \lambda_{2}''(s) \left(F_{3}(r,\rho,\lambda_{1}(s))-F_{3}(r,\rho,\lambda_{2}(s))\right)|\\
&\leq \frac{C}{r} \int_{0}^{\infty} \frac{\rho d\rho}{t^{2}\log^{b+1}(t)} \frac{r^{2} ||e_{1}-e_{2}||_{X}}{\log^{b(2\alpha-3)}(t) \log^{b}(t) \sqrt{\log(\log(t))}}\frac{1}{(1+2(\rho^{2}+r^{2})\lambda_{0,0}(t)^{2\alpha-2}+(\rho^{2}-r^{2})^{2} \lambda_{0,0}(t)^{4\alpha-4})}\end{split}\end{equation}
The $\rho$ integral appearing on the last line of the above estimate has been treated in the $v_{3}$ pointwise estimates, and, in total, we get

\begin{equation} |v_{3,2,b}^{\lambda_{1}}(t,r) - v_{3,2,b}^{\lambda_{2}}(t,r)| \leq \frac{C r ||e_{1}-e_{2}||_{X}}{t^{2}\log^{b+1}(t) \sqrt{\log(\log(t))}}  \end{equation}

We similarly divide $v_{3,1}^{\lambda_{1}}-v_{3,1}^{\lambda_{2}}$ into two parts. First, we consider
\begin{equation}\begin{split} &|\frac{-1}{r} \int_{t}^{\infty} ds \int_{0}^{s-t} \frac{\rho d\rho}{(s-t)} \left(\lambda_{1}''(s)-\lambda_{2}''(s)\right) \left(\frac{-1-\rho^{2}+r^{2}}{\sqrt{(1+\rho^{2}-r^{2})^{2}+4r^{2}}}+F_{3}(r,\rho,\lambda_{1}(s))\right)|\\
&\leq \frac{C}{r} \int_{t}^{t+6r} ds \int_{0}^{s-t} \frac{\rho d\rho}{(s-t)} |\lambda_{1}''(s)-\lambda_{2}''(s)|\\
&+ |\int_{t+6r}^{\infty} \frac{ds}{r} \int_{0}^{s-t} \frac{\rho d\rho}{(s-t)} \left(\lambda_{1}''(s)-\lambda_{2}''(s)\right) \left(\frac{-1-\rho^{2}+r^{2}}{\sqrt{(-1-\rho^{2}+r^{2})^{2}+4r^{2}}}+F_{3}(r,\rho,\lambda_{1}(s))\right)|\\
&\leq \frac{C r ||e_{1}-e_{2}||_{X}}{t^{2}\log^{b+1}(t) \sqrt{\log(\log(t))}} + II\end{split}\end{equation}
where, using a procedure similar to the analogous term treated in the $v_{3}$ pointwise estimates, we have
\begin{equation}\begin{split} II &= |-2 r \int_{6r}^{\infty} dw \left(\lambda_{1}''(t+w)-\lambda_{2}''(t+w)\right)\left(\frac{w}{2(1+w^{2})} - \frac{w}{2(\lambda_{1}(t+w)^{2-2\alpha}+w^{2})}\right)| +E_{II}\end{split}\end{equation}
with $$|E_{II}| \leq \frac{C r ||e_{1}-e_{2}||_{X}}{t^{2}\log^{b+1}(t) \sqrt{\log(\log(t))}}$$
We already estimated the integral appearing in the above expression in \eqref{v31bi1initial} (except with the replacement of $\lambda''$ with $\lambda_{1}''-\lambda_{2}''$). Therefore, the same procedure used there shows
\begin{equation}\begin{split} |II| &\leq C r \log(\log(t)) \sup_{x \geq t}|\lambda_{1}''(x)-\lambda_{2}''(x)|+ \frac{C r ||e_{1}-e_{2}||_{X}}{t^{2}\log^{b+1}(t) \sqrt{\log(\log(t))}}\\
&\leq \frac{C r \log(\log(t)) ||e_{1}-e_{2}||_{X}}{t^{2} \log^{b+1}(t) \sqrt{\log(\log(t))}} +\frac{C r ||e_{1}-e_{2}||_{X}}{t^{2}\log^{b+1}(t) \sqrt{\log(\log(t))}} \end{split}\end{equation}

Next, we consider the second part of $v_{3,1}^{\lambda_{1}}-v_{3,1}^{\lambda_{2}}$:
\begin{equation} \begin{split} &|\frac{-1}{r} \int_{t}^{\infty} ds \int_{0}^{s-t} \frac{\rho d\rho}{(s-t)} \lambda_{2}''(s) \left(F_{3}(r,\rho,\lambda_{1}(s))-F_{3}(r,\rho,\lambda_{2}(s))\right)|\\
&\leq \frac{C}{r} \int_{t}^{\infty} ds \int_{0}^{s-t} \frac{\rho d\rho}{(s-t)} |\lambda_{2}''(s)| \frac{r^{2}(\lambda_{0,0}(s)^{2\alpha-3}) |\lambda_{1}(s)-\lambda_{2}(s)|}{(1+2(\rho^{2}+r^{2})\lambda_{0,0}(t)^{2\alpha-2}+(\rho^{2}-r^{2})^{2}\lambda_{0,0}(t)^{4\alpha-4})}\\
&\leq C r \int_{t}^{t+\lambda_{0,0}(t)^{1-\alpha}} ds \int_{0}^{s-t} \frac{\rho d\rho}{(s-t)} |\lambda_{2}''(s)| |e_{1}(s)-e_{2}(s)| \lambda_{0,0}(s)^{2\alpha-3}\\
&+ C r \int_{t+\lambda_{0,0}(t)^{1-\alpha}}^{\infty} \frac{ds}{(s-t)} |\lambda_{2}''(s)| \frac{||e_{1}-e_{2}||_{X} \lambda_{0,0}(s)^{2\alpha-3}}{\log^{b}(s) \sqrt{\log(\log(s))}} \int_{0}^{\infty}\frac{\rho d\rho}{(1+2(\rho^{2}+r^{2})\lambda_{0,0}(t)^{2\alpha-2} + (\rho^{2}-r^{2})^{2} \lambda_{0,0}(t)^{4\alpha-4})}\\
&\leq C r \int_{t}^{t+\lambda_{0,0}(t)^{1-\alpha}} ds (s-t) |\lambda_{2}''(s)| |e_{1}(s)-e_{2}(s)| \lambda_{0,0}(s)^{2\alpha -3} + \frac{C r ||e_{1}-e_{2}||_{X}}{\log^{b}(t) \sqrt{\log(\log(t))}} \int_{t+\lambda_{0,0}(t)^{1-\alpha}}^{\infty} \frac{\lambda_{0,0}(s)^{2\alpha-3} \lambda_{0,0}(t)^{2-2\alpha}ds}{(s-t) s^{2} \log^{b+1}(s)}\\
&\leq \frac{C r ||e_{1}-e_{2}||_{X}}{t^{2}\log^{b+1}(t) \sqrt{\log(\log(t))}}+ \frac{C r ||e_{1}-e_{2}||_{X}}{t^{2}\sqrt{\log(\log(t))} \log^{b}(t)}\\
&\leq \frac{C r ||e_{1}-e_{2}||_{X}}{t^{2}\sqrt{\log(\log(t))} \log^{b}(t)}\end{split}\end{equation}

Combining the above estimates, we get
\begin{equation} \begin{split} |(v_{3}^{\lambda_{1}}-v_{3}^{\lambda_{2}})(t,r)| &\leq \frac{C r ||e_{1}-e_{2}||_{X}}{t^{2} \sqrt{\log(\log(t))} \log^{b}(t)}\end{split}\end{equation} 
Next, we prove an estimate on $v_{3,1,2}$ which is more useful for large $r$.
\begin{equation} v_{3}^{\lambda}(t,r) = \frac{-1}{r} \int_{t}^{\infty} ds \int_{0}^{s-t} \frac{\rho d\rho}{\sqrt{(s-t)^{2}-\rho^{2}}} \lambda''(s) \left(\frac{-1-\rho^{2}+r^{2}}{\sqrt{(-1-\rho^{2}+r^{2})^{2}+4r^{2}}} + F_{3}(r,\rho,\lambda(s))\right)\end{equation}
and get
\begin{equation}\label{v3liplarger} \begin{split} v_{3}^{\lambda_{1}}-v_{3}^{\lambda_{2}} (t,r) &= \frac{-1}{r} \int_{t}^{\infty} ds \int_{0}^{s-t} \frac{\rho d\rho}{\sqrt{(s-t)^{2}-\rho^{2}}} \left(\lambda_{1}''(s)-\lambda_{2}''(s)\right) \left(\frac{-1-\rho^{2}+r^{2}}{\sqrt{(-1-\rho^{2}+r^{2})^{2}+4r^{2}}} + F_{3}(r,\rho,\lambda_{1}(s))\right)\\
&-\frac{1}{r} \int_{t}^{\infty} ds \int_{0}^{s-t} \frac{\rho d\rho}{\sqrt{(s-t)^{2}-\rho^{2}}} \lambda_{2}''(s) \left(F_{3}(r,\rho,\lambda_{1}(s))-F_{3}(r,\rho,\lambda_{2}(s))\right)\end{split}\end{equation} 
For the first line of the right-hand side of \eqref{v3liplarger}, we have
\begin{equation}\begin{split} &|\frac{-1}{r} \int_{t}^{\infty} ds \int_{0}^{s-t} \frac{\rho d\rho}{\sqrt{(s-t)^{2}-\rho^{2}}} \left(\lambda_{1}''(s)-\lambda_{2}''(s)\right) \left(\frac{-1-\rho^{2}+r^{2}}{\sqrt{(-1-\rho^{2}+r^{2})^{2}+4r^{2}}} + F_{3}(r,\rho,\lambda_{1}(s))\right)|\\
&\leq \frac{C}{r} \int_{t}^{\infty} ds (s-t) \frac{||e_{1}-e_{2}||_{X}}{s^{2} \log^{b+1}(s) \sqrt{\log(\log(s))}}\\
&\leq \frac{C ||e_{1}-e_{2}||_{X}}{r \log^{b}(t) \sqrt{\log(\log(t))}}\end{split}\end{equation}

For the second line on the right-hand side of \eqref{v3liplarger}, we use the same method used previously to estimate 
$$F_{3}(r,\rho,\lambda_{1}) - F_{3}(r,\rho,\lambda_{2})$$
and we get
\begin{equation}\begin{split} &|-\frac{1}{r} \int_{t}^{\infty} ds \int_{0}^{s-t} \frac{\rho d\rho}{\sqrt{(s-t)^{2}-\rho^{2}}} \lambda_{2}''(s) \left(F_{3}(r,\rho,\lambda_{1}(s))-F_{3}(r,\rho,\lambda_{2}(s))\right)|\\
&\leq \frac{C}{r} \int_{t}^{\infty} ds \int_{0}^{s-t} \frac{\rho d\rho}{\sqrt{(s-t)^{2}-\rho^{2}}} |\lambda_{2}''(s)| \frac{\lambda_{0,0}(s)^{-3+2\alpha}}{\log^{b}(s) \sqrt{\log(\log(s))}} \frac{r^{2} ||e_{1}-e_{2}||_{X}}{(1+\lambda_{0,0}(s)^{4\alpha -4}(\rho^{2}-r^{2})^{2}+2\lambda_{0,0}(s)^{2\alpha-2}(\rho^{2}+r^{2}))}\\
&\leq \frac{C}{r} \int_{t}^{\infty} ds \int_{0}^{s-t} \frac{\rho d\rho}{\sqrt{(s-t)^{2}-\rho^{2}}} \frac{|\lambda_{2}''(s)| \lambda_{0,0}(s)^{-3+2\alpha}}{\log^{b}(s) \sqrt{\log(\log(s))}} \frac{r^{2} ||e_{1}-e_{2}||_{X}}{r^{2} \lambda_{0,0}(s)^{2\alpha -2}}\\
&\leq \frac{C ||e_{1}-e_{2}||_{X}}{r \log^{b}(t) \sqrt{\log(\log(t))}}\end{split}\end{equation}
whence, we conclude the final estimate
\begin{equation} |v_{3}^{\lambda_{1}}-v_{3}^{\lambda_{2}}|(t,r) \leq \frac{C ||e_{1}-e_{2}||_{X}}{r \log^{b}(t) \sqrt{\log(\log(t))}}, \quad r > \frac{t}{2}\end{equation}
\end{proof}
Next, we recall the function $E_{5}$ defined in the $v_{3}$ subsection:
\begin{equation} v_{3}^{\lambda}(t,r) = -r \int_{t+6r}^{\infty} ds \lambda''(s) (s-t) \left(\frac{1}{1+(s-t)^{2}}-\frac{1}{\lambda(t)^{2-2\alpha}+(s-t)^{2}}\right) + E_{5}^{\lambda}(t,r) \end{equation}

We will need to estimate $E_{5}^{\lambda_{1}}-E_{5}^{\lambda_{2}}$. For this purpose, some of the estimates already proven for $v_{3}^{\lambda_{1}}-v_{3}^{\lambda_{2}}$ will suffice, but we will need to use a slightly more complicated argument for some terms. In particular, we have
\begin{equation} E_{5}^{\lambda}(t,r) = v_{3,1,a}^{\lambda}(t,r) + v_{3,1,b,i,1}^{\lambda}(t,r) + v_{3,1,b,ii}^{\lambda}(t,r) + v_{3,2}^{\lambda}(t,r)\end{equation}
with
$$v_{3,1,b,i,1}^{\lambda}(t,r) = -2 r \int_{6r}^{\infty} dw \lambda''(t+w) w \left(\frac{-1}{2(\lambda(t+w)^{2-2\alpha}+w^{2})}+\frac{1}{2 (\lambda(t)^{2-2\alpha}+w^{2})}\right)$$
\begin{equation}\begin{split}v_{3,1,b,ii}^{\lambda}(t,r) &= \frac{-1}{r} \int_{t+6r}^{\infty}  \int_{0}^{s-t} \frac{\rho}{(s-t)} \lambda''(s) \left(\left(\frac{-1-\rho^{2}+r^{2}}{\sqrt{(-1-\rho^{2}+r^{2})^{2}+4r^{2}}}+F_{3}(r,\rho,\lambda(s))\right)\right.\\
&-\left. \left(\frac{2 r^{2}}{(1+\rho^{2})^{2}}-\frac{2r^{2}\lambda(s)^{2-2\alpha}}{(\lambda(s)^{2-2\alpha}+\rho^{2})^{2}}\right)\right)d\rho ds\end{split}\end{equation}
and
\begin{equation} v_{3,1,a}^{\lambda}(t,r) = \frac{-1}{r} \int_{t}^{t+6r} ds \int_{0}^{s-t} \frac{\rho d\rho}{(s-t)} \lambda''(s) \left(\frac{-1-\rho^{2}+r^{2}}{\sqrt{(-1-\rho^{2}+r^{2})^{2}+4r^{2}}}+F_{3}(r,\rho,\lambda(s))\right)\end{equation}
(Note that $v_{3,2}$ was already defined when proving estimates on $v_{3}^{\lambda_{1}}-v_{3}^{\lambda_{2}}$).

\begin{equation} \begin{split} v_{3,1,a}^{\lambda_{1}}-v_{3,1,a}^{\lambda_{2}} &= -\frac{1}{r} \int_{t}^{t+6r} ds \int_{0}^{s-t} \frac{\rho d\rho}{(s-t)} \left(\lambda_{1}''(s)-\lambda_{2}''(s)\right) \left(\frac{-1-\rho^{2}+r^{2}}{\sqrt{(-1-\rho^{2}+r^{2})^{2}+4r^{2}}}+F_{3}(r,\rho,\lambda_{1}(s))\right)\\
&-\frac{1}{r} \int_{t}^{t+6r} ds \int_{0}^{s-t} \frac{\rho d\rho}{(s-t)} \lambda_{2}''(s) \left(F_{3}(r,\rho,\lambda_{1}(s))-F_{3}(r,\rho,\lambda_{2}(s))\right)\end{split}\end{equation}
This gives
\begin{equation} \begin{split} |v_{3,1,a}^{\lambda_{1}}-v_{3,1,a}^{\lambda_{2}}| &\leq \frac{C}{r} \int_{t}^{t+6r} ds \int_{0}^{s-t} \frac{\rho d\rho}{(s-t)} |e_{1}''(s)-e_{2}''(s)| \\
&+ \frac{C}{r} \int_{t}^{t+6r} ds \int_{0}^{s-t} \frac{\rho d\rho}{(s-t)} \frac{|\lambda_{2}''(s)| (r^{2} \lambda_{0,0}(t)^{2\alpha -2}) \frac{\lambda_{0,0}(s)^{2\alpha -3}}{\lambda_{0,0}(t)^{2\alpha -2}} |e_{1}(s)-e_{2}(s)|}{(1+2(\rho^{2}+r^{2})\lambda_{0,0}(t)^{2\alpha -2} +(\rho^{2}-r^{2})^{2} \lambda_{0,0}(t)^{4\alpha -4})}\\
&\leq \frac{C r ||e_{1}-e_{2}||_{X}}{t^{2}\log^{b+1}(t) \sqrt{\log(\log(t))}} + \frac{C}{r} \int_{t}^{t+6r}ds  \int_{0}^{s-t} \frac{\rho d\rho}{(s-t)} |\lambda_{2}''(s)| \frac{\lambda_{0,0}(s)^{2\alpha -3}}{\lambda_{0,0}(t)^{2\alpha -2}} |e_{1}(s)-e_{2}(s)|\\
&\leq \frac{C r ||e_{1}-e_{2}||_{X}}{t^{2}\log^{b+1}(t) \sqrt{\log(\log(t))}}\end{split}\end{equation}

For the next term, we have
\begin{equation}\label{v3b1i1lip} \begin{split} &v_{3,1,b,i,1}^{\lambda_{1}}-v_{3,1,b,i,1}^{\lambda_{2}}\\
&= -2 r \int_{6r}^{\infty} dw \left(\lambda_{1}''(t+w)-\lambda_{2}''(t+w)\right) \left(\frac{-w}{2(\lambda_{1}(t+w)^{2-2\alpha}+w^{2})}+\frac{w}{2(\lambda_{1}(t)^{2-2\alpha}+w^{2})}\right)\\
&-2 r \int_{6r}^{\infty} dw \lambda_{2}''(t+w) \left(\frac{-w}{2(\lambda_{1}(t+w)^{2-2\alpha}+w^{2})}+\frac{w}{2(\lambda_{1}(t)^{2-2\alpha}+w^{2})}\right)\\
&+2 r \int_{6r}^{\infty} dw \lambda_{2}''(t+w) \left(\frac{-w}{2(\lambda_{2}(t+w)^{2-2\alpha}+w^{2})}+\frac{w}{2(\lambda_{2}(t)^{2-2\alpha}+w^{2})}\right)\end{split}\end{equation}
For the first line on the right-hand side of \eqref{v3b1i1lip}, we have
\begin{equation}\begin{split} &|-2 r \int_{6r}^{\infty} dw \left(\lambda_{1}''(t+w)-\lambda_{2}''(t+w)\right)\left(\frac{-w}{2(\lambda_{1}(t+w)^{2-2\alpha}+w^{2})}+\frac{w}{2(\lambda_{1}(t)^{2-2\alpha}+w^{2})}\right)|\\
&\leq C r \int_{6r}^{\infty} dw \frac{|e_{1}''(t+w)-e_{2}''(t+w)| w \left(w \lambda_{0,0}(t)^{1-2\alpha} |\lambda_{0,0}'(t)|\right)}{w^{2} \lambda_{0,0}(t+w)^{2-2\alpha} (1+w^{2}\lambda_{0,0}(t)^{2\alpha-2})}\\
&\leq \frac{C r ||e_{1}-e_{2}||_{X}}{t^{3} \log^{b +2}(t)\sqrt{\log(\log(t))}} \int_{0}^{\infty} \frac{dw}{(1+w^{2}\lambda_{0,0}(t)^{2\alpha-2})}\\
&\leq \frac{C r ||e_{1}-e_{2}||_{X}}{t^{3}\log^{b+2}(t) \sqrt{\log(\log(t))}}\end{split}\end{equation}

The second and third lines of \eqref{v3b1i1lip} can be treated together:
\begin{equation}\begin{split} &-2 r \int_{6r}^{\infty} dw \lambda_{2}''(t+w) \left(\frac{-w}{2(\lambda_{1}(t+w)^{2-2\alpha}+w^{2})}+\frac{w}{2(\lambda_{1}(t)^{2-2\alpha}+w^{2})}\right)\\
&+2 r \int_{6r}^{\infty} dw \lambda_{2}''(t+w) \left(\frac{-w}{2(\lambda_{2}(t+w)^{2-2\alpha}+w^{2})}+\frac{w}{2(\lambda_{2}(t)^{2-2\alpha}+w^{2})}\right)\\
&= -r \int_{6r}^{\infty} dw \lambda_{2}''(t+w)w \left(F(t+w)-F(t)\right)\end{split}\end{equation}

where $$F(x) = \frac{\lambda_{1}(x)^{2-2\alpha}-\lambda_{2}(x)^{2-2\alpha}}{(\lambda_{1}(x)^{2-2\alpha}+w^{2})(\lambda_{2}(x)^{2-2\alpha}+w^{2})}$$
We first note that
\begin{equation} |F'(\sigma s +(1-\sigma) t)| \leq \frac{C \lambda_{0,0}(t)^{1-2\alpha} ||e_{1}-e_{2}||_{X}}{t \log^{b}(t) \log^{-b}(s) \log^{b+1}(t) \sqrt{\log(\log(t))} (w^{2}+\lambda_{0,0}(s)^{2-2\alpha})^{2}}, \quad 0 \leq \sigma \leq 1\end{equation}

Then, we get
\begin{equation}\begin{split} &|-r \int_{6r}^{\infty} dw \lambda_{2}''(t+w)w \left(F(t+w)-F(t)\right)| \\
&\leq \frac{C r ||e_{1}-e_{2}||_{X} \log^{-b(1-2\alpha)}(t)}{t \log^{b+1}(t) \sqrt{\log(\log(t))}} \int_{6r}^{\infty} \frac{|\lambda_{2}''(t+w)| dw}{(w^{2}+\lambda_{0,0}(s)^{2-2\alpha})\log^{b}(t) \log^{-b}(t+w)}\\
&\leq \frac{C r \log^{-b(1-2\alpha)}(t) ||e_{1}-e_{2}||_{X}}{t \log^{b+1}(t) \sqrt{\log(\log(t))} \log^{b}(t)} \int_{t+6r}^{\infty} \frac{ds}{s^{2}\log(s) \lambda_{0,0}(s)^{2-2\alpha} (\lambda_{0,0}(s)^{2\alpha -2}(s-t)^{2}+1)}\\
&\leq \frac{C r ||e_{1}-e_{2}||_{X}}{t^{3}\log^{b+2}(t) \sqrt{\log(\log(t))}}\end{split}\end{equation}

It then remains to consider 
\begin{equation}\label{v31biilip}\begin{split}& v_{3,1,b,ii}^{\lambda_{1}}-v_{3,1,b,ii}^{\lambda_{2}} \\
&= \frac{-1}{r} \int_{t+6r}^{\infty}  \int_{0}^{s-t} \frac{\rho }{(s-t)} \left(\lambda_{1}''(s)-\lambda_{2}''(s)\right) \left(\frac{-1-\rho^{2}+r^{2}}{\sqrt{(-1-\rho^{2}+r^{2})^{2}+4r^{2}}}+F_{3}(r,\rho,\lambda_{1}(s)) \right.\\
&\left.- \left(\frac{2 r^{2}}{(1+\rho^{2})^{2}}-\frac{2 r^{2} \lambda_{1}(s)^{2-2\alpha}}{(\lambda_{1}(s)^{2-2\alpha} +\rho^{2})^{2}}\right)\right)d\rho ds\\
&-\frac{1}{r} \int_{t+6r}^{\infty} ds \int_{0}^{s-t} \frac{\rho d\rho}{(s-t)} \lambda_{2}''(s) \left(F_{3}(r,\rho,\lambda_{1}(s) + \frac{2 r^{2} \lambda_{1}(s)^{2-2\alpha}}{(\lambda_{1}(s)^{2-2\alpha}+\rho^{2})^{2}} - \left(F_{3}(r,\rho,\lambda_{2}(s) + \frac{2 r^{2} \lambda_{2}(s)^{2-2\alpha}}{(\lambda_{2}(s)^{2-2\alpha}+\rho^{2})^{2}}\right)\right)\end{split}\end{equation}

The first two lines on the right-hand side of \eqref{v31biilip} are estimated in the same way as the analogous term appearing in the $v_{3}$ pointwise estimates, and we get
\begin{equation}\begin{split}&|\frac{-1}{r} \int_{t+6r}^{\infty}  \int_{0}^{s-t} \frac{\rho }{(s-t)} \left(\lambda_{1}''(s)-\lambda_{2}''(s)\right) \left(\frac{-1-\rho^{2}+r^{2}}{\sqrt{(-1-\rho^{2}+r^{2})^{2}+4r^{2}}}+F_{3}(r,\rho,\lambda_{1}(s)) \right.\\
&\left.- \left(\frac{2 r^{2}}{(1+\rho^{2})^{2}}-\frac{2 r^{2} \lambda_{1}(s)^{2-2\alpha}}{(\lambda_{1}(s)^{2-2\alpha} +\rho^{2})^{2}}\right)\right)d\rho ds|\\
&\leq \frac{C r ||e_{1}-e_{2}||_{X}}{t^{2}\log^{b+1}(t) \sqrt{\log(\log(t))}}\end{split}\end{equation}

To treat the third line on the right-hand side of \eqref{v31biilip}, we start with defining $G_{3,1}$ by
\begin{equation} G_{3,1}(w,r,\lambda) := \int_{0}^{w} \rho \left( F_{3}(r,\rho,\lambda(s)) + \frac{2 r^{2} \lambda(s)^{2-2\alpha}}{(\lambda(s)^{2-2\alpha}+\rho^{2})^{2}}\right)d\rho\end{equation}
Then, we get 
\begin{equation} |G_{3,1}(w,r,\lambda_{1}(s))-G_{3,1}(w,r,\lambda_{2}(s))| \leq \frac{|e_{1}(s)-e_{2}(s)| \lambda_{0,0}(s)^{-3+2\alpha} r^{4}}{(1+\lambda_{0,0}(s)^{2\alpha-2}w^{2})^{2}}\end{equation}
which gives
\begin{equation}\begin{split}&|-\frac{1}{r} \int_{t+6r}^{\infty} ds \int_{0}^{s-t} \frac{\rho d\rho}{(s-t)} \lambda_{2}''(s) \left(F_{3}(r,\rho,\lambda_{1}(s)) + \frac{2 r^{2} \lambda_{1}(s)^{2-2\alpha}}{(\lambda_{1}(s)^{2-2\alpha}+\rho^{2})^{2}} - \left(F_{3}(r,\rho,\lambda_{2}(s)) + \frac{2 r^{2} \lambda_{2}(s)^{2-2\alpha}}{(\lambda_{2}(s)^{2-2\alpha}+\rho^{2})^{2}}\right)\right)|\\
&\leq \frac{C}{r} \int_{t+6r}^{\infty} ds \frac{|\lambda_{0,0}''(s)|}{(s-t)} \frac{|e_{1}(s)-e_{2}(s)| \lambda_{0,0}(s)^{1-2\alpha} r^{4}}{(\lambda_{0,0}(s)^{2-2\alpha})((s-t)^{2})}\\
&\leq \frac{C r ||e_{1}-e_{2}||_{X}}{t^{2}\log^{b+1}(t) \sqrt{\log(\log(t))}}\end{split}\end{equation}

Combining the above, we get
\begin{equation} |\left(E_{5}^{\lambda_{1}}-E_{5}^{\lambda_{2}}\right)(t,r)| \leq \frac{C r ||e_{1}-e_{2}||_{X}}{t^{2}\log^{b+1}(t) \sqrt{\log(\log(t))}}\end{equation}

\begin{lemma}We have the following estimates
\begin{equation}\label{v4lipfinalest}|v_{4}^{\lambda_{1}}-v_{4}^{\lambda_{2}}|(t,r) \leq \begin{cases} \frac{C ||e_{1}-e_{2}||_{X}}{t^{2}\sqrt{\log(\log(t))}\log^{N+3b-1}(t)}, \quad r \leq \frac{t}{2}\\
\frac{C ||e_{1}-e_{2}||_{X}}{t \log^{3b+2N}(t) \sqrt{\log(\log(t))}}, \quad r \geq \frac{t}{2}\end{cases}\end{equation}\end{lemma}
\begin{proof}
We proceed to estimate $v_{4,1,2}:=v_{4}^{\lambda_{1}}-v_{4}^{\lambda_{2}}$, by noting that $v_{4,1,2}$ solves the following equation with $0$ Cauchy data at infinity:
\begin{equation} -\partial_{tt}v_{4,1,2}+\partial_{rr}v_{4,1,2}+\frac{1}{r}\partial_{r}v_{4,1,2}-\frac{v_{4,1,2}}{r^{2}} = v_{4,c}^{\lambda_{1}}-v_{4,c}^{\lambda_{2}}\end{equation}

So, we start by estimating $v_{4,c}^{\lambda_{1}}-v_{4,c}^{\lambda_{2}}$:
\begin{equation} \label{v4clip}\begin{split} |v_{4,c}^{\lambda_{1}}-v_{4,c}^{\lambda_{2}}| &\leq C \frac{\chi_{\geq 1}(\frac{2r}{\log^{N}(t)}) ||e_{1}-e_{2}||_{X}}{\log^{2b}(t) \sqrt{\log(\log(t))}} \frac{|v_{1}^{\lambda_{1}}+v_{2}+v_{3}^{\lambda_{1}}|}{(\lambda_{0,0}(t)^{2}+r^{2})^{2}}\\
&+ C \chi_{\geq 1}(\frac{2r}{\log^{N}(t)}) \frac{||e_{1}-e_{2}||_{X}}{\log^{b}(t) \sqrt{\log(\log(t))}}\frac{1}{t^{2}\log^{2b+1-2\alpha b}(t)r^{3}}\\
&+C\frac{\chi_{\geq 1}(\frac{2r}{\log^{N}(t)}) \lambda_{0,0}(t)^{2}}{(r^{2}+\lambda_{0,0}(t)^{2})^{2}} \left(|v_{1}^{\lambda_{1}}-v_{1}^{\lambda_{2}}|+|v_{3}^{\lambda_{1}}-v_{3}^{\lambda_{2}}|\right)\end{split}\end{equation}
where we used the explicit formula for $F_{0,2}$.\\
\\
We note that the right-hand side of the equation for $v_{1}^{\lambda}$ depends linearly on $\lambda''$. Therefore\\
$v_{1}^{\lambda_{1}}-v_{1}^{\lambda_{2}}=v_{1}^{\lambda_{1}-\lambda_{2}}$, so we can use our estimates for $v_{1}^{\lambda}$ which were previously recorded.\\
We thus have, in addition to \eqref{v3lip}, the following estimates, valid for any $\lambda$ of the form 
$$\lambda(t) = \lambda_{0}(t) + f(t), \quad f \in \overline{B}_{1}(0) \subset X$$
\begin{equation} |v_{1}^{\lambda}+v_{2}+v_{3}^{\lambda}| \leq \begin{cases} \frac{C r}{t^{2}\log^{b}(t)}, \quad r \leq \frac{t}{2}\\
\frac{C \log(r)}{|t-r|}, \quad  t > r \geq \frac{t}{2}\end{cases}\end{equation}

\begin{equation} |v_{1}^{\lambda_{1}}-v_{1}^{\lambda_{2}}| \leq \begin{cases} \frac{C r ||e_{1}-e_{2}||_{X}}{t^{2}\log^{b}(t) \sqrt{\log(\log(t))}}, \quad r \leq \frac{t}{2}\\
\frac{C ||e_{1}-e_{2}||_{X}}{r \sqrt{\log(\log(t))} \log^{b}(t)}, \quad r >\frac{t}{2}\end{cases}\end{equation}
This gives
\begin{equation} \label{v4cestlip}|v_{4,c}^{\lambda_{1}}-v_{4,c}^{\lambda_{2}}| \leq C \chi_{\geq 1}(\frac{2r}{\log^{N}(t)}) ||e_{1}-e_{2}||_{X} \begin{cases} \frac{1}{r^{3}t^{2} \log^{3b}(t) \sqrt{\log(\log(t))}}, \quad r \leq \frac{t}{2}\\
\frac{\log(r)}{r^{4} \log^{2b}(t) \sqrt{\log(\log(t))} |t-r|} + \frac{1}{t^{2} \log^{3b+1-2\alpha b}(t) \sqrt{\log(\log(t))} r^{3}}, \quad t> r > \frac{t}{2}\end{cases}\end{equation}
Now, we start with

\begin{equation}\begin{split} &\left(v_{4}^{\lambda_{1}}-v_{4}^{\lambda_{2}}\right)(t,r) \\
&= \frac{-1}{2\pi} \int_{t}^{\infty} ds \int_{0}^{s-t} \frac{\rho d\rho}{\sqrt{(s-t)^{2}-\rho^{2}}} \int_{0}^{2\pi} d\theta \frac{(r+\rho \cos(\theta))(v_{4,c}^{\lambda_{1}}-v_{4,c}^{\lambda_{2}})(s,\sqrt{r^{2}+\rho^{2}+2 r \rho \cos(\theta)})}{\sqrt{r^{2}+\rho^{2}+2 r \rho \cos(\theta)}}\end{split}\end{equation}
which gives
\begin{equation} |v_{4}^{\lambda_{1}}-v_{4}^{\lambda_{2}}| \leq C \int_{t}^{\infty} ds \int_{0}^{s-t} \frac{\rho d\rho}{\sqrt{(s-t)^{2}-\rho^{2}}} \int_{0}^{2\pi} d\theta |v_{4,c}^{\lambda_{1}}-v_{4,c}^{\lambda_{2}}|(s,\sqrt{r^{2}+\rho^{2}+2 r \rho \cos(\theta)})\end{equation}

Note that \eqref{v4cestlip} has the same right-hand side as \eqref{v4cest1}, except for an extra factor of $\frac{||e_{1}-e_{2}||_{X}}{\sqrt{\log(\log(t))}}$. So, we can inspect the intermediate steps of the procedure used when obtaining pointwise estimates on $v_{4}$, thereby getting
\begin{equation} |v_{4}^{\lambda_{1}}-v_{4}^{\lambda_{2}}|(t,r) \leq \frac{C ||e_{1}-e_{2}||_{X}}{t^{2}\sqrt{\log(\log(t))}\log^{N+3b-1}(t)}, \quad r \leq \frac{t}{2} \end{equation}

In the region $r \geq \frac{t}{2}$, we use the following, slightly different estimate for $v_{1}^{\lambda}+v_{2}+v_{3}^{\lambda}$, again valid for all $\lambda(t) = \lambda_{0}(t) + f(t), \quad f \in \overline{B}_{1}(0) \subset X$:
\begin{equation} |v_{1}^{\lambda}+v_{2}+v_{3}^{\lambda}| \leq \begin{cases} \frac{C r}{t^{2}\log^{b}(t)}, \quad r \leq \frac{t}{2}\\
\frac{C}{\sqrt{r}}, \quad t-\sqrt{t} \leq r \leq t+\sqrt{t}\\
\frac{C \log(r)}{|t-r|}, \quad \frac{t}{2} \leq r \leq t-\sqrt{t}, \text{  or  } r \geq t+\sqrt{t}\end{cases}\end{equation} 
We then get
\begin{equation} \label{v4cestlip2}\begin{split}|v_{4,c}^{\lambda_{1}}-v_{4,c}^{\lambda_{2}}| &\leq C \chi_{\geq 1}(\frac{2r}{\log^{N}(t)}) ||e_{1}-e_{2}||_{X} \\
&\cdot \begin{cases} \frac{1}{r^{3}t^{2} \log^{3b}(t) \sqrt{\log(\log(t))}}, \quad r \leq \frac{t}{2}\\
\frac{1}{\log^{2b}(t) \sqrt{\log(\log(t))} r^{9/2}}, \quad t-\sqrt{t} \leq r \leq t+\sqrt{t}\\
\frac{\log(r)}{r^{4} \log^{2b}(t) \sqrt{\log(\log(t))} |t-r|} + \frac{1}{t^{2} \log^{3b+1-2\alpha b}(t) \sqrt{\log(\log(t))} r^{3}}, \quad t-\sqrt{t} > r > \frac{t}{2} \text{ or } r>t+\sqrt{t}\end{cases}\end{split}\end{equation}
Note that \eqref{v4cestlip2} again has the same right-hand side as \eqref{v4cest2}, except for the extra factor of $\frac{||e_{1}-e_{2}||_{X}}{\sqrt{\log(\log(t))}}$, so, we can infer the following estimates from our study of $v_{4}$:
 \begin{equation} \label{v4c12l2est} ||\widehat{v_{4,c}^{\lambda_{1}}-v_{4,c}^{\lambda_{2}}}(x,\xi)||_{L^{2}(\xi d\xi)} \leq \frac{C ||e_{1}-e_{2}||_{X}}{x^{2} \log^{3b+2N}(x) \sqrt{\log(\log(x))}}\end{equation}

Then, we simply note:
\begin{equation} v_{4,1,2}(t,r) = \int_{t}^{\infty} ds \int_{0}^{\infty} d\xi J_{1}(r\xi) \sin((t-s)\xi) \left(\widehat{v_{4,c}^{\lambda_{1}}-v_{4,c}^{\lambda_{2}}}(s,\xi)\right)\end{equation}
So,
\begin{equation} \begin{split} |v_{4,1,2}(t,r)| &\leq C \int_{t}^{\infty} ds \int_{0}^{\frac{1}{r}}  r \xi |\widehat{v_{4,c}^{\lambda_{1}}-v_{4,c}^{\lambda_{2}}}|(s,\xi) d\xi + C \int_{t}^{\infty} ds \int_{\frac{1}{r}}^{\infty} d\xi \frac{|\widehat{v_{4,c}^{\lambda_{1}}-v_{4,c}^{\lambda_{2}}}|(s,\xi)}{\sqrt{r\xi}}\\
&\leq C r \int_{t}^{\infty} ds \left(\int_{0}^{\frac{1}{r}} \xi d\xi \right)^{1/2} \left(\int_{0}^{\frac{1}{r}} \xi |\widehat{v_{4,c}^{\lambda_{1}}-v_{4,c}^{\lambda_{2}}}(s,\xi)|^{2}d\xi\right)^{1/2} \\
&+ \frac{C}{\sqrt{r}}\int_{t}^{\infty} ds \left(\int_{\frac{1}{r}}^{\infty} \frac{d\xi}{\xi^{2}}\right)^{1/2} \left(\int_{\frac{1}{r}}^{\infty} \xi |\widehat{v_{4,c}^{\lambda_{1}}-v_{4,c}^{\lambda_{2}}}(s,\xi)|^{2} d\xi\right)^{1/2}\\
&\leq C \int_{t}^{\infty} ds ||\widehat{v_{4,c}^{\lambda_{1}}-v_{4,c}^{\lambda_{2}}}(s,\xi)||_{L^{2}(\xi d\xi)}\\
&\leq \frac{C ||e_{1}-e_{2}||_{X}}{t \sqrt{\log(\log(t))} \log^{3b+2N}(t)}\end{split}\end{equation}
which finishes the proof of the lemma.
\end{proof}
\begin{lemma}We have the following pointwise estimate

 \begin{equation} \label{v5lipest} |v_{5}^{\lambda_{1}}-v_{5}^{\lambda_{2}}(t,r)| \leq C \frac{||e_{1}-e_{2}||_{X} \log^{2}(t)}{t^{3} \sqrt{\log(\log(t))} \log^{b}(t)}, \quad r \leq \frac{t}{2}\end{equation}
\end{lemma}

\begin{proof}
As was the case for $v_{4}$, we start with noting that
$$v_{5,1,2}:=v_{5}^{\lambda_{1}}-v_{5}^{\lambda_{2}}$$ solves the equation (with 0 Cauchy data at infinity)
$$-\partial_{tt}v_{5,1,2}+\partial_{rr}v_{5,1,2}+\frac{1}{r}\partial_{r}v_{5,1,2}-\frac{v_{5,1,2}}{r^{2}} = N_{2}(f_{v_{5}}^{\lambda_{1}})-N_{2}(f_{v_{5}}^{\lambda_{2}})$$
where
$$f_{v_{5}}^{\lambda} = v_{1}^{\lambda}+v_{2}+v_{3}^{\lambda}+v_{4}^{\lambda}$$
Collecting the estimates from previous subsections, one estimate on $f(v_{5}^{\lambda})$, valid for all \\
$\lambda= \lambda_{0}+f, \quad f \in \overline{B}_{1}(0) \subset X$, is
\begin{equation} |f_{v_{5}}^{\lambda}| \leq \begin{cases} \frac{C r}{t^{2} \log^{b}(t)}, \quad r \leq \frac{t}{2}\\
\frac{C\log(r)}{|t-r|}, \quad t > r > \frac{t}{2}\end{cases}\end{equation}

Now, we start to estimate the right-hand side of the $v_{5,1,2}$ equation. 
\begin{equation}\label{n2diffmain}\begin{split} N_{2}(f_{v_{5}}^{\lambda_{1}})-N_{2}(f_{v_{5}}^{\lambda_{2}})&= \left(\frac{\sin(2Q_{\frac{1}{\lambda_{1}(t)}}(r))}{2r^{2}}-\frac{\sin(2Q_{\frac{1}{\lambda_{2}(t)}}(r))}{2r^{2}}\right)\left(\cos(2f_{v_{5}}^{\lambda_{1}})-1\right)\\
&+\frac{\sin(2Q_{\frac{1}{\lambda_{2}(t)}}(r))}{2r^{2}}\left(\cos(2f_{v_{5}}^{\lambda_{1}})-\cos(2f_{v_{5}}^{\lambda_{2}})\right)\\
&+\left(\frac{\cos(2Q_{\frac{1}{\lambda_{1}(t)}}(r))-\cos(2Q_{\frac{1}{\lambda_{2}(t)}}(r))}{2r^{2}}\right)\left(\sin(2f_{v_{5}}^{\lambda_{1}})-2f_{v_{5}}^{\lambda_{1}}\right)\\
&+\frac{\cos(2Q_{\frac{1}{\lambda_{2}(t)}}(r))}{2r^{2}}\left(\sin(2f_{v_{5}}^{\lambda_{1}})-2f_{v_{5}}^{\lambda_{1}}-\left(\sin(2f_{v_{5}}^{\lambda_{2}})-2f_{v_{5}}^{\lambda_{2}}\right)\right)\end{split}\end{equation}
For the first line of \eqref{n2diffmain}, we have
\begin{equation}\begin{split}&|\left(\frac{\sin(2Q_{\frac{1}{\lambda_{1}(t)}}(r))}{2r^{2}}-\frac{\sin(2Q_{\frac{1}{\lambda_{2}(t)}}(r))}{2r^{2}}\right)\left(\cos(2f_{v_{5}}^{\lambda_{1}})-1\right)|\\
&\leq \frac{C ||e_{1}-e_{2}||_{X}}{r \log^{b}(t) \sqrt{\log(\log(t))}} \frac{\left(f_{v_{5}}^{\lambda_{1}}\right)^{2}}{(\lambda_{0,0}(t)^{2}+r^{2})}\\
&\leq \frac{C ||e_{1}-e_{2}||_{X}}{r \log^{b}(t) \sqrt{\log(\log(t))}} \begin{cases} \frac{r^{2}}{t^{4} \log^{2b}(t) (\lambda_{0,0}(t)^{2}+r^{2})}, \quad r \leq \frac{t}{2}\\
\frac{\log^{2}(r)}{r^{2}(t-r)^{2}} , \quad t > r \geq \frac{t}{2}\end{cases}\end{split}\end{equation}
For the second line of \eqref{n2diffmain}, we have
\begin{equation}\begin{split}&|\frac{\sin(2Q_{\frac{1}{\lambda_{2}(t)}}(r))}{2r^{2}}\left(\cos(2f_{v_{5}}^{\lambda_{1}})-\cos(2f_{v_{5}}^{\lambda_{2}})\right)|\\
&\leq \frac{C \lambda_{2}(t)}{(r^{2}+\lambda_{2}(t)^{2}) r} \left(|f_{v_{5}}^{\lambda_{1}}|+|f_{v_{5}}^{\lambda_{2}}|\right)\left(|v_{1}^{\lambda_{1}}-v_{1}^{\lambda_{2}}|+|v_{3}^{\lambda_{1}}-v_{3}^{\lambda_{2}}|+|v_{4}^{\lambda_{1}}-v_{4}^{\lambda_{2}}|\right)\\
&\leq \begin{cases} \frac{C r \lambda_{0,0}(t) ||e_{1}-e_{2}||_{X}}{(r^{2}+\lambda_{0,0}(t)^{2}) t^{4} \log^{2b}(t) \sqrt{\log(\log(t))}} + \frac{C ||e_{1}-e_{2}||_{X} \lambda_{0,0}(t)}{(r^{2}+\lambda_{0,0}(t)^{2}) t^{4} \sqrt{\log(\log(t))} \log^{N+4b-1}(t)}, \quad r \leq \frac{t}{2}\\
\frac{C \log(t) ||e_{1}-e_{2}||_{X}}{r^{3} |t-r| t \log^{2b}(t) \sqrt{\log(\log(t))}}, \quad t> r > \frac{t}{2}\end{cases}\end{split}\end{equation}

We consider the third line of \eqref{n2diffmain}, and get
\begin{equation}\begin{split}&|\left(\frac{\cos(2Q_{\frac{1}{\lambda_{1}(t)}}(r))-\cos(2Q_{\frac{1}{\lambda_{2}(t)}}(r))}{2r^{2}}\right)\left(\sin(2f_{v_{5}}^{\lambda_{1}})-2f_{v_{5}}^{\lambda_{1}}\right)|\\
&\leq \frac{C ||e_{1}-e_{2}||_{X}}{\log^{2b}(t) \sqrt{\log(\log(t))}(r^{2}+\lambda_{0,0}(t)^{2})^{2}}\begin{cases} \frac{r^{3}}{t^{6} \log^{3b}(t)}, \quad r \leq \frac{t}{2}\\
\frac{ \log^{3}(r)}{|t-r|^{3}}, \quad t > r > \frac{t}{2}\end{cases}\end{split}\end{equation}

For the fourth line of \eqref{n2diffmain}, we have 
\begin{equation}\begin{split}&|\frac{\cos(2Q_{\frac{1}{\lambda_{2}(t)}}(r))}{2r^{2}}\left(\sin(2f_{v_{5}}^{\lambda_{1}})-2f_{v_{5}}^{\lambda_{1}}-\left(\sin(2f_{v_{5}}^{\lambda_{2}})-2f_{v_{5}}^{\lambda_{2}}\right)\right)|\\
&\leq \frac{C}{r^{2}} \left(|f_{v_{5}}^{\lambda_{1}}|^{2}+|f_{v_{5}}^{\lambda_{2}}|^{2}\right)|f_{v_{5}}^{\lambda_{1}}-f_{v_{5}}^{\lambda_{2}}|\\
&\leq \begin{cases} \frac{C r ||e_{1}-e_{2}||_{X}}{t^{6} \sqrt{\log(\log(t))} \log^{3b}(t)} + \frac{C ||e_{1}-e_{2}||_{X}}{t^{6} \sqrt{\log(\log(t))} \log^{N+5b-1}(t)}, \quad r \leq \frac{t}{2}\\
\frac{C ||e_{1}-e_{2}||_{X} \log^{2}(t)}{r^{3} (t-r)^{2} \log^{b}(t) \sqrt{\log(\log(t))}}, \quad t > r > \frac{t}{2}\end{cases} \end{split}\end{equation}

Combining all of the above estimates, we obtain
\begin{equation}\label{n2diffestfinal}\begin{split} &|N_{2}(f_{v_{5}}^{\lambda_{1}})(t,r)-N_{2}(f_{v_{5}}^{\lambda_{2}})(t,r)|\\
&\leq \begin{cases} \frac{C r \lambda_{0,0}(t) ||e_{1}-e_{2}||_{X}}{(r^{2}+\lambda_{0,0}(t)^{2}) t^{4} \log^{2b}(t) \sqrt{\log(\log(t))}}+\frac{C ||e_{1}-e_{2}||_{X} \lambda_{0,0}(t)}{(r^{2}+\lambda_{0,0}(t)^{2}) t^{4} \sqrt{\log(\log(t))} \log^{N+4b-1}(t)}, \quad r \leq \frac{t}{2}\\
\frac{C \log(t) ||e_{1}-e_{2}||_{X}}{r^{3} \log^{b}(t) \sqrt{\log(\log(t))}}\left(\frac{\log(t)}{(t-r)^{2}}+\frac{1}{t|t-r| \log^{b}(t)} + \frac{\log^{2}(t)}{\log^{b}(t) r |t-r|^{3}}\right), \quad t > r > \frac{t}{2}\end{cases}\end{split}\end{equation}
We now proceed to estimate $v_{5,1,2}$ in the region $r \leq \frac{t}{2}$. We start with
\begin{equation}\begin{split} v_{5,1,2}(t,r) &= \frac{-1}{2\pi} \int_{t}^{\infty} ds \int_{0}^{s-t}\frac{\rho d\rho}{\sqrt{(s-t)^{2}-\rho^{2}}} \\
&\int_{0}^{2\pi} d\theta \left(N_{2}(f_{v_{5}}^{\lambda_{1}})-N_{2}(f_{v_{5}}^{\lambda_{2}})\right)(s,\sqrt{r^{2}+\rho^{2}+2 \rho r \cos(\theta)}) \frac{(r+\rho \cos(\theta))}{\sqrt{r^{2}+2 r \rho \cos(\theta) + \rho^{2}}}\end{split}\end{equation}
We then estimate as follows
\begin{equation}\begin{split} |v_{5,1,2}(t,r)| &\leq C \int_{t}^{\infty} ds \int_{B_{s-t}(0)} \frac{dA(y)}{\sqrt{(s-t)^{2}-|y|^{2}}} |N_{2}(f_{v_{5}}^{\lambda_{1}})-N_{2}(f_{v_{5}}^{\lambda_{2}})|(s,|x+y|)\\
&\leq C \int_{t}^{\infty} ds \int_{B_{s-t}(0) \cap B_{\frac{s}{2}}(-x)} \frac{dA(y)}{\sqrt{(s-t)^{2}-|y|^{2}}} |N_{2}(f_{v_{5}}^{\lambda_{1}})-N_{2}(f_{v_{5}}^{\lambda_{2}})|(s,|x+y|)\\
&+C \int_{t}^{\infty} ds \int_{B_{s-t}(0) \cap (B_{\frac{s}{2}}(-x))^{c}} \frac{dA(y)}{\sqrt{(s-t)^{2}-|y|^{2}}} |N_{2}(f_{v_{5}}^{\lambda_{1}})-N_{2}(f_{v_{5}}^{\lambda_{2}})|(s,|x+y|)\end{split}\end{equation}
Now, we carry out the identical steps which were done to obtain the $v_{5}^{\lambda}$ estimates, and get
\begin{equation} |v_{5,1,2}(t,r)| \leq C \frac{||e_{1}-e_{2}||_{X} \log^{2}(t)}{t^{3} \sqrt{\log(\log(t))} \log^{b}(t)}, \quad r \leq \frac{t}{2}\end{equation}
which completes the proof of the lemma.
\end{proof}

We can now use all of the above estimates to get
\begin{equation} \begin{split} &|\int_{0}^{\infty} \left(\frac{\cos(2Q_{1}(\frac{r}{\lambda_{2}(t)}))-1}{r^{2}\lambda_{2}(t)}\right) \phi_{0}(\frac{r}{\lambda_{2}(t)}) r \left(v_{4}^{\lambda_{1}}-v_{4}^{\lambda_{2}}\right)(t,r)\left(1-\chi_{\geq 1}(\frac{4r}{t})\right) dr| \\
&\leq C \lambda_{0}(t)^{2} \int_{0}^{\frac{t}{2}} \frac{r^{2}}{(r^{2}+\lambda_{0}(t)^{2})^{3}} \frac{||e_{1}-e_{2}||_{X}}{t^{2} \sqrt{\log(\log(t))} \log^{N+3b-1}(t)}dr\leq \frac{C ||e_{1}-e_{2}||_{X}}{t^{2} \sqrt{\log(\log(t))} \log^{N+2b-1}(t)}\end{split}\end{equation}

\begin{equation} \begin{split} &|\int_{0}^{\infty} \left(\frac{\cos(2Q_{1}(\frac{r}{\lambda_{2}(t)}))-1}{r^{2}\lambda_{2}(t)}\right) \phi_{0}(\frac{r}{\lambda_{2}(t)}) r \left(v_{5}^{\lambda_{1}}-v_{5}^{\lambda_{2}}\right)(t,r)\left(1-\chi_{\geq 1}(\frac{4r}{t})\right) dr| \\
&\leq C \lambda_{0}(t)^{2} \int_{0}^{\frac{t}{2}} \frac{r^{2}}{(r^{2}+\lambda_{0}(t)^{2})^{3}} \frac{||e_{1}-e_{2}||_{X} \log^{2}(t)}{t^{3} \sqrt{\log(\log(t))} \log^{b}(t)}dr \leq \frac{C ||e_{1}-e_{2}||_{X} \log^{2}(t)}{t^{3} \sqrt{\log(\log(t))}}\end{split}\end{equation}

\begin{equation} \begin{split} &|\int_{0}^{\infty} \left(\frac{\cos(2Q_{1}(\frac{r}{\lambda_{2}(t)}))-1}{r^{2}\lambda_{2}(t)}\right) \phi_{0}(\frac{r}{\lambda_{2}(t)}) r \left(E_{5}^{\lambda_{1}}-E_{5}^{\lambda_{2}}\right)(t,r) dr|\\
&\leq C \int_{0}^{\infty} \frac{\lambda_{2}(t)^{2} r^{2} }{(\lambda_{2}(t)^{2}+r^{2})^{3}} \left(\frac{r ||e_{1}-e_{2}||_{X}}{t^{2} \log^{b+1}(t) \sqrt{\log(\log(t))}}\right)dr \leq \frac{C ||e_{1}-e_{2}||_{X}}{t^{2}\log^{b+1}(t) \sqrt{\log(\log(t))}} \int_{0}^{\infty} \frac{u^{3} du}{(1+u^{2})^{3}}\\
&\leq \frac{C ||e_{1}-e_{2}||_{X}}{t^{2}\log^{b+1}(t) \sqrt{\log(\log(t))}}\end{split}\end{equation}

\begin{equation} \begin{split} &|\int_{0}^{\infty} \left(\frac{\cos(2Q_{1}(\frac{r}{\lambda_{2}(t)}))-1}{r^{2}\lambda_{2}(t)}\right) \phi_{0}(\frac{r}{\lambda_{2}(t)}) r \chi_{\geq 1}(\frac{2r}{\log^{N}(t)}) \left(-v_{1}^{\lambda_{1}}+v_{1}^{\lambda_{2}}-v_{3}^{\lambda_{1}}+v_{3}^{\lambda_{2}}\right) dr|\\
&\leq C \int_{\frac{\log^{N}(t)}{2}}^{\frac{t}{2}} \frac{\lambda_{2}(t)^{2} r^{2} }{r^{6}} \left(\frac{||e_{1}-e_{2}||_{X} r}{t^{2} \log^{b}(t) \sqrt{\log(\log(t))}}\right)dr + C \int_{\frac{t}{2}}^{\infty} \frac{\lambda_{2}(t)^{2} r^{2} }{r^{6}} \frac{||e_{1}-e_{2}||_{X}}{r \sqrt{\log(\log(t))} \log^{b}(t)}dr\\
&\leq \frac{C ||e_{1}-e_{2}||_{X}}{t^{2} \log^{3b +2N}(t) \sqrt{\log(\log(t))}}\end{split}\end{equation}

\begin{equation} \begin{split} &|\int_{0}^{\infty} \chi_{\geq 1}(\frac{2r}{\log^{N}(t)}) (F_{0,2}^{\lambda_{1}}-F_{0,2}^{\lambda_{2}})(t,r) \frac{\phi_{0}(\frac{r}{\lambda_{1}(t)})}{\lambda_{1}(t)} r dr|+|\int_{0}^{\infty} \chi_{\geq 1}(\frac{2r}{\log^{N}(t)}) F_{0,2}^{\lambda_{2}}(t,r)\left(\frac{\phi_{0}(\frac{r}{\lambda_{1}(t)})}{\lambda_{1}(t)} - \frac{\phi_{0}(\frac{r}{\lambda_{2}(t)})}{\lambda_{2}(t)}\right) r dr|\\
&\leq \frac{C ||e_{1}-e_{2}||_{X}}{t^{2} \log^{3b+1-2\alpha b}(t) \sqrt{\log(\log(t))} \log^{2N}(t)}\end{split}\end{equation}

Combining the above estimates, we firstly have
\begin{equation} |G(t,\lambda_{1}(t))-G(t,\lambda_{2}(t))| \leq \frac{C ||e_{1}-e_{2}||_{X}}{t^{2}\sqrt{\log(\log(t))} \log^{b+1}(t)}\end{equation}

Then, combining this with the earlier estimates of the terms in $RHS(e_{1},t)-RHS(e_{2},t)$ which don't involve $G$, we finally get: there exists $C_{lip}>0$ \emph{independent} of $T_{0}$, such that for $e_{1},e_{2} \in \overline{B_{1}(0)} \subset X$
\begin{equation} |RHS(e_{1},t)-RHS(e_{2},t)| \leq \frac{C_{lip} ||e_{1}-e_{2}||_{X}}{t^{2}\log^{b+1}(t) (\log(\log(t)))^{3/2}}\end{equation}
(Recall that $G$ appears in the expression of $RHS$ in the term $\frac{G(t,\lambda(t))}{\log(\lambda_{0}(t))}$). This concludes the proof of Proposition \ref{rhslipprop}.\\
\\
Using the $L^{1}$ estimate on the resolvent kernel, $r$, we get
\begin{equation}\begin{split} |(T(e_{1})-T(e_{2}))''(t)| &\leq \frac{|RHS(e_{1},t)-RHS(e_{2},t)|}{\alpha} + \frac{2}{\alpha}\sup_{z \geq t}\left(|RHS(e_{1},z)-RHS(e_{2},z)|\right)\\
&\leq \frac{3 C_{lip}}{\alpha} \frac{||e_{1}-e_{2}||_{X}}{t^{2}\log^{b+1}(t) (\log(\log(t)))^{3/2}} \end{split}\end{equation} 
Then, with the same procedure used when estimating $|T(e)'(t)|$ earlier, we get
\begin{equation} |(T(e_{1})-T(e_{2}))'(t)| \leq \frac{3 C_{lip}}{\alpha}\left(1+\frac{1}{100}\right) \frac{||e_{1}-e_{2}||_{X}}{t \log^{b+1}(t) (\log(\log(t)))^{3/2}}\end{equation}
\begin{equation} |(T(e_{1})-T(e_{2}))(t)| \leq \frac{3 C_{lip}}{\alpha}\left(1+\frac{1}{100}\right) \frac{||e_{1}-e_{2}||_{X}}{b \log^{b}(t) (\log(\log(t)))^{3/2}}\end{equation}
whence,
\begin{equation} ||T(e_{1})-T(e_{2})||_{X} \leq \frac{10 C_{lip}}{\alpha \log(\log(T_{0}))} ||e_{1}-e_{2}||_{X}, \quad e_{1},e_{2} \in \overline{B_{1}(0)}\end{equation}
Since $C_{lip}$ is independent of $T_{0}$, and the above estimates are valid for all $T_{0}$ satisfying $$T_{0} > 2e^{e^{1000(b+1)}} +T_{0,1}+T_{0,2}+T_{0,3}$$ if we further restrict $T_{0}$ to satisfy
\begin{equation} T_{0} > 2e^{e^{1000(b+1)}} +T_{0,1}+T_{0,2}+T_{0,3} + e^{e^{\frac{1500 (C_{lip}+1000(1+\frac{1}{b}))}{\alpha}}}\end{equation}
then, $T$ is a strict contraction on the complete metric space $\overline{B_{1}(0)} \subset X$. By Banach's fixed point theorem, $T$ has a fixed point, say $e_{0}$, in $\overline{B_{1}(0)} \subset X$. Then, if $$\lambda(t)=\lambda_{0}(t)+e_{0}(t)$$
we have $$\lambda \in C^{2}([T_{0},\infty)), \quad ||e_{0}||_{X} \leq 1$$
and $\lambda$ solves \eqref{modulationfinal}.

\subsubsection{Estimating $\lambda'''(t)$}
The main proposition of this section is:
\begin{proposition}$\lambda \in C^{3}([T_{0},\infty))$, with 
\begin{equation} |\lambda'''(t)| \leq \frac{C}{t^{3} \log^{b+1}(t)}, \quad t \geq T_{0}\end{equation}
In addition, we have
\begin{equation} \label{dtv1finalest} |\partial_{t}v_{1}(t,r)| \leq \begin{cases} \frac{C r}{t^{3} \log^{b}(t)}, \quad r\leq \frac{t}{2}\\
\frac{C}{r t \log^{b+1}(t)}, \quad r > \frac{t}{2}\end{cases}\end{equation}

\begin{equation} \label{dtv3finalest} |\partial_{t}v_{3}(t,r)| \leq \begin{cases} \frac{C r \log(\log(t))}{t^{3} \log^{b+1}(t)}, \quad r \leq \frac{t}{2}\\
\frac{C}{r t \log^{b+1}(t)}, \quad r > \frac{t}{2}\end{cases}\end{equation}

\begin{equation}\label{dtrv3finalest} |\partial_{tr} v_{3}(t,r)| \leq \frac{C}{t^{3} \log^{b}(t)}\end{equation}

\begin{equation}\label{dtrv1finalest} |\partial_{tr}v_{1}(t,r)| \leq \begin{cases} \frac{C}{t^{3}\log^{b}(t)}, \quad r \leq \frac{t}{2}\\
\frac{C}{r^{2}t \log^{b+1}(t)}, \quad r > \frac{t}{2}\end{cases}\end{equation} 

\begin{equation} \label{dtv4finalest} |\partial_{t}v_{4}(t,r)| \leq \begin{cases} \frac{C r}{t^{3} \log^{3b+2N-2}(t)}, \quad r \leq \frac{t}{2}\\
\frac{C}{t^{2} \log^{3b+2N-1}(t)}, \quad r > \frac{t}{2}\end{cases}\end{equation}

\end{proposition}
\begin{proof}
Since $\lambda(t)=\lambda_{0}(t)+e_{0}(t)$, it suffices to show that $e_{0} \in C^{3}([T_{0},\infty))$, and estimate $e_{0}'''(t)$. Recall that $\lambda$ solves \eqref{modulationfinal}, which can be re-written as
\begin{equation}\label{epppsetup}\lambda''(t) = \frac{RHS_{2}(t)}{g_{2}(t)}\end{equation}
where
$$g_{2}(t) =\frac{-2}{\lambda(t)} + \frac{4 \alpha \log(\lambda(t))}{\lambda(t)} \left(\frac{1}{-1+\lambda(t)^{2\alpha}}\right) -\frac{2(\lambda(t)^{-2\alpha}-1)}{\lambda(t)} g_{5}(t)$$
$$g_{5}(t) = \int_{0}^{\infty} \frac{\chi_{\geq 1}(\frac{2R \lambda(t)}{\log^{N}(t)}) \phi_{0}(R) R^{2} dR}{(R^{2}+1)(\lambda(t)^{-2\alpha}+R^{2})}$$
and
\begin{equation}\begin{split} RHS_{2}(t)&=-\frac{16}{\lambda(t)^{3}} \int_{t}^{\infty} ds \lambda''(s) \left(K_{1}(s-t,\lambda(t))+K(s-t,\lambda(t))\right) + 2 \frac{(\lambda'(t))^{2}}{\lambda(t)^{2}} \\
&+\frac{4b}{\lambda(t) t^{2} \log^{b}(t)} + E_{v_{2},ip}(t,\lambda(t))\\
&+ \langle \left(\frac{\cos(2Q_{\frac{1}{\lambda(t)}})-1}{r^{2}}\right)\left(v_{3}+(v_{4}+v_{5})\left(1-\chi_{\geq 1}(\frac{4r}{t})\right)-\chi_{\geq 1}\left(\frac{2r}{\log^{N}(t)}\right)\left(v_{1}+v_{2}+v_{3}\right)\right)\vert_{r=R\lambda(t)},\phi_{0}\rangle\\
&-4 \int_{0}^{\infty} \chi_{\geq 1}(\frac{2R \lambda(t)}{\log^{N}(t)}) \frac{(\lambda'(t))^{2} R^{2} \phi_{0}(R) dR}{\lambda(t)^{2} (R^{2}+1)^{2}}\end{split}\end{equation}
The point of this re-writing is that the right-hand side of \eqref{epppsetup} is in $C^{1}([T_{0},\infty))$, since $\lambda \in C^{2}([T_{0},\infty))$. So, $\lambda \in C^{3}([T_{0},\infty))$, and we have
\begin{equation} \lambda'''(t) = \partial_{t}\left(\frac{RHS_{2}(t)}{g_{2}(t)}\right)\end{equation}
We will first prove a preliminary estimate on $\lambda'''(t)$, which will then allow us to obtain a more useful formula for $e_{0}'''$. Using the formula for $g_{2}$, and estimates on $e_{0}$ from the fact that $e_{0} \in \overline{B_{1}(0)}$, we get
$$|\frac{1}{g_{2}(t)}| \leq \frac{C}{\log^{b}(t)\log(\log(t))}, \quad |\left(\frac{1}{g_{2}}\right)'(t)| \leq \frac{C}{t \log^{b+1}(t)\log(\log(t))}$$
and
$$|\lambda'''(t)| \leq \frac{C |RHS_{2}'(t)|}{\log^{b}(t)\log(\log(t))} + \frac{C |RHS_{2}(t)|}{t \log^{b+1}(t)\log(\log(t))}$$
For the preliminary estimate, it will suffice to estimate each term in $RHS_{2}$ seperately, despite the fact that there is some cancellation between some terms in $RHS_{2}$. When we prove the final estimate on $\lambda'''(t)$, we will take this cancellation into account.\\
\\
Using the same procedures and estimates used to estimate $G(t,\lambda_{0}(t)+f(t))$ for arbitrary $f \in \overline{B_{1}(0)} \subset X$,
we get
$$|RHS_{2}(t)| \leq \frac{C}{t^{2}}$$
Now, we estimate $|RHS_{2}'(t)|$. Using estimates on $\lambda^{(j)}(t), \quad j=0,1,2$, we get
\begin{equation}\label{RHS2prime}\begin{split} |RHS_{2}'(t)| &\leq |\partial_{t}\left(\frac{-16}{\lambda(t)^{3}} \int_{t}^{\infty} ds \lambda''(s) \left(K_{1}(s-t,\lambda(t))+K(s-t,\lambda(t))\right)\right)| \\
& + \frac{C}{t^{3} \log^{2}(t)} +\frac{C}{t^{3}} + |\partial_{t} E_{v_{2},ip}(t,\lambda(t))|\\
&+|\partial_{t}\left(\langle \left(\frac{\cos(2Q_{\frac{1}{\lambda(t)}})-1}{r^{2}}\right)\left(v_{3}+(v_{4}+v_{5})\left(1-\chi_{\geq 1}(\frac{4r}{t})\right)-\chi_{\geq 1}(\frac{2r}{\log^{N}(t)})\left(v_{1}+v_{2}+v_{3}\right)\right)\vert_{r=R\lambda(t)},\phi_{0}\rangle\right)|\end{split}\end{equation}
For the $E_{v_{2},ip}(t)$ term, we recall the definition in \eqref{ev2ip2}, and start with
\begin{equation}\begin{split} &\partial_{t}\left(\int_{0}^{\frac{1}{2}} d\xi \left(\chi_{\leq \frac{1}{4}}(\xi)-1\right) \frac{\sin(t\xi)}{t^{2}}\left(\frac{(b-1)}{\xi \log^{b}(\frac{1}{\xi})}+\frac{b(b-1)}{\xi \log^{b+1}(\frac{1}{\xi})}\right)\right)\\
&=\partial_{t}\left(\int_{0}^{\frac{t}{2}} du \left(\chi_{\leq \frac{1}{4}}(\frac{u}{t}) -1\right) \frac{\sin(u)}{t^{2}} \left(\frac{b-1}{u \log^{b}(\frac{t}{u})} + \frac{b(b-1)}{u \log^{b+1}(\frac{t}{u})}\right)\right)
\end{split}\end{equation}
After taking the derivative, and integrating by parts once in the remaining integrals, we get
\begin{equation}\begin{split} &\partial_{t}\left(\int_{0}^{\frac{1}{2}} d\xi \left(\chi_{\leq \frac{1}{4}}(\xi)-1\right) \frac{\sin(t\xi)}{t^{2}}\left(\frac{(b-1)}{\xi \log^{b}(\frac{1}{\xi})}+\frac{b(b-1)}{\xi \log^{b+1}(\frac{1}{\xi})}\right)\right)\\
&=\frac{- \sin(\frac{t}{2})}{t^{3}} \left(\frac{b-1}{\log^{b}(2)}+\frac{b(b-1)}{\log^{b+1}(2)}\right) + \text{Err}\end{split}\end{equation}
where
$$|\text{Err}| \leq \frac{C}{t^{4}}$$
On the other hand,
\begin{equation} \begin{split} \partial_{t}\left(\int_{0}^{\frac{t}{2}} \frac{\sin(u)b(b-1) du}{t^{2}u\log^{b+1}(\frac{t}{u})}\right)&= \frac{ \sin(\frac{t}{2}) b(b-1)}{t^{3}\log^{b+1}(2)} + \int_{0}^{\frac{t}{2}} \frac{\sin(u) b(b-1)}{u} \left(\frac{-2}{t^{3}\log^{b+1}(\frac{t}{u})}-\frac{(b+1)}{\log^{b+2}(\frac{t}{u}) t^{3}}\right)du\end{split}\end{equation}
So,
\begin{equation}\label{partofdtev2ip}\begin{split}E_{t,v_{2},ip,1}:= \partial_{t}&\left(2 c_{b}\int_{0}^{\frac{1}{2}} d\xi \left(\chi_{\leq \frac{1}{4}}(\xi)-1\right) \frac{\sin(t\xi)}{t^{2}} \left(\frac{b-1}{\xi \log^{b}(\frac{1}{\xi})} + \frac{b(b-1)}{\xi \log^{b+1}(\frac{1}{\xi})}\right)\right.\\
&+\left.2 c_{b} \left(\int_{0}^{\frac{t}{2}} \frac{\sin(u)(b-1) du}{t^{2}u \log^{b}(\frac{t}{u})} - \frac{(b-1) \pi}{2 t^{2}\log^{b}(t)} + \int_{0}^{\frac{t}{2}} \frac{\sin(u) b(b-1) du}{t^{2} u \log^{b+1}(\frac{t}{u})}\right)\right)\\
&=2 c_{b} \text{Err} + 2 c_{b}\left(\int_{0}^{\frac{t}{2}} du \sin(u) (b-1)\left(\frac{-2}{t^{3}u \log^{b}(\frac{t}{u})}-\frac{b}{\log^{b+1}(\frac{t}{u}) t^{3} u}\right)\right.\\
& +\left. \int_{0}^{\frac{t}{2}} \frac{du}{u} \sin(u) b(b-1)\left(\frac{-2}{t^{3}\log^{b+1}(t)} - \frac{(b+1)}{t^{3}\log^{b+2}(\frac{t}{u})}\right)\right)\\
&+2 c_{b} \left(\frac{(b-1)\pi}{t^{3}\log^{b}(t)} + \frac{b(b-1)\pi}{2 t^{3}\log^{b+1}(t)}\right)\end{split}\end{equation}
The asymptotics of the integrals in lines 3 and 4 of the above expression were computed previously, in the section which constructed $v_{2}$. The following is the most important asymptotic, which shows that there is some cancellation between the integral in line 3 and the terms on line 5.
\begin{equation} -2 \int_{0}^{\frac{t}{2}} \frac{\sin(u)(b-1) du}{t^{3} u \log^{b}(\frac{t}{u})} = \frac{-2 (b-1)}{t^{3}} \frac{\pi}{2} \left(\frac{1}{\log^{b}(t)} + O\left(\frac{1}{\log^{b+1}(t)}\right)\right)\end{equation}
In total, we get 
$$|E_{t,v_{2},ip,1}(t)| \leq \frac{C}{t^{3}\log^{b+1}(t)}$$
Next, we estimate
\begin{equation} \partial_{t}\left(2 c_{b}\lambda(t) \int_{0}^{\infty} d\xi \frac{\sin(t\xi)}{t^{2}} \psi_{v_{2}}(\xi,\lambda(t))\right)\end{equation}
Here, we use estimates on the derivatives of $\psi_{v_{2}}$, proven earlier, when showing that the map $T$ was a contraction, as well as \eqref{kjestimates}, which shows (among other things) that
$$|\partial_{\xi}^{2} \psi_{v_{2}}(\xi,\lambda(t))| \leq C \frac{p(\xi)}{\lambda(t)}, \quad p \in C^{\infty}_{c}([\frac{1}{8},\frac{1}{4}])$$
This gives us
\begin{equation}|\partial_{t}\left(2 c_{b}\lambda(t) \int_{0}^{\infty} d\xi \frac{\sin(t\xi)}{t^{2}} \psi_{v_{2}}(\xi,\lambda(t))\right)| \leq \frac{C}{t^{4}}\end{equation}
Similarly, we use estimates on derivatives of $F_{v_{2}}$ which were proven while showing that $T$ is a contraction, as well as  \eqref{kjestimates}, and asymptotics of $K_{1}$, to get
\begin{equation} |\partial_{t}\left(2 c_{b}\lambda(t) \int_{0}^{\infty} d\xi \chi_{\leq \frac{1}{4}}(\xi) \frac{\sin(t\xi)}{t^{2}} F_{v_{2}}(\xi,\lambda(t))\right)| \leq \frac{C \log(\log(t))}{t^{4}\log^{2b}(t)}\end{equation}
In total, we have
\begin{equation} |\partial_{t}\left(\lambda(t) E_{v_{2},ip}(t,\lambda(t))\right)| \leq \frac{C}{t^{3}\log^{b+1}(t)}\end{equation}
By the same procedure, this estimate is also true for the case $b=1$.\\
\\
Some of the terms in \eqref{RHS2prime} which remain to be estimated will be estimated in two different ways. One way of estimating these terms will give rise to preliminary estimates on $e'''$, which will then allow us to obtain stronger estimates. For this preliminary estimate, we will start with the first line of \eqref{RHS2prime}:
\begin{equation} \partial_{t}\left(\frac{16}{\lambda(t)^{3}} \int_{t}^{\infty} ds \lambda''(s)\left(K(s-t,\lambda(t))+K_{1}(s-t,\lambda(t))\right)\right)=I+II+III+IV\end{equation}
where
\begin{equation}\begin{split}&I=\frac{16}{\lambda(t)^{3}} \int_{t}^{\infty} ds \lambda''(s)\partial_{2}K(s-t,\lambda(t))\lambda'(t)\end{split}\end{equation}

\begin{equation}II=\frac{16}{\lambda(t)^{3}} \int_{t}^{\infty} ds \lambda''(s)\partial_{2}K_{1}(s-t,\lambda(t))\lambda'(t)\end{equation}

\begin{equation}\label{Eterm}\begin{split} III&= -\frac{16}{\lambda(t)^{3}} \int_{t}^{\infty} ds \lambda''(s)\left(\partial_{1}K(s-t,\lambda(t))+\partial_{1}K_{1}(s-t,\lambda(t))\right)\end{split}\end{equation}

\begin{equation}\begin{split}&IV=\frac{-48 \lambda'(t)}{\lambda(t)^{4}}\int_{t}^{\infty} ds \lambda''(s) \left(K(s-t,\lambda(t))+K_{1}(s-t,\lambda(t))\right)\\
&|IV| \leq \frac{C}{t^{3}\log(t)}\end{split}\end{equation}
In order to estimate $I$, let us start with
\begin{equation}\begin{split} \partial_{2}K(x,\lambda(t)) = \int_{0}^{\infty} dR \frac{R}{(1+R^{2})^{3}} \int_{0}^{x} \rho d\rho \left(\frac{1}{\sqrt{x^{2}-\rho^{2}}}-\frac{1}{x}\right) \left(\frac{4\lambda(t)R^{2}(1+\rho^{2}+\lambda(t)^{2}R^{2})}{(4\lambda(t)^{2}R^{2}+(1+\rho^{2}-R^{2}\lambda(t)^{2})^{2})^{3/2}}\right)\end{split}\end{equation}
Then, the same procedure used to obtain \eqref{d2kintest} gives
 \begin{equation}\begin{split} &|\int_{t}^{\infty} ds \lambda''(s) \partial_{2}K(s-t,\lambda(t)) \lambda'(t)| \leq \frac{C}{t^{3} \log^{2b+2}(t)} \int_{t}^{\infty} ds |\partial_{2}K(s-t,\lambda(t))|\\
&\leq \frac{C}{t^{3}\log^{3b+2}(t)}\end{split}\end{equation} 
So, $$|I| \leq \frac{C}{t^{3}\log^{2}(t)}$$
Next, we recall the definition of $K_{1}$:
\begin{equation} K_{1}(x,\lambda(t)) = \int_{0}^{\infty} \frac{r dr}{\lambda(t)^{2}(1+\frac{r^{2}}{\lambda(t)^{2}})^{3}} \int_{0}^{x} \frac{\rho d\rho}{x}\left(1+\frac{r^{2}-1-\rho^{2}}{\sqrt{(r^{2}-1-\rho^{2})^{2}+4r^{2}}}\right)\end{equation}
So,
\begin{equation} |\partial_{2} K_{1}(x,\lambda(t)) \lambda'(t)| \leq C \lambda(t)^{3}|\lambda'(t)| \int_{0}^{\infty} dr \frac{r}{(r^{2}+\lambda(t)^{2})^{3}} \int_{0}^{x} \frac{\rho d\rho}{x} \left(1+\frac{r^{2}-1-\rho^{2}}{\sqrt{(r^{2}-1-\rho^{2})^{2}+4r^{2}}}\right)\end{equation}
and we get
\begin{equation} \begin{split} &|II|\leq |\frac{16}{\lambda(t)^{3}} \int_{t}^{\infty} dx \lambda''(x) \partial_{2}K_{1}(x-t,\lambda(t)) \lambda'(t)| \\
&\leq C |\lambda'(t)| \int_{t}^{\infty} dx |\lambda''(x)| \int_{0}^{\infty} \frac{r dr}{(r^{2}+\lambda(t)^{2})^{3}} \begin{cases} \frac{1}{(x-t)} \int_{0}^{x-t} 2 \rho d\rho, \quad x-t \leq 1\\
\frac{1}{(x-t)} \int_{0}^{\infty} \rho d\rho \left(1+\frac{r^{2}-1-\rho^{2}}{\sqrt{(r^{2}-1-\rho^{2})^{2}+4r^{2}}}\right), \quad x-t \geq 1\end{cases}\\
&\leq C |\lambda'(t)| \int_{t}^{\infty} dx |\lambda''(x)| \int_{0}^{\infty} \frac{r dr}{(r^{2}+\lambda(t)^{2})^{3}} \begin{cases} (x-t), \quad x-t \leq 1\\
\frac{r^{2}}{(x-t)}, \quad x-t \geq 1\end{cases}\\
&\leq C  |\lambda'(t)| \left(\int_{t}^{t+1}\frac{dx}{x^{2}\log^{b+1}(x)} \frac{(x-t)}{\lambda(t)^{4}} + \int_{t+1}^{\infty} \frac{dx}{x^{2}\log^{b+1}(x) (x-t)} \int_{0}^{\infty} \frac{r^{3} dr}{(r^{2}+\lambda(t)^{2})^{3}}\right)\\
&\leq \frac{C}{t^{3}\log^{1-2b}(t)}\end{split}\end{equation}

We will now treat the term $III$. Let $L=K+K_{1}$. Then, we have
\begin{equation} \partial_{1}L(x,\lambda(t)) = L_{a}(x,\lambda(t))+L_{b}(x,\lambda(t))\end{equation}
with
\begin{equation}\begin{split} L_{a}(x,\lambda(t)) &= \int_{0}^{\infty} dR \frac{R}{(1+R^{2})^{3}} \int_{0}^{1} \frac{u du}{\sqrt{1-u^{2}}} \left(1+\frac{R^{2}\lambda(t)^{2}-1-u^{2}x^{2}}{\sqrt{(R^{2}\lambda(t)^{2}-1-u^{2}x^{2})^{2}+4R^{2}\lambda(t)^{2}}}\right)\\
&=\int_{0}^{\infty} dR \frac{R}{(1+R^{2})^{3}} \int_{0}^{x} \frac{\rho d\rho}{x \sqrt{x^{2}-\rho^{2}}}\left(1+\frac{R^{2}\lambda(t)^{2}-1-\rho^{2}}{\sqrt{(R^{2}\lambda(t)^{2}-1-\rho^{2})^{2}+4R^{2}\lambda(t)^{2}}}\right)\end{split}\end{equation}

\begin{equation}\begin{split} L_{b}(x,\lambda(t))&= \int_{0}^{\infty} dR \frac{R}{(1+R^{2})^{3}} \int_{0}^{1} \frac{u du x}{\sqrt{1-u^{2}}}\left(\frac{-8 R^{2}\lambda(t)^{2}u^{2}x}{(4R^{2}\lambda(t)^{2}+(1-R^{2}\lambda(t)^{2}+u^{2}x^{2})^{2})^{3/2}}\right)\\
&=\int_{0}^{\infty} dR \frac{R}{(1+R^{2})^{3}} \int_{0}^{x} \frac{\rho d\rho}{\sqrt{x^{2}-\rho^{2}}} \left(\frac{-8 R^{2}\lambda(t)^{2}\rho^{2}}{x(4R^{2}\lambda(t)^{2}+(1-R^{2}\lambda(t)^{2}+\rho^{2})^{2})^{3/2}}\right)\end{split}\end{equation}

Now, \begin{equation}\label{laest}\begin{split} &\int_{t}^{\infty} ds |L_{a}(s-t,\lambda(t))| \\
&\leq \int_{0}^{\infty} dR \frac{R}{(1+R^{2})^{3}} \int_{0}^{\infty} \rho d\rho \left(1+\frac{R^{2}\lambda(t)^{2}-1-\rho^{2}}{\sqrt{(R^{2}\lambda(t)^{2}-1-\rho^{2})^{2}+4R^{2}\lambda(t)^{2}}}\right) \int_{\rho+t}^{\infty} \frac{ds}{(s-t)\sqrt{(s-t)^{2}-\rho^{2}}}\\
&\leq C \int_{0}^{\infty} dR \frac{R}{(1+R^{2})^{3}} \int_{0}^{\infty}  d\rho \left(1+\frac{R^{2}\lambda(t)^{2}-1-\rho^{2}}{\sqrt{(R^{2}\lambda(t)^{2}-1-\rho^{2})^{2}+4R^{2}\lambda(t)^{2}}}\right)\\
&\leq C \int_{0}^{\infty} dR \frac{R}{(1+R^{2})^{3}} \int_{2(R\lambda(t)+1)}^{\infty}  d\rho \left(1+\frac{R^{2}\lambda(t)^{2}-1-\rho^{2}}{\sqrt{(R^{2}\lambda(t)^{2}-1-\rho^{2})^{2}+4R^{2}\lambda(t)^{2}}}\right)\\
&+C \int_{0}^{\infty} dR \frac{R}{(1+R^{2})^{3}} \int_{0}^{2(R\lambda(t)+1)}  d\rho \left(1+\frac{R^{2}\lambda(t)^{2}-1-\rho^{2}}{\sqrt{(R^{2}\lambda(t)^{2}-1-\rho^{2})^{2}+4R^{2}\lambda(t)^{2}}}\right) \end{split}\end{equation}
\\
\\
When $\rho > 2(R\lambda(t)+1)$, we have
\begin{equation}\begin{split} |1+\frac{R^{2}\lambda(t)^{2}-1-\rho^{2}}{\sqrt{(R^{2}\lambda(t)^{2}-1-\rho^{2})^{2}+4R^{2}\lambda(t)^{2}}}| &= |1+\frac{(-1-\frac{1}{\rho^{2}}+\frac{R^{2}\lambda(t)^{2}}{\rho^{2}})}{\sqrt{1-2(\frac{R^{2}\lambda(t)^{2}}{\rho^{2}}-\frac{1}{\rho^{2}})+\frac{(R^{2}\lambda(t)^{2}+1)^{2}}{\rho^{4}}}}|\\
&\leq C \frac{(R\lambda(t)+1)^{2}}{\rho^{2}}\end{split}\end{equation}

If $\rho < 2(R\lambda(t)+1)$, then, we use
\begin{equation} |1+\frac{R^{2}\lambda(t)^{2}-1-\rho^{2}}{\sqrt{(R^{2}\lambda(t)^{2}-1-\rho^{2})^{2}+4R^{2}\lambda(t)^{2}}}| \leq 2\end{equation}\\
\\
to continue the estimate in \eqref{laest}
\begin{equation} \begin{split} &\int_{t}^{\infty} ds |L_{a}(s-t,\lambda(t))|\\
&\leq C \int_{0}^{\infty} dR \frac{R}{(1+R^{2})^{3}} \int_{2(R\lambda(t)+1)}^{\infty}  d\rho \left(1+\frac{R^{2}\lambda(t)^{2}-1-\rho^{2}}{\sqrt{(R^{2}\lambda(t)^{2}-1-\rho^{2})^{2}+4R^{2}\lambda(t)^{2}}}\right)\\
&+C \int_{0}^{\infty} dR \frac{R}{(1+R^{2})^{3}} \int_{0}^{2(R\lambda(t)+1)}  d\rho \left(1+\frac{R^{2}\lambda(t)^{2}-1-\rho^{2}}{\sqrt{(R^{2}\lambda(t)^{2}-1-\rho^{2})^{2}+4R^{2}\lambda(t)^{2}}}\right)\\
&\leq C \int_{0}^{\infty} dR \frac{R}{(1+R^{2})^{3}} \int_{2(R\lambda(t)+1)}^{\infty} d\rho \frac{(R\lambda(t)+1)^{2}}{\rho^{2}} + C\int_{0}^{\infty} dR \frac{R}{(1+R^{2})^{3}} (R\lambda(t)+1)\\
&\leq C\end{split}\end{equation}\\
\\
Similarly, we have
\begin{equation}\label{Lbest}\begin{split} &|\int_{t}^{\infty} ds L_{b}(s-t,\lambda(t))|\\
&\leq \int_{0}^{\infty} dR \frac{R}{(1+R^{2})^{3}} \int_{0}^{\infty} \rho d\rho \left(\frac{8 R^{2}\lambda(t)^{2}\rho^{2}}{(4R^{2}\lambda(t)^{2}+(1-R^{2}\lambda(t)^{2}+\rho^{2})^{2})^{3/2}}\right) \int_{\rho+t}^{\infty} \frac{ds}{\sqrt{(s-t)^{2}-\rho^{2}}} \frac{1}{(s-t)}\\
&\leq C \int_{0}^{\infty} dR \frac{R}{(1+R^{2})^{3}} \int_{0}^{\infty} d\rho \frac{R^{2}\lambda(t)^{2}\rho^{2}}{(4R^{2}\lambda(t)^{2}+(1-R^{2}\lambda(t)^{2}+\rho^{2})^{2})^{3/2}}\\
&\leq C \int_{0}^{\infty} dR \frac{R}{(1+R^{2})^{3}} \int_{2(R\lambda(t)+1)}^{\infty} d\rho \frac{R^{2}\lambda(t)^{2}\rho^{2}}{(4R^{2}\lambda(t)^{2}+(1-R^{2}\lambda(t)^{2}+\rho^{2})^{2})^{3/2}}\\
&+ C \int_{0}^{\infty} dR \frac{R}{(1+R^{2})^{3}} \int_{0}^{2(R\lambda(t)+1)} d\rho \frac{R^{2}\lambda(t)^{2}\rho^{2}}{(4R^{2}\lambda(t)^{2}+(1-R^{2}\lambda(t)^{2}+\rho^{2})^{2})^{3/2}} \end{split}\end{equation}

When $\rho>2(R\lambda(t)+1)$, we have
\begin{equation} \frac{1}{(4R^{2}\lambda(t)^{2}+(1-R^{2}\lambda(t)^{2}+\rho^{2})^{2})^{3/2}} \leq \frac{C}{\rho^{6}}\end{equation}

When $\rho < 2(R\lambda(t)+1)$, we use
\begin{equation} \frac{1}{(4R^{2}\lambda(t)^{2}+(1-R^{2}\lambda(t)^{2}+\rho^{2})^{2})^{3/2}} \leq \frac{C}{R^{3}\lambda(t)^{3}}\end{equation}\\
\\
and then continue the estimate from \eqref{Lbest}:
\begin{equation} \begin{split} &|\int_{t}^{\infty} ds L_{b}(s-t,\lambda(t))|\\
&\leq C \int_{0}^{\infty} dR \frac{R}{(1+R^{2})^{3}} \int_{2(R\lambda(t)+1)}^{\infty} d\rho \frac{R^{2}\lambda(t)^{2}\rho^{2}}{(4R^{2}\lambda(t)^{2}+(1-R^{2}\lambda(t)^{2}+\rho^{2})^{2})^{3/2}}\\
&+ C \int_{0}^{\infty} dR \frac{R}{(1+R^{2})^{3}} \int_{0}^{2(R\lambda(t)+1)} d\rho \frac{R^{2}\lambda(t)^{2}\rho^{2}}{(4R^{2}\lambda(t)^{2}+(1-R^{2}\lambda(t)^{2}+\rho^{2})^{2})^{3/2}}\\
&\leq C \int_{0}^{\infty} dR \frac{R}{(1+R^{2})^{3}} R^{2}\lambda(t)^{2} \int_{2(R\lambda(t)+1)}^{\infty} d\rho \frac{\rho^{2}}{\rho^{6}} + C \int_{0}^{\infty} dR \frac{R}{(1+R^{2})^{3}} \int_{0}^{2(R\lambda(t)+1)} d\rho \frac{R^{2}\lambda(t)^{2}\rho^{2}}{R^{3}\lambda(t)^{3}}\\
&\leq \frac{C}{\lambda(t)}\end{split}\end{equation}
Then, we combine the above estimates to get 
\begin{equation} \int_{t}^{\infty} ds |\partial_{1}L(s-t,\lambda(t))| \leq \frac{C}{\lambda(t)}\end{equation}
So,
\begin{equation} |III| \leq C \frac{\log^{3b-1}(t)}{t^{2}}\end{equation}
We finally conclude that
\begin{equation}|I|+|II|+|III|+|IV| \leq \frac{C \log^{3b-1}(t)}{t^{2}}\end{equation}
Next, we estimate the last line of \eqref{RHS2prime}. First, we note that
\begin{equation} \begin{split} \partial_{t} \langle \left(\frac{\cos(2Q_{\frac{1}{\lambda(t)}})-1}{r^{2}}\right) f\vert_{r = R\lambda(t)},\phi_{0}\rangle &= \int_{0}^{\infty} \partial_{t}\left(\left(\frac{\cos(2Q_{1}(\frac{r}{\lambda(t)})-1}{r^{2}}\right)\frac{\phi_{0}(\frac{r}{\lambda(t)})}{\lambda(t)^{2}}\right) f(t,r) r dr\\
&+\int_{0}^{\infty} \left(\frac{\cos(2Q_{1}(\frac{r}{\lambda(t)}))-1}{r^{2}}\right) \frac{\phi_{0}(\frac{r}{\lambda(t)})}{\lambda(t)^{2}} \partial_{t}f(t,r) r dr\end{split}\end{equation}
which gives
\begin{equation}\label{dtipest} \begin{split} |\partial_{t}\langle \left(\frac{\cos(2Q_{\frac{1}{\lambda(t)}})-1}{r^{2}}\right) f\vert_{r = R\lambda(t)},\phi_{0}\rangle| &\leq C \int_{0}^{\infty} \frac{r |\lambda'(t)|}{(r^{2}+\lambda(t)^{2})^{3}} |f(t,r)| r dr + C \int_{0}^{\infty} \frac{r \lambda(t)}{(r^{2}+\lambda(t)^{2})^{3}} |\partial_{t}f(t,r)| r dr\end{split}\end{equation}

We now proceed to estimate $\partial_{t}v_{k}, \quad k=1,3,4,5$. We first prove a preliminary estimate on these quantities in order to obtain the preliminary estimate on $\lambda'''$. Once this is done, we can improve our estimates for $\partial_{t}v_{k}$.
\begin{lemma}[Preliminary Estimates on $\partial_{t}v_{k}, \quad k=1,3,4,5$]
We have the following \emph{preliminary} estimates on $\partial_{t}v_{k}, \quad k=1,3,4,5$.\\
\begin{equation} \label{dtv1prelimest} |\partial_{t}v_{1}(t,r)| \leq \frac{C r}{t^{2} \log^{b+1}(t)}\end{equation}
\begin{equation} \label{dtv3prelimest}|\partial_{t}v_{3}(t,r)| \leq \frac{C r}{t^{2} \log^{1+b\alpha}(t)}\end{equation}
\begin{equation} \label{dtv4prelimest} |\partial_{t}v_{4}(t,r)| \leq \frac{C}{t^{2} \log^{3b+2N-1}(t)}, \quad r \leq \frac{t}{2} \end{equation}
\begin{equation} \label{dtv5prelimest} |\partial_{t}v_{5}(t,r)| \leq \frac{C}{t^{7/2} \log^{3b-3+\frac{5N}{2}}(t)}+\frac{C r \log(t)}{t^{4} \log^{b}(t)}, \quad r \leq \frac{t}{2}\end{equation}

\end{lemma} 
\begin{proof}
For $\partial_{t}v_{3}$, we have
\begin{equation}\label{dtv3prelim}\begin{split} \partial_{t}v_{3}(t,r) &= \frac{1}{r} \int_{t}^{\infty} ds \int_{0}^{s-t} \frac{\rho d\rho}{(s-t)\sqrt{(s-t)^{2}-\rho^{2}}} \lambda''(s) \left(\frac{-1-\rho^{2}+r^{2}}{\sqrt{(1+\rho^{2}-r^{2})^{2}+4r^{2}}} + F_{3}(r,\rho,\lambda(s))\right)\\
&-\frac{1}{r} \int_{t}^{\infty} ds \int_{0}^{s-t} \frac{\rho d\rho}{(s-t)\sqrt{(s-t)^{2}-\rho^{2}}}\lambda''(s) \left(\frac{8 \rho^{2} r^{2}}{(4r^{2}+(1+\rho^{2}-r^{2})^{2})^{3/2}} \right.\\
&+\left. \frac{-8 \lambda(s)^{4\alpha -4}\rho^{2} r^{2}}{(4 \lambda(s)^{2\alpha -2} r^{2}+(1+\lambda(s)^{2\alpha -2}(\rho^{2}-r^{2}))^{2})^{3/2}}\right)\end{split}\end{equation}

For the first line on the right-hand side of \eqref{dtv3prelim}, we first note that 
\begin{equation}\begin{split} |1-F_{3}(r,\rho,\lambda(s))| = 1-F_{3}(r,\rho,\lambda(s)) = F_{3}(0,\rho,\lambda(s))-F_{3}(r,\rho,\lambda(s))\end{split}\end{equation}
So,
\begin{equation}\begin{split} |1-F_{3}(r,\rho,\lambda(s))| &\leq C r \int_{0}^{1} \frac{r_{\sigma} \lambda(s)^{2\alpha-2}(1+\lambda(s)^{2\alpha-2}(\rho^{2}+r_{\sigma}^{2}))}{(1+2(\rho^{2}+r_{\sigma}^{2})\lambda(s)^{2\alpha-2}+(\rho^{2}-r_{\sigma}^{2})^{2}\lambda(s)^{4\alpha-4})^{3/2}}d\sigma\end{split}\end{equation}
where
$$r_{\sigma}=(1-\sigma)r$$ 
Then,  because $\lambda'(x) \leq 0, \quad x \geq T_{0}$, we have $\lambda(s)^{2\alpha -2} \geq \lambda(t)^{2\alpha -2}, \quad s \geq t$. So,
\begin{equation} \begin{split} &|\frac{1}{r} \int_{t}^{\infty} ds \int_{0}^{s-t} \frac{\rho d\rho \lambda''(s)}{(s-t)\sqrt{(s-t)^{2}-\rho^{2}}} \left(\frac{-1-\rho^{2}+r^{2}}{\sqrt{(1+\rho^{2}-r^{2})^{2}+4r^{2}}} + 1 + F_{3}(r,\rho,\lambda(s))-1 \right)|\\
&\leq C r \sup_{x \geq t}|\lambda''(x)| \int_{0}^{\infty}d\rho  \int_{0}^{1} \frac{d\sigma}{(1+2(\rho^{2}+r_{\sigma}^{2})+(\rho^{2}-r_{\sigma}^{2})^{2})^{1/2}} \\
&+ C r \sup_{x \geq t} \left(|\lambda''(x)| \lambda(x)^{4\alpha -4}\right) \int_{0}^{\infty} d\rho   \int_{0}^{1}   \frac{\left(\lambda(t)^{2-2\alpha}+\rho^{2}+r_{\sigma}^{2}\right)}{(1+2(\rho^{2}+r_{\sigma}^{2})\lambda(t)^{2\alpha-2}+(\rho^{2}-r_{\sigma}^{2})^{2}\lambda(t)^{4\alpha-4})^{3/2}}d\sigma\end{split}\end{equation}
and where we obtained the second and third lines by switching the $s$ and $\rho$ integration order, and using
$$\int_{\rho+t}^{\infty} \frac{ds}{(s-t) \sqrt{(s-t)^{2}-\rho^{2}}}\leq \frac{C}{\rho}$$
So, it suffices to estimate
\begin{equation}\label{B3def} \begin{split} &B_{3}(y) = \int_{0}^{\infty} \frac{d\rho}{(1+2(\rho^{2}+y^{2})+(\rho^{2}-y^{2})^{2})^{1/2}} \leq C \int_{0}^{2y}\frac{d\rho}{y} + \int_{2y}^{\infty} \frac{d\rho}{(1+\rho^{2})}\leq C\end{split}\end{equation}
This gives
\begin{equation} \begin{split} &|\frac{1}{r} \int_{t}^{\infty} ds \int_{0}^{s-t} \frac{\rho d\rho \lambda''(s)}{(s-t)\sqrt{(s-t)^{2}-\rho^{2}}} \left(\frac{-1-\rho^{2}+r^{2}}{\sqrt{(1+\rho^{2}-r^{2})^{2}+4r^{2}}} + 1 + F_{3}(r,\rho,\lambda(s))-1 \right)|\\
&\leq C \frac{r}{t^{2} \log^{b+1}(t)} + C \frac{r}{t^{2} \log^{1+b\alpha}(t)}\end{split}\end{equation}
We then consider the second line on the right-hand side of \eqref{dtv3prelim}, and use the same procedure as above:
\begin{equation}\label{dtv3prelimline2} \begin{split} &|\frac{-1}{r} \int_{t}^{\infty} ds \int_{0}^{s-t} \frac{\rho d\rho}{(s-t)} \frac{\lambda''(s)}{\sqrt{(s-t)^{2}-\rho^{2}}} \left(\frac{8 \rho^{2} r^{2}}{(4r^{2}+(1+\rho^{2}-r^{2})^{2})^{3/2}}-\frac{8\lambda(s)^{4\alpha-4}\rho^{2}r^{2}}{(4\lambda(s)^{2\alpha-2}r^{2}+(1+\lambda(s)^{2\alpha-2}(\rho^{2}-r^{2}))^{2})^{3/2}}\right)|\\
&\leq C r \int_{0}^{\infty} d\rho \left(\frac{\rho^{2} \left(\sup_{x \geq t}|\lambda''(x)|\right)}{(4r^{2}+(1+\rho^{2}-r^{2})^{2})^{3/2}} + \frac{\rho^{2} \sup_{x \geq t}\left(|\lambda''(x)| \lambda(x)^{4\alpha -4}\right)}{(1+2\lambda(t)^{2\alpha-2}(r^{2}+\rho^{2})+\lambda(t)^{4\alpha-4}(\rho^{2}-r^{2})^{2})^{3/2}}\right)\\
&\leq C r \int_{0}^{\infty} \frac{1}{t^{2} \log^{b+1}(t)(1+2(\rho^{2}+r^{2})+(\rho^{2}-r^{2})^{2})^{1/2}}d\rho \\
&+ C r \int_{0}^{\infty} \frac{\lambda(t)^{3-3\alpha}dq }{t^{2} \log^{b+1}(t) \log^{-4b+4\alpha b}(t) (1+2(q^{2} +r^{2} \lambda(t)^{2\alpha -2})+(q^{2}-r^{2}\lambda(t)^{2\alpha-2})^{2})^{1/2}}\\
&\leq C r \frac{B_{3}(r)}{t^{2} \log^{b+1}(t)} + C r \frac{B_{3}(r \lambda(t)^{\alpha -1})}{t^{2} \log^{b+1}(t)\log^{-b+\alpha b}(t)}\end{split}\end{equation}

This gives
\begin{equation}\begin{split} &|\frac{-1}{r} \int_{t}^{\infty} ds \int_{0}^{s-t} \frac{\rho d\rho}{(s-t)} \frac{\lambda''(s)}{\sqrt{(s-t)^{2}-\rho^{2}}} \left(\frac{8 \rho^{2} r^{2}}{(4r^{2}+(1+\rho^{2}-r^{2})^{2})^{3/2}}-\frac{8\lambda(s)^{4\alpha-4}\rho^{2}r^{2}}{(4\lambda(s)^{2\alpha-2}r^{2}+(1+\lambda(s)^{2\alpha-2}(\rho^{2}-r^{2}))^{2})^{3/2}}\right)|\\
&\leq C \frac{r}{t^{2} \log^{b+1}(t)} + C \frac{r}{t^{2} \log^{1+b\alpha}(t)}\end{split}\end{equation}

In total, we get \eqref{dtv3prelimest}.\\
\\
Next, we have to estimate $\partial_{t}v_{4}$ in the region $r \leq \frac{t}{2}$. Again, we write the formula for $v_{4}$ as
\begin{equation}\label{v4forprelim} v_{4}(t,r) = \frac{-1}{2\pi} \int_{t}^{\infty} ds \int_{0}^{s-t} \frac{\rho d\rho}{\sqrt{(s-t)^{2}-\rho^{2}}} \int_{0}^{2 \pi} d\theta \frac{v_{4,c}(s,\sqrt{r^{2}+2 r \rho \cos(\theta) + \rho^{2}})}{\sqrt{r^{2}+\rho^{2}+2 r \rho \cos(\theta)}} (r+\rho \cos(\theta))\end{equation}
Like done previously for $v_{4}$, we introduce $x =r \textbf{e}_{1} \in \mathbb{R}^{2}$ to ease notation, and get
\begin{equation}\label{dtv4forprelim}\begin{split} \partial_{t}v_{4}(t,r) &= -\int_{t}^{\infty} ds \frac{v_{4,s}(t,r)}{(s-t)} -\frac{1}{2\pi} \int_{t}^{\infty} ds \int_{B_{s-t}(0)} \frac{dA(y)}{\sqrt{(s-t)^{2}-|y|^{2}}}\text{integrand}_{v_{4,1}}\end{split}\end{equation}
where we recall

\begin{equation} v_{4,s}(t,r) = \frac{-1}{2\pi} \int_{B_{s-t}(0)} \frac{dA(y)}{\sqrt{(s-t)^{2}-|y|^{2}}}\frac{v_{4,c}(s,|x+y|)\left((x+y)\cdot \widehat{x}\right)}{|x+y|}\end{equation}

\begin{equation}\begin{split} \text{integrand}_{v_{4,1}} &= \frac{-\partial_{2}v_{4,c}(s,|x+y|)}{|x+y|^{2}(s-t)} \left((x+y)\cdot y\right)\left(\hat{x}\cdot(x+y)\right)\\
&+\frac{v_{4,c}(s,|x+y|)}{(s-t)|x+y|}\left(-y\cdot \hat{x} + \frac{\left(\hat{x}\cdot (x+y)\right)\left(y\cdot(y+x)\right)}{|x+y|^{2}}\right)\end{split}\end{equation}
So,
\begin{equation} |\text{integrand}_{v_{4,1}}| \leq C |\partial_{2}v_{4,c}(s,|x+y|)| + \frac{C |v_{4,c}(s,|x+y|)|}{|x+y|}\end{equation}
This is the same estimate obtained when estimating the integrand of $v_{4}$, except for a factor of $\frac{1}{r}$. So, the second term of \eqref{dtv4forprelim} is estimated as follows:
\begin{equation} \begin{split} &|-\frac{1}{2\pi} \int_{t}^{\infty} ds \int_{B_{s-t}(0)} \frac{dA(y)}{\sqrt{(s-t)^{2}-|y|^{2}}}\text{integrand}_{v_{4,1}}|\leq \frac{C}{t^{2} \log^{3b+2N-1}(t)}, \quad r \leq \frac{t}{2}\end{split}\end{equation}

For the first term in \eqref{dtv4forprelim}, we have
\begin{equation}\label{dtv4intprelimest} \begin{split} &|\int_{t}^{\infty} ds \frac{v_{4,s}(t,r)}{(s-t)}| \\
&\leq C \int_{t}^{\infty} \frac{ds}{(s-t)} \int_{B_{s-t}(0)} \frac{dA(y)}{\sqrt{(s-t)^{2}-|y|^{2}}} |v_{4,c}(s,|x+y|)|\left(\mathbbm{1}_{B_{\frac{s}{2}}(-x)}(y)+\mathbbm{1}_{(B_{\frac{s}{2}}(-x))^{c}}(y)\right) \\
&\leq C \int_{t}^{\infty} ds \int_{0}^{s-t} \rho d\rho \int_{0}^{2\pi} \frac{d\theta}{(s-t)\sqrt{(s-t)^{2}-\rho^{2}}} \frac{1}{s^{2} \log^{N+3b}(s)} \frac{1}{(\log^{2N}(s)+r^{2}+\rho^{2}+2 r \rho\cos(\theta))}\\
&+ C \int_{t}^{\infty} ds \int_{0}^{s-t} \rho d\rho \int_{0}^{2\pi} \frac{d\theta}{(s-t) \sqrt{(s-t)^{2}-\rho^{2}}} \frac{\log(s)}{\log^{2b}(s) s^{2} t} \frac{1}{(s^{2}+\rho^{2}+r^{2}+2 r p \cos(\theta))}\end{split}\end{equation}
We first estimate the third line of \eqref{dtv4intprelimest}:
\begin{equation}\begin{split} &|\int_{t}^{\infty} ds \int_{0}^{s-t} \rho d\rho \int_{0}^{2\pi} \frac{d\theta}{(s-t)\sqrt{(s-t)^{2}-\rho^{2}}} \frac{1}{s^{2} \log^{N+3b}(s)} \frac{1}{(\log^{2N}(s)+r^{2}+\rho^{2}+2 r \rho \cos(\theta))}|\\
& \leq C \int_{0}^{\infty} \rho d\rho \int_{0}^{2\pi} \frac{d\theta}{(t+\rho)^{2} \log^{3b+N}(t)(\log^{2N}(t) + \rho^{2} + r^{2} + 2 r \rho \cos(\theta))} \int_{\rho+t}^{\infty} \frac{ds}{(s-t) \sqrt{(s-t)^{2}-\rho^{2}}}\\
&\leq C \int_{0}^{\infty} \frac{d\rho}{(t+\rho)^{2} \log^{3b+N}(t)} \int_{0}^{2\pi} \frac{d\theta}{(\log^{2N}(t) + \rho^{2}+r^{2}+2 r \rho \cos(\theta))}\\
&\leq C \int_{0}^{\infty} \frac{d\rho}{(t+\rho)^{2} \log^{3b+N}(t)} \frac{1}{\sqrt{(\log^{2N}(t) +(r+\rho)^{2})(\log^{2N}(t) + (r-\rho)^{2})}}\\
&\leq \frac{C}{\log^{3b+N}(t)} \int_{0}^{\infty} \frac{d\rho}{(t+\rho)^{2}} \frac{1}{\sqrt{\log^{2N}(t)+(r+\rho)^{2}}}\frac{1}{\log^{N}(t)}\\
&\leq \frac{C}{t^{2}\log^{3b+2N-1}(t)}\end{split}\end{equation}
Next, we have
\begin{equation}\begin{split} &|\int_{t}^{\infty} ds \int_{0}^{s-t} \rho d\rho \int_{0}^{2\pi} \frac{d\theta}{(s-t) \sqrt{(s-t)^{2}-\rho^{2}}} \frac{\log(s)}{\log^{2b}(s) (s^{2}+\rho^{2}+r^{2}+2 r \rho\cos(\theta)) s^{2} t}|\\
&\leq C \int_{0}^{\infty} \rho d\rho \int_{0}^{2\pi} \frac{d\theta}{\log^{2b}(t)} \frac{\log(t)}{(\rho^{2}+r^{2}+2 r \rho \cos(\theta) + t^{2})} \frac{1}{(\rho+t)}\frac{1}{t^{2}} \int_{\rho+t}^{\infty} \frac{ds}{(s-t) \sqrt{(s-t)^{2}-\rho^{2}}}\\
&\leq \frac{C}{t^{2} \log^{2b}(t)} \int_{0}^{\infty} d\rho \int_{0}^{2\pi} \frac{\log(t) d\theta}{(\rho^{2}+r^{2}+2 r \rho \cos(\theta) +t^{2})(\rho+t)}\\
&\leq \frac{C \log(t)}{t^{3} \log^{2b}(t)} \int_{0}^{\infty} \frac{d\rho}{(\rho+t) \sqrt{t^{2}+\rho^{2}}}\\
& \leq \frac{C\log(t)}{t^{4} \log^{2b}(t)}\end{split}\end{equation}
Combining these, we get \eqref{dtv4prelimest}.\\
\\
Now, we are ready to estimate $\partial_{t}v_{5}$. Again, we have
\begin{equation} \begin{split} \partial_{t}v_{5}&= \frac{1}{2\pi} \int_{t}^{\infty} ds \int_{0}^{s-t} \frac{\rho d\rho}{(s-t)} \frac{1}{\sqrt{(s-t)^{2}-\rho^{2}}} \int_{0}^{2\pi} d\theta \frac{N_{2}(f_{v_{5}})(s,\sqrt{r^{2}+\rho^{2}+2 r \rho \cos(\theta)})(r+\rho \cos(\theta))}{\sqrt{r^{2}+\rho^{2}+2 r \rho \cos(\theta)}} \\
&-\frac{1}{2\pi} \int_{t}^{\infty} ds \int_{B_{s-t}(0)} \frac{dA(y)}{\sqrt{(s-t)^{2}-|y|^{2}}} \text{integrand}_{v_{5,1}}\end{split}\end{equation}
with
\begin{equation}|\text{integrand}_{v_{5,1}}| \leq C |\partial_{2}N_{2}(f_{v_{5}})(s,|x+y|)| + C \frac{|N_{2}(f_{v_{5}})(s,|x+y|)}{|x+y|}\end{equation}
Again, this is the same estimate for the integrand which appeared in the $v_{5}$ pointwise estimates, aside from an extra factor of $\frac{1}{r}$. So, we get
\begin{equation} |-\frac{1}{2\pi} \int_{t}^{\infty} ds \int_{B_{s-t}(0)} \frac{dA(y)}{\sqrt{(s-t)^{2}-|y|^{2}}} \text{integrand}_{v_{5,1}}| \leq \frac{C}{t^{7/2} \log^{3b-3+\frac{5N}{2}}(t)}, \quad r \leq \frac{t}{2}\end{equation}
Then, we need to estimate 
\begin{equation}\begin{split} &\int_{t}^{\infty} ds \int_{B_{s-t}(0)} \frac{dA(y)}{\sqrt{(s-t)^{2}-|y|^{2}}} \frac{|N_{2}(f_{v_{5}})|(s,|x+y|)}{(s-t)}\left(\mathbbm{1}_{B_{\frac{s}{2}}(-x)}(y)+\mathbbm{1}_{(B_{\frac{s}{2}}(-x))^{c}}(y)\right)\\
&\leq C \int_{t}^{\infty} ds \int_{0}^{s-t} \rho d\rho \int_{0}^{2\pi}\frac{d\theta}{\sqrt{(s-t)^{2}-\rho^{2}}} \frac{1}{(s-t)} \left(\frac{r+\rho}{(\lambda(s)^{2} + r^{2}+\rho^{2}+2 r \rho \cos(\theta))s^{4} \log^{3b}(s)} + \frac{s}{s^{6} \log^{3b}(s)}\right)\\
&+ C \int_{t}^{\infty} ds \int_{0}^{s-t} \rho d\rho \int_{0}^{2\pi} \frac{d\theta}{\sqrt{(s-t)^{2}-\rho^{2}}} \frac{1}{(s-t)} \left(\frac{\log^{3}(s)}{s^{2} t^{3}} + \frac{1}{s^{6} \log^{3N+7b-2}(s)}\right)\\
&\leq C \int_{0}^{\infty} \rho d\rho \int_{0}^{2\pi} \frac{d\theta}{\rho} \frac{(r+\rho)}{(1+r^{2}+\rho^{2}+2 r \rho \cos(\theta))}\frac{1}{\log^{b}(t) (\rho+t)^{4}}+ C \int_{t}^{\infty} \frac{ds}{s^{5} \log^{3b}(s)} \\
&+ C \int_{t}^{\infty} ds \left(\frac{\log^{3}(s)}{s^{2} t^{3}} + \frac{1}{s^{6} \log^{3N+7b-2}(s)}\right)\\
&\leq C \frac{(r+1) \log(t)}{t^{4} \log^{b}(t)}+C \frac{\log^{3}(t)}{t^{4}}\\
&\leq C \frac{r \log(t)}{t^{4} \log^{b}(t)} + C \frac{\log^{3}(t)}{t^{4}}, \quad r \leq \frac{t}{2}\end{split}\end{equation}
In total, we have \eqref{dtv5prelimest}.\\
\\
Finally, we estimate $\partial_{t}v_{1}$. Using the same procedure used to estimate $\partial_{t}v_{3}$, we get
\begin{equation}\label{dtv1prelim}\begin{split} \partial_{t}v_{1}(t,r) &= \int_{t}^{\infty} ds \frac{-\lambda''(s)}{r} \int_{0}^{s-t} \frac{\rho d\rho}{(s-t) \sqrt{(s-t)^{2}-\rho^{2}}} \left(1+\frac{r^{2}-1-\rho^{2}}{\sqrt{(r^{2}-1-\rho^{2})^{2}+4r^{2}}}\right)\\
&+\int_{t}^{\infty} ds \frac{\lambda''(s)}{r} \int_{0}^{s-t} \frac{\rho d\rho}{\sqrt{(s-t)^{2}-\rho^{2}}} \frac{8 \rho^{2} r^{2}}{(s-t) (4r^{2}+(1+\rho^{2}-r^{2})^{2})^{3/2}}\end{split}\end{equation}
So, recalling \eqref{B3def}, we have
\begin{equation}\begin{split} |\partial_{t}v_{1}(t,r)| &\leq \frac{C r}{t^{2} \log^{b+1}(t)} \int_{0}^{1}B_{3}(r(1-\sigma)) d\sigma + \frac{C r B_{3}(r)}{t^{2}\log^{b+1}(t)} \leq \frac{C r}{t^{2} \log^{b+1}(t)}\end{split}\end{equation}
\end{proof}

Now, we return to \eqref{dtipest}, and get
\begin{equation}\begin{split} |\partial_{t}\langle \left(\frac{\cos(2Q_{\frac{1}{\lambda(t)}})-1}{r^{2}}\right) v_{3}\vert_{r = R\lambda(t)},\phi_{0}\rangle|&\leq \frac{C \log(\log(t))}{t^{3} \log^{2}(t)} + \frac{C}{t^{2} \log^{-b +1 +b\alpha}(t)}\end{split}\end{equation}
\begin{equation}\begin{split} |\partial_{t}\langle \left(\frac{\cos(2Q_{\frac{1}{\lambda(t)}})-1}{r^{2}}\right) v_{4}\left(1-\chi_{\geq 1}(\frac{4r}{t})\right)\vert_{r = R\lambda(t)},\phi_{0}\rangle|&\leq \frac{C}{t^{3} \log^{2b+2N}(t)} + \frac{C}{t^{2} \log^{b +2N -1}(t)}\end{split}\end{equation}
\begin{equation}\begin{split} |\partial_{t}\langle \left(\frac{\cos(2Q_{\frac{1}{\lambda(t)}})-1}{r^{2}}\right) v_{5}\left(1-\chi_{\geq 1}(\frac{4r}{t})\right)\vert_{r = R\lambda(t)},\phi_{0}\rangle|&\leq \frac{C}{t^{9/2} \log^{2b-2+\frac{5N}{2}}(t)} + \frac{C}{t^{7/2} \log^{b-3+\frac{5N}{2}}(t)}\end{split}\end{equation}
\begin{equation}\begin{split} |\partial_{t}\langle \left(\frac{\cos(2Q_{\frac{1}{\lambda(t)}})-1}{r^{2}}\right) \chi_{\geq 1}(\frac{2r}{\log^{N}(t)})\left(v_{1}+v_{2}+v_{3}\right)\vert_{r = R\lambda(t)},\phi_{0}\rangle|&\leq \frac{C}{t^{3} \log^{2b+2N+1}(t)} + \frac{C}{t^{2} \log^{1+b\alpha +b+2N}(t)}\end{split}\end{equation}
Combining these estimates, we finally conclude
\begin{equation} |RHS_{2}'(t)| \leq \frac{C}{t^{2}\log^{1-3b}(t)}\end{equation} which implies
\begin{equation}\label{lpppprelim} |\lambda'''(t)| \leq \frac{C}{t^{2}\log(\log(t))\log^{1-2b}(t)}\end{equation}
Since $\lambda(t) = \lambda_{0}(t) + e_{0}(t)$, we get 
\begin{equation} \label{epppprelim} |e_{0}'''(t)| \leq \frac{C}{t^{2} \log(\log(t)) \log^{1-2b}(t)}\end{equation}
Now that we have this preliminary estimate, we return to \eqref{epppsetup}, and write $$\lambda(t) = \lambda_{0,0}(t) + e(t)$$
 Note that the function $e$ defined in the above expression is different from the function $e$ which appeared in the construction of $\lambda$ as a solution to the modulation equation. It will suffice for our purposes to have an estimate of the form $|e'''(t)| \leq C |\lambda_{0,0}'''(t)|$, and this is why we will not need to use the slightly more complicated decomposition 
$$\lambda(t) = \lambda_{0}(t) + e_{0}(t) = \lambda_{0,0}(t)+\lambda_{0,1}(t) + e_{0}(t)$$
Then, we differentiate \eqref{epppsetup}, and write the equation in the following way, with any differentiations under the integral sign justified by the preliminary estimate on $e'''$. (Note that $e'''(t) = e_{0}'''(t) +\lambda_{0,1}'''(t)$).
\begin{equation}\label{epppfinaleqn}\begin{split} &-4 \int_{t}^{\infty} \frac{e'''(s) ds}{\log(\lambda_{0,0}(s)) (1+s-t)} -4 \int_{t}^{\infty} \frac{e'''(s) ds}{(1+s-t)^{3} \log(\lambda_{0,0}(s)) (\lambda_{0,0}(t)^{1-\alpha} +s-t)} + 4 \alpha e'''(t)\\
&=-4 \alpha \lambda_{0,0}'''(t) + \frac{-4 (1-\alpha) \lambda(t)^{-\alpha} \lambda'(t)}{\log(\lambda_{0,0}(t))} \int_{t}^{\infty} \frac{e''(s) ds}{(\lambda(t)^{1-\alpha}+s-t)^{2}(1+s-t)^{3}}\\
&+\frac{1}{\log(\lambda_{0,0}(t))}\left(\frac{-4 \alpha \lambda'(t) \lambda''(t)}{\lambda(t)} + \partial_{t}G(t,\lambda(t)) + 4 \partial_{t}\left(\int_{0}^{\infty} \frac{\lambda_{0,0}''(t+w) dw}{(\lambda(t)^{1-\alpha}+w)(1+w)^{3}}\right)\right)\\
&+\frac{1}{\log(\lambda_{0,0}(t))}\left(-E_{\lambda_{0,0}}'(t)+4\alpha (\log(\lambda_{0,0}(t))-\log(\lambda(t)))\lambda'''(t)\right)\\
&-\frac{4}{\log(\lambda_{0,0}(t))} \int_{t}^{\infty} \frac{e'''(s)}{(1+s-t)^{3}} \left(\frac{1}{\lambda_{0,0}(t)^{1-\alpha} +s-t}-\frac{1}{\lambda(t)^{1-\alpha}+s-t}\right) ds\\
&-4 \int_{t}^{\infty} \frac{e'''(s) ds}{1+s-t} \left(\frac{1}{\log(\lambda_{0,0}(s))}-\frac{1}{\log(\lambda_{0,0}(t))}\right)\\
&-4 \int_{t}^{\infty} \frac{e'''(s) ds}{(1+s-t)^{3} (\lambda_{0,0}(t)^{1-\alpha} +s-t)}\left(\frac{1}{\log(\lambda_{0,0}(s))}-\frac{1}{\log(\lambda_{0,0}(t))}\right)\\
&:=RHS_{3}(t) \end{split}\end{equation}
We now proceed to estimate $RHS_{3}$, starting with the terms which do not involve $G$.
\begin{equation} |-4 \alpha \lambda_{0,0}'''(t)| \leq \frac{C}{t^{3} \log^{b+1}(t)}\end{equation}
\begin{equation}\begin{split} |\frac{-4(1-\alpha) \lambda(t)^{-\alpha} \lambda'(t)}{\log(\lambda_{0,0}(t))} \int_{t}^{\infty} \frac{e''(s) ds}{(\lambda(t)^{1-\alpha}+s-t)^{2}(1+s-t)^{3}}|\leq \frac{C}{t^{3} \log^{b+2}(t) (\log(\log(t)))^{3/2}}\end{split}\end{equation}
\begin{equation} |\frac{-4 \alpha \lambda'(t) \lambda''(t)}{\log(\lambda_{0,0}(t)) \lambda(t)}| \leq \frac{C}{t^{3} \log^{b+2}(t) \log(\log(t))}\end{equation}
\begin{equation} \begin{split} &|\frac{4}{\log(\lambda_{0,0}(t))} \partial_{t}\left(\int_{0}^{\infty} \frac{\lambda_{0,0}''(t+w) dw}{(\lambda(t)^{1-\alpha}+w)(1+w)^{3}}\right)|\\
&\leq \frac{C}{\log(\log(t))} \left(\int_{0}^{\infty} \frac{|\lambda_{0,0}'''(t+w)| dw}{(\lambda(t)^{1-\alpha}+w)(1+w)^{3}} + \int_{0}^{\infty} \frac{|\lambda_{0,0}''(t+w)| dw \lambda(t)^{-\alpha} |\lambda'(t)|}{(\lambda(t)^{1-\alpha}+w)^{2}(1+w)^{3}}\right)\\
&\leq \frac{C}{t^{3} \log^{b+1}(t)}\end{split}\end{equation}
By the same procedure used to estimate $E_{\lambda_{0,0}}$, we have
\begin{equation} \frac{|E_{\lambda_{0,0}}'(t)|}{|\log(\lambda_{0,0}(t))|} \leq \frac{C}{t^{3} \log^{b+1}(t) \log(\log(t))}\end{equation}
In addition,
\begin{equation}\begin{split} |\frac{4 \alpha (\log(\lambda_{0,0}(t))-\log(\lambda(t))) \lambda'''(t)}{\log(\lambda_{0,0}(t))}|&\leq \frac{C}{\log(\log(t))}  |\log(1+\frac{e(t)}{\lambda_{0,0}(t)})|\left(\frac{1}{t^{3}\log^{b+1}(t)}+|e'''(t)|\right)\\
&\leq \frac{C}{(\log(\log(t)))^{3/2}}\left(\frac{1}{t^{3}\log^{b+1}(t)}+|e'''(t)|\right)\end{split}\end{equation}
\begin{equation}\begin{split} &|\frac{-4}{\log(\lambda_{0,0}(t))} \int_{t}^{\infty} \frac{e'''(s)}{(1+s-t)^{3}} \left(\frac{1}{\lambda_{0,0}(t)^{1-\alpha} +s-t}-\frac{1}{\lambda(t)^{1-\alpha} +s-t}\right) ds |\\
&\leq \frac{C}{\log(\log(t))} \int_{t}^{\infty} \frac{|e'''(s)|}{(1+s-t)^{3}} |\frac{\lambda(t)^{1-\alpha}-\lambda_{0,0}(t)^{1-\alpha}}{(\lambda_{0,0}(t)^{1-\alpha}+s-t)^{2}}| ds\\
&\leq \frac{C \sup_{x \geq t}|e'''(x)|}{\log(\log(t))} \int_{t}^{\infty} \frac{ds}{\sqrt{\log(\log(t))}} \frac{1}{(1+s-t)^{3} (\lambda_{0,0}(t)^{1-\alpha}+s-t)}\\
&\leq \frac{C \sup_{x \geq t}|e'''(x)|}{\sqrt{\log(\log(t))}}\end{split}\end{equation}
Similarly, 
\begin{equation} \begin{split} &|-4 \int_{t}^{\infty} \frac{e'''(s) ds}{(1+s-t)} \left(\frac{1}{\log(\lambda_{0,0}(s))}-\frac{1}{\log(\lambda_{0,0}(t))}\right)| \leq \frac{C}{t (\log(\log(t)))^{2} \log(t)} \int_{t}^{\infty} |e'''(s)| ds\\
&\leq \frac{C \sup_{x \geq t} \left(|e'''(x)| x^{3/2}\right)}{t^{3/2} (\log(\log(t)))^{2} \log(t)}\end{split}\end{equation}
Finally,
\begin{equation}\begin{split} &|-4 \int_{t}^{\infty} \frac{e'''(s) ds}{(1+s-t)^{3} (\lambda_{0,0}(t)^{1-\alpha}+s-t)} \left(\frac{1}{\log(\lambda_{0,0}(s))}-\frac{1}{\log(\lambda_{0,0}(t))}\right)|\\
&\leq \frac{C \sup_{x \geq t}|e'''(x)|}{t \log(t) (\log(\log(t)))^{2}}\end{split}\end{equation}
Now, we start to estimate the terms arising from $\partial_{t}G(t,\lambda(t))$:
\begin{equation} |\partial_{t}\left(\lambda(t) E_{0,1}(\lambda(t),\lambda'(t),\lambda''(t))\right)| \leq \frac{C}{t^{3} \log^{b+1}(t)} + C|e'''(t)|\end{equation}
\begin{equation}\label{dtk3mk0}\begin{split} &\partial_{t}\left(-16 \int_{t}^{\infty} \lambda''(s) \left(K_{3}(s-t,\lambda(t))-K_{3,0}(s-t,\lambda(t))\right)ds\right)\\
&= -16 \int_{t}^{\infty} \lambda'''(s) \left(K_{3}(s-t,\lambda(t))-K_{3,0}(s-t,\lambda(t))\right) ds\\
&-16 \int_{t}^{\infty} \lambda''(s) \left(\partial_{2}(K_{3}-K_{3,0})(s-t,\lambda(t))\right) \lambda'(t) ds\end{split}\end{equation}
For the first line of the right-hand side of \eqref{dtk3mk0}, we have
\begin{equation} \begin{split} &|-16 \int_{t}^{\infty} \lambda'''(s) \left(K_{3}(s-t,\lambda(t))-K_{3,0}(s-t,\lambda(t))\right) ds|\\
&\leq C \sup_{x \geq t}|\lambda'''(x)| \int_{t}^{\infty} |K_{3}(s-t,\lambda(t))-K_{3,0}(s-t,\lambda(t))| ds\\
&\leq \frac{C}{t^{3} \log^{b+1}(t)} + C \sup_{x \geq t}|e'''(x)|\end{split}\end{equation}
where we recall that $\int_{t}^{\infty} |K_{3}(s-t,\lambda(t))-K_{3,0}(s-t,\lambda(t))| ds$ was estimated in the $v_{3}$ inner product subsection. For the second line of the right-hand side of \eqref{dtk3mk0}, we start with
\begin{equation} K_{3}(w,\lambda(t))-K_{3,0}(w,\lambda(t)) = \left(\frac{w}{1+w^{2}}-\frac{w}{\lambda(t)^{2-2\alpha}+w^{2}}\right) \frac{w^{4}}{4(w^{2}+36\lambda(t)^{2})^{2}} + \frac{1}{4(\lambda(t)^{1-\alpha}+w)(1+w)^{3}}\end{equation}
Then, we get
\begin{equation}\begin{split} |\partial_{2}(K_{3}-K_{3,0})(w,\lambda(t))| &\leq \frac{C \lambda(t)^{\alpha}}{(1+w)^{3}(\lambda(t)+\lambda(t)^{\alpha}w)^{2}}\\
&+\frac{C \lambda(t)^{1+2\alpha} w^{5}}{(36\lambda(t)^{2}+w^{2})^{2} (\lambda(t)^{2}+\lambda(t)^{2\alpha} w^{2})^{2}}\\
&+\frac{C \lambda(t) w^{5}}{(36\lambda(t)^{2}+w^{2})^{3}} |\frac{1}{1+w^{2}}-\frac{1}{(\lambda(t)^{2-2\alpha}+w^{2})}|\end{split}\end{equation}
Then, we note that
\begin{equation} \int_{0}^{\infty} \frac{\lambda(t)^{-\alpha} dw}{(1+w)^{3}(\lambda(t)^{1-\alpha}+w)^{2}} \leq C \log(\log(t))\log^{b}(t)\end{equation}
\begin{equation} \int_{0}^{\infty} \frac{\lambda(t)^{1-2\alpha} w^{5} dw}{(36\lambda(t)^{2}  + w^{2})^{2} (\lambda(t)^{2-2\alpha}+w^{2})^{2}} \leq \frac{C}{\lambda(t)}\end{equation}
\begin{equation}\begin{split} &\int_{0}^{\infty} \frac{\lambda(t) w^{5}}{(36\lambda(t)^{2}+w^{2})^{3}} |\frac{1}{1+w^{2}}-\frac{1}{(\lambda(t)^{2-2\alpha}+w^{2})}| dw\\
&\leq \frac{C \log(\log(t))}{\lambda(t)}\end{split}\end{equation}
This gives
\begin{equation}\begin{split} &|\partial_{t}\left(-16 \int_{t}^{\infty} \lambda''(s) \left(K_{3}(s-t,\lambda(t))-K_{3,0}(s-t,\lambda(t))\right) ds\right)|\\
&\leq C \sup_{x \geq t} |e'''(x)| + \frac{C}{t^{3} \log^{b+1}(t)}\end{split}\end{equation} 
As mentioned before, some terms arising in $\partial_{t}G(t,\lambda_{0}(t)+e(t))$ will be treated differently, now that we have the preliminary estimate on $e'''$. We start with the term
\begin{equation}\begin{split} &\partial_{t}\left(\frac{16}{\lambda(t)^{2}} \int_{t}^{\infty} ds \lambda''(s) \left(K_{1}(s-t,\lambda(t))-\frac{\lambda(t)^{2}}{4(1+s-t)}\right)\right)\\ 
&=\frac{-32 \lambda'(t)}{\lambda(t)^{3}} \int_{t}^{\infty} ds \lambda''(s) \left(K_{1}(s-t,\lambda(t))-\frac{\lambda(t)^{2}}{4(1+s-t)}\right)\\
&+\frac{16}{\lambda(t)^{2}} \int_{t}^{\infty} ds \lambda'''(s) \left(K_{1}(s-t,\lambda(t))-\frac{\lambda(t)^{2}}{4(1+s-t)}\right)\\
&+\frac{16}{\lambda(t)^{2}} \int_{t}^{\infty} ds \lambda''(s) \left(\partial_{2}K_{1}(s-t,\lambda(t)) \lambda'(t) - \frac{\lambda(t) \lambda'(t)}{2 (1+s-t)}\right)\end{split}\end{equation}
For $s-t \leq 1$, we estimate $\partial_{2}K_{1}$ as follows:
\begin{equation} |\partial_{2}K_{1}(s-t,\lambda(t))| \leq \int_{0}^{\infty} \frac{R dR}{(1+R^{2})^{3}} \int_{0}^{s-t} \frac{\rho d\rho}{(s-t)} \frac{4\lambda(t) R^{2}\left(1+\rho^{2}+R^{2}\lambda(t)^{2}\right)}{(4\lambda(t)^{2}R^{2}+(1+\rho^{2}-R^{2}\lambda(t)^{2})^{2})^{3/2}}\end{equation}
We then note that $$4R^{2}\lambda(t)^{2}+(1+\rho^{2}-R^{2}\lambda(t)^{2})^{2} = (1+(\rho+R \lambda(t))^{2})(1+(\rho-R\lambda(t))^{2})$$
So,
\begin{equation}|\frac{4\lambda(t) R^{2}\left(1+\rho^{2}+R^{2}\lambda(t)^{2}\right)}{(4\lambda(t)^{2}R^{2}+(1+\rho^{2}-R^{2}\lambda(t)^{2})^{2})^{3/2}}| \leq C R^{2}\lambda(t)\end{equation}
Then, we obtain
\begin{equation} |\partial_{2}K_{1}(s-t,\lambda(t))| \leq C \lambda(t)(s-t), \quad s-t \leq 1\end{equation} 
From here, we get
\begin{equation} \int_{t}^{t+1} |\lambda''(x)| |\partial_{2}K_{1}(x-t,\lambda(t)) - \frac{\lambda(t)}{2(1+x-t)}| |\lambda'(t)| dx \leq \frac{C}{t^{3}\log^{3b+2}(t)}\end{equation}
For $s-t \geq 1$, we use
\begin{equation} \begin{split} \partial_{2}K_{1}(s-t,\lambda(t))&=\int_{0}^{\infty} \frac{R dR}{(1+R^{2})^{3}} \int_{0}^{\infty} \frac{\rho d\rho}{(s-t)} \frac{4\lambda(t) R^{2}(1+\rho^{2}+\lambda(t)^{2}R^{2})}{(4\lambda(t)^{2} R^{2}+(1+\rho^{2}-R^{2}\lambda(t)^{2})^{2})^{3/2}}\\
&-\int_{0}^{\infty} \frac{R dR}{(1+R^{2})^{3}} \int_{s-t}^{\infty} \frac{\rho d\rho}{(s-t)} \frac{4 \lambda(t)R^{2}(1+\rho^{2}+R^{2}\lambda(t)^{2})}{(4R^{2}\lambda(t)^{2}+(1+\rho^{2}-R^{2}\lambda(t)^{2})^{2})^{3/2}}\\
&=\frac{\lambda(t)}{2(s-t)} -\int_{0}^{\infty} \frac{R dR}{(1+R^{2})^{3}} \int_{s-t}^{\infty} \frac{\rho d\rho}{(s-t)} \frac{4\lambda(t) R^{2}(1+\rho^{2}+R^{2}\lambda(t)^{2})}{(4R^{2}\lambda(t)^{2}+(1+\rho^{2}-R^{2}\lambda(t)^{2})^{2})^{3/2}}\end{split}\end{equation}
Then,
\begin{equation}\begin{split} &\int_{t+1}^{\infty} |\lambda''(s)||\partial_{2}K_{1}(s-t,\lambda(t))-\frac{\lambda(t)}{2(1+s-t)}| ds |\lambda'(t)| \\
&\leq \frac{C}{t^{3}\log^{2b+2}(t)} \int_{t+1}^{\infty} |\partial_{2}K_{1}(s-t,\lambda(t))-\frac{\lambda(t)}{2(1+s-t)}| ds\\
&\leq \frac{C}{t^{3} \log^{2b+2}(t)} \int_{1}^{\infty} dw \left(\frac{\lambda(t)}{2 w}-\frac{\lambda(t)}{2(1+w)}\right)\\
&+\frac{C}{t^{3}\log^{2b+2}(t)} \int_{0}^{\infty} \frac{R dR}{(1+R^{2})^{3}} \int_{1}^{\infty} \rho \log(\rho) \frac{4 \lambda(t) R^{2} (1+\rho^{2}+R^{2}\lambda(t)^{2})}{(4R^{2}\lambda(t)^{2}+(1+\rho^{2}-R^{2}\lambda(t)^{2})^{2})^{3/2}}d\rho\\
&\leq \frac{C}{t^{3}\log^{3b+2}(t)} + \frac{C}{t^{3}\log^{3b+2}(t)} \int_{0}^{\infty} \frac{R^{3} dR}{(1+R^{2})^{3}} \log(2+R\lambda(t))\\
&\leq \frac{C}{t^{3}\log^{3b+2}(t)}\end{split}\end{equation}
We also have
\begin{equation}\begin{split} &|\frac{-32 \lambda'(t)}{\lambda(t)^{3}} \int_{t}^{\infty} ds \lambda''(s) \left(K_{1}(s-t,\lambda(t))-\frac{\lambda(t)^{2}}{4(1+s-t)}\right)|\\
&\leq \frac{C}{t^{3} \log^{2-b}(t)} \int_{t}^{\infty} |K_{1}(s-t,\lambda(t))-\frac{\lambda(t)^{2}}{4(1+s-t)}| ds\\
&\leq \frac{C}{t^{3} \log^{b+2}(t)}\end{split}\end{equation}
Finally, we have
\begin{equation}\begin{split} &|\frac{16}{\lambda(t)^{2}} \int_{t}^{\infty} dx \lambda'''(x) \left(K_{1}(x-t,\lambda(t)) -\frac{\lambda(t)^{2}}{4(1+x-t)}\right)|\\
&\leq \frac{C}{\lambda(t)^{2}} \sup_{x \geq t}(|\lambda'''(x)|) \int_{t}^{\infty} |K_{1}(x-t,\lambda(t))-\frac{\lambda(t)^{2}}{4(1+x-t)}| dx\\
&\leq C \sup_{x \geq t}|\lambda'''(x)|\end{split}\end{equation}
Combining these, we get
\begin{equation}\begin{split} &|\partial_{t}\left(\frac{16}{\lambda(t)^{2}} \int_{t}^{\infty} ds \lambda''(s) \left(K_{1}(s-t,\lambda(t))-\frac{\lambda(t)^{2}}{4(1+s-t)}\right)\right)|\\
&\leq \frac{C}{t^{3}\log^{b+1}(t)} + C \sup_{x \geq t} |e'''(x)|\end{split}\end{equation}

Next, we consider 
\begin{equation}\begin{split} &\partial_{t}\left(\frac{16}{\lambda(t)^{2}} \int_{t}^{\infty} dx \lambda''(x) K(x-t,\lambda(t))\right) \\
&= \frac{-32 \lambda'(t)}{\lambda(t)^{3}} \int_{t}^{\infty} dx \lambda''(x) K(x-t,\lambda(t))+\frac{16}{\lambda(t)^{2}} \int_{t}^{\infty} dx \lambda'''(x) K(x-t,\lambda(t))\\
&+\frac{16}{\lambda(t)^{2}} \int_{t}^{\infty} dx \lambda''(x) \partial_{2}K(x-t,\lambda(t)) \lambda'(t)\end{split}\end{equation}
Note that
\begin{equation}\begin{split} |\frac{-32 \lambda'(t)}{\lambda(t)^{3}} \int_{t}^{\infty} dx \lambda''(x) K(x-t,\lambda(t))| &\leq \frac{C \log^{3b}(t)}{t^{3}\log^{2b+2}(t)} \int_{t}^{\infty} |K(x-t,\lambda(t))| dx\\
&\leq \frac{C}{t^{3} \log^{b+2}(t)}\end{split}\end{equation}
Next,
\begin{equation}\begin{split} |\frac{16}{\lambda(t)^{2}}\int_{t}^{\infty} dx \lambda'''(x) K(x-t,\lambda(t))|&\leq C \sup_{x \geq t} |\lambda'''(x)|\\
&\leq \frac{C}{t^{3}\log^{b+1}(t)} + C \sup_{x \geq t} |e'''(x)|\end{split}\end{equation}
Finally, the integral
\begin{equation} \frac{16}{\lambda(t)^{2}}\int_{t}^{\infty} ds \lambda''(s) \partial_{2}K(s-t,\lambda(t)) \lambda'(t) \end{equation}
was treated during the prelminary estimates (it is equal to $\lambda(t) \cdot I$ in the notation of that section). So, in total, we get
\begin{equation}\begin{split}& |\partial_{t}\left(\frac{16}{\lambda(t)^{2}} \int_{t}^{\infty} dx \lambda''(x) K(x-t,\lambda(t))\right)| \leq \frac{C}{t^{3}\log^{b+1}(t)} + C \sup_{x \geq t} |e'''(x)|\end{split}\end{equation}

Now, we return to \eqref{v3laterest} to prove an estimate on $\partial_{t}v_{3}$ which will be sufficeint to estimate $\partial_{t}v_{4}$. We start with
\begin{equation}\label{dtv3fornonsharpest}\begin{split} \partial_{t}v_{3}(t,r) &=\frac{-1}{r} \int_{0}^{\infty} dw \int_{0}^{w} \frac{\rho d\rho}{\sqrt{w^{2}-\rho^{2}}} \lambda'''(w+t)\left(\frac{-1-\rho^{2}+r^{2}}{\sqrt{(-1-\rho^{2}+r^{2})^{2}+4r^{2}}}+F_{3}(r,\rho,\lambda(t+w))\right)\\
&-\frac{1}{r} \int_{0}^{\infty} dw \int_{0}^{w} \frac{\rho d\rho}{\sqrt{w^{2}-\rho^{2}}} \lambda''(t+w) \partial_{3}F_{3}(r,\rho,\lambda(t+w)) \lambda'(t+w)\end{split}\end{equation}
The first line of \eqref{dtv3fornonsharpest} is treated in the same manner as $v_{3}$ was, while obtaining \eqref{v3laterest}. In particular, the analog of $v_{3,2}$ which appeared in an intermediate step in obtaining \eqref{v3laterest} is estimated by
\begin{equation} \begin{split} & C r \sup_{x \geq t} |\lambda'''(x)| + C r \sup_{x \geq t} \left(|\lambda'''(x)|\lambda(x)^{\alpha-1}\left(\lambda(x)^{\alpha-1}-\lambda(t)^{\alpha-1}\right)\right)\lambda(t)^{2-2\alpha}\\
&\leq \frac{Cr }{t^{3} \log^{b+1}(t)} +  \frac{Cr }{t \log^{b-b\alpha}(t)} \sup_{x \geq t}\left(\frac{|e'''(x)|x}{\lambda(x)^{1-\alpha}}\right)\end{split}\end{equation}
while the analog of $v_{3,1}$ is estimated by $$C r \sup_{x \geq t}\left(\frac{|\lambda'''(x)|}{\lambda(x)^{2-2\alpha}}\right)\log(\log(t))\lambda(t)^{2-2\alpha}$$
To treat the second line of  \eqref{dtv3fornonsharpest}, we start with
\begin{equation} |\partial_{3}F_{3}(r,\rho,\lambda(s))| \leq \frac{C r^{2} \lambda(s)^{2\alpha-3}}{(1+\lambda(s)^{4\alpha-4}(\rho^{2}-r^{2})^{2}+2\lambda(s)^{2\alpha-2}(\rho^{2}+r^{2}))}\end{equation}
Then, we get
\begin{equation} \begin{split} &|\frac{-1}{r} \int_{t}^{\infty} ds \int_{0}^{s-t} \frac{\rho d\rho}{(s-t)} \lambda''(s) \partial_{3}F_{3}(r,\rho,\lambda(s))\lambda'(s)|\\
&\leq \frac{C}{r} \int_{t}^{t+1} \frac{ds}{(s-t)} \frac{1}{s^{3}\log^{2b+2}(s)} \int_{0}^{s-t} \rho d\rho \frac{r^{2} \lambda(s)^{2\alpha-3}}{(1+2\lambda(s)^{2\alpha-2}\rho^{2})}+\frac{C}{r} \int_{t+1}^{\infty} \frac{ds}{s^{3} \log^{2b+2}(s) (s-t)} \int_{0}^{s-t} \rho d\rho |\partial_{3}F_{3}(r,\rho,\lambda(s))|\\
&\leq \frac{C r}{t^{3}\log^{b+1}(t)}\end{split}\end{equation}
and
\begin{equation}\begin{split}&|\frac{-1}{r} \int_{t}^{\infty} ds \int_{0}^{s-t} \rho d\rho \left(\frac{1}{\sqrt{(s-t)^{2}-\rho^{2}}}-\frac{1}{(s-t)}\right)\lambda''(s) \partial_{3}F_{3}(r,\rho,\lambda(s))\lambda'(s)|\\
&\leq C r \int_{0}^{\infty} \frac{\rho d\rho}{(1+\lambda(t)^{4\alpha-4}(\rho^{2}-r^{2})^{2}+2\lambda(t)^{2\alpha-2}(\rho^{2}+r^{2}))}\frac{\log^{(3-2\alpha)b}(t)}{t^{3} \log^{2b+2}(t)}\\
&\leq \frac{C r}{t^{3} \log^{b+2}(t)}\end{split}\end{equation}
Combining these, we get 
\begin{equation}\label{dtv3forfinaleppp} |\partial_{t}v_{3}(t,r)| \leq \frac{C r \log(\log(t))}{t^{3} \log^{b+1}(t)} + \frac{C r}{t} \sup_{x \geq t} \left(\frac{x |e'''(x)|}{\lambda(x)^{2-2\alpha}}\right) \log(\log(t)) \lambda(t)^{2-2\alpha}\end{equation}
Next, we estimate $\partial_{t}E_{5}$. First, we recall the definition of $E_{5}$:
\begin{equation}\label{e5fordtest} \begin{split} E_{5}(t,r)&= -\frac{1}{r} \int_{t}^{t+6r} ds \int_{0}^{s-t} \frac{\rho d\rho}{(s-t)} \lambda''(s) \left(\frac{-1-\rho^{2}+r^{2}}{\sqrt{(-1-\rho^{2}+r^{2})^{2}+4r^{2}}}+F_{3}(r,\rho,\lambda(s))\right)\\
&-\frac{1}{r} \int_{6r}^{\infty} dw \left(\int_{0}^{w} \frac{\rho d\rho}{w} \lambda''(t+w)  \left(\frac{-1-\rho^{2}+r^{2}}{\sqrt{(-1-\rho^{2}+r^{2})^{2}+4r^{2}}}+F_{3}(r,\rho,\lambda(t+w))\right)\right.\\
&\left. - \frac{\lambda''(t+w)}{w} r^{2} w^{2} \left(\frac{1}{1+w^{2}}-\frac{1}{(\lambda(t+w)^{2-2\alpha}+w^{2})}\right)\right)\\
&-r \int_{t+6r}^{\infty} ds \lambda''(s) (s-t)\left(\frac{1}{(\lambda(t)^{2-2\alpha}+(s-t)^{2})} - \frac{1}{\lambda(s)^{2-2\alpha}+(s-t)^{2}}\right)\\
&+v_{3,2}(t,r)\end{split}\end{equation}
We then estimate the time derivative of each of the lines of \eqref{e5fordtest}. For the first line, we have
\begin{equation}\begin{split}&|\partial_{t} \left(-\frac{1}{r} \int_{0}^{6r} dw \int_{0}^{w} \frac{\rho d\rho}{w} \lambda''(t+w) \left(\frac{-1-\rho^{2}+r^{2}}{\sqrt{(-1-\rho^{2}+r^{2})^{2}+4r^{2}}}+F_{3}(r,\rho,\lambda(t+w))\right)\right)|\\
&\leq \frac{C}{r} \int_{0}^{6r} dw \int_{0}^{w} \frac{\rho d\rho}{w} |\lambda'''(t+w)|\cdot|\frac{-1-\rho^{2}+r^{2}}{\sqrt{(-1-\rho^{2}+r^{2})^{2}+4r^{2}}}+F_{3}(r,\rho,\lambda(t+w))|\\
&+\frac{C}{r} \int_{0}^{6r} dw \int_{0}^{w} \frac{\rho d\rho}{w} |\lambda''(t+w) \partial_{3}F_{3}(r,\rho,\lambda(t+w)) \lambda'(t+w)|\\
&\leq \frac{C}{r} \int_{0}^{6r} dw \int_{0}^{w} \frac{\rho d\rho}{w} |\lambda'''(t+w)| + \frac{Cr}{t^{3}\log^{b+1}(t)}\end{split}\end{equation}
where we note that the third line of the inequality above has been estimated already. Next, we have
\begin{equation}\label{intermediateestdte5}\begin{split} &\partial_{t}\left(-\frac{1}{r} \int_{6r}^{\infty} dw \left(\int_{0}^{w} \frac{\rho d\rho}{w} \lambda''(t+w)  \left(\frac{-1-\rho^{2}+r^{2}}{\sqrt{(-1-\rho^{2}+r^{2})^{2}+4r^{2}}}+F_{3}(r,\rho,\lambda(t+w))\right)\right.\right.\\
&\left.\left. - \frac{\lambda''(t+w)}{w} r^{2} w^{2} \left(\frac{1}{1+w^{2}}-\frac{1}{(\lambda(t+w)^{2-2\alpha}+w^{2})}\right)\right)\right)\\
&=\frac{1}{r} \int_{6r}^{\infty} dw \lambda''(t+w) r^{2} w \left(\frac{(2-2\alpha)}{(\lambda(t+w)^{2-2\alpha}+w^{2})^{2}}\lambda(t+w)^{1-2\alpha} \lambda'(t+w)\right)\\
&-\frac{1}{r} \int_{6r}^{\infty} dw \int_{0}^{w} \frac{\rho d\rho}{w} \lambda'''(t+w) \left(\frac{-1-\rho^{2}+r^{2}}{\sqrt{(-1-\rho^{2}+r^{2})^{2}+4r^{2}}}+F_{3}(r,\rho,\lambda(t+w))\right)\\
&+\frac{1}{r} \int_{6r}^{\infty} dw \lambda'''(t+w) r^{2} w \left(\frac{1}{1+w^{2}}-\frac{1}{(\lambda(t+w)^{2-2\alpha}+w^{2})}\right)\\
&-\frac{1}{r} \int_{6r}^{\infty} dw \int_{0}^{w} \frac{\rho d\rho}{w} \lambda''(t+w) \partial_{3}F_{3}(r,\rho,\lambda(t+w))\lambda'(t+w)\end{split}\end{equation}
We get
\begin{equation} \begin{split}&|\frac{1}{r} \int_{6r}^{\infty} dw \lambda''(t+w) r^{2} w \left(\frac{(2-2\alpha)}{(\lambda(t+w)^{2-2\alpha}+w^{2})^{2}}\lambda(t+w)^{1-2\alpha} \lambda'(t+w)\right)|\\
&\leq C r \int_{6r}^{\infty} dw \frac{|\lambda''(t+w)|w}{(\lambda(t+w)^{2-2\alpha}+w^{2})^{2}} \lambda(t+w)^{1-2\alpha} |\lambda'(t+w)|\\
&\leq \frac{C r}{t^{3} \log^{b+2}(t)}\end{split}\end{equation}
By the identical procedure used to estimate  $E_{4}$ in the $v_{3}$ subsection, we have
\begin{equation}\begin{split} &|-\frac{1}{r} \int_{6r}^{\infty} dw \int_{0}^{w} \frac{\rho d\rho}{w} \lambda'''(t+w) \left(\frac{-1-\rho^{2}+r^{2}}{\sqrt{(-1-\rho^{2}+r^{2})^{2}+4r^{2}}}+F_{3}(r,\rho,\lambda(t+w))\right)\\
&+\frac{1}{r} \int_{6r}^{\infty} dw \lambda'''(t+w) r^{2} w \left(\frac{1}{1+w^{2}}-\frac{1}{(\lambda(t+w)^{2-2\alpha}+w^{2})}\right)|\\
&\leq \frac{C r}{t^{3}\log^{b+1}(t)} + C r \sup_{x \geq t}|e'''(x)|\end{split}\end{equation}
We then note that the last line of \eqref{intermediateestdte5} was already estimated, and is bounded above in absolute value by
$$\frac{C r}{t^{3}\log^{b+1}(t)}$$
Then, we estimate
\begin{equation} \begin{split} &|\partial_{t}\left(r \int_{t+6r}^{\infty} ds \lambda''(s) (s-t)\left(\frac{1}{(\lambda(t)^{2-2\alpha}+(s-t)^{2})} - \frac{1}{\lambda(s)^{2-2\alpha}+(s-t)^{2})}\right)\right)|\\
&\leq C r \int_{6r}^{\infty} dw |\lambda'''(t+w)| |\frac{\lambda(t+w)^{2-2\alpha}-\lambda(t)^{2-2\alpha}}{(\lambda(t+w)^{2-2\alpha}+w^{2})^{2}}| w \\
&+ C r |\int_{6r}^{\infty} dw \lambda''(t+w) w \partial_{t}\left(\frac{1}{(\lambda(t)^{2-2\alpha}+w^{2})} - \frac{1}{\lambda(t+w)^{2-2\alpha}+w^{2}}\right)|\end{split}\end{equation}
The second line of the above expression is estimated by
\begin{equation} \begin{split}& r \int_{6r}^{\infty} dw |\lambda'''(t+w)| |\frac{\lambda(t+w)^{2-2\alpha}-\lambda(t)^{2-2\alpha}}{(\lambda(t+w)^{2-2\alpha}+w^{2})^{2}}| w\\
&\leq C r \int_{6r}^{\infty} dw |\lambda'''(t+w)| \frac{\lambda(t)^{1-2\alpha} |\lambda'(t)| w^{2}}{(\lambda(t+w)^{2-2\alpha}+w^{2})^{2}}\\
&\leq \frac{C r}{t \log^{3b+1-3b\alpha}(t)}\left(\frac{1}{t^{3} \log^{1-b+2\alpha b}(t)} + \sup_{x \geq t} \left(\frac{|e'''(x)|}{\lambda(x)^{2-2\alpha}}\right)\right)\end{split}\end{equation}
On the other hand, we have
\begin{equation}\begin{split}&| r \int_{6r}^{\infty} dw \lambda''(t+w) w \partial_{t}\left(\frac{1}{(\lambda(t)^{2-2\alpha}+w^{2})} - \frac{1}{\lambda(t+w)^{2-2\alpha}+w^{2}}\right)|\\
&\leq \frac{Cr}{t^{3}\log^{b+2}(t)}\end{split}\end{equation}
Finally, we recall
\begin{equation}\begin{split} v_{3,2}(t,r) = \frac{-1}{r} \int_{t}^{\infty} ds \int_{0}^{s-t} \rho d\rho& \left(\frac{1}{\sqrt{(s-t)^{2}-\rho^{2}}}-\frac{1}{(s-t)}\right)\lambda''(s) \\
&\left(\frac{-1-\rho^{2}+r^{2}}{\sqrt{(-1-\rho^{2}+r^{2})^{2}+4r^{2}}}+F_{3}(r,\rho,\lambda(s))\right)\end{split}\end{equation}
After taking the time derivative, we estimate the following term with the same argument used for $v_{3,2}$.
\begin{equation}\begin{split} &|\frac{-1}{r} \int_{0}^{\infty} dw \int_{0}^{w} \rho d\rho \left(\frac{1}{\sqrt{w^{2}-\rho^{2}}}-\frac{1}{w}\right) \lambda'''(t+w) \left(\frac{-1-\rho^{2}+r^{2}}{\sqrt{(-1-\rho^{2}+r^{2})^{2}+4r^{2}}}+F_{3}(r,\rho,\lambda(t+w))\right)|\\
&\leq \frac{C r}{t^{3} \log^{b+1}(t)} + \frac{C r \sup_{x \geq t}\left(\frac{|e'''(x)| x}{\lambda(x)^{3-2\alpha}}\right)}{t \log^{3b-2b\alpha}(t)}\end{split}\end{equation}
Finally,
$$|\frac{-1}{r} \int_{0}^{\infty} dw \int_{0}^{w} \rho d\rho \left(\frac{1}{\sqrt{w^{2}-\rho^{2}}}-\frac{1}{w}\right) \lambda''(t+w) \partial_{3}F_{3}(r,\rho,\lambda(t+w)) \lambda'(t+w)| \leq \frac{C r}{t^{3} \log^{b+2}(t)}$$
where we note that the integral has already been estimated while studying $\partial_{t}v_{3}$.

In total, we get
\begin{equation}\begin{split} |\partial_{t}E_{5}(t,r)| &\leq \frac{C r}{t^{3}\log^{b+1}(t)}+\frac{C r \sup_{x \geq t}\left(\frac{|e'''(x)|x}{\lambda(x)^{3-2\alpha}}\right)}{t\log^{(3-2\alpha)b}(t)}\end{split}\end{equation}
Now, we will obtain an estimate on $\partial_{t}v_{4}$ which is better than the preliminary one, \eqref{dtv4prelimest}. In particular, now that we have the preliminary estimate on $\lambda'''$, we can use the same procedure used to estimate $\partial_{t}v_{3}$, to see that $\partial_{t} v_{1}$ solves, with $0$ Cauchy data at infinity, the same equation as $v_{1}$ does, except with $\lambda''$ replaced by $\lambda'''$. This observation, combined with the estimates on $v_{1}$ gives
\begin{equation}\label{dtv1finalestsmallr} |\partial_{t}v_{1}(t,r)| \leq \frac{C r \left(\log(t)+\log(3+2r)\right)}{t} \sup_{x \geq t}\left(|\lambda'''(x)| x\right)\end{equation}
(Note that we can not directly use the $\partial_{t}v_{1}$ analog of \eqref{v1largerest}, becuase the preliminary estimate on $|e'''|$ is not good enough to justify the steps which would produce such an estimate).\\
Now, we use the estimate \eqref{dtv3forfinaleppp}, combined with the above estimate on $\partial_{t}v_{1}$, and the previous estimate on $\partial_{t}v_{2}$, to get
\begin{equation} \begin{split} |\partial_{t}v_{4,c}(t,r)| &\leq \frac{C |\chi_{\geq 1}'(\frac{2r}{\log^{N}(t)})|}{t \log(t)} \left(\frac{\lambda(t)^{2}}{(\lambda(t)^{2}+r^{2})^{2}} |v_{1}+v_{2}+v_{3}| + |F_{0,2}(t,r)|\right) \\
&+ \frac{C \chi_{\geq 1}(\frac{2r}{\log^{N}(t)}) \lambda(t) |\lambda'(t)|}{(r^{2}+\lambda(t)^{2})^{2}} |v_{1}+v_{2}+v_{3}| + \frac{C\chi_{\geq 1}(\frac{2r}{\log^{N}(t)}) \lambda(t)^{2} |\partial_{t}(v_{1}+v_{2}+v_{3})|}{(\lambda(t)^{2}+r^{2})^{2}} \\
&+ C\chi_{\geq 1}(\frac{2r}{\log^{N}(t)}) |\partial_{t}F_{0,2}(t,r)|\end{split}\end{equation}

Using $|\chi_{\geq 1}'(\frac{2r}{\log^{N}(t)})| \cdot \frac{r}{t \log^{N+1}(t)} \leq  C \frac{\mathbbm{1}_{\{r \geq \frac{\log^{N}(t)}{2}\}}}{t \log(t)}$, we get
\begin{equation} \begin{split} |\partial_{t}v_{4,c}(t,r)| &\leq \frac{C \mathbbm{1}_{\{r \geq \frac{\log^{N}(t)}{2}\}}}{\log^{2b}(t) r^{4}} \begin{cases} \frac{r}{t^{3} \log^{b}(t)}, \quad r \leq \frac{t}{2}\\
\frac{\log(r)}{|t-r|} \left(\frac{1}{|t-r|} + \frac{1}{t \log(t)}\right), \quad t>r> \frac{t}{2}\end{cases}\\
&+ \frac{C \mathbbm{1}_{\{r \geq \frac{\log^{N}(t)}{2}\}}}{r^{3}t^{3}} \left(\frac{\log(t)+\log(r)}{\log^{3b+1}(t)}\right) + C \mathbbm{1}_{\{r \geq \frac{\log^{N}(t)}{2}\}} \frac{\lambda(t)^{2-2\alpha}(\log(t)+\log(r))}{r^{3}t \log^{2b}(t)} \sup_{x \geq t} \left(\frac{|e'''(x)| x}{\lambda(x)^{2-2\alpha}}\right)\end{split}\end{equation}

Then, we note that
\begin{equation} \partial_{t}v_{4}(t,r) = \frac{-1}{2\pi} \int_{0}^{\infty} dw \int_{0}^{w} \frac{\rho d\rho}{\sqrt{w^{2}-\rho^{2}}}\int_{0}^{2\pi} d\theta \frac{\partial_{1}v_{4,c}(t+w,\sqrt{r^{2}+\rho^{2}+2 r \rho \cos(\theta)})}{\sqrt{r^{2}+\rho^{2}+2 r \rho \cos(\theta)}}(r+\rho \cos(\theta))\end{equation}

and, carry out the identical procedure done to obtain the estimate on $v_{4}^{\lambda_{1}}-v_{4}^{\lambda_{2}}$. This results in
\begin{equation} |\partial_{t}v_{4}(t,r)| \leq \frac{C}{t^{3}\log^{3b+N-2}(t)} + \frac{C \lambda(t)^{2-2\alpha}}{t \log^{2b+N-3}(t)}\sup_{x \geq t}\left(\frac{x |e'''(x)|}{\lambda(x)^{2-2\alpha}}\right), \quad r \leq \frac{t}{2}\end{equation}

Like previously, in order to estimate $\partial_{t}v_{4}$ in the region $r \geq \frac{t}{2}$, we first record a slightly different estimate on $\partial_{t}v_{4,c}$. To obtain this, we use \eqref{v2sqrtrest} in the region $t-\sqrt{t} \leq r \leq t+\sqrt{t}$, and \eqref{v2singularconeest} for the other parts of the region $r \geq \frac{t}{2}$ to estimate $v_{2}$. For $\partial_{t} v_{2}$, we use \eqref{v2sqrtrest} in the region $t-t^{1/4} \leq r \leq t+t^{1/4}$, and  \eqref{v2singularconeest} in the other parts of the region $r \geq \frac{t}{2}$.\\
\\
This leads to
\begin{equation} \begin{split} |\partial_{t}v_{4,c}(t,r)| &\leq \frac{C \mathbbm{1}_{\{r \geq \frac{\log^{N}(t)}{2}\}}}{\log^{2b}(t) r^{4}} \begin{cases} \frac{r}{t^{3} \log^{b}(t)}, \quad r \leq \frac{t}{2}\\
\frac{\log(r)}{|t-r|} \left(\frac{1}{|t-r|} + \frac{1}{t \log(t)}\right), \quad \frac{t}{2} \leq r \leq t-\sqrt{t},\text{ or } r\geq t+\sqrt{t}\\
\frac{\log(r)}{(t-r)^{2}}+\frac{1}{\sqrt{r} t \log(t)}, \quad t-\sqrt{t} \leq r \leq t-t^{1/4},\text{ or } t+t^{1/4} \leq r \leq t+\sqrt{t}\\
\frac{1}{\sqrt{r}}, \quad t-t^{1/4} \leq r \leq t+t^{1/4}\end{cases}\\
&+ \frac{C \mathbbm{1}_{\{r \geq \frac{\log^{N}(t)}{2}\}}}{r^{3}t^{3}} \left(\frac{\log(t)+\log(r)}{\log^{3b+1}(t)}\right) + C \mathbbm{1}_{\{r \geq \frac{\log^{N}(t)}{2}\}} \frac{\lambda(t)^{2-2\alpha}(\log(t)+\log(r))}{r^{3}t \log^{2b}(t)} \sup_{x \geq t} \left(\frac{|e'''(x)| x}{\lambda(x)^{2-2\alpha}}\right)\end{split}\end{equation}
 First, we obtain
\begin{equation}\label{dtv4cl2estfordtv4final} ||\partial_{t}v_{4,c}||_{L^{2}(r dr)} \leq \frac{C}{t^{3}\log^{3b+2N}(t)}+\frac{C \sup_{x \geq t}\left(\frac{x |e'''(x)|}{\lambda(x)^{2-2\alpha}}\right)}{t \log^{4b-2b\alpha+2N-1}(t)}\end{equation}
Then, we use the same procedure used to estimate $v_{4}^{\lambda_{1}}-v_{4}^{\lambda_{2}}$ in the region $r \geq \frac{t}{2}$, and get
\begin{equation}\label{dtv4finalestlarger}|\partial_{t}v_{4}(t,r)| \leq \frac{C}{t^{2}\log^{3b+2N}(t)}+\frac{C \sup_{x \geq t}\left(\frac{x |e'''(x)|}{\lambda(x)^{2-2\alpha}}\right)}{\log^{4b-2\alpha b+2N-2}(t)}, \quad r \geq \frac{t}{2}\end{equation}

Now, we recall \eqref{dtipest}, and prove new estimates on the terms involving $v_{4},E_{5},v_{1}+v_{2}+v_{3}$, and $F_{0,2}$
\begin{equation}\begin{split} |\partial_{t}\langle \left(\frac{\cos(2Q_{\frac{1}{\lambda(t)}})-1}{r^{2}}\right) v_{4}\left(1-\chi_{\geq 1}(\frac{4r}{t})\right)\vert_{r = R\lambda(t)},\phi_{0}\rangle| &\leq \frac{C}{t^{3}\log^{b+N-2}(t)}+\frac{C \sup_{x \geq t}\left(\frac{x |e'''(x)|}{\lambda(x)^{2-2\alpha}}\right)}{t \log^{N-3+2b-2b\alpha}(t)}\end{split}\end{equation}
\begin{equation}\begin{split} |\partial_{t}\langle \left(\frac{\cos(2Q_{\frac{1}{\lambda(t)}})-1}{r^{2}}\right) E_{5}\vert_{r = R\lambda(t)},\phi_{0}\rangle| &\leq \frac{C}{t^{3}\log(t)}+\frac{C \sup_{x \geq t}\left(\frac{x|e'''(x)|}{\lambda(x)^{3-2\alpha}}\right)}{\log^{(2-2\alpha)b}(t)t}\end{split}\end{equation}
\begin{equation}\begin{split} &|\partial_{t}\langle \left(\frac{\cos(2Q_{\frac{1}{\lambda(t)}})-1}{r^{2}}\right) \chi_{\geq 1}(\frac{2r}{\log^{N}(t)})\left(v_{1}+v_{2}+v_{3}\right)\vert_{r = R\lambda(t)},\phi_{0}\rangle| \leq  \frac{C}{t^{3}\log^{2b+2N}(t)}+\frac{C \lambda(t)^{2-2\alpha}\sup_{x\geq t}\left(\frac{|e'''(x)|x}{\lambda(x)^{2-2\alpha}}\right)}{t \log^{2N+b-1}(t)}\end{split}\end{equation}
\begin{equation} |\partial_{t}\langle \chi_{\geq 1}(\frac{2r}{\log^{N}(t)}) F_{0,2}\vert_{r=R\lambda(t)},\phi_{0}\rangle| \leq \frac{C}{t^{3} \log^{2b+1-2b\alpha +2N}(t)} + \frac{C |e'''(t)|}{\log^{b-2b\alpha+2N}(t)}\end{equation} 
From our previous estimates, we also have
\begin{equation} \begin{split}&|\lambda'(t)| |\langle  \left(\frac{\cos(2Q_{\frac{1}{\lambda(t)}})-1}{r^{2}}\right) \left((v_{4}+v_{5})\left(1-\chi_{\geq 1}(\frac{4r}{t})\right)+E_{5}-\chi_{\geq 1}(\frac{2r}{\log^{N}(t)})\left(v_{1}+v_{2}+v_{3}\right)\right)\vert_{r = R\lambda(t)},\phi_{0}\rangle|\\
&+|\lambda'(t)| |\langle \chi_{\geq 1}(\frac{2r}{\log^{N}(t)}) F_{0,2}\vert_{r=R\lambda(t)},\phi_{0}\rangle|\\
&\leq \frac{C}{t^{3}\log^{b+2}(t)}\end{split}\end{equation}
Combining these, we get
\begin{equation} \begin{split} &|\partial_{t}\left(\lambda(t) \langle  \left(\frac{\cos(2Q_{\frac{1}{\lambda(t)}})-1}{r^{2}}\right) \left((v_{4}+v_{5})\left(1-\chi_{\geq 1}(\frac{4r}{t})\right)+E_{5}-\chi_{\geq 1}(\frac{2r}{\log^{N}(t)})\left(v_{1}+v_{2}+v_{3}\right)\right)\vert_{r = R\lambda(t)},\phi_{0}\rangle\right)|\\
&+|\partial_{t}\left(\lambda(t) \langle \chi_{\geq 1}(\frac{2r}{\log^{N}(t)})F_{0,2}\vert_{r = R\lambda(t)},\phi_{0}\rangle\right)|\\
&\leq \frac{C}{t^{3}\log^{b+1}(t)} + \frac{C \sup_{x \geq t}\left(\frac{x|e'''(x)|}{\lambda(x)^{3-2\alpha}}\right)}{t \log^{(3-2\alpha)b}(t)}\end{split}\end{equation}
where we use \eqref{dtv5prelimest} to estimate $\partial_{t}v_{5}$. Finally, for all remaining terms in $RHS_{3}$, we use the same estimate that was used in obtaining the preliminary estimate on $e'''$, and conclude
\begin{equation} |RHS_{3}(t)| \leq \frac{C}{t^{3} \log^{b+1}(t)} + \frac{C \sup_{x \geq t}\left(x^{3/2}|e'''(x)|\right)}{t^{3/2}\sqrt{\log(\log(t))}} \end{equation}
Note that the left-hand side of \eqref{epppfinaleqn} is the same as that of \eqref{emodulation}, except with $e''$ replaced with $e'''$, and $\lambda_{0}$ replaced with $\lambda_{0,0}$. This will not cause any major differences in the study of \eqref{epppfinaleqn} relative to \eqref{emodulation}, because the key estimates \eqref{kineq} and \eqref{kquotest} are invariant under multiplication of $K(t,s)$ by any non-negative function of $s$, and \eqref{kineq} is still true with the replacement of $\frac{1}{(\lambda_{0}(t)^{1-\alpha}+s-t)}$ by $\frac{1}{(\lambda_{0,0}(t)^{1-\alpha}+s-t)}$. In particular, if we define 
$$K_{2}(t,s) := \frac{\mathbbm{1}_{\leq 0}(s-t)}{\alpha |\log(\lambda_{0,0}(-s))|} \left(\frac{1}{1-s+t}+\frac{1}{(1-s+t)^{3}(\lambda_{0,0}(-t)^{1-\alpha}-s+t)}\right)$$
then, by the above discussion, the resolvent kernel associated to $K_{2}$ in the same way that $r$ was associated to $K$, exists, and satisfies the same estimate as $r$: \eqref{restimate}. Let us denote this resolvent kernel by $r_{2}$.\\
\\
So far, we have that $e'''$ is a solution to \eqref{epppfinaleqn}, which can be re-cast into the form
$$K_{2} * x+x=H_{2}$$
similarly to \eqref{emodulation}. Next, we carry out the following computation, noting that the step which requires Fubini's theorem is justified by the preliminary estimate on $e'''$, \eqref{epppprelim}.
\begin{equation}\begin{split} r_{2}*(K_{2}*x)+r_{2}*x&=r_{2}*H_{2}\\
(r_{2}*K_{2})*x+r_{2}*x&=r_{2}*H_{2}\\
H_{2}-x=K_{2}*x&=r_{2}*H_{2}\end{split}\end{equation}
where we used the equation
$$r_{2}+r_{2}*K_{2}=K_{2}$$
(Recall that $r_{2}$ solves the $r_{2}$ and $K_{2}$ analogs of \eqref{resolventeqns}). Translating $x$ and $H_{2}$ back to $e'''$ and $RHS_{3}$, we have
\begin{equation}\label{eppprep} e'''(t) = \frac{RHS_{3}(t)}{4 \alpha} -\int_{t}^{\infty} \frac{RHS_{3}(z)}{4\alpha} r_{2}(-t,-z) dz, \quad \text{a.e. }t \geq T_{0}\end{equation}\\
\\
$e''', RHS_{3}$ are continuous functions, and the above equation can be rearranged to yield
\begin{equation} \int_{t}^{\infty} \frac{RHS_{3}(z)}{4 \alpha}r_{2}(-t,-z)dz = \frac{RHS_{3}(t)}{4 \alpha}-e'''(t), \quad \text{a.e. }t \geq T_{0}\end{equation}\\
\\
So, $$t \mapsto \int_{t}^{\infty} \frac{RHS_{3}(z)}{4 \alpha}r_{2}(-t,-z)dz$$ agrees with a continuous function almost everywhere, and can thus be extended to a continuous function. Moreover, using \eqref{restimate}, for $r_{2}$ instead of $r$, we have
\begin{equation}  |\int_{t}^{\infty} \frac{RHS_{3}(z)}{4 \alpha}r_{2}(-t,-z)dz| \leq ||\frac{RHS_{3}}{2 \alpha}||_{L^{\infty}([t,\infty))}, \quad \text{a.e. }t \geq T_{0}\end{equation}\\
\\ 
So, the same estimate holds for all $t \geq T_{0}$, where we now identify $$t \mapsto \int_{t}^{\infty} \frac{RHS_{3}(z)}{4 \alpha}r_{2}(-t,-z)dz$$ with its continuous extension described above. Such an identification will be performed without further mention.
Returning to \eqref{eppprep}, we get
\begin{equation}\begin{split} |e'''(t)| &\leq C \sup_{x \geq t} |RHS_{3}(x)| \\
&\leq \frac{C}{t^{3}\log^{b+1}(t)} + \frac{C \sup_{x \geq t}\left(x^{3/2}|e'''(x)|\right)}{t^{3/2}\sqrt{\log(\log(t))}}\end{split}\end{equation}
Thus,
\begin{equation} t^{3/2} |e'''(t)| \leq \frac{C}{t^{3/2}\log^{b+1}(t)} + \frac{C \sup_{x \geq t}\left(x^{3/2} |e'''(x)|\right)}{\sqrt{\log(\log(t))}}\end{equation}
But, as mentioned earlier, $e'''$ is a continuous function on $[T_{0},\infty)$,\\
 so $t\mapsto t^{3/2}|e'''(t)|$ is also continuous on $[T_{0},\infty)$, and
 $$|e'''(t)| t^{3/2} \rightarrow 0, \quad t \rightarrow \infty$$
 by the preliminary estimate on $e'''$.  So, for all $t \geq T_{0}$, there exists $y(t) \geq t$ such that 
$$\sup_{x \geq t}\left(x^{3/2} |e'''(x)|\right) = y(t)^{3/2}|e'''(y(t))|$$
But, then, 
\begin{equation}\begin{split} \sup_{x \geq t}\left(x^{3/2}|e'''(x)|\right) = y(t)^{3/2} |e'''(y(t))| &\leq \frac{C}{y(t)^{3/2} \log^{b+1}(y(t))} + \frac{C \sup_{x \geq y(t)}\left(x^{3/2} |e'''(x)|\right)}{\sqrt{\log(\log(y(t))}}\\
&\leq \frac{C}{t^{3/2} \log^{b+1}(t)} + \frac{C \sup_{x \geq t}\left(x^{3/2} |e'''(x)|\right)}{\sqrt{\log(\log(t))}}\end{split}\end{equation}
So, there exists some absolute constants $C_{p_{2}},C_{p_{1}}>0$ such that, for all $$t \geq C_{p_{1}}+T_{0}$$ we have
\begin{equation} \sup_{x \geq t}\left(x^{3/2}|e'''(x)|\right) \leq \frac{C_{p_{2}}}{t^{3/2}\log^{b+1}(t)}\end{equation}
But, $e''' \in C([T_{0},\infty))$, so, there exists some absolute constant $C_{p_{3}}>0$ such that \begin{equation} |e'''(t)| \leq \frac{C_{p_{3}}}{t^{3}\log^{b+1}(t)}, \quad t \geq T_{0}\end{equation} Recalling $\lambda(t)=\lambda_{0,0}(t)+e(t)$, we have
\begin{equation}\label{lambdapppfinalest} |\lambda'''(t)| \leq \frac{C}{t^{3}\log^{b+1}(t)}, \quad t \geq T_{0}\end{equation}
To finish, we only need to establish the estimates on $\partial_{t}\partial_{r}^{j}v_{k}$ in the proposition statement. We recall that $\partial_{t}v_{1}$ solves the same equation as $v_{1}$, also with $0$ Cauchy data at infinity, except with $\lambda''$ on the right-hand side replaced with $\lambda'''$. Now that we have established \eqref{lambdapppfinalest}, we can justify the steps leading to the $\partial_{t}v_{1}$ analog of \eqref{v1largerest}, which gives
\begin{equation} |\partial_{t}v_{1}(t,r)| \leq \frac{C}{r t \log^{b+1}(t)}, \quad r > \frac{t}{2}\end{equation}
Combined with \eqref{dtv1finalestsmallr}, this gives \eqref{dtv1finalest}. \\
\\
A similar large $r$ estimate can be also proven for $\partial_{t}v_{3}$: Starting with \eqref{dtv3fornonsharpest}, we get
\begin{equation}\begin{split} |\partial_{t}v_{3}(t,r)| &\leq \frac{C}{r} \int_{t}^{\infty} ds |\lambda'''(s)| (s-t)\\
&+\frac{C}{r} \int_{t}^{\infty} ds \int_{0}^{s-t} \frac{\rho d\rho}{\sqrt{(s-t)^{2}-\rho^{2}}} \frac{|\lambda''(s)| |\lambda'(s)| r^{2} \lambda(s)^{2\alpha -3}}{(1+\lambda(s)^{4\alpha-4}(\rho^{2}-r^{2})^{2} + 2 \lambda(s)^{2\alpha -2}(\rho^{2}+r^{2}))}\\
&\leq \frac{C}{r t \log^{b+1}(t)}, \quad r > \frac{t}{2}\end{split}\end{equation}
Combining this with \eqref{dtv3forfinaleppp} gives \eqref{dtv3finalest}.\\
\\
Next, we estimate  $\partial_{tr}v_{k}, \quad k=1,3$. We start with $\partial_{tr}v_{3}$:
\begin{equation}\label{dtv3fordtrv3}\begin{split} \partial_{t}v_{3}(t,r)&=\frac{-1}{r}\int_{t}^{\infty} ds \int_{0}^{s-t} \frac{\rho d\rho}{\sqrt{(s-t)^{2}-\rho^{2}}} \lambda'''(s) \left(\frac{-1-\rho^{2}+r^{2}}{\sqrt{(1+\rho^{2}-r^{2})^{2}+4r^{2}}}+F_{3}(r,\rho,\lambda(s))\right)\\
&-\frac{1}{r} \int_{t}^{\infty} ds \int_{0}^{s-t} \frac{\rho d\rho}{\sqrt{(s-t)^{2}-\rho^{2}}}\lambda''(s) \partial_{3}F_{3}(r,\rho,\lambda(s))\lambda'(s)\end{split}\end{equation}
Since the first line of \eqref{dtv3fordtrv3} is the same expression as $v_{3}$, except with $\lambda'''$ replacing $\lambda''$, we use the same procedure for this line, as was used to estimate $\partial_{r}v_{3}$ . For the second line of \eqref{dtv3fordtrv3}, we start by noting that
\begin{equation}\begin{split} |\partial_{r}\left(\frac{\partial_{3}F_{3}(r,\rho,\lambda(s))}{r}\right)| &\leq \frac{C \lambda(s)^{2\alpha +1}(\lambda(s)^{6}+(\rho^{2}+r^{2}) \lambda(s)^{4+2\alpha} + (\rho^{2}+r^{2})^{2} \lambda(s)^{2+4\alpha})}{\lambda(s)^{10} (1+(\rho^{2}-r^{2})^{2}\lambda(s)^{4\alpha -4} + 2(\rho^{2}+r^{2})\lambda(s)^{2\alpha -2})^{5/2}}\\
&+\frac{C \lambda(s)^{8\alpha +1}(\rho^{2}-r^{2})^{2} (\rho^{2}+r^{2})}{\lambda(s)^{10}(1+(\rho^{2}-r^{2})^{2} \lambda(s)^{4\alpha -4} + 2(\rho^{2}+r^{2})\lambda(s)^{2\alpha -2})^{5/2}}\end{split}\end{equation}
We then estimate the partial $r$ derivative of the second line of \eqref{dtv3fordtrv3} using the same procedure used before to estimate $\partial_{r}v_{3}$, and get
\begin{equation}\begin{split} &|\partial_{r}\left(-\frac{1}{r} \int_{t}^{\infty} ds \int_{0}^{s-t} \frac{\rho d\rho}{\sqrt{(s-t)^{2}-\rho^{2}}}\lambda''(s) \partial_{3}F_{3}(r,\rho,\lambda(s))\lambda'(s)\right)| \leq \frac{C}{t^{3}\log^{1+b}(t)}\end{split}\end{equation}
(Note that we have the factor $|\lambda'(s)| \leq \frac{C}{s \log^{b+1}(s)}$ in the integrand of the second line of \eqref{dtv3fordtrv3}, which explains the gain relative to the analogous term arising in the $\partial_{r} v_{3}$ estimates). In total, we get
\eqref{dtrv3finalest}.\\
\\
As observed before, $\partial_{t}v_{1}$ has the same exact representation formulae as $v_{1}$, except with an extra derivative on $\lambda''$. Therefore, we can use the identical procedure used to estimate $\partial_{r}v_{1}$, to get \eqref{dtrv1finalest}.\\
\\
This gives
\begin{equation}\label{dtrv4cfinal}\begin{split} |\partial_{tr}v_{4,c}(t,r)| &\leq \frac{C |\chi_{\geq 1}''(\frac{2r}{\log^{N}(t)})|}{r^{3} t^{3} \log^{1+N+3b}(t)} + \frac{C |\chi_{\geq 1}'(\frac{2r}{\log^{N}(t)})|}{r^{3} t^{3} \log^{3b+N}(t)}\\
& + C \chi_{\geq 1}(\frac{2r}{\log^{N}(t)}) \begin{cases} \frac{1}{r^{4} t^{3} \log^{3b}(t)}, \quad r \leq \frac{t}{2}\\
\frac{\log(r)}{\log^{2b}(t) r^{4}} \left(\frac{1}{\log(t) r |t-r|t}+\frac{1}{(t-r)^{2} t} + \frac{1}{(t-r)^{3}}\right) + \frac{1}{t^{3} r^{4} \log^{3b}(t)}, \quad t> r > \frac{t}{2}\end{cases}\end{split}\end{equation}
Using the final estimates on $\lambda'''$, we also get
\begin{equation} \frac{|\partial_{t}v_{4,c}(t,r)|}{r} \leq \frac{C \mathbbm{1}_{\{r \geq \frac{\log^{N}(t)}{2}\}}}{r^{4} t^{3} \log^{3b+1-2\alpha b}(t)} + \frac{C \mathbbm{1}_{\{r \geq \frac{\log^{N}(t)}{2}\}}}{\log^{2b}(t) r^{5}} \begin{cases} \frac{r}{t^{3} \log^{b}(t)}, \quad r \leq \frac{t}{2}\\
\frac{\log(r)}{|t-r|} \left(\frac{1}{t \log(t)} + \frac{1}{|t-r|}\right) + \frac{1}{r t \log^{b+1}(t)}, \quad t > r \geq \frac{t}{2}\end{cases}\end{equation}
 Now, we can prove \eqref{dtv4finalest} in the region $r \leq \frac{t}{2}$, by writing
\begin{equation}\begin{split}& \partial_{t}v_{4}(t,r) \\
&= \frac{-r}{2\pi} \int_{t}^{\infty} ds \int_{0}^{s-t} \frac{\rho d\rho}{\sqrt{(s-t)^{2}-\rho^{2}}}\int_{0}^{1} d\beta \int_{0}^{2\pi} d\theta \partial_{r}\left(\frac{\partial_{1}v_{4,c}(s,\sqrt{r^{2}+\rho^{2}+2 r \rho \cos(\theta)})}{\sqrt{r^{2}+\rho^{2}+2 r \rho \cos(\theta)}} (r+\rho \cos(\theta))\right)(s,\rho,\beta r)\end{split}\end{equation}
and using the same procedure used for $v_{4}$. In total, we get
\begin{equation}\label{dtv4refinementnearorigin} |\partial_{t}v_{4}(t,r)| \leq \frac{C r}{t^{3} \log^{3b+2N-2}(t)}, \quad r \leq \frac{t}{2}\end{equation}
Combining this with \eqref{dtv4finalestlarger} (and the final estimate on $\lambda'''$) gives \eqref{dtv4finalest}.

\end{proof}

\subsubsection{Estimating $\lambda''''$}

\begin{proposition} $\lambda \in C^{4}([T_{0},\infty))$ and we have the following estimates:
\begin{equation} |\lambda''''(t)| \leq \frac{C}{t^{4} \log^{b+1}(t)}, \quad t \geq T_{0} \end{equation}
\begin{equation} \label{dttv1finalest} |\partial_{t}^{2} v_{1}(t,r)| \leq \begin{cases} \frac{C r}{t^{4} \log^{b}(t)}, \quad r \leq \frac{t}{2}\\
\frac{C}{r t^{2} \log^{b+1}(t)}, \quad r > \frac{t}{2}\end{cases}\end{equation}

\begin{equation} \label{dttv3finalest} |\partial_{t}^{2} v_{3}(t,r)| \leq \frac{C r \log(\log(t))}{t^{4} \log^{b+1}(t)}\end{equation}

\begin{equation}\label{dttv4finalest} |\partial_{t}^{2}v_{4}(t,r)| \leq \begin{cases} \frac{C}{t^{4} \log^{3b-2+N}(t)}, \quad r \leq \frac{t}{2}\\
\frac{C}{t^{35/12} \log^{2b-1}(t)}, \quad r > \frac{t}{2} \end{cases}\end{equation}

\begin{equation}\label{dttv5finalest}\begin{split}|\partial_{t}^{2}v_{5}(t,r)| &\leq \frac{C}{t^{4}\log^{3N+b-2}(t)}, \quad r \leq \frac{t}{2}\end{split}\end{equation}

\end{proposition}
\begin{proof}
Recalling that $\lambda(t) = \lambda_{0,0}(t)+e(t)$, we will show that $e \in C^{4}([T_{0},\infty))$, and estimate $e''''$. Returning to \eqref{epppsetup}, we have
\begin{equation} \lambda'''(t) = \frac{RHS_{2}'(t)}{g_{2}(t)}-\frac{g_{2}'(t)}{g_{2}(t)^{2}} RHS_{2}(t)\end{equation}
Because $\lambda'''\in C^{0}([T_{0},\infty))$, an inspection of the definition of $RHS_{2}$ shows that $RHS_{2} \in C^{2}([T_{0},\infty))$; so, $\lambda'''\in C^{1}([T_{0},\infty))$, with
$$\lambda''''(t) = \partial_{t}\left(\frac{RHS_{2}'(t)}{g_{2}(t)}-\frac{g_{2}'(t)}{g_{2}(t)^{2}} RHS_{2}(t)\right)$$
The previously obtained estimates on $RHS_{2}, RHS_{2}'$ then show that 
\begin{equation}\label{lambdappppprelim1} |\lambda''''(t)| \leq \frac{C |RHS_{2}''(t)|}{\log(\log(t))\log^{b}(t)}+\frac{C}{t^{3}\log^{2-2b}(t)\log(\log(t))}\end{equation}
Again, we will first obtain a preliminary estimate on $\lambda''''$, which will be improved afterwards. We start by recalling the definition of $RHS_{2}$:
\begin{equation}\begin{split} RHS_{2}(t)&=-\frac{16}{\lambda(t)^{3}} \int_{t}^{\infty} ds \lambda''(s) \left(K_{1}(s-t,\lambda(t))+K(s-t,\lambda(t))\right) + 2 \frac{(\lambda'(t))^{2}}{\lambda(t)^{2}} \\
&+\frac{4b}{\lambda(t) t^{2} \log^{b}(t)} + E_{v_{2},ip}(t,\lambda(t))\\
&+ \langle \left(\frac{\cos(2Q_{\frac{1}{\lambda(t)}})-1}{r^{2}}\right)\left(v_{3}+(v_{4}+v_{5})\left(1-\chi_{\geq 1}(\frac{4r}{t})\right)-\chi_{\geq 1}\left(\frac{2r}{\log^{N}(t)}\right)\left(v_{1}+v_{2}+v_{3}\right)\right)\vert_{r=R\lambda(t)},\phi_{0}\rangle\\
&-4 \int_{0}^{\infty} \chi_{\geq 1}(\frac{2R \lambda(t)}{\log^{N}(t)}) \frac{(\lambda'(t))^{2} R^{2} \phi_{0}(R) dR}{\lambda(t)^{2} (R^{2}+1)^{2}}\end{split}\end{equation}
We then note
\begin{equation} |\partial_{t}^{2}\left(\frac{2\lambda'(t)^{2}}{\lambda(t)^{2}}\right)| \leq \frac{C}{t^{4}\log^{2}(t)}\end{equation}
\begin{equation} |\partial_{t}^{2}\left(\frac{4b}{\lambda(t)t^{2}\log^{b}(t)}\right)| \leq \frac{C}{t^{4}}\end{equation}
\begin{equation}\begin{split} |\partial_{t}^{2}\left(\frac{-4 \lambda'(t)^{2}}{\lambda(t)^{2}} \int_{0}^{\infty} \chi_{\geq 1}(\frac{2R \lambda(t)}{\log^{N}(t)}) \frac{R^{2} \phi_{0}(R) dR}{(R^{2}+1)^{2}}\right)| \leq \frac{C}{t^{4} \log^{2+2b+2N}(t)}\end{split}\end{equation}
 To continue estimating $RHS_{2}''$, we start with
\begin{equation}\label{k1relatedtermsinRHS3pp} \begin{split} &\partial_{t}^{2}\left(\frac{16}{\lambda(t)^{2}} \int_{t}^{\infty} \lambda''(x)K_{1}(x-t,\lambda(t))dx\right)\\
&=\left(\frac{96}{\lambda(t)^{4}} \lambda'(t)^{2} -\frac{32 \lambda''(t)}{\lambda(t)^{3}}\right)\int_{t}^{\infty} \lambda''(x) K_{1}(x-t,\lambda(t)) dx - \frac{64 \lambda'(t)}{\lambda(t)^{3}} \int_{t}^{\infty} \lambda'''(x) K_{1}(x-t,\lambda(t)) dx\\
&-\frac{64 \lambda'(t)}{\lambda(t)^{3}} \int_{t}^{\infty} \lambda''(x) \partial_{2}K_{1}(s-t,\lambda(t))\lambda'(t) dx -\frac{16}{\lambda(t)^{2}} \int_{t}^{\infty} \lambda'''(x) \partial_{1}K_{1}(x-t,\lambda(t)) dx \\
&+ \frac{32}{\lambda(t)^{2}}\int_{t}^{\infty} \lambda'''(x) \partial_{2}K_{1}(x-t,\lambda(t)) \lambda'(t) dx\\
&+\frac{16}{\lambda(t)^{2}} \int_{t}^{\infty} \lambda''(x) \left(\partial_{2}^{2}K_{1}(x-t,\lambda(t)) \lambda'(t)^{2} + \partial_{2}K_{1}(x-t,\lambda(t)) \lambda''(t)\right) dx \end{split}\end{equation}
The only term in \eqref{k1relatedtermsinRHS3pp} which we can not immediately estimate from previous estimates is the one involving $\partial_{2}^{2}K_{1}$. The analogous term involving $K$ is also present in $RHS_{3}''(t)$. We start with
\begin{equation}\begin{split} &|\partial_{22}K(x,\lambda(t))(\lambda'(t))^{2}+\partial_{2}K(x,\lambda(t))\lambda''(t)| \\
&\leq \int_{0}^{\infty} dr \frac{C r\lambda(t)^{2}\left(\lambda'(t)^{2}+\lambda(t)\lambda''(t)\right)}{(r^{2}+\lambda(t)^{2})^{3}} \int_{0}^{x} \rho d\rho \left(\frac{1}{\sqrt{x^{2}-\rho^{2}}}-\frac{1}{x}\right)\left(1+\frac{r^{2}-1-\rho^{2}}{\sqrt{(r^{2}-1-\rho^{2})^{2}+4r^{2}}}\right)\end{split}\end{equation}
which gives
\begin{equation}\begin{split} &|\frac{16}{\lambda(t)^{2}} \int_{t}^{\infty} dx \lambda''(x) \left(\partial_{2}^{2} K(x-t,\lambda(t))\lambda'(t)^{2}+\partial_{2}K(x-t,\lambda(t))\lambda''(t)\right)|\\
&\leq \frac{C}{t^{4} \log^{3b+2}(t)}  \int_{t}^{\infty} dx \int_{0}^{\infty} \frac{r dr}{(r^{2}+\lambda(t)^{2})^{3}} \int_{0}^{x-t} \rho d\rho \left(\frac{1}{\sqrt{(x-t)^{2}-\rho^{2}}}-\frac{1}{x-t}\right)\left(1+\frac{r^{2}-1-\rho^{2}}{\sqrt{(r^{2}-1-\rho^{2})^{2}+4r^{2}}}\right)\\
&\leq \frac{C}{t^{4} \log^{3b+2}(t)} \int_{0}^{\infty} \frac{r dr}{(r^{2}+\lambda(t)^{2})^{3}} \int_{0}^{\infty} \rho d\rho \left(1+\frac{r^{2}-1-\rho^{2}}{\sqrt{(r^{2}-1-\rho^{2})^{2}+4r^{2}}}\right) \int_{\rho+t}^{\infty} \left(\frac{1}{\sqrt{(x-t)^{2}-\rho^{2}}}-\frac{1}{x-t}\right) dx\\
&\leq \frac{C}{t^{4} \log^{3b+2}(t)} \int_{0}^{\infty} dr \frac{r^{3}}{(r^{2}+\lambda(t)^{2})^{3}}\\
&\leq \frac{C}{t^{4}\log^{2+b}(t)}\end{split}\end{equation}
Then, we treat the $K_{1}$ term:
\begin{equation}\begin{split}K_{1}(x,\lambda(t))&=\int_{0}^{\infty} \frac{r dr}{\lambda(t)^{2}(1+\frac{r^{2}}{\lambda(t)^{2}})^{3}} \int_{0}^{x} \frac{\rho d\rho}{x} \left(1+\frac{r^{2}-1-\rho^{2}}{\sqrt{(r^{2}-1-\rho^{2})^{2}+4r^{2}}}\right)\end{split}\end{equation}
Proceeding as for $K$, we get
\begin{equation} \begin{split} &|\frac{16}{\lambda(t)^{2}}\int_{t}^{\infty} dx \lambda''(x) \left(\partial_{2}^{2}K_{1}(x-t,\lambda(t))\lambda'(t)^{2}+\partial_{2}K_{1}(x-t,\lambda(t))\lambda''(t)\right)|\\
&\leq C \int_{t}^{\infty} dx |\lambda''(x)| \frac{1}{t^{2}\log^{2b+1}(t)}\int_{0}^{\infty} \frac{r dr}{(r^{2}+\lambda(t)^{2})^{3}} \frac{1}{(x-t)} \int_{0}^{x-t} \rho d\rho \left(1+\frac{r^{2}-1-\rho^{2}}{\sqrt{(r^{2}-1-\rho^{2})^{2}+4r^{2}}}\right)\\
&\leq C \int_{t}^{\infty} dx |\lambda''(x)| \frac{1}{t^{2}\log^{2b+1}(t)} \int_{0}^{\infty} \frac{r dr}{(r^{2}+\lambda(t)^{2})^{3}} \frac{1}{(x-t)} \begin{cases} \int_{0}^{x-t} 2 \rho d\rho, \quad x-t \leq 1\\
\int_{0}^{\infty} \rho d\rho \left(1+\frac{r^{2}-1-\rho^{2}}{\sqrt{(r^{2}-1-\rho^{2})^{2}+4r^{2}}}\right), \quad x-t \geq 1\end{cases}\\
&\leq C \int_{t}^{t+1} \frac{dx (x-t)}{x^{2}\log^{b+1}(x) t^{2}\log^{2b+1}(t)} \frac{1}{\lambda(t)^{4}}\\
&+C \int_{t+1}^{\infty} \frac{dx}{x^{2}\log^{b+1}(x)(x-t)t^{2}\log(t)}\\
&\leq \frac{C}{t^{4} \log^{1-b}(t)}\end{split}\end{equation}
Combining this with estimates for the other terms in \eqref{k1relatedtermsinRHS3pp} (which are deduced from the procedure used in obtaining preliminary estimates on $e'''$), we get
\begin{equation}\begin{split}&|\partial_{t}^{2} \left(\frac{16}{\lambda(t)^{2}} \int_{t}^{\infty} \lambda''(x) K_{1}(x-t,\lambda(t)) +\frac{16}{\lambda(t)^{2}} \int_{t}^{\infty} \lambda''(x) K(x-t,\lambda(t))\right)| \\
&\leq \frac{C}{t^{3}\log^{1-2b}(t)}\end{split}\end{equation}

The next term in $RHS_{2}''(t)$ which we consider is
\begin{equation}\begin{split} &\partial_{t}^{2}\left(2 c_{b}\lambda(t) \int_{0}^{\infty} d\xi \frac{\sin(t\xi)}{t^{2}} \psi_{v_{2}}(\xi,\lambda(t))\right)\\
&=2 c_{b}\lambda''(t) \int_{0}^{\infty} d\xi \frac{\sin(t\xi)}{t^{2}}\psi_{v_{2}}(\xi,\lambda(t)) + 4 c_{b}\lambda'(t) \int_{0}^{\infty} d\xi \frac{-2 \sin(t\xi)}{t^{3}} \psi_{v_{2}}(\xi,\lambda(t))\\
&+4 c_{b}\lambda'(t) \int_{0}^{\infty} d\xi \frac{\xi \cos(t\xi)}{t^{2}}\psi_{v_{2}}(\xi,\lambda(t))+4 c_{b}\lambda'(t) \int_{0}^{\infty} d\xi \frac{\sin(t\xi)}{t^{2}}\partial_{2}\psi_{v_{2}}\lambda'(t)\\
&+4 c_{b}\lambda(t) \int_{0}^{\infty} d\xi \frac{-2 \xi \cos(t\xi)}{t^{3}}\psi_{v_{2}}(\xi,\lambda(t)) + 2 c_{b}\lambda(t)\int_{0}^{\infty} d\xi \frac{6 \sin(t\xi)}{t^{4}} \psi_{v_{2}}(\xi,\lambda(t)) + 4 c_{b}\lambda(t) \int_{0}^{\infty} d\xi \frac{-2 \sin(t\xi)}{t^{3}} \partial_{2}\psi_{v_{2}} \lambda'(t)\\
&-2c_{b}\lambda(t)\int_{0}^{\infty} d\xi \frac{\xi^{2}\sin(t\xi)}{t^{2}}\psi_{v_{2}}(\xi,\lambda(t)) + 4 c_{b}\lambda(t)\int_{0}^{\infty} d\xi \frac{\xi \cos(t\xi)}{t^{2}}\partial_{2}\psi_{v_{2}}\lambda'(t)\\
&+2 c_{b}\lambda(t) \int_{0}^{\infty} d\xi \frac{\sin(t\xi)}{t^{2}}\left(\partial_{22}\psi_{v_{2}}\lambda'(t)^{2}+\partial_{2}\psi_{v_{2}} \lambda''(t)\right)\end{split}\end{equation} 
Only three integrals can not be immediately estimated based on our previous estimates. These three integrals are estimated using the same procedure used before, by estimating $K_{1}(x)$ and various of its derivatives, using the fact that $K_{1}$ appears in $\psi_{v_{2}}(\xi \lambda(t))$ with its argument $x=\xi \lambda(t)$ satisfying $0 < x < \frac{1}{4}$.  We have
\begin{equation} \begin{split}-2 c_{b}\lambda(t) \int_{0}^{\infty} d\xi \frac{\xi^{2} \sin(t\xi)}{t^{2}}\psi_{v_{2}}(\xi,\lambda(t)) &=\frac{2 c_{b}\lambda(t)}{t^{5}} \int_{0}^{\infty} \cos(t\xi) \partial_{\xi}^{3}\left(\xi^{2}\psi_{v_{2}}(\xi,\lambda(t))\right)d\xi\end{split}\end{equation}
\begin{equation}\begin{split}4 c_{b}\lambda(t)\int_{0}^{\infty} d\xi \frac{\xi \cos(t\xi)}{t^{2}}\partial_{2}\psi_{v_{2}}\lambda'(t)&=-4 c_{b}\lambda(t) \int_{0}^{\infty} d\xi \frac{\cos(t\xi)}{t^{4}} \partial_{\xi}^{2}\left(\xi \partial_{2}\psi_{v_{2}}\right) \lambda'(t)\end{split}\end{equation}
and
\begin{equation} 2 c_{b}\lambda(t) \int_{0}^{\infty} d\xi \frac{\sin(t\xi)}{t^{2}}\left(\partial_{22}\psi_{v_{2}}\lambda'(t)^{2}\right)=2 c_{b}\lambda(t) \int_{0}^{\infty} d\xi \frac{\cos(t\xi)}{t^{3}} \partial_{122}\psi_{v_{2}}(\xi,\lambda(t)) \lambda'(t)^{2}\end{equation}
Using these, and the procedure described above, we get
\begin{equation}\begin{split}&|-2 c_{b}\lambda(t) \int_{0}^{\infty} d\xi \frac{\xi^{2} \sin(t\xi)}{t^{2}}\psi_{v_{2}}(\xi,\lambda(t))|+|4 c_{b}\lambda(t)\int_{0}^{\infty} d\xi \frac{\xi \cos(t\xi)}{t^{2}}\partial_{2}\psi_{v_{2}}\lambda'(t)|\\
&+|2 c_{b}\lambda(t) \int_{0}^{\infty} d\xi \frac{\sin(t\xi)}{t^{2}}\left(\partial_{22}\psi_{v_{2}}\lambda'(t)^{2}\right)|\\
&\leq \frac{C}{t^{5}}\end{split}\end{equation}
Every other term arising in $$\partial_{t}^{2}\left(2 c_{b}\lambda(t) \int_{0}^{\infty} d\xi \frac{\sin(t\xi)}{t^{2}} \psi_{v_{2}}(\xi,\lambda(t))\right)$$ can be estimated by using estimates for $\lambda,\lambda',\lambda''$, and the terms estimated when considering
$$\partial_{t}\left(2 c_{b}\lambda(t) \int_{0}^{\infty} d\xi \frac{\sin(t\xi)}{t^{2}} \psi_{v_{2}}(\xi,\lambda(t))\right)$$
In total, we then get
\begin{equation} |\partial_{t}^{2}\left(2 c_{b}\lambda(t) \int_{0}^{\infty} d\xi \frac{\sin(t\xi)}{t^{2}} \psi_{v_{2}}(\xi,\lambda(t))\right)| \leq \frac{C}{t^{5}}\end{equation}
For the term
\begin{equation} \partial_{t}^{2}\left(2c_{b}\lambda(t)\int_{0}^{\infty} d\xi \frac{\sin(t\xi)}{t^{2}} \chi_{\leq \frac{1}{4}}(\xi) F_{v_{2}}(t,\xi)\right)\end{equation}
we use the identical procedure as above, noting that the only difference is that $\psi_{v_{2}}$ is replaced here by $\chi_{\leq \frac{1}{4}}(\xi) F_{v_{2}}$. Then, we get
\begin{equation}\begin{split}&|-2 c_{b}\lambda(t) \int_{0}^{\infty} d\xi \frac{\xi^{2} \sin(t\xi)}{t^{2}}\chi_{\leq \frac{1}{4}}(\xi)F_{v_{2}}(\xi,\lambda(t))|+|4 c_{b}\lambda(t)\int_{0}^{\infty} d\xi \frac{\xi \cos(t\xi)}{t^{2}}\partial_{2}\left(\chi_{\leq \frac{1}{4}}(\xi)F_{v_{2}}(\xi,\lambda(t))\right)\lambda'(t)|\\
&+|2 c_{b}\lambda(t) \int_{0}^{\infty} d\xi \frac{\sin(t\xi)}{t^{2}}\left(\partial_{22}\left(\chi_{\leq \frac{1}{4}}(\xi)F_{v_{2}}(\xi,\lambda(t))\right)\lambda'(t)^{2}\right)|\\
&\leq \frac{C \log(\log(t))}{t^{5} \log^{2b}(t)}\end{split}\end{equation}
In total, we have
\begin{equation} |\partial_{t}^{2}\left(2c_{b}\lambda(t)\int_{0}^{\infty} d\xi \frac{\sin(t\xi)}{t^{2}} \chi_{\leq \frac{1}{4}}(\xi) F_{v_{2}}(t,\xi)\right)| \leq \frac{C \log(\log(t))}{t^{5}\log^{2b}(t)}\end{equation}
The next term to consider is
\begin{equation}\begin{split}&\partial_{t}^{2}\left(\int_{0}^{\frac{1}{2}} d\xi \left(\chi_{\leq \frac{1}{4}}(\xi)-1\right) \frac{\sin(t\xi)}{t^{2}} \left(\frac{b-1}{\xi \log^{b}(\frac{1}{\xi})}+\frac{b(b-1)}{\xi \log^{b+1}(\frac{1}{\xi})}\right)\right)\\
&=\frac{4}{t^{4}} \left(\frac{b-1}{\log^{b}(2)}+\frac{b(b-1)}{\log^{b+1}(2)}\right)\sin(\frac{t}{2}) + \text{Err}\\
& -\left(\frac{2(b-1)}{\log^{b}(2)} + \frac{2b(b-1)}{\log^{b+1}(2)}\right)\frac{\cos(\frac{t}{2})}{4t^{3}} + \frac{\sin(\frac{t}{2})}{t^{4}}\left(\partial_{\xi}\left(\frac{\xi (b-1)}{\log^{b}(\frac{1}{\xi})}+\frac{\xi b(b-1)}{\log^{b+1}(\frac{1}{\xi})}\right)\vert_{\xi =\frac{1}{2}}\right) +E_{2}\end{split}\end{equation}
where we integrate by parts after differentiating the integral, and 
$$|\text{Err}| \leq \frac{C}{t^{5}}, \quad |E_{2}| \leq \frac{C}{t^{5}}$$
On the other hand, we have
\begin{equation} \begin{split} \partial_{t}^{2}\left(\int_{0}^{\frac{t}{2}} \frac{\sin(u)b(b-1) du}{t^{2}u \log^{b+1}(\frac{t}{u})}\right) &=\frac{\cos(\frac{t}{2}) b(b-1)}{2t^{3}\log^{b+1}(2)} -\frac{3 \sin(\frac{t}{2}) b(b-1)}{t^{4}\log^{b+1}(2)}\\
&+\frac{\sin(\frac{t}{2}) b(b-1)}{t}\left(\frac{-2}{t^{3}\log^{b+1}(2)}-\frac{(b+1)}{t^{3}\log^{b+2}(2)}\right)\\
&+\int_{0}^{\frac{t}{2}} \frac{\sin(u)b(b-1)}{u} \left(\frac{6}{t^{4} \log^{b+1}(\frac{t}{u})}+\frac{5(b+1)}{t^{4} \log^{b+2}(\frac{t}{u})} + \frac{(b+2)(b+1)}{t^{4}\log^{b+3}(\frac{t}{u})}\right)du\end{split}\end{equation}
So,
\begin{equation}\begin{split} &\partial_{t}^{2}\left(\int_{0}^{\frac{1}{2}} d\xi \left(\chi_{\leq \frac{1}{4}}(\xi)-1\right)\frac{\sin(t\xi)}{t^{2}}\left(\frac{b-1}{\xi \log^{b}(\frac{1}{\xi})}+\frac{b(b-1)}{\xi \log^{b+1}(\frac{1}{\xi})}\right)\right)\\
&+\partial_{t}^{2}\left(\int_{0}^{\frac{t}{2}}du \frac{\sin(u) (b-1)}{t^{2}u} \left(\frac{b}{\log^{b+1}(\frac{t}{u})}+\frac{1}{\log^{b}(\frac{t}{u})}\right)\right)\\
&=O\left(\frac{1}{t^{5}}\right)+ \int_{0}^{\frac{t}{2}}du \frac{\sin(u)(b-1)}{u}\left(\frac{6}{t^{4}\log^{b}(\frac{t}{u})}+\frac{5b}{t^{4}\log^{b+1}(\frac{t}{u})}+\frac{b(b+1)}{t^{4}\log^{b+2}(\frac{t}{u})}\right)\\
&+\int_{0}^{\frac{t}{2}}du \frac{\sin(u)b(b-1)}{u}\left(\frac{6}{t^{4} \log^{b+1}(\frac{t}{u})}+\frac{5(b+1)}{t^{4}\log^{b+2}(\frac{t}{u})} + \frac{(b+1)(b+2)}{t^{4}\log^{b+3}(\frac{t}{u})}\right)\\
&=O\left(\frac{1}{t^{5}}\right)+\frac{3(b-1)\pi}{t^{4}\log^{b}(t)} + O\left(\frac{1}{t^{4}\log^{b+1}(t)}\right)\end{split}\end{equation}
where we used the computations from the section which constructed $v_{2}$, to obtain the $\frac{1}{t^{4}\log^{b}(t)}$ term above. We conclude the following estimate on one of the terms arising in $E_{v_{2}}''(t)$:
\begin{equation}\begin{split} &|\partial_{t}^{2}\left(\int_{0}^{\frac{1}{2}} d\xi \left(\chi_{\leq \frac{1}{4}}(\xi)-1\right)\frac{\sin(t\xi)}{t^{2}}\left(\frac{b-1}{\xi \log^{b}(\frac{1}{\xi})}+\frac{b(b-1)}{\xi \log^{b+1}(\frac{1}{\xi})}\right)\right)\\
&+\partial_{t}^{2}\left(\int_{0}^{\frac{t}{2}} \frac{\sin(u) (b-1)}{t^{2}u} \left(\frac{b}{\log^{b+1}(\frac{t}{u})}+\frac{1}{\log^{b}(\frac{t}{u})}\right)du-\left(\frac{(b-1)\pi}{2t^{2}\log^{b}(t)}\right)\right)|\\
&\leq \frac{C}{t^{4}\log^{b+1}(t)}\end{split}\end{equation}
In total, we finally get
\begin{equation} |\partial_{t}^{2}\left(\lambda(t) E_{v_{2},ip}(t,\lambda(t))\right)| \leq \frac{C}{t^{4}\log^{b+1}(t)}\end{equation}
By the same procedure, this estimate is also true for the case $b=1$.\\
\\
Now, we record preliminary estimates on $\partial_{t}^{2}v_{k}, \quad k=1,3,4,5$
\begin{lemma}[Preliminary estimates on $\partial_{t}^{2}v_{k}$]
We have the following \emph{preliminary} estimates on $\partial_{t}^{2}v_{k}$:
\begin{equation}\label{dttv1prelimest} |\partial_{t}^{2}v_{1}(t,r)| \leq \frac{C r}{(1+r)} \frac{1}{t^{3}\log^{b+1}(t)}\end{equation}

\begin{equation} \label{dttv3prelimest} |\partial_{t}^{2}v_{3}(t,r)| \leq \frac{C r}{t^{4}\log^{b+1}(t)}+ \frac{C}{t^{3} \log^{1+2\alpha b-b}(t)}\end{equation}

\begin{equation}\label{dttv4prelimest} |\partial_{t}^{2}v_{4}(t,r)|\leq \begin{cases}\frac{C}{t^{3}\log^{b+2N-2}(t)}, \quad r\leq \frac{t}{2}\\
\frac{C}{t^{2}\log^{3N+b}(t)}, \quad r \geq \frac{t}{2}\end{cases}\end{equation}

\begin{equation}\label{dttv5prelimest}\begin{split}|\partial_{t}^{2}v_{5}(t,r)| &\leq \frac{C}{t^{4}\log^{3N+b-2}(t)}, \quad r \leq \frac{t}{2}\end{split}\end{equation}
\end{lemma}
\begin{proof}
First, we estimate $\partial_{t}^{2}v_{3}$. We start with the formula for $\partial_{t}v_{3}$ used in the process of proving the final estimate on $\lambda'''$:
\begin{equation}\label{dtv3forlambdapppp}\begin{split} \partial_{t}v_{3}(t,r) &= \frac{-1}{r} \int_{t}^{\infty} ds \int_{0}^{s-t} \frac{\rho d\rho}{\sqrt{(s-t)^{2}-\rho^{2}}} \lambda'''(s) \left(\frac{-1-\rho^{2}+r^{2}}{\sqrt{(1+\rho^{2}-r^{2})^{2}+4r^{2}}} + F_{3}(r,\rho,\lambda(s))\right)\\
&-\frac{1}{r} \int_{0}^{\infty} dw \int_{0}^{w} \frac{\rho d\rho}{\sqrt{w^{2}-\rho^{2}}} \lambda''(t+w) \partial_{3}F_{3}(r,\rho,\lambda(t+w))\lambda'(t+w)\end{split}\end{equation}

Denote the first line on the right-hand side of \eqref{dtv3forlambdapppp} by $v_{3,1,t}$. Then, we have

\begin{equation}\label{dttv3prelim1}\begin{split} \partial_{t}v_{3,1,t}(t,r) &= \frac{1}{r} \int_{t}^{\infty} ds \int_{0}^{s-t} \frac{\rho d\rho}{(s-t)\sqrt{(s-t)^{2}-\rho^{2}}} \lambda'''(s) \left(\frac{-1-\rho^{2}+r^{2}}{\sqrt{(1+\rho^{2}-r^{2})^{2}+4r^{2}}} + F_{3}(r,\rho,\lambda(s))\right)\\
&-\frac{1}{r} \int_{t}^{\infty} ds \int_{0}^{s-t} \frac{\rho d\rho}{(s-t)\sqrt{(s-t)^{2}-\rho^{2}}}\lambda'''(s) \left(\frac{8 \rho^{2} r^{2}}{(4r^{2}+(1+\rho^{2}-r^{2})^{2})^{3/2}} \right.\\
&+\left. \frac{-8 \lambda(s)^{4\alpha -4}\rho^{2} r^{2}}{(4 \lambda(s)^{2\alpha -2} r^{2}+(1+\lambda(s)^{2\alpha -2}(\rho^{2}-r^{2}))^{2})^{3/2}}\right)\end{split}\end{equation}

We will need to use a more complicated procedure for \eqref{dttv3prelim1} than what was used to treat the analogous term which arose in the course of obtaining \eqref{dtv3prelimest}. For the first line on the right-hand side of \eqref{dttv3prelim1}, we start with
\begin{equation}\begin{split} |\frac{-1-\rho^{2}+r^{2}}{\sqrt{(1+\rho^{2}-r^{2})^{2}+4r^{2}}}+F_{3}(r,\rho,\lambda(s))| &\leq \begin{cases} 2, \quad \rho \leq 4r\\
|\frac{-1-\rho^{2}+r^{2}}{\sqrt{(1+\rho^{2}-r^{2})^{2}+4r^{2}}}+1|+|-1+F_{3}(r,\rho,\lambda(s))|, \quad \rho > 4r\end{cases}\\
&\leq \begin{cases} 2, \quad \rho\leq 4r\\
\frac{r^{2}}{\rho^{2}}, \quad \rho > 4r\end{cases}\end{split}\end{equation}

Then, we have
\begin{equation} \begin{split} &|\frac{1}{r} \int_{t}^{\infty} ds \int_{0}^{s-t} \frac{\rho d\rho}{(s-t)\sqrt{(s-t)^{2}-\rho^{2}}} \lambda'''(s) \left(\frac{-1-\rho^{2}+r^{2}}{\sqrt{(1+\rho^{2}-r^{2})^{2}+4r^{2}}} + F_{3}(r,\rho,\lambda(s))\right)|\\
&\leq \frac{C}{r} \int_{0}^{\infty} \frac{\rho d\rho}{t^{3}\log^{b+1}(t)} \left(\begin{cases} 2, \quad \rho\leq 4r\\
\frac{r^{2}}{\rho^{2}}, \quad \rho > 4r\end{cases}\right) \int_{\rho+t}^{\infty} ds \frac{1}{(s-t)\sqrt{(s-t)^{2}-\rho^{2}}}\\
&\leq \frac{C}{r} \int_{0}^{4r} \frac{d\rho}{t^{3}\log^{b+1}(t)} + \frac{C}{r} \int_{4r}^{\infty} \frac{r^{2} d\rho}{\rho^{2} t^{3} \log^{b+1}(t)}\\
&\leq \frac{C}{t^{3} \log^{b+1}(t)}\end{split}\end{equation}

We then consider the second line on the right-hand side of \eqref{dttv3prelim1}. First, note that, because $\lambda'(x) \leq 0, \quad x \geq T_{0}$, we have $\lambda(s)^{2\alpha -2} \geq \lambda(t)^{2\alpha -2}, \quad s \geq t$. Then, we start with the case $r \leq \frac{\lambda(t)^{1-\alpha}}{4} \leq \frac{1}{4}$. We have $1 + \rho^{2} \geq 16 r^{2}$, and $1+\lambda(t)^{2\alpha -2} \rho^{2} \geq 16 \lambda(t)^{2\alpha -2}r^{2}$. Whence, we have
$$\frac{1}{(4r^{2}+(1-r^{2}+\rho^{2})^{2})^{3/2}} \leq \frac{C}{(1+\rho^{2})^{3}}$$
$$\frac{1}{(4 \lambda(t)^{2\alpha -2} r^{2} +(1+\lambda(t)^{2\alpha -2} \rho^{2} - \lambda(t)^{2\alpha -2}r^{2})^{2})^{3/2}} \leq \frac{C}{(1+\rho^{2}\lambda(t)^{2\alpha -2})^{3}}$$
So,
\begin{equation}\begin{split} &|-\frac{1}{r} \int_{t}^{\infty} ds \int_{0}^{s-t} \frac{\rho d\rho}{(s-t)\sqrt{(s-t)^{2}-\rho^{2}}}\lambda'''(s) \left(\frac{8 \rho^{2} r^{2}}{(4r^{2}+(1+\rho^{2}-r^{2})^{2})^{3/2}} \right.\\
&+\left. \frac{-8 \lambda(s)^{4\alpha -4}\rho^{2} r^{2}}{(4 \lambda(s)^{2\alpha -2} r^{2}+(1+\lambda(s)^{2\alpha -2}(\rho^{2}-r^{2}))^{2})^{3/2}}\right)|\\
&\leq  \frac{C}{r} \int_{0}^{\infty} \frac{d\rho}{t^{3}\log^{b+1}(t)} \left(\frac{\rho^{2} r^{2}}{(1+\rho^{2})^{3}} + \frac{\rho^{2} r^{2}}{(1+\lambda(t)^{2\alpha -2}\rho^{2})^{3}\log^{4\alpha b-4b}(t)}\right)\\
&\leq \frac{C r}{t^{3} \log^{1+b\alpha}(t)}, \quad r \leq \frac{\lambda(t)^{1-\alpha}}{4}\end{split}\end{equation}
If $ r \geq \frac{\lambda(t)^{1-\alpha}}{4}$, then,
\begin{equation}\begin{split} &|-\frac{1}{r} \int_{t}^{\infty} ds \int_{0}^{s-t} \frac{\rho d\rho}{(s-t)\sqrt{(s-t)^{2}-\rho^{2}}}\lambda'''(s) \left(\frac{8 \rho^{2} r^{2}}{(4r^{2}+(1+\rho^{2}-r^{2})^{2})^{3/2}} \right.\\
&+\left. \frac{-8 \lambda(s)^{4\alpha -4}\rho^{2} r^{2}}{(4 \lambda(s)^{2\alpha -2} r^{2}+(1+\lambda(s)^{2\alpha -2}(\rho^{2}-r^{2}))^{2})^{3/2}}\right)|\\
&\leq \frac{C}{r} \left(\frac{r^{2}}{t^{3} \log^{b+1}(t)} \int_{0}^{4r} \frac{\rho^{2} d\rho}{(4r^{2}+(1-r^{2}+\rho^{2})^{2})^{3/2}} + \frac{r^{2}}{t^{3}\log^{b+1}(t)\log^{4\alpha b-4b}(t)} \int_{0}^{4r} \frac{\rho^{2} d\rho}{(4 \lambda(t)^{2\alpha -2} r^{2} + (1+\lambda(t)^{2\alpha -2}(\rho^{2}-r^{2}))^{2})^{3/2}}\right)\\
&+\frac{C}{r t^{3} \log^{b+1}(t)} \int_{4r}^{\infty} d\rho \left(\frac{\rho^{2} r^{2}}{\rho^{6}} + \frac{\rho^{2}r^{2}}{\log^{4\alpha b-4b}(t)\lambda(t)^{6\alpha -6} \rho^{6}}\right)\\
&\leq \frac{C}{r} \left(\frac{r^{3}}{t^{3} \log^{b+1}(t)} \int_{0}^{\infty} \frac{\rho d\rho}{(4r^{2}+(1-r^{2}+\rho^{2})^{2})^{3/2}} + \frac{r^{3} }{t^{3}\log^{b+1}(t)\log^{4\alpha b-4b}(t)} \int_{0}^{\infty} \frac{\rho d\rho}{(4 \lambda(t)^{2\alpha -2} r^{2} + (1+\lambda(t)^{2\alpha -2}(\rho^{2}-r^{2}))^{2})^{3/2}}\right)\\
&+ \frac{C}{r t^{3} \log^{b+1}(t)} \left(\frac{1}{r} + \frac{1}{r \log^{2b-2\alpha b}(t)}\right)\\
&\leq \frac{C}{t^{3} \log^{-b+2\alpha b+1}(t)}\end{split}\end{equation}

In total, we get
\begin{equation} |\partial_{t}v_{3,1,t}(t,r)| \leq \frac{C}{t^{3} \log^{1+2\alpha b-b}(t)}\end{equation}

For the $t$ derivative of the second line of \eqref{dtv3forlambdapppp}, we start with 
\begin{equation} |\partial_{3}^{2}F_{3}(r,\rho,\lambda(s))| \leq \frac{C |\partial_{3}F_{3}(r,\rho,\lambda(s))|}{\lambda(s)} + \frac{C \lambda(s)^{4\alpha -6} r^{2}}{\lambda(s)^{2\alpha -2}\left(1+(\rho^{2}-r^{2})^{2}\lambda(s)^{4\alpha -4} + 2(\rho^{2}+r^{2})\lambda(s)^{2\alpha -2}\right)}\end{equation}
Then, we use our estimate for $\partial_{3}F_{3}$ proven while estimating $\partial_{t}v_{3}$ previously, and note that the estimate above for $\partial_{3}^{2}F_{3}$ gives rise to an estimate of the $t$ derivative of the integrand of the second line of \eqref{dtv3forlambdapppp} which  is of the same form as the estimate on the integrand used to obtain the final estimate on $\partial_{t}v_{3}$. We can therefore read off an estimate from our previous computations, and get that the partial $t$ derivative of the second line of \eqref{dtv3forlambdapppp} is bounded above in absolute value by 
$$\frac{C r}{t^{4}\log^{b+1}(t)}$$
This gives \eqref{dttv3prelimest}.\\
For $\partial_{t}^{2}v_{1}$, we have
\begin{equation}\label{dttv1prelim}\begin{split} \partial_{tt}v_{1}(t,r) &= \int_{t}^{\infty} ds \frac{-\lambda'''(s)}{r} \int_{0}^{s-t} \frac{\rho d\rho}{\sqrt{(s-t)^{2}-\rho^{2}}(s-t)} \left(1+\frac{r^{2}-1-\rho^{2}}{\sqrt{(r^{2}-1-\rho^{2})^{2}+4r^{2}}}\right)\\
&+\int_{t}^{\infty} ds \frac{\lambda'''(s)}{r} \int_{0}^{s-t} \frac{\rho d\rho}{\sqrt{(s-t)^{2}-\rho^{2}}} \frac{8 \rho^{2}r^{2}}{(s-t)(4r^{2}+(1+\rho^{2}-r^{2})^{2})^{3/2}}\end{split}\end{equation}

For the first line of \eqref{dttv1prelim},we get
\begin{equation}\begin{split} &\int_{t}^{\infty} ds \frac{-\lambda'''(s)}{r} \int_{0}^{s-t} \frac{\rho d\rho}{(s-t) \sqrt{(s-t)^{2}-\rho^{2}}} \left(1+\frac{r^{2}-1-\rho^{2}}{\sqrt{(r^{2}-1-\rho^{2})^{2}+4r^{2}}}\right)\\
&=-\int_{0}^{\infty} \rho d\rho \left(1+\frac{r^{2}-1-\rho^{2}}{\sqrt{(r^{2}-1-\rho^{2})^{2}+4r^{2}}}\right) \int_{\rho +t}^{\infty} ds \frac{\lambda'''(s)}{r(s-t)\sqrt{(s-t)^{2}-\rho^{2}}}\end{split}\end{equation}
Next, if $$f_{v_{1},t}(x,y) = 1+\frac{x-y}{\sqrt{(x-y)^{2}+4x}}, \quad x\geq 0, y \geq 1$$
then,
$$|f_{v_{1},t}(x,y)| \leq \frac{C x}{y^{2}}, \quad x \leq \frac{y}{16}$$
Also, we have $$|f_{v_{1},t}(x,y)| \leq 2$$
Note that if $r \leq \frac{1}{4}$, then, $r^{2} \leq \frac{1+\rho^{2}}{16}$. So, we get
\begin{equation}\begin{split} &|\int_{t}^{\infty} ds \frac{-\lambda'''(s)}{r} \int_{0}^{s-t} \frac{\rho d\rho}{(s-t) \sqrt{(s-t)^{2}-\rho^{2}}} \left(1+\frac{r^{2}-1-\rho^{2}}{\sqrt{(r^{2}-1-\rho^{2})^{2}+4r^{2}}}\right)|\\
&\leq C \frac{\sup_{x \geq t} |\lambda'''(x)|}{r} \begin{cases} \int_{0}^{\infty} d\rho |f_{v_{1},t}(r^{2},1+\rho^{2})|, \quad r < \frac{1}{4}\\
 \int_{0}^{4r} 2 d\rho + C \int_{4r}^{\infty} \frac{r^{2}d\rho}{\rho^{4}}, \quad r > \frac{1}{4}\end{cases}\\
&\leq C \sup_{x \geq t}|\lambda'''(x)| \begin{cases} r, \quad r< \frac{1}{4}\\
1, \quad r>\frac{1}{4}\end{cases}\\
&\leq C \frac{r}{1+r} \sup_{x \geq t}|\lambda'''(x)|\end{split}\end{equation}
For the second line of \eqref{dttv1prelim}, we use a similar procedure as above, again treating first the case $r \leq \frac{1}{4}$, and get
\begin{equation}\begin{split}&|\int_{t}^{\infty} ds \frac{\lambda'''(s)}{r} \int_{0}^{s-t} \frac{\rho d\rho}{\sqrt{(s-t)^{2}-\rho^{2}}} \frac{8 \rho^{2} r^{2}}{(s-t) (4r^{2}+(1-r^{2}+\rho^{2})^{2})^{3/2}}|\\
&\leq \frac{C \sup_{x \geq t}|\lambda'''(x)|}{r} \int_{0}^{\infty} \frac{\rho d\rho \rho^{2} r^{2}}{(4r^{2}+(1-r^{2}+\rho^{2})^{2})^{3/2}} \int_{\rho +t}^{\infty} \frac{ds}{\sqrt{(s-t)^{2}-\rho^{2}}} \frac{1}{(s-t)}\\
&\leq C \frac{\sup_{x \geq t} |\lambda'''(x)|}{r} r^{2} \int_{0}^{\infty} \frac{\rho^{2} d\rho}{(1+\rho^{2})^{3}}\\
&\leq C r \sup_{x \geq t}|\lambda'''(x)|, \quad r \leq \frac{1}{4}\end{split}\end{equation}
where we used the fact that 
$$r \leq \frac{1}{4} \implies 1+\rho^{2}-r^{2} \geq C(1+\rho^{2})$$
Next, we consider the second line of \eqref{dttv1prelim} for $ r > \frac{1}{4}$. Here, we get
\begin{equation}\begin{split} &|\int_{t}^{\infty} ds \frac{\lambda'''(s)}{r} \int_{0}^{s-t} \frac{\rho d\rho}{\sqrt{(s-t)^{2}-\rho^{2}}} \frac{8 \rho^{2} r^{2}}{(s-t) (4r^{2}+(1-r^{2}+\rho^{2})^{2})^{3/2}}|\\
&\leq C \frac{\sup_{x \geq t}|\lambda'''(x)|}{r} \left(\int_{0}^{2r} \frac{\rho^{2} r^{2}d\rho}{(4r^{2}+(1-r^{2}+\rho^{2})^{2})^{3/2}} + \int_{2r}^{\infty} \frac{\rho^{2} r^{2} d\rho}{(4r^{2}+(1-r^{2}+\rho^{2})^{2})^{3/2}}\right)\\
&\leq \frac{C \sup_{x \geq t}|\lambda'''(x)|}{r} \left(2 r \int_{0}^{2r} \frac{r^{2} \rho d\rho}{(4r^{2}+(1-r^{2}+\rho^{2})^{2})^{3/2}} + r^{2} \int_{2r}^{\infty} \frac{\rho^{2} d\rho}{\rho^{6}}\right)\\
&\leq C \sup_{x \geq t} |\lambda'''(x)|, \quad  r > \frac{1}{4}\end{split}\end{equation}
In total, we get \eqref{dttv1prelimest}.\\
Now, we will obtain an estimate on $\partial_{t}^{2}v_{4}$. We write
\begin{equation} v_{4}=v_{4}^{0}(t,r)+v_{4}^{1}(t,r)\end{equation}
where $v_{4}^{0}$ solves the same equation as $v_{4}$, with $0$ Cauchy data at infinity, except with right-hand side equal to $v_{4,c}^{0}$, which is given by
$$v_{4,c}^{0}(t,r) = \chi_{\geq 1}(\frac{2r}{\log^{N}(t)}) \left(\left(\frac{\cos(2Q_{1}(\frac{r}{\lambda(t)}))-1}{r^{2}}\right)\left(v_{1}+v_{2}+v_{3}\right)\right)$$
and similarly for $v_{4}^{1}$, where
$$v_{4,c}^{1}(t,r) = \chi_{\geq 1}(\frac{2r}{\log^{N}(t)}) F_{0,2}(t,r)$$ 
We start with
\begin{equation}\label{dtv4fordttv4prelim} \partial_{t}^{2}v_{4}^{0}(t,r) = \frac{-1}{2\pi} \int_{t}^{\infty} ds \int_{0}^{s-t} \frac{\rho d\rho}{\sqrt{(s-t)^{2}-\rho^{2}}} \int_{0}^{2\pi} d\theta \frac{\partial_{1}^{2}v_{4,c}^{0}(s,\sqrt{r^{2}+\rho^{2}+2 r \rho \cos(\theta)}) \left(r+\rho \cos(\theta)\right)}{\sqrt{r^{2}+\rho^{2}+2 r \rho \cos(\theta)}}\end{equation}

Then, we note that
\begin{equation} \begin{split} |\partial_{t}^{2} v_{4,c}^{0}(t,r)| &\leq C \frac{\mathbbm{1}_{\{r \geq \frac{\log^{N}(t)}{2}\}} |v_{1}+v_{2}+v_{3}|}{t^{2} \log^{2b+1}(t) r^{4}} + \frac{C \mathbbm{1}_{\{r \geq \frac{\log^{N}(t)}{2}\}}|\partial_{t}(v_{1}+v_{2}+v_{3})|}{t \log^{2b+1}(t) r^{4}}\\
&+\frac{C \mathbbm{1}_{\{r \geq \frac{\log^{N}(t)}{2}\}} |\partial_{t}^{2}(v_{1}+v_{2}+v_{3})|}{r^{4} \log^{2b}(t)}\end{split}\end{equation}
For the purposes of obtaining a preliminary estimate on $\partial_{t}^{2}v_{4}^{0}$ in the region $r \leq \frac{t}{2}$, we combine all of our previous estimates in the following way:
\begin{equation} |\partial_{t}^{2}v_{4,c}^{0}(t,r)| \leq \mathbbm{1}_{\{r \geq \frac{\log^{N}(t)}{2}\}} \begin{cases}\frac{C}{r^{4} \log^{2b}(t) t^{3}}\left(\frac{1}{\log^{1+2b\alpha-b}(t)}+\frac{1}{\log^{b}(t)}\right), \quad r \leq \frac{t}{2}\\
\frac{C \log(r)}{r^{4} \log^{2b}(t)} \left(\frac{1}{|t-r|^{3}} + \frac{1}{t \log(t) (t-r)^{2}}+\frac{1}{t^{2} \log(t) |t-r|}\right)+\frac{C}{r^{4} t^{3} \log^{1+b+2b\alpha}(t)}, \quad t > r > \frac{t}{2}\end{cases}\end{equation}
Directly inserting this estimate into our previous formula for $\partial_{t}^{2}v_{4}^{0}(t,r)$, and estimating as in previous sections, we get
\begin{equation}|\partial_{t}^{2}v_{4}^{0}(t,r)| \leq \frac{C}{t^{3}\log^{2N+b-2}(t)}, \quad r \leq \frac{t}{2}\end{equation}
We will use the same procedure used for $\partial_{t}v_{4}$ to estimate $\partial_{t}^{2}v_{4}^{0}$ in the region $r \geq \frac{t}{2}$. As in previous cases, we will use a different combination of $\partial_{t}^{j}v_{2}$ estimates to get a different estimate for $\partial_{t}^{2}v_{4,c}^{0}$:
\begin{equation} |\partial_{t}^{2}v_{4,c}^{0}(t,r)| \leq  \begin{cases}\frac{C \mathbbm{1}_{\{r \geq \frac{\log^{N}(t)}{2}\}}}{r^{4} \log^{b}(t) t^{3}}, \quad r \leq \frac{t}{2}\\
\frac{C}{\log^{2b}(t) r^{9/2}}, \quad t-t^{1/6} \leq r \leq t+t^{1/6}\\
\frac{C \log(r)}{r^{4} \log^{2b}(t)} \left(\frac{1}{|t-r|^{3}} + \frac{1}{t \log(t) (t-r)^{2}}+\frac{1}{t^{2} \log(t) |t-r|}\right)+\frac{C}{r^{4} t^{3} \log^{1+b+2b\alpha}(t)}, \quad t-t^{1/6} > r > \frac{t}{2}\text{, or }r>t+t^{1/6}\end{cases}\end{equation}
We then get
\begin{equation} |\partial_{t}^{2}v_{4}^{0}(t,r)| \leq \frac{C}{t^{2}\log^{3N+b}(t)}, \quad r \geq \frac{t}{2}\end{equation}
which gives
\begin{equation} |\partial_{t}^{2}v_{4}^{0}(t,r)|\leq \begin{cases}\frac{C}{t^{3}\log^{2N+b-2}(t)}, \quad r\leq \frac{t}{2}\\
\frac{C}{t^{2}\log^{3N+b}(t)}, \quad r \geq \frac{t}{2}\end{cases}\end{equation}
Now, we estimate $\partial_{t}^{2}v_{4}^{1}$. This time, we follow the procedure used to obtain the preliminary estimate on $\partial_{t}v_{4}$:
\begin{equation}\label{dttv41prelim}\begin{split} \partial_{t}^{2} v_{4}^{1}(t,r) &= \frac{1}{2\pi}\int_{t}^{\infty} \frac{ds}{(s-t)} \int_{0}^{s-t} \frac{\rho d\rho}{\sqrt{(s-t)^{2}-\rho^{2}}} \int_{0}^{2\pi} d\theta \frac{\partial_{1}v_{4,c}^{1}(s,\sqrt{r^{2}+\rho^{2}+2 r \rho \cos(\theta)}) (r+\rho \cos(\theta))}{\sqrt{r^{2}+\rho^{2}+2 r \rho \cos(\theta)}}\\
&-\frac{1}{2\pi} \int_{t}^{\infty} ds \int_{B_{s-t}(0)} \frac{dA(y)}{\sqrt{(s-t)^{2}-|y|^{2}}} \text{integrand}_{v_{4,2}^{1}}\end{split}\end{equation}
where
\begin{equation}\begin{split}\text{integrand}_{v_{4,2}^{1}} &= \frac{-\partial_{12} v_{4,c}^{1}(s,|x+y|)}{|x+y|^{2}(s-t)} ((x+y)\cdot y)(\widehat{x}\cdot(x+y))\\
&+\frac{\partial_{1}v_{4,c}^{1}(s,|x+y|)}{(s-t)|x+y|} \left(-y\cdot \widehat{x} + \frac{(\widehat{x}\cdot(x+y))(y\cdot(y+x))}{|x+y|^{2}}\right)\end{split}\end{equation}
which gives
\begin{equation} |\text{integrand}_{v_{4,2}^{1}}| \leq C |\partial_{12}v_{4,c}^{1}| + C \frac{|\partial_{1}v_{4,c}^{1}|}{|x+y|}\end{equation}
We start with the term on the first line of \eqref{dttv41prelim}:
\begin{equation}\begin{split}&|\frac{1}{2\pi}\int_{t}^{\infty} \frac{ds}{(s-t)} \int_{0}^{s-t} \frac{\rho d\rho}{\sqrt{(s-t)^{2}-\rho^{2}}} \int_{0}^{2\pi} d\theta \frac{\partial_{1}v_{4,c}^{1}(s,\sqrt{r^{2}+\rho^{2}+2 r \rho \cos(\theta)}) (r+\rho \cos(\theta))}{\sqrt{r^{2}+\rho^{2}+2 r \rho \cos(\theta)}}|\\
&\leq C \int_{t}^{\infty} \frac{ds}{(s-t)} \int_{0}^{s-t} \frac{\rho d\rho}{\sqrt{(s-t)^{2}-\rho^{2}}} \int_{0}^{2\pi} \frac{d\theta}{s^{3} \log^{3b+1+N-2b\alpha}(s)(\log^{2N}(s)+\rho^{2}+r^{2}+2 r \rho \cos(\theta))}\\
&\leq C \int_{0}^{\infty} \rho d\rho \int_{0}^{2\pi} \frac{d\theta}{(\rho+t)^{3} \log^{3b+1+N-2b\alpha}(t)(\log^{2N}(t)+r^{2}+\rho^{2}+2r \rho \cos(\theta))} \int_{\rho+t}^{\infty} \frac{ds}{(s-t)\sqrt{(s-t)^{2}-\rho^{2}}}\\
&\leq \frac{C}{t^{3} \log^{3b+2N-2b\alpha}(t)}\end{split}\end{equation}
For the second line of \eqref{dttv41prelim}, we use the same procedure used to estimate $v_{4}$ previously, and we get
\begin{equation}\begin{split}|-\frac{1}{2\pi} \int_{t}^{\infty} ds \int_{B_{s-t}(0)} \frac{dA(y)}{\sqrt{(s-t)^{2}-|y|^{2}}} \text{integrand}_{v_{4,2}^{1}}| &\leq \frac{C}{t^{3} \log^{3b-1+2N-2\alpha b}(t)}\end{split}\end{equation}
In total, we get
\begin{equation} |\partial_{t}^{2}v_{4}^{1}(t,r)| \leq \frac{C}{t^{3} \log^{3b+2N-1-2\alpha b}(t)}, \quad r \geq 0\end{equation}
Combining with the estimate on $\partial_{t}^{2}v_{4}^{0}$ above, we get \eqref{dttv4prelimest}.\\
Next, we estimate $\partial_{t}^{2}v_{5}$. We start with
\begin{equation}\label{dttn2f}\begin{split}|\partial_{t}^{2}N_{2}(f)(t,r)| &\leq \left(\frac{C \lambda(t) \lambda'(t)^{2}}{r(r^{2}+\lambda(t)^{2})^{2}} + \frac{C |\lambda''(t)|}{r(r^{2}+\lambda(t)^{2})}\right)|f(t,r)|^{2}\\
&+\frac{C |\lambda'(t)|}{r(r^{2}+\lambda(t)^{2})} |f(t,r)\partial_{t}f(t,r)| + \frac{C \lambda(t)}{r(r^{2}+\lambda(t)^{2})} \left((\partial_{t}f(t,r))^{2}+|f(t,r)\partial_{tt}f(t,r)|\right)\\
&+\frac{C \left((\lambda'(t))^{2}+\lambda(t)|\lambda''(t)|\right)}{(r^{2}+\lambda(t)^{2})^{2}} |f(t,r)^{3}| + \frac{C \lambda(t)|\lambda'(t)|}{(r^{2}+\lambda(t)^{2})^{2}} |f(t,r)^{2}\partial_{t}f(t,r)|\\
&+\frac{C}{r^{2}}\left(|f(t,r)| (\partial_{t}f(t,r))^{2}+(f(t,r))^{2}|\partial_{tt}f(t,r)|\right)\end{split}\end{equation}
where $f=v_{1}+v_{2}+v_{3}+v_{4}$.\\ 
For the purposes of estimating $\partial_{t}^{2}v_{5}$ in the region $r \leq \frac{t}{2}$, we return to \eqref{dttn2f}, and use \eqref{v2singularconeest} for all quantities involving $v_{2}$ in the region $r \geq \frac{t}{2}$, with the following exceptions. For the term involving $|f_{v_{5}}| \cdot |\partial_{t}f_{v_{5}}|^{2}$ on the last line of \eqref{dttn2f}, we take the following average of \eqref{v2singularconeest} and \eqref{v2sqrtrest} to estimate $(\partial_{t}v_{2})^{2}$ in the region $r \geq \frac{t}{2}$:
$$(\partial_{t}v_{2}(t,r))^{2} \leq C \left(\frac{\log(r)}{|t-r|^{2}}\right)^{3/2} \cdot \frac{1}{r^{1/4}} $$ 
Similarly, for the term involving $f_{v_{5}}^{2} |\partial_{tt}f_{v_{5}}|$, we estimate $v_{2}^{2}|\partial_{tt}v_{2}|$ by
$$(v_{2}(t,r))^{2} |\partial_{tt}v_{2}| \leq \frac{C \log(r)^{2}}{|t-r|^{4} \sqrt{r}}$$
(This is so that we will not have to have an extra argument for a term analogous to the last line of \eqref{v52intest}). This procedure gives
\begin{equation} |\partial_{t}^{2}N_{2}(f)(t,r)| \leq C \begin{cases} \frac{1}{t^{5} \log^{b}(t)(r^{2}+\lambda(t)^{2})}, \quad r \leq \frac{t}{2}\\
\frac{\log(r)}{r^{3} t^{2} |t-r| \log^{3N+2b}(t)} + \frac{\log^{2}(r)}{r^{2}(t-r)^{2} t^{2} \log^{3N+b}(t)} + \frac{\log^{3}(r)}{t^{2} r^{2} \log(t) |t-r|^{3}} + \frac{\log^{5/2}(r)}{r^{9/4} (t-r)^{4}}, \quad \frac{t}{2} \leq r <t\end{cases}\end{equation}
Using these estimates in the formula
\begin{equation}\begin{split} \partial_{t}^{2}v_{5}(t,r) &= \frac{-1}{2\pi} \int_{t}^{\infty}ds \int_{0}^{s-t} \frac{\rho d\rho}{\sqrt{(s-t)^{2}-\rho^{2}}} \int_{0}^{2\pi}d\theta \frac{\partial_{1}^{2}N_{2}(f)(s,\sqrt{r^{2}+\rho^{2}+2 r \rho \cos(\theta)})}{\sqrt{r^{2}+\rho^{2}+2 r \rho \cos(\theta)}}(r+\rho \cos(\theta))\end{split}\end{equation}
gives \eqref{dttv5prelimest}.
\end{proof}
 Then, we use all of our estimates above to get
\begin{equation} |\partial_{t}^{2}\left(\langle\left(\frac{\cos(2Q_{\frac{1}{\lambda(t)}}(r))-1}{r^{2}}\right)v_{3}\vert_{r=R\lambda(t)}, \phi_{0}\rangle\right)| \leq \frac{C}{t^{3}\log^{1+2\alpha b-3b}(t)}\end{equation}
\begin{equation} |\partial_{t}^{2}\left(\langle\left(\frac{\cos(2Q_{\frac{1}{\lambda(t)}}(r))-1}{r^{2}}\right)v_{4}\left(1-\chi_{\geq 1}(\frac{4r}{t})\right)\vert_{r=R\lambda(t)}, \phi_{0}\rangle\right)| \leq \frac{C}{t^{3}\log^{2N-2-b}(t)}\end{equation}
\begin{equation} |\partial_{t}^{2}\left(\langle\left(\frac{\cos(2Q_{\frac{1}{\lambda(t)}}(r))-1}{r^{2}}\right)v_{5}\left(1-\chi_{\geq 1}(\frac{4r}{t})\right)\vert_{r=R\lambda(t)}, \phi_{0}\rangle\right)| \leq \frac{C}{t^{4}\log^{3N-b-2}(t)}\end{equation}
\begin{equation} |\partial_{t}^{2}\left(\langle\left(\frac{\cos(2Q_{\frac{1}{\lambda(t)}}(r))-1}{r^{2}}\right)\left(\chi_{\geq 1}(\frac{2r}{\log^{N}(t)})(v_{1}+v_{2}+v_{3})\right)\vert_{r=R\lambda(t)}, \phi_{0}\rangle\right)| \leq \frac{C}{t^{3}\log^{1+3N+2\alpha b}(t)}\end{equation}

Finally, assembling together all of our previous estimates, we conclude
$$|RHS_{2}''(t)| \leq \frac{C}{t^{3}\log^{1-3b}(t)}$$
Returning to \eqref{lambdappppprelim1}, we get
\begin{equation} |\lambda''''(t)| \leq \frac{C}{t^{3}\log^{1-2b}(t)\log(\log(t))}\end{equation}
Now that we have this preliminary estimate, we return to \eqref{modulationfinal}, substitute 
$$\lambda(t) = \lambda_{0,0}(t)+e(t)$$
and differentitate twice in $t$, to get
\begin{equation}\label{eppppforfinalest}\begin{split} &-4 \int_{t}^{\infty} \frac{e''''(s) ds}{\log(\lambda_{0,0}(s))(1+s-t)} + 4 \alpha e''''(t) - 4 \int_{t}^{\infty} \frac{e''''(s) ds}{\log(\lambda_{0,0}(s))(\lambda_{0,0}(t)^{1-\alpha}+s-t)(1+s-t)^{3}}\\
&=\frac{-E_{\lambda_{0,0}}''(t)}{\log(\lambda_{0,0}(t))} -\frac{4\alpha}{\log(\lambda_{0,0}(t))}\left(\frac{\lambda''(t)}{\lambda(t)}-\frac{\lambda'(t)^{2}}{\lambda(t)^{2}}\right)\lambda''(t) - \frac{8 \alpha \lambda'(t) \lambda'''(t)}{\lambda(t) \log(\lambda_{0,0}(t))}\\
&-4 \alpha \left(\frac{\log(\lambda(t))}{\log(\lambda_{0,0}(t))}-1\right) e''''(t) -\frac{4 \alpha \log(\lambda(t))}{\log(\lambda_{0,0}(t))} \lambda_{0,0}''''(t) + \frac{4}{\log(\lambda_{0,0}(t))} \partial_{t}^{2}\left(\int_{0}^{\infty} \frac{\lambda_{0,0}''(t+w)}{(\lambda(t)^{1-\alpha}+w)(1+w^{3})} dw\right)\\
&+4 \int_{0}^{\infty} \frac{e''''(t+w)}{\log(\lambda_{0,0}(t))} \left(\frac{1}{(\lambda(t)^{1-\alpha}+w)}-\frac{1}{(\lambda_{0,0}(t)^{1-\alpha}+w)}\right)\frac{dw}{(1+w)^{3}} \\
&-\frac{8(1-\alpha)\lambda(t)^{-\alpha}\lambda'(t)}{\log(\lambda_{0,0}(t))} \int_{0}^{\infty} \frac{e'''(t+w)}{(\lambda(t)^{1-\alpha}+w)^{2}(1+w)^{3}} dw\\
&-\frac{4(1-\alpha)(-\alpha \lambda(t)^{-\alpha -1}(\lambda'(t))^{2} + \lambda(t)^{-\alpha} \lambda''(t))}{\log(\lambda_{0,0}(t))} \int_{0}^{\infty} \frac{e''(t+w)dw}{(1+w)^{3}(\lambda(t)^{1-\alpha}+w)^{2}}\\
&+\frac{8}{\log(\lambda_{0,0}(t))} \int_{0}^{\infty} \frac{e''(t+w)}{(1+w)^{3}} \frac{(1-\alpha)^{2} \lambda(t)^{-2\alpha}\lambda'(t)^{2} dw}{(\lambda(t)^{1-\alpha}+w)^{3}} + \frac{\partial_{t}^{2} G(t,\lambda_{0,0}(t)+e(t))}{\log(\lambda_{0,0}(t))}\\
&+4 \int_{t}^{\infty} e''''(s) \left(\frac{1}{\log(\lambda_{0,0}(t))}-\frac{1}{\log(\lambda_{0,0}(s))}\right) \frac{ds}{(1+s-t)}\\
&+4 \int_{t}^{\infty} \frac{e''''(s) ds}{(1+s-t)^{3} (\lambda_{0,0}(t)^{1-\alpha}+s-t)}\left(\frac{1}{\log(\lambda_{0,0}(t))}-\frac{1}{\log(\lambda_{0,0}(s))}\right)\\
&:=RHS_{4}(t)\end{split}\end{equation}
Like before, we start by estimating all the terms on the right-hand side of \eqref{eppppforfinalest} which do not involve $G$:
\begin{equation} |\frac{-E_{\lambda_{0,0}}''(t)}{\log(\lambda_{0,0}(t))}| \leq \frac{C}{t^{4} \log^{b+1}(t) \log(\log(t))}\end{equation}
\begin{equation} |\frac{-4 \alpha}{\log(\lambda_{0,0}(t))} \left(\frac{\lambda''(t)}{\lambda(t)} -\frac{\lambda'(t)^{2}}{\lambda(t)^{2}}\right)\lambda''(t)| \leq \frac{C}{t^{4} \log^{b+2}(t) \log(\log(t))}\end{equation}
\begin{equation} |\frac{-8 \alpha \lambda'(t) \lambda'''(t)}{\lambda(t) \log(\lambda_{0,0}(t))}| \leq \frac{C}{t^{4} \log^{b+2}(t) \log(\log(t))}\end{equation}
\begin{equation} |-4 \alpha \left(\frac{\log(\lambda_{0,0}(t) + e(t))}{\log(\lambda_{0,0}(t))} -1\right) e''''(t)| \leq \frac{C |e''''(t)|}{(\log(\log(t)))^{3/2}}\end{equation}
\begin{equation} |\frac{-4 \alpha \log(\lambda(t))}{\log(\lambda_{0,0}(t))} \lambda_{0,0}''''(t)| \leq \frac{C}{t^{4} \log^{b+1}(t)}\end{equation}
Then, we note that the following estimate:
\begin{equation} |\partial_{t}^{2}\left(\frac{\lambda_{0,0}''(t+w)}{w+\lambda(t)^{1-\alpha}}\right)| \leq \frac{C}{t^{2} (t+w)^{2} \log^{b+1}(t+w) (w+\lambda(t)^{1-\alpha})}\end{equation}
implies
\begin{equation} |\frac{4}{\log(\lambda_{0,0}(t))} \int_{0}^{\infty} \partial_{t}^{2} \left(\frac{\lambda_{0,0}''(t+w)}{w+\lambda(t)^{1-\alpha}}\right) \frac{dw}{(1+w)^{3}}| \leq \frac{C}{t^{4} \log^{b+1}(t)}\end{equation}
Next, we have
\begin{equation} |4 \int_{0}^{\infty} \frac{e''''(t+w)}{\log(\lambda_{0,0}(t))} \left(\frac{\lambda_{0,0}(t)^{1-\alpha}-\lambda(t)^{1-\alpha}}{(\lambda(t)^{1-\alpha}+w)(\lambda_{0,0}(t)^{1-\alpha}+w)}\right) \frac{dw}{(1+w)^{3}}| \leq \frac{C \sup_{x \geq t} |e''''(x)|}{\sqrt{\log(\log(t))}}\end{equation}
\begin{equation} |\frac{-8(1-\alpha) \lambda(t)^{-\alpha} \lambda'(t)}{\log(\lambda_{0,0}(t))} \int_{0}^{\infty} \frac{e'''(t+w) dw}{(\lambda(t)^{1-\alpha}+w)^{2}(1+w)^{3}}| \leq \frac{C}{t^{4} \log^{b+2}(t)}\end{equation}
\begin{equation} |\frac{-4 (1-\alpha)(-\alpha \lambda(t)^{-\alpha -1}(\lambda'(t))^{2} + \lambda(t)^{-\alpha} \lambda''(t))}{\log(\lambda_{0,0}(t))} \int_{0}^{\infty} \frac{e''(t+w) dw}{(1+w)^{3} (\lambda(t)^{1-\alpha}+w)^{2}}| \leq \frac{C}{t^{4} \log^{b+2}(t)}\end{equation}
\begin{equation} |\frac{8}{\log(\lambda_{0,0}(t))} \int_{0}^{\infty} \frac{e''(t+w)}{(1+w)^{3}} \frac{(1-\alpha)^{2} \lambda(t)^{-2\alpha} \lambda'(t)^{2} dw}{(\lambda(t)^{1-\alpha} +w)^{3}}| \leq \frac{C}{t^{4} \log^{b+3}(t)}\end{equation}
\begin{equation} |4 \int_{t}^{\infty} e''''(s) \left(\frac{1}{\log(\lambda_{0,0}(t))} - \frac{1}{\log(\lambda_{0,0}(s))}\right) \frac{ds}{1+s-t}| \leq \frac{C \sup_{x \geq t}\left(x^{2} |e''''(x)|\right)}{t^{2} \log(t) (\log(\log(t)))^{2}}\end{equation}
and
\begin{equation} |4 \int_{t}^{\infty} \frac{e''''(s) ds}{(1+s-t)^{3} (\lambda_{0,0}(t)^{1-\alpha}+s-t)} \left(\frac{1}{\log(\lambda_{0,0}(t))}-\frac{1}{\log(\lambda_{0,0}(s))}\right)| \leq \frac{C \sup_{x \geq t} \left(|e''''(x) x^{2}\right)}{t^{3} \log(t) \log(\log(t))}\end{equation}
Combining these, we get that all of the terms on the right-hand side of \eqref{eppppforfinalest} which do not involve $G$ are bounded above in absolute value by
$$\frac{C}{t^{4} \log^{b+1}(t)} + \frac{C \sup_{x \geq t} \left(x^{2} |e''''(x)|\right)}{\sqrt{\log(\log(t))} t^{2}}$$
We now estimate terms in $\partial_{t}^{2}G(t,\lambda(t))$, starting with the term involving $K_{1}$. We again use the preliminary estimate on $e''''$ to justify any differentiations under the integral sign, to get
\begin{equation}\label{k1dttterms} \begin{split} &\partial_{t}^{2}\left(\frac{16}{\lambda(t)^{2}} \int_{t}^{\infty} dx \lambda''(x) \left(K_{1}(x-t,\lambda(t))-\frac{\lambda(t)^{2}}{4(1+x-t)}\right)\right)\\
&=\left(\frac{96 \lambda'(t)^{2}}{\lambda(t)^{4}}-\frac{32 \lambda''(t)}{\lambda(t)^{3}}\right)\int_{0}^{\infty} dw \lambda''(t+w) \left(K_{1}(w,\lambda(t))-\frac{\lambda(t)^{2}}{4(1+w)}\right)\\
&-\frac{64 \lambda'(t)}{\lambda(t)^{3}} \int_{0}^{\infty} dw \lambda'''(w+t)\left(K_{1}(w,\lambda(t))-\frac{\lambda(t)^{2}}{4(1+w)}\right) \\
&-\frac{64 \lambda'(t)}{\lambda(t)^{3}}\int_{0}^{\infty} dw \lambda''(w+t)\left(\partial_{2}K_{1}(w,\lambda(t))-\frac{\lambda(t)}{2(1+w)}\right)\lambda'(t)\\
&+\frac{16}{\lambda(t)^{2}}\int_{0}^{\infty} dw \lambda''''(w+t)\left(K_{1}(w,\lambda(t))-\frac{\lambda(t)^{2}}{4(1+w)}\right)\\
&+\frac{32}{\lambda(t)^{2}}\int_{0}^{\infty} dw \lambda'''(w+t)\left(\partial_{2}K_{1}(w,\lambda(t))-\frac{\lambda(t)}{2(1+w)}\right)\lambda'(t)\\
&+\frac{16}{\lambda(t)^{2}}\int_{0}^{\infty} dw \lambda''(w+t)\left(\left(\partial_{2}^{2}K_{1}(w,\lambda(t))\lambda'(t)-\frac{\lambda'(t)}{2(1+w)}\right)\lambda'(t)+\left(\partial_{2}K_{1}(w,\lambda(t))-\frac{\lambda(t)}{2(1+w)}\right)\lambda''(t)\right)\end{split}\end{equation}
The only term not involving $\lambda''''$ which has not already been estimated is the first term on the last line of the above expression. For $w \geq 1$, we start with
\begin{equation} K_{1}(w,\lambda(t))=\int_{0}^{\infty} \frac{r dr}{\lambda(t)^{2}(1+\frac{r^{2}}{\lambda(t)^{2}})^{3}} \int_{0}^{w} \frac{\rho d\rho}{w} \left(1+\frac{r^{2}-1-\rho^{2}}{\sqrt{(r^{2}-1-\rho^{2})^{2}+4r^{2}}}\right)\end{equation}
whence
\begin{equation}\begin{split} \partial_{2}^{2}K_{1}(w,\lambda(t))&=\int_{0}^{\infty}dr \frac{6 \lambda(t)^{2} r\left(\lambda(t)^{4}-5\lambda(t)^{2}r^{2}+2r^{4}\right)}{(\lambda(t)^{2}+r^{2})^{5}}\left(\frac{r^{2}}{w}-\int_{w}^{\infty} \frac{\rho d\rho}{w}\left(1+\frac{r^{2}-1-\rho^{2}}{\sqrt{(r^{2}-1-\rho^{2})^{2}+4r^{2}}}\right)\right)\\
&=\frac{1}{2w} -\int_{0}^{\infty}dr \frac{6 \lambda(t)^{2} r(\lambda(t)^{4}-5\lambda(t)^{2}r^{2}+2r^{4})}{(\lambda(t)^{2}+r^{2})^{5}}\int_{w}^{\infty} \frac{\rho d\rho}{w}\left(1+\frac{r^{2}-1-\rho^{2}}{\sqrt{(r^{2}-1-\rho^{2})^{2}+4r^{2}}}\right)\end{split}\end{equation}
Then,
\begin{equation} \begin{split} &|\frac{16}{\lambda(t)^{2}}\int_{1}^{\infty} dw \lambda''(w+t) \left(\partial_{22}K_{1}(w,\lambda(t))\lambda'(t)-\frac{\lambda'(t)}{2(1+w)}\right)\lambda'(t)|\\
&\leq \frac{C |\lambda'(t)|}{t^{3}\log^{2}(t)} \int_{1}^{\infty} dw \left(\frac{1}{2w}-\frac{1}{2(1+w)}\right)\\
&+\frac{C |\lambda'(t)|}{t^{3}\log^{2}(t)} \int_{0}^{\infty} \frac{r \lambda(t)^{2} dr}{(\lambda(t)^{2}+r^{2})^{3}} \int_{1}^{\infty} d\rho \int_{1}^{\rho} \frac{\rho dw}{w} \left(1+\frac{r^{2}-1-\rho^{2}}{\sqrt{(r^{2}-1-\rho^{2})^{2}+4r^{2}}}\right)\\
&\leq \frac{C}{t^{4}\log^{b+3}(t)} + \frac{C \lambda(t)^{2}}{t^{4}\log^{b+3}(t)} \int_{0}^{\infty} \frac{r^{3}\log(2+r^{2})dr}{(\lambda(t)^{2}+r^{2})^{3}}\\
&\leq \frac{C}{t^{4}\log^{b+3}(t)}\end{split}\end{equation}
For $w \leq 1$, we first change variables, and start with
\begin{equation} K_{1}(w,\lambda(t)) = \int_{0}^{\infty} \frac{R dR}{(1+R^{2})^{3}} \int_{0}^{w} \frac{\rho d\rho}{w} \left(1+\frac{R^{2}\lambda(t)^{2}-1-\rho^{2}}{\sqrt{(R^{2}\lambda(t)^{2}-1-\rho^{2})^{2}+4R^{2}\lambda(t)^{2}}}\right)\end{equation}
Then,
\begin{equation} |\partial_{2}^{2}K_{1}(w,\lambda(t))| \leq  C \int_{0}^{\infty} \frac{R dR}{(1+R^{2})^{3}} \int_{0}^{w}\frac{\rho d\rho}{w} R^{2}\end{equation}
which gives 
\begin{equation}  \begin{split} &|\frac{16}{\lambda(t)^{2}}\int_{0}^{1} dw \lambda''(w+t) \left(\partial_{22}K_{1}(w,\lambda(t))\lambda'(t)-\frac{\lambda'(t)}{2(1+w)}\right)\lambda'(t)|\\
&\leq \frac{C}{t^{3}\log^{2}(t)} \int_{0}^{1} dw \left(w+\frac{1}{1+w}\right)\frac{1}{t \log^{b+1}(t)} \leq \frac{C}{t^{4}\log^{b+3}(t)}\end{split}\end{equation}
We now consider the expression \eqref{k1dttterms}, except with all instances of $$K_{1}(x-t,\lambda(t))-\frac{\lambda(t)^{2}}{4(1+x-t)}$$ replaced with $$K(x-t,\lambda(t))$$ (This expression also appears in $\partial_{t}^{2}(G(t,\lambda(t)))$). All of the integrals involving $K$ which do not involve $\lambda''''$ have already been estimated, and we get
\begin{equation}\begin{split} &|\partial_{t}^{2}\left(\frac{16}{\lambda(t)^{2}} \int_{t}^{\infty} dx \lambda''(x) \left(K_{1}(x-t,\lambda(t))-\frac{\lambda(t)^{2}}{4(1+x-t)}\right)\right)|+|\partial_{t}^{2}\left(\frac{16}{\lambda(t)^{2}} \int_{t}^{\infty} dx \lambda''(x) K(x-t,\lambda(t))\right)|\\
&\leq \frac{C }{t^{4}\log^{b+1}(t)} +C \sup_{x \geq t}|e''''(x)|\end{split}\end{equation}
From the explicit expressions, we also get
\begin{equation} |\partial_{t}^{2}\left(-\lambda(t) E_{0,1}(\lambda(t),\lambda'(t),\lambda''(t))\right)| \leq \frac{C}{t^{4}\log^{b+1}(t)}+C |e''''(t)|\end{equation}
Next, we start with
\begin{equation} |\partial_{t}^{2}\left(K_{3}(w,\lambda(t))\right)| \leq \frac{C}{t^{2} \log(t)} |\frac{w}{1+w^{2}}-\frac{w}{\lambda(t)^{2-2\alpha}+w^{2}}| + \frac{C w}{t^{2} \log^{(2-2\alpha)b+1}(t) (\lambda(t)^{2-2\alpha}+w^{2})^{2}}\end{equation}
and
\begin{equation} |\partial_{t}^{2} K_{3,0}(w,\lambda(t))| \leq \frac{C}{t^{2} \log^{b+1-b\alpha}(t)} \frac{1}{(\lambda(t)^{1-\alpha}+w)^{2}(1+w)^{3}}\end{equation}
and use our previous estimates on $K_{3}-K_{3,0}$, and $\partial_{t}\left((K_{3}-K_{3,0})(w,\lambda(t))\right)$ to get
\begin{equation} |\partial_{t}^{2} \left(-16 \int_{t}^{\infty} \lambda''(s) \left(K_{3}(s-t,\lambda(t))-K_{3,0}(s-t,\lambda(t))\right) ds\right)| \leq C \sup_{x \geq t} |e''''(x)| + \frac{C}{t^{4} \log^{b+1}(t)}\end{equation}
Next, we will need an estimate on $\partial_{t}^{2} v_{3}$ which is different from that recorded while obtaining the preliminary estimate on $\lambda''''$. We start with
\begin{equation} \begin{split} \label{dttv3forlambdappppfinal} \partial_{t}^{2} v_{3}(t,r) &= \frac{-1}{r} \int_{0}^{\infty} dw \int_{0}^{w} \frac{\rho d\rho}{\sqrt{w^{2}-\rho^{2}}} \lambda''''(t+w) \left(\frac{-1-\rho^{2}+r^{2}}{\sqrt{(1+\rho^{2}-r^{2})^{2}+4r^{2}}}+F_{3}(r,\rho,\lambda(t+w))\right)\\
&-\frac{2}{r} \int_{0}^{\infty} dw \int_{0}^{w} \frac{\rho d\rho}{\sqrt{w^{2}-\rho^{2}}} \lambda'''(t+w) \partial_{t}(F_{3}(r,\rho,\lambda(t+w)))\\
&-\frac{1}{r} \int_{0}^{\infty} dw \int_{0}^{w} \frac{\rho d\rho}{\sqrt{w^{2}-\rho^{2}}} \lambda''(t+w) \partial_{t}^{2} (F_{3}(r,\rho,\lambda(t+w)))\end{split}\end{equation}
The second and third lines of \eqref{dttv3forlambdappppfinal} were already estimated in the course of obtaining the preliminary estimate on $\lambda''''$, and do not need to be estimated any differently here. The first line of \eqref{dttv3forlambdappppfinal}, with the replacement of $\lambda''''$ with $\lambda'''$ has already been estimated while obtaining the final estimate on $\lambda'''$. Therefore, we can read off that
\begin{equation}\label{dttv3finalsupest} |\partial_{t}^{2} v_{3}(t,r)| \leq C r \sup_{x \geq t} \left(\frac{x |e''''(x)|}{\lambda(x)^{2-2\alpha}}\right) \frac{\log(\log(t)) \lambda(t)^{2-2\alpha}}{t} +\frac{C r \log(\log(t))}{t^{4} \log^{b+1}(t)}\end{equation}
Next, we estimate $\partial_{t}^{2}E_{5}(t,r)$ by proceeding line-by-line on the expression \eqref{e5fordtest}. Denoting, by $E_{5,t,t,i}$, the second partial $t$ derivative of the $i$th line of \eqref{e5fordtest}, we get
\begin{equation}\label{e5tt1}\begin{split} E_{5,t,t,1}(t,r) &= \frac{-1}{r} \int_{0}^{6r} dw \int_{0}^{w} \frac{\rho d\rho}{w} \lambda''''(t+w) \left(\frac{-1-\rho^{2}+r^{2}}{\sqrt{(1+\rho^{2}-r^{2})^{2}+4r^{2}}}+F_{3}(r,\rho,\lambda(t+w))\right)\\
&-\frac{2}{r} \int_{0}^{6r} \int_{0}^{w} \frac{\rho d\rho}{w} \lambda'''(t+w) \partial_{t}(F_{3}(r,\rho,\lambda(t+w)))\\
&-\frac{1}{r} \int_{0}^{6r} dw \int_{0}^{w} \frac{\rho d\rho}{w} \lambda''(t+w) \partial_{t}^{2} (F_{3}(r,\rho,\lambda(t+w)))\end{split}\end{equation}
Note that, with the replacement of $\lambda''''$ with $\lambda'''$, the first line of \eqref{e5tt1} was already estimated in the course of obtaining the final estimate on $\lambda'''$. The same is true for the second line of \eqref{e5tt1}, except with the replacement of $\lambda'''$ with $\lambda''$. On the other hand, for the third line of \eqref{e5tt1}, we use our previous estimate of $\partial_{t}^{2}(F_{3}(r,\rho,\lambda(t+w)))$ to get 
\begin{equation} |E_{5,t,t,1}(t,r)| \leq \frac{C r}{t^{4} \log^{b+1}(t)} + C r \sup_{x \geq t}\left(|e''''(x)|\right)\end{equation}
Next, we have
\begin{equation}\label{e5tt2} \begin{split} E_{5,t,t,2}(t,r) &= \frac{-1}{r} \int_{6r}^{\infty} dw \int_{0}^{w} \frac{\rho d\rho}{w} \lambda''''(t+w) \left(\frac{-1-\rho^{2}+r^{2}}{\sqrt{(1+\rho^{2}-r^{2})^{2}+4r^{2}}}+F_{3}(r,\rho,\lambda(t+w))\right)\\
&-\frac{2}{r} \int_{6r}^{\infty}dw \int_{0}^{w} \frac{\rho d\rho}{w} \lambda'''(t+w) \partial_{t}(F_{3}(r,\rho,\lambda(t+w))) \\
&-\frac{1}{r} \int_{6r}^{\infty} dw \int_{0}^{w} \frac{\rho d\rho}{w} \lambda''(t+w) \partial_{t}^{2}(F_{3}(r,\rho,\lambda(t+w)))\\
&+\frac{1}{r} \int_{6r}^{\infty} dw \lambda''''(t+w) r^{2} w \left(\frac{1}{1+w^{2}}-\frac{1}{\lambda(t+w)^{2-2\alpha}+w^{2}}\right)\\
&+\frac{2}{r} \int_{6r}^{\infty} dw \frac{\lambda'''(t+w) r^{2} w}{(\lambda(t+w)^{2-2\alpha}+w^{2})^{2}} (2-2\alpha) \lambda(t+w)^{1-2\alpha}\lambda'(t+w)\\
&+\frac{1}{r} \int_{6r}^{\infty} dw \frac{\lambda''(t+w) r^{2} w (2-2\alpha) \left((1-2\alpha) \lambda(t+w)^{-2\alpha}(\lambda'(t+w))^{2} + \lambda(t+w)^{1-2\alpha} \lambda''(t+w)\right)}{(\lambda(t+w)^{2-2\alpha}+w^{2})^{2}}\\
&-\frac{2}{r} \int_{6r}^{\infty} dw \frac{\lambda''(t+w) r^{2} w \left((2-2\alpha) \lambda(t+w)^{1-2\alpha} \lambda'(t+w)\right)^{2}}{(\lambda(t+w)^{2-2\alpha}+w^{2})^{3}}\end{split}\end{equation}
Only two lines in \eqref{e5tt2} can not be estimated simply by comparing to an analogous term arising in the $\partial_{t}E_{5}$ estimates. They are estimated as follows: 
\begin{equation} \begin{split} &|\frac{-1}{r} \int_{6r}^{\infty} dw \int_{0}^{w} \frac{\rho d\rho}{w} \lambda''(t+w) \partial_{t}^{2} (F_{3}(r,\rho,\lambda(t+w)))| \\
&\leq \frac{C}{r} \int_{6r}^{\infty} dw \int_{0}^{w} \frac{\rho d\rho}{w(t+w)^{4} \log^{2b+2}(t+w)} \left(\frac{C r^{2} \lambda(t+w)^{2\alpha -3}}{(1+\lambda(t+w)^{4\alpha -4}(\rho^{2}-r^{2})^{2}+2 \lambda(t+w)^{2\alpha -2}(\rho^{2}+r^{2}))}\right)\\
&\leq \frac{C}{r} \int_{0}^{\lambda(t)^{1-\alpha}} \frac{dw}{w t^{4} \log^{2b+2}(t)} \int_{0}^{w} r^{2} \lambda(t+w)^{2\alpha -3} \rho d\rho\\
&+\frac{C}{r} \int_{\lambda(t)^{1-\alpha}}^{\infty} \frac{dw}{w} \int_{0}^{w} \frac{\rho d\rho}{(t+w)^{4} \log^{2b+2}(t+w)} \frac{r^{2} \lambda(t+w)^{2\alpha -3}}{(1+\lambda(t+w)^{4\alpha -4} (\rho^{2}-r^{2})^{2} + 2 \lambda(t+w)^{2\alpha -2} (\rho^{2}+r^{2}))}\\
&\leq \frac{C r}{t^{4} \log^{b+1}(t)}\end{split}\end{equation}
\begin{equation}\begin{split} &|\frac{-1}{r} \int_{6r}^{\infty} \frac{dw \lambda''(t+w) r^{2} w (2-2\alpha) \left((1-2\alpha)\lambda(t+w)^{-2\alpha}(\lambda'(t+w))^{2} + \lambda(t+w)^{1-2\alpha} \lambda''(t+w)\right)}{(\lambda(t+w)^{2-2\alpha}+w^{2})^{2}}|\\
&\leq \frac{C r}{t^{4} \log^{b+2}(t)}\end{split}\end{equation}
where we use the identical procedure used to estimate an analogous term arising while obtaining the final estimate on $\lambda'''$. Then, we get
\begin{equation} |E_{5,t,t,2}(t,r)| \leq \frac{C r}{t^{4} \log^{b+1}(t)} + C r \sup_{x \geq t} |e''''(x)|\end{equation}
Next, we have
\begin{equation} \label{e5tt3} \begin{split} E_{5,t,t,3}(t,r)&= -r \int_{6r}^{\infty} dw \lambda''''(t+w) w \left(\frac{1}{\lambda(t)^{2-2\alpha}+w^{2}}-\frac{1}{\lambda(t+w)^{2-2\alpha}+w^{2})}\right)\\
&-2 r \int_{6r}^{\infty} dw \lambda'''(t+w) w \left(\frac{-(2-2\alpha) \lambda(t)^{1-2\alpha} \lambda'(t)}{(\lambda(t)^{2-2\alpha} + w^{2})^{2}} + \frac{(2-2\alpha) \lambda(t+w)^{1-2\alpha} \lambda'(t+w)}{(\lambda(t+w)^{2-2\alpha} + w^{2})^{2}} \right)\\
&-r \int_{6r}^{\infty} dw \lambda''(t+w) w \partial_{t} \left(\frac{-(2-2\alpha) \lambda(t)^{1-2\alpha} \lambda'(t)}{(\lambda(t)^{2-2\alpha}+w^{2})^{2}} + \frac{(2-2\alpha) \lambda(t+w)^{1-2\alpha} \lambda'(t+w)}{(\lambda(t+w)^{2-2\alpha} + w^{2})^{2}}\right)\end{split}\end{equation}
The first and second lines of \eqref{e5tt3} can be estimated based on estimates of analogous terms arising in the estimation of $\partial_{t}E_{5}$ done while obtaining the final estimate on $\lambda'''$. On the other hand, by appropriately using 
$$|f(t+w) - f(t)| \leq \sup_{\sigma \in [0,1]}|f'(t+\sigma w)| \cdot w$$
we estimate the third line on the right-hand side of \eqref{e5tt3}, and in total, we get
\begin{equation} |E_{5,t,t,3}(t,r)| \leq \frac{C r}{t^{4} \log^{b+2}(t)} + \frac{C r}{t \log^{3b+1-3b\alpha}(t)} \sup_{x \geq t} \left(\frac{|e''''(x)|}{\lambda(x)^{2-2\alpha}}\right)\end{equation}
Finally, we have
\begin{equation}\label{e5tt4}\begin{split} E_{5,t,t,4}(t,r) &= \frac{-1}{r} \int_{0}^{\infty} dw \int_{0}^{w} \rho d\rho \left(\frac{1}{\sqrt{w^{2}-\rho^{2}}}-\frac{1}{w}\right) \lambda''''(t+w) \left(\frac{-1-\rho^{2}+r^{2}}{\sqrt{(1+\rho^{2}-r^{2})^{2}+4r^{2}}}+F_{3}(r,\rho,\lambda(t+w))\right)\\
&-\frac{2}{r} \int_{0}^{\infty} dw \int_{0}^{w} \rho d\rho \left(\frac{1}{\sqrt{w^{2}-\rho^{2}}}-\frac{1}{w}\right) \lambda'''(t+w) \partial_{t}\left(F_{3}(r,\rho,\lambda(t+w))\right)\\
&-\frac{1}{r} \int_{0}^{\infty} dw \int_{0}^{w} \rho d\rho \left(\frac{1}{\sqrt{w^{2}-\rho^{2}}}-\frac{1}{w}\right) \lambda''(t+w) \partial_{t}^{2} \left(F_{3}(r,\rho,\lambda(t+w))\right)\end{split}\end{equation}
The only term in \eqref{e5tt4} which can not be estimated simply by reading off estimates of analogous terms arising in $\partial_{t}E_{5}$ estimates from before is the one involving $\partial_{t}^{2} \left(F_{3}(r,\rho,\lambda(t+w))\right)$. For this term, we have
\begin{equation}\begin{split} &|-\frac{1}{r} \int_{0}^{\infty} dw \int_{0}^{w} \rho d\rho \left(\frac{1}{\sqrt{w^{2}-\rho^{2}}}-\frac{1}{w}\right) \lambda''(t+w) \partial_{t}^{2} \left(F_{3}(r,\rho,\lambda(t+w))\right)|\\
&\leq C r \int_{0}^{\infty} \rho d\rho \int_{\rho}^{\infty} dw \left(\frac{1}{\sqrt{w^{2}-\rho^{2}}}-\frac{1}{w}\right) \frac{\log^{3b-2\alpha b}(t)}{t^{4} \log^{2b+2}(t) (1+\lambda(t)^{4\alpha -4}(\rho^{2}-r^{2})^{2} + 2 \lambda(t)^{2\alpha -2}(\rho^{2}+r^{2}))}\\
&\leq \frac{C r}{t^{4} \log^{b+2}(t)}\end{split}\end{equation}
Combining these, we get
\begin{equation} |\partial_{t}^{2} E_{5}(t,r)| \leq \frac{C r}{t^{4} \log^{b+1}(t)} + \frac{C r}{t \log^{(3-2\alpha)b}(t)} \sup_{x \geq t}\left(\frac{|e''''(x)| x}{\lambda(x)^{3-2\alpha}}\right)\end{equation}
Next, we need to record new estimates on $\partial_{t}^{2}v_{4}$. First, we note that, by the same procedure used to obtain estimates on $\partial_{t}v_{1}$, we have
\begin{equation}\label{dttv1finalsmallrsupest} |\partial_{t}^{2}v_{1}(t,r)| \leq C r \left(\frac{\log(t)+\log(2+r^{2})}{t}\right) \sup_{x \geq t}\left(|\lambda''''(x)| x\right)\end{equation} 
Note that this time, the preliminary estimate on $\lambda''''$ is indeed strong enough to justify the steps leading up to a $\partial_{t}^{2}v_{1}$ analog of \eqref{v1largerest}, but this will be unnecessary for our purposes. We then use \eqref{dttv3finalsupest}, and the same estimates for all other $\partial_{t}^{j}v_{k}$ used to obtain the preliminary estimate on $\partial_{t}^{2}v_{4,c}$, and get
\begin{equation}\label{dttv4cfinalinsidecone}\begin{split} |\partial_{t}^{2}v_{4,c}(t,r)| &\leq \frac{C \mathbbm{1}_{\geq 1}(\frac{2r}{\log^{N}(t)})}{r^{4} \log^{2b}(t)} \begin{cases} \frac{r(\log(t)+\log(2+r^{2}))}{t} \left(\frac{1}{t^{3} \log^{b+1}(t)} + \sup_{x \geq t} \left(\frac{|e''''(x)| x}{\lambda(x)^{2-2\alpha}}\right)\lambda(t)^{2-2\alpha}\right), \quad r \leq \frac{t}{2}\\
\frac{r(\log(t)+\log(2+r^{2}))}{t} \left(\frac{1}{t^{3} \log^{b+1}(t)} + \sup_{x \geq t} \left(\frac{|e''''(x)| x}{\lambda(x)^{2-2\alpha}}\right)\lambda(t)^{2-2\alpha}\right)\\
+\frac{\log(r)}{|t-r|^{3}}+\frac{\log(r)}{t \log(t) (t-r)^{2}}+\frac{\log(r)}{t^{2} |t-r| \log(t)}, \quad t>r>\frac{t}{2}\end{cases}\end{split}\end{equation}
which leads, via the same procedure used to obtain the preliminary estimate on $\partial_{t}^{2} v_{4}$, to
\begin{equation} |\partial_{t}^{2} v_{4}(t,r)| \leq \frac{C}{t^{4} \log^{3b+N-2}(t)} + \frac{C \lambda(t)^{2-2\alpha}}{t \log^{2b-3+N}(t)} \sup_{x \geq t} \left(\frac{|e''''(x)| x}{\lambda(x)^{2-2\alpha}}\right), \quad r \leq \frac{t}{2}\end{equation}
As usual, for the following estimate, we modify \eqref{dttv4cfinalinsidecone} by using $\frac{C}{r^{9/2} \log^{2b}(t)}$ in the region $t-t^{1/6} \leq r \leq t+t^{1/6}$. (Note that the preliminary estimate on $\lambda''''$ implies the following).
$$\sup_{x \geq t}\left(\frac{|e''''(x)| x}{\lambda(x)^{2-2\alpha}}\right) \lambda(t)^{2-2\alpha} \frac{1}{r^{4} \log^{2b-1}(t)} \leq \frac{C}{r^{9/2} \log^{2b}(t)}, \quad t-t^{1/6} \leq r \leq t+t^{1/6}$$ 

This gives
\begin{equation} ||\partial_{t}^{2} v_{4,c}||_{L^{2}(r dr)} \leq \frac{C}{t^{47/12} \log^{2b-1}(t)} + \frac{C \lambda(t)^{2-2\alpha}}{\log^{2N+2b-1}(t) t} \sup_{x \geq t} \left(\frac{|e''''(x)| x}{\lambda(x)^{2-2\alpha}}\right)\end{equation}

In total, this gives
\begin{equation} |\partial_{t}^{2}v_{4}(t,r)| \leq \begin{cases} \frac{C}{t^{4} \log^{3b+N-2}(t)} + \frac{C \lambda(t)^{2-2\alpha}}{t \log^{2b-3+N}(t)} \sup_{x \geq t} \left(\frac{|e''''(x)| x}{\lambda(x)^{2-2\alpha}}\right), \quad r \leq \frac{t}{2}\\
\frac{C}{t^{35/12} \log^{2b-1}(t)} + \frac{C \lambda(t)^{2-2\alpha}}{\log^{2N+2b-2}(t)} \sup_{x \geq t}\left(\frac{|e''''(x)| x}{\lambda(x)^{2-2\alpha}}\right), \quad r \geq \frac{t}{2}\end{cases}\end{equation}

Using  the same estimates that were used for $\partial_{t}^{j}v_{5}$ in the course of proving the preliminary estimate on $\lambda''''$, we get

\begin{equation} |\partial_{t}^{2}\left(-\lambda(t) \langle \left(\frac{\cos(2Q_{\frac{1}{\lambda(t)}}(r))-1}{r^{2}}\right)E_{5}\vert_{r=R\lambda(t)},\phi_{0}\rangle\right)| \leq \frac{C}{t^{4} \log^{b+1}(t)} + \frac{C \sup_{x \geq t} \left(\frac{|e''''(x)| x}{\lambda(x)^{3-2\alpha}}\right)}{t \log^{(3-2\alpha)b}(t)} \end{equation}

\begin{equation} |\partial_{t}^{2}\left(-\lambda(t) \langle \left(\frac{\cos(2Q_{\frac{1}{\lambda(t)}}(r))-1}{r^{2}}\right)v_{4}\left(1-\chi_{\geq 1}(\frac{4r}{t})\right)\vert_{r=R\lambda(t)},\phi_{0}\rangle\right)| \leq \frac{C}{t^{4} \log^{2b+N-2}(t)} + \frac{C \lambda(t)^{2-2\alpha} \sup_{x \geq t} \left(\frac{|e''''(x)| x}{\lambda(x)^{2-2\alpha}}\right)}{t \log^{b-3+N}(t)} \end{equation}

\begin{equation} |\partial_{t}^{2}\left(-\lambda(t) \langle \left(\frac{\cos(2Q_{\frac{1}{\lambda(t)}}(r))-1}{r^{2}}\right)v_{5}\left(1-\chi_{\geq 1}(\frac{4r}{t})\right)\vert_{r=R\lambda(t)},\phi_{0}\rangle\right)| \leq \frac{C}{t^{4}\log^{3N-2}(t)} \end{equation}

\begin{equation}\begin{split} &|\partial_{t}^{2}\left(-\lambda(t) \langle \left(\frac{\cos(2Q_{\frac{1}{\lambda(t)}}(r))-1}{r^{2}}\right)\chi_{\geq 1}(\frac{2r}{\log^{N}(t)})\left(v_{1}+v_{2}+v_{3}\right)\vert_{r=R\lambda(t)},\phi_{0}\rangle\right)| \\
&\leq \frac{C}{t^{4} \log^{3b+2N}(t)} + \frac{C \lambda(t)^{2-2\alpha} \sup_{x \geq t} \left(\frac{|e''''(x)| x}{\lambda(x)^{2-2\alpha}}\right)}{t \log^{2N+2b-1}(t)} \end{split}\end{equation}

\begin{equation}\begin{split} &|\partial_{t}^{2}\left(\lambda(t) \langle \chi_{\geq 1}(\frac{2r}{\log^{N}(t)})F_{0,2}\vert_{r=R\lambda(t)},\phi_{0}\rangle\right)| \leq \frac{C}{t^{4} \log^{3b+1+2N-2b\alpha}(t)} + \frac{C |e''''(t)|}{\log^{2b+2N-2b\alpha}(t)}\end{split}\end{equation}

Combining these, we get
\begin{equation} |\partial_{t}^{2} G(t,\lambda_{0,0}(t)+e(t))| \leq \frac{C}{t^{4} \log^{b+1}(t)} + C \sup_{x \geq t} \left(\frac{|e''''(x)| x}{\lambda(x)^{3-2\alpha}}\right) \frac{1}{t \log^{(3-2\alpha)b}(t)}\end{equation}

and
\begin{equation} |RHS_{4}(t)| \leq \frac{C}{t^{4} \log^{b+1}(t)} + \frac{C \sup_{x \geq t} \left(\frac{x^{2}|e''''(x)|}{\lambda(x)^{3-2\alpha}}\right)}{\sqrt{\log(\log(t))} t^{2} \log^{(3-2\alpha)b}(t)}\end{equation}

We now return to \eqref{eppppforfinalest}, and note that it is of the form
\begin{equation} -4 \int_{t}^{\infty} \frac{e''''(s) ds}{\log(\lambda_{0,0}(s))(1+s-t)} + 4\alpha e''''(t) - 4 \int_{t}^{\infty} \frac{e''''(s) ds}{\log(\lambda_{0,0}(s))(\lambda_{0,0}(t)^{1-\alpha}+s-t)(1+s-t)^{3}}=RHS_{4}(t)\end{equation}
with
\begin{equation} |RHS_{4}(t)| \leq \frac{C}{t^{4} \log^{b+1}(t)} + C \frac{\sup_{x \geq t} \left(\frac{x^{2} |e''''(x)|}{\lambda(x)^{3-2\alpha}}\right)}{\sqrt{\log(\log(t))}t^{2} \log^{(3-2\alpha)b}(t)}\end{equation}
We are now in the same situation as for $e'''$, and repeating the procedure used there, we get
\begin{equation} |e''''(t)| \leq \frac{C}{t^{4}\log^{b+1}(t)}\end{equation}
Using the explicit formula for $\lambda_{0,0}$, we finally conclude
\begin{equation}|\lambda''''(t)| \leq \frac{C}{t^{4}\log^{b+1}(t)}, \quad t \geq T_{0}\end{equation}
To finish the proof of the proposition, we recall \eqref{v1largerest}, and the fact that $\partial_{t}^{2}v_{1}$ has the same formula as $v_{1}$, except for $\lambda''''$ replacing $\lambda''$. Using our estimate on $\lambda''''$, we get
\begin{equation} |\partial_{t}^{2}v_{1}(t,r)| \leq \frac{C}{r t^{2} \log^{b+1}(t)}\end{equation}
Combining this with \eqref{dttv1finalsmallrsupest}, we get \eqref{dttv1finalest}.
\end{proof}
\subsection{Estimates on $\partial_{r}^{k}\partial_{t}^{j}F_{4}$}
Later on, when we start to construct the exact solution to \eqref{wavemaps}, we will utilize the orthogonality 
$$\langle F_{4},\phi_{0}(\frac{\cdot}{\lambda(t)})\rangle =0, \quad t \geq T_{0}$$ 
by integrating by parts in various oscillatory integrals involving $F_{4}$. Therefore, we will need estimates on certain derivatives of $F_{4}$ in order to control the integrands which will arise in this process:
\begin{proposition} For $0 \leq k \leq 2, \quad 0 \leq j \leq 1, \quad j+k \leq 2$, we have
\begin{equation} \label{f4estimates} t^{j} r^{k} |\partial_{r}^{k}\partial_{t}^{j} F_{4}(t,r)| \leq \frac{C \mathbbm{1}_{\{r \leq \log^{N}(t)\}} r}{t^{2} \log^{3b+1-2\alpha b}(t) (r^{2}+\lambda(t)^{2})^{2}}+\frac{C \mathbbm{1}_{\{r \leq \frac{t}{2}\}} r}{t^{2} \log^{5b+2N-2}(t) (r^{2}+\lambda(t)^{2})^{2}}\end{equation}

In addition, we have
\begin{equation} \label{dttf4estimate} |\partial_{t}^{2} F_{4}(t,r)| \leq \frac{C \mathbbm{1}_{\{r \leq \log^{N}(t)\}} r}{t^{4} \log^{3b+1-2\alpha b}(t) (r^{2}+\lambda(t)^{2})^{2}}+\frac{C \mathbbm{1}_{\{r \leq \frac{t}{2}\}} r}{t^{4} \log^{5b+2N-2}(t) (r^{2}+\lambda(t)^{2})^{2}} + \frac{C \mathbbm{1}_{\{r \leq \frac{t}{2}\}}}{t^{4} \log^{5b+N-2}(t) (r^{2}+\lambda(t)^{2})^{2}}\end{equation}
\end{proposition}
\begin{proof} We start with a lemma which captures a delicate leading order cancellation near the origin between $v_{1}$ and $v_{2}$:
\begin{lemma}[Near origin cancellation between $v_{1}$ and $v_{2}$]
For $0 \leq k,j \leq 2, \quad k+j \leq 2$, we have
\begin{equation} t^{j}r^{k} |\partial_{r}^{k}\partial_{t}^{j} (v_{1}+v_{2})| \leq \frac{C r \log(\log(t))}{t^{2} \log^{b+1}(t)}, \quad r \leq \log^{N}(t)\end{equation}
\end{lemma}
\begin{proof}

We use \eqref{v1smallrest} and \eqref{v2precisenearorigin}, to get
\begin{equation} v_{1}(t,r) + v_{2}(t,r) = r \left(\int_{t}^{\infty} \frac{\lambda''(x)}{1+x-t} dx -\frac{b}{t^{2}\log^{b}(t)}\right) +\text{Err}(t,r) + E_{v_{2}}(t,r), \quad 0 \leq r \leq \frac{t}{2}\end{equation}
Then, using the modulation equation, and the previous estimates on its terms, we get
\begin{equation} \label{v1plusv2delicate} |v_{1}(t,r) + v_{2}(t,r)| \leq \frac{C r \log(\log(t))}{t^{2} \log^{b+1}(t)} + \frac{C r \log(3+2r)}{t^{2} \log^{b+1}(t)} + \frac{Cr}{t^{2} \log^{b+1}(t)} + \frac{Cr^{2}}{t^{3} \log^{b}(t)}, \quad r \leq \frac{t}{2}\end{equation}
Note that, for (e.g) $r \leq \log^{N}(t)$, this estimate is slightly better than what would result from estimating $v_{1}$ and $v_{2}$ seperately.\\
\\
We use \eqref{drv1smallrest} and \eqref{v2precisenearorigin} to obtain the analogous estimate after taking an $r$ derivative:
\begin{equation} \partial_{r}v_{1}(t,r)+\partial_{r}v_{2}(t,r) = \int_{t}^{\infty} \frac{\lambda''(s) ds}{1+s-t} -\frac{b}{t^{2}\log^{b}(t)} + E_{\partial_{r}v_{1}}(t,r)+E_{\partial_{r} v_{2}}(t,r)\end{equation}
Again, the modulation equation, combined with the estimates for $E_{\partial_{r}v_{1}}$ and $E_{\partial_{r}v_{2}}$ from the previous subsections, give
\begin{equation}\label{drv1plusv2delicate} |\partial_{r}v_{1}(t,r) + \partial_{r}v_{2}(t,r)| \leq \frac{C \log(\log(t))}{t^{2} \log^{b+1}(t)} + \frac{C \log(3+2r)}{t^{2} \log^{b+1}(t)} + \frac{C}{t^{2} \log^{b+1}(t)} + \frac{Cr}{t^{3} \log^{b}(t)}, \quad r \leq \frac{t}{2}\end{equation}

We first note that, for $j=1,2$, $\partial_{t}^{j}v_{1}$ solves the same equation with $0$ Cauchy data at infinity as $v_{1}$ does, except with $\lambda^{(2+j)}(t)$ on the right-hand side. Then, we use the fact that $\lambda$ solves
\begin{equation} -4 \int_{t}^{\infty} \frac{\lambda''(x) dx}{1+x-t} + \frac{4b}{t^{2} \log^{b}(t)} + 4 \alpha \log(\lambda(t)) \lambda''(t) - 4 \int_{t}^{\infty} \frac{\lambda''(s) ds}{(\lambda(t)^{1-\alpha}+s-t)(1+s-t)^{3}} = G(t,\lambda(t))\end{equation}
and differentiate $j$ times for $j=1,2$, to get
\begin{equation} |-4\int_{t}^{\infty} \frac{\lambda'''(s)}{1+s-t} ds -\frac{8b}{t^{3} \log^{b}(t)}| \leq \frac{C \log(\log(t))}{t^{3} \log^{b+1}(t)}\end{equation}
and
\begin{equation} |\int_{t}^{\infty} \frac{\lambda''''(s) ds}{1+s-t} -\frac{6b}{t^{4} \log^{b}(t)}| \leq \frac{C \log(\log(t))}{t^{4} \log^{b+1}(t)}\end{equation}
Then, using \eqref{v1smallrest},\eqref{v2precisenearorigin}, we get
\begin{equation}\label{dtv1plusv2delicate} |\partial_{t}(v_{1}+v_{2})(t,r)| \leq C \frac{r \log(\log(t))}{t^{3} \log^{b+1}(t)}, \quad r \leq  \log^{N}(t)\end{equation}
and
\begin{equation}\label{dttv1plusv2delicate} |\partial_{t}^{2}(v_{1}+v_{2})(t,r)| \leq C \frac{r \log(\log(t))}{t^{4} \log^{b+1}(t)}, \quad r \leq  \log^{N}(t)\end{equation}
Next, we use the equations solved by $v_{1},v_{2}$, to get
\begin{equation} r^{2} \partial_{rr}(v_{1}+v_{2}) = r^{2} \left(\frac{-2 r \lambda''(t)}{1+r^{2}}\right) + r^{2} \partial_{tt}(v_{1}+v_{2})-r\partial_{r}(v_{1}+v_{2}) + (v_{1}+v_{2})\end{equation}
which gives
\begin{equation}\label{drrv1plusv2delicate} r^{2} |\partial_{rr}(v_{1}+v_{2})| \leq C \frac{r \log(\log(t))}{t^{2} \log^{b+1}(t)},\quad r \leq \log^{N}(t)\end{equation}
Finally, we study $\partial_{tr}(v_{1}+v_{2})$. Because $\partial_{t}v_{1}$ has the same representation formula as $v_{1}$, except with $\lambda''$ replaced by $\lambda'''$, we can use the same procedure used for $\partial_{r}v_{1}$, to get
\begin{equation} \partial_{tr}v_{1}(t,r) = \int_{t}^{\infty} \frac{\lambda'''(s) ds}{1+s-t} + E_{\partial_{tr}v_{1}}(t,r)\end{equation} 
with
\begin{equation} |E_{\partial_{tr}v_{1}}(t,r)| \leq \frac{C \log(3+2r)}{t^{3} \log^{b+1}(t)}\end{equation}

Then, we use \eqref{v2precisenearorigin} to conclude
\begin{equation}\label{dtrv1plusv2delicate} |\partial_{tr}(v_{1}+v_{2})(t,r)| \leq \frac{C \log(\log(t))}{t^{3} \log^{b+1}(t)}, \quad r \leq \log^{N}(t)\end{equation}
\end{proof}
Next, we note that our previous estimate for $\partial_{r}v_{3}$, \eqref{drv3est}, sufficed for all of our purposes up to now. For estimating derivatives of $F_{4}$, however, we will require a different estimate which will turn out to be better in the region $r \leq  \log^{N}(t)$, and which will require a more complicated argument than before. This refinement will also lead to slightly better estimates on $\partial_{tr}v_{3}$ and $r^{2} \partial_{rr}v_{3}$ near the origin.
\begin{lemma}[Near origin refinements of various $v_{3}$ related estimates]
\begin{equation}\label{drv3nearoriginrefinement} |\partial_{r}v_{3}(t,r)| \leq \frac{C}{t^{2} \log^{b+1}(t)} \left(\log(\log(t))+\log(6+6r)\right)\end{equation}

\begin{equation} |\partial_{tr}v_{3}(t,r)| \leq \frac{C \log(\log(t))}{t^{3} \log^{b+1}(t)}, \quad r \leq \log^{N}(t)\end{equation}

\begin{equation}\label{rsquareddrrv3} r^{2} |\partial_{rr}v_{3}| \leq \frac{C r \log(\log(t))}{t^{2} \log^{b+1}(t)}, \quad r \leq \log^{N}(t)\end{equation}

\end{lemma}
\begin{proof}
 We first use the same decomposition on $\partial_{r}v_{3}$ used before, and only need to use a different argument for one term, namely:
$$ -\int_{t}^{\infty} \frac{ds}{(s-t)} \int_{0}^{s-t} \rho d\rho \frac{\lambda''(s)}{r} \partial_{r}\left(\frac{-1-\rho^{2}+r^{2}}{\sqrt{(-1-\rho^{2}+r^{2})^{2}+4r^{2}}}+F_{3}(r,\rho,\lambda(s))\right)$$
We treat different regions of the variable $s-t$. First, from our previous $\partial_{r}v_{3}$ estimate, we have
\begin{equation}\begin{split} &|\int_{0}^{s-t} \rho d\rho \frac{\lambda''(s)}{r} \partial_{r}\left(\frac{-1-\rho^{2}+r^{2}}{\sqrt{(-1-\rho^{2}+r^{2})^{2}+4r^{2}}}\right)| \leq C(s-t)^{2} |\lambda''(s)|, \quad s-t \leq \frac{1}{2} \end{split}\end{equation}
which gives
\begin{equation}\begin{split} &|-\int_{t}^{t+\frac{1}{2}} \frac{ds}{(s-t)} \int_{0}^{s-t} \rho d\rho \frac{\lambda''(s)}{r} \partial_{r}\left(\frac{-1-\rho^{2}+r^{2}}{\sqrt{(-1-\rho^{2}+r^{2})^{2}+4r^{2}}}\right)|\\
&\leq \frac{C}{t^{2} \log^{b+1}(t)}\end{split}\end{equation}
Next, we have
\begin{equation} \begin{split} &|-\int_{t+\frac{1}{2}}^{t+6} \frac{ds}{(s-t)} \int_{0}^{s-t} \rho d\rho \frac{\lambda''(s)}{r} \partial_{r}\left(\frac{-1-\rho^{2}+r^{2}}{\sqrt{(-1-\rho^{2}+r^{2})^{2}+4r^{2}}}\right)|\\
&\leq \int_{t+\frac{1}{2}}^{t+6} \frac{ds}{(s-t)} \int_{0}^{\infty} \rho d\rho \frac{|\lambda''(s)|}{r} |\partial_{r}\left(\frac{-1-\rho^{2}+r^{2}}{\sqrt{(-1-\rho^{2}+r^{2})^{2}+4r^{2}}}\right)|\\
&\leq C \int_{t+\frac{1}{2}}^{t+6} \frac{ds}{(s-t)} |\lambda''(s)|\\
&\leq C \frac{1}{t^{2}\log^{b+1}(t)} \end{split}\end{equation}
For convenience, we recall the definition of $F_{3}$:
$$F_{3}(r,\rho,\lambda(s))=\frac{1-(r^{2}-\rho^{2})\lambda(s)^{2\alpha -2}}{\sqrt{4r^{2}\lambda(s)^{2\alpha -2} + (1-(r^{2}-\rho^{2})\lambda(s)^{2\alpha -2})^{2}}}$$
Similarly, we have
\begin{equation} |\partial_{r}F_{3}(r,\rho,\lambda(s))| \leq \frac{C r \lambda(s)^{2+2\alpha}((\rho^{2}+r^{2})\lambda(s)^{2\alpha -2} +1)}{\lambda(s)^{4} (1+2(\rho^{2}+r^{2})\lambda(s)^{2\alpha -2}+(\rho^{2}-r^{2})^{2} \lambda(s)^{4\alpha -4})^{3/2}}\end{equation}
For the region $s-t \leq \lambda(t)^{1-\alpha}$, we continue this estimate, and get
\begin{equation} |\partial_{r}F_{3}(r,\rho,\lambda(s))| \leq C r \lambda(s)^{2\alpha -2}\end{equation}
So,
\begin{equation}\begin{split} &|\int_{t}^{t+\lambda(t)^{1-\alpha}} \frac{ds}{(s-t)} \int_{0}^{s-t} \frac{\rho d\rho \lambda''(s)}{r} \partial_{r}F_{3}(r,\rho,\lambda(s))|\leq C \int_{t}^{t+\lambda(t)^{1-\alpha}} \frac{ds}{(s-t)} \int_{0}^{s-t} \frac{\rho d\rho}{r} |\lambda''(s)| r \lambda(s)^{2\alpha -2}\\
&\leq \frac{C}{t^{2} \log^{b+1}(t)}\end{split}\end{equation}

\begin{equation}\begin{split} &|\int_{t+\lambda(t)^{1-\alpha}}^{t+6} \frac{ds}{(s-t)} \int_{0}^{s-t} \frac{\rho d\rho \lambda''(s)}{r} \partial_{r}F_{3}(r,\rho,\lambda(s))| \leq C \int_{t+\lambda(t)^{1-\alpha}}^{t+6} \frac{ds}{(s-t)} \int_{0}^{\infty} \frac{\rho d\rho |\lambda''(s)|}{r} |\partial_{r}F_{3}(r,\rho,\lambda(s))|\\
&\leq \frac{C}{r} \int_{t+\lambda(t)^{1-\alpha}}^{t+6} \frac{ds}{(s-t)} |\lambda''(s)| r \\
&\leq \frac{C \log(\log(t))}{t^{2} \log^{b+1}(t)}\end{split}\end{equation}

Then, we treat the region $ 6 \leq s-t \leq 6+6r$ in a similar fashion:
\begin{equation}\begin{split} &|-\int_{t+6}^{t+6+6r} \frac{ds}{(s-t)} \int_{0}^{s-t} \rho d\rho \frac{\lambda''(s)}{r} \partial_{r}\left(\frac{-1-\rho^{2}+r^{2}}{\sqrt{(-1-\rho^{2}+r^{2})^{2}+4r^{2}}}+F_{3}(r,\rho,\lambda(s))\right)|\\
&\leq \int_{t+6}^{t+6+6r} \frac{ds}{(s-t)} \int_{0}^{\infty} \rho d\rho \frac{|\lambda''(s)|}{r} \left(|\partial_{r}\frac{-1-\rho^{2}+r^{2}}{\sqrt{(-1-\rho^{2}+r^{2})^{2}+4r^{2}}}|+|\partial_{r}F_{3}(r,\rho,\lambda(s))|\right)\\
&\leq C \int_{t+6}^{t+6+6r} \frac{ds}{(s-t)} |\lambda''(s)|\\
&\leq \frac{C \log(6+6r)}{t^{2} \log^{b+1}(t)}\end{split}\end{equation}

Now, we study the region $s-t \geq 6+6r$:

 We first note that
\begin{equation} \int_{0}^{\infty} \partial_{r} \left(\frac{-1-\rho^{2}+r^{2}}{\sqrt{(1+\rho^{2}-r^{2})^{2}+4r^{2}}}+F_{3}(r,\rho,\lambda(s))\right) \rho d\rho =0\end{equation}
which follows from the fact that the integrand of the expression above is equal to
\begin{equation} \partial_{\rho}\left(r\left(\frac{-1+\rho^{2}-r^{2}}{\sqrt{\rho^{4}-2\rho^{2}(-1+r^{2})+(1+r^{2})^{2}}}+\frac{1+(r^{2}-\rho^{2})\lambda(s)^{2\alpha -2}}{\sqrt{1+2(\rho^{2}+r^{2})\lambda(s)^{2\alpha -2}+(\rho^{2}-r^{2})^{2} \lambda(s)^{4\alpha -4}}}\right)\right)\end{equation}

Then, we use
\begin{equation} r\left(\frac{-1+(s-t)^{2}-r^{2}}{\sqrt{(s-t)^{4}-2(s-t)^{2}(-1+r^{2})+(1+r^{2})^{2}}}\right) = r\left(1+E_{\partial_{r}v_{3},1}\right)\end{equation}
with
$$|E_{\partial_{r}v_{3},1}| \leq C \frac{(1+r)^{2}}{(s-t)^{2}}, \quad s-t \geq 6+6r$$

and
\begin{equation}r\left(\frac{1+(r^{2}-(s-t)^{2})\lambda(s)^{2\alpha -2}}{\sqrt{1+2((s-t)^{2}+r^{2})\lambda(s)^{2\alpha -2}+((s-t)^{2}-r^{2})^{2} \lambda(s)^{4\alpha -4}}}\right)=r\left(-1+E_{\partial_{r}v_{3},2}\right)\end{equation}
with
$$|E_{\partial_{r}v_{3},2}| \leq C \frac{(1+r)^{2}}{(s-t)^{2}}, \quad s-t \geq 6+6r$$

which gives 
\begin{equation}\begin{split} &|-\int_{t+6+6r}^{\infty} \frac{ds}{(s-t)} \int_{0}^{s-t} \rho d\rho \frac{\lambda''(s)}{r} \partial_{r}\left(\frac{-1-\rho^{2}+r^{2}}{\sqrt{(-1-\rho^{2}+r^{2})^{2}+4r^{2}}}+F_{3}(r,\rho,\lambda(s))\right)|\\
&=|\int_{t+6+6r}^{\infty} \frac{ds}{(s-t)} \int_{s-t}^{\infty} \rho d\rho \frac{\lambda''(s)}{r} \partial_{r}\left(\frac{-1-\rho^{2}+r^{2}}{\sqrt{(-1-\rho^{2}+r^{2})^{2}+4r^{2}}}+F_{3}(r,\rho,\lambda(s))\right)|\\
&=|\int_{t+6+6r}^{\infty} \frac{ds}{(s-t)} \frac{\lambda''(s)}{r} r\left(E_{\partial_{r}v_{3},1}+E_{\partial_{r}v_{3},2}\right)|\\
&\leq C \int_{t+6+6r}^{\infty} \frac{ds}{(s-t)} |\lambda''(s)| \frac{(1+r)^{2}}{(s-t)^{2}}\\
&\leq \frac{C}{t^{2} \log^{b+1}(t)}\end{split}\end{equation}

In total, we get
\begin{equation} |\partial_{r}v_{3}(t,r)| \leq \frac{C}{t^{2} \log^{b+1}(t)} \left(\log(\log(t))+\log(6+6r)\right)\end{equation}
Note that this estimate is better than \eqref{drv3est}, in the region $r \leq \log^{N}(t)$. 
For $\partial_{tr}v_{3}$, we recall that the partial $r$ derivative of the second line of \eqref{dtv3fordtrv3} was bounded above in absolute value by $\frac{C}{t^{3} \log^{b+1}(t)}$, and an estimate of the $r$ derivative of the first line of \eqref{dtv3fordtrv3} was 
inferred from an estimate on $\partial_{r}v_{3}$. Using the above near origin refinement of the $\partial_{r}v_{3}$ estimate, instead of the previous one, gives
\begin{equation} |\partial_{tr}v_{3}(t,r)| \leq \frac{C \log(\log(t))}{t^{3} \log^{b+1}(t)}, \quad r \leq \log^{N}(t)\end{equation}
Now, we can prove a different estimate for $ r^{2}\partial_{r}^{2} v_{3}$ than what follows from previous work. 
\begin{equation} r^{2} \partial_{rr}v_{3} = r^{2} F_{0,1}(t,r) + r^{2} \partial_{tt}v_{3}-r \partial_{r}v_{3}+v_{3}\end{equation}
Using \eqref{drv3nearoriginrefinement}, as well as our previous estimates on $v_{3}$, we get
\begin{equation} r^{2} |\partial_{rr}v_{3}| \leq \frac{C r \log(\log(t))}{t^{2} \log^{b+1}(t)}, \quad r \leq \log^{N}(t)\end{equation}
\end{proof}
Finally, the estimate for $\partial_{t}v_{5}$ used when estimating $\lambda'''$ was based on estimates for $N_{2}(f)$ and $\partial_{r}N_{2}(f)$ which only used at most two derivatives of $\lambda$. For future use, we will prove a stronger estimate on $\partial_{t}v_{5}$, which uses the final estimate on $\lambda'''$. (Note that this is the reason why no estimates on $\partial_{t}v_{5}$ were presented in the proposition statement in the $\lambda'''$ section). As part of the process, we will also obtain an estimate on $\partial_{tr}v_{4}$.
\begin{lemma}[Improved $\partial_{t}v_{5}$, $\partial_{tr}v_{4}$ estimates]
\begin{equation}\label{dtv5nearoriginrefinement} |\partial_{t}v_{5}(t,r)| \leq \frac{C r}{t^{9/2} \log^{\frac{3N}{2}+3b-2}(t)}, \quad r \leq \frac{t}{2}\end{equation}
\begin{equation} \label{dtrv4finalest} |\partial_{tr}v_{4}(t,r)| \leq \begin{cases} \frac{C}{t^{3} \log^{3b+2N-2}(t)}, \quad r \leq \frac{t}{2}\\
 \frac{C}{t^{35/12} \log^{2b-1}(t)}, \quad r \geq \frac{t}{2}\end{cases}\end{equation}
\end{lemma}
\begin{proof}
 For this estimate, we start with
\begin{equation}\begin{split} \partial_{t}v_{5}(t,r) = \frac{-r}{2 \pi} \int_{0}^{1} d\beta \int_{t}^{\infty} ds \int_{B_{s-t}(0)} \frac{dA(y)}{\sqrt{(s-t)^{2}-|y|^{2}}} &\left(\frac{\partial_{12}N_{2}(f_{v_{5}})(s,|\beta x+y|)\left((\beta x+y)\cdot \hat{x}\right)^{2}}{|\beta x+y|^{2}}\right.\\
&-\frac{\partial_{1}N_{2}(f_{v_{5}})(s,|\beta x+y|)\left((\beta x+y)\cdot \hat{x}\right)^{2}}{|\beta x+y|^{3}}\\
&\left. + \frac{\partial_{1}N_{2}(f_{v_{5}})(s,|\beta x+y|)}{|\beta x+y|}\right)\end{split}\end{equation}  
This gives
\begin{equation}\label{dtv5nearoriginsetup} |\partial_{t}v_{5}(t,r)| \leq C r \int_{0}^{1} d\beta \int_{t}^{\infty} ds \int_{B_{s-t}(0)} \frac{dA(y)}{\sqrt{(s-t)^{2}-|y|^{2}}} \left(|\partial_{12}N_{2}(f_{v_{5}})|(s,|\beta x+y|)+\frac{|\partial_{1}N_{2}(f_{v_{5}})|(s,|\beta x+y|)}{|\beta x+y|}\right)\end{equation}
In order to proceed, we will have to estimate $\partial_{tr}v_{4}$, which was not done previously. We start by noting that, the procedure used to obtain \eqref{dtv4refinementnearorigin} was to write
\begin{equation} \partial_{t}v_{4}(t,r) = r \int_{0}^{1} d\beta \partial_{tr}v_{4}(t,r\beta)\end{equation}
and then, to estimate $\partial_{tr}v_{4}(t,r\beta)$ uniformly for $0 \leq \beta \leq 1$. So, in the region $r \leq \frac{t}{2}$, we have
\begin{equation} |\partial_{tr}v_{4}(t,r)| \leq \frac{C}{t^{3} \log^{3b+2N-2}(t)}, \quad r \leq \frac{t}{2}\end{equation}
For the region $r \geq \frac{t}{2}$, we start by introducing the vector field 
$$V=t \partial_{t}+r\partial_{r}$$
and recalling that
\begin{equation} -\partial_{tt}v_{4}+\partial_{rr}v_{4}+\frac{1}{r} \partial_{r}v_{4}-\frac{v_{4}}{r^{2}}=v_{4,c}\end{equation}
So, we have
\begin{equation}\left(-\partial_{tt}+\partial_{rr}+\frac{1}{r} \partial_{r}-\frac{1}{r^{2}}\right)\left(V(\partial_{t}v_{4})\right) = V\left(\partial_{t}v_{4,c}\right) + 2 \partial_{t}v_{4,c}\end{equation}
We use \eqref{dtrv4cfinal}, \eqref{dttv4cfinalinsidecone}, and 
$$|\partial_{t}^{2} v_{4,c}(t,r)|+|\partial_{tr} v_{4,c}(t,r)| \leq \frac{C}{t^{9/2} \log^{2b}(t)}, \quad t-t^{1/6} \leq r \leq t+t^{1/6}$$   
to get
\begin{equation} |V(\partial_{t}v_{4,c})|(t,r) \leq C  \begin{cases} \frac{\mathbbm{1}_{\{r \geq \frac{\log^{N}(t)}{2}\}}}{r^{3} t^{3} \log^{3b}(t)}, \quad r \leq \frac{t}{2}\\
\frac{\log(r)}{r^{3} \log^{2b}(t) |t-r|^{3}} + \frac{\log(r)}{r^{3} t \log^{2b}(t) (t-r)^{2}} + \frac{\log(r)}{\log^{2b+1}(t) r^{4} t |t-r|} + \frac{\log(r)}{r^{3} \log^{3b+1}(t) t^{3}}, \quad \frac{t}{2} \leq r \leq t-t^{1/6}\text{, or } r\geq t+t^{1/6}\\
\frac{1}{t^{7/2} \log^{2b}(t)}, \quad t-t^{1/6} \leq r \leq t+t^{1/6}\end{cases}\end{equation}
and we get
\begin{equation}||V(\partial_{t}v_{4,c})||_{L^{2}(r dr)} \leq \frac{C}{t^{35/12} \log^{2b-1}(t)}\end{equation}
Next, we recall \eqref{dtv4cl2estfordtv4final}, and get
\begin{equation} ||\partial_{t}v_{4,c}||_{L^{2}(r dr)} \leq \frac{C}{t^{3} \log^{2N+3b}(t)}\end{equation}
Then, we apply the same  procedure used before to estimate (e.g.) $\partial_{t}v_{4}$ in the region $r \geq \frac{t}{2}$, to the equation for $V(\partial_{t}v_{4})$, and get
\begin{equation} |V(\partial_{t}v_{4})|(t,r) \leq \frac{C}{t^{23/12} \log^{2b-1}(t)}\end{equation}
But, then,
\begin{equation}\begin{split} \partial_{tr}v_{4}(t,r) &= \frac{(t\partial_{t}^{2} v_{4} + r \partial_{tr}v_{4})}{r} - \frac{t}{r} \partial_{t}^{2}v_{4}\\
&=\frac{V(\partial_{t}v_{4})}{r} - \frac{t}{r} \partial_{t}^{2}v_{4}\end{split}\end{equation}
and we have already estimated $\partial_{t}^{2} v_{4}$, during our study of $\lambda''''$. So, we get
\begin{equation} |\partial_{tr}v_{4}(t,r)| \leq \frac{C}{t^{35/12} \log^{2b-1}(t)}, \quad r \geq \frac{t}{2}\end{equation}
Now, we can estimate $\partial_{tr}N_{2}(f)$, and get
\begin{equation} |\partial_{tr}N_{2}(f)|(t,r) \leq \begin{cases} \frac{C}{t^{5} \log^{3b}(t) (r^{2}+\lambda(t)^{2})}, \quad r \leq \frac{t}{2}\\
\frac{C \log^{3}(r)}{r^{2}|t-r|^{5}} + \frac{C \log^{3}(r)}{r^{5/2} \sqrt{t} (t-r)^{4}} + \frac{C \log^{2}(r)}{r^{2} t^{3/2} |t-r|^{3} \log^{3b-1+\frac{5N}{2}}(t)} , \quad t>r>\frac{t}{2}\end{cases}\end{equation}
(Note that we only need to estimate $\partial_{tr}N_{2}(f_{v_{5}})$ in the region $r < t$, which is why we used $\frac{1}{|t-r|} \geq \frac{1}{t}$ to simply the above estimate. In previous estimates, we did not proceed analogously because we eventually used a single estimate for the entire region $\frac{t}{2} \leq r \leq t-t^{\gamma} \text{, or }r \geq t+t^{\gamma}$ for some $\gamma>0$).\\
\\
Using the estimates on $\partial_{t}^{j}v_{k}, \quad j=0,1, \quad k=1,2,3,4$, we get
\begin{equation} |\partial_{t}N_{2}(f)(t,r)| \leq \begin{cases} \frac{C r}{t^{5} \log^{3b}(t) (r^{2}+\lambda(t)^{2})}, \quad r \leq \frac{t}{2}\\
\frac{C}{r^{3}t^{7/2} \log^{\frac{3N}{2}+5b}(t)} + \frac{C \log^{3}(r)}{r^{2} (t-r)^{4}}, \quad t>r>\frac{t}{2}\end{cases}\end{equation}
Now, we return to \eqref{dtv5nearoriginsetup}, and use the same procedure used to estimate $v_{5}$ in the region $r \leq \frac{t}{2}$ earlier. (In particular, we handle the contribution to $\partial_{t}v_{5}$ coming from the term $\frac{\log^{3}(r)}{r^{2}|t-r|^{5}}$, which arises in the estimate for $\partial_{tr}N_{2}(f_{v_{5}})$, using the same procedure as in \eqref{v5specialterm}). This results in 
\begin{equation} |\partial_{t}v_{5}(t,r)| \leq \frac{C r}{t^{9/2} \log^{\frac{3N}{2}+3b-2}(t)}, \quad r \leq \frac{t}{2}\end{equation}
\end{proof}
Now, we proceed to estimate $F_{4}$, and various of its derivatives. We recall 
\begin{equation}\begin{split} F_{4}(t,r) &= \left(1-\chi_{\geq 1}(\frac{2r}{\log^{N}(t)})\right)\left(F_{0,2}(t,r)+\left(\frac{\cos(2Q_{\frac{1}{\lambda(t)}}(r))-1}{r^{2}}\right)\left(v_{1}+v_{2}+v_{3}\right)\right)\\
&+\left(1-\chi_{\geq 1}(\frac{4r}{t})\right)\left(\frac{\cos(2Q_{\frac{1}{\lambda(t)}}(r))-1}{r^{2}}\right)\left(v_{4}+v_{5}\right)\end{split}\end{equation}

Then, using \eqref{v3laterest}, \eqref{v4finalest}, \eqref{v5finalest}, and \eqref{v1plusv2delicate}, we get 
\begin{equation}\label{F4pointwise}\begin{split} |F_{4}(t,r)| &\leq C \frac{\left(1-\chi_{\geq 1}(\frac{2r}{\log^{N}(t)})\right)}{(r^{2}+\lambda(t)^{2})^{2}}\left(\frac{r}{t^{2} \log^{3b+1-2b\alpha}(t)}\right)\\
&+C\frac{ \left(1-\chi_{\geq 1}(\frac{4r}{t})\right) }{(r^{2}+\lambda(t)^{2})^{2}}\frac{r}{t^{2} \log^{5b+2N-1}(t)}\end{split}\end{equation}

Next, we will need to record estimates on $\partial_{t}F_{4}$, and $r \partial_{r}F_{4}$. By combining \eqref{dtv4finalest}, and \eqref{dtv5nearoriginrefinement}, we get
\begin{equation} |\partial_{t}v_{4}+\partial_{t}v_{5}|(t,r) \leq \frac{C r}{t^{3} \log^{3b+2N-2}(t)}, \quad r \leq \frac{t}{2}\end{equation}
This, combined with the explicit formula for $F_{0,2}$, \eqref{dtv1plusv2delicate}, and \eqref{dtv3finalest}, gives 
\begin{equation}\label{dtf4pointwise} |\partial_{t}F_{4}(t,r)| \leq \frac{C \mathbbm{1}_{\{r \leq \log^{N}(t)\}} r}{t^{3} \log^{3b+1-2b\alpha}(t)(r^{2}+\lambda(t)^{2})^{2}} + \frac{C \mathbbm{1}_{\{r \leq \frac{t}{2}\}}}{(r^{2}+\lambda(t)^{2})^{2}} \frac{r}{t^{3} \log^{5b+2N-2}(t)}\end{equation}
where we estimate derivatives on the cutoff functions by (e.g.)
$$|\chi_{\geq 1}'(\frac{2r}{\log^{N}(t)})| \frac{N r}{t \log^{N+1}(t)} \leq \frac{C \mathbbm{1}_{\{r \leq \log^{N}(t)\}}}{t \log(t)}$$

Similarly, for the partial $r$ derivative of $F_{4}$, we use the explicit formula for $F_{0,2}$, \eqref{drv3nearoriginrefinement}, \eqref{drv1plusv2delicate}, \eqref{drv4finalest}, and \eqref{drv5finalest}, to get
\begin{equation}\label{rdrf4pointwise} r |\partial_{r}F_{4}(t,r)| \leq \frac{C r \mathbbm{1}_{\{r \leq \log^{N}(t)\}}}{t^{2} \log^{3b+1-2b\alpha}(t)(r^{2}+\lambda(t)^{2})^{2}}+C\frac{ \mathbbm{1}_{\{r \leq \frac{t}{2}\}}}{(r^{2}+\lambda(t)^{2})^{2}}\frac{r}{t^{2} \log^{5b+2N-1}(t)}\end{equation}
Treating the terms involving $\chi_{\geq 1}$ and $F_{0,2}$ as before, and using \eqref{dttv1plusv2delicate}, \eqref{dttv4finalest}, and \eqref{dttv5finalest}, we get
\begin{equation}\label{dttf4pointwise} |\partial_{t}^{2}F_{4}(t,r)| \leq \frac{C \mathbbm{1}_{\{r \leq \log^{N}(t)\}} r}{t^{4} (r^{2}+\lambda(t)^{2})^{2} \log^{3b+1-2b\alpha}(t)} + \frac{C \mathbbm{1}_{\{r \leq \frac{t}{2}\}} r}{t^{4} \log^{5b+2N-2}(t)(r^{2}+\lambda(t)^{2})^{2}} + \frac{C \mathbbm{1}_{\{r \leq \frac{t}{2}\}}}{t^{4} \log^{5b+N-2}(t)(r^{2}+\lambda(t)^{2})^{2}}\end{equation}
Next, we note that, exactly as was the case with $\partial_{tr}v_{4}$, we can infer the following estimate on $\partial_{tr}v_{5}$ by inspecting \eqref{dtv5nearoriginrefinement}, and the procedure used to obtain it:
\begin{equation}\label{dtrv5nearorigin} |\partial_{tr}v_{5}(t,r)| \leq \frac{C }{t^{9/2} \log^{\frac{3N}{2}+3b-2}(t)}, \quad r \leq \frac{t}{2}\end{equation}
We use the same procedure, \eqref{dtrv1plusv2delicate}, \eqref{dtrv5nearorigin}, and \eqref{dtrv4finalest}, to get
\begin{equation}\label{rdtrf4pointwise}r|\partial_{rt}F_{4}(t,r)| \leq C \frac{\mathbbm{1}_{\{r \leq \log^{N}(t)\}} r }{t^{3} \log^{3b+1-2b\alpha}(t)  (r^{2}+\lambda(t)^{2})^{2}} + C \frac{\mathbbm{1}_{\{r \leq \frac{t}{2}\}} r }{t^{3} \log^{5b+2N-2}(t)(r^{2}+\lambda(t)^{2})^{2}}\end{equation}
Next, we use the same procedure used for $v_{1}+v_{2},v_{3}$ to estimate $r^{2}\partial_{r}^{2} v_{4}, r^{2}\partial_{r}^{2}v_{5}$:
\begin{equation} r^{2} \partial_{r}^{2} v_{4} = r^{2} v_{4,c}(t,r) + r^{2} \partial_{tt}v_{4}-r \partial_{r}v_{4}+v_{4}\end{equation}
So,
\begin{equation} \label{rsquareddrrv4}\begin{split}|r^{2} \partial_{r}^{2}v_{4}(t,r)| &\leq \frac{C r}{t^{2} \log^{3b+2N}(t)} + \frac{C r^{2}}{t^{4} \log^{3b-2+N}(t)} + \frac{C r}{t^{2} \log^{3b+2N-1}(t)} + \frac{C r}{t^{2} \log^{3b+2N-1}(t)}\\
&\leq \frac{C r}{t^{2} \log^{3b+2N-1}(t)}, \quad r \leq \frac{t}{2}\end{split}\end{equation}
and
\begin{equation}\label{rsquareddrrv5}\begin{split} |r^{2} \partial_{r}^{2} v_{5}(t,r)| &\leq \frac{C r}{t^{4} \log^{3b}(t)} + \frac{C r^{2}}{t^{4} \log^{3N+b-2}(t)} + \frac{C r}{t^{7/2} \log^{\frac{5N}{2}+3b-3}(t)} + \frac{C r}{t^{7/2} \log^{\frac{5N}{2}+3b-3}(t)}\\
&\leq \frac{C r}{t^{3} \log^{3N+b-2}(t)}, \quad r \leq \frac{t}{2}\end{split}\end{equation}

By using  \eqref{drrv1plusv2delicate}, \eqref{rsquareddrrv3},  \eqref{rsquareddrrv4}, and \eqref{rsquareddrrv5}, we get 
\begin{equation} \label{drrf4pointwise} r^{2}|\partial_{rr}F_{4}(t,r)| \leq C \frac{\mathbbm{1}_{\{r \leq \log^{N}(t)\}}r}{t^{2} \log^{3b+1-2b\alpha}(t)(r^{2}+\lambda(t)^{2})^{2}} + \frac{C \mathbbm{1}_{\{r \leq \frac{t}{2}\}} r}{t^{2} \log^{5b+2N-1}(t)(r^{2}+\lambda(t)^{2})^{2}}\end{equation}
\end{proof}
\subsection{Estimates on $F_{5}$}
Recall that $F_{5}$ was one of the $v_{6}$-independent error terms on the right-hand side of \eqref{v6eqn} which was not included in the modulation equation, and is therefore, not necessarily orthogonal to $\phi_{0}(\frac{\cdot}{\lambda(t)})$. Hence, we must prove that it decays sufficiently quickly in sufficiently many norms. This is the result of the following lemma.
\begin{lemma}
We have the following estimates on $F_{5}$: 
\begin{equation} \label{f5l2est}\frac{1}{\lambda(x)^{2}} ||F_{5}(x,r)||_{L^{2}(r dr)} \leq \frac{C \log^{6+2b}(x)}{x^{17/4}}\end{equation}
\begin{equation}\label{f5h1dotest}\frac{||F_{5}(x,\cdot\lambda(x))||_{\dot{H}^{1}_{e}}}{\lambda(x)} \leq \frac{C \log^{6+b}(x)}{x^{35/8}}\end{equation}
\end{lemma}
\begin{proof}
We recall \eqref{f5def}:
\begin{equation}\begin{split} F_{5}(t,r)&=N_{2}(v_{5})(t,r)+\frac{\sin(2(v_{1}+v_{2}+v_{3}+v_{4}))}{2r^{2}} \left(\cos(2Q_{\frac{1}{\lambda(t)}}+2v_{5})-\cos(2Q_{\frac{1}{\lambda(t)}})\right)\\
&+\left(\frac{\cos(2(v_{1}+v_{2}+v_{3}+v_{4}))-1}{2r^{2}}\right)\left(\sin(2Q_{\frac{1}{\lambda(t)}}+2v_{5})-\sin(2Q_{\frac{1}{\lambda(t)}})\right)\end{split}\end{equation}
Then, we start with
\begin{equation} |N_{2}(v_{5})|(t,r) \leq \begin{cases} \frac{C r}{(r^{2}+\lambda(t)^{2}) t^{7} \log^{7b-6+5N}(t)}, \quad r \leq \frac{t}{2}\\
\frac{C \log^{8}(t)}{r^{7/2} \log^{b}(t) t^{7/2}}, \quad r \geq \frac{t}{2}\end{cases}\end{equation}
Next, we have
\begin{equation}\begin{split} &|\frac{\sin(2(v_{1}+v_{2}+v_{3}+v_{4}))}{2r^{2}} \left(\cos(2Q_{\frac{1}{\lambda(t)}})\left(\cos(2v_{5})-1\right)-\sin(2Q_{\frac{1}{\lambda(t)}})\sin(2v_{5})\right)|\\
&\leq C \frac{|v_{1}+v_{2}+v_{3}+v_{4}|}{r^{2}} \left(v_{5}(t,r)^{2} + \frac{r \lambda(t)}{r^{2}+\lambda(t)^{2}}|v_{5}(t,r)|\right)\end{split}\end{equation}
\begin{equation} \begin{split} &|\left(\frac{\cos(2(v_{1}+v_{2}+v_{3}+v_{4}))-1}{2r^{2}}\right)\left(\sin(2Q_{\frac{1}{\lambda(t)}})\left(\cos(2v_{5})-1\right)+\cos(2Q_{\frac{1}{\lambda(t)}})\sin(2v_{5})\right)|\\
&\leq \frac{C (v_{1}+v_{2}+v_{3}+v_{4})^{2}}{r^{2}} \left(\frac{r \lambda(t) v_{5}^{2}}{r^{2}+\lambda(t)^{2}} + |v_{5}|\right)\end{split}\end{equation}
Using our previous pointwise estimates on $v_{1},v_{2},v_{3},v_{4},v_{5}$, we get
\begin{equation} |F_{5}(t,r)| \leq \begin{cases} \frac{C r}{(r^{2}+\lambda(t)^{2}) t^{11/2} \log^{5b-3+\frac{5N}{2}}(t)}, \quad r \leq \frac{t}{2}\\
\frac{C \log(r) \log^{4}(t)}{r^{3} t^{2} \log^{b}(t) |t-r|} + \frac{C \log^{4}(t) \log^{2}(r)}{(t-r)^{2} r^{5/2} t^{3/2}} + \frac{C \log^{4}(t)}{r^{7/2} t^{3} \log^{\frac{3N}{2}+4b-1}(t)}, \quad \frac{t}{2} \leq r \leq t-\sqrt{t}, \text{ or }r > t+\sqrt{t}\\
\frac{\log^{4}(t)}{r^{7/2} t^{3/2}}, \quad t-\sqrt{t} \leq r \leq t+\sqrt{t} \end{cases}\end{equation}
This gives
\begin{equation} \frac{1}{\lambda(x)^{2}} ||F_{5}(x,r)||_{L^{2}(r dr)} \leq \frac{C \log^{6+2b}(x)}{x^{17/4}}\end{equation}\\
We then proceed to estimate $||F_{5}(t,\cdot\lambda(t))||_{\dot{H}^{1}_{e}}$\\
We start with
\begin{equation} \label{drn2v5terms}|\partial_{r}N_{2}(v_{5})(t,r)| \leq \frac{C \lambda(t)}{r^{2}(r^{2}+\lambda(t)^{2})} v_{5}(t,r)^{2} + \frac{C \lambda(t) |v_{5}(t,r) \partial_{r}v_{5}(t,r)|}{r(r^{2}+\lambda(t)^{2})} + \frac{C |\partial_{r}v_{5}(t,r)| v_{5}(t,r)^{2}}{r^{2}} + \frac{C |v_{5}(t,r)|^{3}}{r^{3}}\end{equation}
For the two terms on the right-hand side of the above equation which involve $\partial_{r}v_{5}$, we estimate as follows:
\begin{equation} \begin{split} ||\frac{\lambda(t) |v_{5}(t,r) \partial_{r}v_{5}(t,r)|}{r(r^{2}+\lambda(t)^{2})}\vert_{r = R\lambda(t)}||_{L^{2}( R dR)} &\leq C \lambda(t) ||\frac{v_{5}(t,r)}{r(r^{2}+\lambda(t)^{2})}||_{L^{\infty}} \cdot ||(\partial_{2}v_{5})(t,\cdot \lambda(t))||_{L^{2}(R dR)}\\
&\leq \frac{C}{t^{21/4} \log^{b-6+\frac{5N}{2}}(t)}\end{split}\end{equation}
\begin{equation}\begin{split} ||\frac{|\partial_{r}v_{5}(t,r)| v_{5}(t,r)^{2}}{r^{2}}\vert_{r = R\lambda(t)}||_{L^{2}(R dR)} &\leq ||(\partial_{2}v_{5})(t,\cdot \lambda(t))||_{L^{2}(R dR)} \cdot ||\frac{v_{5}^{2}}{r^{2}}||_{L^{\infty}}\\
&\leq \frac{C \log^{11+b}(t)}{t^{31/4}}\end{split}\end{equation}
For the other terms in \eqref{drn2v5terms}, we use the $v_{5}$ pointwise estimates. In total, we get
\begin{equation} ||(\partial_{2}N_{2}(v_{5}))(t,\cdot \lambda(t))||_{L^{2}(R dR)} \leq \frac{C}{t^{21/4} \log^{b-6+\frac{5N}{2}}(t)}\end{equation}
Next, we have
\begin{equation} \begin{split} &\partial_{r}\left(\frac{\sin(2(v_{1}+v_{2}+v_{3}+v_{4}))}{2r^{2}}\left(\cos(2Q_{\frac{1}{\lambda(t)}}+2v_{5})-\cos(2Q_{\frac{1}{\lambda(t)}})\right)\right)\\
&= \frac{\cos(2(v_{1}+v_{2}+v_{3}+v_{4}))\partial_{r}(v_{1}+v_{2}+v_{3}+v_{4})}{r^{2}} \left(\cos(2Q_{\frac{1}{\lambda(t)}}+2v_{5})-\cos(2Q_{\frac{1}{\lambda(t)}})\right)\\
&-\frac{\sin(2(v_{1}+v_{2}+v_{3}+v_{4}))}{r^{3}} \left(\cos(2Q_{\frac{1}{\lambda(t)}}+2v_{5})-\cos(2Q_{\frac{1}{\lambda(t)}})\right)\\
&+\frac{\sin(2(v_{1}+v_{2}+v_{3}+v_{4}))}{r^{2}} \left(-\sin(2Q_{\frac{1}{\lambda(t)}}+2v_{5})\partial_{r}(Q_{\frac{1}{\lambda(t)}}+v_{5})+\sin(2Q_{\frac{1}{\lambda(t)}})\partial_{r}Q_{\frac{1}{\lambda(t)}}\right)\end{split}\end{equation}
For $\partial_{r}v_{k}, \quad k=1,2,4$, we use the pointwise estimates from the previous subsections. On the other hand, for $\partial_{r}v_{3}$, we use the energy estimate procedure, previously used for $\partial_{r}v_{5}$, to get
\begin{equation} \begin{split} &||\partial_{r}v_{3}(t)||_{L^{2}(r dr)} \leq C \int_{t}^{\infty} ||F_{0,1}(s)||_{L^{2}(r dr)} ds \leq C \int_{t}^{\infty} \frac{ \sqrt{\log(\log(s))}}{s^{2} \log^{b+1}(s)} ds\\
&\leq \frac{C \sqrt{\log(\log(t))}}{t \log^{b+1}(t)}\end{split}\end{equation}
Then, we have
\begin{equation}\begin{split} &||\frac{\cos(2(v_{1}+v_{2}+v_{3}+v_{4}))\partial_{r}(v_{1}+v_{2}+v_{3}+v_{4})}{r^{2}} \left(\cos(2Q_{\frac{1}{\lambda(t)}}+2v_{5})-\cos(2Q_{\frac{1}{\lambda(t)}})\right)\vert_{r=R\lambda(t)}||_{L^{2}(R dR)}\\
&\leq ||\frac{\partial_{r}(v_{1}+v_{2}+v_{4})}{r^{2}} \left(\cos(2Q_{\frac{1}{\lambda(t)}}+2v_{5})-\cos(2Q_{\frac{1}{\lambda(t)}})\right)\vert_{r=R\lambda(t)}||_{L^{2}(R dR)}\\
&+||(\partial_{2}v_{3})(t,R \lambda(t))||_{L^{2}(R dR)}||\frac{ \left(\cos(2Q_{\frac{1}{\lambda(t)}}+2v_{5})-\cos(2Q_{\frac{1}{\lambda(t)}})\right)}{r^{2}}||_{L^{\infty}}\\
&\leq \frac{C \sqrt{\log(\log(t))}}{t^{9/2} \log^{2b-2+\frac{5N}{2}}(t)}\end{split}\end{equation}
Using the pointwise estimates on $v_{5}$, we get
\begin{equation}\begin{split} &||-\frac{\sin(2(v_{1}+v_{2}+v_{3}+v_{4}))}{r^{3}} \left(\cos(2Q_{\frac{1}{\lambda(t)}}+2v_{5})-\cos(2Q_{\frac{1}{\lambda(t)}})\right)\vert_{r = R\lambda(t)}||_{L^{2}(R dR)} \\
&\leq \frac{C}{t^{11/2} \log^{3b-3+\frac{5N}{2}}(t)}\end{split}\end{equation}

We use
\begin{equation} \label{3drv5}\begin{split}&|\frac{\sin(2(v_{1}+v_{2}+v_{3}+v_{4}))}{r^{2}} \left(-\sin(2Q_{\frac{1}{\lambda(t)}}+2v_{5})\partial_{r}(Q_{\frac{1}{\lambda(t)}}+v_{5})+\sin(2Q_{\frac{1}{\lambda(t)}})\partial_{r}Q_{\frac{1}{\lambda(t)}}\right)|\\
&\leq \frac{C|v_{1}+v_{2}+v_{3}+v_{4}|}{r^{2}}\left(\frac{r \lambda(t)^{2} v_{5}^{2}}{(r^{2}+\lambda(t)^{2})^{2}} + \frac{\lambda(t) |v_{5}(t,r)|}{r^{2}+\lambda(t)^{2}} + \frac{|\partial_{r}v_{5}| r \lambda(t)}{r^{2}+\lambda(t)^{2}} + |v_{5}\partial_{r}v_{5}|\right)\end{split}\end{equation}
The first term in \eqref{3drv5} which involves $\partial_{r}v_{5}$ is estimated as follows:
\begin{equation}\label{3cdrv5}\begin{split} &||\frac{|v_{1}+v_{2}+v_{3}+v_{4}|}{r^{2}}\left(\frac{|\partial_{r}v_{5}| r \lambda(t)}{r^{2}+\lambda(t)^{2}}\right)\vert_{r=R\lambda(t)}||_{L^{2}(R dR)}\\
&\leq C ||\frac{|v_{1}+v_{2}+v_{3}+v_{4}|}{r^{2}}\left(\frac{|\partial_{r}v_{5}| r \lambda(t)}{r^{2}+\lambda(t)^{2}}\right)\vert_{r=R\lambda(t)}||_{L^{2}((0,\frac{t}{2 \lambda(t)}),R dR)} \\
&+ C ||(\partial_{2}v_{5})(t,R \lambda(t))||_{L^{2}(R dR)} \cdot ||\frac{(v_{1}+v_{2}+v_{3}+v_{4})}{r(r^{2}+\lambda(t)^{2})} \lambda(t)||_{L^{\infty}(r \geq \frac{t}{2})}\end{split}\end{equation}
where the second line of \eqref{3cdrv5} is estimated using the pointwise estimates on $\partial_{r}v_{5}$ in the region $r \leq \frac{t}{2}$. In total, we get
\begin{equation}\begin{split} &||\frac{|v_{1}+v_{2}+v_{3}+v_{4}|}{r^{2}}\left(\frac{|\partial_{r}v_{5}| r \lambda(t)}{r^{2}+\lambda(t)^{2}}\right)\vert_{r=R\lambda(t)}||_{L^{2}(R dR)}\leq \frac{C \log^{4}(t)}{t^{21/4}}\end{split}\end{equation}
The second term of \eqref{3drv5} which involves $\partial_{r}v_{5}$ is estimated as follows:
\begin{equation} ||\frac{v_{1}+v_{2}+v_{3}+v_{4}}{r}||_{L^{\infty}}\cdot ||\frac{v_{5}}{r}||_{L^{\infty}} \cdot ||(\partial_{2}v_{5})(t,\cdot \lambda(t))||_{L^{2}(R dR)} \leq \frac{C \log^{8}(t)}{t^\frac{25}{4} \lambda(t)}\end{equation}
The other terms of \eqref{3drv5} are estimated using the $v_{k}$ pointwise estimates for $1\leq k \leq 5$. In total, we get
\begin{equation}\begin{split} &||\frac{\sin(2(v_{1}+v_{2}+v_{3}+v_{4}))}{r^{2}} \left(-\sin(2Q_{\frac{1}{\lambda(t)}}+2v_{5})\partial_{r}(Q_{\frac{1}{\lambda(t)}}+v_{5})+\sin(2Q_{\frac{1}{\lambda(t)}})\partial_{r}Q_{\frac{1}{\lambda(t)}}\right)\vert_{r=R\lambda(t)}||_{L^{2}(R dR)} \\
&\leq \frac{C \log^{4}(t)}{t^{21/4}}\end{split}\end{equation}
Finally, we have
\begin{equation}\label{stardrv5}\begin{split} &\partial_{r} \left(\left(\frac{\cos(2(v_{1}+v_{2}+v_{3}+v_{4}))-1}{2r^{2}}\right)\left(\sin(2Q_{\frac{1}{\lambda(t)}}+2v_{5})-\sin(2Q_{\frac{1}{\lambda(t)}})\right)\right)\\
&= \frac{-\sin(2(v_{1}+v_{2}+v_{3}+v_{4}))}{r^{2}} \partial_{r}(v_{1}+v_{2}+v_{3}+v_{4}) \left(\sin(2Q_{\frac{1}{\lambda(t)}}+2v_{5})-\sin(2Q_{\frac{1}{\lambda(t)}})\right)\\
&-\frac{(\cos(2(v_{1}+v_{2}+v_{3}+v_{4}))-1)}{r^{3}}\left(\sin(2Q_{\frac{1}{\lambda(t)}}+2v_{5})-\sin(2Q_{\frac{1}{\lambda(t)}})\right)\\
&+\left(\frac{\cos(2(v_{1}+v_{2}+v_{3}+v_{4}))-1}{r^{2}}\right)\left(\cos(2Q_{\frac{1}{\lambda(t)}}+2v_{5})(\partial_{r}Q_{\frac{1}{\lambda(t)}}+\partial_{r}v_{5})-\cos(2Q_{\frac{1}{\lambda(t)}})\partial_{r}Q_{\frac{1}{\lambda(t)}}\right)\end{split}\end{equation}
The second line of \eqref{stardrv5} is estimated by
\begin{equation} \begin{split} &||\frac{-\sin(2(v_{1}+v_{2}+v_{3}+v_{4}))}{r^{2}} \partial_{r}(v_{1}+v_{2}+v_{3}+v_{4}) \left(\sin(2Q_{\frac{1}{\lambda(t)}}+2v_{5})-\sin(2Q_{\frac{1}{\lambda(t)}})\right)\vert_{r = R\lambda(t)}||_{L^{2}(R dR)} \\
&\leq C ||\frac{v_{1}+v_{2}+v_{3}+v_{4}}{r}||_{L^{\infty}} \left(||\frac{v_{5} \partial_{r}(v_{1}+v_{2}+v_{4})}{r}\vert_{r=R\lambda(t)}||_{L^{2}(R dR)}+ ||\frac{v_{5}}{r}||_{L^{\infty}} ||(\partial_{2}v_{3})(t,\cdot\lambda(t))||_{L^{2}(R dR)}\right)\\
&\leq \frac{C \log^{6+b}(t)}{t^{35/8}}\end{split}\end{equation}
The thrid line of \eqref{stardrv5} is estimated by
\begin{equation}\begin{split} &||-\frac{(\cos(2(v_{1}+v_{2}+v_{3}+v_{4}))-1)}{r^{3}}\left(\sin(2Q_{\frac{1}{\lambda(t)}}+2v_{5})-\sin(2Q_{\frac{1}{\lambda(t)}})\right)\vert_{r=R\lambda(t)}||_{L^{2}(R dR)}\\
&\leq C ||\frac{v_{1}+v_{2}+v_{3}+v_{4}}{r}||_{L^{\infty}}^{2} ||\frac{v_{5}}{r}\vert_{r=R\lambda(t)}||_{L^{2}(R dR)}\\
&\leq \frac{C \log^{6+b}(t)}{t^{5}}\end{split}\end{equation}
Finally, the fourth line of \eqref{stardrv5} is estimated by
\begin{equation}\begin{split} &||\left(\frac{\cos(2(v_{1}+v_{2}+v_{3}+v_{4}))-1}{r^{2}}\right)\left(\cos(2Q_{\frac{1}{\lambda(t)}}+2v_{5})(\partial_{r}Q_{\frac{1}{\lambda(t)}}+\partial_{r}v_{5})-\cos(2Q_{\frac{1}{\lambda(t)}})\partial_{r}Q_{\frac{1}{\lambda(t)}}\right)\vert_{r=R\lambda(t)}||_{L^{2}(R dR)}\\
&\leq ||\frac{v_{1}+v_{2}+v_{3}+v_{4}}{r}||_{L^{\infty}}^{2} \left(||\left(\lambda(t) \frac{v_{5}^{2}}{\lambda(t)^{2}+r^{2}}+\frac{r \lambda(t)^{2} v_{5}}{(r^{2}+\lambda(t)^{2})^{2}}\right)\vert_{r=R\lambda(t)}||_{L^{2}(R dR)} + ||(\partial_{2}v_{5})(t,\cdot \lambda(t))||_{L^{2}(R dR)}\right)\\
&\leq \frac{C \log^{5+b}(t)}{t^{19/4}}\end{split}\end{equation}
Finally, we use our pointwise estimates on all $v_{k}$ to get
\begin{equation} \left(\frac{1}{\lambda(t)^{2}} \int_{0}^{\infty} \frac{(F_{5}(t,R\lambda(t)))^{2}}{R^{2}} R dR\right)^{1/2} \leq \frac{C \log^{6+b}(t)}{t^{21/4}}\end{equation}
Combining all of these estimates, we get
\begin{equation} \frac{1}{\lambda(t)}||F_{5}(t,\cdot\lambda(t))||_{\dot{H}^{1}_{e}} \leq \frac{C \log^{6+b}(t)}{t^{35/8}}\end{equation}
\end{proof}

\subsection{Estimates on $F_{6}$}
We prove the analgous estimates on $F_{6}$:
\begin{lemma} 

\begin{equation}\label{f6l2est} \frac{1}{\lambda(t)^{2}} ||F_{6}(t,r)||_{L^{2}(r dr)} \leq \frac{C}{t^{4} \log^{3b+2N-1}(t)}\end{equation}

\begin{equation}\label{f6h1dotest} \frac{1}{\lambda(t)}||F_{6}(t,r)||_{\dot{H}^{1}_{e}} \leq \frac{C}{t^{9/2} \log^{4b-1+\frac{5N}{2}}(t)}\end{equation}
\end{lemma}

\begin{proof}

We start with
\begin{equation}\begin{split} ||F_{6}(t,r)||_{L^{2}(r dr)}^{2} &\leq C \int_{\frac{t}{4}}^{\frac{t}{2}} \frac{\lambda(t)^{4} r (v_{4}(t,r)^{2}+v_{5}(t,r)^{2})}{r^{8}} dr + C \int_{\frac{t}{2}}^{\infty} \frac{\lambda(t)^{4} (v_{4}^{2}+v_{5}^{2}) r dr}{r^{8}}\\
&\leq \frac{C}{t^{8} \log^{10b+4N-2}(t)} \end{split}\end{equation}
where we used \eqref{v4finalest}, \eqref{v5finalest}. This concludes the proof of \eqref{f6l2est}. 

Next, we have
\begin{equation} \begin{split} \partial_{r}F_{6}(t,r) &= \chi_{\geq 1}'(\frac{4r}{t}) \left(\frac{4}{t}\right) \left(\frac{\cos(2Q_{\frac{1}{\lambda(t)}}(r))-1}{r^{2}}\right) \left(v_{4}+v_{5}\right)+\chi_{\geq 1}(\frac{4r}{t})\left(v_{4}+v_{5}\right) \partial_{r}\left(\frac{\cos(2Q_{\frac{1}{\lambda(t)}}(r))-1}{r^{2}}\right)\\
&+\chi_{\geq 1}(\frac{4r}{t}) \left(\frac{\cos(2Q_{\frac{1}{\lambda(t)}}(r))-1}{r^{2}}\right) \partial_{r}(v_{4}+v_{5})\end{split}\end{equation}
We estimate the $L^{2}$ norm as follows.
\begin{equation}\begin{split} ||\partial_{r}F_{6}(t,r)||_{L^{2}(r dr)} &\leq \frac{C}{t^{9/2} \log^{5b-1+\frac{5N}{2}}(t)} + ||\partial_{r}v_{5}||_{L^{2}(r dr)}\cdot ||\frac{\chi_{\geq 1}(\frac{4r}{t}) \lambda(t)^{2}}{(r^{2}+\lambda(t)^{2})^{2}}||_{L^{\infty}_{r}}\\
&\leq \frac{C}{t^{9/2} \log^{5b-1+\frac{5N}{2}}(t)}\end{split}\end{equation}
where we used \eqref{v4finalest}, \eqref{drv4finalest}, \eqref{v5finalest}, and \eqref{v5energyest}. The last term to estimate is
\begin{equation} ||\frac{F_{6}(t,r)}{r}||_{L^{2}(r dr)} \leq \frac{C}{t^{5} \log^{5b+2N-1}(t)}\end{equation}
where we used \eqref{v4finalest} and \eqref{v5finalest}. Combining these, we get
\begin{equation} \frac{1}{\lambda(t)}||F_{6}(t,r)||_{\dot{H}^{1}_{e}} \leq \frac{C}{t^{9/2} \log^{4b-1+\frac{5N}{2}}(t)}\end{equation}

\end{proof}
\subsection{Estimates on $v_{corr}$-dependent quantities}
Finally, we define $v_{corr}:=v_{1}+v_{2}+v_{3}+v_{4}+v_{5}$ and record some estimates on $v_{corr}$-dependent quantities which will appear as coefficients of various error terms involving the final correction, which is to be constructed in the next section.
\begin{lemma}\label{vcorrests}
We have the following estimates
\begin{equation}\label{vcorrcof} ||\frac{v_{corr}(x,R\lambda(x))}{R \lambda(x)}||_{L^{\infty}}^{2}+||\frac{v_{corr}(x,R\lambda(x))}{R\lambda(x)^{2}(1+R^{2})}||_{L^{\infty}} \leq \frac{C \log(\log(x))}{x^{2}\log(x)}\end{equation}

\begin{equation}\label{1cof} 1+||\frac{v_{corr}(x,R\lambda(x))}{R}||_{L^{\infty}}+||\partial_{R}(v_{corr}(x,R\lambda(x)))||_{L^{\infty}} \leq C\end{equation}

\begin{equation}\label{vcorrdrvcorrcof}\begin{split} &||\frac{v_{corr}(x,R\lambda(x)) \partial_{R}(v_{corr}(x,R\lambda(x)))}{R\lambda(x)^{2}}||_{L^{\infty}_{R}((0,1))}+||\frac{v_{corr}(x,R\lambda(x)) \partial_{R}(v_{corr}(x,R\lambda(x)))}{R^{2}\lambda(x)^{2}}||_{L^{\infty}_{R}((1,\infty))}\\
&+||\frac{\partial_{R}(v_{corr}(x,R\lambda(x)))}{(1+R^{2})\lambda(x)^{2}}||_{L^{\infty}} \\
&\leq \frac{C \log(\log(x))}{x^{2}\log(x)}\end{split}\end{equation}

\end{lemma}
\begin{proof}
  From \eqref{v1plusv2delicate}, \eqref{v3laterest}, \eqref{v3largerest}, \eqref{v4finalest}, and \eqref{v5finalest}, we get
\begin{equation} |v_{corr}(t,r)| \leq \begin{cases} \frac{C r \log(\log(t))}{t^{2} \log^{b+1}(t)}, \quad r \leq \log^{N}(t)\\
\frac{C r}{t^{2} \log^{b}(t)}, \quad \log^{N}(t) \leq r \leq \frac{t}{2}\\
\frac{C}{\sqrt{r}}, \quad \frac{t}{2} < r\end{cases}\end{equation}
This gives
\begin{equation} ||\frac{v_{corr}(x,R\lambda(x))}{R \lambda(x)}||_{L^{\infty}}^{2}+||\frac{v_{corr}(x,R\lambda(x))}{R\lambda(x)^{2}(1+R^{2})}||_{L^{\infty}} \leq \frac{C \log(\log(x))}{x^{2}\log(x)} \end{equation}
Then, we use \eqref{drv1plusv2delicate},\eqref{drv3nearoriginrefinement}, \eqref{drv3est}, \eqref{drv4finalest}, and \eqref{drv5finalest} to get
\begin{equation} |\partial_{R}(v_{corr}(x,R\lambda(x)))| \leq C\lambda(x) \begin{cases} \frac{\log(\log(x))}{x^{2} \log^{b+1}(x)}, \quad R \lambda(x) \leq \log^{N}(x)\\
\frac{1}{x^{2} \log^{b}(x)}, \quad  \log^{N}(x) \leq R\lambda(x) \leq \frac{x}{2}\\
\frac{1}{\sqrt{x}}, \quad R \lambda(x) > \frac{x}{2}\end{cases}\end{equation}
This implies

\begin{equation} 1+||\frac{v_{corr}(x,R\lambda(x))}{R}||_{L^{\infty}}+||\partial_{R}(v_{corr}(x,R\lambda(x)))||_{L^{\infty}} \leq C\end{equation}

Next, we have
\begin{equation} |\frac{v_{corr}(x,R\lambda(x))}{R \lambda(x)}| \cdot |\frac{\partial_{R}(v_{corr}(x,R\lambda(x)))}{\lambda(x)}| \leq  C \frac{(\log(\log(x)))^{2}}{x^{4} \log^{2b+2}(x)}, \quad R  \leq 1\end{equation}

\begin{equation}|\frac{v_{corr}(x,R\lambda(x))}{R \lambda(x)}| \cdot |\frac{\partial_{R}(v_{corr}(x,R\lambda(x)))}{R \lambda(x)}| \leq \begin{cases} \frac{C (\log(\log(x)))^{2}}{R x^{4} \log^{2b+2}(x)}, \quad 1 \leq R  \leq \frac{\log^{N}(x)}{\lambda(x)}\\
\frac{C}{R x^{4} \log^{2b}(x)}, \quad \log^{N}(x) \leq R \lambda(x) \leq \frac{x}{2}\\
\frac{C}{R^{5/2} \lambda(x)^{3/2} \sqrt{x}}, \quad R \lambda(x) > \frac{x}{2}\end{cases}\end{equation}

and

\begin{equation} \frac{|\partial_{R}(v_{corr}(x,R\lambda(x)))|}{(1+R^{2})\lambda(x)^{2}} \leq \begin{cases} \frac{C \log(\log(x))}{x^{2} \log(x) (1+R^{2})}, \quad R \lambda(x) \leq \log^{N}(x)\\
\frac{1}{x^{2}(1+R^{2})}, \quad \log^{N}(x) \leq R\lambda(x) \leq \frac{x}{2}\\
\frac{\log^{b}(x)}{\sqrt{x}}\cdot \frac{1}{1+R^{2}}, \quad R\lambda(x) > \frac{x}{2}\end{cases}\end{equation}
which imply
\begin{equation}\begin{split} &||\frac{v_{corr}(x,R\lambda(x)) \partial_{R}(v_{corr}(x,R\lambda(x)))}{R\lambda(x)^{2}}||_{L^{\infty}_{R}((0,1))}+||\frac{v_{corr}(x,R\lambda(x)) \partial_{R}(v_{corr}(x,R\lambda(x)))}{R^{2}\lambda(x)^{2}}||_{L^{\infty}_{R}((1,\infty))}\\
&+||\frac{\partial_{R}(v_{corr}(x,R\lambda(x)))}{(1+R^{2})\lambda(x)^{2}}||_{L^{\infty}} \\
&\leq \frac{C \log(\log(x))}{x^{2}\log(x)}\end{split}\end{equation}

\end{proof}

\section{Solving the final equation}
The full equation to solve is \eqref{v6eqn} with 0 Cauchy data at infinity. For ease of notation, let us set $u=v_{6}$, and re-write \eqref{v6eqn} as
\begin{equation}\label{linearproblem}\begin{split}-\partial_{tt}u+\partial_{rr}u+\frac{1}{r} \partial_{r}u-\frac{\cos(2Q_{\frac{1}{\lambda(t)}}(r))}{r^{2}} u &= F(t,r)+F_{3}(t,r)
 \end{split}\end{equation}
where 
$$F(t,r) = F_{4}(t,r) + F_{5}(t,r)+F_{6}(t,r)$$
and we recall that $F_{4},F_{5}$, and $F_{6}$ are defined in \eqref{f4def}, \eqref{f5def}, and \eqref{f6def}, and are estimated in theorem \ref{approxsolnthm}.
and \begin{equation}\label{f3def}F_{3} = N(u)+L_{1}(u)\end{equation}
with
$$N(f) = \left(\frac{\sin(2f)-2f}{2r^{2}}\right)\cos(2Q_{\frac{1}{\lambda(t)}}) + \left(\frac{\cos(2f)-1}{2r^{2}}\right)\sin(2(Q_{\frac{1}{\lambda(t)}}+v_{corr}))$$
$$L_{1}(f)=\frac{\sin(2f)}{2r^{2}} \cos(2Q_{\frac{1}{\lambda(t)}})(\cos(2v_{corr})-1) -\frac{\sin(2f)}{2r^{2}}\sin(2Q_{\frac{1}{\lambda(t)}})\sin(2v_{corr})$$
$$v_{corr}=v_{1}+v_{2}+v_{3}+v_{4}+v_{5}$$

Note that we will utilize the crucial fact that $$\langle F_{4}(t,\cdot),\phi_{0}(\frac{\cdot}{\lambda(t)})\rangle =0$$
\subsection{The equation for $\mathcal{F}(u)$}
We will make appropriate changes of variables in order to (formally) derive the equation for the distorted Fourier transform, discussed in section 4 of \cite{kst}, of $u$. (Note, however, that we will not renormalize the time variable, unlike in \cite{kst}.) We will denote the distorted Fourier transform of a function $f$ by $\mathcal{F}(f)$. Let $$u(t,r) = v(t,\frac{r}{\lambda(t)})$$
Then, if we evaluate the equation for $u$ at the point $(t,R \lambda(t))$, we obtain
\begin{equation} \label{vlinearproblem}\begin{split}&-\partial_{11}v(t,R) +2 \frac{\lambda'(t)}{\lambda(t)} R \partial_{12}v(t,R) + \left(\frac{\lambda''(t)}{\lambda(t)}-2\frac{\lambda'(t)^{2}}{\lambda(t)^{2}}\right)R \partial_{2}v(t,R) -\frac{\lambda'(t)^{2}}{\lambda(t)^{2}} R^{2} \partial_{22}v(t,R) \\
&+ \frac{1}{\lambda(t)^{2}}\left(\partial_{22}v(t,R) + \frac{1}{R}\partial_{2}v(t,R) - \frac{\cos(2Q_{1}(R))}{R^{2}} v(t,R)\right)=F(t,R \lambda(t))+F_{3}(t,R\lambda(t))\end{split}\end{equation}
Now, let $$v(t,R) = \frac{w(t,R)}{\sqrt{R}}$$ to get
\begin{equation}\label{wlinearproblem}\begin{split}&-\partial_{11}w(t,R)-\frac{\lambda'(t)}{\lambda(t)} \partial_{1}w(t,R) + 2\frac{\lambda'(t)}{\lambda(t)} \partial_{1}(R \partial_{2}w)(t,R) + \left(\frac{-\lambda''(t)}{2\lambda(t)}+\frac{1}{4}\frac{\lambda'(t)^{2}}{\lambda(t)^{2}}\right)w(t,R) \\
&+ \left(\frac{\lambda''(t)}{\lambda(t)} - \frac{\lambda'(t)^{2}}{\lambda(t)^{2}}\right) R\partial_{2}w(t,R) -\frac{\lambda'(t)^{2}}{\lambda(t)^{2}} R^{2} \partial_{22}w(t,R) \\
&+ \frac{1}{\lambda(t)^{2}}\left(\partial_{22}w(t,R) -\left(\frac{3}{4R^{2}}-\frac{8}{(1+R^{2})^{2}}\right)w(t,R)\right)=\sqrt{R}F(t,R \lambda(t))+\sqrt{R}F_{3}(t,R\lambda(t))\end{split}\end{equation}
Next, from (5.1) of \cite{kst}, we have
$$\mathcal{F}(R\partial_{R}w) = -2 \xi \partial_{\xi} \mathcal{F}(w)+K(\mathcal{F}(w))$$
where we will use various estimates on $K$, proven in \cite{kst}, later on. Making the final change of variable
$$y(t,\xi) = \mathcal{F}(w)(t,\xi \lambda(t)^{2})$$
and evaluating the distorted Fourier transform of \eqref{wlinearproblem} at the point $(t,\omega \lambda(t)^{2})$, we get
\begin{equation}\label{ylinearproblem}\begin{split} &-\partial_{tt} y(t,\omega) - \omega y(t,\omega) -\frac{\lambda'(t)}{\lambda(t)} \partial_{t}y(t,\omega) + \frac{2 \lambda'(t)}{\lambda(t)} K\left(\partial_{1}y(t,\frac{\cdot}{\lambda(t)^{2}})\right)(\omega \lambda(t)^{2}) +\left(\frac{-\lambda''(t)}{2\lambda(t)}+\frac{\lambda'(t)^{2}}{4 \lambda(t)^{2}}\right)y(t,\omega) \\
&+ \frac{\lambda''(t)}{\lambda(t)} K\left(y(t,\frac{\cdot}{\lambda(t)^{2}})\right)(\omega \lambda(t)^{2}) +2\frac{\lambda'(t)^{2}}{\lambda(t)^{2}}\left([\xi \partial_{\xi},K](y(t,\frac{\cdot}{\lambda(t)^{2}}))\right)(\omega \lambda(t)^{2}) \\
&-\frac{\lambda'(t)^{2}}{\lambda(t)^{2}} K\left(K(y(t,\frac{\cdot}{\lambda(t)^{2}}))\right)(\omega \lambda(t)^{2})=\mathcal{F}(\sqrt{\cdot}F(t,\cdot \lambda(t)))(\omega \lambda(t)^{2})+\mathcal{F}(\sqrt{\cdot}F_{3}(u(y))(t,\cdot \lambda(t)))(\omega \lambda(t)^{2})\end{split}\end{equation}\\
\\
where we write $F_{3}(u(y))$ to emphasize the dependence of $F_{3}$ on $y$, which is related to $u$ via 
\begin{equation}\label{yu}y(t,\xi) = \mathcal{F}(\sqrt{\cdot} u(t,\cdot \lambda(t)))(\xi \lambda(t)^{2})\end{equation}
\subsection{Estimates on $F_{2}$}
Let \begin{equation}\begin{split}F_{2}(y)(t,\omega) &=-\frac{\lambda'(t)}{\lambda(t)} \partial_{t}y(t,\omega) + \frac{2 \lambda'(t)}{\lambda(t)} K\left(\partial_{1}y(t,\frac{\cdot}{\lambda(t)^{2}})\right)(\omega \lambda(t)^{2}) +\left(\frac{-\lambda''(t)}{2\lambda(t)}+\frac{\lambda'(t)^{2}}{4 \lambda(t)^{2}}\right)y(t,\omega) \\
&+ \frac{\lambda''(t)}{\lambda(t)} K\left(y(t,\frac{\cdot}{\lambda(t)^{2}})\right)(\omega \lambda(t)^{2}) +2\frac{\lambda'(t)^{2}}{\lambda(t)^{2}}\left([\xi \partial_{\xi},K](y(t,\frac{\cdot}{\lambda(t)^{2}}))\right)(\omega \lambda(t)^{2}) \\
&-\frac{\lambda'(t)^{2}}{\lambda(t)^{2}} K\left(K(y(t,\frac{\cdot}{\lambda(t)^{2}}))\right)(\omega \lambda(t)^{2})\end{split}\end{equation}\\
\\
Note that \eqref{ylinearproblem} becomes
\begin{equation}\label{finalyproblem} \partial_{tt}y+\omega y = -\mathcal{F}(\sqrt{\cdot}F(t,\cdot \lambda(t)))(\omega \lambda(t)^{2})+F_{2}(y)(t,\omega) -\mathcal{F}(\sqrt{\cdot}F_{3}(u(y))(t,\cdot \lambda(t)))(\omega \lambda(t)^{2})\end{equation}
Our goal is to prove the following proposition:
\begin{proposition}\label{f2prop} There exists $C_{1}>0$ such that, for all $y$ satisfying 
$$ y(t,\omega)  \sqrt{\rho(\omega \lambda(t)^{2})} \langle \omega \lambda(t)^{2}\rangle \in C^{0}_{t}([T_{0},\infty),L^{2}(d\omega))$$
and
$$\partial_{t}y(t,\omega)  \langle \sqrt{\omega} \lambda(t)\rangle \sqrt{\rho(\omega \lambda(t)^{2})} \in C^{0}_{t}([T_{0},\infty), L^{2}(d\omega))$$
we have the following inequalities, for $x \geq T_{0}$:
\begin{equation}\label{yinhomest}||F_{2}(y)(x)||_{L^{2}(\rho(\omega \lambda(x)^{2}) d\omega)} \leq \frac{C_{1}}{x \log(x)}\left(||\partial_{t}y(x)||_{L^{2}(\rho(\omega \lambda(x)^{2})d\omega)} + \frac{1}{x} ||y(x)||_{L^{2}(\rho(\omega \lambda(x)^{2}) d\omega)}\right)\end{equation}\\
\\
and
\begin{equation}\label{lyinhomest}\begin{split} &||\sqrt{\omega} \lambda(x) F_{2}(y)(x)||_{L^{2}(\rho(\omega \lambda(x)^{2}) d\omega)}\\
&\leq \frac{C_{1}}{x\log(x)}\left(||\sqrt{\omega}\lambda(x)\partial_{t}y(x)||_{L^{2}(\rho(\omega \lambda(x)^{2}) d\omega)} + ||\partial_{t}y(x)||_{L^{2}(\rho(\omega \lambda(x)^{2})d\omega)}\right)\\
&+\frac{C_{1}}{x^{2}\log(x)}\left(||\sqrt{\omega} \lambda(x) y(x)||_{L^{2}(\rho(\omega \lambda(x)^{2})d\omega)}+||y(x)||_{L^{2}(\rho(\omega \lambda(x)^{2})d\omega)}\right)\end{split}\end{equation}
\end{proposition}
\begin{proof}
By the symbol type bounds on $a$ from Proposition 4.7 of \cite{kst}, we have
$$|\frac{\xi \rho'(\xi)}{\rho(\xi)}| = |\frac{\xi(a \overline{a}'+\overline{a}a')}{|a|^{2}}| \leq C$$
Then, by (5.3) of \cite{kst}, there exists a constant $C$ such that, for $f \in C^{\infty}_{c}((0,\infty))$ and $\alpha =0, \frac{1}{2}$,  we have
$$||\langle \xi \rangle ^{\alpha}(Kf)||_{L^{2}(\rho d\xi)} \leq C \left( ||\langle \xi \rangle ^{\alpha} f||_{L^{2}(\rho d\xi)} + ||\langle \xi \rangle^{\alpha} (K_{0}(f))||_{L^{2}(\rho d\xi)}\right)$$
In addition, we use Proposition 5.2 of \cite{kst} to get, for $\alpha=0, \frac{1}{2}$,
$$||\langle \xi \rangle ^{\alpha+1/2}K_{0}f||_{L^{2}(\rho d\xi)} \leq C ||\langle \xi \rangle ^{\alpha}f||_{L^{2}(\rho d\xi)}$$
$$||\langle \xi \rangle^{\alpha}[\xi \partial_{\xi},K_{0}]f||_{L^{2}(\rho d\xi)} \leq C ||\langle \xi \rangle^{\alpha}f||_{L^{2}(\rho d\xi)}$$
So, for $\alpha=0, \frac{1}{2}$
$$||\langle \xi \rangle ^{\alpha}Kf||_{L^{2}(\rho d\xi)} \leq C ||\langle \xi \rangle ^{\alpha}f||_{L^{2}(\rho d\xi)}$$\\
\\
Also, since
$$[\xi \partial_{\xi},K]f = [\xi \partial_{\xi},K_{0}]f - \xi \partial_{\xi}\left(\frac{\xi\rho'(\xi)}{\rho(\xi)}\right) f$$
the symbol type bounds on $a$ imply that for $\alpha = 0,\frac{1}{2}$, 
$$||\langle \xi \rangle^{\alpha}[\xi \partial_{\xi},K]f||_{L^{2}(\rho d\xi)} \leq C ||\langle \xi \rangle^{\alpha}f||_{L^{2}(\rho d\xi)}$$\\
\\
For later convenience, we will record estimates on some terms appearing in \eqref{ylinearproblem}, treating homogeneous components of a norm with weight $\langle \omega \lambda(x)^{2} \rangle^{1/2}$ seperately:
\begin{equation}\begin{split} ||K(\partial_{1}y(x,\frac{\cdot}{\lambda(x)^{2}}))(\omega \lambda(x)^{2})||_{L^{2}(\rho(\omega \lambda(x)^{2}) d\omega)}&=\left(\int_{0}^{\infty} \rho(\xi) \left(K(\partial_{1}y(x,\frac{\cdot}{\lambda(x)^{2}}))(\xi)\right)^{2} \frac{d\xi}{\lambda(x)^{2}}\right)^{1/2}\\
&\leq \frac{C}{\lambda(x)} \left(\int_{0}^{\infty} \rho(\xi) \left(\partial_{1}y(x,\frac{\xi}{\lambda(x)^{2}})\right)^{2} d\xi\right)^{1/2}\\
&\leq C ||\partial_{1}y(x)||_{L^{2}(\rho(\omega \lambda(x)^{2})d\omega)} \end{split}\end{equation}

\begin{equation}\begin{split}&||\sqrt{\omega} \lambda(x) K(\partial_{1}y(x,\frac{\cdot}{\lambda(x)^{2}}))(\omega \lambda(x)^{2})||^{2}_{L^{2}(\rho(\omega \lambda(x)^{2}) d\omega)} \\
&\leq C \int_{0}^{\infty} \left(K(\partial_{1}y(x,\frac{\cdot}{\lambda(x)^{2}}))(\omega \lambda(x)^{2})\right)^{2} \left(1+\lambda(x)^{4}\omega^{2}\right)^{1/2} \rho(\omega \lambda(x)^{2}) d\omega\\
&\leq C \int_{0}^{\infty} \left(K(\partial_{1}y(x,\frac{\cdot}{\lambda(x)^{2}}))(\xi)\right)^{2} \sqrt{1+\xi^{2}} \frac{\rho(\xi) d\xi}{\lambda(x)^{2}}\\
&\leq \frac{C}{\lambda(x)^{2}} \int_{0}^{\infty} \left(\partial_{1}y(x,\frac{\xi}{\lambda(x)^{2}})\right)^{2} \sqrt{1+\xi^{2}} \rho(\xi) d\xi\\
&\leq C \int_{0}^{\infty} \left(\partial_{1}y(x,\omega)\right)^{2} \left(1+\omega \lambda(x)^{2}\right) \rho(\omega \lambda(x)^{2}) d\omega\end{split}\end{equation}\\
\\
So, \begin{equation}\begin{split}&||\sqrt{\omega}\lambda(x) K(\partial_{1}y(x,\frac{\cdot}{\lambda(x)^{2}}))(\omega \lambda(x)^{2}) ||_{L^{2}(\rho(\omega \lambda(x)^{2})d\omega)} \\
&\leq C\left(||\partial_{1}y(x)||_{L^{2}(\rho(\omega \lambda(x)^{2}) d\omega)}+||\sqrt{\omega} \lambda(x)\partial_{1}y(x) ||_{L^{2}(\rho(\omega \lambda(x)^{2}) d\omega)}\right)\end{split}\end{equation}

Similarly, $$||\left([\xi \partial_{\xi},K](y(x,\frac{\cdot}{\lambda(x)^{2}}))\right)(\omega \lambda(x)^{2})||_{L^{2}(\rho(\omega \lambda(x)^{2}) d\omega)} \leq C ||y(x)||_{L^{2}(\rho(\omega \lambda(x)^{2}) d\omega)}$$

$$||K\left(K(y(x,\frac{\cdot}{\lambda(x)^{2}}))\right)(\omega \lambda(x)^{2})||_{L^{2}(\rho(\omega \lambda(x)^{2}) d\omega)} \leq C ||y(x)||_{L^{2}(\rho(\omega \lambda(x)^{2}) d\omega)}$$

$$||K\left(y(x,\frac{\cdot}{\lambda(x)^{2}})\right)(\omega \lambda(x)^{2})||_{L^{2}(\rho(\omega \lambda(x)^{2})d\omega)} \leq C ||y(x)||_{L^{2}(\rho(\omega \lambda(x)^{2}) d\omega)}$$

\begin{equation}\begin{split}&||\sqrt{\omega} \lambda(x)\left([\xi \partial_{\xi},K](y(x,\frac{\cdot}{\lambda(x)^{2}}))\right)(\omega \lambda(x)^{2})||_{L^{2}(\rho(\omega \lambda(x)^{2}) d\omega)} \\
&\leq C\left(||y(x)||_{L^{2}(\rho(\omega \lambda(x)^{2}) d\omega)}+||\sqrt{\omega} \lambda(x) y(x) ||_{L^{2}(\rho(\omega \lambda(x)^{2}) d\omega)}\right)\end{split}\end{equation}

\begin{equation}\begin{split}&||\sqrt{\omega} \lambda(x)\left(K\left(y(x,\frac{\cdot}{\lambda(x)^{2}})\right)(\omega \lambda(x)^{2})\right)||_{L^{2}(\rho(\omega \lambda(x)^{2}) d\omega)} \\
&\leq C\left(||y(x)||_{L^{2}(\rho(\omega \lambda(x)^{2}) d\omega)}+||\sqrt{\omega} \lambda(x) y(x) ||_{L^{2}(\rho(\omega \lambda(x)^{2}) d\omega)}\right)\end{split}\end{equation}

\begin{equation}\begin{split}&||\sqrt{\omega} \lambda(x)\left(K\left(K(y(x,\frac{\cdot}{\lambda(x)^{2}}))\right)(\omega \lambda(x)^{2})\right)||_{L^{2}(\rho(\omega \lambda(x)^{2}) d\omega)} \\
&\leq C\left(||y(x)||_{L^{2}(\rho(\omega \lambda(x)^{2}) d\omega)}+||\sqrt{\omega} \lambda(x) y(x) ||_{L^{2}(\rho(\omega \lambda(x)^{2}) d\omega)}\right)\end{split}\end{equation} 
We thus conclude the inequalities \eqref{yinhomest} and \eqref{lyinhomest}.\end{proof}

\subsection{$F_{3}$ Estimates}
Our goal in this subsection is to prove the following proposition.
\begin{proposition}\label{f3prop} For $y$ satisfying 
$$ y(t,\omega)  \sqrt{\rho(\omega \lambda(t)^{2})} \langle \omega \lambda(t)^{2}\rangle \in C^{0}_{t}([T_{0},\infty),L^{2}(d\omega))$$ 
let $F_{3}(u(y))$ be given by the expression \eqref{f3def}, where
$$u(t,r) = \sqrt{\frac{\lambda(t)}{r}} \mathcal{F}^{-1}\left(y(t,\frac{\cdot}{\lambda(t)^{2}})\right)(\frac{r}{\lambda(t)}), \quad r >0$$
Then, there exists an absolute constant $C>0$ \emph{independent} of $y$, such that
\begin{equation}\label{f3est}\begin{split} &||\mathcal{F}(\sqrt{\cdot}F_{3}(u(y))(t,\cdot \lambda(t)))(\omega \lambda(t)^{2})||_{L^{2}(\rho(\omega \lambda(t)^{2})d\omega)} \\
&\leq C ||y(t)||_{L^{2}(\rho(\omega \lambda(t)^{2})d\omega)} \left(||\frac{v_{corr}(t,R\lambda(t))}{R\lambda(t)}||_{L^{\infty}}^{2}+||\frac{v_{corr}(t,R\lambda(t))}{R\lambda(t)^{2}(1+R^{2})}||_{L^{\infty}}\right)\\
&+C||\langle \omega \lambda(t)^{2}\rangle y(t)||^{3}_{L^{2}(\rho(\omega \lambda(t)^{2})d\omega)}\\
&+\frac{C}{\lambda(t)}||\langle \sqrt{\omega}\lambda(t)\rangle y(t)||^{2}_{L^{2}(\rho(\omega \lambda(t)^{2})d\omega)}\left(1+||\frac{v_{corr}(t,R\lambda(t))}{R}||_{L^{\infty}}\right)\end{split}\end{equation}\\
\\
and
\begin{equation}\label{lf3est}\begin{split}&||\sqrt{\omega}\lambda(t) \mathcal{F}(\sqrt{\cdot}F_{3}(u(y))(t,\cdot\lambda(t)))(\omega \lambda(t)^{2})||_{L^{2}(\rho(\omega \lambda(t)^{2})d\omega)}\\
&\leq C||\langle\sqrt{\omega}\lambda(t)\rangle y(t)||_{L^{2}(\rho(\omega \lambda(t)^{2})d\omega)}\cdot\left(||\frac{v_{corr}(t,R\lambda(t))}{R\lambda(t)}||_{L^{\infty}}^{2}+||\frac{v_{corr}(t,R\lambda(t))}{R\lambda(t)^{2}(1+R^{2})}||_{L^{\infty}}\right.\\
&\left.+||\frac{\partial_{R}(v_{corr}(t,R\lambda(t)))}{(1+R^{2})\lambda(t)^{2}}||_{L^{\infty}}+||\frac{v_{corr}(t,R\lambda(t))\partial_{R}(v_{corr}(t,R\lambda(t)))}{R\lambda(t)^{2}}||_{L^{\infty}_{R}((0,1))}+||\frac{v_{corr}(t,R\lambda(t))\partial_{R}(v_{corr}(t,R\lambda(t)))}{R^{2}\lambda(t)^{2}}||_{L^{\infty}_{R}((1,\infty))}\right)\\
&+C ||\langle \omega \lambda(t)^{2}\rangle y(t)||_{L^{2}(\rho(\omega \lambda(t)^{2})d\omega)}^{3}\\
&+\frac{C}{\lambda(t)}\left(1+||\frac{v_{corr}(t,R\lambda(t))}{R}||_{L^{\infty}}+||\partial_{R}(v_{corr}(t,R\lambda(t)))||_{L^{\infty}}\right)||\langle \omega \lambda(t)^{2}\rangle y(t)||^{2}_{L^{2}(\rho(\omega \lambda(t)^{2})d\omega)}\end{split}\end{equation}
where
$$v_{corr}=v_{1}+v_{2}+v_{3}+v_{4}+v_{5}$$
\end{proposition}
\begin{proof}
To prove these estimates, it will be convenient to use the distorted Fourier transforms of \cite{bkt}, associated to the operators $L^{*}L$ and $LL^{*}$, where
$$L=\partial_{r}-\frac{\cos(Q_{1})}{r}, \quad L^{*}=-\partial_{r}-\frac{(\cos(Q_{1})+1)}{r}$$
Let us use $\tilde{\phi}_{\xi}(r)$ to denote the eigenfunctions denoted by $\phi_{\xi}(r)$ in \cite{bkt}, and $\phi(r,\xi)$ to denote the eigenfunctions of \cite{kst}. By the definitions of these eigenfunctions, there exists $f$ such that
\begin{equation}\label{efunctranslation}\sqrt{r}\tilde{\phi}_{\sqrt{\xi}}(r)=f(\sqrt{\xi})\phi(r,\xi)\end{equation}
To determine the expression for $f$, let $g \in \mathcal{S}$, with $g(0)=0$. Using the definitions of the Fourier transforms from \cite{bkt} and \cite{kst}, and \eqref{efunctranslation}, we have
\begin{equation} \begin{split} \mathcal{F}(g)(\xi) = \int_{0}^{\infty} \phi(r, \xi) g(r) dr = \int_{0}^{\infty} \frac{\sqrt{r} \tilde{\phi}_{\sqrt{\xi}}(r)}{f(\sqrt{\xi})} g(r) dr = \frac{1}{f(\sqrt{\xi})} \mathcal{F}_{H}(\frac{g(\cdot)}{\sqrt{\cdot}})(\sqrt{\xi})\end{split}\end{equation}
where $\mathcal{F}_{H}$ is the operator defined in section 3.1 of \cite{bkt}
Using the inversion formula from \cite{kst} we get
 \begin{equation} \begin{split}g(r) = \int_{0}^{\infty} \phi(r,\xi) \mathcal{F}(g)(\xi) \rho(\xi) d\xi = \int_{0}^{\infty} \frac{\sqrt{r}}{f(\sqrt{\xi})^{2}} \mathcal{F}_{H}(\frac{g(\cdot)}{\sqrt{\cdot}})(\sqrt{\xi}) \rho(\xi) \tilde{\phi}_{\sqrt{\xi}}(r) d\xi\end{split}\end{equation} 
So, for $g \in \mathcal{S}$, with $g(0)=0$, we have, for all $r \neq 0$, 
\begin{equation} \frac{g(r)}{\sqrt{r}} = \int_{0}^{\infty} \frac{2}{f(u)^{2}} \mathcal{F}_{H}(\frac{g(\cdot)}{\sqrt{\cdot}})(u) \rho(u^{2}) \tilde{\phi}_{u}(r) u du\end{equation}
Comparing this to the inversion formula from \cite{bkt}, we get
$$f(u) = \sqrt{2 u \rho(u^{2})}$$
In order to estimate the $F_{3}$ terms, we will also use the following lemma
\begin{lemma}\label{ytouregularity}There exists $C>0$ such that, for all $\overline{y}$ with
$\overline{y}(\xi) \langle \xi \rangle \in L^{2}((0,\infty),\rho(\xi) d\xi)$, if $\overline{v}$ is given by
$$\overline{v}(r) = \frac{1}{\sqrt{r}} \int_{0}^{\infty} \overline{y}(\xi) \phi(r,\xi) \rho(\xi) d\xi, \quad r>0$$
then, 
$\overline{v} \in C^{0}(0,\infty)$, and $\frac{\overline{v}(r)}{r \langle \log(r) \rangle}$ admits a continuous extension to $[0,\infty)$ with $\frac{\overline{v}(r)}{r \langle \log(r) \rangle}(0)= \lim_{r \rightarrow \infty} \overline{v}(r) = 0$. 
$$\overline{v},L(\overline{v}), L^{*}L(\overline{v}) \in L^{2}((0,\infty), rdr)$$
with
\begin{equation}\label{l2conversion} ||\overline{v}(r)||_{L^{2}(r dr)} = ||\mathcal{F}^{-1}(\overline{y})(r)||_{L^{2}(dr)} = ||\overline{y}||_{L^{2}(\rho(\xi) d\xi)}\end{equation}
\begin{equation}\label{lconversion}\begin{split} &||\sqrt{\xi} \overline{y}(\xi)||_{L^{2}(\rho(\xi) d\xi)}^{2} = ||L\overline{v}||_{L^{2}(r dr)}^{2}\end{split}\end{equation}
and
\begin{equation}\label{lstarlconversion}\begin{split} &||\xi \overline{y}(\xi)||_{L^{2}(\rho(\xi) d\xi)}^{2} = ||L^{*}L\overline{v}||_{L^{2}(r dr)}^{2}\end{split}\end{equation}
Moreover, $\overline{v} \in \dot{H}^{1}_{e}$, with
\begin{equation}\label{lfcoercive}||\overline{v}||_{\dot{H}^{1}_{e}} \leq C \left(||\overline{v}||_{L^{2}(r dr)} + ||L(\overline{v})||_{L^{2}(r dr)}\right)\end{equation}
\begin{equation}\label{h1dotecoercive2}|\overline{v}(r)| \leq C ||\overline{v}||_{\dot{H}^{1}_{e}}, \quad r \geq 0\end{equation}
\begin{equation}\label{foverxinfinity} |\frac{\overline{v}(r)}{r \langle \log(r) \rangle}| \leq C\left(||\overline{v}||_{L^{2}(r dr)} + ||L(\overline{v})||_{L^{2}(r dr)} + ||L^{*}L\overline{v}||_{L^{2}(r dr)}\right) , \quad r> 0\end{equation}
\begin{equation}\label{lstarcoercive} \int_{0}^{\infty} \left(\frac{\left(L(\overline{v})(r)\right)^{2}}{r^{2}(1+r^{2})}\right) rdr\leq C ||L^{*}L\overline{v}||_{L^{2}(r dr)}^{2}  \end{equation}
\end{lemma}
\begin{proof} 
If $f \in C^{1}([0,\infty))\cap L^{2}((0,\infty), rdr)$, $Lf \in L^{2}((0,\infty), rdr)$, and $f(0)= \lim_{r \rightarrow \infty} f(r) =0$, then, since
$$|\frac{\cos(Q_{1}(r))-1}{r}| \leq 1, \quad r > 0$$
we have
\begin{equation} \partial_{r} f -\frac{f}{r} = Lf + \left(\frac{\cos(Q_{1}(r))-1}{r}\right) f \in L^{2}((0,\infty), rdr)\end{equation}
Then, if $M>1$,
\begin{equation}\begin{split} \int_{\frac{1}{M}}^{M} (\partial_{r} f-\frac{f}{r})^{2} r dr &= \int_{\frac{1}{M}}^{M} \left((\partial_{r}f)^{2}-\frac{\partial_{r}(f^{2})}{r} + \frac{f^{2}}{r^{2}}\right) r dr \\
&=\int_{0}^{\infty} \mathbbm{1}_{[\frac{1}{M},M]}(r) \left((\partial_{r}f)^{2}+\frac{f^{2}}{r^{2}}\right) r dr - (f(M))^{2}+(f(\frac{1}{M}))^{2} \end{split}\end{equation}
Letting $M \rightarrow \infty$, and using the monotone convergence theorem, we have
\begin{equation} \label{lfcoercive2}||f||_{\dot{H}^{1}_{e}}^{2} = \int_{0}^{\infty} \left((\partial_{r}f)^{2}+\frac{f^{2}}{r^{2}}\right) r dr \leq C\left(||Lf||_{L^{2}(r dr)}^{2}+||f||_{L^{2}(r dr)}^{2}\right)\end{equation}

Next, for $f \in C^{1}([0,\infty))\cap \dot{H}^{1}_{e}$ satisfying $f(0)=0$, we have
\begin{equation} f^{2}(r)=\int_{0}^{r} \left(2 f(s) f'(s)\right) ds =2 \int_{0}^{r} \frac{f(s)}{\sqrt{s}} (f'(s) \sqrt{s}) ds \leq 2 ||\frac{f}{r}||_{L^{2}(r dr)} ||f'||_{L^{2}(r dr)} \end{equation}
So, \begin{equation}\label{h1dotecoercive3} ||f||_{\infty} \leq C ||\frac{f}{r}||^{1/2}_{L^{2}(r dr)}||f'||^{1/2}_{L^{2}(r dr)} \leq C ||f||_{\dot{H}^{1}_{e}}\end{equation}

Next, for any $g \in C^{1}((0,\infty)) \cap C^{0}([0,\infty))$ such that $g(0) = \lim_{r \rightarrow \infty} g(r)= 0$, $L^{*}g \in L^{2}((0,\infty), rdr)$, and for any $M \geq 1$,
\begin{equation}\begin{split} &\int_{\frac{1}{M}}^{M} (L^{*} g)^{2} r dr = \int_{\frac{1}{M}}^{M} (g'(r))^{2} r dr + \int_{\frac{1}{M}}^{M} \left(\cos(Q_{1}(r))+1\right) \frac{d}{dr} \left(g^{2}(r)\right) dr + \int_{\frac{1}{M}}^{M} \left(\frac{\cos(Q_{1}(r))+1}{r}\right)^{2} g(r)^{2} r dr\\
&=(\cos(Q_{1}(M))+1)g(M)^{2}-\left(\cos(Q_{1}(\frac{1}{M}))+1\right)g(\frac{1}{M})^{2}+\int_{0}^{\infty} \mathbbm{1}_{[\frac{1}{M},M]}(r) \left((g'(r))^{2}+\frac{4 g(r)^{2}}{r^{2}(1+r^{2})}\right) r dr\end{split}\end{equation}
By the monotone convergence theorem,
\begin{equation}\label{lstarcoercive2} ||L^{*}g||_{L^{2}(r dr)}^{2} = \int_{0}^{\infty} \left((g'(r))^{2}+\frac{4g(r)^{2}}{r^{2}(1+r^{2})}\right) rdr\end{equation}
 If $\overline{y}$ is as in the lemma statement, for $M \geq 4$, define 
$$\overline{v}_{M}(r) := \frac{1}{\sqrt{r}}\int_{0}^{\infty} \overline{y}(\xi) \phi(r,\xi) \chi_{\leq 1}(\frac{\xi}{M}) \rho(\xi) d\xi, \quad r>0$$
We will now record some estimates on $\partial_{r}^{k} \phi(r,\xi), \quad k \leq 1$, which will allow us to prove a certain regularity of $\overline{v}$ and $\overline{v_{M}}$. From \cite{kst}, we have
\begin{equation} \frac{1}{\sqrt{r}} \phi(r,\xi) = \frac{1}{2} \phi_{0}(r) + \frac{1}{r} \sum_{j=1}^{\infty} (r^{2} \xi)^{j} \phi_{j}(r^{2}), \quad r^{2} \xi \leq 4\end{equation}
and
\begin{equation} \frac{1}{\sqrt{r}}\phi(r,\xi) = \frac{2 \text{Re}\left(a(\xi) \psi^{+}(r,\xi)\right)}{\sqrt{r}}, \quad r^{2} \xi > 4\end{equation}
Therefore,
\begin{equation}\label{efuncest}|\frac{1}{r^{3/2} \langle \log(r) \rangle} \phi(r,\xi)| \leq \begin{cases} \frac{C}{\langle \log(r) \rangle} \left(\frac{\phi_{0}(r)}{r}+\frac{\log(1+r^{2})}{r^{2}}\right), \quad r^{2} \xi \leq 4\\
\frac{C |a(\xi)|}{\xi^{1/4} r^{1/2}\cdot r \langle \log(r) \rangle}, \quad r^{2} \xi > 4\end{cases} \end{equation}
We use
$$\frac{\langle \log(\xi)\rangle^{2}}{\langle \log(r)\rangle^{2}} \leq C, \quad r^{2} \xi \leq 4, \quad \xi \geq 1$$
and the fact that 
$$r \mapsto \frac{1}{\sqrt{r} \langle \log(r) \rangle} \text{ is decreasing on  } (0,\infty)$$
which gives
\begin{equation} \frac{1}{\sqrt{r} \langle \log(r) \rangle} \leq \frac{C \xi^{1/4}}{\langle \log(\xi)\rangle}, \quad r^{2} \xi > 4\end{equation} 
in \eqref{efuncest} to get, for all $r \geq 0$

\begin{equation}|\frac{\phi(r,\xi)}{r^{3/2}\langle \log(r)\rangle}| \leq  C \frac{|a(\xi)|\sqrt{\xi}}{\langle \log(\xi)\rangle}+ C \begin{cases}\frac{1}{\langle \log(\xi)\rangle}, \quad \xi \geq 1\\
1 , \quad \xi \leq 1\end{cases}\end{equation}
Then, if 
$$g(\xi) = \left(\frac{|a(\xi)|^{2}\xi}{\langle \log(\xi)\rangle^{2}}+\begin{cases}\frac{1}{\langle \log(\xi)\rangle^{2}}, \quad \xi \geq 1\\
1 , \quad \xi \leq 1\end{cases}\right)\frac{\rho(\xi)}{\langle \xi\rangle^{2}}$$
we have, for all $r \geq 0$
\begin{equation} |\overline{y}(\xi) \frac{\phi(r,\xi)}{r^{3/2} \langle \log(r)\rangle}\rho(\xi)| \leq |\overline{y}(\xi)| \langle \xi \rangle \sqrt{\rho(\xi)} \cdot \left(|\frac{\phi(r,\xi)}{r^{3/2}\langle \log(r)\rangle}| \frac{\sqrt{\rho(\xi)}}{\langle \xi \rangle}\right) \leq C|\overline{y}(\xi)| \langle \xi \rangle \sqrt{\rho(\xi)} \cdot \sqrt{g(\xi)}  \end{equation}
But, $$|\overline{y}(\xi)| \langle \xi \rangle \sqrt{\rho(\xi)} \cdot \sqrt{g(\xi)} \in L^{1}((0,\infty), d\xi)$$
by Cauchy-Schwartz, due to the assumptions on $\overline{y}$ and the fact that
\begin{equation} \int_{0}^{\infty} g(\xi) d\xi \leq C\end{equation}
Therefore, by the dominated convergence theorem, $\overline{v}$, defined in the lemma statement satisfies $\frac{\overline{v}(r)}{r \langle \log(r) \rangle}$ is continuous on $(0,\infty)$, and we have
\begin{equation} |\frac{\overline{v}(r)}{r \langle \log(r) \rangle}| \leq  \int_{0}^{\infty} |\overline{y}(\xi) \frac{\phi(r,\xi)}{r^{3/2} \langle \log(r)\rangle}\rho(\xi)|d\xi \leq C ||\overline{y}(\xi) \langle \xi \rangle||_{L^{2}(\rho(\xi) d\xi)}, \quad r>0\end{equation}
and
\begin{equation} \lim_{r \rightarrow 0} \frac{\overline{v}(r)}{r \langle \log(r) \rangle} = \int_{0}^{\infty} \overline{y}(\xi) \lim_{r\rightarrow 0}\left(\frac{\phi(r,\xi)}{r^{3/2} \langle \log(r) \rangle}\right) \rho(\xi) d\xi=0\end{equation}
Similarly,
\begin{equation} |\frac{\phi(r,\xi)}{\sqrt{r}} \frac{\sqrt{\rho(\xi)}}{\langle \xi \rangle}| \leq C \sqrt{g_{2}(\xi)}\end{equation}
with
$$g_{2}(\xi) = \left( |a(\xi)|^{2}+\begin{cases} 1, \quad \xi \leq 1\\
\frac{1}{\xi}, \quad \xi > 1\end{cases}\right) \frac{\rho(\xi)}{\langle \xi\rangle^{2}} \in L^{1}((0,\infty),d\xi)$$
So, again by the dominated convergence theorem,
\begin{equation} \lim_{r \rightarrow \infty} \overline{v}(r) = \int_{0}^{\infty} \overline{y}(\xi) \lim_{r\rightarrow \infty}\left(\frac{\phi(r,\xi)}{\sqrt{r}}\right) \rho(\xi) d\xi=0\end{equation}
The same argument shows that
\begin{equation}\lim_{r \rightarrow \infty}\overline{v_{M}}(r)=0\end{equation}
From \eqref{efuncest}, we also have
\begin{equation} |\frac{\phi(r,\xi)}{r^{3/2}}| \leq \begin{cases} C, \quad r^{2} \xi \leq 4\\
C |a(\xi)| \sqrt{\xi}, \quad r^{2} \xi >4\end{cases}\end{equation}
whence,
\begin{equation}|\frac{\phi(r,\xi)}{r^{3/2}}| \leq C+ C|a(\xi)| \sqrt{\xi}, \quad r \geq 0\end{equation}
Then, the dominated convergence theorem shows that $\overline{v_{M}}$ satisfies that $\frac{\overline{v_{M}}(r)}{r}$ admits a continuous extension to a function defined on $[0,\infty)$. Similarly,
\begin{equation} |\partial_{r}\left(\frac{\phi(r,\xi)}{\sqrt{r}}\right)| \leq \begin{cases} C\left(|\phi_{0}'(r)| + \xi \log(1+r^{2})\right), \quad r^{2} \xi \leq 4\\
\frac{C |a(\xi)| \xi^{1/4}}{r^{1/2}}, \quad r^{2} \xi > 4\end{cases} \end{equation}
which shows
\begin{equation} |\partial_{r}\left(\frac{\phi(r,\xi)}{\sqrt{r}}\right)| \leq C + C |a(\xi)| \sqrt{\xi}, \quad r \geq 0 \end{equation}
Therefore, $\overline{v_{M}} \in C^{1}([0,\infty))$. Moreover,
\begin{equation} L(\overline{v_{M}})(0) = \int_{0}^{\infty} \overline{y}(\xi) \lim_{r \rightarrow 0}L\left(\frac{\phi(r,\xi)}{\sqrt{r}}\right) \chi_{\leq 1}\left(\frac{\xi}{M}\right) \rho(\xi) d\xi=0\end{equation}
where we used $L(\phi_{0})=0$. Again, by the dominated convergence theorem,
\begin{equation} \lim_{r \rightarrow \infty} \partial_{r}\overline{v_{M}}(r) = \int_{0}^{\infty} \overline{y}(\xi) \lim_{r\rightarrow \infty} \partial_{r}\left(\frac{\phi(r,\xi)}{\sqrt{r}}\right) \chi_{\leq 1}\left(\frac{\xi}{M}\right)\rho(\xi) d\xi=0\end{equation}
Finally, we have
\begin{equation} |\partial_{r}^{2}\left(\frac{\phi(r,\xi)}{\sqrt{r}}\right)| \leq C \left(1+\sqrt{\xi} + \xi |a(\xi)|\right), \quad r \geq 0 \end{equation}
and the same dominated convergence theorem based procedure used above shows that $\overline{v_{M}} \in C^{2}((0,\infty))$.

 Next, we have
\begin{equation} ||\overline{v}(r)||_{L^{2}(r dr)} = ||\mathcal{F}^{-1}(\overline{y})(r)||_{L^{2}(dr)} = ||\overline{y}||_{L^{2}(\rho(\xi) d\xi)}\end{equation}
\begin{equation}\begin{split} &||\sqrt{\xi} \overline{y}(\xi)||_{L^{2}(\rho(\xi) d\xi)}^{2} = \int_{0}^{\infty} \xi |\mathcal{F}(\sqrt{\cdot} \overline{v}(\cdot))(\xi)|^{2} \rho(\xi) d\xi = \int_{0}^{\infty} \xi \frac{|\mathcal{F}_{H}(\overline{v})(\sqrt{\xi})|^{2}}{2 \sqrt{\xi}}d\xi \\
&= \int_{0}^{\infty} \eta^{2} |\mathcal{F}_{H}(\overline{v})(\eta)|^{2} d\eta = ||\eta \mathcal{F}_{H}(\overline{v})(\eta)||_{L^{2}(d\eta)}^{2}\\
&=||L\overline{v}||_{L^{2}(r dr)}^{2}\end{split}\end{equation}
where we used the $L^{2}$ isometry property of $\mathcal{F}_{H}$ from \cite{bkt}. 
Finally,
\begin{equation}\begin{split} &||\xi \overline{y}(\xi)||_{L^{2}(\rho(\xi) d\xi)}^{2} = \int_{0}^{\infty} \xi^{2} \frac{|\mathcal{F}_{H}(\overline{v})(\sqrt{\xi})|^{2}}{2 \sqrt{\xi}} d\xi = \int_{0}^{\infty} \eta^{4} |\mathcal{F}_{H}(\overline{v})(\eta)|^{2} d\eta\\
&=\int_{0}^{\infty} |\mathcal{F}_{H}(L^{*}L \overline{v})(\eta)|^{2} d\eta = ||L^{*}L\overline{v}||_{L^{2}(r dr)}^{2}\end{split}\end{equation}
This shows that $\overline{v},\overline{v_{M}},L(\overline{v}),L(\overline{v_{M}}),L^{*}L(\overline{v}),L^{*}L(\overline{v_{M}}) \in L^{2}((0,\infty),r dr)$. Combining these facts with our estimates on $\frac{\phi(r,\xi)}{\sqrt{r}}\frac{\sqrt{\rho(\xi)}}{\langle \xi \rangle}$ and
\begin{equation} |\overline{v}(r) - \overline{v_{M}}(r)| \leq \int_{0}^{\infty} |\overline{y}(t,\xi) \frac{\phi(r,\xi)}{\sqrt{r}}| \cdot |\chi_{\leq 1}(\frac{\xi}{M})-1| \rho(\xi) d\xi\end{equation}
the Dominated convergence theorem gives
$$\overline{v_{M}} \rightarrow \overline{v}, \text{ pointwise, and in }L^{2}((0,\infty),rdr), \quad M \rightarrow \infty$$
$$L(\overline{v_{M}}) \rightarrow L(\overline{v}), \text{in }L^{2}((0,\infty),rdr), \quad M \rightarrow \infty$$
$$L^{*}L(\overline{v_{M}}) \rightarrow L^{*}L(\overline{v}), \text{in }L^{2}((0,\infty),rdr), \quad M \rightarrow \infty$$
 We conclude the proof of the lemma by noting that \eqref{lfcoercive2} and \eqref{h1dotecoercive3} hold for $\overline{v_{M}}$, and \eqref{lstarcoercive2} holds for $g=L(\overline{v_{M}})$. So, by approximation, we have \eqref{lfcoercive}, \eqref{h1dotecoercive2}, and \eqref{lstarcoercive}. 
\end{proof}

 Now, we can estimate the $F_{3}(u(y))$ terms, for $y$ such that $ y(t,\omega) \sqrt{\rho(\omega \lambda(t)^{2})} \langle \omega \lambda(t)^{2}\rangle \in C^{0}_{t}([T_{0},\infty),L^{2}(d\omega))$. This is sufficient for our purposes, since all $y$ in the space $Z$ ( which is the space in which we will construct a solution to \eqref{ylinearproblem}, and is defined later on) satisfy this condition.
Recall that 
$$F_{3}(u)(t,r) = N(u)(t,r)+L_{1}(u)(t,r)$$ where $u$ and $y$ are related by
\begin{equation}\label{ytou}y(t,\xi) = \mathcal{F}(\sqrt{\cdot} u(t,\cdot \lambda(t)))(\xi \lambda(t)^{2})\end{equation}
and
\begin{equation}\label{l1errorterm}L_{1}(u)(t,r) = \frac{\sin(2 u(t,r))}{2r^{2}}\left(\cos(2Q_{\frac{1}{\lambda(t)}}(r))\left(\cos(2v_{corr})-1\right) - \sin(2Q_{\frac{1}{\lambda(t)}}) \sin(2v_{corr})\right)\end{equation}
\begin{equation}\label{nerrorterm}N(u)(t,r)=\left(\frac{\sin(2u(t,r))-2u(t,r)}{2r^{2}}\right)\cos(2Q_{\frac{1}{\lambda(t)}}(r)) + \left(\frac{\cos(2u(t,r))-1}{2r^{2}}\right)\sin(2Q_{\frac{1}{\lambda(t)}}(r)+2v_{corr})\end{equation}\\
\\
We start with the $L^{2}$ estimate on the $L_{1}(u)$ term. Using the same procedure as in \eqref{l2conversion}, we get
\begin{equation} \begin{split}||\mathcal{F}(\sqrt{\cdot} L_{1}(u)(t,\cdot \lambda(t)))(\omega \lambda(t)^{2})||_{L^{2}(\rho(\omega \lambda(t)^{2})d\omega)}^{2} &=\frac{1}{\lambda(t)^{2}} \int_{0}^{\infty} R (L_{1}(u)(t,R\lambda(t)))^{2} dR\end{split}\end{equation}
Then,
\begin{equation} \begin{split} &\frac{1}{\lambda(t)^{2}} \int_{0}^{\infty} R (L_{1}(u)(t,R\lambda(t)))^{2} dR \\
&\leq \frac{C}{\lambda(t)^{2}} \int_{0}^{\infty} R \frac{(u(t,R\lambda(t)))^{2}}{R^{4}\lambda(t)^{4}} \left(v_{corr}(t,R\lambda(t))\right)^{4} dR+\frac{C}{\lambda(t)^{2}} \int_{0}^{\infty} R \frac{(u(t,R\lambda(t)))^{2}}{R^{4}\lambda(t)^{4}} \frac{R^{2}}{(1+R^{2})^{2}} \left(v_{corr}(t,R\lambda(t))\right)^{2} dR\end{split}\end{equation}
Using the functions $v$ and $w$ (introduced when deriving the equation for $y$) defined by 
$$u(t,r) = v(t,\frac{r}{\lambda(t)})$$
$$v(t,R) = \frac{w(t,R)}{\sqrt{R}}$$
we have
\begin{equation}\begin{split} ||u(t,\cdot \lambda(t))||^{2}_{L^{2}(R dR)} = ||v(t,\cdot)||^{2}_{L^{2}(R dR)} = ||w(t)||^{2}_{L^{2}(dR)} = ||\mathcal{F}(w)(t)||_{L^{2}(\rho(\xi) d\xi)}^{2} = \lambda(t)^{2} ||y(t)||_{L^{2}(\rho(\omega \lambda(t)^{2})d\omega)}^{2}\end{split}\end{equation}\\
\\
So, we end up with \begin{equation}\begin{split} &||\mathcal{F}(\sqrt{\cdot} L_{1}(u)(t,\cdot \lambda(t)))(\omega \lambda(t)^{2})||_{L^{2}(\rho(\omega \lambda(t)^{2})d\omega)} \\
&\leq C ||y(t)||_{L^{2}(\rho(\omega \lambda(t)^{2})d\omega)}\left(||\frac{v_{corr}(t,R \lambda(t))}{R\lambda(t)}||_{L^{\infty}_{R}}^{2}+ ||\frac{v_{corr}(t,R\lambda(t))}{R\lambda(t)^{2}(1+R^{2})}||_{L^{\infty}_{R}}\right)\end{split}\end{equation} \\
\\
Next, we apply \eqref{lconversion} to our current setting to get
\begin{equation} \begin{split} ||\sqrt{\omega} \lambda(t) \mathcal{F}(\sqrt{\cdot} L_{1}(u)(t,\cdot \lambda(t)))(\omega \lambda(t)^{2})||_{L^{2}(\rho(\omega \lambda(t)^{2})d\omega)}^{2}&=\frac{1}{\lambda(t)^{2}}||L(L_{1}(u)(t,\cdot \lambda(t)))||_{L^{2}(R dR)}^{2}\end{split}\end{equation} 
Using \eqref{l1errorterm}, we get 
\begin{equation}\begin{split} |\partial_{R}(L_{1}(u)(t,R\lambda(t)))| &\leq C \left(|Lv(t,R)|+\frac{|v(t,R)|}{R}\right)\left(\frac{(v_{corr}(t,R\lambda(t)))^{2}}{R^{2}\lambda(t)^{2}}+\frac{|v_{corr}(t,R\lambda(t))|}{R\lambda(t)^{2}(1+R^{2})}\right)\\
&+C \frac{|v(t,R)|}{R^{2}\lambda(t)^{2}} \left(\frac{R |\partial_{R}(v_{corr}(t,R\lambda(t)))|}{1+R^{2}}+|\partial_{R}(v_{corr}(t,R\lambda(t))) v_{corr}(t,R\lambda(t))|\right)\end{split}\end{equation}\\
\\
On the other hand, \begin{equation} |\frac{L_{1}(u)(t,R\lambda(t))}{R}| \leq \frac{C |v(t,R)|}{R^{3}\lambda(t)^{2}} \left(v_{corr}(t,R\lambda(t))^{2}+\frac{|v_{corr}(t,R\lambda(t))|R}{(1+R^{2})}\right)\end{equation}\\
\\
So, we get
\begin{equation} \begin{split} &||L(L_{1}(u)(t,\cdot\lambda(t)))||_{L^{2}(R dR)}\\
&\leq C \left(||Lv(t)||_{L^{2}(R dR)}+||v(t)||_{\dot{H}^{1}_{e}}\right)\left(||\frac{v_{corr}(t,R\lambda(t)}{R\lambda(t)}||_{L^{\infty}_{R}}^{2} + ||\frac{v_{corr}(t,R\lambda(t))}{R\lambda(t)^{2}(1+R^{2})}||_{L^{\infty}_{R}}\right)\\
&+C ||v(t)||_{\dot{H}^{1}_{e}}\cdot||\frac{\partial_{R}(v_{corr}(t,R\lambda(t)))}{(1+R^{2})\lambda(t)^{2}}||_{L^{\infty}_{R}} + ||v(t)||_{\dot{H}^{1}_{e}}\cdot ||\frac{v_{corr}(t,R\lambda(t))\partial_{R}(v_{corr}(t,R\lambda(t)))}{R\lambda(t)^{2}}||_{L^{\infty}_{R}((0,1))}\\
&+C ||v(t)||_{L^{2}((1,\infty),R dR)} \cdot ||\frac{\partial_{R}(v_{corr}(t,R\lambda(t)))\cdot v_{corr}(t,R\lambda(t))}{R^{2}\lambda(t)^{2}}||_{L^{\infty}_{R}((1,\infty))}\end{split}\end{equation} \\
\\
Now, we  use \eqref{l2conversion} and \eqref{lconversion} to translate the right-hand side in terms of $y$: 
\begin{equation} ||Lv(t)||_{L^{2}(R dR)} = \lambda(t) ||\sqrt{\omega} \lambda(t) y(t)||_{L^{2}(\rho(\omega \lambda(t)^{2})d\omega)}\end{equation}
\begin{equation} ||v(t)||_{L^{2}(R dR)} = \lambda(t) ||y(t)||_{L^{2}(\rho(\omega \lambda(t)^{2})d\omega)}\end{equation}
Then, we get
\begin{equation}\begin{split} &||\sqrt{\omega} \lambda(t) \mathcal{F}(\sqrt{\cdot}L_{1}(u)(t,\cdot \lambda(t)))(\omega \lambda(t)^{2})||_{L^{2}(\rho(\omega \lambda(t)^{2})d\omega)}\\
&\leq C \left(||\sqrt{\omega} \lambda(t) y(t)||_{L^{2}(\rho(\omega \lambda(t)^{2})d\omega)} + ||y(t)||_{L^{2}(\rho(\omega \lambda(t)^{2}) d\omega)}\right)\\
&\cdot \left(||\frac{v_{corr}(t,R\lambda(t))}{R\lambda(t)}||_{L^{\infty}_{R}}^{2}+||\frac{v_{corr}(t,R\lambda(t))}{R\lambda(t)^{2}(1+R^{2})}||_{L^{\infty}_{R}}+||\frac{\partial_{R}(v_{corr}(t,R\lambda(t)))}{(1+R^{2})\lambda(t)^{2}}||_{L^{\infty}_{R}}\right.\\
&\left. + ||\frac{v_{corr}(t,R\lambda(t))\partial_{R}(v_{corr}(t,R\lambda(t)))}{R\lambda(t)^{2}}||_{L^{\infty}_{R}((0,1))}+||\frac{v_{corr}(t,R\lambda(t))\partial_{R}(v_{corr}(t,R\lambda(t)))}{R^{2}\lambda(t)^{2}}||_{L^{\infty}_{R}((1,\infty))}\right)\end{split}\end{equation}\\
\\
Now, we treat the $N$ terms. Recall that
\begin{equation}\begin{split} N(u)(t,R\lambda(t)) &= \left(\frac{\sin(2v(t,R))-2v(t,R)}{2R^{2}\lambda(t)^{2}}\right)\cos(2Q_{1}(R))\\
&+\left(\frac{\cos(2v(t,R))-1}{2R^{2}\lambda(t)^{2}}\right)\sin(2Q_{1}(R)+2v_{corr}(t,R\lambda(t)))\end{split}\end{equation}\\
\\
We then use \eqref{foverxinfinity} and the previous estimates to conclude that, if $R<1$, then,
\begin{equation} \begin{split} |N(u)(t,R\lambda(t))| &\leq C\left(||v(t)||^{2}_{\dot{H}^{1}_{e}} + ||L^{*}Lv(t)||_{L^{2}(R dR)}^{2}\right)\left(\log^{2}(\frac{1}{R})+1\right) \cdot \frac{|v(t,R)|}{\lambda(t)^{2}}\\
&+C ||v(t)||_{\dot{H}^{1}_{e}}\frac{|v(t,R)|}{R \lambda(t)^{2}} \left(1+\frac{|v_{corr}(t,R\lambda(t))|}{R}\right)\end{split}\end{equation}\\
\\
On the other hand, if $R>1$, then, 
\begin{equation} |N(u)(t,R\lambda(t))| \leq C \frac{|v(t,R)|^{3}}{\lambda(t)^{2}} + C \frac{(v(t,R))^{2}}{\lambda(t)^{2}} \left(1+\frac{|v_{corr}(t,R\lambda(t))|}{R}\right)\end{equation}\\
\\
Then, we get 
\begin{equation}\label{nul2prelimest}\begin{split} \frac{1}{\lambda(t)^{2}} ||N(u)(t,\cdot \lambda(t))||_{L^{2}(R dR)}^{2}&\leq \frac{C}{\lambda(t)^{6}} \left(||v(t)||_{L^{2}(R dR)}^{6}+||Lv(t)||_{L^{2}(R dR)}^{6}+||L^{*}Lv(t)||_{L^{2}(R dR)}^{6}\right)\\
&+\frac{C}{\lambda(t)^{6}} \left(||v(t)||^{4}_{L^{2}(R dR)}+||Lv(t)||^{4}_{L^{2}(R dR)}\right)\left(1+||\frac{v_{corr}(t,R\lambda(t))}{R}||_{L^{\infty}_{R}}^{2}\right)\end{split}\end{equation}\\
\\
Finally, we apply \eqref{lstarlconversion} in our current setting:
\begin{equation} \begin{split} \int_{0}^{\infty} (L^{*}Lv(t,R))^{2} R dR &=\lambda(t)^{2}||\omega \lambda(t)^{2} y(t)||_{L^{2}(\rho(\omega \lambda(t)^{2})d\omega)}^{2}\end{split}\end{equation}\\
\\
Using the same procedure as used for $L_{1}$ above, we then translate the rest of the right hand side, and the left hand side of \eqref{nul2prelimest} in terms of $y$, and $\mathcal{F}(\sqrt{\cdot}N(u)(t,\cdot \lambda(t)))(\omega \lambda(t)^{2})$, respectively, to get
\begin{equation}\begin{split} &||\mathcal{F}(\sqrt{\cdot}N(u)(t,\cdot\lambda(t)))(\omega \lambda(t)^{2})||_{L^{2}(\rho(\omega \lambda(t)^{2})d\omega)}\\
&\leq C \left(||y(t)||^{3}_{L^{2}(\rho(\omega \lambda(t)^{2})d\omega)}+||\sqrt{\omega}\lambda(t)y(t)||^{3}_{L^{2}(\rho(\omega \lambda(t)^{2})d\omega)}+||\omega \lambda(t)^{2} y(t)||^{3}_{L^{2}(\rho(\omega \lambda(t)^{2})d\omega)}\right)\\
&+\frac{C}{\lambda(t)} \left(||y(t)||^{2}_{L^{2}(\rho(\omega \lambda(t)^{2})d\omega)} + ||\sqrt{\omega}\lambda(t)y(t)||^{2}_{L^{2}(\rho(\omega \lambda(t)^{2})d\omega)}\right)\left(1+||\frac{v_{corr}(t,R\lambda(t))}{R}||_{L^{\infty}}\right)\end{split}\end{equation}\\
\\
We will now estimate $||\sqrt{\omega}\lambda(t) \mathcal{F}(\sqrt{\cdot}N(u)(t,\cdot\lambda(t)))(\omega \lambda(t)^{2})||_{L^{2}(\rho(\omega \lambda(t)^{2})d\omega)}$, starting with
\begin{equation} \begin{split}|\partial_{R}(N(u)(t,R\lambda(t)))| &\leq \frac{C |v(t,R)|^{3}}{R^{3}\lambda(t)^{2}} + C \frac{|\partial_{R}v(t,R)| (v(t,R))^{2}}{R^{2}\lambda(t)^{2}}\\
&+C \frac{|Q_{1}'(R)+\partial_{R}(v_{corr}(t,R\lambda(t)))|}{R^{2}\lambda(t)^{2}} (v(t,R))^{2}\\
&+C \frac{|Q_{1}(R)|+|v_{corr}(t,R\lambda(t))|}{R^{2}\lambda(t)^{2}} |v(t,R)|\left(\frac{|v(t,R)|}{R}+|\partial_{R}v(t,R)|\right)\end{split}\end{equation}\\
\\
Then, we use our previous estimates to get, for $R \leq 1$:
\begin{equation}\begin{split} &|\partial_{R}(N(u)(t,R\lambda(t)))| \\
&\leq \frac{C}{\lambda(t)^{2}} \left(||v(t)||_{\dot{H}^{1}_{e}}^{3}+||L^{*}Lv(t)||_{L^{2}(R dR)}^{3}\right)\left(1+\log^{2}(\frac{1}{R})\right)^{3/2}\\
&+C \frac{\langle \log(R) \rangle^{2}}{\lambda(t)^{2}} \left(||v(t)||_{L^{2}(R dR)}^{2} + ||Lv(t)||_{L^{2}(R dR)}^{2}+||L^{*}Lv(t)||_{L^{2}(R dR)}^{2}\right)\left(|Lv(t,R)|+\frac{|v(t,R)|}{R}\right)\\
&+C \frac{\left(1+||\partial_{R}(v_{corr}(t,R\lambda(t)))||_{L^{\infty}}\right)}{\lambda(t)^{2}}\left(||v(t)||^{2}_{L^{2}(R dR)} +||Lv(t)||^{2}_{L^{2}(R dR)} +||L^{*}Lv(t)||_{L^{2}(R dR)}^{2}\right)\left(\log^{2}(\frac{1}{R})+1\right)\\
&+C \frac{\left(||v(t)||_{L^{2}(R dR)} +||Lv(t)||_{L^{2}(R dR)} +||L^{*}Lv(t)||_{L^{2}(R dR)}\right) \langle \log(R)\rangle |\partial_{R}v|}{\lambda(t)^{2}}\left(1+||\frac{v_{corr}(t,R\lambda(t))}{R}||_{L^{\infty}}\right)\\
&+\frac{C}{\lambda(t)^{2}}\left(1+||\frac{v_{corr}(t,R\lambda(t))}{R}||_{L^{\infty}}\right)\left(||v(t)||_{\dot{H}^{1}_{e}}^{2}+||L^{*}Lv(t)||_{L^{2}(R dR)}^{2}\right)\left(\log^{2}(\frac{1}{R})+1\right)\end{split}\end{equation}\\
\\
For $R \geq 1$, we have
\begin{equation} \begin{split} |\partial_{R}(N(u)(t,R\lambda(t)))| &\leq \frac{C}{\lambda(t)^{2}} ||v(t)||_{\dot{H}^{1}_{e}}^{2}|v(t,R)| + \frac{C ||v(t)||_{\dot{H}^{1}_{e}}^{2}}{\lambda(t)^{2}}\left(|Lv(t,R)|+\frac{|v(t,R)|}{R}\right)\\
&+ C \frac{\left(1+||\partial_{R}(v_{corr}(t,R\lambda(t)))||_{L^{\infty}}+||\frac{v_{corr}(t,R\lambda(t))}{R}||_{L^{\infty}}\right)}{\lambda(t)^{2}}||v(t)||_{\dot{H}^{1}_{e}} |v(t,R)|\\
&+\frac{C}{\lambda(t)^{2}}\left(1+||\frac{v_{corr}(t,R\lambda(t))}{R}||_{L^{\infty}}\right)\left(|Lv(t,R)|+\frac{|v(t,R)|}{R}\right)||v(t)||_{\dot{H}^{1}_{e}}\end{split}\end{equation}\\
\\
Then, we get
\begin{equation}\begin{split} &||\partial_{R}(N(u)(t,R\lambda(t)))||_{L^{2}(R dR)} \\
&\leq \frac{C}{\lambda(t)^{2}} \left(||v(t)||_{L^{2}(R dR)}^{3}+||Lv(t)||_{L^{2}(R dR)}^{3}+||L^{*}Lv(t)||_{L^{2}(R dR)}^{3}\right)\\
&+\frac{C}{\lambda(t)^{2}}\left(1+||\partial_{R}(v_{corr}(t,R\lambda(t)))||_{L^{\infty}}+||\frac{v_{corr}(t,R\lambda(t))}{R}||_{L^{\infty}}\right)\\
&\cdot\left(||v(t)||^{2}_{L^{2}(R dR)}+||Lv(t)||^{2}_{L^{2}(R dR)}+||L^{*}Lv(t)||_{L^{2}(R dR)}^{2}\right)\end{split}\end{equation}\\
\\
Finally, we consider, for $R \leq 1$:
\begin{equation}\begin{split} |\frac{N(u)(t,R\lambda(t))}{R}| &\leq \frac{C}{\lambda(t)^{2}} \left(||v(t)||_{\dot{H}^{1}_{e}}^{3}+||L^{*}Lv(t)||_{L^{2}(R dR)}^{3}\right)\left(1+\log^{2}(\frac{1}{R})\right)^{3/2}\\
&+\frac{C}{\lambda(t)^{2}} \left(||v(t)||_{\dot{H}^{1}_{e}}^{2}+||L^{*}Lv(t)||_{L^{2}(R dR)}^{2}\right)\left(1+\log^{2}(\frac{1}{R})\right)\left(1+||\frac{v_{corr}(t,R\lambda(t))}{R}||_{L^{\infty}}\right)\end{split}\end{equation}\\
\\
and for $R >1$, we have
\begin{equation} |\frac{N(u)(t,R\lambda(t))}{R}| \leq C\frac{||v(t)||_{\dot{H}^{1}_{e}}^{2}}{\lambda(t)^{2}}|v(t,R)| + C\frac{||v(t)||_{\dot{H}^{1}_{e}}|v(t,R)|}{\lambda(t)^{2}}\left(1+||\frac{v_{corr}(t,R\lambda(t))}{R}||_{L^{\infty}}\right)\end{equation}\\
\\
Then, we get
\begin{equation}\begin{split} ||L(N(u)(t,\cdot\lambda(t)))||_{L^{2}(R dR)}&\leq \frac{C}{\lambda(t)^{2}}\left(||v(t)||_{L^{2}(R dR)}^{3}+||Lv(t)||^{3}_{L^{2}(R dR)}+||L^{*}Lv(t)||^{3}_{L^{2}(R dR)}\right)\\
&+\frac{C}{\lambda(t)^{2}} \left(1+||\partial_{R}(v_{corr}(t,R\lambda(t)))||_{L^{\infty}}+||\frac{v_{corr}(t,R\lambda(t))}{R}||_{L^{\infty}}\right)\\
&\cdot \left(||v(t)||^{2}_{L^{2}(R dR)}+||Lv(t)||^{2}_{L^{2}(R dR)}+||L^{*}Lv(t)||^{2}_{L^{2}(R dR)}\right)\end{split}\end{equation}\\
\\
We use the same procedure as in the previous estimates, to translate the left and right hand sides of the above estimate, and combine everything, to get \eqref{f3est} and \eqref{lf3est}.\end{proof}

Before we proceed, we will need one more estimate. First, because $\lambda$ is decreasing, if $x \geq t$, then, $\lambda(x) \leq \lambda(t)$. Next, we use Proposition 4.7 b of \cite{kst} to conclude that there exists a constant $C_{1}>0$ such that
\begin{equation} \begin{split} &\frac{1}{C_{1} \xi \log^{2}(\xi)} \leq \rho(\xi) \leq \frac{C_{1}}{\xi \log^{2}(\xi)}, \quad 0 < \xi < \frac{1}{2e^{2}}\\
&\frac{\xi}{C_{1}} \leq \rho(\xi) \leq C_{1}\xi, \quad \xi > \frac{1}{2e^{2}} \end{split}\end{equation} 
Then, if $x \geq t$, if $\omega\lambda(t)^{2} \leq \frac{1}{2 e^{2}}$, then, $\omega \lambda(x)^{2} \leq \frac{1}{2 e^{2}}$, and 
\begin{equation}\begin{split} \frac{\rho(\omega \lambda(t)^{2})}{\rho(\omega \lambda(x)^{2})} &\leq C \frac{\omega \lambda(x)^{2} \log^{2}(\omega \lambda(x)^{2})}{\omega \lambda(t)^{2} \log^{2}(\omega \lambda(t)^{2})} \\
&\leq C\end{split}\end{equation}
where we used the fact that $x\mapsto x \log^{2}(x)$ is increasing on $(0,\frac{1}{2e^{2}})$.\\
\\
If $\omega \lambda(t)^{2} \geq \frac{1}{2e^{2}}$, but $\omega \lambda(x)^{2} \leq \frac{1}{2e^{2}}$, then,
\begin{equation}\begin{split}  \frac{\rho(\omega \lambda(t)^{2})}{\rho(\omega \lambda(x)^{2})} &\leq C \omega \lambda(t)^{2}\omega \lambda(x)^{2}\log^{2}(\omega \lambda(x)^{2})\\
&\leq C \frac{\lambda(t)^{2}}{\lambda(x)^{2}} \left(\omega \lambda(x)^{2}\right)^{2} \log^{2}(\omega \lambda(x)^{2})\\
&\leq \frac{C \lambda(t)^{2}}{\lambda(x)^{2}}\end{split}\end{equation}\\
\\
Finally, if $\omega \lambda(t)^{2} \geq \frac{1}{2e^{2}}$, and $\omega \lambda(x)^{2} \geq \frac{1}{2e^{2}}$, then, 
\begin{equation}\frac{\rho(\omega\lambda(t)^{2})}{\rho(\omega \lambda(x)^{2})} \leq C \frac{\lambda(t)^{2}}{\lambda(x)^{2}}\end{equation}\\
\\
In all cases, we have: there exists $C_{\rho}>0$ such that, if $x \geq t$, then,
\begin{equation} \label{rhoscaling} \frac{\rho(\omega \lambda(t)^{2})}{\rho(\omega \lambda(x)^{2})} \leq C_{\rho} \frac{\lambda(t)^{2}}{\lambda(x)^{2}}\end{equation}

Before setting up the final iteration, we will need to estimate various oscillatory integrals involving $F_{4}$:

\subsection{Estimates on $F_{4}$-related oscillatory integrals}
\begin{lemma} We have the following estimates
\begin{equation} ||\int_{t}^{\infty} \frac{\sin((t-x)\sqrt{\omega})}{\sqrt{\omega}} \mathcal{F}(\sqrt{\cdot}F_{4}(x,\cdot\lambda(x)))(\omega \lambda(x)^{2}) dx||_{L^{2}(\rho(\omega \lambda(t)^{2})d\omega)} \leq \frac{C (\log(\log(t))^{2}}{t^{2}\log^{b+1-2\alpha b}(t)}\end{equation}
\begin{equation}||\int_{t}^{\infty} \cos((t-x)\sqrt{\omega}) \mathcal{F}(\sqrt{\cdot} F_{4}(x,\cdot \lambda(x)))(\omega \lambda(x)^{2}) dx||_{L^{2}(\rho(\omega \lambda(t)^{2}) d\omega)}\leq \frac{C (\log(\log(t)))^{2}}{t^{3} \log^{b+1-2\alpha b}(t)}\end{equation}
\begin{equation} ||\sqrt{\omega}\lambda(t) \int_{t}^{\infty} \frac{\sin((t-x)\sqrt{\omega})}{\sqrt{\omega}} \mathcal{F}(\sqrt{\cdot}F_{4}(x,\cdot\lambda(x)))(\omega \lambda(x)^{2}) dx||_{L^{2}(\rho(\omega \lambda(t)^{2}) d\omega)} \leq \frac{C (\log(\log(t)))^{2}}{t^{2} \log^{1+b-2\alpha b}(t)}\end{equation}
\begin{equation} ||\sqrt{\omega} \lambda(t) \partial_{t}\left(\int_{t}^{\infty} \frac{\sin((t-x)\sqrt{\omega})}{\sqrt{\omega}} \mathcal{F}(\sqrt{\cdot}F_{4}(x,\cdot\lambda(x)))(\omega \lambda(x)^{2}) dx\right)||_{L^{2}(\rho(\omega\lambda(t)^{2})d\omega)} \leq \frac{C (\log(\log(t)))^{2}}{t^{3} \log^{1+b-2\alpha b}(t)} \end{equation}
\begin{equation}||\omega \lambda(t)^{2} \int_{t}^{\infty} \frac{\sin((t-x)\sqrt{\omega})}{\sqrt{\omega}} \mathcal{F}(\sqrt{\cdot} F_{4}(x,\cdot\lambda(x)))(\omega \lambda(x)^{2}) dx||_{L^{2}(\rho(\omega \lambda(t)^{2}) d\omega)}\leq \frac{C}{t^{2} \log^{1+b-2b\alpha}(t)}\end{equation}
\end{lemma}
\begin{proof}
We start with
\begin{equation}\begin{split} &\int_{t}^{\infty} \frac{\sin((t-x)\sqrt{\omega})}{\sqrt{\omega}} \mathcal{F}(\sqrt{\cdot} F_{4}(x,\cdot \lambda(x)))(\omega \lambda(x)^{2}) dx\\
&=-\frac{\mathcal{F}(\sqrt{\cdot} F_{4}(t,\cdot\lambda(t)))(\omega \lambda(t)^{2})}{\omega}\\
&-\int_{t}^{\infty} \frac{\cos((t-x) \sqrt{\omega})}{\omega} \left(\mathcal{F}(\sqrt{\cdot} F_{4}(x,\cdot\lambda(x)))'(\omega \lambda(x)^{2}) \cdot 2 \lambda(x) \lambda'(x) \omega + \mathcal{F}(\sqrt{\cdot} \partial_{x}\left(F_{4}(x,\cdot \lambda(x))\right))(\omega \lambda(x)^{2})\right) dx\end{split}\end{equation}

Then,
\begin{equation} \begin{split} -\frac{\mathcal{F}(\sqrt{\cdot} F_{4}(t,\cdot\lambda(t)))(\omega \lambda(t)^{2})}{\omega}&=\frac{-1}{\omega} \int_{0}^{\frac{2}{\sqrt{\omega} \lambda(t)}} \phi(r,\omega \lambda(t)^{2}) \sqrt{r} F_{4}(t,r\lambda(t)) dr\\
&-\frac{1}{\omega} \int_{\frac{2}{\sqrt{\omega} \lambda(t)}}^{\infty} \phi(r,\omega \lambda(t)^{2}) \sqrt{r} F_{4}(t,r\lambda(t)) dr\end{split}\end{equation}

In the region $r \sqrt{\omega \lambda(t)^{2}} \leq 2$, we use proposition 4.4 of \cite{kst} to make the decomposition
\begin{equation} \phi(r,\omega \lambda(t)^{2}) = \widetilde{\phi_{0}}(r) + \frac{1}{\sqrt{r}} \sum_{j=1}^{\infty} (r^{2} \omega \lambda(t)^{2})^{j} \phi_{j}(r^{2})\end{equation}
where we denote, by $\widetilde{\phi_{0}}$, what was denoted by $\phi_{0}$ in \cite{kst}. In our notation, 
$$\widetilde{\phi_{0}}(r) = \frac{\sqrt{r}}{2} \phi_{0}(r)$$
The first term to consider is then
\begin{equation} -\frac{1}{\omega} \int_{0}^{\frac{2}{\sqrt{\omega}\lambda(t)}} \widetilde{\phi_{0}}(r) \sqrt{r}F_{4}(t,r\lambda(t)) dr = \frac{-1}{2\omega} \int_{0}^{\frac{2}{\sqrt{\omega} \lambda(t)}} r \phi_{0}(r) F_{4}(t,r\lambda(t)) dr\end{equation}
We will then consider several cases of $\omega$.\\
\\
\textbf{Case 1}: $1 \leq \frac{2}{\sqrt{\omega} \lambda(t)} \leq \frac{\log^{N}(t)}{\lambda(t)}$. In this region, we use the orthogonality of $F_{4}(t,\cdot \lambda(t))$ to $\phi_{0}$, to get
\begin{equation} \frac{-1}{2\omega} \int_{0}^{\frac{2}{\sqrt{\omega} \lambda(t)}} r \phi_{0}(r) F_{4}(t,r\lambda(t)) dr = \frac{1}{2\omega} \int_{\frac{2}{\sqrt{\omega} \lambda(t)}}^{\infty} \phi_{0}(r) F_{4}(t,r\lambda(t)) r dr\end{equation}
which gives
\begin{equation} \begin{split} &|\frac{-1}{2\omega} \int_{0}^{\frac{2}{\sqrt{\omega} \lambda(t)}} r \phi_{0}(r) F_{4}(t,r\lambda(t)) dr|\\
&\leq \frac{C}{\omega} \int_{\frac{2}{\sqrt{\omega} \lambda(t)}}^{\frac{\log^{N}(t)}{\lambda(t)}} \frac{\phi_{0}(r)}{\lambda(t)^{4} r^{4}} \left(\frac{r \lambda(t)}{t^{2} \log^{3b+1-2\alpha b}(t)}\right) r dr+\frac{C}{\omega} \int_{\frac{2}{\sqrt{\omega} \lambda(t)}}^{\frac{t}{2\lambda(t)}} \frac{\phi_{0}(r)}{r^{4} \lambda(t)^{4} \log^{5b+2N-2}(t)} \frac{r \lambda(t)}{t^{2}} r dr\\
&\leq \frac{C}{t^{2} \log^{2b+1-2\alpha b}(t)}, \quad 1 \leq \frac{2}{\sqrt{\omega} \lambda(t)} \leq \frac{\log^{N}(t)}{\lambda(t)}\end{split}\end{equation}\\
\\
\textbf{Case 2}: $\frac{2}{\sqrt{\omega} \lambda(t)} \leq 1$. Here, we have
\begin{equation}\begin{split} &|\frac{-1}{2\omega} \int_{0}^{\frac{2}{\sqrt{\omega} \lambda(t)}} r \phi_{0}(r) F_{4}(t,r\lambda(t)) dr|\\
&\leq \frac{C}{\omega} \int_{0}^{\frac{2}{\sqrt{\omega} \lambda(t)}} \phi_{0}(r) \left(\frac{r \lambda(t)}{\lambda(t)^{4} (r^{2}+1)^{2} t^{2} \log^{3b+1-2\alpha b}(t)}\right) r dr + \frac{C}{\omega} \int_{0}^{\frac{2}{\sqrt{\omega} \lambda(t)}} \frac{\phi_{0}(r) r \lambda(t) r dr}{\lambda(t)^{4} (r^{2}+1)^{2} \log^{5b+2N-2}(t) t^{2}}\\
&\leq \frac{C}{\omega \left(\sqrt{\omega} \lambda(t)\right)^{4} t^{2} \log^{1-2\alpha b}(t)}, \quad \frac{2}{\sqrt{\omega} \lambda(t)} \leq 1\end{split}\end{equation}\\
\\
\textbf{Case 3}: $\frac{\log^{N}(t)}{\lambda(t)} \leq \frac{2}{\sqrt{\omega} \lambda(t)} \leq \frac{t}{2\lambda(t)}$. We use the orthogonality condition here, and recall that the $v_{1}+v_{2}+v_{3}$ and $F_{0,2}$ terms are supported in the region $r \leq \log^{N}(t)$. Then, in the integral below, only the $v_{4}+v_{5}$ term contributes:
\begin{equation}\begin{split} &|\frac{-1}{2\omega} \int_{0}^{\frac{2}{\sqrt{\omega} \lambda(t)}} r \phi_{0}(r) F_{4}(t,r\lambda(t)) dr| = |\frac{1}{2\omega} \int_{\frac{2}{\sqrt{\omega} \lambda(t)}}^{\infty} \phi_{0}(r) F_{4}(t,r\lambda(t)) r dr|\\
&\leq \frac{C}{\omega} \int_{\frac{2}{\sqrt{\omega} \lambda(t)}}^{\frac{t}{2\lambda(t)}} \frac{r dr}{t^{2} \log^{2b+2N-2}(t)(r^{2}+1)^{2}}\\
&\leq \frac{C}{t^{2} \log^{4b+2N-2}(t)}, \quad \frac{\log^{N}(t)}{\lambda(t)} \leq \frac{2}{\sqrt{\omega} \lambda(t)} \leq \frac{t}{2\lambda(t)}\end{split}\end{equation}\\
\\
\textbf{Case 4}: $\frac{2}{\sqrt{\omega} \lambda(t)} \geq \frac{t}{2\lambda(t)}$. We use the orthogonality condition, and note that $F_{4}(t,r) =0, \quad r \geq \frac{t}{2}$. This gives
\begin{equation}\begin{split} &|\frac{-1}{2\omega} \int_{0}^{\frac{2}{\sqrt{\omega} \lambda(t)}} r \phi_{0}(r) F_{4}(t,r\lambda(t)) dr| = |\frac{1}{2\omega} \int_{\frac{2}{\sqrt{\omega} \lambda(t)}}^{\infty} \phi_{0}(r) F_{4}(t,r\lambda(t)) r dr|=0\end{split}\end{equation}
Combining these estimates, we get
\begin{equation} |\frac{-1}{2\omega} \int_{0}^{\frac{2}{\sqrt{\omega} \lambda(t)}} r \phi_{0}(r) F_{4}(t,r\lambda(t)) dr| \leq \begin{cases} \frac{C}{\omega^{3} \lambda(t)^{4} t^{2} \log^{1-2\alpha b}(t)}, \quad \frac{2}{\sqrt{\omega} \lambda(t)} \leq 1\\
\frac{C}{t^{2} \log^{2b+1-2\alpha b}(t)}, \quad 1 \leq \frac{2}{\sqrt{\omega} \lambda(t)} \leq \frac{\log^{N}(t)}{\lambda(t)}\\
\frac{C}{t^{2} \log^{4b+2N-2}(t)}, \quad \frac{\log^{N}(t)}{\lambda(t)} \leq \frac{2}{\sqrt{\omega} \lambda(t)} \leq \frac{t}{2\lambda(t)}\\
0, \quad \frac{t}{2\lambda(t)} \leq \frac{2}{\sqrt{\omega} \lambda(t)}\end{cases}\end{equation}
Using proposition 4.7 b of \cite{kst}, we have (for example)
$$\rho(x) \leq \begin{cases}\frac{C}{x \log^{2}(x)}, \quad x \leq \frac{1}{4}\\
C x, \quad x \geq \frac{1}{4}\end{cases}$$
This leads to
\begin{equation}\label{1a} ||\frac{-1}{2\omega} \int_{0}^{\frac{2}{\sqrt{\omega} \lambda(t)}} r \phi_{0}(r) F_{4}(t,r\lambda(t)) dr||_{L^{2}(\rho(\omega \lambda(t)^{2})d\omega)} \leq \frac{C}{t^{2} \log^{b+1-2\alpha b}(t)}\end{equation}

The next integral to consider is 
\begin{equation} -\frac{1}{\omega} \sum_{j=1}^{\infty} \int_{0}^{\frac{2}{\sqrt{\omega} \lambda(t)}} r^{2j} \omega ^{j} \lambda(t)^{2j} \phi_{j}(r^{2}) F_{4}(t,r\lambda(t)) dr\end{equation}
For this integral, there is of course no orthogonality to exploit. In all cases of $\omega$, we will use the estimate (from proposition 4.4 of \cite{kst})
$$|\phi_{j}(r^{2})| \leq \frac{3 C_{1}^{j}}{(j-1)!} \log(1+r^{2})$$ 

We again treat various cases of $\omega$.\\
\\
\textbf{Case 1}: $1\leq \frac{2}{\sqrt{\omega} \lambda(t)} \leq \frac{\log^{N}(t)}{\lambda(t)}$. Then, we have
\begin{equation}\begin{split} &|-\frac{1}{\omega} \sum_{j=1}^{\infty} \int_{0}^{\frac{2}{\sqrt{\omega} \lambda(t)}} r^{2j} \omega ^{j} \lambda(t)^{2j} \phi_{j}(r^{2}) F_{4}(t,r\lambda(t)) dr| \\
&\leq \frac{C}{\omega} \sum_{j=1}^{\infty} \int_{0}^{\frac{2}{\sqrt{\omega} \lambda(t)}} r^{2j} \omega^{j} \lambda(t)^{2j} \frac{C_{1}^{j}}{(j-1)!} \left(\frac{\log(\log(t)) r}{\lambda(t)^{3} (r^{2}+1)^{2} t^{2} \log^{3b+1-2b\alpha}(t)} + \frac{\log(t) r \lambda(t)}{t^{2} \lambda(t)^{4} (r^{2}+1)^{2} \log^{5b+2N-2}(t)}\right)dr\\
&\leq \frac{C}{\omega} \sum_{j=1}^{\infty} \left(\int_{0}^{1} \frac{r^{2j+1} \omega^{j} \lambda(t)^{2j} C_{1}^{j}}{(j-1)!} \frac{\log(\log(t))}{t^{2} \log^{1-2b\alpha}(t)} dr + \int_{1}^{\frac{2}{\sqrt{\omega}\lambda(t)}} \frac{\omega^{j} \lambda(t)^{2j} C_{1}^{j}}{(j-1)!} \frac{\log(\log(t))}{t^{2} \log^{1-2b\alpha}(t)} r^{2j-3}dr\right)\\
&\leq \frac{C (\log(\log(t)))^{2}}{t^{2} \log^{1+2b-2b\alpha}(t)}\end{split}\end{equation}
where we used the support properties of the various terms in $F_{4}$, and the largeness of $N$.\\
\\
\textbf{Case 2}: $\frac{2}{\sqrt{\omega} \lambda(t)} \leq 1$. Here, we have
\begin{equation}\begin{split}&|-\frac{1}{\omega} \sum_{j=1}^{\infty} \int_{0}^{\frac{2}{\sqrt{\omega} \lambda(t)}} r^{2j} \omega ^{j} \lambda(t)^{2j} \phi_{j}(r^{2}) F_{4}(t,r\lambda(t)) dr| \\
&\leq \frac{C}{\omega} \sum_{j=1}^{\infty} \int_{0}^{\frac{2}{\sqrt{\omega}\lambda(t)}} r^{2j} \omega^{j} \lambda(t)^{2j} \frac{C_{1}^{j}}{(j-1)!} \log(1+r^{2}) \left(\frac{r \lambda(t)}{\lambda(t)^{4}(r^{2}+1)^{2}} \frac{1}{t^{2}\log^{3b+1-2b\alpha}(t)}+\frac{r \lambda(t)}{\lambda(t)^{4}(r^{2}+1)^{2}} \frac{1}{t^{2} \log^{5b+2N-2}(t)}\right) dr\\
&\leq \frac{C}{\omega} \sum_{j=1}^{\infty} \left(\int_{0}^{\frac{2}{\sqrt{\omega}\lambda(t)}} r^{2j+3} dr\right) \frac{\omega^{j} \lambda(t)^{2j} C_{1}^{j}}{(j-1)! t^{2} \log^{1-2b\alpha}(t)}\\
&\leq \frac{C}{\omega^{3} \lambda(t)^{4} t^{2} \log^{1-2b\alpha}(t)}, \quad \frac{2}{\sqrt{\omega}\lambda(t)} \leq 1\end{split}\end{equation}\\
\\
\textbf{Case 3}: $\frac{\log^{N}(t)}{\lambda(t)} \leq \frac{2}{\sqrt{\omega}\lambda(t)} \leq \frac{t}{2\lambda(t)}$. Again, we use the support properties of the various terms in $F_{4}$, to get
\begin{equation}\begin{split}&|-\frac{1}{\omega} \sum_{j=1}^{\infty} \int_{0}^{\frac{2}{\sqrt{\omega} \lambda(t)}} r^{2j} \omega ^{j} \lambda(t)^{2j} \phi_{j}(r^{2}) F_{4}(t,r\lambda(t)) dr| \\
&\leq \frac{C}{\omega} \sum_{j=1}^{\infty} \left(\int_{0}^{\frac{\log^{N}(t)}{\lambda(t)}} \frac{r^{2j+1} dr}{(r^{2}+1)^{2}}\right) \frac{\omega^{j} \lambda(t)^{2j} C_{1}^{j}}{(j-1)!} \frac{\log(\log(t))}{t^{2} \log^{1-2b\alpha}(t)}+\frac{C}{\omega} \sum_{j=1}^{\infty} \left(\int_{0}^{\frac{2}{\sqrt{\omega}\lambda(t)}} \frac{r^{2j+1} dr}{(r^{2}+1)^{2}}\right) \frac{\omega^{j} \lambda(t)^{2j} C_{1}^{j}}{(j-1)! t^{2} \log^{2b+2N-3}(t)}\\
&\leq \frac{C (\log(\log(t)))^{2}}{t^{2} \log^{1+2b-2b\alpha}(t)}, \quad \frac{\log^{N}(t)}{\lambda(t)} \leq \frac{2}{\sqrt{\omega} \lambda(t)} \leq \frac{t}{2\lambda(t)}\end{split}\end{equation}\\
\\
\textbf{Case 4}: $\frac{2}{\sqrt{\omega}\lambda(t)} \geq \frac{t}{2\lambda(t)}$.
\begin{equation}\begin{split}&|-\frac{1}{\omega} \sum_{j=1}^{\infty} \int_{0}^{\frac{2}{\sqrt{\omega} \lambda(t)}} r^{2j} \omega ^{j} \lambda(t)^{2j} \phi_{j}(r^{2}) F_{4}(t,r\lambda(t)) dr| \\
&\leq \frac{C}{\omega} \sum_{j=1}^{\infty} \left(\int_{0}^{\frac{\log^{N}(t)}{\lambda(t)}} \frac{r^{2j+1} dr}{(r^{2}+1)^{2}}\right) \frac{\omega^{j} \lambda(t)^{2j} C_{1}^{j}}{(j-1)!} \frac{\log(\log(t))}{t^{2} \log^{1-2b\alpha}(t)}+\frac{C}{\omega} \sum_{j=1}^{\infty} \left(\int_{0}^{\frac{t}{2\lambda(t)}} \frac{r^{2j+1} dr}{(r^{2}+1)^{2}}\right) \frac{\omega^{j} \lambda(t)^{2j} C_{1}^{j}}{(j-1)! t^{2} \log^{2b+2N-3}(t)}\\
&\leq \frac{C (\log(\log(t)))^{2}}{t^{2} \log^{1+2b-2b\alpha}(t)}, \quad  \frac{2}{\sqrt{\omega} \lambda(t)} \geq \frac{t}{2\lambda(t)}\end{split}\end{equation}\\
\\
Then, we get
\begin{equation}\label{1b} ||-\frac{1}{\omega} \sum_{j=1}^{\infty} \int_{0}^{\frac{2}{\sqrt{\omega} \lambda(t)}} r^{2j} \omega ^{j} \lambda(t)^{2j} \phi_{j}(r^{2}) F_{4}(t,r\lambda(t)) dr||_{L^{2}(\rho(\omega \lambda(t)^{2}) d\omega)} \leq \frac{C (\log(\log(t)))^{2}}{t^{2} \log^{b+1-2b\alpha}(t)}\end{equation}
The third integral to estimate is
\begin{equation} \begin{split} &-\frac{1}{\omega} \int_{\frac{2}{\sqrt{\omega} \lambda(t)}}^{\infty} \phi(r,\omega \lambda(t)^{2}) \sqrt{r} F_{4}(t,r\lambda(t)) dr = -\frac{2}{\omega} \text{Re}\left(\int_{\frac{2}{\sqrt{\omega}\lambda(t)}}^{\infty} a(\omega \lambda(t)^{2}) \psi^{+}(r,\omega\lambda(t)^{2}) \sqrt{r} F_{4}(t,r\lambda(t)) dr \right)\end{split}\end{equation}
where we used propositions 4.6 and 4.7 of \cite{kst}. The estimates from these propositions also imply
\begin{equation}\begin{split}&|-\frac{1}{\omega} \int_{\frac{2}{\sqrt{\omega} \lambda(t)}}^{\infty} \phi(r,\omega \lambda(t)^{2}) \sqrt{r} F_{4}(t,r\lambda(t)) dr| \leq \frac{C |a(\omega \lambda(t)^{2})|}{\omega^{5/4}\lambda(t)^{1/2}} \int_{\frac{2}{\sqrt{\omega} \lambda(t)}}^{\infty} \sqrt{r} |F_{4}(t,r\lambda(t))| dr\end{split}\end{equation}\\
\\
\textbf{Case 1}: $1 \leq \frac{2}{\sqrt{\omega} \lambda(t)} \leq \frac{\log^{N}(t)}{\lambda(t)}$.
\begin{equation}\begin{split} &|-\frac{1}{\omega} \int_{\frac{2}{\sqrt{\omega} \lambda(t)}}^{\infty} \phi(r,\omega \lambda(t)^{2}) \sqrt{r} F_{4}(t,r\lambda(t)) dr|\\
&\leq \frac{C |a(\omega \lambda(t)^{2})|}{\omega^{5/4}\lambda(t)^{1/2}}\left(\int_{\frac{2}{\sqrt{\omega} \lambda(t)}}^{\frac{\log^{N}(t)}{\lambda(t)}} \frac{\sqrt{r} dr}{\lambda(t)^{3} r^{3} t^{2} \log^{3b+1-2b\alpha}(t)} + \int_{\frac{2}{\sqrt{\omega}\lambda(t)}}^{\infty} \frac{r^{3/2}}{\lambda(t)^{3} r^{4} \log^{5b+2N-2}(t) t^{2}}\right)\\
&\leq C \frac{|a(\omega\lambda(t)^{2})|}{\omega^{5/4} \lambda(t)^{1/2}} \cdot \frac{\omega^{3/4} \lambda(t)^{3/2}}{t^{2} \log^{3b+1-2\alpha b}(t) \lambda(t)^{3}}\\
&\leq \frac{C |a(\omega \lambda(t)^{2})|}{\omega} \cdot \frac{\sqrt{\omega} \lambda(t)}{t^{2} \log^{1-2\alpha b}(t)}, \quad 1 \leq \frac{2}{\sqrt{\omega} \lambda(t)} \leq \frac{\log^{N}(t)}{\lambda(t)}\end{split}\end{equation}\\
\\
\textbf{Case 2}: $\frac{2}{\sqrt{\omega}\lambda(t)} \leq 1$. 
\begin{equation}\begin{split} &|-\frac{1}{\omega} \int_{\frac{2}{\sqrt{\omega} \lambda(t)}}^{\infty} \phi(r,\omega \lambda(t)^{2}) \sqrt{r} F_{4}(t,r\lambda(t)) dr|\\
&\leq \frac{C |a(\omega \lambda(t)^{2})|}{\omega^{5/4} \lambda(t)^{1/2}}\left(\frac{1}{t^{2} \log^{1-2\alpha b}(t)} \int_{0}^{\frac{\log^{N}(t)}{\lambda(t)}} \frac{r^{3/2} dr}{(r^{2}+1)^{2}} + \frac{1}{t^{2} \log^{2b+2N-2}(t)} \int_{0}^{\infty} \frac{r^{3/2} dr}{(r^{2}+1)^{2}}\right)\\
&\leq \frac{C |a(\omega \lambda(t)^{2})|}{\omega^{5/4} \lambda(t)^{1/2} t^{2} \log^{1-2\alpha b}(t)}, \quad \frac{2}{\sqrt{\omega} \lambda(t)} \leq 1\end{split}\end{equation}\\
\\
\textbf{Case 3}: $\frac{\log^{N}(t)}{\lambda(t)} \leq \frac{2}{\sqrt{\omega} \lambda(t)} \leq \frac{t}{2\lambda(t)}$. In this region, the $F_{0,2}$ and $v_{1}+v_{2}+v_{3}$ terms do not contribute to the integral because of the support properties of $1-\chi_{\geq 1}(\frac{r}{\log^{N}(t)})$. We get
\begin{equation}\begin{split}&|-\frac{1}{\omega} \int_{\frac{2}{\sqrt{\omega} \lambda(t)}}^{\infty} \phi(r,\omega \lambda(t)^{2}) \sqrt{r} F_{4}(t,r\lambda(t)) dr|\leq \frac{C |a(\omega \lambda(t)^{2})|}{\omega^{5/4} \lambda(t)^{1/2}} \int_{\frac{2}{\sqrt{\omega}\lambda(t)}}^{\infty} \frac{\sqrt{r} dr}{\lambda(t)^{3} r^{3} \log^{5b+2N-2}(t) t^{2}}\\
&\leq \frac{C |a(\omega \lambda(t)^{2})|}{\omega} \cdot \frac{\sqrt{\omega} \lambda(t)}{t^{2} \log^{2b+2N-2}(t)}\end{split}\end{equation}\\
\\
\textbf{Case 4}: $\frac{2}{\sqrt{\omega} \lambda(t)} \geq \frac{t}{2\lambda(t)}$. In this case, the integral is zero, because of the support properties of $F_{4}$.\\
\\
Combining the above estimates, we get
\begin{equation} |-\frac{1}{\omega} \int_{\frac{2}{\sqrt{\omega} \lambda(t)}}^{\infty} \phi(r,\omega \lambda(t)^{2}) \sqrt{r} F_{4}(t,r\lambda(t)) dr| \leq C |a(\omega \lambda(t)^{2})| \begin{cases} \frac{1}{\omega^{5/4} \lambda(t)^{1/2} t^{2} \log^{1-2\alpha b}(t)}, \quad \frac{2}{\sqrt{\omega} \lambda(t)} \leq 1\\
\frac{1}{\sqrt{\omega} t^{2} \log^{1+b-2\alpha b}(t)}, \quad 1 \leq \frac{2}{\sqrt{\omega} \lambda(t)} \leq \frac{\log^{N}(t)}{\lambda(t)}\\
\frac{1}{\sqrt{\omega} t^{2} \log^{3b+2N-2}(t)}, \quad \frac{\log^{N}(t)}{\lambda(t)} \leq \frac{2}{\sqrt{\omega} \lambda(t)} \leq \frac{t}{2\lambda(t)}\\
0, \quad \frac{t}{2\lambda(t)} \leq \frac{2}{\sqrt{\omega} \lambda(t)}\end{cases}\end{equation}
which gives
\begin{equation}\label{1c} ||-\frac{1}{\omega} \int_{\frac{2}{\sqrt{\omega} \lambda(t)}}^{\infty} \phi(r,\omega \lambda(t)^{2}) \sqrt{r} F_{4}(t,r\lambda(t)) dr||_{L^{2}(\rho(\omega \lambda(t)^{2})d\omega)} \leq \frac{C \sqrt{\log(\log(t))}}{t^{2} \log^{1+b-2\alpha b}(t)}\end{equation}
(Note that, by proposition 4.7 b of \cite{kst}, $\rho(\xi) = \frac{1}{\pi |a(\xi)|^{2}}$).
Now, we consider
\begin{equation} -\int_{t}^{\infty} dx \frac{\cos((t-x) \sqrt{\omega})}{\omega} \mathcal{F}(\sqrt{\cdot} F_{4}(x,\cdot \lambda(x)))'(\omega \lambda(x)^{2}) \cdot 2 \omega \lambda(x) \lambda'(x)\end{equation}
We start by noting that, in the region $r \leq \frac{2}{\sqrt{\xi}}$,
\begin{equation} \partial_{2}\phi(r,\xi) = \partial_{\xi}\left(\widetilde{\phi_{0}}(r) + \frac{1}{\sqrt{r}} \sum_{j=1}^{\infty}(r^{2j} \xi^{j} \phi_{j}(r^{2}))\right)\end{equation}
So,
\begin{equation}\label{ftransprime}\begin{split} \mathcal{F}(\sqrt{\cdot} F_{4}(x,\cdot \lambda(x)))'(\omega \lambda(x)^{2}) &= \int_{0}^{\frac{2}{\sqrt{\omega} \lambda(x)}} \sqrt{r} F_{4}(x,r\lambda(x)) \cdot \frac{1}{\sqrt{r}} \sum_{j=1}^{\infty} (r^{2j} j (\omega \lambda(x)^{2})^{j-1} \phi_{j}(r^{2})) dr\\
&+\int_{\frac{2}{\sqrt{\omega} \lambda(x)}}^{\infty} \sqrt{r} F_{4}(x,r\lambda(x)) \partial_{2}\phi(r,\omega\lambda(x)^{2}) dr\end{split}\end{equation} 
We start with the first term:
\begin{equation}\begin{split} &|-\int_{t}^{\infty} dx \frac{\cos((t-x) \sqrt{\omega})}{\omega} \left(\int_{0}^{\frac{2}{\sqrt{\omega} \lambda(x)}} \sqrt{r} F_{4}(x,r\lambda(x)) \cdot \frac{1}{\sqrt{r}} \sum_{j=1}^{\infty} (r^{2j} j (\omega \lambda(x)^{2})^{j-1} \phi_{j}(r^{2})) dr\right) \cdot 2 \omega \lambda(x) \lambda'(x)|\\
&\leq  C \sum_{j=1}^{\infty} j \frac{\omega^{j}}{\omega} \int_{t}^{\infty} dx \lambda(x)^{2j-1} |\lambda'(x)| \int_{0}^{\frac{2}{\sqrt{\omega}\lambda(x)}} |F_{4}(x,r\lambda(x))| r^{2j} \phi_{j}(r^{2}) dr\\
&\leq C \sum_{j=1}^{\infty} \frac{j \omega^{j} C_{1}^{j}}{(j-1)!} \int_{t}^{\infty} dx \frac{\lambda(x)^{2j-1} |\lambda'(x)|}{\omega} \int_{0}^{\frac{2}{\sqrt{\omega} \lambda(x)}} |F_{4}(x,r\lambda(x))| r^{2j} \log(1+r^{2}) dr\end{split}\end{equation}
The $r$ integral in the last line of the above expression was estimated before, and we get
\begin{equation}\begin{split} &|-\int_{t}^{\infty} dx \frac{\cos((t-x) \sqrt{\omega})}{\omega} \left(\int_{0}^{\frac{2}{\sqrt{\omega} \lambda(x)}} \sqrt{r} F_{4}(x,r\lambda(x)) \cdot \frac{1}{\sqrt{r}} \sum_{j=1}^{\infty} (r^{2j} j (\omega \lambda(x)^{2})^{j-1} \phi_{j}(r^{2})) dr\right) \cdot 2 \omega \lambda(x) \lambda'(x)|\\
&\leq \int_{t}^{\infty} \frac{dx}{x \log(x)} \begin{cases} \frac{C}{\omega^{3} \lambda(x)^{4} x^{2} \log^{1-2b\alpha}(x)}, \quad \frac{2}{\sqrt{\omega}\lambda(x)} \leq 1\\
\frac{C (\log(\log(x)))^{2}}{x^{2} \log^{1+2b-2b\alpha}(x)}, \quad 1 \leq \frac{2}{\sqrt{\omega}\lambda(x)}\end{cases}\end{split}\end{equation}
Finally,
\begin{equation}\label{2a} \begin{split} &||-\int_{t}^{\infty} dx \frac{\cos((t-x) \sqrt{\omega})}{\omega} \left(\int_{0}^{\frac{2}{\sqrt{\omega} \lambda(x)}} \sqrt{r} F_{4}(x,r\lambda(x)) \cdot \frac{1}{\sqrt{r}} \sum_{j=1}^{\infty} (r^{2j} j (\omega \lambda(x)^{2})^{j-1} \phi_{j}(r^{2})) dr\right) \cdot 2 \omega \lambda(x) \lambda'(x)||_{L^{2}(\rho(\omega \lambda(t)^{2}) d\omega)}\\
&\leq C \int_{t}^{\infty} dx ||\frac{1}{\omega}\cdot\left(\int_{0}^{\frac{2}{\sqrt{\omega} \lambda(x)}} \sqrt{r} F_{4}(x,r\lambda(x)) \cdot \frac{1}{\sqrt{r}} \sum_{j=1}^{\infty} (r^{2j} j (\omega \lambda(x)^{2})^{j-1} \phi_{j}(r^{2})) dr\right) \cdot 2 \omega \lambda(x) \lambda'(x)||_{L^{2}(\rho(\omega \lambda(t)^{2}) d\omega)}\\
&\leq C \int_{t}^{\infty} dx \left(\frac{\lambda(t)}{\lambda(x)}\right)||\frac{1}{\omega}\cdot\left(\int_{0}^{\frac{2}{\sqrt{\omega} \lambda(x)}} \sqrt{r} F_{4}(x,r\lambda(x)) \cdot \frac{1}{\sqrt{r}} \sum_{j=1}^{\infty} (r^{2j} j (\omega \lambda(x)^{2})^{j-1} \phi_{j}(r^{2})) dr\right) \cdot 2 \omega \lambda(x) \lambda'(x)||_{L^{2}(\rho(\omega \lambda(x)^{2}) d\omega)}\\
&\leq \frac{C (\log(\log(t)))^{2}}{t^{2} \log^{2+b-2b\alpha}(t)}\end{split}\end{equation}
where we used \eqref{rhoscaling}.
The next integral to treat is 
\begin{equation}\begin{split}&|-2 \omega \int_{t}^{\infty} dx \frac{\cos((t-x)\sqrt{\omega})}{\omega} \lambda(x) \lambda'(x) \int_{\frac{2}{\sqrt{\omega} \lambda(x)}}^{\infty} \sqrt{r} F_{4}(x,r\lambda(x)) \partial_{2}\phi(r,\omega \lambda(x)^{2}) dr|\\
&\leq C \int_{t}^{\infty} \frac{dx}{x \log^{2b+1}(x)} \int_{\frac{2}{\sqrt{\omega}\lambda(x)}}^{\infty} \sqrt{r} |F_{4}(x,r\lambda(x))| |\partial_{2}\phi(r,\omega \lambda(x)^{2})| dr\end{split}\end{equation}
We again use 
$$\phi(r,\xi) = 2 \text{Re}(a(\xi) \psi^{+}(r,\xi)), \quad r \geq \frac{2}{\sqrt{\xi}}$$
and the following symbol type estimate on $a$ from proposition 4.7 of \cite{kst}:
$$|a^{(k)}(\xi)| \leq \frac{C_{k} |a(\xi)|}{\xi^{k}}$$
as well as, from proposition 4.6 of \cite{kst}
\begin{equation} \psi^{+}(r,\xi) = \frac{e^{i r \sqrt{\xi}}}{\xi^{1/4}} \sigma(r\sqrt{\xi},r), \quad r \sqrt{\xi} \geq 2\end{equation}
and the asymptotic series representation for $\sigma$, to get
\begin{equation} |\partial_{2}\phi(r,\xi)| \leq \frac{C |a(\xi)|}{\xi^{3/4}} r, \quad r \sqrt{\xi} \geq 2\end{equation}
This gives
\begin{equation}\begin{split}&|-2 \omega \int_{t}^{\infty} dx \frac{\cos((t-x)\sqrt{\omega})}{\omega} \lambda(x) \lambda'(x) \int_{\frac{2}{\sqrt{\omega} \lambda(x)}}^{\infty} \sqrt{r} F_{4}(x,r\lambda(x)) \partial_{2}\phi(r,\omega \lambda(x)^{2}) dr|\\
&\leq \frac{C |a(\omega \lambda(x)^{2})|}{\omega^{3/4} \lambda(x)^{3/2}} \int_{t}^{\infty} \frac{dx}{x \log^{2b+1}(x)} \int_{\frac{2}{\sqrt{\omega}\lambda(x)}}^{\infty} |F_{4}(x,r\lambda(x))| r^{3/2} dr\end{split}\end{equation}
Note that estimates on this integral can not quite be infered from estimates on a previously treated integral. So, we start considering cases of $\omega$:\\
\\
\textbf{Case 1}: $1 \leq \frac{2}{\sqrt{\omega}\lambda(x)} \leq \frac{\log^{N}(x)}{\lambda(x)}$. Here, we get
\begin{equation}\begin{split}  &\frac{|a(\omega \lambda(x)^{2})|}{\omega^{3/4} \lambda(x)^{3/2}}\int_{\frac{2}{\sqrt{\omega}\lambda(x)}}^{\infty} |F_{4}(x,r\lambda(x))| r^{3/2} dr\\
&\leq C\frac{|a(\omega \lambda(x)^{2})|}{\omega^{3/4} \lambda(x)^{3/2}}\left( \int_{\frac{2}{\sqrt{\omega} \lambda(x)}}^{\frac{\log^{N}(x)}{\lambda(x)}} \frac{r^{5/2} \lambda(x)}{\lambda(x)^{4}(r^{2}+1)^{2} x^{2} \log^{3b+1-2b\alpha}(x)}dr+ \int_{\frac{2}{\sqrt{\omega} \lambda(x)}}^{\frac{x}{2\lambda(x)}} \frac{r^{3/2} dr}{\lambda(x)^{3} r^{3} x^{2} \log^{5b+2N-2}(x)}\right)\\
&\leq \frac{C |a(\omega \lambda(x)^{2})|}{x^{2} \log^{1-2b\alpha}(x) \sqrt{\omega}\lambda(x)}, \quad 1 \leq \frac{2}{\sqrt{\omega} \lambda(x)} \leq \frac{\log^{N}(x)}{\lambda(x)}\end{split}\end{equation}\\
\\
\textbf{Case 2}: $\frac{2}{\sqrt{\omega} \lambda(x)} \leq 1$. 
\begin{equation}\begin{split} &\frac{|a(\omega \lambda(x)^{2})|}{\omega^{3/4} \lambda(x)^{3/2}}\int_{\frac{2}{\sqrt{\omega}\lambda(x)}}^{\infty} |F_{4}(x,r\lambda(x))| r^{3/2} dr\\
&\leq \frac{C |a(\omega \lambda(x)^{2})|}{\omega^{3/4} \lambda(x)^{3/2}} \left(\int_{0}^{\frac{\log^{N}(x)}{\lambda(x)}} \frac{r^{5/2} dr}{(r^{2}+1)^{2} x^{2} \log^{1-2b\alpha}(x)} + \int_{0}^{\infty} \frac{r^{5/2} dr}{x^{2}(r^{2}+1)^{2} \log^{2b+2N-2}(x)}\right)\\
&\leq \frac{C |a(\omega \lambda(x)^{2})|}{\omega^{3/4} \lambda(x)^{3/2} x^{2} \log^{1-2b\alpha}(x)}, \quad \frac{2}{\sqrt{\omega}\lambda(x)} \leq 1\end{split}\end{equation}\\
\\
\textbf{Case 3}: $\frac{x}{2\lambda(x)} \geq \frac{2}{\sqrt{\omega} \lambda(x)} \geq \frac{\log^{N}(x)}{\lambda(x)}$. In this region, only the $v_{4}+v_{5}$ term in $F_{4}$ contributes:
\begin{equation} \begin{split} &\frac{|a(\omega \lambda(x)^{2})|}{\omega^{3/4} \lambda(x)^{3/2}}\int_{\frac{2}{\sqrt{\omega}\lambda(x)}}^{\infty} |F_{4}(x,r\lambda(x))| r^{3/2} dr\\
&\leq \frac{C |a(\omega \lambda(x)^{2})|}{\omega^{3/4}\lambda(x)^{3/2}} \int_{\frac{2}{\sqrt{\omega}\lambda(x)}}^{\frac{x}{2\lambda(x)}} \frac {r^{3/2} dr}{\lambda(x)^{3}r^{3} x^{2} \log^{5b+2N-2}(x)}\\
&\leq \frac{C |a(\omega \lambda(x)^{2})|}{\sqrt{\omega} \lambda(x)^{3/2} x^{2} \log^{\frac{5b}{2}+2N-2}(x)}, \quad \frac{x}{2\lambda(x)} \geq \frac{2}{\sqrt{\omega}\lambda(x)} \geq \frac{\log^{N}(x)}{\lambda(x)}\end{split}\end{equation}\\
\\
\textbf{Case 4}: $\frac{2}{\sqrt{\omega}\lambda(x)} \geq \frac{x}{2\lambda(x)}$. In this region, the integral to estimate is zero, because of the support properties of $F_{4}$.

Combining these estimates, we get
\begin{equation}\label{2bptwse} \frac{|a(\omega \lambda(x)^{2})|}{\omega^{3/4} \lambda(x)^{3/2}}\int_{\frac{2}{\sqrt{\omega}\lambda(x)}}^{\infty} |F_{4}(x,r\lambda(x))| r^{3/2} dr \leq C|a(\omega \lambda(x)^{2})| \cdot \begin{cases} \frac{1}{x^{2} \log^{1-2b\alpha}(x) \sqrt{\omega} \lambda(x)}, \quad 1 \leq \frac{2}{\sqrt{\omega}\lambda(x)} \leq \frac{\log^{N}(x)}{\lambda(x)}\\
\frac{1}{\omega^{3/4}\lambda(x)^{3/2} x^{2} \log^{1-2b\alpha}(x)}, \quad \frac{2}{\sqrt{\omega}\lambda(x)} \leq 1\\
\frac{1}{\sqrt{\omega}\lambda(x)^{3/2} x^{2} \log^{\frac{5b}{2}+2N-2}(x)}, \quad \frac{x}{2\lambda(x)} \geq \frac{2}{\sqrt{\omega}\lambda(x)} \geq \frac{\log^{N}(x)}{\lambda(x)}\\
0, \quad \frac{2}{\sqrt{\omega}\lambda(x)} \geq \frac{x}{2\lambda(x)}\end{cases}\end{equation}
Using the same procedure as in \eqref{2a}, we get

\begin{equation}\label{2b} ||-2 \omega \int_{t}^{\infty} dx \frac{\cos((t-x)\sqrt{\omega})}{\omega} \lambda(x) \lambda'(x) \int_{\frac{2}{\sqrt{\omega} \lambda(x)}}^{\infty} \sqrt{r} F_{4}(x,r\lambda(x)) \partial_{2}\phi(r,\omega \lambda(x)^{2}) dr||_{L^{2}(\rho(\omega\lambda(t)^{2})d\omega)} \leq \frac{C \sqrt{\log(\log(t))}}{t^{2} \log^{2+b-2b\alpha}(t)}\end{equation}
To treat the next integral:
\begin{equation} -\int_{t}^{\infty} \frac{\cos((t-x)\sqrt{\omega})}{\omega} \mathcal{F}(\sqrt{\cdot}\partial_{x}(F_{4}(x,\cdot\lambda(x))))(\omega \lambda(x)^{2}) dx\end{equation}
we first note that
$$\int_{0}^{\infty} \phi_{0}(r) F_{4}(x,r\lambda(x)) r dr =0, \quad x \geq T_{0}$$
implies
\begin{equation} \int_{0}^{\infty} \phi_{0}(r) \partial_{x}(F_{4}(x,r\lambda(x))) r dr =0, \quad x \geq T_{0}\end{equation}
So, $\partial_{x}(F_{4}(x,r \lambda(x)))$ is still orthogonal to $\phi_{0}(r)$. Also, by noting the symbol-type estimates, \eqref{f4estimates}, and inspecting the procedure used to estimate $\frac{\mathcal{F}(\sqrt{\cdot} F_{4}(x,\cdot \lambda(x)))(\omega \lambda(x)^{2})}{\omega}$, we get 
$$||\frac{\mathcal{F}(\sqrt{\cdot}\partial_{x}(F_{4}(x,\cdot \lambda(x))))(\omega\lambda(x)^{2})}{\omega}||_{L^{2}(\rho(\omega \lambda(t)^{2}) d\omega)} \leq \frac{C \lambda(t)}{\lambda(x)} \left(\frac{(\log(\log(x)))^{2}}{x^{3} \log^{b+1-2\alpha b}(x)}\right)$$
where we again use \eqref{rhoscaling}. Using Minkowski's inequality, we get
\begin{equation}\label{3} ||-\int_{t}^{\infty} \frac{\cos((t-x)\sqrt{\omega})}{\omega} \mathcal{F}(\sqrt{\cdot}\partial_{x}(F_{4}(x,\cdot \lambda(x))))(\omega \lambda(x)^{2}) dx||_{L^{2}(\rho(\omega \lambda(t)^{2}) d\omega)} \leq \frac{C (\log(\log(t))^{2}}{t^{2} \log^{b+1-2\alpha b}(t)}\end{equation}
Combining this with our other estimates in this section, we finally get
\begin{equation} ||\int_{t}^{\infty} \frac{\sin((t-x)\sqrt{\omega})}{\sqrt{\omega}} \mathcal{F}(\sqrt{\cdot}F_{4}(x,\cdot\lambda(x)))(\omega \lambda(x)^{2}) dx||_{L^{2}(\rho(\omega \lambda(t)^{2})d\omega)} \leq \frac{C (\log(\log(t))^{2}}{t^{2}\log^{b+1-2\alpha b}(t)}\end{equation}
Next, we estimate 
\begin{equation}\begin{split} &\partial_{t}\left(\int_{t}^{\infty} \frac{\sin((t-x)\sqrt{\omega})}{\sqrt{\omega}} \mathcal{F}(\sqrt{\cdot}F_{4}(x,\cdot\lambda(x)))(\omega \lambda(x)^{2}) dx\right) = \int_{t}^{\infty} \cos((t-x)\sqrt{\omega}) \mathcal{F}(\sqrt{\cdot} F_{4}(x,\cdot\lambda(x)))(\omega \lambda(x)^{2}) dx\\
&=\int_{t}^{\infty} \frac{\sin((t-x)\sqrt{\omega})}{\sqrt{\omega}} \left(\mathcal{F}(\sqrt{\cdot} F_{4}(x,\cdot\lambda(x)))'(\omega \lambda(x)^{2}) \cdot 2 \lambda(x) \lambda'(x) \omega + \mathcal{F}(\sqrt{\cdot} \partial_{x}(F_{4}(x,\cdot \lambda(x))))(\omega \lambda(x)^{2})\right) dx\\
&=-\left(\mathcal{F}(\sqrt{\cdot}F_{4}(t,\cdot\lambda(t)))'(\omega \lambda(t)^{2}) \cdot 2 \lambda(t) \lambda'(t) + \frac{\mathcal{F}(\sqrt{\cdot} \partial_{t}(F_{4}(t,\cdot\lambda(t))))(\omega\lambda(t)^{2})}{\omega}\right)\\
&-\int_{t}^{\infty} \frac{\cos((t-x)\sqrt{\omega})}{\omega} \left(\mathcal{F}(\sqrt{\cdot} F_{4}(x,\cdot\lambda(x)))''(\omega \lambda(x)^{2})\cdot 4 \lambda(x)^{2}\lambda'(x)^{2} \omega^{2} + 2 \mathcal{F}(\sqrt{\cdot} \partial_{x}(F_{4}(x,\cdot\lambda(x))))'(\omega \lambda(x)^{2})\cdot 2 \lambda(x)\lambda'(x)\omega \right.\\
&\left.+ \mathcal{F}(\sqrt{\cdot} F_{4}(x,\cdot\lambda(x)))'(\omega \lambda(x)^{2}) \cdot 2\omega((\lambda'(x))^{2}+\lambda(x)\lambda''(x))+\mathcal{F}(\sqrt{\cdot}\partial_{x}^{2}(F_{4}(x,\cdot\lambda(x))))(\omega\lambda(x)^{2})\right) dx\end{split}\end{equation}
Note that, while obtaining \eqref{2a} and \eqref{2b}, we showed that 
$$||\frac{\mathcal{F}(\sqrt{\cdot} F_{4}(x,\cdot \lambda(x)))'(\omega \lambda(x)^{2})}{\omega}\cdot 2 \omega \lambda(x) \lambda'(x)||_{L^{2}(\rho(\omega \lambda(x)^{2}) d\omega)} \leq C \frac{(\log(\log(x)))^{2}}{x^{3} \log^{2+b-2b\alpha}(x)}$$
Similarly, from the procedure used to obtain \eqref{3}, we infer
\begin{equation} ||\frac{\mathcal{F}(\sqrt{\cdot} \partial_{t}(F_{4}(t,\cdot\lambda(t))))(\omega\lambda(t)^{2})}{\omega}||_{L^{2}(\rho(\omega \lambda(t)^{2}) d\omega)} \leq \frac{C (\log(\log(t)))^{2}}{t^{3} \log^{b+1-2\alpha b}(t)}\end{equation}
Next, we consider the term involving $\mathcal{F}(\sqrt{\cdot}F_{4}(x,\cdot\lambda(x)))''$. For this, we start by studying, with $f(x,r)=\sqrt{r} F_{4}(x,r\lambda(x))$,
\begin{equation} \mathcal{F}(f(x))''(\xi) = \int_{0}^{\frac{2}{\sqrt{\xi}}}f(x,r) \partial_{2}^{2}\phi(r,\xi) dr + \int_{\frac{2}{\sqrt{\xi}}}^{\infty} f(x,r) \partial_{2}^{2}\phi(r,\xi) dr\end{equation}
In the region $r \sqrt{\xi} \leq 2$, we use
\begin{equation} \partial_{2}^{2}\phi(r,\xi) = \frac{1}{\sqrt{r}} \sum_{j=2}^{\infty} j(j-1)\xi^{j-2} r^{2j}\phi_{j}(r^{2}), \quad r \sqrt{\xi} \leq 2\end{equation}
Then, 
\begin{equation}\begin{split} &|-4 \int_{t}^{\infty} dx \cos((t-x)\sqrt{\omega}) \lambda(x)^{2}\lambda'(x)^{2} \omega \sum_{j=2}^{\infty} j(j-1)(\omega \lambda(x)^{2})^{j-2} \int_{0}^{\frac{2}{\sqrt{\omega}\lambda(x)}} F_{4}(x,r\lambda(x)) r^{2j} \phi_{j}(r^{2}) dr|\\
&\leq C \sum_{j=2}^{\infty} j(j-1) \omega^{j-1} \int_{t}^{\infty} dx \left(\frac{\lambda'(x)}{\lambda(x)}\right)^{2} \lambda(x)^{2j} \int_{0}^{\frac{2}{\sqrt{\omega}\lambda(x)}}|F_{4}(x,r\lambda(x))| r^{2j} |\phi_{j}(r^{2})| dr\end{split}\end{equation}
we then apply the same procedure as in \eqref{2a}, and get
\begin{equation}\begin{split} &||-4 \int_{t}^{\infty} dx \cos((t-x)\sqrt{\omega}) \lambda(x)^{2}\lambda'(x)^{2} \omega \sum_{j=2}^{\infty} j(j-1)(\omega \lambda(x)^{2})^{j-2} \int_{0}^{\frac{2}{\sqrt{\omega}\lambda(x)}} F_{4}(x,r\lambda(x)) r^{2j} \phi_{j}(r^{2}) dr||_{L^{2}(\rho(\omega \lambda(t)^{2}) d\omega)} \\
&\leq \frac{C (\log(\log(t)))^{2}}{t^{3} \log^{3+b-2b\alpha}(t)}\end{split}\end{equation}
Next, we have
\begin{equation} \partial_{2}^{2}\phi(r,\xi) = (\partial_{2}^{2}\phi(r,\xi))_{0}+(\partial_{2}^{2}\phi(r,\xi))_{1}\end{equation}
with
\begin{equation}\begin{split} (\partial_{2}^{2}\phi(r,\xi))_{0} &= 2 \text{Re}\left(a''(\xi) \psi^{+}(r,\xi) + 2 a'(\xi) \partial_{2}\psi^{+}(r,\xi)\right)\\
&+2 \text{Re}\left(a(\xi)\left(\partial_{\xi}^{2}\left(\frac{e^{i r \sqrt{\xi}}}{\xi^{1/4}} \sigma(r\sqrt{\xi},r)\right)+\frac{r^{2}}{4 \xi^{5/4}}e^{ir\sqrt{\xi}} \sigma(r\sqrt{\xi},r)\right)\right)\end{split}\end{equation}
and
\begin{equation}(\partial_{2}^{2}\phi(r,\xi))_{1}=-2\text{Re}\left(a(\xi) \frac{r^{2}}{4 \xi^{5/4}}e^{ir\sqrt{\xi}}\sigma(r\sqrt{\xi},r)\right)\end{equation}
\begin{equation}|(\partial_{2}^{2}\phi(r,\xi))_{0}| \leq \frac{C r |a(\xi)|}{\xi^{7/4}}\end{equation}
Again, with 
$$f(x,r)=\sqrt{r}F_{4}(x,r\lambda(x))$$
we have
\begin{equation}\begin{split} &\int_{\frac{2}{\sqrt{\xi}}}^{\infty} f(x,r) (\partial_{2}^{2}\phi(r,\xi))_{1} dr = -2 \text{Re}\left(\frac{a(\xi)}{4\xi^{5/4}}\int_{\frac{2}{\sqrt{\xi}}}^{\infty} f(x,r) r^{2} e^{i r \sqrt{\xi}} \sigma(r\sqrt{\xi},r) dr\right)\\
&=\frac{1}{2}\text{Re}\left(\frac{a(\xi)}{\xi^{5/4}} \frac{4 f(x,\frac{2}{\sqrt{\xi}})}{i \xi^{3/2}}\sigma(2,\frac{2}{\sqrt{\xi}}) e^{2i}\right)\\
&+\frac{1}{2} \text{Re}\left(\frac{a(\xi)}{\xi^{5/4}} \int_{\frac{2}{\sqrt{\xi}}}^{\infty} \frac{e^{i r \sqrt{\xi}}}{i\sqrt{\xi}} \left(\partial_{r}f(x,r) r^{2}\sigma(r\sqrt{\xi},r)+2 f(x,r) r \sigma(r\sqrt{\xi},r) + r^{2} f(x,r) \left(\sqrt{\xi} \partial_{1}\sigma(r\sqrt{\xi},r)+\partial_{2}\sigma(r\sqrt{\xi},r)\right)\right)dr\right)\end{split}\end{equation}
Again, using Proposition 4.6 of \cite{kst} to estimate the $\sigma$ terms, we get
\begin{equation}|\int_{\frac{2}{\sqrt{\xi}}}^{\infty} f(x,r) (\partial_{2}^{2}\phi(r,\xi))_{1} dr| \leq \frac{C |a(\xi)|}{\xi^{11/4}} |f(x,\frac{2}{\sqrt{\xi}})| + C \frac{|a(\xi)|}{\xi^{7/4}} \int_{\frac{2}{\sqrt{\xi}}}^{\infty} \left(r|\partial_{r}f(x,r)|+|f(x,r)|\right) rdr\end{equation}
Then,
\begin{equation}\begin{split}&|-\int_{t}^{\infty} \frac{\cos((t-x)\sqrt{\omega})}{\omega}\left(\int_{\frac{2}{\sqrt{\omega}\lambda(x)}}^{\infty} \sqrt{r} F_{4}(x,r\lambda(x)) \left((\partial_{2}^{2}\phi(r,\xi))_{0}+(\partial_{2}^{2}\phi(r,\xi))_{1}\right)\vert_{\xi=\omega \lambda(x)^{2}}dr\right)\cdot 4 \lambda(x)^{2}\lambda'(x)^{2} \omega^{2} dx|\\
&\leq C \int_{t}^{\infty} \left(\int_{\frac{2}{\sqrt{\omega}\lambda(x)}}^{\infty} \left(|F_{4}(x,r\lambda(x))|r^{3/2}+r^{3/2}\cdot r\lambda(x)|\partial_{2}F_{4}(x,r\lambda(x))|\right)dr\right)\frac{|a(\omega \lambda(x)^{2})|}{\omega^{3/4}\lambda(x)^{3/2}}\frac{dx}{x^{2}\log^{2b+2}(x)}\\
&+C \int_{t}^{\infty} \frac{|a(\omega \lambda(x)^{2})|}{\omega^{2}\lambda(x)^{6}} \frac{|F_{4}(x,\frac{2}{\sqrt{\omega}})|}{x^{2}\log^{4b+2}(x)} dx\end{split}\end{equation}

For the last term, we have
\begin{equation}\begin{split} \int_{t}^{\infty} ||\frac{|a(\omega \lambda(x)^{2})|}{\omega^{2}\lambda(x)^{6}} \frac{|F_{4}(x,\frac{2}{\sqrt{\omega}})|}{x^{2}\log^{4b+2}(x)}||_{L^{2}(\rho(\omega \lambda(t)^{2}) d\omega)} dx  &\leq C\int_{t}^{\infty} \frac{\lambda(t)}{\lambda(x)} \frac{1}{x^{2} \log^{2b+2}(x) \lambda(x)^{4}} \left(\int_{0}^{\infty} |F_{4}(x,y)|^{2} y^{8} \frac{dy}{y^{3}}\right)^{1/2} dx\\
&\leq \frac{C \sqrt{\log(\log(t))}}{t^{3} \log^{b+3-2\alpha b}(t)}\end{split}\end{equation}

So,
\begin{equation}||\int_{t}^{\infty} \frac{|a(\omega \lambda(x)^{2})|}{\omega^{2}\lambda(x)^{6}} \frac{|F_{4}(x,\frac{2}{\sqrt{\omega}})|}{x^{2}\log^{4b+2}(x)} dx||_{L^{2}(\rho(\omega \lambda(t)^{2}) d\omega)}\leq \frac{C \sqrt{\log(\log(t))}}{t^{3} \log^{b+3-2\alpha b}(t)}\end{equation}

On the other hand, by the same procedure used to obtain \eqref{2bptwse} and \eqref{2b}, we get
\begin{equation}\begin{split}&||\int_{t}^{\infty} \left(\int_{\frac{2}{\sqrt{\omega}\lambda(x)}}^{\infty} \left(|F_{4}(x,r\lambda(x))|r^{3/2}+r^{3/2}\cdot r\lambda(x)|\partial_{2}F_{4}(x,r\lambda(x))|\right)dr\right)\frac{|a(\omega \lambda(x)^{2})|}{\omega^{3/4}\lambda(x)^{3/2}}\frac{dx}{x^{2}\log^{2b+2}(x)}||_{L^{2}(\rho(\omega \lambda(t)^{2}) d\omega)} \\
&\leq \frac{C \sqrt{\log(\log(t))}}{t^{3} \log^{3+b-2b\alpha}(t)}\end{split}\end{equation}

Next, by noting the symbol-type character of the estimates \eqref{f4estimates}, and inspecting \eqref{2a} and \eqref{2b}, we deduce the following estimate 
\begin{equation}||-\int_{t}^{\infty} \frac{\cos((t-x)\sqrt{\omega})}{\omega}\cdot 2\mathcal{F}(\sqrt{\cdot} \partial_{x}(F_{4}(x,\cdot\lambda(x))))'(\omega \lambda(x)^{2}) \cdot 2\lambda(x) \lambda'(x) \omega dx||_{L^{2}(\rho(\omega \lambda(t)^{2})d\omega)} \leq \frac{C (\log(\log(t)))^{2}}{t^{3} \log^{2+b-2b\alpha}(t)}\end{equation}

The following term was already estimated via \eqref{2a}, \eqref{2b} (except with different $\lambda$-dependent coefficients). Taking into account the estimates on $\lambda',\lambda''$, we get
\begin{equation} ||-\int_{t}^{\infty} \frac{\cos((t-x)\sqrt{\omega})}{\omega} \mathcal{F}(\sqrt{\cdot}F_{4}(x,\cdot\lambda(x)))'(\omega\lambda(x)^{2}) \cdot 2 \omega\left(\lambda'(x)^{2}+\lambda(x)\lambda''(x)\right) dx||_{L^{2}(\rho(\omega \lambda(t)^{2}) d\omega)} \leq \frac{C (\log(\log(t)))^{2}}{t^{3} \log^{2+b-2b\alpha}(t)}\end{equation}

Finally, we start with estimating $\frac{\mathcal{F}(\sqrt{\cdot} \partial_{x}^{2}(F_{4}(x,\cdot\lambda(x))))(\omega \lambda(x)^{2})}{\omega}$. Firstly, we note that $\partial_{x}^{2}(F_{4}(x,r\lambda(x)))$ is still orthogonal to $\phi_{0}(r)$. Then, we repeat the same procedure used to estimate $\frac{\mathcal{F}(\sqrt{\cdot} F_{4}(x,\cdot\lambda(x)))(\omega \lambda(x)^{2})}{\omega}$. By again noting the symbol-type behavior of \eqref{f4estimates}, the only contributions to $\frac{\mathcal{F}(\sqrt{\cdot} \partial_{x}^{2}(F_{4}(x,\cdot\lambda(x))))(\omega \lambda(x)^{2})}{\omega}$ which need to be checked are those due to the last term in \eqref{dttf4estimate}, which is not directly comparable to terms arising in \eqref{f4estimates}. To be clear, we still use the orthogonality of the full function $\partial_{x}^{2}(F_{4}(x,r\lambda(x)))$ to $\phi_{0}(r)$; but, after using the orthogonality as needed, we can then deduce estimates on all integrals which do not involve the last term in \eqref{dttf4estimate} just by comparison with an analogous term in \eqref{f4estimates}, as described above. We start with\\
\\
\textbf{Case 1}: $\frac{2}{\sqrt{\omega} \lambda(x)} \leq 1$.
\begin{equation}\begin{split}&|\frac{-1}{\omega} \int_{0}^{\frac{2}{\sqrt{\omega} \lambda(x)}} \widetilde{\phi_{0}}(r) \sqrt{r} \frac{\mathbbm{1}_{\{r \lambda(x) \leq \frac{x}{2}\}}}{x^{4} \log^{b+N-2}(x)(r^{2}+1)^{2}} dr| \leq \frac{C}{\omega} \int_{0}^{\frac{2}{\sqrt{\omega}\lambda(x)}} \frac{\phi_{0}(r) r dr}{x^{4} \log^{b+N-2}(x) (r^{2}+1)^{2}}\\
&\leq \frac{C}{\omega^{5/2} \lambda(x)^{3} x^{4} \log^{b+N-2}(x)}, \quad \frac{2}{\sqrt{\omega}\lambda(x)} \leq 1\end{split}\end{equation}\\
\\
\textbf{Case 2}: $1 \leq \frac{2}{\sqrt{\omega} \lambda(x)} \leq \frac{x}{2\lambda(x)}$. Here, we first use the orthogonality of $\partial_{x}^{2}(F_{4}(x,r\lambda(x)))$ to $\phi_{0}(r)$. Using the procedure described above, we only need to estimate
\begin{equation} \begin{split} &|\frac{1}{\omega} \int_{\frac{2}{\sqrt{\omega} \lambda(x)}}^{\infty} \widetilde{\phi_{0}}(r) \sqrt{r} \frac{\mathbbm{1}_{\{r \lambda(x) \leq \frac{x}{2}\}}}{x^{4} \log^{b+N-2}(x)(r^{2}+1)^{2}} dr| \leq \frac{C}{\omega} \int_{\frac{2}{\sqrt{\omega} \lambda(x)}}^{\frac{x}{2\lambda(x)}} \frac{dr}{x^{4} \log^{b+N-2}(x) r^{4}}\\
&\leq \frac{C \sqrt{\omega}}{x^{4} \log^{4b+N-2}(x)}, \quad 1 \leq \frac{2}{\sqrt{\omega}\lambda(x)} \leq \frac{x}{2\lambda(x)}\end{split}\end{equation}\\
\\
\textbf{Case 3}: $\frac{x}{2\lambda(x)} \leq \frac{2}{\sqrt{\omega}\lambda(x)}$. In this case, after again using the orthogonality of $\partial_{x}^{2}(F_{4}(x,r\lambda(x)))$ to $\phi_{0}(r)$, the only integral to be checked is zero, due to the support properties of $F_{4}$. Next, we have\\
\\
\textbf{Case 1}: $1 \leq \frac{2}{\sqrt{\omega}\lambda(x)} \leq \frac{x}{2\lambda(x)}$.
\begin{equation} \begin{split} &\frac{1}{\omega} \sum_{j=1}^{\infty} \int_{0}^{\frac{2}{\sqrt{\omega}\lambda(x)}} r^{2j} \omega^{j} \lambda(x)^{2j} |\phi_{j}(r^{2})| \frac{\mathbbm{1}_{\{r \lambda(x) \leq \frac{x}{2}\}}}{x^{4} \log^{b+N-2}(x)(r^{2}+1)^{2}} dr\leq \frac{C}{\omega} \sum_{j=1}^{\infty} \frac{\omega^{j} \lambda(x)^{2j} C_{1}^{j}}{(j-1)!} \int_{0}^{\frac{2}{\sqrt{\omega}\lambda(x)}} \frac{r^{2j} \log(1+r^{2}) dr}{x^{4} \log^{b+N-2}(x) (r^{2}+1)^{2}}\\
&\leq \frac{C}{x^{4} \log^{3b+N-3}(x)}, \quad 1 \leq \frac{2}{\sqrt{\omega}\lambda(x)} \leq \frac{x}{2\lambda(x)}\end{split}\end{equation}
where we treat the integral in exactly the same way we did in obtaining \eqref{1a}\\
\\
\textbf{Case 2}: $\frac{2}{\sqrt{\omega}\lambda(x)} \leq 1$.
\begin{equation} \begin{split} &\frac{1}{\omega} \sum_{j=1}^{\infty} \int_{0}^{\frac{2}{\sqrt{\omega}\lambda(x)}} r^{2j} \omega^{j} \lambda(x)^{2j} |\phi_{j}(r^{2})| \frac{\mathbbm{1}_{\{r \lambda(x) \leq \frac{x}{2}\}}}{x^{4} \log^{b+N-2}(x)(r^{2}+1)^{2}} dr \leq \frac{C}{\omega} \sum_{j=1}^{\infty} \frac{\omega^{j} \lambda(x)^{2j} C_{1}^{j}}{(j-1)! x^{4} \log^{b+N-2}(x)} \int_{0}^{\frac{2}{\sqrt{\omega}\lambda(x)}} r^{2j+2} dr\\
&\leq \frac{C}{\omega^{5/2} \lambda(x)^{3} x^{4} \log^{b+N-2}(x)}, \quad \frac{2}{\sqrt{\omega}\lambda(x)} \leq 1\end{split}\end{equation}\\
\\
\textbf{Case 3}: $\frac{2}{\sqrt{\omega} \lambda(x)} \geq \frac{x}{2\lambda(x)}$. 
\begin{equation}\begin{split}&\frac{1}{\omega} \sum_{j=1}^{\infty} \int_{0}^{\frac{2}{\sqrt{\omega}\lambda(x)}} r^{2j} \omega^{j} \lambda(x)^{2j} |\phi_{j}(r^{2})| \frac{\mathbbm{1}_{\{r \lambda(x) \leq \frac{x}{2}\}}}{x^{4} \log^{b+N-2}(x)(r^{2}+1)^{2}} dr \leq \frac{C}{\omega} \sum_{j=1}^{\infty} \frac{\omega^{j}\lambda(x)^{2j}C_{1}^{j}}{(j-1)!} \int_{0}^{\frac{x}{2\lambda(x)}} \frac{r^{2j} \log(1+r^{2}) dr}{x^{4} \log^{b+N-2}(x) (r^{2}+1)^{2}}\\
&\leq \frac{C}{x^{4} \log^{3b+N-3}(x)}, \quad \frac{2}{\sqrt{\omega} \lambda(x)} \geq \frac{x}{2\lambda(x)}\end{split}\end{equation}\\
\\
Finally, we consider the following integral in multiple cases:\\
\textbf{Case 1}: $\frac{2}{\sqrt{\omega}\lambda(x)} \leq 1$.
\begin{equation}\begin{split} &\frac{1}{\omega} \int_{\frac{2}{\sqrt{\omega}\lambda(x)}}^{\infty} \frac{2 |\text{Re}\left(a(\omega \lambda(x)^{2}) \psi^{+}(r,\omega\lambda(x)^{2})\right)|\sqrt{r}\mathbbm{1}_{\{r \lambda(x) \leq \frac{x}{2}\}}}{x^{4} \log^{b+N-2}(x)(r^{2}+1)^{2}} dr \leq \frac{C}{\omega} \frac{|a(\omega \lambda(x)^{2})|}{\omega^{1/4} \lambda(x)^{1/2}} \int_{\frac{2}{\sqrt{\omega}\lambda(x)}}^{\frac{x}{2\lambda(x)}} \frac{\sqrt{r} dr}{x^{4} \log^{b+N-2}(x)  (r^{2}+1)^{2}}\\
&\leq \frac{C|a(\omega \lambda(x)^{2})|}{\omega^{5/4} \sqrt{\lambda(x)} x^{4} \log^{b+N-2}(x)}, \quad \frac{2}{\sqrt{\omega}\lambda(x)} \leq 1\end{split}\end{equation}\\
\\
\textbf{Case 2}: $1 \leq \frac{2}{\sqrt{\omega}\lambda(x)} \leq \frac{x}{2\lambda(x)}$
\begin{equation}\begin{split}&\frac{1}{\omega} \int_{\frac{2}{\sqrt{\omega}\lambda(x)}}^{\infty} 2 |\text{Re}\left(a(\omega \lambda(x)^{2}) \psi^{+}(r,\omega\lambda(x)^{2})\right)|\sqrt{r}\frac{\mathbbm{1}_{\{r \lambda(x) \leq \frac{x}{2}\}}}{x^{4} \log^{b+N-2}(x)(r^{2}+1)^{2}} dr\leq \frac{C |a(\omega \lambda(x)^{2})|}{\omega^{5/4}\lambda(x)^{1/2}} \int_{\frac{2}{\sqrt{\omega}\lambda(x)}}^{\frac{x}{2\lambda(x)}} \frac{dr}{r^{7/2} x^{4} \log^{b+N-2}(x)}\\
&\leq \frac{C |a(\omega \lambda(x)^{2})|}{x^{4} \log^{3b+N-2}(x)}, \quad 1 \leq \frac{2}{\sqrt{\omega}\lambda(x)} \leq \frac{x}{2\lambda(x)}\end{split}\end{equation}\\
\\
\textbf{Case 3}: $\frac{2}{\sqrt{\omega}\lambda(x)} > \frac{x}{2\lambda(x)}$. In this case, the integral to estimate is zero. Combining all of the above estimates, and using the procedure described above, we get
\begin{equation} ||-\int_{t}^{\infty} \frac{\cos((t-x)\sqrt{\omega})}{\omega} \mathcal{F}(\sqrt{\cdot}\partial_{x}^{2}(F_{4}(x,\cdot\lambda(x))))(\omega \lambda(x)^{2}) dx||_{L^{2}(\rho(\omega \lambda(t)^{2}) d\omega)} \leq \frac{C (\log(\log(t)))^{2}}{t^{3} \log^{b+1-2\alpha b}(t)}\end{equation}
In total, we finally get
\begin{equation}\label{dtsineintf4}||\int_{t}^{\infty} \cos((t-x)\sqrt{\omega}) \mathcal{F}(\sqrt{\cdot} F_{4}(x,\cdot \lambda(x)))(\omega \lambda(x)^{2}) dx||_{L^{2}(\rho(\omega \lambda(t)^{2}) d\omega)}\leq \frac{C (\log(\log(t)))^{2}}{t^{3} \log^{b+1-2\alpha b}(t)}\end{equation}

The next integral to estimate is
\begin{equation}\begin{split}&\lambda(t) \int_{t}^{\infty} \sin((t-x)\sqrt{\omega}) \mathcal{F}(\sqrt{\cdot}F_{4}(x,\cdot\lambda(x)))(\omega \lambda(x)^{2}) dx\\
&= \lambda(t)\left(-\frac{\mathcal{F}(\sqrt{\cdot} F_{4}(t,\cdot \lambda(t)))(\omega \lambda(t)^{2})}{\sqrt{\omega}}-\int_{t}^{\infty} \frac{\cos((t-x)\sqrt{\omega})}{\sqrt{\omega}} \partial_{x}\left(\mathcal{F}(\sqrt{\cdot}F_{4}(x,\cdot\lambda(x)))(\omega \lambda(x)^{2})\right) dx\right)\\
&=\frac{-\lambda(t) \mathcal{F}(\sqrt{\cdot}F_{4}(t,\cdot\lambda(t)))(\omega \lambda(t)^{2})}{\sqrt{\omega}}-\lambda(t)\int_{t}^{\infty} \frac{\sin((t-x)\sqrt{\omega})}{\omega} \partial_{x}^{2}\left(\mathcal{F}(\sqrt{\cdot}F_{4}(x,\cdot\lambda(x)))(\omega\lambda(x)^{2})\right) dx\end{split}\end{equation}
In order to estimate $\frac{-\lambda(t) \mathcal{F}(\sqrt{\cdot}F_{4}(t,\cdot\lambda(t)))(\omega \lambda(t)^{2})}{\sqrt{\omega}}$ it suffices to multiply the pointwise in $\omega$ estimates that we already obtained on $\frac{\mathcal{F}(\sqrt{\cdot}F_{4}(t,\cdot\lambda(t)))(\omega\lambda(t)^{2})}{\omega}$, by $\sqrt{\omega}\lambda(t)$, and then take the $L^{2}(\rho(\omega \lambda(t)^{2})d\omega)$ norm. In fact, we only need to consider the region $\frac{2}{\sqrt{\omega}\lambda(t)} \leq 1$, since, in the other regions, the factor $\sqrt{\omega}\lambda(t)$ that we multiply by, is less than $2$. Doing this procedure for each of the terms appearing in \eqref{1a}, \eqref{1b}, and \eqref{1c}, and combining the resulting estimates, we get
\begin{equation} |\frac{-\lambda(t) \mathcal{F}(\sqrt{\cdot}F_{4}(t,\cdot\lambda(t)))(\omega \lambda(t)^{2})}{\sqrt{\omega}}| \leq \frac{C}{\omega^{5/2} \lambda(t)^{3} t^{2} \log^{1-2\alpha b}(t)} + \frac{C \lambda(t)^{1/2} |a(\omega \lambda(t)^{2})|}{\omega^{3/4} t^{2} \log^{1-2\alpha b}(t)}, \quad \frac{2}{\sqrt{\omega} \lambda(t)} \leq 1\end{equation}
and
\begin{equation}\left(\int_{\frac{4}{\lambda(t)^{2}}}^{\infty} \rho(\omega \lambda(t)^{2}) \frac{d\omega}{\omega^{5} \lambda(t)^{6} t^{4} \log^{2-4\alpha b}(t)}\right)^{1/2} \leq \frac{C}{t^{2} \log^{1+b-2\alpha b}(t)}\end{equation}
\begin{equation} \left(\int_{\frac{4}{\lambda(t)^{2}}}^{\infty} \frac{|a(\omega \lambda(t)^{2})|^{2}}{\omega^{3/2}} \rho(\omega \lambda(t)^{2}) d\omega\right)^{1/2} \cdot \frac{C}{t^{2} \log^{1+\frac{b}{2}-2\alpha b}(t)} \leq \frac{C}{t^{2} \log^{1+b-2\alpha b}(t)}\end{equation}
Then, by the remarks preceeding these estimates, we get
\begin{equation} ||\frac{\lambda(t) \mathcal{F}(\sqrt{\cdot} F_{4}(t,\cdot\lambda(t)))(\omega \lambda(t)^{2})}{\sqrt{\omega}}||_{L^{2}(\rho(\omega \lambda(t)^{2}) d\omega)} \leq \frac{C (\log(\log(t)))^{2}}{t^{2} \log^{1+b-2\alpha b}(t)}\end{equation}
Then, we use the identical procedure as was used to obtain \eqref{dtsineintf4}, and get
\begin{equation} ||\lambda(t) \int_{t}^{\infty} \frac{\sin((t-x)\sqrt{\omega})}{\omega} \partial_{x}^{2}\left(\mathcal{F}(\sqrt{\cdot} F_{4}(x,\cdot\lambda(x)))(\omega \lambda(x)^{2})\right) dx||_{L^{2}(\rho(\omega \lambda(t)^{2}) d\omega)} \leq \frac{C (\log(\log(t)))^{2}}{t^{3} \log^{2b+1-2\alpha b}(t)}\end{equation}
Combining these, we get
\begin{equation} ||\sqrt{\omega}\lambda(t) \int_{t}^{\infty} \frac{\sin((t-x)\sqrt{\omega})}{\sqrt{\omega}} \mathcal{F}(\sqrt{\cdot}F_{4}(x,\cdot\lambda(x)))(\omega \lambda(x)^{2}) dx||_{L^{2}(\rho(\omega \lambda(t)^{2}) d\omega)} \leq \frac{C (\log(\log(t)))^{2}}{t^{2} \log^{1+b-2\alpha b}(t)}\end{equation}
Next, we consider
\begin{equation}\label{sqrtomegadt}\begin{split} &\sqrt{\omega} \lambda(t) \partial_{t}\left(\int_{t}^{\infty} \frac{\sin((t-x)\sqrt{\omega})}{\sqrt{\omega}} \mathcal{F}(\sqrt{\cdot}F_{4}(x,\cdot\lambda(x)))(\omega \lambda(x)^{2}) dx\right) \\
&= \frac{-\lambda(t) \partial_{t}\left(\mathcal{F}(\sqrt{\cdot} F_{4}(t,\cdot \lambda(t)))(\omega \lambda(t)^{2})\right)}{\sqrt{\omega}} - \int_{t}^{\infty} \lambda(t) \cdot \frac{\cos((t-x)\sqrt{\omega})}{\sqrt{\omega}} \partial_{x}^{2}\left(\mathcal{F}(\sqrt{\cdot}F_{4}(x,\cdot\lambda(x)))(\omega\lambda(x)^{2})\right) dx\end{split}\end{equation}
For $\frac{-\lambda(t) \left(\mathcal{F}(\sqrt{\cdot} \partial_{t}(F_{4}(t,\cdot \lambda(t))))(\omega \lambda(t)^{2})\right)}{\sqrt{\omega}}$, we again need only multiply the pointwise estimates for $\frac{\mathcal{F}(\sqrt{\cdot} \partial_{t}(F_{4}(t,\cdot \lambda(t))))(\omega \lambda(t)^{2})}{\omega}$ (which were previously inferred from pointwise estimates for $\frac{\mathcal{F}(\sqrt{\cdot}(F_{4}(t,\cdot \lambda(t))))(\omega \lambda(t)^{2})}{\omega}$ and noting the symbol-type nature of the estimate \eqref{f4estimates}) by $\sqrt{\omega} \lambda(t)$ in the region $\sqrt{\omega} \lambda(t) \geq 2$, and then take the $L^{2}(\rho(\omega \lambda(t)^{2}) d\omega)$ norm. This results in 
\begin{equation} ||-\sqrt{\omega} \lambda(t) \left(\frac{\mathcal{F}(\sqrt{\cdot} \partial_{t}(F_{4}(t,\cdot\lambda(t))))(\omega\lambda(t)^{2})}{\omega}\right)||_{L^{2}(\rho(\omega \lambda(t)^{2}) d\omega)} \leq \frac{C (\log(\log(t)))^{2}}{t^{3} \log^{1+b-2\alpha b}(t)}\end{equation}
Note, however, that doing this same procedure for the term $-\sqrt{\omega}\lambda(t)\left(\frac{\mathcal{F}(\sqrt{\cdot}F_{4}(t,\cdot\lambda(t)))'(\omega \lambda(t)^{2}) \cdot 2 \omega \lambda(t) \lambda'(t)}{\omega}\right)$ would result in an estimate which is not square integrable against the measure $\rho(\omega \lambda(t)^{2}) d\omega$. Instead, we have to integrate by parts in an appropriate $r$ integral, to gain extra decay in $\omega$. In particular, we first make the decomposition as in \eqref{ftransprime}. Then, we can multiply pointwise estimates on the first term in the decomposition \eqref{ftransprime} by $\sqrt{\omega}\lambda(t)$, and proceed as with our previous estimates. This results in a contribution to the overall $L^{2}(\rho(\omega \lambda(t)^{2}) d\omega)$ norm of $-\sqrt{\omega}\lambda(t)\left(\frac{\mathcal{F}(\sqrt{\cdot}F_{4}(t,\cdot\lambda(t)))'(\omega \lambda(t)^{2}) \cdot 2 \omega \lambda(t) \lambda'(t)}{\omega}\right)$ bounded above by: 
\begin{equation} ||-2 \lambda(t)^{2} \lambda'(t) \sqrt{\omega} \left(\int_{0}^{\frac{2}{\sqrt{\omega} \lambda(t)}} \sqrt{r} F_{4}(t,r\lambda(t)) \partial_{2}\phi(r,\omega \lambda(t)^{2}) dr\right)||_{L^{2}(\rho(\omega \lambda(t)^{2}) d\omega)} \leq \frac{C (\log(\log(t)))^{2}}{t^{3} \log^{b+2-2\alpha b}(t)}\end{equation}
 For the second term in the decomposition as in \eqref{ftransprime},namely,
$$ -2 \sqrt{\omega} \lambda(t)^{2}\lambda'(t) \int_{\frac{2}{\sqrt{\omega}\lambda(t)}}^{\infty} \sqrt{r} F_{4}(t,r\lambda(t)) \partial_{2}\phi(r,\omega\lambda(t)^{2}) dr$$
we start with
\begin{equation} \partial_{2}\phi(r,\xi) = (\partial_{2}\phi(r,\xi))_{0}+ (\partial_{2}\phi(r,\xi))_{1} \end{equation}
where
\begin{equation} (\partial_{2}\phi(r,\xi))_{1} = \text{Re}\left(\frac{a(\xi) i r}{\xi^{3/4}} e^{i r \sqrt{\xi}} \sigma(r\sqrt{\xi},r)\right)\end{equation}
and
\begin{equation} |(\partial_{2}\phi(r,\xi))_{0}| \leq \frac{C |a(\xi)|}{\xi^{5/4}}, \quad r \sqrt{\xi} \geq 2\end{equation}
Then, 
\begin{equation} \begin{split} &-2  \lambda(t)^{2}\lambda'(t) \sqrt{\omega} \int_{\frac{2}{\sqrt{\omega}\lambda(t)}}^{\infty} \sqrt{r} F_{4}(t,r\lambda(t)) (\partial_{2}\phi(r,\omega\lambda(t)^{2}))_{1} dr\\
&=-2 \lambda(t)^{2} \lambda'(t) \sqrt{\omega} \text{Re}\left(\frac{a(\omega \lambda(t)^{2}) i}{\omega^{3/4} \lambda(t)^{3/2}} \int_{\frac{2}{\sqrt{\omega}\lambda(t)}}^{\infty} r^{3/2} F_{4}(t,r\lambda(t)) e^{i r \sqrt{\omega}\lambda(t)} \sigma(r\sqrt{\omega}\lambda(t),r) dr\right)\\
&= \frac{-2 \lambda(t)^{1/2} \lambda'(t)}{\omega^{1/4}} \text{Re}\left(\frac{-a(\omega \lambda(t)^{2}) 2^{3/2} F_{4}(t,\frac{2}{\sqrt{\omega}}) \sigma(2,\frac{2}{\sqrt{\omega}\lambda(t)}) e^{2i}}{\omega^{5/4} \lambda(t)^{5/2}}\right)\\
&+\frac{2 \lambda(t)^{1/2} \lambda'(t)}{\omega^{1/4}}\text{Re}\left(a(\omega \lambda(t)^{2}) \int_{\frac{2}{\sqrt{\omega}\lambda(t)}}^{\infty} \frac{e^{i r \sqrt{\omega}\lambda(t)}}{\sqrt{\omega}\lambda(t)} \partial_{r}\left(r^{3/2} F_{4}(t,r\lambda(t)) \sigma(r\sqrt{\omega}\lambda(t),r)\right) dr\right)\end{split}\end{equation}

So, 
\begin{equation}\begin{split} &|-2 \sqrt{\omega} \lambda(t)^{2}\lambda'(t) \int_{\frac{2}{\sqrt{\omega}\lambda(t)}}^{\infty} \sqrt{r} F_{4}(t,r\lambda(t)) \partial_{2}\phi(r,\omega\lambda(t)^{2}) dr|\\
&\leq \frac{C |\lambda'(t)| |a(\omega \lambda(t)^{2})|}{\omega^{3/2}\lambda(t)^{2}} |F_{4}(t,\frac{2}{\sqrt{\omega}})| + \frac{C |\lambda'(t)| |a(\omega \lambda(t)^{2})|}{\lambda(t)^{1/2} \omega^{3/4}} \int_{\frac{2}{\sqrt{\omega}\lambda(t)}}^{\infty} r^{1/2} \left(|F_{4}(t,r\lambda(t))|+r\lambda(t)|\partial_{2}F_{4}(t,r\lambda(t))|\right) dr\end{split}\end{equation}
where we used the estimates on $\sigma$ following from proposition 4.6 of \cite{kst}.\\
\\
This gives
\begin{equation} \begin{split} &|| -2 \sqrt{\omega} \lambda(t)^{2}\lambda'(t) \int_{\frac{2}{\sqrt{\omega}\lambda(t)}}^{\infty} \sqrt{r} F_{4}(t,r\lambda(t)) \partial_{2}\phi(r,\omega\lambda(t)^{2}) dr||_{L^{2}((\frac{4}{\lambda(t)^{2}},\infty),\rho(\omega \lambda(t)^{2})d\omega)}\\
&\leq \frac{C |\lambda'(t)|}{\lambda(t)^{2}} \left(\int_{0}^{\lambda(t)} |F_{4}(t,y)|^{2} y^{3} dy\right)^{1/2} + \frac{C |\lambda'(t)|}{\lambda(t)^{1/2}} \left(\int_{\frac{4}{\lambda(t)^{2}}}^{\infty} \frac{d\omega}{\omega^{3/2}}\left(\int_{0}^{\infty} \frac{r^{3/2}}{t^{2} \log^{1-2\alpha b}(t) (r^{2}+1)^{2}} dr\right)^{2}\right)^{1/2}\\
&\leq \frac{C}{t^{3} \log^{b+2-2\alpha b}(t)}\end{split}\end{equation}

Combining these estimates, and appropriately using \eqref{2b}, for the region $\sqrt{\omega}\lambda(t) \leq 2$, we get
\begin{equation}\label{sqrtomegaf4hatprime}||-\sqrt{\omega}\lambda(t)\left(\frac{\mathcal{F}(\sqrt{\cdot} F_{4}(t,\cdot\lambda(t)))'(\omega\lambda(t)^{2})\cdot 2 \omega \lambda(t)\lambda'(t)}{\omega}\right)||_{L^{2}(\rho(\omega \lambda(t)^{2}) d\omega)} \leq \frac{C (\log(\log(t)))^{2}}{t^{3} \log^{b+2-2b\alpha}(t)}\end{equation}
Now, we start to treat the term inside the $x$ integral in \eqref{sqrtomegadt}. Here, we expand
\begin{equation}\begin{split} &\partial_{x}^{2}\left(\mathcal{F}(\sqrt{\cdot}F_{4}(x,\cdot\lambda(x)))(\omega \lambda(x)^{2})\right) \\
&= \mathcal{F}(\sqrt{\cdot}F_{4}(x,\cdot\lambda(x)))''(\omega \lambda(x)^{2}) \cdot(2 \omega \lambda(x)\lambda'(x))^{2} + 2 \mathcal{F}(\sqrt{\cdot} \partial_{x}(F_{4}(x,\cdot \lambda(x))))'(\omega \lambda(x)^{2}) \cdot 2 \omega \lambda(x)\lambda'(x)\\
&+\mathcal{F}(\sqrt{\cdot} F_{4}(x,\cdot\lambda(x)))'(\omega \lambda(x)^{2}) \cdot 2 \omega ((\lambda'(x))^{2}+\lambda(x)\lambda''(x)) + \mathcal{F}(\sqrt{\cdot}\partial_{x}^{2}(F_{4}(x,\cdot\lambda(x))))(\omega \lambda(x)^{2})\end{split}\end{equation}
We start by considering
\begin{equation}\label{dxf4hatprime}\frac{\sqrt{\omega}\lambda(t)}{\omega} \cdot \left(2 \mathcal{F}(\sqrt{\cdot} \partial_{x}(F_{4}(x,\cdot \lambda(x))))'(\omega \lambda(x)^{2}) \cdot 2 \omega \lambda(x)\lambda'(x)\right)\end{equation}
Recall that the last term we estimated was
$$-\sqrt{\omega}\lambda(t) \cdot 2 \lambda(t)\lambda'(t) \mathcal{F}(\sqrt{\cdot} F_{4}(t,\cdot\lambda(t)))'(\omega \lambda(t)^{2})$$
and we used only \eqref{F4pointwise}, as well as \eqref{rdrf4pointwise}, because we needed to integrate by parts in the $r$ variable in one of the terms. We then repeat the same procedure, with the only difference being 
$$\partial_{1}F_{4}(x,r\lambda(x))+r \lambda'(x) \partial_{2}F_{4}(x,r\lambda(x))\text{ replacing }F_{4}(t,r\lambda(t))$$ 
and 
$$\lambda(t)\partial_{12}F_{4}(x,r\lambda(x))+\lambda'(x) \partial_{2}F_{4}(x,r\lambda(x))+r \lambda(x) \lambda'(x) \partial_{2}^{2}F_{4}(x,r\lambda(x))\text{ replacing } \lambda(x) \partial_{2}F_{4}(x,r\lambda(x))$$ 
By noting the symbol-type nature of the estimate \eqref{f4estimates}, we get
\begin{equation} \begin{split} &||\sqrt{\omega}\lambda(t) \cdot 4 \lambda(x) \lambda'(x) \mathcal{F}(\sqrt{\cdot} \partial_{x}(F_{4}(x,\cdot \lambda(x))))'(\omega \lambda(x)^{2})||_{L^{2}(\rho(\omega \lambda(t)^{2}) d\omega)}\\
&\leq C \left(\frac{\lambda(t)}{\lambda(x)}\right)^{2} ||\sqrt{\omega}\lambda(x) \cdot 4 \lambda(x) \lambda'(x) \mathcal{F}(\sqrt{\cdot} \partial_{x}(F_{4}(x,\cdot \lambda(x))))'(\omega \lambda(x)^{2})||_{L^{2}(\rho(\omega \lambda(x)^{2}) d\omega)}\\
&\leq C \left(\frac{\lambda(t)}{\lambda(x)}\right)^{2} \frac{(\log(\log(x)))^{2}}{x^{4} \log^{b+2-2b\alpha}(x)}\end{split}\end{equation}
and this gives
\begin{equation} ||-\lambda(t) \int_{t}^{\infty} \frac{\cos((t-x)\sqrt{\omega})}{\sqrt{\omega}} \left(2 \mathcal{F}(\sqrt{\cdot} \partial_{x}(F_{4}(x,\cdot \lambda(x))))'(\omega \lambda(x)^{2}) \cdot 2 \omega \lambda(x)\lambda'(x)\right) dx||_{L^{2}(\rho(\omega \lambda(t)^{2})d\omega)} \leq \frac{C (\log(\log(t)))^{2}}{t^{3} \log^{b+2-2b\alpha}(t)}\end{equation}
Next, we consider
\begin{equation} \frac{\mathcal{F}(\sqrt{\cdot}F_{4}(x,\cdot\lambda(x)))'(\omega \lambda(x)^{2})\cdot 2 \omega ((\lambda'(x))^{2}+\lambda(x)\lambda''(x))\cdot \sqrt{\omega}\lambda(t)}{\omega}\end{equation}
We treat this term identically to how \eqref{sqrtomegaf4hatprime} was treated, noting that the only difference between the two terms is a coefficient which depends on absolute constants and $\lambda$. We therefore get
\begin{equation}\begin{split} &||-\lambda(t) \int_{t}^{\infty} \frac{\cos((t-x)\sqrt{\omega})}{\sqrt{\omega}} \left(\mathcal{F}(\sqrt{\cdot}F_{4}(x,\cdot\lambda(x)))'(\omega \lambda(x)^{2})\cdot 2 \omega ((\lambda'(x))^{2}+\lambda(x)\lambda''(x))\right) dx||_{L^{2}(\rho(\omega \lambda(t)^{2})d\omega)}\\
&\leq \int_{t}^{\infty} dx \frac{C}{x} \left(\frac{\lambda(t)}{\lambda(x)}\right)^{2} \frac{(\log(\log(x)))^{2}}{x^{3} \log^{b+2-2b\alpha}(x)}\\
&\leq \frac{C (\log(\log(t)))^{2}}{t^{3} \log^{b+2-2b\alpha}(t)}\end{split}\end{equation}

Next, we study 
$$\frac{\sqrt{\omega}\lambda(t)}{\omega} \cdot \mathcal{F}(\partial_{x}^{2}\left(\sqrt{\cdot} F_{4}(x,\cdot\lambda(x))\right))(\omega\lambda(x)^{2})$$
by multiplying our previous pointwise in $\omega$ estimates on $\frac{\mathcal{F}(\partial_{x}^{2}\left(\sqrt{\cdot}F_{4}(x,\cdot\lambda(x))\right))(\omega\lambda(x)^{2})}{\omega}$ by $\left(\frac{\lambda(t)}{\lambda(x)}\right)\sqrt{\omega} \lambda(x)$. We need only check the contributions to $ \left(\frac{\lambda(t)}{\lambda(x)}\right) \cdot ||\frac{\sqrt{\omega}\lambda(x)}{\omega} \cdot \mathcal{F}(\partial_{x}^{2}\left(\sqrt{\cdot} F_{4}(x,\cdot\lambda(x))\right))(\omega\lambda(x)^{2})||_{L^{2}(\rho(\omega \lambda(t)^{2}) d\omega)}$ coming from the region $\sqrt{\omega}\lambda(x) \geq 2$, just as for previous terms:\\
The integrals to check are:
\begin{equation}\begin{split} &\left(\frac{\lambda(t)}{\lambda(x)}\right)\left(\int_{\frac{4}{\lambda(x)^{2}}}^{\infty} \frac{\rho(\omega \lambda(t)^{2}) d\omega}{\omega^{4} \lambda(x)^{4} x^{8} \log^{2b+2N-4}(x)}\right)^{1/2}\leq C \left(\frac{\lambda(t)}{\lambda(x)}\right)^{2}\left(\int_{\frac{4}{\lambda(x)^{2}}}^{\infty} \frac{\rho(\omega \lambda(x)^{2}) d\omega}{\omega^{4} \lambda(x)^{4} x^{8} \log^{2b+2N-4}(x)}\right)^{1/2}\\
&\leq C\left(\frac{\lambda(t)}{\lambda(x)}\right)^{2} \frac{1}{x^{4} \log^{2b+N-2}(x)}\end{split}\end{equation}
\begin{equation}\begin{split}&\left(\frac{\lambda(t)}{\lambda(x)}\right)^{2}\left(\int_{\frac{4}{\lambda(x)^{2}}}^{\infty} \frac{\omega \lambda(x)^{2} d\omega}{\omega^{5}}\right)^{1/2} \frac{1}{\lambda(x)^{3}x^{4}\log^{1-2\alpha b}(x)}\leq C \left(\frac{\lambda(t)}{\lambda(x)}\right)^{2} \frac{1}{x^{4} \log^{b+1-2\alpha b}(x)}\end{split}\end{equation}
and
\begin{equation} \begin{split} &\left(\frac{\lambda(t)}{\lambda(x)}\right)^{2} \left(\int_{\frac{4}{\lambda(x)^{2}}}^{\infty} \frac{d\omega}{\omega^{3/2} x^{8} \log^{b+2-4\alpha b}(x)}\right)^{1/2} \leq C \left(\frac {\lambda(t)}{\lambda(x)}\right)^{2} \frac{1}{x^{4} \log^{b+1-2\alpha b}(x)}\end{split}\end{equation}
In total, we get
\begin{equation} ||\sqrt{\omega} \lambda(t) \frac{\mathcal{F}(\partial_{x}^{2}(\sqrt{\cdot} F_{4}(x,\cdot\lambda(x))))(\omega\lambda(x)^{2})}{\omega}||_{L^{2}(\rho(\omega \lambda(t)^{2})d\omega)}\leq C \left(\frac{\lambda(t)}{\lambda(x)}\right)^{2} \frac{(\log(\log(x)))^{2}}{x^{4} \log^{1+b-2\alpha b}(x)}\end{equation}
and this gives
\begin{equation} ||\int_{t}^{\infty} \cos((t-x)\sqrt{\omega})  \mathcal{F}(\sqrt{\cdot}\partial_{x}^{2}(F_{4}(x,\cdot \lambda(x))))(\omega \lambda(x)^{2}) \frac{\sqrt{\omega}\lambda(t)}{\omega} dx ||_{L^{2}(\rho(\omega \lambda(t)^{2})d\omega)}\leq \frac{C (\log(\log(t)))^{2}}{t^{3} \log^{1+b-2\alpha b}(t)}\end{equation}
The last term we need to consider here is 
$$\frac{\sqrt{\omega}\lambda(t)}{\omega}\mathcal{F}(\sqrt{\cdot}F_{4}(x,\cdot\lambda(x)))''(\omega \lambda(x)^{2}) \cdot\left(2\omega \lambda(x)\lambda'(x)\right)^{2}$$
We already estimated 
$$\frac{\mathcal{F}(\sqrt{\cdot}F_{4}(x,\cdot\lambda(x)))''(\omega \lambda(x)^{2}) \cdot\left(2\omega \lambda(x)\lambda'(x)\right)^{2}}{\omega}$$
So, we need only prove new estimates in the region $\sqrt{\omega}\lambda(x) \geq 2$ (by writing $\sqrt{\omega}\lambda(t) = \left(\frac{\lambda(t)}{\lambda(x)}\right) \cdot \sqrt{\omega} \lambda(x)$).\\
So, we again write
\begin{equation}\label{f4hatdoubleprimesetup}\begin{split} &\mathcal{F}(\sqrt{\cdot}F_{4}(x,\cdot\lambda(x)))''(\omega \lambda(x)^{2})\\
&=\int_{0}^{\frac{2}{\sqrt{\omega}\lambda(x)}} \sqrt{r} F_{4}(x,r\lambda(x)) \partial_{2}^{2}\phi(r,\omega\lambda(x)^{2}) dr + \int_{\frac{2}{\sqrt{\omega}\lambda(x)}}^{\infty} \sqrt{r} F_{4}(x,r\lambda(x)) \partial_{2}^{2}\phi(r,\omega\lambda(x)^{2}) dr\end{split}\end{equation}
and for the first term on the second line of \eqref{f4hatdoubleprimesetup}, we simply multiply our previous pointwise in $\omega$ estimate on $\left(\int_{0}^{\frac{2}{\sqrt{\omega}\lambda(x)}} \sqrt{r} F_{4}(x,r\lambda(x)) \partial_{2}^{2}\phi(r,\omega\lambda(x)^{2}) dr\right) \cdot\frac{\left(2\omega \lambda(x)\lambda'(x)\right)^{2}}{\omega}$ by $\frac{\lambda(t)}{\lambda(x)} \sqrt{\omega}\lambda(x)$ in the region $\sqrt{\omega}\lambda(x) \geq 2$, and estimate the $L^{2}(\rho(\omega\lambda(t)^{2}) d\omega)$ norm, as before. We have
\begin{equation} \sqrt{\omega}\lambda(x)\cdot \frac{\omega^{2}\lambda(x)^{2}\lambda'(x)^{2}}{\omega}\cdot|\int_{0}^{\frac{2}{\sqrt{\omega}\lambda(x)}} \sqrt{r} F_{4}(x,r\lambda(x)) \partial_{2}^{2}\phi(r,\omega\lambda(x)^{2}) dr| \leq \frac{C}{x^{4} \log^{3-2b\alpha}(x)\omega^{5/2} \lambda(x)^{3}}, \quad \sqrt{\omega}\lambda(x) \geq 2\end{equation}

and this leads to
\begin{equation}||\sqrt{\omega}\lambda(t)\left(\frac{\int_{0}^{\frac{2}{\sqrt{\omega}\lambda(x)}} \sqrt{r} F_{4}(x,r\lambda(x)) \partial_{2}^{2}\phi(r,\omega\lambda(x)^{2}) dr}{\omega} \cdot \omega^{2}\lambda(x)^{2}(\lambda'(x))^{2}\right)||_{L^{2}(\rho(\omega \lambda(t)^{2})d\omega)} \leq C \left(\frac{\lambda(t)^{2}}{\lambda(x)^{2}}\right)\frac{\log(\log(x)))^{2}}{x^{4}\log^{3+b-2b\alpha}(x)}\end{equation}

For the second term on the second line of \eqref{f4hatdoubleprimesetup}, we can not simply multiply our pointwise in $\omega$ estimates on $\left(\int_{\frac{2}{\sqrt{\omega}\lambda(x)}}^{\infty} \sqrt{r} F_{4}(x,r\lambda(x)) \partial_{2}^{2}\phi(r,\omega\lambda(x)^{2}) dr\right) \cdot\frac{\left(2\omega \lambda(x)\lambda'(x)\right)^{2}}{\omega}$ by $\sqrt{\omega}\lambda(t)$, since doing so would result in an estimate that is not square integrable against the measure $\rho(\omega\lambda(t)^{2}) d\omega$. So, we have to integrate by parts in appropriate $r$ integrals, just as in a previous situation. To be precise, we consider
\begin{equation}\label{new2bterm} \int_{\frac{2}{\sqrt{\omega}\lambda(x)}}^{\infty} \sqrt{r} F_{4}(x,r\lambda(x)) \partial_{2}^{2}\phi(r,\omega\lambda(x)^{2}) dr\end{equation}
and write, for $r \sqrt{\xi} \geq 2$:
\begin{equation}\label{d2phiseconddecomp}\begin{split} \partial_{2}^{2}\phi(r,\xi)&= 2\text{Re}\left(\left(\frac{a''(\xi)}{\xi^{1/4}} \sigma(r\sqrt{\xi},r) + 2 a'(\xi) \left(\frac{-\sigma(r\sqrt{\xi},r)}{4\xi^{5/4}}+\frac{r\partial_{1}\sigma(r\sqrt{\xi},r)}{2\xi^{3/4}}\right)\right.\right.\\
&+\left.\left.a(\xi)\left(\frac{5 \sigma(r\sqrt{\xi},r)}{16 \xi^{9/4}}-\frac{r\partial_{1}\sigma(r\sqrt{\xi},r)}{2\xi^{7/4}}+\frac{i r^{2} \partial_{1}\sigma(r\sqrt{\xi},r)}{2 \xi^{5/4}}+\frac{r^{2} \partial_{1}^{2}\sigma(r\sqrt{\xi},r)}{4\xi^{5/4}}\right)\right)e^{ir\sqrt{\xi}}\right)\\
&+2 \text{Re}\left(\left(2 \frac{a'(\xi) i r \sigma(r\sqrt{\xi},r)}{2 \xi^{3/4}}-\frac{i r a(\xi) \sigma(r\sqrt{\xi},r)}{2\xi^{7/4}}\right)e^{i r \sqrt{\xi}}\right)\\
&-2 \text{Re}\left(\frac{a(\xi) r^{2}}{4 \xi^{5/4}} e^{i r \sqrt{\xi}} \sigma(r\sqrt{\xi},r)\right)\end{split}\end{equation} 
For each term on the third line of \eqref{d2phiseconddecomp}, inserted into \eqref{new2bterm}, we will integrate by parts once in $r$. The insertion of the first term on the third line of \eqref{d2phiseconddecomp} gives
\begin{equation}\label{2atildesetup}\begin{split} &2 \text{Re}\left(\int_{\frac{2}{\sqrt{\omega}\lambda(x)}}^{\infty} \sqrt{r} F_{4}(x,r\lambda(x)) \cdot \frac{ a'(\omega \lambda(x)^{2}) i r}{ \omega^{3/4} \lambda(x)^{3/2}}\sigma(r\sqrt{\omega}\lambda(x),r) e^{i r \sqrt{\omega}\lambda(x)} dr\right)\\
&=2 \text{Re}\left(\frac{i a'(\omega \lambda(x)^{2})}{\omega^{3/4}\lambda(x)^{3/2}}\left(\frac{-2^{3/2} F_{4}(x,\frac{2}{\sqrt{\omega}})}{\omega^{3/4} \lambda(x)^{3/2} i \sqrt{\omega} \lambda(x)} \sigma(2,\frac{2}{\sqrt{\omega}\lambda(x)})e^{2i}-\int_{\frac{2}{\sqrt{\omega}\lambda(x)}}^{\infty} \frac{e^{i r \sqrt{\omega}\lambda(x)}}{i \sqrt{\omega}\lambda(x)} \partial_{r}\left(r^{3/2} F_{4}(x,r\lambda(x))\sigma(r\sqrt{\omega}\lambda(x),r)\right) dr\right)\right)\end{split}\end{equation}
Then we estimate each term seperately:
\begin{equation}\begin{split} &|2 \text{Re}\left(\frac{i a'(\omega \lambda(x)^{2})}{\omega^{3/4}\lambda(x)^{3/2}}\left(\frac{-2^{3/2} F_{4}(x,\frac{2}{\sqrt{\omega}})}{\omega^{3/4} \lambda(x)^{3/2} i \sqrt{\omega} \lambda(x)} \sigma(2,\frac{2}{\sqrt{\omega}\lambda(x)})e^{2i}\right)\right)|\\
&\leq C \frac{|a'(\omega \lambda(x)^{2})|}{\omega^{3/4}\lambda(x)^{3/2}} \frac{|F_{4}(x,\frac{2}{\sqrt{\omega}})|}{\omega^{3/4}\lambda(x)^{3/2} \sqrt{\omega}\lambda(x)}\end{split}\end{equation}
The contribution to the $L^{2}((\frac{4}{\lambda(x)^{2}},\infty),\rho(\omega \lambda(t)^{2}) d\omega)$ norm of the integrand of \eqref{sqrtomegadt} due to the above term is thus estimated by
\begin{equation}\begin{split}&\left(\int_{\frac{4}{\lambda(x)^{2}}}^{\infty} \left(\frac{|a'(\omega \lambda(x)^{2})|}{\omega^{3/4}\lambda(x)^{3/2}} \frac{|F_{4}(x,\frac{2}{\sqrt{\omega}})| }{\omega^{3/4}\lambda(x)^{3/2} \sqrt{\omega}\lambda(x)}\right)^{2} \omega^{4} \cdot \left(\frac{\sqrt{\omega} \lambda(t)}{\omega}\right)^{2} \frac{\rho(\omega \lambda(t)^{2}) d\omega}{\left(x^{2} \log^{4b+2}(x)\right)^{2}}\right)^{1/2}\\
&\leq C \left(\frac{\lambda(t)}{\lambda(x)}\right)^{2} \left(\int_{0}^{\lambda(x)} y^{6} |F_{4}(x,y)|^{2} \frac{dy}{y^{3}}\right)^{1/2} \frac{1}{\lambda(x)^{5} x^{2} \log^{4b+2}(x)}\\
&\leq C \left(\frac{\lambda(t)}{\lambda(x)}\right)^{2} \frac{1}{x^{4} \log^{b+3-2b\alpha}(x)}\end{split}\end{equation}
For the next term of \eqref{2atildesetup}, we use Proposition 4.6 of \cite{kst} again, to get
\begin{equation}\begin{split}&|2 \text{Re}\left(\frac{i a'(\omega \lambda(x)^{2})}{\omega^{3/4}\lambda(x)^{3/2}}\left(-\int_{\frac{2}{\sqrt{\omega}\lambda(x)}}^{\infty} \frac{e^{i r \sqrt{\omega}\lambda(x)}}{i \sqrt{\omega}\lambda(x)} \partial_{r}\left(r^{3/2} F_{4}(x,r\lambda(x))\sigma(r\sqrt{\omega}\lambda(x),r)\right) dr\right)\right)|\\
&\leq \frac{C |a(\omega \lambda(x)^{2})|}{\omega^{7/4}\lambda(x)^{7/2}}\frac{1}{\sqrt{\omega}\lambda(x)} \int_{\frac{2}{\sqrt{\omega}\lambda(x)}}^{\infty} \sqrt{r} \left(|F_{4}(x,r\lambda(x))|+r\lambda(x)|(\partial_{2}F_{4})(x,r\lambda(x))|\right)dr\\
&\leq \frac{C |a(\omega \lambda(x)^{2})|}{\omega^{9/4} \lambda(x)^{9/2} x^{2} \log^{1-2b\alpha}(x)}\end{split}\end{equation}
Then, the contribution of this term to the $L^{2}((\frac{4}{\lambda(x)^{2}},\infty),\rho(\omega \lambda(t)^{2}) d\omega)$ norm of the  integrand of \eqref{sqrtomegadt} is estimated by:
\begin{equation}\begin{split} &\left(\int_{\frac{4}{\lambda(x)^{2}}}^{\infty} \frac{d\omega}{\omega^{9/2}} \omega^{4} \left(\frac{\sqrt{\omega}\lambda(x)}{\omega}\right)^{2}\right)^{1/2} \left(\frac{\lambda(t)}{\lambda(x)}\right)^{2} \frac{1}{x^{2}\log^{1-2b\alpha}(x)} \cdot \frac{1}{\lambda(x)^{9/2}} \cdot \frac{1}{x^{2} \log^{4b+2}(x)}\\
&\leq C \left(\frac{\lambda(t)}{\lambda(x)}\right)^{2} \frac{1}{x^{4} \log^{b+3-2b\alpha}(x)}\end{split}\end{equation}
By comparing the first and second terms on the third line of \eqref{d2phiseconddecomp}, and recalling the symbol-type estimates on $a$ from Proposition 4.7 of \cite{kst}, we can estimate the second term on the third line of \ref{d2phiseconddecomp} by repeating the same exact procedure used to estimate the first term  of the third line of \eqref{d2phiseconddecomp}. \\
\\
Next, we treat the fourth line of \eqref{d2phiseconddecomp}. Here, we integrate by parts twice in the $r$ integral resulting from substitution of the fourth line of \eqref{d2phiseconddecomp} into \eqref{new2bterm}. With $\xi=\omega \lambda(t)^{2}$, we have
\begin{equation}\begin{split}&\int_{\frac{2}{\sqrt{\xi}}}^{\infty} \sqrt{r} F_{4}(x,r\lambda(x)) 2 \text{Re}\left(\frac{-a(\xi) r^{2} e^{i r \sqrt{\xi}} \sigma(r\sqrt{\xi},r)}{4\xi^{5/4}}\right) dr\\
&=\text{Re}\left(\frac{-a(\xi)}{2\xi^{5/4}}\left(\frac{-2^{5/2} F_{4}(x,\frac{2}{\sqrt{\xi}}\lambda(x))\sigma(2,\frac{2}{\sqrt{\xi}})e^{2i}}{i \xi^{7/4}}+\frac{1}{i \sqrt{\xi}} \cdot \left(\partial_{r}\left(r^{5/2} F_{4}(x,r\lambda(x)) \sigma(r\sqrt{\xi},r)\right) \frac{e^{i r \sqrt{\xi}}}{i \sqrt{\xi}}\right)\vert_{r=\frac{2}{\sqrt{\xi}}}\right)\right)\\
&+\text{Re}\left(\frac{-a(\xi)}{2\xi^{5/4}}\cdot \frac{-1}{\xi}\int_{\frac{2}{\sqrt{\xi}}}^{\infty} e^{i r \sqrt{\xi}}\partial_{r}^{2}\left(r^{5/2} F_{4}(x,r\lambda(x))\sigma(r\sqrt{\xi},r)\right) dr\right)\end{split}\end{equation}

Next, we again note the symbol-type character of the estimate \eqref{f4estimates}, to get
\begin{equation} \label{3tildeintermediate}\begin{split}&|\int_{\frac{2}{\sqrt{\xi}}}^{\infty} \sqrt{r} F_{4}(x,r\lambda(x)) 2 \text{Re}\left(\frac{-a(\xi) r^{2} e^{i r \sqrt{\xi}} \sigma(r\sqrt{\xi},r)}{4\xi^{5/4}}\right) dr|\cdot 4 \omega^{2} \lambda(x)^{2}\lambda'(x)^{2} \frac{\sqrt{\omega} \lambda(x)}{\omega}\\
&\leq \frac{C \omega^{2} \lambda(x)}{\sqrt{\omega}x^{2}\log^{4b+2}(x)} \left(\frac{|a(\omega \lambda(x)^{2})|}{\omega^{3}\lambda(x)^{6}} F_{4,est}(x,\frac{2 \lambda(x)}{\sqrt{\xi}})+ \frac{|a(\omega \lambda(x)^{2})|}{\omega^{9/4} \lambda(x)^{9/2}} \int_{\frac{2}{\sqrt{\omega}\lambda(x)}}^{\infty} r^{1/2} F_{4,est}(x,r\lambda(x)) dr\right)\end{split}\end{equation}
where $F_{4,est}$ is the expression which appears on the right-hand side of \eqref{f4estimates}. Then, the same procedure used to treat \eqref{2atildesetup} also applies to treat \eqref{3tildeintermediate}, and we get
\begin{equation}\begin{split}&||\int_{\frac{2}{\sqrt{\xi}}}^{\infty} \sqrt{r} F_{4}(x,r\lambda(x)) 2 \text{Re}\left(\frac{-a(\xi) r^{2} e^{i r \sqrt{\xi}} \sigma(r\sqrt{\xi},r)}{4\xi^{5/4}}\right) dr\cdot 4 \omega^{2} \lambda(x)^{2}\lambda'(x)^{2} \frac{\sqrt{\omega} \lambda(x)}{\omega}\left(\frac{\lambda(t)}{\lambda(x)}\right)||_{L^{2}((\frac{4}{\lambda(x)^{2}},\infty),\rho(\omega \lambda(t)^{2}) d\omega)}\\
&\leq C \left(\frac{\lambda(t)}{\lambda(x)}\right)^{2} \frac{1}{x^{4} \log^{b+3-2b\alpha}(x)}\end{split}\end{equation}
Next, we study the first two lines of \eqref{d2phiseconddecomp}, which are given below:
\begin{equation}\begin{split}(\partial_{2}^{2}\phi(r,\xi))_{2}&=2\text{Re}\left(\left(\frac{a''(\xi)}{\xi^{1/4}} \sigma(r\sqrt{\xi},r) + 2 a'(\xi) \left(\frac{-\sigma(r\sqrt{\xi},r)}{4\xi^{5/4}}+\frac{r\partial_{1}\sigma(r\sqrt{\xi},r)}{2\xi^{3/4}}\right)\right.\right.\\
&+\left.\left.a(\xi)\left(\frac{5 \sigma(r\sqrt{\xi},r)}{16 \xi^{9/4}}-\frac{r\partial_{1}\sigma(r\sqrt{\xi},r)}{2\xi^{7/4}}+\frac{i r^{2} \partial_{1}\sigma(r\sqrt{\xi},r)}{2 \xi^{5/4}}+\frac{r^{2} \partial_{1}^{2}\sigma(r\sqrt{\xi},r)}{4\xi^{5/4}}\right)\right)e^{ir\sqrt{\xi}}\right)\end{split}\end{equation} We note that
\begin{equation}|(\partial_{2}^{2}\phi(r,\xi))_{2}| \leq \frac{C |a(\xi)|}{\xi^{9/4}}, \quad r \geq \frac{2}{\sqrt{\xi}}\end{equation}
So,
\begin{equation} \begin{split}|\int_{\frac{2}{\sqrt{\omega}\lambda(x)}}^{\infty} \sqrt{r} F_{4}(x,r\lambda(x)) (\partial_{2}^{2}\phi(r,\omega\lambda(x)^{2}))_{2} dr|&\leq \frac{C |a(\omega\lambda(x)^{2})|}{\omega^{9/4}\lambda(x)^{9/2}} \int_{\frac{2}{\sqrt{\omega}\lambda(x)}}^{\infty} \sqrt{r}|F_{4}(x,r\lambda(x))| dr\end{split}\end{equation}
Finally,
\begin{equation}\begin{split}&\left(\int_{\frac{4}{\lambda(x)^{2}}}^{\infty} \rho(\omega \lambda(t)^{2}) \left(\left(\int_{\frac{2}{\sqrt{\omega}\lambda(x)}}^{\infty} \sqrt{r} F_{4}(x,r\lambda(x)) (\partial_{2}^{2}\phi(r,\omega\lambda(x)^{2}))_{2} dr\right) \cdot 4 \omega^{2} \lambda(x)^{2} \lambda'(x)^{2} \frac{\sqrt{\omega} \lambda(t)}{\omega}\right)^{2} d\omega\right)^{1/2} \\
&\leq C \left(\frac{\lambda(t)}{\lambda(x)}\right)^{2} \frac{1}{x^{4} \log^{b+3-2b\alpha}(x)}\end{split}\end{equation}

Combining all of our estimates, we get
\begin{equation}||\sqrt{\omega}\lambda(t)\left(\frac{\mathcal{F}(\sqrt{\cdot} F_{4}(x,\cdot\lambda(x)))''(\omega\lambda(x)^{2})}{\omega} \omega^{2} \lambda(x)^{2}\lambda'(x)^{2}\right)||_{L^{2}(\rho(\omega\lambda(t)^{2})d\omega)} \leq C\left(\frac{\lambda(t)}{\lambda(x)}\right)^{2} \frac{\log(\log(x))^{2}}{x^{4} \log^{3+b-2b\alpha}(x)}\end{equation}
and
\begin{equation}\label{dxsquaredintegrand} ||-\int_{t}^{\infty} \cos((t-x)\sqrt{\omega}) \frac{\sqrt{\omega}\lambda(t)}{\omega} \partial_{x}^{2}\left(\mathcal{F}(\sqrt{\cdot}F_{4}(x,\cdot\lambda(x)))(\omega\lambda(x)^{2})\right) dx||_{L^{2}(\rho(\omega \lambda(t)^{2}) d\omega)} \leq \frac{C (\log(\log(t)))^{2}}{t^{3} \log^{1+b-2\alpha b}(t)}\end{equation}

which imply
\begin{equation} ||\sqrt{\omega} \lambda(t) \partial_{t}\left(\int_{t}^{\infty} \frac{\sin((t-x)\sqrt{\omega})}{\sqrt{\omega}} \mathcal{F}(\sqrt{\cdot}F_{4}(x,\cdot\lambda(x)))(\omega \lambda(x)^{2}) dx\right)||_{L^{2}(\rho(\omega\lambda(t)^{2})d\omega)} \leq \frac{C (\log(\log(t)))^{2}}{t^{3} \log^{1+b-2\alpha b}(t)} \end{equation}

The last quantity to estimate is
\begin{equation}\label{omegalambdasquared}\begin{split} &\omega\lambda(t)^{2}\left(\int_{t}^{\infty} \frac{\sin((t-x)\sqrt{\omega}}{\sqrt{\omega}} \mathcal{F}(\sqrt{\cdot}F_{4}(x,\cdot\lambda(x)))(\omega\lambda(x)^{2}) dx\right)\\
&=-\lambda(t)^{2}\mathcal{F}(\sqrt{\cdot}F_{4}(t,\cdot\lambda(t)))(\omega\lambda(t)^{2}) - \sqrt{\omega}\lambda(t)^{2}\int_{t}^{\infty} \frac{\sin((t-x)\sqrt{\omega})}{\omega} \partial_{x}^{2}\left(\mathcal{F}(\sqrt{\cdot}F_{4}(x,\cdot\lambda(x)))(\omega\lambda(x)^{2})\right) dx\end{split}\end{equation}

We can estimate the second term on the second line of \eqref{omegalambdasquared}, simply by using the estimates which gave \eqref{dxsquaredintegrand}. For the first term of the second line of \eqref{omegalambdasquared}, we have, with $\xi = \omega \lambda(t)^{2}$,
\begin{equation}\begin{split} &\lambda(t)^{2}\left(\int_{0}^{\infty} \rho(\omega \lambda(t)^{2}) \left(\mathcal{F}(\sqrt{\cdot} F_{4}(t,\cdot \lambda(t)))(\omega \lambda(t)^{2})\right)^{2} d\omega\right)^{1/2} = \lambda(t)^{2} \left(\int_{0}^{\infty} \rho(\xi) \left(\mathcal{F}(\sqrt{\cdot} F_{4}(t,\cdot \lambda(t)))(\xi)\right)^{2} \frac{d\xi}{\lambda(t)^{2}}\right)^{1/2}\\
&=\lambda(t) \left(\int_{0}^{\infty} r \left(F_{4}(t,r\lambda(t))\right)^{2} dr\right)^{1/2}\\
&\leq C \lambda(t) \left(\int_{0}^{\frac{\log^{N}(t)}{\lambda(t)}} \frac{r^{3} \lambda(t)^{2} dr}{\lambda(t)^{8} (r^{2}+1)^{4} t^{4} \log^{6b+2-4b\alpha}(t)} + \int_{0}^{\frac{t}{2\lambda(t)}} \frac{r^{3} \lambda(t)^{2} dr}{\lambda(t)^{8} (r^{2}+1)^{4} t^{4} \log^{10b+4N-4}(t)} \right)^{1/2}\\
&\leq \frac{C}{t^{2} \log^{b+1-2b\alpha}(t)} + \frac{C}{t^{2} \log^{3b+2N-2}(t)}\\
&\leq \frac{C}{t^{2} \log^{b+1-2b\alpha}(t)}\end{split}\end{equation}

So, 
\begin{equation}||\omega \lambda(t)^{2} \int_{t}^{\infty} \frac{\sin((t-x)\sqrt{\omega})}{\sqrt{\omega}} \mathcal{F}(\sqrt{\cdot} F_{4}(x,\cdot\lambda(x)))(\omega \lambda(x)^{2}) dx||_{L^{2}(\rho(\omega \lambda(t)^{2}) d\omega)}\leq \frac{C}{t^{2} \log^{1+b-2b\alpha}(t)}\end{equation}
which finishes the proof of the lemma \end{proof}

To proceed, we quickly translate our estimates \eqref{f5f6l2thm} and \eqref{f5f6h1thm}, by noting
\begin{equation}\begin{split} ||\mathcal{F}(\sqrt{\cdot} \left(F_{5}+F_{6}\right)(x,\cdot \lambda(x)))(\omega \lambda(x)^{2})||_{L^{2}(\rho(\omega \lambda(x)^{2}) d\omega)}^{2}&=\int_{0}^{\infty} \frac{r}{\lambda(x)^{4}} \left(F_{5}+F_{6}\right)^{2}(x,r) dr\end{split}\end{equation}

\begin{equation} ||\sqrt{\omega} \lambda(x) \mathcal{F}(\sqrt{\cdot} \left(F_{5}+F_{6}\right)(x,\cdot \lambda(x)))(\omega \lambda(x)^{2})||_{L^{2}(\rho(\omega \lambda(x)^{2}) d\omega)}^{2} = \frac{1}{\lambda(x)^{2}} \int_{0}^{\infty} (L(\left(F_{5}+F_{6}\right)(x,\cdot \lambda(x))))^{2}(R) R dR\end{equation}
and
$$|L(f)(r)| \leq C\left(|f'(r)|+\frac{|f(r)|}{r}\right)$$
So, we have
\begin{equation} ||\mathcal{F}(\sqrt{\cdot} \left(F_{5}+F_{6}\right)(x,\cdot \lambda(x)))(\omega \lambda(x)^{2})||_{L^{2}(\rho(\omega \lambda(x)^{2}) d\omega)} \leq \frac{C}{x^{4} \log^{3b+2N-1}(x)}\end{equation}
and
\begin{equation} ||\sqrt{\omega} \lambda(x) \mathcal{F}(\sqrt{\cdot} \left(F_{5}+F_{6}\right)(x,\cdot \lambda(x)))(\omega \lambda(x)^{2})||_{L^{2}(\rho(\omega \lambda(x)^{2}) d\omega)}\leq C \frac{\log^{6+b}(x)}{x^{35/8}}\end{equation}

Now, we recall $F(t,r) = F_{4}(t,r)+F_{5}(t,r)+F_{6}(t,r)$, and estimate the following quantities:
\begin{equation} \label{linsoln}\begin{split} &||\int_{t}^{\infty} \frac{\sin((t-x)\sqrt{\omega})}{\sqrt{\omega}} \mathcal{F}(\sqrt{\cdot}F(x,\cdot\lambda(x)))(\omega \lambda(x)^{2}) dx||_{L^{2}(\rho(\omega \lambda(t)^{2})d\omega)}\\
&=||\int_{t}^{\infty} \frac{\sin((t-x)\sqrt{\omega})}{\sqrt{\omega}} \left(\mathcal{F}(\sqrt{\cdot}F_{4}(x,\cdot\lambda(x)))(\omega \lambda(x)^{2})+\mathcal{F}(\sqrt{\cdot}\left(F_{5}+F_{6}\right)(x,\cdot\lambda(x)))(\omega \lambda(x)^{2})\right) dx||_{L^{2}(\rho(\omega \lambda(t)^{2})d\omega)}\\
&\leq ||\int_{t}^{\infty} \frac{\sin((t-x)\sqrt{\omega})}{\sqrt{\omega}} \mathcal{F}(\sqrt{\cdot}F_{4}(x,\cdot\lambda(x)))(\omega \lambda(x)^{2}) dx||_{L^{2}(\rho(\omega \lambda(t)^{2})d\omega)} \\
&+ ||\int_{t}^{\infty} \frac{\sin((t-x)\sqrt{\omega})}{\sqrt{\omega}} \mathcal{F}(\sqrt{\cdot}\left(F_{5}+F_{6}\right)(x,\cdot\lambda(x)))(\omega \lambda(x)^{2}) dx||_{L^{2}(\rho(\omega \lambda(t)^{2})d\omega)}\\
&\leq \frac{C (\log(\log(t)))^{2}}{t^{2} \log^{b+1-2\alpha b}(t)} + C \int_{t}^{\infty} (x-t) \left(||\mathcal{F}(\sqrt{\cdot} \left(F_{5}+F_{6}\right)(x,\cdot\lambda(x)))(\omega \lambda(x)^{2})||_{L^{2}(\rho(\omega \lambda(t)^{2})d\omega)}\right)dx\\
&\leq \frac{C (\log(\log(t)))^{2}}{t^{2} \log^{b+1-2\alpha b}(t)} + C \int_{t}^{\infty} (x-t) \left(\frac{\lambda(t)}{\lambda(x)}\right) \frac{dx}{x^{4} \log^{3b+2N-1}(x)} \\
&\leq \frac{C (\log(\log(t)))^{2}}{t^{2} \log^{b+1-2\alpha b}(t)}\end{split}\end{equation}
where we used the fact that,
 \begin{equation}\begin{split}&||\mathcal{F}(\sqrt{\cdot}\left(F_{5}+F_{6}\right)(x,\cdot \lambda(x)))(\omega \lambda(x)^{2})||_{L^{2}(\rho(\omega \lambda(t)^{2})d\omega)} \\
& = \left(\int_{0}^{\infty} \left(\mathcal{F}(\sqrt{\cdot}\left(F_{5}+F_{6}\right)(x,\cdot \lambda(x)))(\omega \lambda(x)^{2})\right)^{2} \frac{\rho(\omega \lambda(t)^{2})}{\rho(\omega \lambda(x)^{2})} \rho(\omega \lambda(x)^{2}) d\omega\right)^{1/2}\\
&\leq C \frac{\lambda(t)}{\lambda(x)} ||\mathcal{F}(\sqrt{\cdot}\left(F_{5}+F_{6}\right)(x,\cdot \lambda(x)))(\omega \lambda(x)^{2})||_{L^{2}(\rho(\omega \lambda(x)^{2})d\omega)}\end{split}\end{equation}
where we used \eqref{rhoscaling}.\\
\\
Similarly,
\begin{equation} \label{dtlinsoln} \begin{split} &||\int_{t}^{\infty} \cos((t-x) \sqrt{\omega}) \mathcal{F}(\sqrt{\cdot}F(x,\cdot\lambda(x)))(\omega \lambda(x)^{2}) dx||_{L^{2}(\rho(\omega \lambda(t)^{2}) d\omega)}\\
&\leq ||\int_{t}^{\infty} \cos((t-x) \sqrt{\omega}) \mathcal{F}(\sqrt{\cdot}F_{4}(x,\cdot\lambda(x)))(\omega \lambda(x)^{2}) dx||_{L^{2}(\rho(\omega \lambda(t)^{2}) d\omega)}\\
&+||\int_{t}^{\infty} \cos((t-x) \sqrt{\omega}) \mathcal{F}(\sqrt{\cdot}\left(F_{5}+F_{6}\right)(x,\cdot\lambda(x)))(\omega \lambda(x)^{2}) dx||_{L^{2}(\rho(\omega \lambda(t)^{2}) d\omega)}\\
&\leq \frac{C (\log(\log(t)))^{2}}{t^{3} \log^{b+1-2\alpha b}(t)} + C \int_{t}^{\infty} \left(\frac{\lambda(t)}{\lambda(x)}\right)\left(||\mathcal{F}(\sqrt{\cdot} \left(F_{5}+F_{6}\right)(x,\cdot\lambda(x)))(\omega \lambda(x)^{2})||_{L^{2}(\rho(\omega \lambda(x)^{2})d\omega)} \right) dx\\
&\leq \frac{C (\log(\log(t)))^{2}}{t^{3} \log^{b+1-2\alpha b}(t)}+\frac{C}{t^{3} \log^{3b+2N-1}(t)}\\
&\leq \frac{C (\log(\log(t)))^{2}}{t^{3} \log^{b+1-2\alpha b}(t)}\end{split}\end{equation}

\begin{equation} \label{llinsoln} \begin{split} &||\sqrt{\omega}\lambda(t) \int_{t}^{\infty} \frac{\sin((t-x)\sqrt{\omega})}{\sqrt{\omega}} \mathcal{F}(\sqrt{\cdot}F(x,\cdot\lambda(x)))(\omega\lambda(x)^{2}) dx||_{L^{2}(\rho(\omega \lambda(t)^{2}) d\omega)}\\
&\leq ||\sqrt{\omega}\lambda(t) \int_{t}^{\infty} \frac{\sin((t-x)\sqrt{\omega})}{\sqrt{\omega}} \mathcal{F}(\sqrt{\cdot}F_{4}(x,\cdot\lambda(x)))(\omega\lambda(x)^{2}) dx||_{L^{2}(\rho(\omega \lambda(t)^{2}) d\omega)}\\
&+||\sqrt{\omega}\lambda(t) \int_{t}^{\infty} \frac{\sin((t-x)\sqrt{\omega})}{\sqrt{\omega}} \mathcal{F}(\sqrt{\cdot}\left(F_{5}+F_{6}\right)(x,\cdot\lambda(x)))(\omega\lambda(x)^{2}) dx||_{L^{2}(\rho(\omega \lambda(t)^{2}) d\omega)}\\
&\leq \frac{C(\log(\log(t)))^{2}}{t^{2} \log^{1+b-2\alpha b}(t)} + \lambda(t) \int_{t}^{\infty} \left(\frac{\lambda(t)}{\lambda(x)}\right)\left(||\mathcal{F}(\sqrt{\cdot} \left(F_{5}+F_{6}\right)(x,\cdot\lambda(x)))(\omega \lambda(x)^{2})||_{L^{2}(\rho(\omega \lambda(x)^{2})d\omega)} \right) dx\\
&\leq \frac{C(\log(\log(t)))^{2}}{t^{2} \log^{1+b-2\alpha b}(t)}+\frac{C}{t^{3} \log^{4b+2N-1}(t)}\\
&\leq \frac{C(\log(\log(t)))^{2}}{t^{2} \log^{1+b-2\alpha b}(t)}\end{split}\end{equation}

\begin{equation}\label{ldtlinsoln} \begin{split} &||\int_{t}^{\infty} \cos((t-x)\sqrt{\omega}) \sqrt{\omega}\lambda(x) \left(\frac{\lambda(t)}{\lambda(x)}\right) \mathcal{F}(\sqrt{\cdot} F(x,\cdot\lambda(x)))(\omega \lambda(x)^{2}) dx||_{L^{2}(\rho(\omega\lambda(t)^{2})d\omega)}\\
&\leq ||\int_{t}^{\infty} \cos((t-x)\sqrt{\omega}) \sqrt{\omega}\lambda(x) \left(\frac{\lambda(t)}{\lambda(x)}\right) \mathcal{F}(\sqrt{\cdot} F_{4}(x,\cdot\lambda(x)))(\omega \lambda(x)^{2}) dx||_{L^{2}(\rho(\omega\lambda(t)^{2})d\omega)}\\
&+||\int_{t}^{\infty} \cos((t-x)\sqrt{\omega}) \sqrt{\omega}\lambda(x) \left(\frac{\lambda(t)}{\lambda(x)}\right) \mathcal{F}(\sqrt{\cdot} \left(F_{5}+F_{6}\right)(x,\cdot\lambda(x)))(\omega \lambda(x)^{2}) dx||_{L^{2}(\rho(\omega\lambda(t)^{2})d\omega)}\\
&\leq \frac{C(\log(\log(t)))^{2}}{t^{3} \log^{1+b-2\alpha b}(t)} + C \int_{t}^{\infty} \left(\frac{\lambda(t)}{\lambda(x)}\right)^{2} \left(||\sqrt{\omega} \lambda(x) \mathcal{F}(\sqrt{\cdot} \left(F_{5}+F_{6}\right)(x,\cdot\lambda(x)))(\omega\lambda(x)^{2})||_{L^{2}(\rho(\omega \lambda(x)^{2})d\omega)} \right) dx\\
&\leq \frac{C(\log(\log(t)))^{2}}{t^{3} \log^{1+b-2\alpha b}(t)} +C \int_{t}^{\infty} \left(\frac{\lambda(t)}{\lambda(x)}\right)^{2} \frac{\log^{b+6}(x)}{x^{35/8}} dx\\
&\leq \frac{C(\log(\log(t)))^{2}}{t^{3} \log^{1+b-2\alpha b}(t)}\end{split}\end{equation}
and, finally,
\begin{equation} \label{lstarllinsoln} \begin{split} &||\int_{t}^{\infty} \sin((t-x)\sqrt{\omega}) \sqrt{\omega} \lambda(t)^{2} \mathcal{F}(\sqrt{\cdot} F(x,\cdot\lambda(x)))(\omega\lambda(x)^{2}) dx||_{L^{2}(\rho(\omega \lambda(t)^{2})d\omega)}\\
&\leq  ||\int_{t}^{\infty} \sin((t-x)\sqrt{\omega}) \sqrt{\omega} \lambda(t)^{2} \mathcal{F}(\sqrt{\cdot} F_{4}(x,\cdot\lambda(x)))(\omega\lambda(x)^{2}) dx||_{L^{2}(\rho(\omega \lambda(t)^{2})d\omega)}\\
&+||\int_{t}^{\infty} \sin((t-x)\sqrt{\omega}) \sqrt{\omega} \lambda(t)^{2} \mathcal{F}(\sqrt{\cdot} \left(F_{5}+F_{6}\right)(x,\cdot\lambda(x)))(\omega\lambda(x)^{2}) dx||_{L^{2}(\rho(\omega \lambda(t)^{2})d\omega)}\\
&\leq \frac{C}{t^{2} \log^{1+b-2\alpha b}(t)} + C \int_{t}^{\infty} \lambda(t) \left(\frac{\lambda(t)}{\lambda(x)}\right)^{2} \left(||\lambda(x) \mathcal{F}(\sqrt{\omega} \sqrt{\cdot}\left(F_{5}+F_{6}\right)(x,\cdot\lambda(x)))(\omega\lambda(x)^{2})||_{L^{2}(\rho(\omega\lambda(x)^{2})d\omega)}\right) dx\\
&\leq \frac{C}{t^{2} \log^{1+b-2\alpha b}(t)}+C \int_{t}^{\infty} \frac{\lambda(t)^{3}}{\lambda(x)^{2}} \frac{\log^{b+6}(x)}{x^{35/8}} dx\\
&\leq \frac{C}{t^{2} \log^{1+b-2\alpha b}(t)}\end{split}\end{equation}

\subsection{Setup of the final iteration}
Let $\epsilon$ be given by
$$\epsilon = 2b+\frac{1}{2}(1-2\alpha b)$$
Note that $2b+\frac{1}{2} > \epsilon > 2b$. Also, note that \eqref{vcorrcofthm}, \eqref{1cofthm}, and \eqref{vcorrdrvcorrcofthm} show that

\begin{equation} \label{correst1}||\frac{v_{corr}(x,R\lambda(x))}{R \lambda(x)}||_{L^{\infty}}^{2}+||\frac{v_{corr}(x,R\lambda(x))}{R\lambda(x)^{2}(1+R^{2})}||_{L^{\infty}} \leq \frac{C}{x^{2} \log^{\epsilon-2b}(x)}\end{equation}

\begin{equation}\label{correst2} 1+||\frac{v_{corr}(x,R\lambda(x))}{R}||_{L^{\infty}}+||\partial_{R}(v_{corr}(x,R\lambda(x)))||_{L^{\infty}} \leq C\end{equation}

\begin{equation}\label{correst3}\begin{split} &||\frac{v_{corr}(x,R\lambda(x)) \partial_{R}(v_{corr}(x,R\lambda(x)))}{R\lambda(x)^{2}}||_{L^{\infty}_{R}((0,1))}+||\frac{v_{corr}(x,R\lambda(x)) \partial_{R}(v_{corr}(x,R\lambda(x)))}{R^{2}\lambda(x)^{2}}||_{L^{\infty}_{R}((1,\infty))}\\
&+||\frac{\partial_{R}(v_{corr}(x,R\lambda(x)))}{(1+R^{2})\lambda(x)^{2}}||_{L^{\infty}} \\
&\leq \frac{C}{x^{2}\log^{\epsilon-2b}(x)}\end{split}\end{equation}

Let $(Z,||\cdot||_{Z})$ be the normed vector space defined as follows. $Z$ is the set of (equivalence classes) of measureable functions $y: [T_{0},\infty)\times (0,\infty) \rightarrow \mathbb{R}$ such that
$$ y(t,\omega) t^{2} \log^{\frac{\epsilon}{2}}(t) \sqrt{\rho(\omega \lambda(t)^{2})} \langle \omega \lambda(t)^{2}\rangle \in C^{0}_{t}([T_{0},\infty),L^{2}(d\omega))$$
$$\partial_{t}y(t,\omega) t^{3} \log^{\frac{\epsilon}{2}}(t) \langle \sqrt{\omega} \lambda(t)\rangle \sqrt{\rho(\omega \lambda(t)^{2})} \in C^{0}_{t}([T_{0},\infty), L^{2}(d\omega))$$
and $||y||_{Z} < \infty$
where
\begin{equation}\label{znorm}\begin{split}||y||_{Z} &= \sup_{t \geq T_{0}}\left(t^{2}\log^{\frac{\epsilon}{2}}(t)\left(||y(t)||_{L^{2}(\rho(\omega \lambda(t)^{2})d\omega)}+||\lambda(t)\sqrt{\omega}y(t)||_{L^{2}(\rho(\omega \lambda(t)^{2})d\omega)}+ ||\omega \lambda(t)^{2}y(t)||_{L^{2}(\rho(\omega \lambda(t)^{2})d\omega)}\right) \right.\\
&\left.+ t^{3}\log^{\frac{\epsilon}{2}}(t)\left(||\partial_{t}y(t)||_{L^{2}(\rho(\omega \lambda(t)^{2})d\omega)}+||\lambda(t)\sqrt{\omega} \partial_{t}y(t)||_{L^{2}(\rho(\omega \lambda(t)^{2})d\omega)} \right)\right)\end{split}\end{equation}\\
\\
Define $T$ on $Z$ by
\begin{equation} T(y)(t,\omega) = -\int_{t}^{\infty} \frac{\sin((t-x)\sqrt{\omega})}{\sqrt{\omega}}\left(F_{2}(y)(x,\omega) - \mathcal{F}(\sqrt{\cdot}F(x,\cdot\lambda(x)))(\omega \lambda(x)^{2}) - \mathcal{F}(\sqrt{\cdot}F_{3}(u(y))(x,\cdot\lambda(x)))(\omega \lambda(x)^{2})\right)dx\end{equation}\\
\\
Note that a fixed point of $T$ is a solution to \eqref{ylinearproblem} with $0$ Cauchy data at infinity.  We will prove the following proposition which implies that $T$ indeed has a fixed point in $\overline{B_{1}(0)}\subset Z$.
\begin{proposition} There exists $T_{4}>0$ such that, for all $T_{0}>T_{4}$, $T$ is a strict contraction on $\overline{B}_{1}(0) \subset Z$\end{proposition}
We use \eqref{linsoln}, Propositions \ref{f2prop} and \ref{f3prop}, and the equations \eqref{correst1}, \eqref{correst2}, and \eqref{correst3}, to get the following (note that in the following estimates, $C>0$ denotes a constant (which might involve $C_{\rho}$) whose value may change from line to line, but which is \emph{independent} of $T_{0}$). Also, for ease of notation, we will denote $F_{3}(u(y))$ by $F_{3}$ until otherwise mentioned.
\begin{equation}\begin{split} ||T(y)(t)||_{L^{2}(\rho(\omega \lambda(t)^{2})d\omega)} &\leq \frac{C (\log(\log(t)))^{2}}{t^{2} \log^{b+1-2\alpha b}(t)} + C \int_{t}^{\infty} \frac{\lambda(t)}{\lambda(x)} \frac{1}{x^{3}\log^{\frac{\epsilon}{2}+1}(x)} dx\\
&+C \int_{t}^{\infty} x \left(\frac{\lambda(t)}{\lambda(x)}\right)\left(\frac{1}{x^{2}\log^{\frac{\epsilon}{2}}(x)}\left(\frac{1}{x^{2}\log^{\epsilon-2b}(x)}\right)+\frac{1}{\lambda(x) x^{4}\log^{\epsilon}(x)}\right) dx\\
&\leq \frac{C (\log(\log(t)))^{2}}{t^{2} \log^{b+1-2\alpha b}(t)} + \frac{C}{t^{2}\log^{\frac{\epsilon}{2}+1}(t)}+\frac{C}{t^{2}\log^{3 \frac{\epsilon}{2}-2b}(t)} + \frac{C}{t^{2}\log^{\epsilon-b}(t)}\\
&\leq \frac{C (\log(\log(t)))^{2}}{t^{2} \log^{b+1-2\alpha b}(t)}+ C \frac{\log^{-\epsilon+2b}(t)+\log^{-\frac{\epsilon}{2}+b}(t)+\log^{-1}(t)}{t^{2}\log^{\frac{\epsilon}{2}}(t)}\end{split}\end{equation}\\
\\
Next,
\begin{equation} \partial_{t}T(y)(t,\omega) = -\int_{t}^{\infty} \cos((t-x)\sqrt{\omega})\left(F_{2}(y)(x,\omega)-\mathcal{F}(\sqrt{\cdot}F(x,\cdot \lambda(x)))(\omega \lambda(x)^{2})-\mathcal{F}(\sqrt{\cdot}F_{3}(x,\cdot\lambda(x)))(\omega \lambda(x)^{2})\right) dx\end{equation}\\
\\
and the same procedure as above gives
\begin{equation} \begin{split} ||\partial_{t}T(y)(t)||_{L^{2}(\rho(\omega \lambda(t)^{2})d\omega)} &\leq \frac{C (\log(\log(t)))^{2}}{t^{3} \log^{b+1-2\alpha b}(t)} + C \int_{t}^{\infty} \frac{\lambda(t)}{\lambda(x)} \frac{1}{x^{4}\log^{\frac{\epsilon}{2}+1}(x)} dx\\
&+C \int_{t}^{\infty}  \left(\frac{\lambda(t)}{\lambda(x)}\right)\left(\frac{1}{x^{2}\log^{\frac{\epsilon}{2}}(x)}\left(\frac{1}{x^{2}\log^{\epsilon-2b}(x)}\right)+\frac{1}{\lambda(x) x^{4}\log^{\epsilon}(x)}\right) dx\\
&\leq \frac{C (\log(\log(t)))^{2}}{t^{3} \log^{b+1-2\alpha b}(t)}+ C \frac{ \log^{-\epsilon+2b}(t)+\log^{-\frac{\epsilon}{2}+b}(t)+\log^{-1}(t)}{t^{3}\log^{\frac{\epsilon}{2}}(t)}\end{split}\end{equation}\\
\\
Similarly, \begin{equation} \sqrt{\omega} \lambda(t) T(y)(t,\omega) = -\lambda(t) \int_{t}^{\infty} \sin((t-x)\sqrt{\omega})\left(F_{2}(y)(x,\omega)-\mathcal{F}(\sqrt{\cdot}F(x,\cdot \lambda(x)))(\omega \lambda(x)^{2})-\mathcal{F}(\sqrt{\cdot}F_{3}(x,\cdot\lambda(x)))(\omega \lambda(x)^{2})\right) dx\end{equation}\\
\\
and the identical argument as for the previous two terms gives
\begin{equation}\begin{split} ||\sqrt{\omega} \lambda(t) T(y)(t,\omega)||_{L^{2}(\rho(\omega \lambda(t)^{2})d\omega)}&\leq \frac{C(\log(\log(t)))^{2}}{t^{2} \log^{1+b-2\alpha b}(t)}+ C \lambda(t) \frac{\log^{-\epsilon+2b}(t)+\log^{-\frac{\epsilon}{2}+b}(t)+\log^{-1}(t)}{t^{3}\log^{\frac{\epsilon}{2}}(t)}\\
&\leq \frac{C(\log(\log(t)))^{2}}{t^{2} \log^{1+b-2\alpha b}(t)}+ C  \frac{\log^{-\epsilon+2b}(t)+\log^{-\frac{\epsilon}{2}+b}(t)+\log^{-1}(t)}{t^{3}\log^{\frac{\epsilon}{2}+b}(t)} \end{split}\end{equation}\\
\\
The next term is
\begin{equation}\begin{split} &\lambda(t) \sqrt{\omega} \partial_{t}T(y)(t,\omega) \\
&= -\int_{t}^{\infty} \cos((t-x)\sqrt{\omega}) \left(\frac{\lambda(t)}{\lambda(x)}\right) \sqrt{\omega} \lambda(x) \left(F_{2}(y)(x,\omega) - \mathcal{F}(\sqrt{\cdot}F(x,\cdot \lambda(x)))(\omega \lambda(x)^{2})-\mathcal{F}(\sqrt{\cdot}F_{3}(x,\cdot\lambda(x)))(\omega \lambda(x)^{2})\right) dx \end{split}\end{equation}\\
\\
and we get
\begin{equation}\begin{split} &||\lambda(t) \sqrt{\omega} \partial_{t}T(y)(t,\omega)||_{L^{2}(\rho(\omega \lambda(t)^{2})d\omega)} \\
&\leq\frac{C(\log(\log(t)))^{2}}{t^{3} \log^{1+b-2\alpha b}(t)} +C \int_{t}^{\infty} \left(\frac{\lambda(t)}{\lambda(x)}\right)^{2} \frac{C}{x^{4}\log^{\frac{\epsilon}{2}+1}(x)} dx+\int_{t}^{\infty} \left(\frac{\lambda(t)}{\lambda(x)}\right)^{2} \frac{1}{x^{4} \log^{\frac{3 \epsilon}{2}-2b}(x)} dx\\
&+C \int_{t}^{\infty} \left(\frac{\lambda(t)}{\lambda(x)}\right)^{2} \log^{b}(x) \frac{1}{x^{4} \log^{\epsilon}(x)} dx\\
&\leq \frac{C(\log(\log(t)))^{2}}{t^{3} \log^{1+b-2\alpha b}(t)}+ C \frac{\log^{b-\frac{\epsilon}{2}}(t) + \log^{-1}(t) + \log^{-\epsilon+2b}(t)}{t^{3}\log^{\frac{\epsilon}{2}}(t)}\end{split}\end{equation}\\
\\
\begin{equation}\begin{split} &\omega \lambda(t)^{2} T(y)(t,\omega) \\
&= - \int_{t}^{\infty} \sin((t-x)\sqrt{\omega}) \frac{\lambda(t)^{2}}{\lambda(x)} \lambda(x) \sqrt{\omega}\left(F_{2}(y)(x,\omega) - \mathcal{F}(\sqrt{\cdot}F(x,\cdot \lambda(x)))(\omega \lambda(x)^{2})-\mathcal{F}(\sqrt{\cdot}F_{3}(x,\cdot\lambda(x)))(\omega \lambda(x)^{2})\right) dx \end{split}\end{equation}\\
\\
and the same procedure as above gives
\begin{equation}\begin{split} &||\omega \lambda(t)^{2} T(y)(t,\omega)||_{L^{2}(\rho(\omega \lambda(t)^{2})d\omega)} \\
&\leq \frac{C}{t^{2} \log^{1+b-2\alpha b}(t)} +C \int_{t}^{\infty} \left(\frac{\lambda(t)^{2}}{\lambda(x)}\right)\left(\frac{\lambda(t)}{\lambda(x)}\right) \frac{C}{x^{3}\log^{\frac{\epsilon}{2}}(x)} \frac{1}{x\log(x)} dx\\
&+C \int_{t}^{\infty} \left(\frac{\lambda(t)^{2}}{\lambda(x)}\right)\left(\frac{\lambda(t)}{\lambda(x)}\right) \frac{C}{x^{2}\log^{\frac{\epsilon}{2}}(x)} \frac{1}{x^{2}\log^{\epsilon-2b}(x)} dx\\
&+C \int_{t}^{\infty} \left(\frac{\lambda(t)^{2}}{\lambda(x)}\right)\left(\frac{\lambda(t)}{\lambda(x)}\right)\frac{\log^{b}(x)}{x^{4}\log^{\epsilon}(x)} dx\\
&\leq \frac{C}{t^{2} \log^{1+b-2\alpha b}(t)}+ C \frac{ \log^{-b-1}(t) + \log^{-b-\epsilon+2b}(t)+\log^{-\frac{\epsilon}{2}}(t)}{\log^{\frac{\epsilon}{2}}(t) t^{3}}\end{split}\end{equation}
Moreover, by, for example the Dominated convergence theorem,
\begin{equation}  T(y)(t,\omega) t^{2} \log^{\frac{\epsilon}{2}}(t) \sqrt{\rho(\omega \lambda(t)^{2})} \langle \omega \lambda(t)^{2}\rangle \in C^{0}_{t}([T_{0},\infty),L^{2}(d\omega))\end{equation}
and
\begin{equation}\partial_{t}T(y)(t,\omega) t^{3} \log^{\frac{\epsilon}{2}}(t) \langle \sqrt{\omega} \lambda(t)\rangle \sqrt{\rho(\omega \lambda(t)^{2})} \in C^{0}_{t}([T_{0},\infty), L^{2}(d\omega))\end{equation}
So, if $T_{0}$ is large enough, then, $T(y) \in \overline{B}_{\frac{1}{2}}(0) \subset Z$ if $y \in \overline{B}_{1}(0) \subset Z$. We will now show that $T$ is a strict contraction on $\overline{B}_{1}(0) \subset Z$. Let $y_{1},y_{2} \in Z$ satisfy $$||y_{1}||_{Z},||y_{2}||_{Z} \leq 1$$\\
\\
\begin{equation}\label{Tislip} T(y_{1})-T(y_{2}) = -\int_{t}^{\infty} \frac{\sin((t-x)\sqrt{\omega})}{\sqrt{\omega}} \left(F_{2}(y_{1})-F_{2}(y_{2})-\left(\mathcal{F}(\sqrt{\cdot}(F_{3}(u(y_{1}))-F_{3}(u(y_{2})))(x,\cdot \lambda(x))))(\omega \lambda(x)^{2})\right)\right)dx\end{equation} First, note that $F_{2}$ is linear in $y$, so 
$$F_{2}(y_{1})-F_{2}(y_{2}) = F_{2}(y_{1}-y_{2})$$
Next, we treat $F_{3}$ starting with the $L_{1}$ terms. We will denote by $u_{i}$ the function associated to $y_{i}$ via \eqref{ytou}. We will also use $\overline{v}_{i}, w_{i}$ to denote the functions associated to $u_{i}$ in the same way $v$ and $w$ were used in the above discussion\\
\\
Recall that
\begin{equation}\begin{split} &(L_{1}(u_{1})-L_{1}(u_{2}))(t,r) \\
&= \left(\frac{\sin(2u_{1}(t,r))-\sin(2u_{2}(t,r))}{2r^{2}}\right)\left(\cos(2Q_{1}(\frac{r}{\lambda(t)}))\left(\cos(2v_{corr}(t,r))-1\right)-\sin(2Q_{1}(\frac{r}{\lambda(t)}))\sin(2v_{corr}(t,r))\right)\end{split}\end{equation}\\
\\
Since $$|\sin(2u_{1}) - \sin(2u_{2})| \leq 2 |u_{1}-u_{2}|$$
we get (after esimating in terms of $u_{i}$ and then translating to $y_{i}$ in exactly the same manner as done above)
\begin{equation}\begin{split} &||\mathcal{F}(\sqrt{\cdot}(L_{1}(u_{1})-L_{1}(u_{2}))(t,\cdot \lambda(t)))(\omega \lambda(t)^{2})||_{L^{2}(\rho(\omega \lambda(t)^{2})d\omega)} \\
&\leq C ||y_{1}(t)-y_{2}(t)||_{L^{2}(\rho(\omega \lambda(t)^{2})d\omega)} \left(||\frac{v_{corr}(t,R\lambda(t))}{R \lambda(t)}||^{2}_{L^{\infty}} + ||\frac{v_{corr}(t,R\lambda(t))}{R(1+R^{2})\lambda(t)^{2}}||_{L^{\infty}}\right)\end{split}\end{equation}\\
\\
Next, we estimate 
\begin{equation}\begin{split} &|\partial_{R}\left((L_{1}(u_{1})-L_{1}(u_{2}))(t,R\lambda(t))\right)|\\
&\leq \frac{C \left(|\overline{v}_{1}-\overline{v}_{2}|\cdot|\partial_{R}\overline{v}_{1}| + |\partial_{R}(\overline{v}_{1}-\overline{v}_{2})|\right)}{R^{2}\lambda(t)^{2}}\left(|v_{corr}(t,R\lambda(t)|^{2}+\frac{R}{(1+R^{2})}|v_{corr}(t,R\lambda(t))\right)\\
&+\frac{C |\overline{v}_{1}-\overline{v}_{2}|}{R^{3}\lambda(t)^{2}}\left(|v_{corr}(t,R\lambda(t))|^{2}+\frac{R |v_{corr}|}{(1+R^{2})}\right)\\
&+\frac{C |\overline{v}_{1}-\overline{v}_{2}|}{R^{2}\lambda(t)^{2}}\left(|v_{corr}(t,R\lambda(t))\partial_{R}(v_{corr}(t,R\lambda(t)))| + \frac{R |\partial_{R}(v_{corr}(t,R\lambda(t)))}{(1+R^{2})}\right)\end{split}\end{equation}\\
\\
Then, we get
\begin{equation}\begin{split} &||\partial_{R}((L_{1}(u_{1})-L_{1}(u_{2}))(t,R\lambda(t)))||_{L^{2}(R dR)} \\
&\leq C \left(||\overline{v}_{1}-\overline{v}_{2}||_{L^{2}(R dR)} + ||L(\overline{v}_{1}-\overline{v}_{2})||_{L^{2}(R dR)}\right)\\
&\cdot \left(\left(1+||\partial_{R}(\overline{v}_{1}(t,R\lambda(t)))||_{L^{2}(R dR)}\right)\left(||\frac{v_{corr}(t,R\lambda(t))}{R\lambda(t)}||_{L^{\infty}}^{2} + ||\frac{v_{corr}(t,R\lambda(t))}{R \lambda(t)^{2}(1+R^{2})}||_{L^{\infty}}\right)\right.\\
&\left.+ ||\frac{v_{corr}(t,R\lambda(t))\partial_{R}(v_{corr}(t,R\lambda(t)))}{R\lambda(t)^{2}}||_{L^{\infty}(R \leq 1)} + ||\frac{v_{corr}(t,R\lambda(t)) \partial_{R}(v_{corr}(t,R\lambda(t)))}{R^{2}\lambda(t)^{2}}||_{L^{\infty}(R \geq 1)} + ||\frac{\partial_{R}(v_{corr}(t,R\lambda(t)))}{\lambda(t)^{2}(1+R^{2})}||_{L^\infty}\right)\end{split}\end{equation}\\
\\
Similarly,
\begin{equation}\begin{split}|\frac{L_{1}(u_{1})(t,R\lambda(t))-L_{1}(u_{2})(t,R\lambda(t))}{R}|&\leq \frac{C |\overline{v}_{1}-\overline{v}_{2}|(t,R)}{R^{3}\lambda(t)^{2}}\left(|v_{corr}(t,R\lambda(t))|^{2}+\frac{R |v_{corr}(t,R\lambda(t))|}{(R^{2}+1)}\right)\end{split}\end{equation}\\
\\
We combine these to estimate $||L\left((L_{1}(u_{1})-L_{1}(u_{2}))(t,R\lambda(t))\right)||_{L^{2}(R dR)}$, and then, as in the previous estimates, translate the right-hand side in terms of $y_{i}$, and use the estiamtes on our ansatz to get
\begin{equation}\begin{split} &||\lambda(t) \sqrt{\omega} \mathcal{F}(\sqrt{\cdot} \left((L_{1}(u_{1})-L_{1}(u_{2})(t,\cdot \lambda(t)))\right))(\omega \lambda(t)^{2})||_{L^{2}(\rho(\omega \lambda(t)^{2}) d\omega)} \\
&\leq C \frac{||y_{1}-y_{2}||_{L^{2}(\rho(\omega \lambda(t)^{2}) d\omega)} + ||\sqrt{\omega} \lambda(t) (y_{1}-y_{2})||_{L^{2}(\rho(\omega \lambda(t)^{2}) d\omega)}}{t^{2} \log^{\epsilon-2b}(t)}\\
&\leq C \frac{||y_{1}-y_{2}||_{Z}}{t^{4}\log^{\frac{3 \epsilon}{2}-2b}(t)}\end{split}\end{equation}\\
\\
Now, we treat the nonlinear terms. First, note that, if
$$n_{1}(x) = \sin(2x)-2x$$
then, by the mean value theorem,
\begin{equation} \begin{split} |n_{1}(x_{1})-n_{1}(x_{2})| &\leq |x_{1}-x_{2}| \text{max}_{\theta \in [0,1]} |n_{1}'(\theta x_{1}+(1-\theta)x_{2})|\\
&\leq |x_{1}-x_{2}| \text{max}_{\theta \in [0,1]}|4 \sin^{2}(\theta x_{1}+(1-\theta)x_{2})|\\
&\leq C |x_{1}-x_{2}| \left(|x_{1}|^{2}+|x_{2}|^{2}\right)\end{split}\end{equation}\\
\\
Similarly, 
\begin{equation} |\left(\cos(2 x_{1})-1\right)-\left(\cos(2 x_{2})-1\right)| \leq C \left(|x_{1}|+|x_{2}|\right) |x_{1}-x_{2}|\end{equation}\\
\\
So,
\begin{equation} \begin{split} |(N(u_{1})-N(u_{2}))(t,R\lambda(t))| &\leq C\frac{|\overline{v}_{1}-\overline{v}_{2}|\left(\overline{v}_{1}^{2}+\overline{v}_{2}^{2}\right)}{R^{2}\lambda(t)^{2}}\\
&+C \frac{|\overline{v}_{1}-\overline{v}_{2}|\left(|\overline{v}_{1}|+|\overline{v}_{2}|\right)}{R^{2}\lambda(t)^{2}} \left(|Q_{1}(R)|+|v_{corr}(t,R\lambda(t))|\right)\end{split}\end{equation}\\
\\
If $R <1$, we use the estimates \eqref{foverxinfinity} to get
\begin{equation} \frac{|\overline{v}_{1}(t,R)|}{\sqrt{R}\lambda(t)} \leq C \frac{||\overline{v}_{1}(t)||_{\dot{H}^{1}_{e}}}{\lambda(t)} + C \frac{||L^{*}L\overline{v}_{1}(t)||_{L^{2}(R dR)}}{\lambda(t)}\end{equation}\\
\\
and if $R>1$, then, we use \eqref{h1dotecoercive3} to get
\begin{equation} \frac{|\overline{v}_{1}(t,R)|}{\sqrt{R} \lambda(t)} \leq \frac{C}{\lambda(t)} ||\overline{v}_{1}(t)||_{\dot{H}^{1}_{e}}\end{equation}\\
\\
In total, we obtain
\begin{equation} \begin{split} &||N(u_{1})-N(u_{2})(t,\cdot\lambda(t))||_{L^{2}(R dR)} \\
&\leq C ||\overline{v}_{1}-\overline{v}_{2}||_{\dot{H}^{1}_{e}} \left(||\langle \omega \lambda(t)^{2}\rangle y_{1}||_{L^{2}(\rho(\omega \lambda(t)^{2}) d\omega)}^{2} + ||\langle \omega \lambda(t)^{2}\rangle y_{2}||_{L^{2}(\rho(\omega \lambda(t)^{2}) d\omega)}^{2} \right)\\
&+C ||\overline{v}_{1}-\overline{v}_{2}||_{\dot{H}^{1}_{e}} \left(\frac{||\langle \sqrt{\omega} \lambda(t)\rangle y_{1}||_{L^{2}(\rho(\omega \lambda(t)^{2}) d\omega)}+||\langle \sqrt{\omega} \lambda(t)\rangle y_{2}||_{L^{2}(\rho(\omega \lambda(t)^{2}) d\omega)}}{\lambda(t)}\right)\left(1+||\frac{v_{corr}(t,R\lambda(t))}{R}||_{L^{\infty}}\right)\\
\end{split}\end{equation}
which gives
\begin{equation} \begin{split} &||\mathcal{F}(\sqrt{\cdot} (N(u_{1})-N(u_{2}))(t,\cdot\lambda(t)))(\omega \lambda(t)^{2})||_{L^{2}(\rho(\omega \lambda(t)^{2})d\omega)} \leq \frac{C}{\lambda(t)} \frac{||y_{1}-y_{2}||_{Z}}{t^{4}\log^{\epsilon}(t)}\end{split}\end{equation}
where we used the estimates on the ansatz, and the fact that 
$$||y_{i}||_{Z} \leq 1, \quad i=1,2$$\\
\\
Next, we will estimate $\partial_{R}\left((N(u_{1})-N(u_{2}))(t,R\lambda(t))\right)$, treating the following expression one line at a time.
\begin{equation}\label{drndiff}\begin{split} \partial_{R}\left((N(u_{1})-N(u_{2}))(t,R\lambda(t))\right) &= \frac{\cos(2Q_{1}(R))}{2R^{2}\lambda(t)^{2}} \left(2\left(\cos(2\overline{v}_{1}(t,R))-1\right)\partial_{R}\overline{v}_{1}-2\left(\cos(2\overline{v}_{2})-1\right)\partial_{R}\overline{v}_{2}\right)\\
&+\partial_{R}\left(\frac{\cos(2Q_{1}(R))}{2R^{2}\lambda(t)^{2}}\right)\left(n_{1}(\overline{v}_{1})-n_{2}(\overline{v}_{2})\right)\\
&-2 \left(\frac{\sin(2Q_{1}(R)+2v_{corr}(t,R\lambda(t)))}{2R^{2}\lambda(t)^{2}}\right)\left(\sin(2\overline{v}_{1})\partial_{R}\overline{v}_{1}-\sin(2\overline{v}_{2})\partial_{R}\overline{v}_{2}\right)\\
&+\left(\cos(2\overline{v}_{1})-\cos(2\overline{v}_{2})\right)\partial_{R}\left(\frac{\sin(2Q_{1}(R)+2v_{corr}(t,R\lambda(t)))}{2R^{2}\lambda(t)^{2}}\right)\end{split}\end{equation}\\
\\
For $R \leq 1$, we estimate the first line of \eqref{drndiff} by
\begin{equation}\begin{split}&|\frac{\cos(2Q_{1}(R))}{2R^{2}\lambda(t)^{2}} \left(2\left(\cos(2\overline{v}_{1}(t,R))-1\right)\partial_{R}\overline{v}_{1}-2\left(\cos(2\overline{v}_{2})-1\right)\partial_{R}\overline{v}_{2}\right)| \\
&\leq \frac{C}{\lambda(t)^{2}} ||\frac{\overline{v}_{1}}{\sqrt{R}}||_{L^{\infty}}^{2} |\frac{L(\overline{v}_{1}-\overline{v}_{2})}{R}| + \frac{C}{\lambda(t)^{2}} \frac{\overline{v}_{1}(t,R)^{2}}{R^{2}} \frac{|\overline{v}_{1}-\overline{v}_{2}|}{R}\\
&+\frac{C}{\lambda(t)^{2}} ||\frac{\overline{v}_{1}-\overline{v}_{2}}{\sqrt{R}}||_{L^{\infty}} ||\frac{|\overline{v}_{1}|+|\overline{v}_{2}|}{\sqrt{R}}||_{L^{\infty}} \frac{|L \overline{v}_{2}|}{R}+ \frac{C}{\lambda(t)^{2}} |\frac{\overline{v}_{1}-\overline{v}_{2}}{R}|\left(\frac{|\overline{v}_{1}|+|\overline{v}_{2}|}{R}\right) \frac{|\overline{v}_{2}|}{R}\end{split}\end{equation}\\
\\
where we used the fact that
\begin{equation} \begin{split} &2(\cos(2\overline{v}_{1}(t,R))-1)\partial_{R}\overline{v}_{1}(t,R) - 2(\cos(2\overline{v}_{2}(t,R))-1)\partial_{R}\overline{v}_{2}(t,R) \\
&=2\left(\cos(2\overline{v}_{1}(t,R))-1\right)\left(\partial_{R}(\overline{v}_{1}-\overline{v}_{2})\right)+2\partial_{R}\overline{v}_{2}\left(\cos(2\overline{v}_{1})-\cos(2\overline{v}_{2})\right)\end{split}\end{equation}\\
\\
to get
\begin{equation} \begin{split} &|2(\cos(2\overline{v}_{1}(t,R))-1)\partial_{R}\overline{v}_{1}(t,R) - 2(\cos(2\overline{v}_{2}(t,R))-1)\partial_{R}\overline{v}_{2}(t,R)| \\
& \leq C \overline{v}_{1}(t,R)^{2} |\partial_{R}(\overline{v}_{1}-\overline{v}_{2})| + C |\overline{v}_{1}-\overline{v}_{2}| |\partial_{R}\overline{v}_{2}| \left(|\overline{v}_{1}|+|\overline{v}_{2}|\right)\end{split}\end{equation}\\
\\
The second line of \eqref{drndiff} is estimated by
\begin{equation}\begin{split} &\partial_{R}\left(\frac{\cos(2Q_{1}(R))}{2R^{2}\lambda(t)^{2}}\right)\left(n_{1}(\overline{v}_{1})-n_{2}(\overline{v}_{2})\right)\leq \frac{C}{\lambda(t)^{2}}\left(\frac{\overline{v}_{1}^{2}}{R^{2}}+\frac{\overline{v}_{2}^{2}}{R^{2}}\right)\frac{|\overline{v}_{1}-\overline{v}_{2}|}{R}\end{split}\end{equation}\\
\\
The third line of \eqref{drndiff} is estimated by
\begin{equation}\begin{split} &-2 \left(\frac{\sin(2Q_{1}(R)+2v_{corr}(t,R\lambda(t)))}{2R^{2}\lambda(t)^{2}}\right)\left(\sin(2\overline{v}_{1})\partial_{R}\overline{v}_{1}-\sin(2\overline{v}_{2})\partial_{R}\overline{v}_{2}\right)\\
&\leq \frac{C}{\lambda(t)^{2}}\left(1+||\frac{v_{corr}(t,R\lambda(t))}{R}||_{L^{\infty}}\right)\left(||\overline{v}_{1}||_{\dot{H}^{1}_{e}}\frac{|L(\overline{v}_{1}-\overline{v}_{2})|}{R} + ||\overline{v}_{1}-\overline{v}_{2}||_{\dot{H}^{1}_{e}}|\frac{L\overline{v}_{2}}{R}|+\left(\frac{|\overline{v}_{1}|+|\overline{v}_{2}|}{R}\right)\frac{|\overline{v}_{1}-\overline{v}_{2}|}{R}\right)\end{split}\end{equation}\\
\\
where we used 
\begin{equation}\begin{split} \sin(2\overline{v}_{1})\partial_{R}\overline{v}_{1}-\sin(2\overline{v}_{2})\partial_{R}\overline{v}_{2} &= \sin(2\overline{v}_{1})\left(\partial_{R}(\overline{v}_{1}-\overline{v}_{2})\right) + \left(\sin(2\overline{v}_{1})-\sin(2\overline{v}_{2})\right)\partial_{R}\overline{v}_{2}\end{split}\end{equation}\\
\\
The final line of \eqref{drndiff} is estimated by
\begin{equation} \begin{split} &\left(\cos(2\overline{v}_{1})-\cos(2\overline{v}_{2})\right)\partial_{R}\left(\frac{\sin(2Q_{1}(R)+2v_{corr}(t,R\lambda(t)))}{2R^{2}\lambda(t)^{2}}\right)\\
&\leq C \left(\frac{|\overline{v}_{1}|+|\overline{v}_{2}|}{R}\right) \frac{|\overline{v}_{1}-\overline{v}_{2}|}{R\lambda(t)^{2}} \left(1+||\partial_{R}(v_{corr}(t,R\lambda(t)))||_{L^{\infty}}+||\frac{v_{corr}(t,R\lambda(t))}{R}||_{L^{\infty}}\right)\end{split}\end{equation}\\
\\
Now, we will estimate the contribution of each of these terms to $||\partial_{R}\left((N(u_{1})-N(u_{2}))(t,R\lambda(t))\right)||_{L^{2}((0,1),R dR)}$.  We use \eqref{foverxinfinity} to control $||\frac{\overline{v}_{1}}{\sqrt{R}}||_{L^{\infty}}$, and get
\begin{equation}\begin{split} &\int_{0}^{1} \frac{C}{\lambda(t)^{4}} ||\frac{\overline{v}_{1}}{\sqrt{R}}||_{L^{\infty}}^{4} \left(\frac{L(\overline{v}_{1}-\overline{v}_{2})}{R}\right)^{2} R dR\\
&\leq C \lambda(t)^{2} ||\langle \omega \lambda(t)^{2} \rangle y_{1}||_{L^{2}(\rho(\omega \lambda(t)^{2}) d\omega)}^{4} ||\langle \omega \lambda(t)^{2} \rangle (y_{1}-y_{2})||_{L^{2}(\rho(\omega \lambda(t)^{2}) d\omega)}^{2} \end{split}\end{equation}\\
\\
Next, we get
\begin{equation}\begin{split} &\int_{0}^{1} \frac{C}{\lambda(t)^{4}} \frac{(\overline{v}_{1}(t,R))^{4}}{R^{4}} \frac{|\overline{v}_{1}-\overline{v}_{2}|^{2}}{R} dR\\
&\leq \frac{C}{\lambda(t)^{4}}\int_{0}^{1} \left(\left(||\overline{v}_{1}||_{\dot{H}^{1}_{e}}^{4}+||L^{*}L\overline{v}_{1}||_{L^{2}(R dR)}^{4}\right)\left(\log^{2}(\frac{1}{R})+1\right)^{2}\right)\left(||\overline{v}_{1}-\overline{v}_{2}||_{\dot{H}^{1}_{e}}^{2}+||L^{*}L(\overline{v}_{1}-\overline{v}_{2})||_{L^{2}(R dR)}^{2}\right)\\
&\cdot R \left(\log^{2}(\frac{1}{R})+1\right) dR\\
&\leq C \lambda(t)^{2} ||\langle \omega \lambda(t)^{2}\rangle y_{1}||_{L^{2}(\rho(\omega \lambda(t)^{2}) d\omega)}^{4} ||\langle \omega \lambda(t)^{2}\rangle (y_{1}-y_{2})||_{L^{2}(\rho(\omega \lambda(t)^{2}) d\omega)}^{2}\end{split}\end{equation}\\
\\
The next term to consider is
\begin{equation}\begin{split} &\int_{0}^{1} \frac{C}{\lambda(t)^{4}} ||\frac{\overline{v}_{1}-\overline{v}_{2}}{\sqrt{R}}||_{L^{\infty}}^{2} ||\frac{|\overline{v}_{1}|+|\overline{v}_{2}|}{\sqrt{R}}||_{L^{\infty}}^{2} \frac{(L\overline{v}_{2})^{2}}{R^{2}} R dR\\
&\leq \frac{C}{\lambda(t)^{4}} \left(||\overline{v}_{1}-\overline{v}_{2}||_{\dot{H}^{1}_{e}}^{2} + ||L^{*}L(\overline{v}_{1}-\overline{v}_{2})||_{L^{2}(R dR)}^{2}\right)\\
&\cdot\left(||\overline{v}_{1}||_{\dot{H}^{1}_{e}}^{2}+||L^{*}L\overline{v}_{1}||_{L^{2}(R dR)}^{2}+||\overline{v}_{2}||_{\dot{H}^{1}_{e}}^{2}+||L^{*}L\overline{v}_{2}||_{L^{2}(R dR)}^{2}\right)\int_{0}^{1} \frac{(L\overline{v}_{2}(t,R))^{2}}{R^{2}} R dR\\
&\leq C \lambda(t)^{2} ||\langle \omega \lambda(t)^{2}\rangle (y_{1}-y_{2})||_{L^{2}(\rho(\omega \lambda(t)^{2}) d\omega)}^{2} \cdot \left(||\langle \omega \lambda(t)^{2}\rangle y_{1}||_{L^{2}(\rho(\omega \lambda(t)^{2}) d\omega)}^{2}+||\langle \omega \lambda(t)^{2}\rangle y_{2}||_{L^{2}(\rho(\omega \lambda(t)^{2}) d\omega)}^{2}\right) \\
&\cdot||\langle \omega \lambda(t)^{2}\rangle y_{2}||_{L^{2}(\rho(\omega \lambda(t)^{2}) d\omega)}^{2}\end{split}\end{equation}\\
\\
Next, we have
\begin{equation} \begin{split} & \frac{C}{\lambda(t)^{4}} \int_{0}^{1} \left( |\frac{\overline{v}_{1}-\overline{v}_{2}}{R}|\left(\frac{|\overline{v}_{1}|+|\overline{v}_{2}|}{R}\right) \frac{|\overline{v}_{2}|}{R}\right)^{2} R dR\\
&\leq \frac{C}{\lambda(t)^{4}} \int_{0}^{1} \left(||\overline{v}_{1}-\overline{v}_{2}||_{\dot{H}^{1}_{e}}^{2}+||L^{*}L(\overline{v}_{1}-\overline{v}_{2})||_{L^{2}(R dR)}^{2}\right)\\
&\cdot\left(||\overline{v}_{1}||_{\dot{H}^{1}_{e}}^{4}+||L^{*}L\overline{v}_{1}||_{L^{2}(R dR)}^{4}+||\overline{v}_{2}||_{\dot{H}^{1}_{e}}^{4}+||L^{*}L\overline{v}_{2}||_{L^{2}(R dR)}^{4}\right) \left(\log^{2}(\frac{1}{R})+1\right)^{3} R dR\\
&\leq C \lambda(t)^{2} ||\langle \omega \lambda(t)^{2}\rangle (y_{1}-y_{2})||_{L^{2}(\rho(\omega \lambda(t)^{2}) d\omega)}^{2} \cdot \left(||\langle \omega \lambda(t)^{2}\rangle y_{1}||_{L^{2}(\rho(\omega \lambda(t)^{2}) d\omega)}^{4}+||\langle \omega \lambda(t)^{2} \rangle y_{2}||_{L^{2}(\rho(\omega \lambda(t)^{2}) d\omega)}^{4}\right)\end{split}\end{equation}
For the second line of \eqref{drndiff},
\begin{equation}\begin{split} & \frac{C}{\lambda(t)^{4}} \int_{0}^{1} \left(\frac{\overline{v}_{1}^{4}}{R^{4}}+\frac{\overline{v}_{2}^{4}}{R^{4}}\right)\frac{|\overline{v}_{1}-\overline{v}_{2}|^{2}}{R^{2}} R dR\\
&\leq \frac{C}{\lambda(t)^{4}}\left(||\overline{v}_{1}||_{\dot{H}^{1}_{e}}^{4}+||L^{*}L\overline{v}_{1}||_{L^{2}(R dR)}^{4}+||\overline{v}_{2}||_{\dot{H}^{1}_{e}}^{4}+||L^{*}L\overline{v}_{2}||_{L^{2}(R dR)}^{4}\right)\\
&\cdot \left(||\overline{v}_{1}-\overline{v}_{2}||_{\dot{H}^{1}_{e}}^{2}+||L^{*}L(\overline{v}_{1}-\overline{v}_{2})||_{L^{2}(R dR)}^{2}\right)\int_{0}^{1} \left(\log^{2}(\frac{1}{R})+1\right)^{3} R dR\\
&\leq C \lambda(t)^{2}\left(||\langle \omega \lambda(t)^{2}\rangle y_{1}||_{L^{2}(\rho(\omega \lambda(t)^{2}) d\omega)}^{4}+||\langle \omega \lambda(t)^{2}\rangle y_{2}||_{L^{2}(\rho(\omega \lambda(t)^{2}) d\omega)}^{4}\right)||\langle \omega \lambda(t)^{2}\rangle (y_{1}-y_{2})||_{L^{2}(\rho(\omega \lambda(t)^{2}) d\omega)}^{2}\end{split}\end{equation}
Next, we treat the third line of \eqref{drndiff}, 
\begin{equation} \begin{split} &\int_{0}^{1} \frac{1}{\lambda(t)^{4}} \left(1+||\frac{v_{corr}(t,R\lambda(t))}{R}||_{L^{\infty}}^{2}\right)\\
&\cdot \left(||\overline{v}_{1}||_{\dot{H}^{1}_{e}}^{2}\frac{(L(\overline{v}_{1}-\overline{v}_{2}))^{2}}{R^{2}} +||\overline{v}_{1}-\overline{v}_{2}||_{\dot{H}^{1}_{e}}^{2} \frac{(L\overline{v}_{2})^{2}}{R^{2}} + \frac{(\overline{v}_{1})^{2}+(\overline{v}_{2})^{2}}{R^{2}} \frac{|\overline{v}_{1}-\overline{v}_{2}|^{2}}{R^{2}}\right)R dR\\
&\leq C \left(1+||\frac{v_{corr}(t,R\lambda(t))}{R}||_{L^{\infty}}^{2}\right) ||\langle \omega \lambda(t)^{2}\rangle (y_{1}-y_{2})||_{L^{2}(\rho(\omega \lambda(t)^{2}) d\omega)}^{2}\\
&\cdot \left(||\langle \omega \lambda(t)^{2}\rangle y_{2}||_{L^{2}(\rho(\omega \lambda(t)^{2}) d\omega)}^{2}+||\langle \omega \lambda(t)^{2}\rangle y_{1}||_{L^{2}(\rho(\omega \lambda(t)^{2}) d\omega)}^{2}\right)\end{split}\end{equation}
The fourth line of \eqref{drndiff} is treated as follows.
\begin{equation} \begin{split} &\int_{0}^{1} \left(\frac{|\overline{v}_{1}|+|\overline{v}_{2}|}{R} \frac{|\overline{v}_{1}-\overline{v}_{2}|}{R\lambda(t)^{2}} \left(1+||\partial_{R}(v_{corr}(t,R\lambda(t)))||_{L^{\infty}}+||\frac{v_{corr}(t,R\lambda(t))}{R}||_{L^{\infty}}\right)\right)^{2} R dR\\
&\leq \frac{C}{\lambda(t)^{4}} \left(1+||\partial_{R}(v_{corr}(t,R\lambda(t)))||_{L^{\infty}}^{2}+||\frac{v_{corr}(t,R\lambda(t))}{R}||_{L^{\infty}}^{2}\right)\\
&\cdot \left(||\overline{v}_{1}||_{\dot{H}^{1}_{e}}^{2}+||L^{*}L\overline{v}_{1}||_{L^{2}(R dR)}^{2}+||\overline{v}_{2}||_{\dot{H}^{1}_{e}}^{2}+||L^{*}L\overline{v}_{2}||_{L^{2}(R dR)}^{2}\right)\\
&\cdot\left(||\overline{v}_{1}-\overline{v}_{2}||_{\dot{H}^{1}_{e}}^{2}+||L^{*}L(\overline{v}_{1}-\overline{v}_{2})||_{L^{2}(R dR)}^{2}\right)\int_{0}^{1} \left(\log^{2}(\frac{1}{R})+1\right)^{2} R dR\\
&\leq C \left(1+||\partial_{R}(v_{corr}(t,R\lambda(t)))||_{L^{\infty}}^{2}+||\frac{v_{corr}(t,R\lambda(t))}{R}||_{L^{\infty}}^{2}\right)\\
&\cdot ||\langle \omega \lambda(t)^{2}\rangle (y_{1}-y_{2})||_{L^{2}(\rho(\omega \lambda(t)^{2}) d\omega)}^{2} \left(||\langle \omega \lambda(t)^{2}\rangle y_{1}||_{L^{2}(\rho(\omega \lambda(t)^{2}) d\omega)}^{2}+||\langle \omega \lambda(t)^{2}\rangle y_{2}||_{L^{2}(\rho(\omega \lambda(t)^{2}) d\omega)}^{2}\right)\end{split}\end{equation}
We combine the above estimates to get
\begin{equation} \begin{split} &\int_{0}^{1}\left(\partial_{R}\left((N(u_{1})-N(u_{2}))(t,R\lambda(t))\right)\right)^{2} R dR \\
&\leq C \lambda(t)^{2} \left(||\langle \omega \lambda(t)^{2}\rangle y_{1}||_{L^{2}(\rho(\omega \lambda(t)^{2}) d\omega)}^{4}+||\langle \omega \lambda(t)^{2}\rangle y_{2}||_{L^{2}(\rho(\omega \lambda(t)^{2}) d\omega)}^{4}\right) ||\langle \omega \lambda(t)^{2}\rangle(y_{1}-y_{2})||_{L^{2}(\rho(\omega \lambda(t)^{2}) d\omega)}^{2}\\
&+C \left(1+||\partial_{R}(v_{corr}(t,R\lambda(t)))||_{L^{\infty}}^{2}+||\frac{v_{corr}(t,R\lambda(t))}{R}||_{L^{\infty}}^{2}\right)\\
&\cdot||\langle \omega \lambda(t)^{2}\rangle (y_{1}-y_{2})||_{L^{2}(\rho(\omega \lambda(t)^{2}) d\omega)}^{2}\left(||\langle \omega \lambda(t)^{2}\rangle y_{2}||_{L^{2}(\rho(\omega \lambda(t)^{2}) d\omega)}^{2}+||\langle \omega \lambda(t)^{2}\rangle y_{1}||_{L^{2}(\rho(\omega \lambda(t)^{2}) d\omega)}^{2}\right)\end{split}\end{equation}\\
\\
When $R\geq 1$, we can estimate \eqref{drndiff} by
\begin{equation} \begin{split} &|\partial_{R}\left((N(u_{1})-N(u_{2}))(t,R\lambda(t))\right)|\\
&\leq \frac{C}{\lambda(t)^{2}} ||\overline{v}_{1}||^{2}_{\dot{H}^{1}_{e}}|L(\overline{v}_{1}-\overline{v}_{2})|+\frac{C}{\lambda(t)^{2}} ||\overline{v}_{1}||_{\dot{H}^{1}_{e}}^{2} |\overline{v}_{1}-\overline{v}_{2}| +\frac{C}{\lambda(t)^{2}} ||\overline{v}_{1}-\overline{v}_{2}||_{\dot{H}^{1}_{e}} \left(||\overline{v}_{1}||_{\dot{H}^{1}_{e}}+||\overline{v}_{2}||_{\dot{H}^{1}_{e}}\right)\left(|L\overline{v}_{2}|+|\overline{v}_{2}|\right)\\
&+\frac{C}{\lambda(t)^{2}} \left(||\overline{v}_{1}||^{2}_{\dot{H}^{1}_{e}}+||\overline{v}_{2}||^{2}_{\dot{H}^{1}_{e}}\right) |\overline{v}_{1}-\overline{v}_{2}|\\
&+\frac{C}{\lambda(t)^{2}} \left(||\overline{v}_{1}||_{\dot{H}^{1}_{e}}+||\overline{v}_{2}||_{\dot{H}^{1}_{e}}\right) |\overline{v}_{1}-\overline{v}_{2}| \left(1+||\partial_{R}(v_{corr}(t,R\lambda(t)))||_{L^{\infty}}+||\frac{v_{corr}(t,R\lambda(t))}{R}||_{L^{\infty}}\right)\\
&+\frac{C}{\lambda(t)^{2}} \left(1+||\frac{v_{corr}(t,R\lambda(t))}{R}||_{L^{\infty}}\right)\left(||\overline{v}_{1}||_{\dot{H}^{1}_{e}}|L(\overline{v}_{1}-\overline{v}_{2})| + ||\overline{v}_{1}-\overline{v}_{2}||_{\dot{H}^{1}_{e}} |L\overline{v}_{2}| + \left(||\overline{v}_{1}||_{\dot{H}^{1}_{e}}+||\overline{v}_{2}||_{\dot{H}^{1}_{e}}\right) |\overline{v}_{1}-\overline{v}_{2}|\right) \end{split}\end{equation}\\
\\
Taking the $L^{2}$ norm over $R\geq 1$, and combining with our previous estimates, we get
\begin{equation} \begin{split} &||\partial_{R}\left((N(u_{1})-N(u_{2}))(t,R\lambda(t))\right)||_{L^{2}(R dR)}\\
&\leq C \lambda(t) \left(||\langle \omega \lambda(t)^{2} \rangle y_{1}||_{L^{2}(\rho(\omega \lambda(t)^{2}) d\omega)}^{2}+||\langle \omega \lambda(t)^{2} \rangle y_{2}||_{L^{2}(\rho(\omega \lambda(t)^{2}) d\omega)}^{2}\right) ||\langle \omega \lambda(t)^{2}\rangle (y_{1}-y_{2})||_{L^{2}(\rho(\omega \lambda(t)^{2}) d\omega)}\\
&+C \left(1+||\frac{v_{corr}(t,R\lambda(t))}{R}||_{L^{\infty}} + ||\partial_{R}(v_{corr}(t,R\lambda(t)))||_{L^{\infty}}\right) \\
&\cdot \left(||\langle \omega \lambda(t)^{2} \rangle y_{2}||_{L^{2}(\rho(\omega \lambda(t)^{2}) d\omega)}+||\langle \omega \lambda(t)^{2} \rangle y_{1}||_{L^{2}(\rho(\omega \lambda(t)^{2}) d\omega)}\right)||\langle \omega \lambda(t)^{2}\rangle (y_{1}-y_{2})||_{L^{2}(\rho(\omega \lambda(t)^{2}) d\omega)}\end{split}\end{equation}\\
\\
It only remains to consider the following.
\begin{equation}\label{ndiffoverr}\begin{split} |\frac{N(u_{1})(t,R\lambda(t))-N(u_{2})(t,R\lambda(t))}{R}|&\leq  C \frac{|\overline{v}_{1}-\overline{v}_{2}|\left((\overline{v}_{1})^{2}+(\overline{v}_{2})^{2}\right)}{R^{3}\lambda(t)^{2}}\\
&+\frac{C\left(|\overline{v}_{1}|+|\overline{v}_{2}|\right)|\overline{v}_{1}-\overline{v}_{2}|}{R^{3}\lambda(t)^{2}} \left(|Q_{1}(R)|+|v_{corr}(t,R\lambda(t))|\right)\end{split}\end{equation}\\
\\
So,
\begin{equation}\begin{split} &\int_{0}^{1} \frac{\left(N(u_{1})(t,R\lambda(t))-N(u_{2})(t,R\lambda(t))\right)^{2}}{R^{2}} R dR \\
&\leq \frac{C}{\lambda(t)^{4}} \int_{0}^{1} \left(||\overline{v}_{1}-\overline{v}_{2}||_{\dot{H}^{1}_{e}}^{2}+||L^{*}L(\overline{v}_{1}-\overline{v}_{2})||_{L^{2}(R dR)}^{2}\right)\\
&\cdot\left(||\overline{v}_{1}||_{\dot{H}^{1}_{e}}^{4}+||L^{*}L\overline{v}_{1}||_{L^{2}(R dR)}^{4}+||\overline{v}_{2}||_{\dot{H}^{1}_{e}}^{4}+||L^{*}L\overline{v}_{2}||_{L^{2}(R dR)}^{4}\right)\left(\log^{2}(\frac{1}{R})+1\right)^{3} R dR\\
&+\frac{C}{\lambda(t)^{4}}\left(1+||\frac{v_{corr}(t,R\lambda(t))}{R}||_{L^{\infty}}^{2}\right) \int_{0}^{1} \left(||\overline{v}_{1}||_{\dot{H}^{1}_{e}}^{2}+||L^{*}L\overline{v}_{1}||_{L^{2}(R dR)}^{2}+||\overline{v}_{2}||_{\dot{H}^{1}_{e}}^{2}+||L^{*}L\overline{v}_{2}||_{L^{2}(R dR)}^{2}\right)\\
&\cdot\left(||\overline{v}_{1}-\overline{v}_{2}||_{\dot{H}^{1}_{e}}^{2}+||L^{*}L(\overline{v}_{1}-\overline{v}_{2})||_{L^{2}(R dR)}^{2}\right) \left(\log^{2}(\frac{1}{R})+1\right)^{2} R dR\\
&\leq C \lambda(t)^{2} ||\langle \omega \lambda(t)^{2} \rangle (y_{1}-y_{2})||_{L^{2}(\rho(\omega \lambda(t)^{2}) d\omega)}^{2} \left(||\langle \omega \lambda(t)^{2} \rangle y_{1}||_{L^{2}(\rho(\omega \lambda(t)^{2}) d\omega)}^{4}+||\langle \omega \lambda(t)^{2} \rangle y_{2}||_{L^{2}(\rho(\omega \lambda(t)^{2}) d\omega)}^{4}\right)\\
&+C \left(1+||\frac{v_{corr}(t,R\lambda(t))}{R}||_{L^{\infty}}^{2}\right)\\
&\cdot ||\langle \omega \lambda(t)^{2} \rangle (y_{1}-y_{2})||_{L^{2}(\rho(\omega \lambda(t)^{2}) d\omega)}^{2} \left(||\langle \omega \lambda(t)^{2} \rangle y_{1}||_{L^{2}(\rho(\omega \lambda(t)^{2}) d\omega)}^{2}+||\langle \omega \lambda(t)^{2} \rangle y_{2}||_{L^{2}(\rho(\omega \lambda(t)^{2}) d\omega)}^{2}\right) \end{split}\end{equation}\\
\\
We return to \eqref{ndiffoverr}, and study
\begin{equation} \begin{split} &\left(\int_{1}^{\infty} \frac{\left((N(u_{1})-N(u_{2}))(t,R\lambda(t))\right)^{2}}{R^{2}} R dR\right)^{1/2}\\
&\leq \frac{C}{\lambda(t)^{2}} ||\overline{v}_{1}-\overline{v}_{2}||_{L^{2}(R dR)} \left(||\overline{v}_{1}||_{\dot{H}^{1}_{e}}^{2}+||\overline{v}_{2}||_{\dot{H}^{1}_{e}}^{2}\right)\\
&+\frac{C}{\lambda(t)^{2}} \left(||\overline{v}_{1}||_{\dot{H}^{1}_{e}}+||\overline{v}_{2}||_{\dot{H}^{1}_{e}}\right)||\overline{v}_{1}-\overline{v}_{2}||_{L^{2}(R dR)} \left(1+||\frac{v_{corr}(t,R\lambda(t))}{R}||_{L^{\infty}}\right)\end{split}\end{equation}\\
\\
We now combine the above estimates, translating between norms of $v_{i}$ and norms of $y_{i}$ as previously, and use the fact that $||y_{i}||_{Z} \leq 1$, as well as the estimates of the ansatz to get
\begin{equation}\begin{split} &||\lambda(t) \sqrt{\omega} \mathcal{F}(\sqrt{\cdot}(N(u_{1})-N(u_{2}))(t,\cdot \lambda(t)))(\omega \lambda(t)^{2})||_{L^{2}(\rho(\omega \lambda(t)^{2})d\omega)}\\
&\leq C \frac{||y_{1}-y_{2}||_{Z}}{t^{4}\log^{\epsilon -b}(t)}\end{split}\end{equation}\\
\\
We now return to \eqref{Tislip} to get
\begin{equation} \begin{split} ||(T(y_{1})-T(y_{2}))(t)||_{L^{2}(\rho(\omega \lambda(t)^{2})d\omega)}&\leq C \int_{t}^{\infty} x \left(\frac{\lambda(t)}{\lambda(x)}\right) \left(\frac{||y_{1}-y_{2}||_{Z}}{x^{4}\log^{1+\frac{\epsilon}{2}}(x)}\right) dx\\
&+C \int_{t}^{\infty} x \left(\frac{\lambda(t)}{\lambda(x)}\right) \frac{||y_{1}-y_{2}||_{Z}}{x^{4}\log^{3\frac{\epsilon}{2}-2b}(x)} dx\\
&+C \int_{t}^{\infty} x \left(\frac{\lambda(t)}{\lambda(x)}\right) \frac{||y_{1}-y_{2}||_{Z}}{x^{4}\log^{\epsilon-b}(x)} dx\\
&\leq C ||y_{1}-y_{2}||_{Z} \frac{\log^{-1}(t)+\log^{-\epsilon +2b}(t)+\log^{-\frac{\epsilon}{2}+b}(t)}{t^{2}\log^{\frac{\epsilon}{2}}(t)}\end{split}\end{equation}\\
\\
\begin{equation}\begin{split} \partial_{t}(T(y_{1})-T(y_{2})) &=-\int_{t}^{\infty} \cos((t-x)\sqrt{\omega})F_{2}(y_{1}-y_{2}) dx\\
&+\int_{t}^{\infty} \cos((t-x)\sqrt{\omega}) \mathcal{F}(\sqrt{\cdot} (F_{3}(y_{1})-F_{3}(y_{2})(x,\cdot \lambda(x))))(\omega \lambda(x)^{2}) dx\end{split}\end{equation}\\
\\
So, the same argument as in the previous estimate gives
\begin{equation} ||\partial_{t}(T(y_{1})-T(y_{2}))||_{L^{2}(\rho(\omega \lambda(t)^{2})d\omega)} \leq C ||y_{1}-y_{2}||_{Z} \frac{\log^{-1}(t)+\log^{-\epsilon+2b}(t)+\log^{-\frac{\epsilon}{2}+b}(t)}{t^{3}\log^{\frac{\epsilon}{2}}(t)}\end{equation}\\
\\
Similarly, 
\begin{equation}\begin{split} \sqrt{\omega} \lambda(t)(T(y_{1})-T(y_{2})) &=-\int_{t}^{\infty}\lambda(t) \sin((t-x)\sqrt{\omega})F_{2}(y_{1}-y_{2}) dx\\
&+\lambda(t)\int_{t}^{\infty} \sin((t-x)\sqrt{\omega}) \mathcal{F}(\sqrt{\cdot} (F_{3}(y_{1})-F_{3}(y_{2})(x,\cdot \lambda(x))))(\omega \lambda(x)^{2}) dx\end{split}\end{equation}\\
\\
and the identical argument gives us
\begin{equation} ||\sqrt{\omega} \lambda(t) (T(y_{1})-T(y_{2}))(t)||_{L^{2}(\rho(\omega \lambda(t)^{2})d\omega)} \leq C ||y_{1}-y_{2}||_{Z} \frac{\log^{-1}(t)+\log^{-\epsilon+2b}(t)+\log^{-\frac{\epsilon}{2}+b}(t)}{t^{3}\log^{\frac{\epsilon}{2}+b}(t)}\end{equation}\\
\\
Next, we have 
\begin{equation}\begin{split} \sqrt{\omega} \lambda(t)\partial_{t}(T(y_{1})-T(y_{2})) &=-\int_{t}^{\infty}\left(\frac{\lambda(t)}{\lambda(x)}\right)  \cos((t-x)\sqrt{\omega})\lambda(x) \sqrt{\omega}F_{2}(y_{1}-y_{2}) dx\\
&+\int_{t}^{\infty} \left(\frac{\lambda(t)}{\lambda(x)}\right) \cos((t-x)\sqrt{\omega}) \lambda(x) \sqrt{\omega} \mathcal{F}(\sqrt{\cdot} (F_{3}(y_{1})-F_{3}(y_{2})(x,\cdot \lambda(x))))(\omega \lambda(x)^{2}) dx\end{split}\end{equation}\\
\\
So, \begin{equation}\begin{split}&||\sqrt{\omega} \lambda(t)\partial_{t}(T(y_{1})-T(y_{2}))(t)||_{L^{2}(\rho(\omega \lambda(t)^{2})d\omega)}\\
&\leq C \int_{t}^{\infty} \left(\frac{\lambda(t)}{\lambda(x)}\right)^{2} \frac{||y_{1}-y_{2}||_{Z}}{x^{4}\log^{1+\frac{\epsilon}{2}}(x)} dx\\
&+C \int_{t}^{\infty}\left(\frac{\lambda(t)}{\lambda(x)}\right)^{2} \frac{||y_{1}-y_{2}||_{Z}}{x^{4}\log^{\frac{3 \epsilon}{2}-2b}(x)} dx\\
&+C \int_{t}^{\infty} \left(\frac{\lambda(t)}{\lambda(x)}\right)^{2} \frac{||y_{1}-y_{2}||_{Z}}{x^{4}\log^{\epsilon-b}(x)} dx\\
&\leq C ||y_{1}-y_{2}||_{Z} \frac{\log^{-1}(t)+\log^{-\epsilon+2b}(t)+\log^{-\frac{\epsilon}{2}+b}(t)}{t^{3}\log^{\frac{\epsilon}{2}}(t)}\end{split}\end{equation}\\
\\
The identical procedure shows that
\begin{equation}\begin{split}&||\omega \lambda(t)^{2}(T(y_{1})-T(y_{2}))(t)||_{L^{2}(\rho(\omega \lambda(t)^{2})d\omega)}\\
&\leq C ||y_{1}-y_{2}||_{Z} \frac{\log^{-1}(t)+\log^{-\epsilon+2b}(t)+\log^{-\frac{\epsilon}{2}+b}(t)}{t^{3}\log^{\frac{\epsilon}{2}+b}(t)}\end{split}\end{equation}\\
\\
Thus, $T$ is a strict contraction on $\overline{B}_{1}(0) \subset Z$, for $T_{0}$ large enough; so, $T$ has a fixed point, say $y_{0}$, in $\overline{B}_{1}(0)\subset Z$ \qed

\section{The energy of the solution, and its decomposition as in Theorem 1.1}
Let us define
$$v_{6}(t,r):=\begin{cases}\sqrt{\frac{\lambda(t)}{r}}\left(\mathcal{F}^{-1}(y_{0}(t,\frac{\cdot}{\lambda(t)^{2}}))\right)(\frac{r}{\lambda(t)}), \quad r > 0\\
0, \quad r=0\end{cases}$$
Note that $v_{6}(t,\cdot) \in C^{0}([0,\infty))$, by the same argument as in Lemma \eqref{ytouregularity}.
Inspecting the derivation of \eqref{ylinearproblem}, we see that we have a solution to \eqref{wavemaps}:
\begin{equation}\label{finalsoln}u(t,r) = Q_{\frac{1}{\lambda(t)}}(r) + \sum_{k=1}^{6} v_{k}(t,r)  \end{equation}
Here, we study the energy, \eqref{equivariantenergy}, of our solution, and prove that it has a decomposition as in the main theorem statement. First, we note that $\partial_{t}Q_{\frac{1}{\lambda(t)}} \not \in L^{2}((0,\infty), rdr)$, because $\phi_{0} \not \in L^{2}((0,\infty),rdr)$, so we have to first capture a delicate cancellation between the large $r$ behavior of $\partial_{t}Q_{\frac{1}{\lambda(t)}}$ and $\partial_{t}v_{1}$ before we can even show that the solution has finite energy. We consider the region $r \geq t$, and use the representation formula for $v_{1}$, \eqref{v1formula}, which implies 
\begin{equation}\label{dtv1forenergy}\begin{split} \partial_{t}v_{1}(t,r) &= \int_{t}^{\infty} ds \frac{\lambda'''(s)}{r} \int_{0}^{s-t} \frac{\rho d\rho}{\sqrt{(s-t)^{2}-\rho^{2}}} \left(1+\frac{r^{2}-1-\rho^{2}}{\sqrt{(1+r^{2}+\rho^{2})^{2}-4r^{2}\rho^{2}}}\right)\\
&=\int_{t}^{t+\frac{r}{6}} ds \frac{\lambda'''(s)}{r} \int_{0}^{s-t} \frac{\rho d\rho}{\sqrt{(s-t)^{2}-\rho^{2}}} \left(1+\frac{r^{2}-1-\rho^{2}}{\sqrt{(1+r^{2}+\rho^{2})^{2}-4r^{2}\rho^{2}}}\right)\\
&+\int_{t+\frac{r}{6}}^{\infty} ds \frac{\lambda'''(s)}{r} \int_{0}^{s-t} \frac{\rho d\rho}{(s-t)} \left(1+\frac{r^{2}-1-\rho^{2}}{\sqrt{(1+r^{2}+\rho^{2})^{2}-4r^{2}\rho^{2}}}\right)\\
&+\int_{t+\frac{r}{6}}^{\infty} ds \frac{\lambda'''(s)}{r} \int_{0}^{s-t} \rho d\rho \left(\frac{1}{\sqrt{(s-t)^{2}-\rho^{2}}}-\frac{1}{(s-t)}\right) \left(1+\frac{r^{2}-1-\rho^{2}}{\sqrt{(1+r^{2}+\rho^{2})^{2}-4r^{2}\rho^{2}}}\right)\end{split}\end{equation}
For the second line of \eqref{dtv1forenergy}, we have
\begin{equation}\label{dtv1preciseenergy}\begin{split} &\int_{t}^{t+\frac{r}{6}} ds \frac{\lambda'''(s)}{r} \int_{0}^{s-t} \frac{\rho d\rho}{\sqrt{(s-t)^{2}-\rho^{2}}} \left(1+\frac{r^{2}-1-\rho^{2}}{\sqrt{(1+r^{2}+\rho^{2})^{2}-4r^{2}\rho^{2}}}\right) \\
&= \int_{t}^{t+\frac{r}{6}} \frac{ds \lambda'''(s)}{r} \int_{0}^{s-t} \frac{\rho d\rho}{\sqrt{(s-t)^{2}-\rho^{2}}} \left(1+1+O\left(\frac{r^{2}}{(r^{2}-1-\rho^{2})^{2}}\right)\right) \\
&=2 \int_{t}^{t+\frac{r}{6}} \frac{\lambda'''(s)}{r}(s-t) ds +E_{\partial_{t}v_{1}}(t,r)\end{split}\end{equation}
where
\begin{equation} |E_{\partial_{t} v_{1}}(t,r)| \leq \frac{C}{r^{3}} \int_{t}^{t+\frac{r}{6}} |\lambda'''(s)| (s-t) ds \leq \frac{C}{r^{3} t \log^{b+1}(t)}, \quad r \geq t\end{equation}
For the third line of \eqref{dtv1forenergy}, we have
\begin{equation} \begin{split} &|\int_{t+\frac{r}{6}}^{\infty} ds \frac{\lambda'''(s)}{r} \int_{0}^{s-t} \frac{\rho d\rho}{(s-t)} \left(1+\frac{r^{2}-1-\rho^{2}}{\sqrt{(1+r^{2}+\rho^{2})^{2}-4r^{2}\rho^{2}}}\right)|\\
&\leq C \int_{t+\frac{r}{6}}^{\infty} \frac{|\lambda'''(s)| ds}{r (s-t)} \int_{0}^{\infty} \rho d\rho\left(1+\frac{r^{2}-1-\rho^{2}}{\sqrt{(1+r^{2}+\rho^{2})^{2}-4r^{2}\rho^{2}}}\right)\leq \frac{C}{r^{2} \log^{b+1}(r)}, \quad r \geq t\end{split}\end{equation}
Finally, the fourth line of \eqref{dtv1forenergy} is treated as follows:
\begin{equation}\begin{split}&|\int_{t+\frac{r}{6}}^{\infty} ds \frac{\lambda'''(s)}{r} \int_{0}^{s-t} \rho d\rho \left(\frac{1}{\sqrt{(s-t)^{2}-\rho^{2}}}-\frac{1}{(s-t)}\right) \left(1+\frac{r^{2}-1-\rho^{2}}{\sqrt{(1+r^{2}+\rho^{2})^{2}-4r^{2}\rho^{2}}}\right)|\\
&\leq \frac{C}{r^{3} \log^{b+1}(r)} \cdot \frac{1}{r} \int_{t+\frac{r}{6}}^{\infty} ds \int_{0}^{s-t} \rho d\rho \left(\frac{1}{\sqrt{(s-t)^{2}-\rho^{2}}}-\frac{1}{(s-t)}\right) \left(1+\frac{r^{2}-1-\rho^{2}}{\sqrt{(1+r^{2}+\rho^{2})^{2}-4r^{2}\rho^{2}}}\right)\\
&\leq \frac{C}{r^{4} \log^{b+1}(r)} \int_{0}^{\infty} \rho d\rho \left(1+\frac{r^{2}-1-\rho^{2}}{\sqrt{(1+r^{2}+\rho^{2})^{2}-4r^{2}\rho^{2}}}\right) \int_{\rho+t}^{\infty} ds \left(\frac{1}{\sqrt{(s-t)^{2}-\rho^{2}}}-\frac{1}{(s-t)}\right)\\
&\leq \frac{C}{r^{2} \log^{b+1}(r)}, \quad r \geq t\end{split}\end{equation}
Finally, we further treat the first term on the last line of \eqref{dtv1preciseenergy}
\begin{equation}\begin{split}2 \int_{t}^{t+\frac{r}{6}} \frac{\lambda'''(s)}{r}(s-t) ds &= \frac{2}{r}\left(\frac{r}{6} \lambda''(t+\frac{r}{6})-\lambda'(t+\frac{r}{6})+\lambda'(t)\right)\\
&=\frac{2\lambda'(t)}{r} + E_{\partial_{t}v_{1},1}(t,r)\end{split}\end{equation}
where
\begin{equation} |E_{\partial_{t}v_{1},1}(t,r)| \leq \frac{C}{r^{2} \log^{b+1}(r)}, \quad r \geq t\end{equation}
Then, we note that
\begin{equation} \partial_{t}Q_{\frac{1}{\lambda(t)}}(r) = \frac{-2 r \lambda'(t)}{r^{2}+\lambda(t)^{2}}\end{equation}
So,
\begin{equation} |\partial_{t}\left(Q_{\frac{1}{\lambda(t)}}(r) + v_{1}(t,r)\right)| \leq \frac{C}{r^{2} \log^{b+1}(r)}, \quad r \geq t\end{equation}
Using the estimates on $\partial_{t}v_{1}$, we then get
\begin{equation} ||\partial_{t}\left(Q_{\frac{1}{\lambda(t)}}(r) +v_{1}(t,r)\right)||^{2}_{L^{2}(r dr)} \leq \frac{C}{t^{2} \log^{2b}(t)}\end{equation}
Then, we recall 
$$E(u,v) = \pi \left(||v||_{L^{2}(r dr)}^{2}+||u||_{\dot{H}^{1}_{e}}^{2}\right)$$
and note that energy estimates for the equations solved by $v_{k}$ for $k=3,4,5$ give
\begin{equation}\begin{split}E(v_{3}(t),\partial_{t}v_{3}(t)) &\leq C \left(\int_{t}^{\infty} ||F_{0,1}(s)||_{L^{2}(r dr)} ds\right)^{2} \leq \frac{C \log(\log(t))}{t^{2} \log^{2b+2}(t)}\end{split}\end{equation}

\begin{equation}\begin{split} E(v_{4}(t),\partial_{t}v_{4}) &\leq C \left(\int_{t}^{\infty} ||v_{4,c}(s)||_{L^{2}(r dr)} ds\right)^{2}\leq \frac{C}{t^{2} \log^{4N+4b}(t)}\end{split}\end{equation}

\begin{equation}\begin{split} E(v_{5}(t),\partial_{t}v_{5}(t)) &\leq C \left(\int_{t}^{\infty} ||N_{2}(f_{v_{5}})(s)||_{L^{2}(r dr)} ds\right)^{2} \leq \frac{C \log^{6}(t)}{t^{7/2}}\end{split}\end{equation}

Next, we consider $v_{2}$ for $b \neq 1$. By Plancherel, we have
\begin{equation}\label{v2energy}\begin{split} &E(v_{2}(t),\partial_{t}v_{2}(t))=E(v_{2}(0),\partial_{t}v_{2}(0)) = \frac{16 b^{2}}{\pi (b-1)^{2}} \int_{0}^{\infty} \frac{(\chi_{\leq \frac{1}{4}}(\xi))^{2}}{\log^{2b-2}(\frac{1}{\xi})} d\xi\\
& \frac{16 b^{2}}{\pi(b-1)^{2}} \int_{0}^{\frac{1}{8}} \frac{d\xi}{\log^{2b-2}(\frac{1}{\xi})} \leq E(v_{2},\partial_{t}v_{2}) \leq \frac{16b^{2}}{\pi(b-1)^{2}} \int_{0}^{\frac{1}{4}} \frac{d\xi}{\log^{2b-2}(\frac{1}{\xi})}\end{split}\end{equation}
where we used properties of $\chi_{\leq \frac{1}{4}}$. This gives, for $b \neq 1$,
\begin{equation} \frac{16b^{2} }{\pi(b-1)^{2}}\Gamma(3-2b,\log(8)) \leq E(v_{2},\partial_{t}v_{2}) \leq \frac{16b^{2}}{\pi(b-1)^{2}} \Gamma(3-2b,\log(4))\end{equation}
By inspection of $v_{2,0}$ for $b=1$, we see that $E(v_{2},\partial_{t}v_{2}) < \infty$ for $b=1$.

Using our estimates on $v_{1}$ and $\partial_{r}v_{1}$, we get
\begin{equation} ||v_{1}||^{2}_{\dot{H}^{1}_{e}} \leq \frac{C}{t^{2} \log^{2b}(t)}\end{equation}

Finally, we treat the $v_{6}$ term in \eqref{finalsoln}. Theorem 5.1 of \cite{kst} (the transferrence identity) shows that
\begin{equation}\label{dty0energy}\begin{split} &||\partial_{t}v_{6}||_{L^{2}(r dr)}\leq C ||\frac{\lambda'(t)}{\sqrt{r \lambda(t)}} \mathcal{F}^{-1}(y_{0}(t,\frac{\cdot}{\lambda(t)^{2}}))(\frac{r}{\lambda(t)})||_{L^{2}(r dr)}\\
&+C ||\sqrt{\frac{\lambda(t)}{r}} \mathcal{F}^{-1}(\partial_{1}y_{0}(t,\frac{\cdot}{\lambda(t)^{2}}))(\frac{r}{\lambda(t)})||_{L^{2}(r dr)} + C ||\frac{\lambda'(t)}{\lambda(t)} \sqrt{\frac{\lambda(t)}{r}} \mathcal{F}^{-1}(\mathcal{K}(y_{0}(t,\frac{\cdot}{\lambda(t)^{2}})))(\frac{r}{\lambda(t)})||_{L^{2}(r dr)}\end{split}\end{equation}

The first line of \eqref{dty0energy} is handled by re-scaling, and applying the $L^{2}$ isometry property of $\mathcal{F}(\cdot)$:
\begin{equation}\begin{split} &\int_{0}^{\infty} \frac{(\lambda'(t))^{2}}{ r \lambda(t)} \left(\mathcal{F}^{-1}(y_{0}(t,\frac{\cdot}{\lambda(t)^{2}}))(\frac{r}{\lambda(t)})\right)^{2} r dr = (\lambda'(t))^{2} \lambda(t)^{2} \int_{0}^{\infty} |y_{0}(t,\omega)|^{2} \rho(\omega \lambda(t)^{2}) d\omega\\
&\leq C \frac{\lambda'(t)^{2}\lambda(t)^{2}}{t^{4} \log^{\epsilon}(t)}\end{split}\end{equation}

The second line of \eqref{dty0energy} is treated by  again re-scaling and using the $L^{2}$ isometry property of $\mathcal{F}(\cdot)$, as well as the same estimates on $\mathcal{K}$ that we used while estimating $F_{2}$ to get
\begin{equation}\begin{split} &\int_{0}^{\infty} \frac{\lambda(t)}{r} |\mathcal{F}^{-1}(\partial_{1}y_{0}(t,\frac{\cdot}{\lambda(t)^{2}}))|^{2}(\frac{r}{\lambda(t)}) r dr + \frac{\lambda'(t)^{2}}{\lambda(t)} \int_{0}^{\infty} \frac{1}{r} |\mathcal{F}^{-1}(\mathcal{K}(y_{0}(t,\frac{\cdot}{\lambda(t)})))|^{2}(\frac{r}{\lambda(t)}) r dr\\
&\leq C \lambda(t)^{4} ||\partial_{t}y_{0}(t)||_{L^{2}(\rho(\omega \lambda(t)^{2})d\omega)}^{2} + C \lambda'(t)^{2}\lambda(t)^{2} ||y_{0}(t)||_{L^{2}(\rho(\omega \lambda(t)^{2}) d\omega)}^{2}\\
&\leq \frac{C}{t^{6} \log^{4b+\epsilon}(t)}\end{split}\end{equation}
This gives
\begin{equation} ||\partial_{t}v_{6}||_{L^{2}(r dr)}^{2} \leq \frac{C}{t^{6} \log^{4b+\epsilon}(t)}\end{equation}
From the definition of $v_{6}$, we have (using the same argument used when estimating $F_{2},F_{3}$)
\begin{equation}||v_{6}(t)||_{\dot{H}^{1}_{e}} = ||v_{6}(t,\cdot \lambda(t))||_{\dot{H}^{1}_{e}} \leq C \left(||v_{6}(t,\cdot \lambda(t))||_{L^{2}(R dR)}+||L(v_{6}(t,\cdot \lambda(t)))||_{L^{2}(R dR)}\right)\end{equation}
and
\begin{equation} ||v_{6}(t,\cdot \lambda(t))||_{L^{2}(R dR)} = \lambda(t) ||y_{0}(t)||_{L^{2}(\rho(\omega\lambda(t)^{2})d\omega)}\end{equation}
\begin{equation} ||L(v_{6}(t,\cdot \lambda(t)))||_{L^{2}(R dR)}= \lambda(t) ||\sqrt{\omega}\lambda(t)y_{0}(t)||_{L^{2}(\rho(\omega \lambda(t)^{2})d\omega)}\end{equation}
which gives
\begin{equation}||v_{6}(t)||_{\dot{H}^{1}_{e}} \leq \frac{C}{t^{2} \log^{b+\frac{\epsilon}{2}}(t)}\end{equation} 

Finally, we have
\begin{equation}\begin{split} E_{\text{WM}}(u,\partial_{t}u) &= \pi \left(||\partial_{t}u||_{L^{2}(r dr)}^{2}+||\frac{\sin(u)}{r}||_{L^{2}(r dr)}^{2} + ||\partial_{r}u||_{L^{2}(r dr)}^{2}\right)\end{split}\end{equation}
Then, we use 
\begin{equation}\begin{split} \int_{0}^{\infty} \frac{\sin^{2}(u-Q_{\frac{1}{\lambda(t)}}+Q_{\frac{1}{\lambda(t)}})}{r^{2}} r dr \leq C\left(\int_{0}^{\infty} \frac{\sin^{2}(Q_{\frac{1}{\lambda(t)}})}{r^{2}} r dr + \int_{0}^{\infty} \frac{|u-Q_{\frac{1}{\lambda(t)}}|^{2}}{r^{2}} r dr\right)\end{split}\end{equation}
to get
\begin{equation}\begin{split}&E_{\text{WM}}(u,\partial_{t}u)\\
&\leq C\left(||\partial_{t}\left(Q_{\frac{1}{\lambda(t)}}+v_{1}\right)||_{L^{2}(r dr)}^{2} + \sum_{k=2}^{6}||\partial_{t}v_{k}||_{L^{2}(r dr)}^{2} + ||\partial_{r}Q_{\frac{1}{\lambda(t)}}||_{L^{2}(r dr)}^{2} + ||\frac{\sin(Q_{\frac{1}{\lambda(t)}})}{r}||_{L^{2}(r dr)}^{2} + \sum_{k=1}^{6}||v_{k}||_{\dot{H}^{1}_{e}}^{2}\right)\end{split} \end{equation}
By combining our above estimates, and recalling
$$||\partial_{r}Q_{\frac{1}{\lambda(t)}}||_{L^{2}(r dr)}^{2} + ||\frac{\sin(Q_{\frac{1}{\lambda(t)}})}{r}||_{L^{2}(r dr)}^{2} =4$$
we get
$$E_{\text{WM}}(u,\partial_{t}u) < \infty$$
and
\begin{equation}\label{energydiff}\begin{split} ||\partial_{t}(u-v_{2})||_{L^{2}(r dr)}^{2} + ||u-Q_{\frac{1}{\lambda(t)}}-v_{2}||_{\dot{H}^{1}_{e}}^{2} \leq \frac{C}{t^{2} \log^{2b}(t)}\end{split}\end{equation} 
Finally, we note that the remark after the main theorem concerning the regularity of $v_{6}$ follows from the definition of the space $Z$, the continuity of dilation on $L^{2}$, and lemma 9.1 of \cite{kst}. This completes the proof of the main theorem.
\appendix
\section{Proof of Theorem \ref{genthm}}
In this appendix, we will summarize the extra arguments needed to prove Theorem \ref{genthm}. To prove Theorem \ref{genthm}, we use a slightly different starting point than for the proof of Theorem \ref{mainthm}. In particular, fix $b>0$, $\lambda_{0,0,b} \in \Lambda$, and let $T_{0} > C_{1}$, where $C_{1}>100$ is some sufficiently large constant depending on $\lambda_{0,0,b}$. Then, we define the $X$ norm as in the proof of Theorem \ref{mainthm}, and write $\lambda = \lambda_{0,0,b}+e_{1}$, where $e_{1} \in \overline{B}_{1}(0) \subset X$. For all such $\lambda$, we then define $v_{1}$ exactly as in the main body of the paper. The main difference is a modification of $v_{2}$: Let $\psi \in C^{\infty}([0,\infty))$ satisfy
$$\psi(t) = \begin{cases} 1, \quad t \geq 200\\
0, \quad t \leq 100\end{cases} \quad 0 \leq \psi(t) \leq 1$$
and define
$$F(t) = \left(4 \int_{t}^{\infty} \frac{\lambda_{0,0,b}''(s) ds}{1+s-t}\right) \psi(t)$$
Then, $v_{2}$ is exactly as in the main body of the paper, except with $v_{2,0}$, the initial velocity of $v_{2}$, given below
\begin{equation} \widehat{v_{2,0}}(\xi) = -\frac{1}{\xi \pi} \int_{0}^{\infty} F(t) \sin(t\xi) dt\end{equation}
The point of this definition is that, by the sine transform inversion, $\widehat{v_{2,0}}$ solves
\begin{equation}\label{v2hateqn}-2 \int_{0}^{\infty} \sin(t\xi) \xi \widehat{v_{2,0}}(\xi) d\xi = F(t), \quad t \geq 0\end{equation}
$v_{3}$ is defined exactly as in the main body of the paper. Several estimates on $v_{k}$ in the main body of the paper used the fact that $\lambda'(x) <0$, which we no longer have in the setting of this appendix. Unless specified later on, analogs of all of these estimates are still true in the setting of this appendix, and can be proven by instead using that, for some $C_{2},C_{3}>0$, 
\begin{equation}\label{lambdacomp}\frac{C_{2}}{\log^{b}(t)} \leq \lambda(t) \leq \frac{C_{3}}{\log^{b}(t)}\end{equation}
and $t \rightarrow \frac{1}{\log^{b}(t)} \text{ is decreasing}$. The definitions of $v_{4},v_{5}$, and the equation resulting from $u_{ansatz}$ are the same as previously. The inner product of the $v_{1}$ linear error term with the re-scaled $\phi_{0}$ is unchanged. On the other hand, for $v_{2}$, we have
\begin{equation}\begin{split} \int_{0}^{\infty} \left(\frac{\cos(2Q_{1}(R))-1}{R^{2}\lambda(t)^{2}}\right) v_{2}(t,R\lambda(t)) \phi_{0}(R) R dR &= -2 \int_{0}^{\infty} \frac{\sin(t \xi)}{\lambda(t)} \xi \widehat{v_{2,0}}(\xi)d\xi -2 \int_{0}^{\infty} \sin(t\xi) \xi^{2}\left(K_{1}(\xi \lambda(t))-\frac{1}{\xi \lambda(t)}\right) \widehat{v_{2,0}}(\xi) d\xi\\
&= \frac{4}{\lambda(t)} \int_{t}^{\infty} \frac{\lambda_{0,0,b}''(s) ds}{1+s-t}  + E_{v_{2},ip}(t,\lambda(t)), \quad t \geq T_{0}\end{split}\end{equation}
where
$$E_{v_{2},ip}(t,\lambda(t)) = -2 \int_{0}^{\infty} \sin(t\xi) \xi^{2}\left(K_{1}(\xi\lambda(t))-\frac{1}{\xi\lambda(t)}\right) \widehat{v_{2,0}}(\xi) d\xi$$
and we used the fact that $\psi(t)=1, \quad t \geq 100$. In order to prove the pointwise estimates on $v_{2}$, we require some estimates on $\partial_{\xi}\left(\xi \widehat{v_{2,0}}\right)$. For these, we use 
\begin{equation} \xi \widehat{v_{2,0}}(\xi) = \frac{-1}{\pi} \int_{0}^{\infty} F(\frac{\sigma}{\xi}) \sin(\sigma) \frac{d\sigma}{\xi}\end{equation}
and can then differentiate under the integral sign, using the symbol type estimates on $\lambda_{0,0,b}$. The most crucial point to mention regarding the pointwise behavior of $v_{2}$ is the analogue of \eqref{v2precisenearorigin}, since it is this which allows us to prove the crucial near origin estimates on $F_{4}$. In the setting of this appendix, we prove the analogue of \eqref{v2precisenearorigin} by noting, for $r \leq \frac{t}{2}$,
\begin{equation}\label{v2precisegen}\begin{split} v_{2}(t,r) &= \int_{0}^{\infty} J_{1}(r\xi) \sin(t\xi) \widehat{v_{2,0}}(\xi) d\xi = \frac{r}{\pi} \int_{0}^{\pi} \frac{\sin^{2}(\theta) d\theta}{2} \int_{0}^{\infty} \xi \left(\sin(\xi(t+r\cos(\theta)))+\sin(\xi(t-r\cos(\theta)))\right) \widehat{v_{2,0}}(\xi) d\xi\\
&= \frac{-r}{4\pi} \int_{0}^{\pi} \sin^{2}(\theta) \left(F(t+r\cos(\theta))+F(t-r\cos(\theta))\right) d\theta\end{split}\end{equation}
and then using
$$F(t\pm r \cos(\theta)) = F(t) + \text{Err}(t,r), \quad |\text{Err}(t,r)| \leq \frac{C r }{t^{3} \log^{b}(t)}, \quad r \leq \frac{t}{2}, \quad t \geq T_{0}$$
The arguments for derivatives of $v_{2}$ are done similarly. The linear error term associated to $v_{3}$ is studied similarly to the main body of the paper. The modulation equation for $\lambda$ then has the same form as previously, except with the modified definition of $E_{v_{2},ip}$ given above, and the replacement
$$\frac{4b}{\lambda(t) t^{2} \log^{b}(t)} \rightarrow \frac{4}{\lambda(t)} \int_{t}^{\infty} \frac{\lambda_{0,0,b}''(s) ds}{1+s-t}$$
Recalling that $\lambda = \lambda_{0,0,b}+e_{1}$, we then substitute $e_{1}(t) = \lambda_{0,1}(t) + e(t)$, where
$$\lambda_{0,1}(t) = \int_{t}^{\infty} \int_{t_{1}}^{\infty} \frac{\lambda_{0,0,b}''(t_{2}) \log(\lambda_{0,0,b}(t_{2})))}{\log(t_{2})} dt_{2}dt_{1}$$ 
and proceed exactly as previously. The crucial kernel estimate \eqref{kineq} was previously proven using an argument which used $\lambda_{0,0}'(x)+\lambda_{0,1}'(x) \leq 0$, which is no longer true in this context. On the other hand, \eqref{kineq} is still true in the context of this appendix, and is proven with the same calculation of $\partial_{st} k(s,t)$ as before. Instead of using $\lambda_{0,0}'(t)+\lambda_{0,1}'(t) \leq 0$, we simply use $$-\lambda_{0}'(-t)(1-\alpha)\lambda_{0}(-t)^{-\alpha}+1 \geq \frac{7}{8}, \quad t \leq -T_{0}$$
which follows from the lower bound on $T_{0}$ imposed at the beginning of the entire argument of this appendix, and where $\lambda_{0}(t) = \lambda_{0,0,b}(t)+\lambda_{0,1}(t)$. Similarly, using the fact that $\lambda_{0}$ is comparable to $\frac{1}{\log^{b}(t)}$, which is a decreasing function, we carry out the same procedure as previously, to obtain the important resolvent kernel estimate \eqref{restimate}. In this context, the analog of \eqref{restimate} has some absolute constant $C$, rather than precisely 2, appearing on the right-hand side. Recalling the comment just above \eqref{lambdacomp}, the rest of the estimates on $v_{k}$ required to construct $\lambda(t)$ and estimate $\lambda^{(k)}(t), \quad k \leq 4$ follow from an argument similar to that used previously. \\
\\
Recall that the modulation equation in the setting of this appendix is of the same form as the one in the main body of the paper. In addition, the definition of $v_{1}$ is the same in both settings. Finally, \eqref{v2precisegen} is true. Combined, these imply that we can prove the crucial pointwise estimates on $F_{4}$, and its derivatives with the same procedure as before. Also, all other subsections involved in the construction of the ansatz section can be established similarly. Finally, we can then complete the argument exactly as before, with $E(v_{e},\partial_{t}\left(Q_{\frac{1}{\lambda(t)}}+v_{e}\right))$ having the same decay in $t$ as previously.

Department of Mathematics, University of California, Berkeley\\
\emph{E-mail address:} mpillai@berkeley.edu


\begin{thebibliography}{9}
\bibitem{bkt} I. Bejenaru, J. Krieger, D. Tataru, \emph{A codimension two stable manifold of near soliton equivariant wave maps}. Analysis and PDE, 6,4,(829), (2013).

\bibitem{ckls1} R. C$\hat{o}$te, C.E. Kenig, A. Lawrie, and W. Schlag, \emph{Characterization of Large Energy Solutions of the Equivariant Wave Map Problem: I}. Amer. J. Math. 137 (2015), no. 1, 139-207.

\bibitem{ckls} R. C$\hat{o}$te, C.E. Kenig, A. Lawrie, and W. Schlag, \emph{Characterization of Large Energy Solutions of the Equivariant Wave Map Problem: II}. Amer. J. Math. 137 (2015), no. 1, 209-250.

\bibitem{dk} R. Donninger, J. Krieger, \emph{Nonscattering solutions and blowup at infinity for the critical wave equation}. Math. Ann. 357 (2013), no. 1, 89-163.


\bibitem{gk} C. Gao, J. Krieger. \emph{Optimal polynomial blow-up range for critical wave maps}. Commun. Pure Appl. Anal., 14(5): 1705-1741, 2015.

\bibitem{gr} I. S. Gradshteyn and I.M. Ryzhik, \emph{Table of Integrals, Series, and Products}, Fifth Edition. Academic Press Inc, San Diego, CA, 1994.

\bibitem{inteqns}
G. Gripenberg, S.-O. Londen, and O. Staffans.
\textit{Volterra Integral and Functional Equations}.
Cambridge University Press, Cambridge, 1990.

\bibitem{j} J. Jendrej, \emph{Construction of two-bubble solutions for energy-critical wave equations}. arXiv:1602.06524v2 [math.AP], preprint.

\bibitem{jl} J. Jendrej, A. Lawrie, \emph{Two-bubble dynamics for threshold solutions to the wave maps equation}. Invent. Math. 213 (2018), 1249.

\bibitem{ktv} H. Koch, D. Tataru, M. Visan.
\textit{Dispersive Equations and Nonlinear Waves}.
Springer Basel, 2014.

\bibitem{km} J. Krieger, S. Miao, \emph{On stability of blow up solutions for the critical co-rotational Wave Maps problem}. arXiv:1803.02706v1 [math.AP], preprint.

\bibitem{kst} J. Krieger, W. Schlag, D. Tataru, \emph{Renormalization and blow up for charge one equivariant critical wave maps}. Invent. Math. 171 (2008), no. 3, 543-615.

\bibitem{kstym} J. Krieger, W. Schlag, D. Tataru, \emph{Renormalization and blow up for the critical Yang-Mills problem}. Adv. Math. 221 (2009), no. 5, 1445-1521.

\bibitem{kstslw} J. Krieger, W. Schlag, D. Tataru, \emph{Slow blow-up solutions for the $H^1(\mathbb{R}^3)$ critical focusing semi-linear wave equation in $\mathbb{R}^3$}. Duke Math. J. 147 (2009), no. 1, 1-53.

\bibitem{lo}A. Lawrie, SJ. Oh. \emph{A Refined Threshold Theorem for (1 + 2)-Dimensional Wave Maps into Surfaces}. Comm. Math. Phys. 342: 989, 2016.

\bibitem{rr} P. Rapha$\ddot{e}$l and I. Rodnianski, \emph{Stable blow up dynamics for the critical co-reotational wave maps and equivariant Yang-Mills problems}. Publ. Math. Inst. Hautes $\acute{E}$tudes Sci., pages 1-122, 2012.

\bibitem{rs} I. Rodnianski, J. Sterbenz, \emph{On the formation of singularities in the critical $O(3)$ $\sigma$-model}. Ann. of Math. (2), 172(1):187-242, 2010.

\bibitem{r} C. Rodriguez, \emph{Threshold dynamics for corotational wave maps}. arXiv:1809.01745 [math.AP], preprint.


\bibitem{stz} J. Shatah, A. Tahvildar-Zadeh, \emph{On the Cauchy Problem for Equivariant Wave Maps}. Comm. Pure Appl. Math. 47: 719-754, 1994.

\bibitem{stthreshold1} J. Sterbenz and D. Tataru, \emph{Energy dispersed large data wave maps in 2+1 dimensions}. Comm. Math. Phys. 298(1):139-230, 2010.

\bibitem{stthreshold2} J. Sterbenz and D. Tataru, \emph{Regularity of wave maps in dimension 2+1}. Comm. Math. Phys. 298(1):231-264, 2010.




\end{thebibliography}
\end{document}